\documentclass{article}
\usepackage{amsfonts}
\usepackage{amsmath,amsthm}

\usepackage[all]{xy}
\usepackage{graphicx}
\usepackage{hyperref}
 
\usepackage{amssymb}
\usepackage{latexsym}

\DeclareMathAlphabet\EuFrak{U}{euf}{m}{n}	
\SetMathAlphabet\EuFrak{bold}{U}{euf}{b}{n}	

\newcommand{\lra}{\longrightarrow}

\newcommand{\ovl}{\overline}
\newcommand{\unl}{\underline}
\newcommand{\wa}{\widehat}
\newcommand{\wt}{\widetilde}

\newcommand{\e}{\`e }

\newcommand{\sC}{{\it C*}-}
\newcommand{\bC}{{\mathbb C}}
\newcommand{\bR}{{\mathbb R}}
\newcommand{\bT}{{\mathbb T}}

\newcommand{\bZ}{{\mathbb Z}}

\newcommand{\bN}{{\mathbb N}}
\newcommand{\bK}{{\mathbb K}}

\newcommand{\bQ}{{\mathbb Q}}

\newcommand{\ud}{{{\mathbb U}(d)}}

\newcommand{\eps}{\varepsilon}

\newcommand{\mA}{\mathcal A}
\newcommand{\mB}{\mathcal B}
\newcommand{\mC}{\mathcal C}
\newcommand{\mD}{\mathcal D}
\newcommand{\mE}{\mathcal E}
\newcommand{\mF}{\mathcal F}
\newcommand{\mG}{\mathcal G}
\newcommand{\mH}{\mathcal H}
\newcommand{\mI}{\mathcal I}

\newcommand{\mK}{\mathcal K}
\newcommand{\mL}{\mathcal L}
\newcommand{\mM}{\mathcal M}
\newcommand{\mN}{\mathcal N}
\newcommand{\mO}{\mathcal O}
\newcommand{\mP}{\mathcal P}

\newcommand{\mR}{\mathcal R}
\newcommand{\mS}{\mathcal S}

\newcommand{\mX}{\mathcal X}

\newcommand{\mV}{\mathcal V}

\newcommand{\res}{{\mathrm{Res}} }

\newtheorem{thm}{Teorema}[section]
\newtheorem{cor}[thm]{Corollario}
\newtheorem{lem}[thm]{Lemma}
\newtheorem{prop}[thm]{Proposizione}

\newtheorem{defn}[thm]{Definizione}

\newtheorem{rem}{Osservazione}[section]

\theoremstyle{definition}
\newtheorem{ex}{Esempio}[section]

\theoremstyle{remark}

\numberwithin{equation}{section}

\begin{document}

\author{{\sf Ezio Vasselli}
}

\title{ Alcune Note di Analisi Matematica }
\maketitle


\

\tableofcontents

\newpage
\section{Introduzione.}
\label{sec_intro}

In queste note vengono discussi argomenti si solito trattati in corsi di Analisi reale e complessa, \ Analisi Funzionale ed Analisi superiore. Viene assunta la conoscenza dei fondamenti di Topologia, Algebra lineare ed Analisi (nozioni elementari sulla continuit\'a, calcolo differenziale ed integrale in una e pi\'u variabili reali).

Essendo mia intenzione restare intorno a 200 pagine di lunghezza penso di avere abbondato in concisione, il che spero abbia agevolato la chiarezza dell'esposizione piuttosto che pregiudicarla. 
D'altra parte ho cercato di essere auto-contenuto nei limiti del possibile; in particolare, le dimostrazioni dei risultati principali sono svolte in modo completo, e quelle volte in cui, invece, viene fornito solo un argomento di massima ci\'o \e debitamente evidenziato. A volte vi sono accenni ad argomenti non prettamente attinenti l'analisi (ad esempio, l'omologia in relazione al teorema di Brouwer in \S \ref{sec_fp} o la dualit\'a di de Rham in \S \ref{sec_fdif}): ci\'o \e voluto ed \e inteso a stimolare la curiosit\'a di chi legge, nonch\'e a supportare il punto di vista che la matematica non \e divisa in compartimenti stagni.

Ho tenuto quel livello di astrazione che considero utile per la comprensione dei risultati. Ad esempio, penso sia controproducente trattare i teoremi di passaggio al limite sotto il segno di integrale limitandosi al caso della retta reale, quando le stesse dimostrazioni si applicano a pi\'u generici spazi misurabili. A parte l'ovvio vantaggio di avere teoremi validi in un ambito pi\'u generale, muovendosi ad un livello pi\'u astratto si ha la possibilit\'a di capire quali sono le propriet\'a dell'oggetto "concreto" (nella fattispecie, la retta reale) cruciali per la dimostrazione dei risultati. D'altro canto, spero di aver inserito un numero accettabile di esempi.

Ci sono argomenti che mi riprometto di includere in una futura versione, pur cercando di non superare i limiti di cui ho scritto nelle righe precedenti. Tra questi ci sono senz'altro un'esposizione completa dei teoremi di Urysohn, Tietze e Stone-Weierstrass (\S \ref{sec_SW}) ed una discussione sulla mo- nodromia nell'ambito delle funzioni di variabile complessa (in particolare il logaritmo). Ovviamente la scelta degli argomenti trattati, e la misura del loro approfondimento, sono del tutto personali e quindi opinabili.

Molti esercizi sono ripresi da altre fonti; nella maggiorparte dei casi ho inserito la referenza originale, dove spesso (e volentieri, suppongo) il lettore potr\'a trovarne la soluzione. Alcuni esercizi provengono da prove di esame per concorsi di ricercatore, e sono stati inseriti in quanto mi sono sembrati interessanti a livello pedagogico. In altri casi gli esercizi sono dei veri e propri complementi, come ad esempio il lemma di Borel-Cantelli (Es.\ref{sec_MIS}.5), il lemma di Riesz (Es.\ref{sec_afunct}.4), e le convoluzioni di misure (Es.\ref{sec_func_n}.7).

\

A prescindere dalle fonti utilizzate \e possibile che queste note non siano scevre di inesattezze, errori matematici o di esposizione, ed ogni segnalazione in merito \e benvenuta.

\

Segue un elenco dei capitoli con relativi commenti e referenze.

In \S \ref{sec_top} richiamiamo alcuni risultati di topologia generale di interesse in analisi, inclusi i teoremi di Tietze e Stone-Weierstrass, e dimostriamo i teoremi delle contrazioni e di Ascoli-Arzel\'a.

Rudimenti sulle equazioni differenziali ordinarie vengono dati in \S \ref{sec_eq_diff}. Per questa sezione mi sono basato su \cite{Giu2,Tes}, ad eccezione dell'Esempio \ref{ex_AcqP}, ripreso da vecchie dispense del Prof. P. Acquistapace, e del Teorema di Peano.

In \S \ref{sec_MIS} vengono trattati argomenti classici come la teoria della misura e dell'integrazione secondo Lebesgue, i teoremi sul passaggio al limite sotto il segno di integrale e le funzioni AC-BV. Fonti principali sono \cite{Roy,Kol,Car}.

La sezione \ref{sec_Lp} contiene una discussione sugli spazi $L^p$, con particolare attenzione alla completezza (Teorema di Fischer-Riesz) ed alla dualit\'a di Riesz.

\S \ref{sec_func_n} \e una raccolta piuttosto eterogenea di appunti su funzioni di pi\'u variabili: include una dimostrazione del Teorema della funzione implicita, una breve discussione sulle forme differenziali, i fondamenti del calcolo variazionale e, ad un livello un p\'o pi\'u avanzato, i teoremi di Fubini e le convoluzioni, queste ultime (brevemente) discusse anche dai punti di vista dell'analisi funzionale e dell'analisi armonica astratta. Fonti principali sono \cite{Giu2,Bre}.

La sezione \ref{sec_afunct} \e sicuramente quella pi\'u approfondita e concerne l'analisi funzionale. Accanto ai rudimenti sugli spazi di Banach e di Hilbert, ed a risultati orientati verso la teoria delle equazioni alle derivate parziali (Teoremi di Stampacchia-Lax-Milgram, teorema di Schauder), il lettore trover\'a un'esposizione dei fondamenti delle algebre di operatori e delle distribuzioni, argomenti, questi ultimi, di interesse in fisica matematica ed in meccanica quantistica in particolare. Particolarmente corposa \e la sezione degli esercizi, dove vengono approfonditi gli aspetti inerenti la teoria spettrale e le connessioni tra spazi di Hilbert, teoria della misura ed algebre di operatori. Le fonti principali sono \cite{Bre,Kol,Rud,Ped}.

L'analisi di Fourier viene trattata in \S \ref{sec_fourier}. Oltre agli argomenti classici, vengono dati alcuni accenni ai gruppi topologici ed alla trasformata di Fourier astratta. Per le serie di Fourier ho seguito \cite{Giu2}, mentre per la trasformata di Fourier mi sono basato su \cite{dBar}.

\S \ref{sec_compl} concerne i fondamenti dell'analisi complessa ed \e fortemente debitrice degli appunti di un corso di Istituzioni di geometria superiore tenuto dal Prof. E. Arbarello alcuni anni fa (met\'a anni novanta). Una versione "ufficiale" di questi appunti, scritta dai Prof. Arbarello e Salvati-Manni, \e reperibile alla pagina web
{\footnote{Va da s\'e che le citazioni relative a pagine web potrebbero diventare obsolete.}}
\cite{Arb}.

In \S \ref{sec_sobolev} vengono dati alcuni accenni sugli spazi di Sobolev e dimostrati risultati di esistenza ed unicit\'a per problemi alle derivate parziali. Ho seguito in modo piuttosto pedissequo \cite{Bre}.

\

\newpage
\section{Alcuni risultati di topologia generale.}
\label{sec_top}

In questa sezione richiamiamo, senza pretese di esaustivit\'a, alcune nozioni di topologia generale che verranno usate spesso nel seguito.
Successivamente dimostriamo alcuni risultati a priori prettamente topologici, come il teorema delle contrazioni e quello di Ascoli-Arzel\'a, i quali hanno per\'o importanti applicazioni in analisi.

\subsection{Metrizzabilit\'a e compattezza.}

Sia $X$ uno spazio topologico con topologia $\tau X$.
Un {\em intorno di} $x \in X$ \e un sottoinsieme $U$ di $X$, tale che $x \in U' \subseteq U$ per qualche aperto $U' \in \tau X$.
Diciamo che $X$ \e {\em separabile} se esiste una successione $X_0 := \{ x_n \in X \}$ {\em densa} in $X$, il che vuol dire che per ogni aperto $Y \neq \emptyset$ esiste $n \in \bN$ tale che $x_n \in Y$.
Lo spazio $X$ si dice {\em compatto} se ogni ricoprimento aperto 
$\{ Y_\alpha \in \tau X \}$, $\cup_\alpha Y_\alpha = X$,
ammette un sottoricoprimento finito
$\{ Y_{\alpha_1} , \ldots , Y_{\alpha_n} \}$, $\cup_i^n Y_{\alpha_i} = X$.

Richiamiamo ora la nozione di {\em spazio metrico}. Se $X$ \e un insieme, allora una {\em metrica} su $X$ \e il dato di una funzione
\[
d : X \times X \to \bR^+
\]
tale che
\[
d(x,x) = 0 
\ \ , \ \
d(x,y) + d(y,z) \leq d(x,z)
\ \ , \ \
x,y,z \in X
\ .
\]
Data una metrica, \e possibile definire su $X$ la topologia avente come sottobase la famiglia dei {\em dischi}
\[
\Delta(x,\delta) := \{ y \in X : d(x,y) < \delta \}
\ \ , \ \
x \in X , \delta > 0
\ .
\]
Una successione $\{ x_n \} \subseteq X$ si dice {\em di Cauchy} se per ogni $\eps > 0$ esiste $n_0 \in \bN$ tale che $d(x_n,x_m) < \eps$, $\forall n,m \geq n_0$. Diremo che $\{ x_n \}$ {\em converge} ad $x \in X$ se per ogni $\eps > 0$ esiste $n_0 \in \bN$ tale che $d(x,x_n) < \eps$, $\forall n > n_0$. Uno spazio metrico $(X,d)$ si dice {\em completo} se ogni successione di Cauchy converge ad un elemento di $X$.
\begin{prop}
\label{prop_smc}
Uno spazio metrico compatto $( X,d )$ \e separabile.
\end{prop}

\begin{proof}[Dimostrazione]
Per ogni $n \in \bN$, consideriamo il ricoprimento $\Delta_n :=$ $\left\{ \Delta (x,1/n) , x \in X \right\}$, ne estraiamo un sottoricoprimento finito $\Delta_n^f$ e denotiamo con $X_n$ l'insieme dei centri dei dischi in $\Delta_n^f$. $X_0 :=$ $\cup_n X_n$ fornisce il sottoinsieme denso e numerable desiderato.
\end{proof}

Ricordiamo che uno spazio topologico $X$ \e {\em sequenzialmente compatto} se ogni successione $\{ x_n \} \subset X$ ammette una sottosuccessione convergente ad un elemento di $x$. In genere uno spazio compatto non \e sequenzialmente compatto, tuttavia si ha il seguente risultato:
\begin{prop}{ \bf{\cite[Teo.11.7]{Cam}}.}
\label{prop_comp}
Sia $X$ uno spazio metrico. Allora $X$ \e sequenzialmente compatto se e solo se \e compatto.
\end{prop}

Nell'ambito degli spazi metrici trova naturale collocazione la nozione di {\em uniforme} continuit\'a.
Dati gli spazi metrici $(X,d)$, $(X',d')$, un'applicazione $f : X \to Y$ si dice {\em uniformemente continua} se per ogni $\eps > 0$ esiste $\delta > 0$ tale che $d'(f(x),f(y)) < \eps$ per ogni coppia $x,y$ tale che $d(x,y) < \delta$. 
Enfatizziamo il fatto che, a differenza dell'usuale continuit\'a definita sugli spazi metrici, $\delta$ non dipende dalla scelta di $x,y$. 
Il {\em Teorema di Heine-Cantor} afferma che se $X$ \e metrico e compatto, allora ogni applicazione continua a valori in uno spazio metrico $Y$ \e anche uniformemente continua.

\subsection{Alcune propriet\'a delle funzioni continue.}
\label{sec_SW}

Sia $X$ uno spazio topologico e $C(X)$ l'algebra delle funzioni continue su $X$ a valori reali
{\footnote{Con il termine {\em algebra} intendiamo uno spazio vettoriale equipaggiato con un prodotto (associativo e distributivo). E' chiaro che le funzioni continue formano un'algebra rispetto alle operazioni di combinazione lineare e moltiplicazione.}}.
Un ben noto teorema di Weierstrass afferma che se $X$ \e compatto allora
\[
\| f \|_\infty := \sup_{x \in X} |f(x)| < + \infty  \ ;
\]
l'applicazione $\| \cdot \|_\infty$ soddisfa le propriet\'a
\[
\| f \|_\infty = 0 \ \Rightarrow \ f=0
\ \ , \ \
\| f+g \|_\infty \leq \| f \|_\infty + \| g \|_\infty
\ , \
\forall f,g \in C(X) \ ,
\]
dunque \e una {\em norma} nel senso di \S \ref{sec_afunct} (o, equivalentemente, $C(X)$ \e uno {\em spazio normato}).
Assumiamo ora, pi\'u in generale, che $X$ sia uno spazio {\em localmente compatto}, il che vuol dire che ogni $x \in X$ ammette un intorno compatto; diciamo che $f \in C(X)$ si annulla all'infinito se per ogni $\eps > 0$ esiste un compatto $K_\eps \subset X$ tale che
\begin{equation}
\label{eq_norm_infty}
\sup_{x \in X - K_\eps} |f(x)| < \eps \ .
\end{equation}
Denotiamo con $C_0(X)$ l'insieme delle funzioni continue che si annullano all'infinito, il quale, al pari di $C(X)$, \e un'algebra. Ora, (\ref{eq_norm_infty}) ed il teorema di Weierstrass implicano che 
\[
\| f \|_\infty = 
\sup_{x \in K_\eps}|f(x)| + \sup_{x \in X-K_\eps}|f(x)| < 
\sup_{x \in K_\eps}|f(x)| + \eps < 
+ \infty
\ \ , \ \
\forall f \in C_0(X)
\ ,
\]
per cui $\| \cdot \|_\infty$ \e ben definita su $C_0(X)$, il quale \e quindi uno spazio normato.
Nel risultato seguente stabiliamo che $C(X)$, $C_0(X)$ sono completi, ovvero {\em spazi di Banach} nel senso di \S \ref{sec_afunct}.
\begin{prop}
\label{prop_C0X}
Sia $X$ uno spazio localmente compatto ed $\{ f_n \} \subset C_0(X)$ tale che
\[
\| f_n - f_m \|_\infty \stackrel{n,m}{\to} 0
\ .
\]
Allora esiste ed \e unica $f \in C_0(X)$ tale che $\lim_n \| f-f_n \|_\infty = 0$, ed $\{ f_n \}$ converge uniformemente ad $f$.
Lo stesso vale per successioni in $C(X)$ nel caso in cui $X$ sia compatto, con $f \in C(X)$.
\end{prop}

\begin{proof}[Dimostrazione]
Scelto $\eps > 0$ esiste $n_\eps \in \bN$ tale che
\begin{equation}
\label{eq.convCX}
\| f_m-f_n \|_\infty < \eps
\ , \
\forall n,m \geq n_\eps
\ \Rightarrow \
|f_n(x)-f_m(x)| < \eps
\ , \
\forall x \in X
\ .
\end{equation}
Dunque ogni successione 
$\{ f_n(x) \}$, $x \in X$,
\e di Cauchy in $\bR$ ed esiste il limite $f(x) := \lim_n f_n(x)$. Ci\'o definisce un'unica funzione 
$f : X \to \bR$,
alla quale $\{ f_n \}$ converge uniformemente grazie al fatto che $n_\eps$ non dipende da $x$. Grazie all'uniformit\'a della convergenza troviamo
\[
|f(x)-f_n(x)| < \eps 
\ , \
\forall n \geq n_\eps
\ , \
x \in X
\ \Leftrightarrow \
\| f-f_n \|_\infty < \eps
\ , \
\forall n \geq n_\eps
\ ,
\]
e quindi $\lim_n \| f-f_n \|_\infty = 0$.
Per verificare che $f$ \e continua scegliamo $\eps > 0$ ed osserviamo che esiste $n_\eps \in \bN$ tale che $\| f-f_{n_\eps} \|_\infty < \eps / 3$; per continuit\'a di $f_{n_\eps}$ esiste un intorno $U_\eps  \ni x$ tale che $| f_{n_\eps}(x)-f_{n_\eps}(x') | < \eps / 3$ per ogni $x' \in U_\eps$, per cui (sommando e sottraendo $f_{n_\eps}(x)$, $f_{n_\eps}(x')$) abbiamo la stima
\[
| f(x)-f(x') | \leq 
2 \| f-f_{n_\eps} \|_\infty + | f_{n_\eps}(x)-f_{n_\eps}(x') | < \eps \ .
\]
Dunque $f$ \e continua. 
Verifichiamo infine, nel caso localmente compatto, che $f$ si annulla all'infinito: scelto $\eps > 0$ sappiamo che esistono $n_\eps \in \bN$ con $\| f-f_{n_\eps} \|_\infty < \eps / 2$ ed un compatto $K_\eps$ con $\sup_{x \in X-K_\eps}|f_{n_\eps}(x)| < \eps / 2$. Per cui,
\[
\sup_{x \in X - K_\eps} |f(x)| 
\leq
\sup_{x \in X - K_\eps} |f(x) - f_{n_\eps}(x)|   +   \sup_{x \in X-K_\eps} |f_{n_\eps}(x)|
\leq
\| f-f_{n_\eps} \|_\infty  +  \sup_{X-K_\eps} |f_{n_\eps}(x) |
<
\eps 
\ ,
\]
e concludiamo che $f$ si annulla all'infinito.
\end{proof}

Nel caso di funzioni a valori complessi i risultati precedenti rimangono validi, e scriveremo rispettivamente $C(X,\bC)$, $C_0(X,\bC)$ per denotare i relativi spazi di Banach (nonch\'e algebre).
Ad uso futuro, denotiamo con $C_c(X) \subset C_0(X)$ lo spazio vettoriale delle funzioni a supporto compatto, il quale, a differenza di $C_0(X)$, {\em non} \e uno spazio di Banach (vedi Esercizio \ref{sec_top}.7).

\

\noindent \textbf{Cenni sul Teorema di Tietze.} Ricordiamo che uno spazio topologico $X$ \e {\em normale} se per ogni coppia $W,W'$ di chiusi disgiunti esistono aperti $U \supset W$, $U' \supset W'$, disgiunti anch'essi. 
E' semplice verificare che ogni spazio metrico \e normale (\cite[Prop.9.8]{Cam}). Inoltre, ogni spazio compatto e di Hausdorff \e normale (vedi \cite[Theorem 1.6.6]{Ped}).

\begin{thm} \textbf{(Il Lemma di Urysohn, \cite[Theorem 1.5.6]{Ped})}
\label{thm_ury}
Sia $X$ uno spazio normale e $W,W' \subset X$ chiusi disgiunti. Allora esiste $f : X \to [0,1]$ continua tale che $f|_W = 0$ e $f|_{W'} = 1$.
\end{thm}

\begin{thm} \textbf{(Tietze, \cite[Theorem 1.5.8]{Ped})}
Sia $X$ uno spazio normale, $W \subset X$ chiuso ed $f : W \to \bR$ continua. Allora esiste $\wt f \in C(X)$ tale che $\wt f|_W = f$.
\end{thm}

\

\noindent \textbf{Cenni sul teorema di Stone-Weierstrass.} Possiamo ora dare l'enunciato di un importante risultato di approsimazione.

\begin{defn}
Diciamo che un sottoinsieme $V$ di $C(X)$ \textbf{separa i punti di $X$} se per ogni $x , x' \in X$, $x \neq x'$, esiste $f \in V$ tale che $f(x) \neq f(x')$.
\end{defn}

\begin{ex}{\it
Sia $X := [0,1]$. Allora l'insieme $V$ delle funzioni del tipo $f(x) = a +bx$, $x \in X$, $a,b \in \bR$, separa i punti di $X$. Osservare che l'algebra generata da $V$ 
{\footnote{
Per definizione, l'algebra generata da un sottoinsieme $S$ di $C(X)$ \e lo spazio vettoriale generato da prodotti del tipo $f_1 \cdots f_n$, dove $n \in \bN$, $f_1 , \ldots , f_n \in S$.
}}
coincide con l'insieme dei polinomi nella variabile $x \in [0,1]$.
}
\end{ex}

\begin{thm}[Stone-Weierstrass]
\label{thm_sw}
Sia $X$ uno spazio compatto di Hausdorff ed $\mA \subset C(X)$ un'al- gebra che contiene le funzioni costanti e che separa i punti di $X$. Allora $\mA$ \e densa in $C(X)$ nella topologia della convergenza uniforme.
\end{thm}

\begin{cor}
\label{cor_SW}
Sia $X$ uno spazio localmente compatto di Hausdorff, ed $\mA \subset C_0(X)$ un'algebra che separa i punti di $X$ e tale che per ogni $x \in X$ esista $f \in \mA$ con $f(x) \neq 0$. Allora $\mA$ \e densa in $C_0(X)$ nella topologia della convergenza uniforme.
\end{cor}

I risultati precedenti si estendono al caso complesso aggiungendo l'ipotesi che $\mA$ sia chiusa rispetto al passaggio alla funzione coniugata
\begin{equation}
\label{eq.adj}
f \mapsto f^* : f^*(x) := \ovl{f(x)} 
\ , \
\forall x \in X
\ \ , \ \ 
f \in C(X,\bC)
\end{equation}
(ovvero, se $f \in \mA$ allora $f^* \in \mA$); per una loro dimostrazione, peraltro "elementare" nel senso che non richiede nozioni non standard, rimandiamo a \cite[\S 4.3]{Ped}.

\begin{rem}
{\it
La condizione (\ref{eq.adj}) \e indispensabile per ottenere la densit\'a di $\mA$ in $C(X,\bC)$ nel caso complesso. Ad esempio, prendiamo la palla unitaria $\Delta := \{ z \in \bC : |z| \leq 1 \}$ e l'algebra $\mO(\Delta)$ delle funzioni analitiche in $\Delta$ (vedi \S \ref{sec_anal}). E' chiaro che $\mO(\Delta)$ contiene le costanti e separa i punti di $\Delta$ (infatti, presi $z \neq z' \in \Delta$ la funzione $f(z) := z$, $z \in \Delta$, \e analitica e tale che $f(z) \neq f(z')$). Tuttavia, $\mO(\Delta)$ \e ben lungi dall'essere denso in $C(\Delta,\bC)$; fosse cos\'i troveremmo
$\mO(\Delta) = C(\Delta,\bC)$,
essendo $\mO(\Delta)$ completo rispetto alla norma $\| \cdot \|_\infty$ (vedi Teo.\ref{thm.morera.2}). Ma ci\'o  \e assurdo, in quanto la funzione $f^*(z) := \ovl{z}$, $z \in \Delta$, \e continua ma non analitica, come si dimostra usando le equazioni di Cauchy-Riemann (Lemma \ref{lem_CR}); l'esempio di $f^*$ mostra anche che $\mO(\Delta)$ non \e chiuso rispetto al passaggio al coniugato, il che spiega il motivo per cui il teorema di Stone-Weierstrass non \e applicabile a $\mO(\Delta)$, il quale \e quindi un sottospazio proprio di $C(\Delta,\bC)$.
}
\end{rem}

Il teorema di Stone-Weierstrass generalizza il classico teorema di densit\'a di Weierstrass (che si ritrova per $X = [0,1]$ ed $\mA$ l'algebra dei polinomi, vedi \cite[Teo.2.8.1]{Giu2}), nonch\'e il teorema di densit\'a dei polinomi trigonometrici nell'algebra delle funzioni continue e periodiche su $[0,2\pi]$ (vedi \cite[Theorem 4.25]{Rud}).
Osserviamo che le dimostrazioni dei due risultati di cui sopra, a differenza di Teo.\ref{thm_sw}, si basano sull'uso delle convoluzioni (\S \ref{sec_conv}).
Il Lemma seguente verr\'a utilizzato nel seguito (Prop.\ref{prop_tens_prod}):

\begin{lem}
\label{lem_ury}
Per ogni spazio localmente compatto di Hausdorff $X$, valgono le seguenti propriet\'a:
(1) Per ogni $x \in X$ esiste $f \in C_0(X)$ tale che $f(x) \neq 0$;
(2) $C_0(X)$ separa i punti di $X$.
\end{lem}

\begin{proof}[Dimostrazione]
Sia $Y \subset X$ compatto e tale che $x$ appartenga alla parte interna $\dot{Y}$. Essendo $X$ di Hausdorff, $Y$ \e anche chiuso (vedi \cite[Prop.10.6]{Cam} o \cite[1.6.5]{Ped}). Essendo $Y$ compatto e di Hausdorff, esso \e anche normale. Possiamo ora dimostrare i due punti dell'enunciato:
(1) Sia $U \subset \dot{Y}$, $U \ni x$; allora $W = Y-U$ \e chiuso sia in $Y$ che in $X$, e chiaramente disgiunto da $\{ x \}$. Per il Lemma di Urysohn esiste $f \in C(Y)$ tale che $f(x) = 1$ e $f|_W = 0$. Del resto, per costruzione $f$ si annulla sulla frontiera di $Y$, dunque estendiamo $f$ ad $X$ definendo $f|_{X-Y} := 0$ e cos\'i otteniamo la funzione cercata.
(2) Se $x' \neq x$, allora possiamo assumere che sia $x$ che $x'$ siano contenuti in un compatto $Y$. Scegliendo un intorno $U \subset \dot{Y}$, $U \ni x$, tale che $x' \notin U$ e ragionando come nel caso precedente concludiamo che esiste $f \in C_0(X)$ tale che $f(x)=1$, $f(x')=0$. 
\end{proof}

\subsection{Teoremi di punto fisso.}
\label{sec_fp}

Il teorema seguente \e il pi\'u classico tra quelli noti come {\em teoremi di punto fisso}. Motivato dalla que- stione dell'esistenza di soluzioni di equazioni differenziali (vedi \S \ref{sec_eq_diff}), ha segnato un importante passo in avanti dal punto di vista concettuale, quello per il quale una funzione si pu\'o riguardare come un "punto" in uno spazio topologico.
Tra le varie applicazioni menzioniamo il teorema di Cauchy (Teo.\ref{thm_Cauchy_PDO}), il teorema delle funzioni implicite (Teo.\ref{thm_din1}), ed i teoremi di Stampacchia-Lax-Milgram (Teo.\ref{thm_SLM1}).

\begin{thm}[Teorema delle contrazioni, Banach-Caccioppoli]
\label{thm_contr}
Sia $( X , d )$ uno spazio metrico completo e $T : X \to X$ un'applicazione continua tale che esista $\alpha \in (0,1)$ con $d ( Tx , Tx' ) \leq$ $\alpha d (x,x')$, $\forall x,x' \in X$. Allora esiste ed \e unico $\ovl x \in X$ tale che $T \ovl x =$ $\ovl x$.
\end{thm}

\begin{proof}[Dimostrazione]
Poniamo $x_n :=$ $T^n x$ e stimiamo
\[
d (x_{n+1} , x_n) 
\leq 
\alpha d ( x_n , x_{n-1} ) 
\leq \ldots \leq
\alpha^n d (x_1 , x)
\ .
\]
Inoltre, per diseguaglianza triangolare, 
\[
d ( x_{m+1} , x_n )
\leq
\sum_{k=n+1}^{m+1} d ( x_k , x_{k-1} )
\leq
\sum_{k=n+1}^{m+1} \alpha^k d ( x_1 , x )
=
\alpha^n \sum_{k=1}^{m-n} \alpha^k d ( x_1 , x )
\leq
 d ( x_1 , x ) \frac{\alpha^n}{1-\alpha}
\ ,
\]
dunque (avendosi $0 < \alpha < 1$) $\left\{ T^nx \right\}$ \e di Cauchy. Il punto limite $\ovl x$ soddisfa
per costruzione l'uguaglianza
\[
T \ovl x = T \lim_n T^nx  = \lim_n T^{n+1}x = \ovl x  \ ,
\]
dunque \e un punto fisso. Inoltre, se $x' \in X$ soddisfa $Tx' = x'$ allora troviamo $d (\ovl x,x') =$ $d ( T \ovl x , Tx') \leq$ $\alpha d (\ovl x , x')$, per cui $d (\ovl x , x') = 0$.
\end{proof}

Nelle righe che seguono discutiamo un altro risultato di punto fisso, il {\em Teorema di Brouwer}, il quale ha conseguenze importanti sia in analisi che in geometria. Nella sua forma pi\'u semplice, quella in dimensione uno, esso \e conseguenza di un teorema di Bolzano, il {\em teorema del valore intermedio}, il quale afferma che se $f : [a,b] \to \bR$ \e continua ed $f(a) f(b) < 0$, allora esiste $x \in (a,b)$ tale che $f(x) = 0$:
\begin{prop}\textbf{(Teorema di Brouwer in dimensione uno)}
Sia $f : [0,1] \to [0,1]$ continua. Allora esiste $x \in [0,1]$ tale che $f(x) = x$.
\end{prop}

\begin{proof}[Dimostrazione]
Se $f(0) = 0$ oppure $f(1) = 1$ non vi \e nulla da dimostrare, per cui assumiamo $f(0) \neq 0$ e $f(1) \neq 1$. Applichiamo allora il teorema del valore intermedio a $g(x) := f(x) - x$, $x \in [0,1]$. 
\end{proof}

Ora, si ha la seguente generalizzazione del risultato precedente:
\begin{thm}\textbf{(Teorema di Brouwer)}
Sia $S \subset \bR^n$, $n \in \bN$, un insieme convesso, compatto e non vuoto, ed $f : S \to S$ continua. Allora esiste $x \in S$ tale che $f(x) = x$.
\end{thm}
Ci limiteremo qui ad esporre l'idea della dimostrazione di un caso particolare del teorema di Brouwer, la quale fa uso dei spazi di {\em omologia}, riguardo i quali rimandiamo a \S \ref{sec_fdif} e, pi\'u in dettaglio, \cite[Cap.5]{Arb} (per un approccio diverso si veda \cite[\S 6.8]{Acq}). 
Per ogni $n \in \bN$, denotiamo con $D^n \subset \bR^n$ la palla unitaria (chiusa) e con $S^{n-1} \subset \bR^n$ la sfera unitaria, che identifichiamo con il bordo $\partial D^n \subset D^n$.
\begin{thm}
\label{teo_bro}
Sia $n \in \bN$, ed $f : D^n \to D^n$ un'applicazione continua. Allora esiste $x \in D^n$ tale che $f(x) = x$.
\end{thm}

\begin{proof}[Sketch della dimostrazione]
Supponiamo per assurdo che $f$ non abbia punti fissi. Allora per ogni $x \in D^n$ \e ben definito il punto $F(x) \in \partial D^n \simeq S^{n-1}$ come l'intersezione tra $\partial D^n$ e la retta passante per $x$ ed $f(x)$. Otteniamo cos\'i un'applicazione continua
\[
F : D^n \to S^{n-1}
\ \ {\mathrm{tale \ che}} \ \
F(x) = x \ , \ \forall x \in S^{n-1}
\ .
\]
Ora, l'idea \e quella di dimostrare che $F$ \e un {\em ritratto per deformazione}
{\footnote{
In generale, dato uno spazio topologico $X$ ed $S \subset X$, un ritratto per deformazione \e un'applicazione continua $F : X \to S$ tale che: (1) $F|_S$ \e l'identit\'a di $S$; (2) esiste un'applicazione continua (detta {\em omotopia}) $H : X \times [0,1] \to X$ tale che $H(x,0) = F(x)$, $H(x,1) = x$, $\forall x \in X$. Su questi argomenti rimandiamo ancora a \cite[Cap.5]{Arb}.
}}
, il che implica, per propriet\'a generali degli spazi di omologia, che si ha un'applicazione lineare iniettiva
\[
\wt F : H_{n-1}(S^{n-1}) \to H_{n-1}(D^n)
\ \ , \ \
n \in \bN
\ .
\]
D'altro canto 
$H_{n-1}(S^{n-1}) = \bR$
e
$H_{n-1}(D^n) = \{ 0 \}$
(vedi \cite[\S 5.5]{Arb}); ci\'o \e palesemente assurdo ($\bR$ non pu\'o essere un sottospazio di $\{ 0 \}$) cosicch\'e troviamo la contraddizione cercata.
\end{proof}

E' possibile dare versioni del teorema di Brouwer per spazi di Banach e localmente convessi. A proposito di questi risultati, motivati da questioni di esistenza di soluzioni di equazioni differenziali, si veda \S \ref{sec_gfp}.

\subsection{Il Teorema di Ascoli-Arzel\'a.}
\label{sec_AA}

Da un punto di vista topologico il teorema di Ascoli-Arzel\'a si pu\'o interpretare come una ge- neralizzazione del teorema di Bolzano-Weierstreiss agli spazi funzionali. Al livello analitico \e uno strumento importante per la verifica della compattezza di applicazioni definite sugli spazi di Banach classici ($C^k$,$L^p$), nonch\'e base per una dimostrazione del teorema di Peano di esistenza di soluzioni per le equazioni differenziali ordinarie.

Ricordiamo che un {\em precompatto} \e un sottoinsieme di uno spazio topologico con chiusura compatta. Cominciamo esibendo un insieme limitato, ma non precompatto nella topologia della convergenza uniforme, nello spazio funzionale $C([0,1])$.

\begin{ex}{\it 
La successione $\left\{ f_n \right\} \subset C([0,1])$, 
\[
f_n (x) :=
\left\{
\begin{array}{ll}
nx \ \ , \ \ x \in [0,1/n]
\\
1  \ \ , \ \ x \in [1/n,1]
\end{array}
\right.
\]
\e limitata e puntualmente convergente alla funzione caratteristica $\chi_{(0,1]}$ (che \e discontinua). Dunque $\{ f_n \}$ non ammette sottosuccessioni uniformemente convergenti e quindi, pur essendo limitata (si ha $\| f_n \|_\infty = 1$, $\forall n \in \bN$), non \e precompatta in $C([0,1])$.
}\end{ex}

\begin{defn}
Sia $X$ uno spazio metrico compatto ed $R$ uno spazio metrico. Una famiglia $\mF \subset$ $C(X,R)$ is dice \textbf{equicontinua} se per ogni $\eps > 0$ esiste un $\delta > 0$ tale che $d ( f(x') , f(x)) < \eps$ per ogni $f \in \mF$ ed $x,x' \in X$ tali che $d(x,x') < \delta$.
\end{defn}

Sottolineiamo il fatto che nella definizione precedente $\delta$ dipende {\em soltanto} da $\eps$ e non da $f$ od $x,x'$. Introduciamo la notazione
\[
\mF_x 
\ := \
\{ f(x) : f \in \mF \}
\ 
\subseteq R
\ \ , \ \
x \in X
\ .
\]

\begin{thm}[Ascoli-Arzel\'a]
\label{thm_AA}
Sia $X$ uno spazio metrico compatto, $R$ uno spazio metrico completo ed $\mF \subseteq$ $C(X,R)$. Allora $\mF$ \e precompatto nella topologia della convergenza uniforme se e soltanto se $\mF$ \e equicontinuo ed ogni $\mF_x$, $x \in X$, \e precompatto in $R$.
\end{thm}

Consideriamo ora il caso $R = \bR$: il teorema di Bolzano-Weierstrass ci assicura che se $\mF$ \e {\em equilimitato} (ovvero $\sup_{f \in \mF} \| f \|_\infty < \infty$)
{\footnote{A volte diremo, pi\'u semplicemente, che $\mF$ \e {\em limitato}.}}, 
allora ogni $\mF_x \subset \bR$ \e limitato e quindi precompatto. D'altra parte, se $\mF$ \e precompatto allora \e certamente limitato (vedi Esercizio \ref{sec_top}.3). Per cui, dato per buono Teo.\ref{thm_AA}, otteniamo i seguenti risultati:
\begin{thm}[Ascoli-Arzel\'a, forma classica]
\label{thm_AA1}
Sia $X$ uno spazio metrico compatto ed $\mF \subseteq C(X)$. Allora $\mF$ \e precompatto se e soltanto se $\mF$ \e equicontinuo ed equilimitato.
\end{thm}

\begin{cor}
\label{cor_AA}
Sia $X$ uno spazio metrico compatto ed $\left\{ f_n \right\} \subset$ $C(X)$ una successione equilimitata ed equicontinua. Allora esiste una sottosuccessione $\left\{ f_{n_k} \right\}$ uniformemente convergente.
\end{cor}

\begin{ex}{\it
Sia $X := [0,1]$ ed $\alpha \in (0,1]$. Consideriamo la famiglia delle funzioni holderiane di ordine $\alpha$ rispetto ad una fissata costante $c > 0$:
\[
\mF_{\alpha,c}
\ := \ 
\{ f \in C(X) : | f(x)-f(y) | \leq c |x-y|^\alpha , x,y \in X  \}
\ .
\]
Allora $\mF_{\alpha,c}$ \e equicontinua (infatti, per ogni $\eps > 0$ possiamo prendere un qualsiasi 
$\delta < (c^{-1} \eps)^{1 / \alpha}$,
indipendentemente da $f$), per cui ogni sottoinsieme limitato di $\mF_{\alpha,c}$ \e precompatto in $C(X)$. In particolare, ogni successione limitata $\{ f_n \} \subset \mF_{\alpha,c}$ ammette una sottosuccessione uniformemente convergente.
}
\end{ex}

Passiamo ora a dare la dimostrazione di Teo.\ref{thm_AA}. A tale scopo osserviamo che $C(X,R)$ \e uno spazio metrico rispetto alla distanza 
\[
d_\infty(f,\wt f) \ := \ \sup_x d(f(x),\wt f(x)) 
\ \ , \ \
f , \wt f \in C(X,R)
\ ,
\]
la quale induce la topologia della convergenza uniforme. Per cui un sottoinsieme di $C(X,R)$ \e compatto se e solo se \e sequenzialmente compatto (vedi Prop.\ref{prop_comp}).

\begin{proof}[Dimostrazione del Teorema \ref{thm_AA}]
{\em Sia $\mF$ precompatto} ed $\ovl \mF$ la sua chiusura in $C(X,R)$; nelle righe che seguono mostriamo che ogni $\mF_x$ \e precompatto in $R$ e che $\mF$ \e equicontinuo. 
Sia $x \in X$ ed $\{ f_n(x) , f_n \in \mF \}_n$ una successione in $\mF_x$; essendo $\mF$ precompatto, esiste $f \in \ovl \mF \subset C(X,R)$ tale che $d_\infty( f_{n_k} , f ) \to 0$ per qualche sottosuccessione $\{ f_{n_k} \}$. Ci\'o implica che
$d ( f_{n_k}(x) , f(x) ) \leq d_\infty ( f , f_{n_k} ) \to 0$,
per cui $\{ f_{n_k}(x) \}$ \e convergente ed $\mF_x$ precompatto.
Il fatto che $\mF$ \e equicontinuo si dimostra con il classico argomento "$3$-$\eps$": scelto $\eps > 0$, effettuiamo un ricoprimento finito di $\ovl \mF$ con dischi $\Delta ( f_k , \eps )$, $k = 1 , \ldots , n$; visto che $n$ \e finito, possiamo considerare i $\delta_1  , \ldots , \delta_n$ delle uniformi continuit\'a di $f_1 , \ldots , f_n$ e definire 
$\delta := \inf_{k=1,\ldots,n} \delta_k > 0$, 
in maniera tale che
\[
d ( \  f_k(x) \ , \  f_k(x') \ ) \ < \ \eps
\ \ , \ \
\forall x,x' \in X 
\ , \ 
d(x,x') < \delta
\ .
\]
Osserviamo che il nostro $\delta$ \e definito solo in base alla scelta del ricoprimento di $\mF$. Presa quindi $f \in \mF$ abbiamo $f \in \Delta ( f_k , \eps )$ per un qualche indice $k$; per cui, preso $x \in X$ ed $x' \in \Delta (x,\delta)$, troviamo
\[
d ( f(x)     , f(x')    )
\ \leq \
d (  f(x)    , f_k(x)   )  +  
d (  f_k(x)  , f_k(x')  )  + 
d (  f_k(x') , f(x')    )
\ \leq \ 
3 \eps
\ .
\]
Essendo $\delta$ funzione solo di $\eps$ (e non di $f$) concludiamo che $\mF$ \e equicontinuo.
{\em Assumiamo ora che $\mF$ sia equicontinuo e tale che ogni $\mF_x$, $x \in X$, sia precompatto}, e mostriamo che $\mF$ \e precompatto; il nostro compito \e verificare che presa una successione $\left\{ f_n \right\} \subseteq$ $\mF$, questa ammette una sottosuccessione uniformemente convergente. 
Come primo passo, consideriamo un sottoinsieme denso e numerabile $X_0 \subset X$ (vedi Prop.\ref{prop_smc}), e costruiamo una sottosuccessione di $\left\{ f_n \right\}$ convergente in modo puntuale in $X_0$. Posto $X_0 :=$ $\left\{ x_m \right\}$, osserviamo che, essendo $\mF_{x_1}$ precompatto, la successione
$\left\{ f_n (x_1) \right\}$
ammette una sottosuccessione convergente, che denotiamo con $\left\{ f_{1,n} (x_1) \right\}$. 
Passiamo quindi a considerare $\left\{ f_{1,n} (x_2) \right\}$ ed ad estrarre una sottosuccessione convergente $\left\{ f_{2,n} (x_2) \right\}$, ottendendo cos\'i che $\left\{ f_{2,n} \right\}$ converge in $x_1$ ed $x_2$. Procedendo induttivamente otteniamo una collezione di sottosuccessioni $\left\{ f_{m,n} \right\}$ tale che $\left\{ g_n := f_{n,n} \right\}$ converge puntualmente in $X_0$.
Infine, dimostriamo che $\{ g_n \}$ converge uniformemente in $X$. Consideriamo $\eps > 0$; per equicontinuit\'a esiste $\delta_\eps > 0$ tale che $d(x,x') < \delta_\eps$ implica 
\[
d ( g_n (x) , g_n(x') ) \ < \ \eps/3
\ \ , \ \
\forall n \in \bN
\ .
\]
Scegliamo quindi $m_\eps \in \bN$ tale che $m_\eps > 1 / \delta_\eps$, cosicch\'e $\Delta^f_{1 / m_\eps}$ ricopre $X$ (vedi dimostrazione di Prop.\ref{prop_smc}). Osserviamo che poich\'e $\left\{ g_n \right\}$ converge puntualmente in $X_{m_\eps} \subset X_0$, esiste $n_\eps \in \bN$ tale che per ogni $n,m > n_\eps$ risulta
\[
d ( g_n (y) , g_m (y)  ) 
\ < \ 
\eps / 3 
\ \ , \ \ 
\forall y \in X_{m_\eps} 
\]
(osservare che $X_{m_\eps}$ \e finito, altrimenti avremmo dei problemi inerenti la convergenza non uniforme di $\left\{ g_n \right\}$). Sottolineiamo che il nostro $n_\eps \in \bN$ dipende, in ultima analisi, solo da $\eps$. Ora, se $x \in X$ troviamo $x \in \Delta (y,\delta_\eps)$ per qualche $y \in X_{m_\eps}$ (infatti, $1/{m_\eps} < \delta_\eps$), e
\[
d ( g_n(x) , g_m(x) )
\ \leq \
d ( g_n(x)  , g_n(y)  ) + 
d ( g_n(y)  , g_m(y)  ) +  
d ( g_m(y)  , g_m(x)  )
\ <  \
\eps
\ .
\]
\end{proof}

\subsection{Esercizi.}

\noindent \textbf{Esercizio \ref{sec_top}.1.} {\it Fissato $\alpha > 0$, denotiamo con $L_\alpha$ l'insieme delle funzioni $f : [0,1] \to \bR$ continue a tratti e tali che $\sup_{x \in [0,1]} |f(x)| \leq \alpha$. Si mostri che la famiglia 
\[
\mF_\alpha
\ := \
\{ F : [0,1] \to \bR \ , \ F(x) := \int_0^x f(t) \ dt \ , x \in [0,1]  \ : \ f \in L_\alpha  \}
\]
\'e equicontinua e limitata.
}

\

\noindent \textbf{Esercizio \ref{sec_top}.2.} {\it Sia $\alpha \in (0,1)$ e $\tau \in C(\bR)$ una funzione con costante di Lipschitz $\alpha$ (ovvero: $|\tau(x)-\tau(y)| \leq \alpha |x-y|$, $\forall x,y \in \bR$). Scelto $c > 0$, si mostri che l'applicazione
\[
T : C([0,1]) \to C([0,1])
\ \ , \ \
f \mapsto Tf 
\ : \
Tf(s) := c + \int_0^s \tau \circ f (t) \ dt
\ , \
\forall s \in [0,1]
\ ,
\]
\'e una contrazione. Si dimostri che esiste ed \e unica $f_0 \in C([0,1])$ tale che $f_0 = Tf_0$. Si dimostri inoltre che $f_0$ \e derivabile e che $f_0' = \tau \circ f_0$. Infine, si calcoli $f_0$ nel caso $\tau(x) = \alpha x$, $\forall x \in \bR$.

\

\noindent (Suggerimenti: per il secondo quesito ovviamente si applica il teorema delle contrazioni. Riguardo il terzo quesito si applichi il teorema fondamentale del calcolo derivando membro a membro l'uguaglian- za $f_0 = Tf_0$. Riguardo il quarto quesito, si osservi che derivando membro a membro l'uguaglianza $f_0 = Tf_0$ in questo caso si ottiene la pi\'u semplice delle equazioni differenziali ordinarie).
}

\

\noindent \textbf{Esercizio \ref{sec_top}.3.} {\it Sia $X$ uno spazio metrico compatto. Si mostri che se $\mF \subset C(X)$ \e precompatto allora \e limitato.

\

\noindent (Suggerimento: ragionando per assurdo, si assuma che $\mF$ sia non limitato e si deduca che esiste una successione $\{ f_n \} \subset \mF$ con $\| f_n \|_\infty \geq n$, $\forall n \in \bN$; si osservi infine che tale successione non pu\'o avere sottosuccessioni convergenti).
}

\

\noindent \textbf{Esercizio \ref{sec_top}.4.} {\it Sia $X$ uno spazio metrico compatto e $K \in C(X \times X)$. Per ogni $x \in X$ si definisca $\kappa_x(y) := K(x,y)$, $y \in X$, e si dimostrino le seguenti propriet\'a:
\textbf{(1)} $\kappa_x \in C(X)$ per ogni $x \in X$;
\textbf{(2)} La famiglia $\mF := \{ \kappa_x \}_{x \in X} \subset C(X)$ \e equicontinua.

\

\noindent (Suggerimenti: si osservi che $X \times X$ \e metrico
{\footnote{
Qui usiamo la metrica euclidea
$d_2( (x,y) , (x',y') ) := \sqrt{ d(x,x')^2 + d(y,y')^2 }$, $x,x',y,y' \in X$.
}}
e compatto, per cui $K$ \e uniformemente continua; si osservi quindi che
$d_2( (x,y) , (x,y') ) = d(y,y')$, $\forall x,y,y' \in X$).
}

\

\noindent \textbf{Esercizio \ref{sec_top}.5.} {\it Sia $X$ uno spazio metrico compatto ed $\{ f_n \} \subset C(X)$ una successione precompatta e convergente puntualmente ad $f : X \to \bR$. Si mostri che $\{ f_n \}$ converge uniformemente ad $f$ (cosicch\'e $f$ \e continua).

\

\noindent (Suggerimento: si supponga per assurdo che esiste $\eps > 0$ tale che $\| f-f_k \|_\infty \geq \eps$ per ogni $k$ appartenente ad un sottoinsieme infinito di $\bN$. Essendo $\{ f_n \}$ precompatta, esiste una sottosuccessione $\{ f_{k_i} \}$ di $\{ f_k \}$ uniformemente convergente; ma $\{ f_{k_i} \}$ deve necessariamente convergere ad $f$ per convergenza puntuale di $\{ f_n \}$, il che fornisce la contraddizione cercata).
}

\

\noindent \textbf{Esercizio \ref{sec_top}.6.} {\it Sia $\{ f_n : \bR \to \bR \}$ una successione di funzioni convesse puntualmente convergente ad $f : \bR \to \bR$. Presi $\alpha < a < b < \beta$, si verifichino le seguenti propriet\'a:
\textbf{(1)} Le successioni
\[
m_n := ( f_n(a) - f_n(\alpha) )(a-\alpha)^{-1}
\ \ , \ \
M_n := (f_n(\beta)-f_n(b))(\beta-b)^{-1}
\ \ , \ \
n \in \bN
\ ,
\]
sono limitate;
\textbf{(2)} Usando il punto precedente e Prop.\ref{prop_conv}(1), si mostri che esiste $L > 0$ tale che
\[
|f_n(x)-f_n(y)| \ \leq \ L|x-y|
\ \ , \ \
\forall n \in \bN
\ , \
x,y \in [a,b]
\ .
\]
\textbf{(3)} Usando il punto precedente, si verifichi che esiste $C > 0$ tale che
\[
|f_n(x)| \ \leq \ C + L(b-a)
\ \ , \ \
\forall n \in \bN
\ , \
x \in [a,b]
\ .
\]
\textbf{(4)} Usando l'Esercizio \ref{sec_top}.5 ed i punti precedenti, si verifichi che $\{ f_n|_{[a,b]} \}$ converge uniformemente.

\

\noindent (Suggerimento: per il punto 4 si usi il teorema di Ascoli-Arzel\'a).

}

\

\noindent \textbf{Esercizio \ref{sec_top}.7.} {\it Presa la successione di funzioni
\[
f_n : \bR \to \bR 
\ , \ 
n \in \bN
\ \ : \ \
f_n(x) := 
\left\{
\begin{array}{ll}
e^{-x^2} \ \ , \ \ |x| < n
\\
e^{-n^2}[ 1-2n(|x|-n) ] \ \ , \ \ |x| \in [n,n+(2n)^{-1}]
\\
0 \ \ , \ \ |x| > n+(2n)^{-1} \ ,
\end{array}
\right.
\]
si mostri che: \textbf{(1)} $f_n \in C_c(\bR)$, $\forall n \in \bN$; \textbf{(2)} $\{ f_n \}$ converge uniformemente ad
$f \in C_0(\bR) - C_c(\bR)$, $f(x) := e^{-x^2}$, $x \in \bR$.
}

\newpage
\section{Equazioni differenziali ordinarie.}
\label{sec_eq_diff}

La ricerca di soluzioni di equazioni differenziali \e una delle questioni caratterizzanti dell'analisi. A livello storico, impulso determinante per lo studio delle equazioni differenziali \e stata la meccanica newtoniana, nell'ambito della quale hanno il ruolo di tradurre in termini matematici princ\'ipi fondamentali quali quelli della dinamica.

Un'equazione differenziale si pu\'o pensare come un insieme di relazioni (algebriche, nei casi pi\'u elementari) che legano la funzione incognita alle sue derivate. La funzione incognita si interpreta tipicamente come una grandezza che evolve nel tempo, si pensi ad esempio alla posizione nello spazio di un corpo materiale; dunque la sua conoscenza implica, nella misura in cui l'equazione (ben) descriva un sistema fisico, la possibilit\'a di prevederne il comportamento.

\

Diamo ora alcune definizioni pi\'u precise.
Sia $n \in \bN$ ed $A \subseteq \bR^{n+1}$ aperto; un'equazione differenziale ordinaria si presenta come un'espressione del tipo
\[
f ( u , u' , u'' , \ldots , u^{(n)} ) = 0
\ ,
\]
dove $f : A \to \bR$ \e una funzione (solitamente) continua, ed $u : I \to \bR$, $I \subseteq \bR$ aperto, \e la funzione incognita. Osserviamo che \e sempre possibile ricondursi ad un problema di primo grado ($n = 1$), sostituendo $u$ con la funzione 
\[
( u_0 , u_1 , \ldots , u_{n-1} ) : I \to \bR^n
\ \ , \ \
u_0 := u
\ , \
u_i := u_{i-1}'
\ , \
i = 1 , \ldots , n-1
\ ,
\]
per cui in generale tratteremo problemi del tipo
\begin{equation}
\label{eq_PDO1}
u' = f(t,u)
\ ,
\end{equation}
dove $u : I \to \bR^n$, $f : A \to \bR^n$, con $I \subseteq \bR$, $A \subseteq \bR^{n+1}$ aperti.
Il problema (\ref{eq_PDO1}) pu\'o essere arricchito con ulteriori condizioni che $u$, e/o la sua derivata, devono soddisfare. In questa sezione consideriamo il {\em Problema di Cauchy}:
\begin{equation}
\label{eq_PDO2}
\left\{
\begin{array}{ll}
u' = f(t,u)   \\
u (t_0) = u_0 \ ,
\end{array}
\right.
\end{equation}
dove $f : A \to \bR^n$, $( t_0 , u_0 ) \in A \subseteq \bR^{n+1}$.
Ora, in generale il calcolo esplicito della soluzione di (\ref{eq_PDO2}) \e un compito impossibile; \e allora importante produrre teoremi che ne assicurino l'esistenza e (possibilmente) l'unicit\'a.

\subsection{Il teorema di Cauchy.}

\begin{thm}
\label{thm_Cauchy_PDO}
Sia dato il problema (\ref{eq_PDO2}) con $f \in C(A,\bR^n)$ e due dischi $I := ( t_0-r , t_0+r )$, $J := \Delta ( u_0 , r' )$, tali che: 
(1) $I \times J \subseteq A$; 
(2) $f$ \e Lipschitz in $J$, ovvero esiste $L > 0$ tale che
\[
| f (s,v) - f(s,w) | \leq L | v - w |
\ \ , \ \
s \in I 
\ , \
v,w \in J
\ .
\]
Allora esiste ed \e unica la soluzione del problema (\ref{eq_PDO2}), definita in un opportuno intorno $I_0 :=$ $( t_0-r_0,t_0+r_0 )$.
\end{thm}

\begin{proof}[Dimostrazione]
Si tratta di applicare il teorema delle contrazioni. Come prima cosa, definiamo
\[
M := \sup_{I \times J} | f (s,v) | 
\]
e scegliamo $r_0 > 0$ tale che
\[
r_0 < \min \left\{  r  , \frac{r'}{M} , \frac{1}{L}  \right\}
\ .
\]
Definiamo quindi $I_0 :=$ $( t_0-r_0,t_0+r_0 )$, e
\[
X 
:= 
\left\{   
u \in C(I_0) : \|  u - u_0  \|_\infty < r'
\right\}
\ .
\]
Per costruzione $X$ \e uno spazio metrico completo, una volta equipaggiato della
distanza indotta dalla norma dell'estremo superiore. Il contenuto intuitivo della
definizione precedente \e che stiamo considerando funzioni la cui immagine sia
contenuta nel disco $J$ all'interno del quale $f$ \e Lipschitz.
Introduciamo ora l'{\em operatore di Volterra}
\begin{equation}
\label{eq_tc1}
F : X \to X
\ \ , \ \
Fu(t) := u_0 + \int_{t_0}^t f ( s , u(s) ) \ ds
\ , \
t \in I_0
\ .
\end{equation}
Come prima cosa, verifichiamo che $F$ sia ben definito, ovvero che si abbia effettivamente
$Fu \in X$. Chiaramente $Fu$ \e continua in $I_0$; inoltre
\[
\|  Fu - u_0 \|_\infty
\leq
\int_{t_0}^t \sup_{I_0} | f ( s , u(s)) | \ ds
\leq
M r_0 < r'
\ ,
\]
dunque $Fu \in X$. Verifichiamo ora che $F$ \e una contrazione:
\[
\|  Fu - Fz  \|
\leq
\int_{t_0}^t | f (s,u(s)) - f(s,z(s)) | \ ds
\leq
L \| u - z \|_\infty r_0
< 
\| u - z \|_\infty
\]
(infatti, $Lr_0 < 1$). Dunque, per Teo.\ref{thm_contr} esiste ed \e unica $\ovl u$ tale che $F \ovl u = \ovl u$, ovvero
\[
\ovl u (t) = u_0 + \int_{t_0}^t f ( s , \ovl u(s) ) \ ds
\ , \
t \in I_0
\ .
\]
L'espressione precedente ci dice che $\ovl u \in C^1(I_0)$, essendo essa una primitiva di $f(\cdot,\ovl u(\cdot)) \in C(I_0)$. Derivando membro a membro concludiamo che $\ovl u$ \e la soluzione del problema (\ref{eq_PDO2}).
\end{proof}

La condizione di {\em locale lipschitzianit\'a} per $f(t,\cdot)$ (ovvero il punto (2) dell'enunciato) \e soddisfatta se $f$ \e convessa come funzione della variabile $u$ (infatti ogni funzione convessa \e localmente Lipschitz, vedi \S \ref{sec_jensen}). Un'altra condizione sufficiente per la condizione di Lipschitz (nel caso $n=1$) \e che la derivata parziale di $f$ rispetto ad $u$ sia continua in un intorno di $u_0$, come si verifica facilmente applicando il teorema del valor medio ad $f (t,\cdot)$.

Infine, osserviamo che esplicitando la successione introdotta nella dimostrazione del teorema delle contrazioni in (\ref{eq_tc1}), si ottiene la {\em successione di Peano-Picard}
\begin{equation}
\label{eq_PP1}
y_0 := u_0
\ \ , \ \
y_{n+1} (t) := u_0 + \int_{t_0}^t f ( s , y_n(s) ) \ ds
\ \ , \ \
n \in \bN
\ ,
\end{equation}
la quale fornisce una successione convergente alla soluzione del problema di Cauchy. Tuttavia, ad eccezione di casi particolarmente favorevoli, il calcolo degli integrali nell'espressione precedente \e in genere difficoltoso, ed in tal caso \e conveniente procedere per approssimazione. Ad esempio, quando $f$ \e di classe $C^\infty$ si pu\'o procedere attraverso sviluppi in serie di Taylor, in modo da ridursi ad integrare dei polinomi; per esempi in tal senso rimandiamo a \cite{Sal}.

\subsection{Prolungamento delle soluzioni.}

Consideriamo il problema di Cauchy (\ref{eq_PDO2}) definito su un intervallo $I :=$ $(t_0-\delta,t_0+\delta)$. Una soluzione {\em locale} di (\ref{eq_PDO2}) \e il dato di una coppia $(J,u)$, dove $J$ \e un intervallo aperto contenuto in $I$ ed $u \in C^1(J,\bR^n)$ \e una soluzione di (\ref{eq_PDO2}); per brevit\'a, talvolta nel seguito scriveremo semplicemente $u$ invece di $(J,u)$. Denotiamo con $\mC$ l'insieme di tali soluzioni locali. 
\begin{defn}
Siano $(J_1,u_1)$, $(J_2,u_2) \in$ $\mC$. Diciamo che $(J_2,u_2)$ \e \textbf{un prolungamento} di $(J_1,u_1)$ se $J_1 \subset J_2$ e $u_2 |_{J_1} = u_1$, ed in tal caso scriveremo $(J_1,u_1) \prec (J_2,u_2)$ o (pi\'u brevemente)$u_1 \prec u_2$. 
\end{defn}
In base alla definizione precedente, $(\mC,\prec)$ \e un insieme parzialmente ordinato. Una soluzione locale $u$ di (\ref{eq_PDO2}) si dice {\em massimale} se \e massimale rispetto alla relazione d'ordine $\prec$ (ovvero $v \in \mC$, $u \prec v$ $\Rightarrow$ $u=v$).

\begin{thm}
Sia dato il problema di Cauchy (\ref{eq_PDO2}) con $f$ localmente Lipschitz nella seconda variabile. Allora ogni soluzione locale di (\ref{eq_PDO2}) ammette un prolungamento massimale.
\end{thm}

\begin{proof}[Dimostrazione]
Sia $(J,u)$ una soluzione locale di (\ref{eq_PDO2}), $\mC_{J,u} := \{ (J_1,u_1) \in \mC : u \prec u_1 \}$ e $(J_1,u_1)$, $(J_2,u_2) \in \mC_{J,u}$. Allora $u_1(t_0) = u_2(t_0) = u_0$ e per unicit\'a della soluzione troviamo $u_1 |_J = u_2 |_J = u$. Ora, avendosi $\emptyset \neq J \subset J_1 \cap J_2$, abbiamo che $\wt J := J_1 \cup J_2$ \e un intervallo, e definendo $\wt u \in C^1(\wt J,\bR^n)$, $\wt u|_{J_1} = u_1$, $\wt u|_{J_2} = u_2$, troviamo facilmente che $(\wt J,\wt u) \in \mC_{J,u}$. Dunque $\mC_{J,u}$ \e un insieme diretto, e quindi ammette un elemento massimale.  
\end{proof}

Sia $(J,u)$, $J := (a,b)$, una soluzione locale e $(J_1,u_1)$ un prolungamento con $\ovl J \subset J_1$. Poich\'e $u_1 \in C^1(J_1,\bR^n)$ troviamo che deve essere necessariamente
\begin{equation}
\label{eq_PR}
L := \lim_{t \to b^-} u(t) \neq \pm \infty \ ,
\end{equation}
in quanto $L = u_1(b)$. In effetti, (\ref{eq_PR}) \e anche condizione sufficiente affinch\'e esista un prolungamento di $(J,u)$: infatti, considerando il problema di Cauchy
\[
\left\{
\begin{array}{ll}
u' = f(t,u) \\
u (b) = L
\end{array}
\right.
\]
ed una sua soluzione $(J_2,u_2)$, possiamo facilmente costruire il prolungamento $J_1 := J \cup J_2$, $u_1 \in C^1(J_1,\bR^n)$, $u_1 |_J = u$, $u_1 |_{J_2} = u_2$. Conseguenza di quanto appena affermato \e il seguente teorema.
\begin{thm}[Fuga dai compatti]
\label{thm_FUGA}
Sia $(I,u)$, $I = (\ovl a , \ovl b)$, una soluzione massimale di (\ref{eq_PDO2}). Per ogni compatto $K \subset A$ esiste $\delta > 0$ tale che per ogni $t \in I - ( \ovl a + \delta , \ovl b - \delta )$ risulta $(t,u(t)) \in A-K$.  
\end{thm}

\begin{proof}[Dimostrazione]
Se $(t_0,u_0)$ appartiene alla frontiera di $K$ possiamo risolvere il problema di Cauchy con dato iniziale $u(t_0) = u_0$, la cui soluzione $\wt u$ fornisce dei punti $(t,\wt u(t)) \in A-K$. Del resto, $\wt u$ deve essere restrizione della soluzione massimale $u$, per cui $(t,u(t)) \in A-K$.
\end{proof}

Dalle considerazioni precedenti segue che \e interessante stabilire quando una soluzione locale soddisfa (\ref{eq_PR})
{\footnote{Quando, al contrario, $\lim_{t \to b^-} |u(t)| = \infty$ allora diciamo che $u$ {\em esplode}. In tal caso troviamo necessariamente $\lim_{t \to b^-} |u'(t)| = \infty$.}}, 
visto che in tal caso essa \e prolungabile. Il seguente Lemma fornisce una condizione sufficiente per evitare {\em esplosioni} di soluzioni locali nel caso $n=1$. 
\begin{lem}
\label{lem_PS1}
Sia $u \in C^1(a,b)$ con $\eps \geq 0$, $L > 0$ tali che 
\begin{equation}
\label{eq_lPS1}
|u'(t)| \leq \eps + L |u(t)| \ \ , \ \ t \in (a,b) \ .
\end{equation}
Allora per ogni $t,t_0 \in (a,b)$ risulta
\begin{equation}
\label{eq_PS0}
|u(t)| \leq \left(  \frac{\eps}{L} + |u(t_0)|  \right) e^{L|t-t_0|} \ .
\end{equation}
\end{lem}

\begin{proof}[Dimostrazione]
Usiamo il seguente trucco: preso $\lambda > 0$, definiamo
\[
z(t) := \sqrt{\lambda^2 + u(t)^2}
\ \Rightarrow \
z'(t) = \frac{u'(t)u(t)}{\sqrt{\lambda^2+u(t)^2}}
\ .
\]
Usando (\ref{eq_lPS1}) e $|u(t)| \leq z(t)$ troviamo
\[
z'(t) \leq
\frac{|u(t)|}{\sqrt{\lambda^2+u(t)^2}} \ |u'(t)| \leq
|u'(t)| \leq 
\eps + L |u(t)| \leq
\eps + L z(t)
\ ,
\]
per cui
\[
\frac{Lz'(t)}{\eps+Lz(t)} \leq L
\ \Rightarrow \
\ln \frac{\eps+Lz(t)}{\eps+Lz(t_0)} \leq L |t-t_0|
\ \Rightarrow \
\eps + L z(t) \leq (\eps+Lz(t_0)) e^{L(t-t_0)} \ ,
\]
e ancora (avendosi chiaramente $Lz(t) \leq \eps+Lz(t)$)
\[
z(t) \leq \frac{1}{L} (\eps+Lz(t_0)) e^{L(t-t_0)}
\ \stackrel{\lambda \to 0}{\Rightarrow} \ 
|u(t)| \leq \frac{1}{L} (\eps+L|u(t_0)|) e^{L(t-t_0)}
\ .
\]
\end{proof}

\begin{thm}
\label{teo_PS2}
Sia $I := ( t_0-r , t_0+r )$ ed $f \in C(I \times \bR)$ localmente Lipschitz nella seconda variabile. Supponiamo che per ogni compatto $K \subset I \times \bR$ esistano $\eps_K , L_K \geq 0$ tali che
\begin{equation}
\label{eq_PS1}
|f(t,u)| \ \leq \ \eps_K + L_K |u| \ \ , \ \ (t,u) \in I \times \bR \ .
\end{equation}
Allora il problema di Cauchy (\ref{eq_PDO2}) ammette una soluzione massimale $u \in C^1(I)$.
\end{thm}

\begin{proof}[Dimostrazione]
Applicando il lemma precedente ad $u'(t) = f(t,u(t))$ otteniamo la stima (\ref{eq_PS0}), la quale implica che non si hanno esplosioni in nessun punto di $I$.
\end{proof}

\begin{ex}
\label{ex_AcqP}
{\it
Studiamo il problema di Cauchy
\begin{equation}
\label{eq_AcqP}
\left\{
\begin{array}{ll}
u' = 2tu^2
\\
u(t_0) = u_0
\end{array}
\right.
\end{equation}
al variare di $(t_0,u_0) \in \bR^2$. Innanzitutto, osserviamo che $f (t,x) := 2tx^2$ \e definita su tutto $\bR^2$ ed \e localmente Lipschitz, per cui possiamo applicare il teorema di Cauchy per ogni condizione iniziale $(t_0,u_0) \in \bR^2$. Tuttavia le soluzioni $u$ variano sensibilmente al variare di $(t_0,u_0)$ in $\bR^2$. Innanzitutto escludiamo il caso $u_0 = 0$, poich\'e fornisce la soluzione banale $u = 0$. Al che, risolvendo per separazione di variabili troviamo
\[
t_0 = 0 \ , \ u_0 > 0
\ \Rightarrow \
u(t) = \frac{ u_0 }{ 1 - u_0 t^2 }
\ \Rightarrow \
\]
\[
\lim_{ |t| \to u_0^{-1/2} } |u(t)| = + \infty
\ \Rightarrow \
I_{max}(0,u_0) = ( - u_0^{-1/2} , u_0^{-1/2} )
\ ,
\]
dove, in generale, $I_{max}(t_0,u_0)$ denota il dominio della soluzione massimale del problema (\ref{eq_AcqP}) con condizione iniziale $(t_0,u_0)$. Se $u_0 < 0$ allora la situazione cambia drasticamente:
\[
t_0 = 0 \ , \ u_0 < 0
\ \Rightarrow \
u(t) = \frac{ u_0 }{ 1 + |u_0| t^2 }
\ \Rightarrow \
I_{max}(0,u_0) = \bR
\ .
\]
Passiamo ora ad analizzare la situazione per $t_0 \neq 0$:
\[
u_0 > 0
\ \Rightarrow \
u(t) = \frac{ u_0 }{ 1 - u_0 (t^2 - t_0^2) }
\ \Rightarrow \
\]
\[
\lim_{ |t| \to \sqrt{ t_0^2+u_0^{-1} } } |u(t)| = + \infty
\ \Rightarrow \
I_{max}(t_0,u_0) = 
\left( - \sqrt{ t_0^2+u_0^{-1} } \ , \ + \sqrt{ t_0^2+u_0^{-1} }  \right)
\ .
\]
D'altro canto
\[
u_0 < 0
\ \Rightarrow \
u(t) = \frac{ u_0 }{ 1 + |u_0| (t^2 - t_0^2) }
\ ;
\]
ora, abbiamo due casi:
\[
1 - |u_0| t_0^2 > 0
\ \Leftrightarrow \
t_0^2-|u_0|^{-1} < 0 
\ \Rightarrow \
I_{max}(t_0,u_0) = \bR
\ ,
\]
altrimenti
\[
t_0^2-|u_0|^{-1} \geq 0 
\ \Rightarrow \
\lim_{ |t| \to \sqrt{t_0^2-|u_0|^{-1}} } |u(t)| = + \infty
\ \Rightarrow \
\]
\[
I_{max}(t_0,u_0) =
\left\{
\begin{array}{ll}
\left( \sqrt{ t_0^2 - |u_0|^{-1} } \ , \  + \infty \right) \ \ , \ \ t_0 > 0
\\
\left( - \infty \ , \ - \sqrt{ t_0^2 - |u_0|^{-1} }  \right) \ \ , \ \ t_0 < 0 
\ .
\end{array}
\right.
\]
}
\end{ex}

\subsection{Dipendenza continua dai dati iniziali.}

Nelle applicazioni fisiche il ruolo della condizione iniziale di un problema di Cauchy \`e quello del valore assunto da una determinata grandezza fisica, diciamo $u_0$, misurata al tempo $t = t_0$. Poich\'e la misura di una grandezza fisica comporta inevitabilmente un errore di rilevazione (per quanto piccolo), \e importante caratterizzare quei problemi di Cauchy tali che a fronte di piccole discrepanze $u_0 \neq v_0$ producano soluzioni $(I,u)$, $(J,v)$ che differiscano "poco" in $I \cap J$. 
In termini rigorosi, consideriamo il problema differenziale
\begin{equation}
\label{eq_DC01}
u' = f(t,u)
\ \ , \ \
A \subset \bR^{n+1}
\ \ , \ \
f \in C(A,\bR^n)
\end{equation}
e, fissato $t_0 \in \bR$, consideriamo un disco $\Delta \subset \bR^n$ tale che $(t_0,x) \in A$ per ogni $x \in \Delta$. Assumiamo che esiste un intervallo $I \subseteq \bR$, $I \ni t_0$, tale che ogni problema di Cauchy 
\[
u' = f(t,u)
\ \ , \ \
u(t_0) = x
\ \ , \ \
x \in \Delta
\ ,
\]
abbia soluzione unica $(I_x,u_x)$ con $I_x \supseteq I$. Diciamo che si ha {\em dipendenza continua dai dati iniziali in $I$} se l'applicazione
\[
\Delta \to C^1(I,\bR^n)
\ \ , \ \
x \mapsto u_x|_I
\ ,
\]
\e continua, dove $C^1(I,\bR^n)$ \e equipaggiato con la topologia della convergenza uniforme; questa condizione ci assicura che per ogni $\eps > 0$ esiste un $\delta$ tale che
\[
x,y \in \Delta \ , \ |x-y| < \delta
\ \Rightarrow \ 
\sup_{t \in I} \, | u_x(t) - u_y(t) | < \eps
\ .
\]
Nelle applicazioni, per parlare di dipendenza continua occorre anche richiedere una propriet\'a non prettamente matematica, ovvero che $I$, il quale si interpreta come l'intervallo temporale nel quale la soluzione di (\ref{eq_DC01}) \e attendibile nel descrivere il comportamento del sistema fisico in oggetto, sia fisicamente significativo. Chiaramente, il termine {\em fisicamente significativo} \e ambiguo: un intervallo $I$ di pochi secondi pu\'o essere ben accettabile in problema d'urto, e totalmente inadeguato, invece, in un modello astronomico.

I risultati classici che permettono di dedurre la propriet\'a di dipendenza continua sono noti come {\em il Lemma di Gronwall}. Fenomeni fisici che conducono a problemi di Cauchy che non hanno dipendenza continua dai dati iniziali sono di solito associati ai cosiddetti {\em sistemi caotici}, e formano a tutt'oggi un'area di ricerca molto importante. Come esempio di sistema caotico portiamo il celebre {\em modello di Lorenz} correlato alle previsioni metereologiche (\cite[\S 9.2]{Tes},\cite{Hil}): presi $\sigma , \rho , \beta > 0$, esso si ottiene come il problema differenziale (\ref{eq_DC01}) associato alla funzione
\begin{equation}
\label{eq.Lorenz}
f : \bR \times \bR^3 \to \bR^3
\ \ , \ \
f(t,u) :=
( -\sigma (u_1-u_2) \ , \  \rho u_1 - u_2 - u_1u_3 \ , \ u_1u_2-\beta u_3 )
\ .
\end{equation}

\begin{lem}[La diseguaglianza di Gronwall]
Sia $0 \in I \subseteq \bR$ un intervallo ed $u \in C(I,\bR)$ tale che esistano $\beta , \alpha \in C(I,\bR)$, $\beta \geq 0$, con
\begin{equation}
\label{eq_DC02}
u(t)
\ \leq \
\alpha (t) + \int_0^t \beta (s) u(s) \ ds
\ \ , \ \
t \in [0,T] \subset I
\ .
\end{equation}
Allora 
\begin{equation}
\label{eq_DC03}
u(t)
\ \leq \
\alpha (t) + 
\int_0^t \alpha(s) \beta(s) \exp \left( \int_s^t \beta (\tau)  d \tau \right) ds
\ \ , \ \
t \in [0,T]
\ .
\end{equation}
\end{lem}

\begin{proof}[Dimostrazione]
Poniamo $\phi(t) := \exp \left( - \int_0^t \beta (s) ds \right)$, $t \in I$, cosicch\'e
$\phi(t)^{-1} \phi(s) = \exp \left(\int_s^t \beta (s) ds \right)$.
Usando (\ref{eq_DC02}) si trova
\[
\frac{d}{dt} \left\{  \phi(t) \int_0^t \beta (s) u(s) \ ds \right\}
\ = \
\beta (t) \phi(t) \left( -\int_0^t \beta(s)u(s) \ ds + u(t) \right)
\ \leq \
\alpha(t) \beta(t) \phi(t)
\ .
\]
Integrando questa diseguaglianza rispetto a $t$ e dividendo per $\phi(t)$ si trova
\[
\int_0^t \beta (s) u(s) \ ds 
\ \leq \
\phi (t)^{-1} \int_0^t \alpha(s) \beta(s) \phi(s) \ ds
\ ,
\]
per cui sommando $\alpha(t)$ ed applicando (\ref{eq_DC02}) troviamo (\ref{eq_DC03}).
\end{proof}

\begin{cor}
Siano $\beta > 0 , \alpha , \gamma \in \bR$. Se
\begin{equation}
\label{eq_DC04}
u(t)
\ \leq \
\alpha + \int_0^t ( \beta u(s) + \gamma ) \ ds 
\ \ , \ \
t \in [0,T]
\ ,
\end{equation}
allora 
$u(t) \leq \alpha e^{\beta t} + \gamma \beta^{-1} ( e^{\beta t} - 1 )$,
$t \in [0,T]$.
\end{cor}

\begin{proof}[Dimostrazione]
Definendo $\wt u := u + \gamma \beta^{-1}$ troviamo che (\ref{eq_DC04}) si scrive
\[
\wt u(t)
\ \leq \
\alpha + \gamma \beta^{-1} + \int_0^t \beta \wt u (s) \ ds
\ \ , \ \
t \in [0,T]
\ ,
\]
dunque possiamo applicare il Lemma precedente. Calcolando esplicitamente i (semplici) integrali coinvolti otteniamo la tesi.
\end{proof}

\begin{thm}[Dipendenza continua]
Siano $f,g \in C(A,\bR^n)$ con $f$ avente costante di Lipschitz $L$. Se $(I_x,u_x)$, $(J_y,v_y)$ sono soluzioni rispettivamente dei problemi di Cauchy
\begin{equation}
\label{eq_DC05a}
\left\{
\begin{array}{ll}
u' = f(t,u)
\\
u(t_0) = x
\end{array}
\right.
\ \ \ , \ \ \
\left\{
\begin{array}{ll}
v' = g(t,v)
\\
v(t_0) = y \ ,
\end{array}
\right.
\end{equation}
allora 
\begin{equation}
\label{eq_DC05b}
| u_x(t) - v_y(t) |
\ \leq \
| x-y | e^{L(t-t_0)} + ML^{-1} ( e^{L(t-t_0)} - 1 )
\ \ , \ \
t \in I_x \cap J_y
\ ,
\end{equation}
dove $M := \| f-g \|_\infty$.
\end{thm}

\begin{proof}[Dimostrazione]
Possiamo supporre senza ledere la generalit\'a che $t_0 = 0$. Integrando e sottraendo le (\ref{eq_DC05a}) abbiamo
\[
|u_x(t)-v_y(t)|
\ \leq \
|x-y| + \int_0^t | f(s,u_x(s)) - g(s,v_y(s)) | \ ds
\ ,
\]
e stimando le funzioni integrande troviamo (sommando e sottraendo $f(s,v_y(s))$)
\[
| f(s,u_x(s)) - g(s,v_y(s)) | 
\ \leq \
L | u_x(s) - v_y(s) |+ M
\ ,
\]
da cui
\[
|u_x(t)-v_y(t)|
\ \leq \
|x-y| + \int_0^t ( L | u_x(s) - v_y(s) |+ M ) \ ds
\ .
\]
La tesi segue dunque applicando il Corollario precedente.
\end{proof}

Poniamo ora $f=g$ e consideriamo un disco $\Delta \subset \bR^n$ tale che $\{ t_0 \} \times \Delta \subset A$. Supposto che $\cap_{x \in \Delta}I_x$ contenga un intervallo $I$ non banale e di lunghezza $|I|$ finita, in conseguenza del teorema precedente troviamo, per ogni $x,y \in \Delta$, $t \in I$,
\[
| u_x(t) - u_y(t) |
\ \leq \
|x-y| e^{L|t-t_0|}
\ \leq \
|x-y| e^{L|I|}
\ ,
\]
ovvero la dipendenza continua nel senso della definizione data all'inizio della sezione. 
Osserviamo che l'ipotesi di lipschitzianit\'a per $f$ pu\'o essere rilassata ad una lipschitzianit\'a locale se ci si restringe a domini compatti $K \subset A$ (si veda \cite{Giu2}); tuttavia, occorre fare attenzione al fatto che se si ottengono costanti di Lipschitz $L_K$ che tendono ad infinito al crescere di $K$, la stima precedente diventa inutile e non si pu\'o parlare di dipendenza continua in $I$.

\subsection{Il teorema di Peano.}

In effetti l'esistenza di una soluzione locale del problema di Cauchy \e dimostrabile con la sola ipotesi di continuit\'a per $f$, attraverso il {\em teorema di Peano}. Il prezzo da pagare per la maggiore generalit\'a di questo risultato \e la perdita dell'unicit\'a della soluzione. Dal punto di vista del metodo della dimostrazione, segnaliamo che questa si basa sul teorema di Ascoli-Arzel\'a (Teo.\ref{thm_AA}).

\begin{thm}[Peano]
\label{thm_P}
Sia dato il problema (\ref{eq_PDO2}) con $A := [ t_0 - r , t_0 + r ] \times \bR^n$ per qualche $r > 0$, $n \in \bN$. Se $f \in C(A,\bR^n)$ \e continua e limitata allora esiste una soluzione del problema (\ref{eq_PDO2}) definita sull'intervallo $[ t_0 , t_0 + r ]$.
\end{thm}

\begin{proof}[Dimostrazione]
L'idea \e quella di applicare il teorema di Ascoli-Arzel\'a ad un'opportuna successione di tipo Peano-Picard. A tale scopo, per ogni $m \in \bN$ consideriamo la partizione 
$P_m := \left\{ t_k^{(m)} := t_0 + kr/m  \right\}$ 
di $[t_0 , t_0 + r]$, e definiamo
\[
y_m (t) := u_k^{(m)} + f ( t_k^{(m)} , u_k^{(m)} ) (t-t_k^{(m)}) 
\ \ , \ \ 
t \in [ t_k^{(m)} , t_{k+1}^{(m)} )
\ , \
k = 0,1, \ldots , m
\ ,
\]
dove i coefficienti $u_k^{(m)}$, $k \in \bN$, sono definiti per iterazione,
\[
u_{k+1}^{(m)} := u_k^{(m)} + f (t_k^{(m)} , u_k^{(m)} ) \frac{r}{m}
\ \ , \ \
k = 1,2, \ldots , m-1 
\ .
\]
%
Le funzioni $y_m$ cos\'i costruite sono lineari a tratti e derivabili in $( t_0 , t_0 + r ) - P_m$, con derivata $y'_m (t) = f ( t_k^{(m)} , u_k^{(m)} )$,
$t \in ( t_k^{(m)} , t_{k+1}^{(m)} )$. Cosicch\'e introducendo le funzioni costanti a tratti
\[
\phi_m : [ t_0 , t_0 + r ] \to \bR^n
\ , \
\phi_m(t) := f ( t_k^{(m)} , u_k^{(m)} )
\ , \ 
t \in [ t_k^{(m)} , t_{k+1}^{(m)} )
\ ,
\]
troviamo
\begin{equation}
\label{eq_P01}
\| \phi_m \|_\infty \leq \| f \|_\infty
\ \ , \ \
\phi_m = y'_m  \ {\mathrm{in}} \  [ t_0 , t_0 + r ] - P_m
\ ,
\end{equation}
e, applicando il teorema fondamentale del calcolo,
\begin{equation}
\label{eq_P02}
y_m(t) = 
u_0 + \int_{t_0}^t \phi_m (s) \ ds 
\ \ , \ \
t \in [t_0,t_0+r]
\ .
\end{equation}
Ora, \e un fatto generale che {\em una funzione regolare a tratti $F : [t_0,t_0+r] \to \bR^n$ con derivata limitata ammette costante di Lipschitz} $\| F' \|_\infty \sqrt{n}$. La disuguaglianza in (\ref{eq_P01}) assicura che ci\'o \e vero per le nostre $y_m$, che si trovano cos\'i ad avere la stessa costante di Lipschitz $\| f \|_\infty \sqrt{n}$. 
In tal modo, abbiamo verificato che $\left\{ y_m \right\}$ \e una successione equicontinua; poich\'e questa \e anche equilimitata (infatti (\ref{eq_P02}) implica $\| y_m \| \leq \| u_0 \| + r \| f \|_\infty$), concludiamo per Ascoli-Arzel\'a che esiste una sottosuccessione, {\em che denotiamo per brevit\'a sempre con $\{ y_m \}$}, convergente ad un limite $y \in C([t_0,t_0+r],\bR^n)$.

Vogliamo ora mostrare che $y$ \e effettivamente una soluzione del nostro problema di Cauchy.
A tale scopo osserviamo che, preso $t \in [t_0,t_0+r]$, una volta scelta la partizione $P_m$, $m \in \bN$, esso apparterr\'a ad uno, ed uno solo, intervallo $[ t_k^{(m)} , t_{k+1}^{(m)} )$, in maniera tale che
\begin{equation}
\label{eq_peano0}
|t-t_k^{(m)}| \leq rm^{-1}
\ ;
\end{equation}
definiamo allora $s_{t,m} := t_k^{(m)}$ ed $u_{t,m} := u_k^{(m)} = y_m(s_{t,m})$. Usando (\ref{eq_peano0}) concludiamo che $\lim_m s_{t,m} = t$, e che la convergenza di tali successioni \e uniforme al variare di $t$ in $[t_0,t_0+r]$ (infatti il termine $rm^{-1}$ in (\ref{eq_peano0}) non dipende da $t$!).
Inoltre, abbiamo la stima
\[
|u_{t,m}-y(t)| 
\ = \
| y_m(s_{t,m})-y(t) | 
\ \leq \ 
| y_m(s_{t,m})-y_m(t) | + | y_m(t)-y(t) | 
\ .
\]
Essendo $\{ y_m \}$ equicontinua abbiamo che, scelto $\eps > 0$, esiste $\delta > 0$ tale che $| y_m(t') - y_m(t'') | < \eps$ per $|t'-t''| < \delta$, uniformemente in $m \in \bN$; d'altro canto esiste certamente un $m_0 \in \bN$ tale che $|s_{t,m}-t | \leq rm^{-1} < \delta$ per ogni $m > m_0$, per cui
$\lim_m | y_m(s_{t,m})-y_m(t) | = 0$ uniformemente in $t$. D'altra parte $\lim_m \| y_m - y \|_\infty = 0$, cosicch\'e
\[
|u_{t,m}-y(t)| 
\ \stackrel{m}{\lra} \
0
\ ,
\]
ed \e importante osservare che la convergenza di $\{ u_{t,m} \}_m$ \e uniforme al variare di $t$ in $[t_0,t_0+r]$. Dunque abbiamo
\begin{equation}
\label{eq_peano1}
( s_{t,m} , u_{t,m} ) 
\ \stackrel{m}{\lra} \
( t , y(t) )
\ ,
\end{equation}
uniformemente in $t$ grazie alle considerazioni precedenti.
Ora, poich\'e $| u_{t,m} | \leq \| y_m \|_\infty \leq \| f \|_\infty$ ogni successione $( s_{t,m} , u_{t,m} )$ \e contenuta nel compatto 
$K := [t_0,t_0+r] \times \ovl{\Delta(0,\| f \|_\infty)}$, 
ed $f$ \e uniformemente continua in $K$ (per Heine-Cantor). Per cui, preso $t \in [t_0,t_0+r]$ troviamo, usando (\ref{eq_peano1}),
\[
| \phi_m(t)  - f(t,y(t)) | 
\ = \
| f(s_{t,m},u_{t,m}) - f(t,y(t)) | 
\ \stackrel{m}{\lra} \
0
\ ,
\]
uniformemente al variare di $t$ in $[t_0 , t_0 + r]$; dunque possiamo passare al limite sotto il segno di integrale in (\ref{eq_P02}) e concludere
\[
y(t) 
\ =  \
u_0 + \int_{t_0}^t f(s,y(s)) \ ds 
\ \ , \ \
t \in [t_0,t_0+r]
\ .
\]
Infine, lo stesso argomento usato nella dimostrazione del Teorema di Cauchy permette di concludere che $y \in C^1([t_0,t_0+r],\bR^n)$, e che essa \e soluzione cercata.
\end{proof}

La dimostrazione precedente in teoria potrebbe essere facilmente tradotta in un algoritmo, visto che la soluzione $y$ viene trovata per mezzo di una successione costruita per iterazione; si ha per\'o un inconveniente, derivante dal fatto che non abbiamo una costruzione esplicita della sottosuccessione convergente $\{ y_m \}$, la cui esistenza viene dedotta con un argomento di compattezza (Ascoli-Arzel\'a). Da qui l'utilit\'a di algoritmi, come quello di Runge-Kutta, che permettono di {\em costruire} (sempre per iterazione) successioni convergenti alla soluzione cercata.

La condizione di limitatezza per $f$ pu\'o essere rimossa, se ci restringiamo a cercare soluzioni in piccolo:
\begin{cor}
Sia $A \subset \bR^{n+1}$ aperto, $(t_0,u_0) \in A$, $f \in C(A,\bR^n)$. Allora il problema di Cauchy (\ref{eq_PDO2}) ammette una soluzione in piccolo $u$ definita su $[t_0,t_0+\tau]$, $\tau \leq r$.
\end{cor}

\begin{proof}[Dimostrazione]
E' sufficiente considerare una funzione continua e limitata $\tilde f$ che coincida con $f$ in un rettangolo chiuso della forma $[t_0 , t_0 + \tau] \times \ovl{\Delta (u_0,R)} \subseteq$ $A$, ed applicare il teorema precedente al problema di Cauchy associato.
\end{proof}

\begin{ex}{\it 
Il problema
\[
\left\{
\begin{array}{ll}
u' = 3 u^{2/3}
\\
u(0) = 0
\end{array}
\right.
\]
ammette le soluzioni $u_0 \equiv 0$ e $u (t) =$ $t^3$. Invitiamo il lettore a verificare che $f (s,v) = 3v^{2/3}$ \e continua ma non Lipschitz in un intorno di $0$, cosicch\'e non \e possibile applicare il teorema di Cauchy. Un'altra propriet\'a di facile verifica (che lasciamo per esercizio) \e che ogni soluzione del precedente problema \e $\geq u_0$, la quale \e quindi il minimo dell'insieme delle soluzioni.
}\end{ex}

\noindent \textbf{Cenni sul pennello di Peano.} Consideriamo il nostro problema di Cauchy (\ref{eq_PDO2}) con $f : A \to \bR^n$ soddisfacente le condizioni di Teo.\ref{thm_P} (ovvero $f$ continua e limitata, $A :=$ $[ t_0 - r , t_0 + r ] \times \bR^n \subset$ $\bR^{n+1}$ per qualche $r > 0$). Se $u$ \e una soluzione di questo problema, allora abbiamo una costante di Lipschitz $\| f \|_\infty \sqrt{n}$ indotta dalla diseguaglianza
\begin{equation}
\label{eq_u_L}
| u'(t) | \leq \|  f  \|_\infty 
\ \ , \ \ 
t \in [t_0-r,t_0+r]
\ ;
\end{equation}
dalla condizione di Lipschitz segue una condizione di limitatezza
\begin{equation}
\label{eq_u_lim}
| u(t) | - | u_0 |
\leq 
\|  f  \|_\infty \sqrt{n} r
\ .
\end{equation}
Osserviamo che n\'e (\ref{eq_u_L}) n\'e (\ref{eq_u_lim}) dipendono dalla particolare soluzione $u$. Ne consegue che {\em l'insieme delle soluzioni del problema (\ref{eq_PDO2}) \e equicontinuo ed equilimitato in $[t_0,t_0+r]$}.

\begin{lem} 
Sia $I \subset \bR$ un intervallo chiuso e limitato, ed $\mF \subset$ $C(I,\bR)$ un insieme equicontinuo, equilimitato, e totalmente ordinato rispetto all'ordinamento standard di $C(I,R)$. Posto $\ovl \phi : I \to \bR$, $\ovl \phi (t) :=$ $\sup_{\phi \in \mF } \phi (t)$, $t \in I$, risulta che per ogni $\eps > 0$ esiste $\phi_\eps \in \mF$ tale che $\| \ovl \phi - \phi_\eps \|_\infty < \eps$.
\end{lem}

\begin{proof}[Dimostrazione]
Omessa, ma semplice.
\end{proof}

\begin{thm}[Peano]
L'insieme delle soluzioni del problema (\ref{eq_PDO2}) ammette un elemento massimale $\ovl u$ ed un elemento minimale $\unl u$. Per ogni punto $p :=$ $(t,y)$ appartenente alla regione $R$ compresa tra i grafici di $\ovl u$ e $\unl u$ esiste almeno una soluzione $u$ tale che $u(t) = y$.
\end{thm}

\begin{proof}[Sketch della dimostrazione]
Applicando il lemma precedente, ed un passaggio al limite sotto il segno dell'integrale di Volterra, concludiamo che ogni catena contenuta nell'insieme $\mF$ delle soluzioni di (\ref{eq_PDO2}) ammette un elemento massimale appartenente ad $\mF$. Applicando Zorn, concludiamo che esistono $\ovl u$ ed $\unl u$ come nell'enunciato. Infine, l'affermazione inerente $p$ si dimostra risolvendo il problema di Cauchy "all'indietro" con dato iniziale $u_0 = y$.
\end{proof}

Osserviamo che {\em qualora si abbia anche unicit\'a della soluzione} (ad esempio, quando $f$ \e localmente Lipschitz nella variabile $u$), allora per ogni $p = (t,y) \in R$ {\em esiste ed \e unica} la soluzione $u$ tale che $u(t) = y$. Dunque $R$ \e "spazzata" dai grafici delle soluzioni di (\ref{eq_PDO2}), i quali non si intersecano. E' questa l'idea alla base del concetto di {\em foliazione}.

\subsection{Esercizi.}

\noindent \textbf{Esercizio \ref{sec_eq_diff}.1.} {\it Effettuare uno studio qualitativo della soluzione del problema di Cauchy 
\[
u' = \sqrt{|u|} + u^2
\ \ , \ \
u(0) = u_0
\ ,
\]
al variare della condizione iniziale $u_0 \in \bR$.
}

\

\noindent \textbf{Esercizio \ref{sec_eq_diff}.2.} {\it Sia $f : \bR \to \bR$ una funzione continua tale che
\begin{equation}
\label{eq.es.2.c}
x f(x) \geq 0 
\ , \
\forall x \in \bR
\ \ , \ \ 
\int_0^{+ \infty} \frac{1}{1+f(x)} \ dx  \ = \ + \infty
\ ,
\end{equation}
e si consideri il problema di Cauchy
\begin{equation}
\label{eq.es.2}
\left\{
\begin{array}{ll}
u' = f(t+|u|)
\\
u(0) = 0 \ .
\end{array}
\right.
\end{equation}
\textbf{(1)} Si mostri che (\ref{eq.es.2}) ammette almeno una soluzione $u$;
\textbf{(2)} si verifichi che si ha $0 = f(0) = u'(0)$;
\textbf{(3)} si mostri che $u'(t) \geq 0$, $u(t) \geq 0$, $\forall t \geq 0$;
\textbf{(4)} posto $v(t) := t + u(t)$, si verifichi che 
          \[
          v'(t) = 1+f(v(t)) \ \ , \ \ v(0)=0 \ \ , \ \ v(t) \geq 0 \ ,  
          \]
          per ogni $t \geq 0$ appartenente al dominio di $u$;
\textbf{(5)} usando il punto precedente, e (\ref{eq.es.2.c}), si verifichi che  
          $u$ non esplode in nessun $T \in \bR$.

\

\noindent (Suggerimenti: per (1) si usi il teorema di Peano; per (2) e (3) si usi la prima delle condizioni in (\ref{eq.es.2.c}), la quale implica che $f(0)=0$; per (4) si osservi che, grazie ai punti precedenti,
$f(t+|u(t)|) = f(t+u(t))$,
cosicch\'e 
$v'(t) = 1+f(v(t))$;
per (5) si osservi che, supponendo per assurdo che 
\[
\infty = \lim_{t \to T^-} u(t) = \lim_{t \to T^-} v(t)
\ ,
\]
integrando per sostituzione si trova la contraddizione
\[
T
\ = \
\int_0^T \frac{v'(t)}{1+f(v(t))} \ dt
\ = \
\int_0^\infty \frac{1}{1+f(v)} \ dv
\ \stackrel{(\ref{eq.es.2.c})}{=} \
\infty
\ .
\]
}

\

\noindent \textbf{Esercizio \ref{sec_eq_diff}.3.} {\it Discutere esistenza ed unicit\'a della soluzione dell'\textbf{equazione di Liouville}
\[
u'' + e^{\lambda u} = 0 \ \ , \ \  \lambda \in \bR \ ,
\]
al variare della condizione iniziale $u(0) = u_0 \in \bR$.
}

\

\noindent \textbf{Esercizio \ref{sec_eq_diff}.4.} {\it Sia $T > 0$ e $\phi \in C^2(\bR)$ una funzione $T$-periodica, strettamente positiva e tale che sia $\phi'$ che $\phi''$ si annullino solo due volte (ciascuna) in $[0,T)$. Si supponga per semplicit\'a che $\min \phi = \phi(0)$ e che $\max \phi = \phi (a)$ con $a \in (0,T)$. \textbf{(a)} Si tracci un grafico approssimativo di $\phi$. \textbf{(b)} Posto 
\[
N := \{ (x,y) \in \bR^2 : \phi(x) = \phi(y)  \} \ \ , \ \ N_0 := N \cap [0,T)^2 \ ,
\]
si provi che $N$ \e $[0,T)^2$-periodico in $\bR^2$. \textbf{(c)} Si provi che $N_0$ \e composto da due archi di classe $C^1$ che si intersecano ad angolo retto. \textbf{(d)} Si consideri il problema di Cauchy 
\begin{equation}
\label{eq_exdiff}
y' = \frac{\phi(x)}{\phi(y)} - 1  
\ \ , \ \
y(0) = y_0 \ .
\end{equation}
Si provi che per ogni $y_0$ esiste un'unica soluzione definita su tutto $\bR$. \textbf{(e)} Si provi che ogni soluzione \e limitata su $\bR$. 


\

\noindent (Suggerimenti: per il punto (c) si usi il Teorema delle funzioni implicite, mentre per i punti (d)-(e) si metta in relazione il grafico di $N_0$ con lo studio qualitativo delle soluzioni di (\ref{eq_exdiff})).
}

\

\noindent \textbf{Esercizio \ref{sec_eq_diff}.5.} {\it Si mostri che il problema di Cauchy 
\[
u' = u^2 - e^{t^2} + 1
\ \ , \ \
u(0) = 0
\ ,
\]
ha soluzione massimale definita su tutto $\bR$. (Suggerimento: si riscontri che non si hanno esplosioni).
}

\

\noindent \textbf{Esercizio \ref{sec_eq_diff}.6.} {\it Si verifichi che la funzione di Lorenz (\ref{eq.Lorenz}) non \e lipschitziana nella variabile $u \in \bR^3$.}

\newpage
\section{Teoria della misura e dell'integrazione.}
\label{sec_MIS}

La teoria dell'integrazione si \e emancipata dalle propriet\'a di regolarit\'a (continuit\'a, continuit\'a a tratti, ...) grazie all'approccio di Lebesgue, attraverso il quale \e possibile integrare funzioni altrimenti intrattabili dal punto di vista dell'integrale di Riemann. Nelle sezioni seguenti tratteggeremo le propriet\'a elementari  degli spazi di misura con particolare cura, ovviamente, del caso della retta reale, e poi passeremo a trattare la teoria dell'integrale.

\subsection{Spazi misurabili.}
\label{sec_MIS1}

\noindent \textbf{$\sigma$-algebre e funzioni semplici.} Sia $X$ un insieme. Un'{\em algebra di Boole} su $X$ \e il dato di un sottoinsieme $\mR$ dell'insieme delle parti $2^X$, chiuso sotto le operazioni di unione e passaggio al complementare (e ci\'o implica che $\mR$ \e chiusa anche rispetto all'intersezione), e tale che $\emptyset , X \in \mR$. 
Una {\em $\sigma$-algebra} \e un'algebra di Boole $\mM$ chiusa rispetto ad unioni {\em numerabili}, il che implica che $\mM$ \e chiusa anche rispetto ad intersezioni numerabili. Dato un insieme $Y$ ed un sottoinsieme $N \subseteq 2^Y$, definiamo {\em la $\sigma$-algebra generata da $N$} come l'intersezione di tutte le $\sigma$-algebre che contengono $N$. Denotiamo tale $\sigma$-algebra con il simbolo $\sigma N$.

\begin{defn}
\label{def_mis_astr}
Siano $X,Y$ insiemi ed $\mM \subseteq 2^X$, $\mN \subseteq 2^Y$ $\sigma$-algebre. Un'applicazione $f : X \to Y$ si dice \textbf{misurabile} se $f^{-1}(B) \in \mM$ per ogni $B \in \mN$.
\end{defn}

Per ogni $A \in \mM$, introduciamo {\em la funzione caratteristica}
\[
\chi_A(x) := 
\left\{
\begin{array}{ll}
1 \ , \ x \in A
\\
0 \ , \ x \notin A \ ;
\end{array}
\right.
\]
osserviamo che
\[
\chi_{A \cap B} = \chi_A \chi_B
\ \ , \ \
\chi_{A \cup B} = \sup \{ \chi_A , \chi_B \}
\ \ , \ \
\chi_{ A^c } = 1 - \chi_A 
\ ,
\]
dove 
\[ A^c \ := \ X-A \ .\] 
Una {\em funzione semplice} \e una combinazione lineare finita a coefficienti in $\bR$ di funzioni caratteristiche:
\[
\varphi := \sum_i \lambda_i \chi_{A_i} 
\ \ , \ \
\lambda_i \in \bR
\ \ , \ \
A_i \in \mM 
\ .
\]
Osserviamo che la rappresentazione di $\varphi$ come combinazione lineare di funzioni caratteristiche non \e unica, ad esempio 
$\lambda \chi_A = \lambda \chi_{U} + \lambda \chi_V$ per ogni $\lambda \in \bR$ ed $A = U \cup V$ con $U \cap V = \emptyset$.
L'insieme delle funzioni semplici \e chiuso rispetto a moltiplicazioni scalari, prodotti e combinazioni lineari, dunque costituisce un'{\em algebra} che denotiamo con $S(X)$.

\begin{ex} \label{def_Bor}
{\it
Sia $X$ uno spazio topologico con topologia $\tau X \subseteq 2^X$. La $\sigma$-algebra dei \textbf{boreliani} \e la $\sigma$-algebra $\beta X := \sigma (\tau X)$, e per costruzione contiene sia gli aperti che i chiusi di $X$.
Se $X$ \e di Hausdorff ed a base numerabile allora ogni $x \in X$ ammette un sistema numerabile di intorni $\{ A_n \}$ tale che $\cap_nA_n = \{ x \}$, dunque $\{ x \} \in \beta X$; di conseguenza, ogni sottoinsieme numerabile di $X$ \e boreliano.
Una funzione $f : X \to \bR$ si dice \textbf{boreliana} se $f^{-1}(I) \in \beta X$ per ogni aperto $I \subset \bR$. Ad esempio, ogni funzione continua \e boreliana, ed ogni funzione caratteristica $\chi_A$, $A \in \tau X$, \e boreliana (ma, in genere, non continua: si prenda ad esempio $X = \bR$ con la topologia usuale e si verifichi che $\chi_I$, $I := (0,1)$, \e boreliana, oltre che, ovviamente, discontinua).
}
\end{ex}

\

\noindent \textbf{Misure.} E' conveniente introdurre ora l'insieme dei {\em reali estesi}, o {\em retta reale estesa}
\[
\wa \bR := \bR \cup \{ - \infty , \infty \} \ .
\]
Gran parte della usuale struttura algebrica dei reali pu\'o essere esportata alla retta reale estesa,
\[
\left\{
\begin{array}{ll}
a \pm \infty := \pm \infty
\ , \
b \cdot (\pm \infty) := \pm \infty
\ , \
c \cdot (\pm \infty) := \mp \infty
\ , \
a / \pm \infty := 0 \ ,
\\
\forall a \in \bR \ , \ b>0 \ , \ c<0
\ ,
\end{array}
\right.
\]
tranne le operazioni $\infty - \infty$, $\pm \infty / \pm \infty$, che vengono lasciate indeterminate. Invece, a differenza di quanto accade nella teoria dei limiti definiamo $0 \cdot \pm \infty := 0$.
Riguardo la relazione d'ordine, poniamo $- \infty < a < \infty$ per ogni $a \in \bR$.

\begin{defn}
\label{def_MIS}
Uno \textbf{spazio misurabile} (o \textbf{di misura}) \e il dato di una terna $( X , \mM , \mu )$, dove $X$ \e un insieme, $\mM \subseteq 2^X$ \e una $\sigma$-algebra, e
\[
\mu : \mM \to \wa \bR^+ := \bR^+ \cup \{ + \infty \}
\]
\e una funzione, detta \textbf{misura}, tale che: 
\textbf{(1)} $\mu \emptyset = 0$; 
\textbf{(2)} $\mu$ soddisfa la propriet\'a di \textbf{additivit\'a numerabile}
\begin{equation}
\label{eq_MIS01}
\mu \left( \dot{\cup}_n A_n \right) = \sum_n \mu A_n
\ \ , \ \
\forall A_n \in \mM \ , \ A_n \cap A_m = \emptyset \ , \ n \neq m
\ .
\end{equation}
\end{defn}

\begin{rem} \label{def_MIS02}
{\it
\textbf{(1)} Considerando successioni $\{ A_n \}$ tali che $A_n = \emptyset$, $\forall n > 2$, troviamo che $\mu$ \e \textbf{additiva}, ovvero 
\[
\mu ( A_1 \dot{\cup} A_2) = \mu A_1 + \mu A_2
\ \ , \ \
\forall A_1 , A_2 \in \mM
\ , \
A_1 \cap A_2 = \emptyset \ .
\]
\textbf{(2)} Presi $A , A' \in \mM$ con $A \subset A'$ abbiamo $A_0 := A' - A \in \mM$ e $\mu A' = \mu A + \mu A_0$; dunque $\mu$ \e \textbf{monot\'ona}, ovvero 
\[
\mu A \leq \mu A' \ \ , \ \ \forall A \subseteq A' \ .
\]
\textbf{(3)} Sia $(X,\mM,\mu)$ uno spazio misurabile ed $A \in \mM$. Definiamo $\mM_A := \{ E \cap A , E \in \mM \}$ e $\mu_A : \mM_A \to \wa \bR$, $\mu_A (E \cap A ) := \mu (E \cap A)$. Allora $( A , \mM_A , \mu_A )$ \e uno spazio misurabile, che chiamiamo la \textbf{restrizione} di $(X,\mM,\mu)$ ad $A$.
}
\end{rem}

\begin{lem}
\label{lem_caressa}
Sia $(X,\mM,\mu)$ uno spazio di misura. Se $\{ E_n \} \subseteq \mM$ \e una successione tale che $\mu E_1 < + \infty$ e $E_{n+1} \subseteq E_n$ per ogni $n \in \bN$, allora $\mu (\cap_n E_n) = \lim_n \mu E_n$.
\end{lem}

\begin{proof}[Dimostrazione]
Posto $E := \cap_n E_n$ abbiamo 
$E_1 = E \dot{\cup} \dot{\bigcup}_n ( E_n - E_{n+1} )$
e quindi, per additivit\'a numerabile,
$\mu E_1 = \mu E + \sum_n \mu(E_n - E_{n+1})$.
Del resto $E_n = E_{n+1} \dot{\cup} (E_n - E_{n+1})$, per cui 
\[
\mu E_n - \mu E_{n+1} = \mu (E_n - E_{n+1}) \ ,
\]
e concludiamo che
\[
\mu E_1 = \mu E + \sum_n (\mu E_n - \mu E_{n+1}) = \mu E + \mu E_1 - \lim_n \mu E_n
\ .
\]
\end{proof}

Nei punti seguenti introduciamo alcune terminologie.
\begin{itemize}
\item Diciamo che uno spazio misurabile $( X , \mM , \mu )$ ha {\em misura finita} se $\mu X < \infty$; in particolare, diremo che $\mu$ \e una {\em misura di probabilit\'a} se $\mu X = 1$.
\item Uno spazio misurabile $(X,\mM,\mu)$ si dice {\em $\sigma$-finito} se esiste una successione $\{ A_n \}$ di insiemi di misura finita tali che $X = \cup_n A_n${\footnote{Osservare che, sfruttando le propriet\'a elementari delle $\sigma$-algebre, non \e restrittivo supporre che $A_n \cap A_m = \emptyset$, $n \neq m$.}}.
\item Uno spazio misurabile $(X,\mM,\mu)$ si dice {\em completo} se per ogni $A \in \mM$ con $\mu A = 0$ e $B \subseteq A$ risulta $B \in \mM$. Chiaramente in tal caso $\mu B = 0$. Osserviamo che, qualora $(X,\mM,\mu)$ non sia completo, \e sempre possibile definire la $\sigma$-algebra 
$\mM^* := \sigma ( \mM \cup \mM_0 )$, dove
$\mM_0 := \{ A \subset X : A \subseteq E \in \mM , \mu E = 0 \}$,
ed estendere $\mu$ ad $\mM^*$ ponendo $\mu A := 0$, $A \in \mM_0$ (vedi \cite[Prop.11.1.4]{Roy}).
\item Uno spazio misurabile $(X,\mM,\mu)$ si dice {\em localmente finito} se per ogni $x \in X$ esiste $V \in \mM$ tale che $x \in V$ e $0 < \mu V < + \infty$.
\item Sia $X$ uno spazio topologico. Una misura $\mu : \mM \to \wa \bR$ si dice {\em di Borel} se $\tau X \subset \mM$ (ed in tal caso $\mM$ contiene la $\sigma$-algebra dei boreliani).
\item Due misure $\mu , \nu : \mM \to \wa \bR^+$ sono {\em mutualmente singolari} se esistono $A,B \in \mM$ tali che $A \cup B = X$ e $\mu A = \nu B = 0$. Ed in tal caso, scriviamo $\mu \perp \nu$.
\end{itemize}

\begin{ex}[La misura di enumerazione]
\label{ex_misnum}
{\it
Consideriamo l'insieme $\bN$ dei naturali, la $\sigma$-algebra $2^\bN$ dei sottoinsiemi di $\bN$, e la funzione
\[
\mu : 2^\bN \to \wa \bR^+ 
\ : \
\mu A :=
\left\{
\begin{array}{ll}
|A| \ \ , \ \  A \ {finito \ ,}
\\
+ \infty \ \ , \ \ {altrimenti \ ,}
\end{array}
\right.
\]
dove $|A|$ \e la cardinalit\'a di $A \subseteq \bN$. Allora $\mu$ \e una misura localmente finita, ma non finita. L'unico sottoinsieme di $\bN$ di misura nulla \e l'insieme vuoto, per cui $\mu$ \e anche completa.
}
\end{ex}

\

\noindent \textbf{Misure con segno.} Sia $X$ un insieme ed $\mM$ una $\sigma$-algebra su $X$. Una {\em misura con segno} \e una funzione $\mu : \mM \to \wa \bR$ tale che:
\textbf{(1)} $\mu$ assume {\em solo uno} tra i valori $+ \infty , - \infty$;
\textbf{(2)} $\mu \emptyset = 0$;
\textbf{(3)} se $E = \dot{\cup}_n E_n$, con $\{ E_n \} \subseteq \mM$, allora $\mu E  = \sum_n \mu E_n$, e la convergenza della serie \e assoluta quando $\mu E \neq \pm \infty$.

Diciamo che $A \in \mM$ \e {\em positivo (negativo)} se $\mu A' \geq 0$ ($\leq 0$) per ogni $A' \subseteq A$, $A' \in \mM$.

\begin{lem}
Sia $(X,\mM,\mu)$ uno spazio di misura con segno. Allora:
(1) Se $E \subset \mM$ \e positivo allora ogni $E' \subset E$, $E' \in \mM$, \e positivo;
(2) $E := \cup_n E_n$ \e positivo per ogni successione $\{ E_n \}$ di insiemi positivi;
(3) Se $\mu E \in (0,\infty)$ allora esiste $A \subset E$ positivo con $\mu A > 0$.
\end{lem}

\begin{proof}[Dimostrazione]
(1) E' del tutto ovvia.
(2) Per ogni $A \subseteq E$, $A \in \mM$, poniamo 
\[
A_n 
\ := \
A \cap E_n \cap (\cap_{i=1}^{n-1} E_i^c)
\ .
\]
Chiaramente ogni $A_n$ \e misurabile e contenuto nel positivo $E_n$, per cui \e esso stesso positivo. Inoltre, essendo gli $A_n$ disgiunti troviamo
\[
\mu A = \sum_n \mu A_n \geq 0 \ ,
\]
per cui $E$ \e positivo.
(3) Se $E$ \e positivo allora non vi \e nulla da dimostrare, per cui assumiamo che esiste $E_1 \subset E$ tale che $\mu E_1 < 0$. Definiamo $n_1 \in \bN$ come il pi\'u piccolo intero tale che 
$\mu E_1 < -1/n_1$.
Procedendo ricorsivamente, se $E - \cup_i^k E_i$ non \e positivo definiamo $n_{k+1}$ come il pi\'u piccolo intero tale che esista $E_{k+1} \in \mM$ con
\[
E_{k+1} \subset E - \cup_{i=1}^k E_i
\ \ , \ \
\mu E_{k+1} < -1/n_{k+1}
\ .
\]
Osserviamo che gli $E_k$ sono mutualmente disgiunti e definiamo $A := E - \dot{\cup}_i E_i$; allora 
\[
E = A \dot{\cup} (\dot{\cup}_k E_k)
\ \Rightarrow \
\mu E = \mu A + \sum_k \mu E_k \in (0,\infty)
\ .
\]
La condizione precedente ci dice che $\sum_k 1/n_k$ converge, per cui $\lim_k n_k = \infty$. Vogliamo ora mostrare che $A$ \e positivo. A tale scopo prendiamo $\eps > 0$ ed osserviamo che esiste $k \in \bN$ tale che $1/(n_k-1) < \eps$; poich\'e per costruzione $A \subset E - \cup_i^k E_i$, troviamo che $A$ non pu\'o contenere insiemi di misura minore di $-1/(n_k-1) > - \eps$, e per arbitrariet\'a di $\eps$ concludiamo che $A$ non contiene insiemi con misura negativa.
\end{proof}

\begin{prop}[La decomposizione di Hahn] 
\label{prop_hahn}
Sia $( X,\mM,\mu )$ uno spazio di misura con segno. Allora esistono un insieme positivo $X^+$ ed un insieme negativo $X^-$ tali che $X^+ \cap X^- = \emptyset$ e $X^+ \cup X^- = X$. 
\end{prop}

\begin{proof}[Dimostrazione]
Possiamo supporre che $\mu$ non assuma il valore $+ \infty$ (altrimenti il ragionamento che segue si applica con ovvie modifiche). Sia 
$\lambda := \sup \{ \mu A : A \ {\mathrm{positivo}} \ \} < \infty$; 
allora esiste una successione $\{ A_n \}$ di insiemi positivi tali che $\mu A_n \stackrel{n}{\to} \lambda$, e definiamo 
\[
X^+ := \cup_n A_n
\ \ , \ \
X^- := X - X^+
\ .
\]
E' chiaro che $X^+ \cap X^- = \emptyset$ e $X^+ \cup X^- = X$, per cui rimane da verificare che $X^+$ \e positivo ed $X^-$ negativo. Che $X^+$ sia positivo segue dal Lemma precedente, punto (2), il che implica $\mu X^+ \leq \lambda$; ma del resto
\[
\mu X^+ = \mu A_n + \mu(A-A_n) \geq \mu A_n
\ , \forall n \in \bN 
\ \Rightarrow \
\mu X^+ \geq \lambda 
\ ,
\]
e quindi $\mu X^+ = \lambda$.
Passando a $X^-$, supponiamo per assurdo che esista $B \subseteq X^-$ con $\mu B > 0$. Per il Lemma precedente, punto (3), esiste $B'$ positivo con $B' \subseteq B \subseteq X^-$, e per costruzione $B' \cap X^+ = \emptyset$. Ora, $B' \dot{\cup} X^+$ \e positivo (ancora per il Lemma precedente), e quindi
\[
\lambda \geq \mu (B' \dot{\cup} X^+) = \mu B' + \mu X^+ = \mu B' + \lambda \ ,
\]
il che \e assurdo. Dunque $X^-$ \e negativo.
\end{proof}

Osserviamo che la decomposizione di Hahn non \e unica, in quanto se ne pu\'o perturbare la costruzione con insiemi di misura nulla. Definendo invece
\[
\mu^+ A :=   \mu (A \cap X^+)
\ \ , \ \
\mu^- A := - \mu (A \cap X^-)
\ \ , \ \
A \in \mM
\ ,
\]
otteniamo due misure mutualmente singolari ed univocamente definite. Abbiamo cos\'i dimostrato:

\begin{prop}[La decomposizione di Jordan]
Sia $( X,\mM,\mu )$ uno spazio di misura con segno. Allora esistono, e sono uniche, due misure $\mu^+$ e $\mu^-$ mutualmente singolari tali che $\mu = \mu^+ - \mu^-$.
\end{prop}

La {\em variazione totale di} $\mu$ si definisce come la misura (positiva)
\begin{equation}
\label{def_vartot}
|\mu| : \mM \to \wa \bR^+
\ \ , \ \
| \mu | A := \mu^+ A + \mu^- A
\ \ , \ \
A \in \mM
\ .
\end{equation}
Sia ora $X$ un insieme e $\beta \subseteq 2^X$ una $\sigma$-algebra (in effetti, abbiamo in mente il caso in cui $X$ \e uno spazio topologico e $\beta$ la $\sigma$-algebra dei boreliani). Definiamo lo spazio delle {\em misure con segno finite}
\begin{equation}
\label{def_l1}
\Lambda_\beta^1(X)
\ := \
\{ 
\mu : \mM \to \wa \bR
\ \ {\mathrm{misura \ con \ segno}} 
\ \ : \ \
\mM \supseteq \beta
\ \ {\mathrm{e}} \ \ 
|\mu|(X) < \infty
\}
\ .
\end{equation}
Prese $\mu , \mu' \in \Lambda_\beta^1(X)$, $\mu : \mM \to \wa \bR$, $\mu' : \mM' \to \wa \bR$, osserviamo che $\beta \subseteq \mM \cap \mM' \neq \emptyset$, per cui ha senso considerare, preso $\lambda \in \bR$, l'applicazione
\[
\mu + \lambda \mu' : \mM \cap \mM' \to \wa \bR
\ \ , \ \
\{ \mu + \lambda \mu' \} E := \mu E + \lambda \mu' E
\ .
\]
Una verifica immediata mostra che $\mu + \lambda \mu'$ appartiene in effetti a $\Lambda_\beta^1(X)$, il quale diventa cos\'i uno spazio vettoriale. Definendo
\begin{equation}
\label{def_vartot2}
\| \mu \| := | \mu | (X)
\ \ , \ \
\mu \in \Lambda_\beta^1(X)
\ ,
\end{equation}
otteniamo una norma nel senso di \S \ref{sec_afunct} (per i dettagli si veda l'Esercizio \ref{sec_MIS}.7), dunque $\Lambda_\beta^1(X)$ \e uno spazio normato.

\

\noindent \textbf{Cenni sulle misure complesse.} Sia $X$ un insieme ed $\mM$ una $\sigma$-algebra su $X$. Una {\em misura complessa} \e un'applicazione del tipo
\[
\mu = \mu_1 + i \mu_2 : \mM \to \bC \ ,
\]
dove $\mu_1 , \mu_2$ sono misure con segno {\em finite} definite su $\mM$. 
Diremo che $\mu$ \e completa, boreliana, di Radon, etc., se lo sono sia $\mu_1$ che $\mu_2$.
Analogamente al caso reale, fissata una $\sigma$-algebra $\beta \subseteq 2^X$ troviamo che l'insieme delle misure complesse con dominio $\supseteq \beta$ \e uno spazio vettoriale complesso, che denotiamo con $\Lambda_\beta^1(X,\bC)$.
E' possibile definire il valore assoluto della misura complessa $\mu$, nel modo che segue:
\begin{equation}
\label{eq.mod}
|\mu|E 
\ := \
\sup \{ \sum_n |\mu E_n| : \{ E_n \} \subseteq \mM , \dot{\cup}_n E_n = E \}
\ \ , \ \
\forall E \in \mM
\ .
\end{equation}
Non \e difficile verificare che $|\mu| : \mM \to \bR$ \e una misura finita nel senso usuale, cosicch\'e definendo
\[
\| \mu \| := |\mu|(X)
\ \ , \ \
\forall \mu \in \Lambda_\beta^1(X,\bC)
\ ,
\]
otteniamo una norma.

\

\noindent \textbf{La costruzione di Carath\'eodory.} 
Nelle righe che seguono riportiamo uno dei metodi pi\'u comuni per costruire misure, di fatto un'astrazione di alcuni dei passi necessari per la definizione della misura di Lebesgue sulla retta reale, come vedremo nel seguito (\S \ref{sec_lebR}).
Sia $X$ un insieme; una {\em misura esterna} \e un'applicazione
$\mu^* : 2^X \to \wa \bR^+$
tale che: 

\

\noindent (1) $\mu^*(\emptyset) = 0$; 

\noindent (2) $\mu^* A \leq \mu^* B$, $\forall A \subseteq B$; 

\noindent (3) $\mu^* (\cup_k A_k) \leq \sum_k \mu^* A_k$, $\forall \{ A_k \}_{k \in \bN} \subseteq 2^X$ ({\em subadditivit\'a}). 

\

\noindent Un insieme $E \subseteq X$ si dice {\em misurabile secondo Carath\'eodory} \, se
\begin{equation}
\label{eq_subadd00}
\mu^* A = \mu^*(E \cap A) + \mu^*(E^c \cap A)
\ \ , \ \
\forall A \in 2^X
\ .
\end{equation}
Denotiamo con $\mM_*$ la classe degli insiemi misurabili secondo Carath\'eodory e definiamo $\mu_* := \mu^*|_{\mM_*}$. Osserviamo che per dimostrare che un generico $E \in 2^X$ appartiene ad $\mM_*$ \e sufficiente verificare, per subadditivit\'a, che
\begin{equation}
\label{eq_subadd0}
\mu^* A \ \geq \ \mu^*(E \cap A) + \mu^*(E^c \cap A)
\ \ , \ \
\forall A \in 2^X
\ .
\end{equation}

\begin{lem}
\label{lem_car}
$\mM_*$ \e una $\sigma$-algebra, e $(X,\mM_*,\mu_*)$ \e uno spazio misurabile completo.
\end{lem}

\begin{proof}[Dimostrazione]
L'insieme vuoto appartiene chiaramente ad $\mM_*$; inoltre, il complementare di ogni elemento di $\mM_*$ appartiene ad $\mM_*$, per cui anche $X$ appartiene ad $\mM_*$.
Presa una successione $\{ E_n \} \subseteq \mM_*$, per ricorsivit\'a otteniamo, per ogni $A \in 2^X$,
\begin{equation}
\label{eq_subadd}
\begin{array}{ll}
\mu^* A & =
\mu^*(A \cap E_1) + \mu^*(A \cap E_1^c)
\\ & =
\mu^*(A \cap E_1) + \mu^*(A \cap E_1^c \cap E_2) + \mu^*(A \cap E_1^c \cap E_2^c)
\\ & =
\ldots
\\ & =
\mu^*(A \cap E_1) + 
\sum_{k=2}^n \mu^* \left( A \cap \bigcap_{j<k} E_j^c \cap E_k \right) +
\mu^* \left( A \cap \bigcap_{k=1}^n E_k^c \right)
\ .
\end{array}
\end{equation}
Ora, possiamo assumere che $E_n \subseteq (\cup_{k<n}E_k)^c$, infatti ci\'o non altera l'unione $n$-esima $E^{(n)} := \cup_k^nE_k$. Per cui risulta
\[
\bigcap_{k=1}^n E_k^c 
\ = \ 
E^{(n),c}
\ \ , \ \
A \cap \bigcap_{j<k} E_j^c \cap E_k 
\ = \ 
A \cap E_k \cap E^{(k-1),c} 
\ = \ 
A \cap E_k 
\ ;
\]
usando monoton\'ia e subadditivit\'a numerabile di $\mu^*$, concludiamo che 
\[
\begin{array}{ll}
\mu^*A & =  \
\mu^*(A \cap E_1) + 
\sum_{k=2}^n \mu^*(A \cap E_k \cap E^{(k-1),c}) + 
\mu^* ( A \cap E^{(n),c} )
\\ & = \
\sum_{k=1}^n \mu^*(A \cap E_k) + \mu^* ( A \cap E^{(n),c} )
\\ & \geq \
\mu^*(A \cap E^{(n)}) + \mu^* ( A \cap E^{(n),c} ) \ .
\end{array}
\]
Per cui (\ref{eq_subadd0}) \e verificata e $\mM_*$ \e una $\sigma$-algebra.
Dimostriamo che $\mu_*$ \e numerabilmente additiva: preso $E := \dot{\cup}_n E_n$, $\{ E_n \} \subset \mM_*$, sostituendo $A$ con $E$ in (\ref{eq_subadd}) otteniamo $\mu_* E = \sum_n \mu_* E_n$. 
Infine verifichiamo che $\mu_*$ \e completa. A tale scopo consideriamo $E \in \mM_*$
tale che $\mu_*E = 0$ e mostriamo che se $E' \subset E$ allora $E \in \mM_*$; preso
$A \in 2^X$ osserviamo che $E' \cap A \subset E \cap A$ e $E^c \cap A \subset {E'}^c \cap A$,
cosicch\'e, usando (\ref{eq_subadd00}), la monoton\'ia di $\mu^*$ ed il fatto che $\mu^*E = 0$,
\[
\mu^* (E' \cap A) \leq \mu^* (E \cap A) = 0
\ \ , \ \
\mu^*A = \mu^* (E^c \cap A) \leq \mu^* ({E'}^c \cap A) \leq \mu^*A
\ .
\]
Concludiamo che $\mu^* ({E'}^c \cap A) = \mu^*A$, per cui 
$\mu^* A = \mu^*(E' \cap A) + \mu^*( {E'}^c \cap A)$,
ovvero $E' \in \mM_*$.
\end{proof}


\

\noindent \textbf{Ulteriori strutture su spazi misurabili.} Nelle applicazioni capita spesso che un insieme $X$ sia equipaggiato di ulteriore struttura (di spazio topologico, vettoriale, ...); per cui, qualora si considerino delle misure definite su $X$ \e opportuno che presentino delle propriet\'a di compatibilit\'a rispetto a tale struttura. Nelle righe seguenti diamo una breve rassegna delle costruzioni pi\'u importanti in questo senso. 

\

\noindent {\em 1. Misure di Borel e regolarit\'a esterna.} Sia $(X,\mM,\mu)$ uno spazio misurabile di Borel. Diciamo che $\mu$ soddisfa la propriet\'a di {\em regolarit\'a esterna} se
\begin{equation}
\label{eq_MIS_RE}
\mu A = \inf \{ \mu U : A \subset U , U \in \tau X  \}
\ \ , \ \
A \in \mM
\ .
\end{equation}
Nella definizione precedente gli aperti giocano il ruolo di una famiglia di insiemi "aventi delle buone propriet\'a", la cui misura sia facilmente calcolabile (si pensi alla lunghezza degli intervalli, come vedremo in seguito), ed attraverso i quali sia possibile "approssimare" la misura di un arbitrario $A \in \mM$ nei termini dell'estremo inferiore (\ref{eq_MIS_RE}).

\begin{rem}\label{oss_cbor}{\it
Segue banalmente da (\ref{eq_MIS_RE}) che $\mu A = \mu \ovl A$ per ogni $A \in \mM$; infatti se un aperto $U$ contiene $A$, allora contiene anche chiusura di $A$.
}
\end{rem}

\noindent {\em 2. Misure di Radon e regolarit\'a interna.} Sia $(X,\mM,\mu)$ uno spazio misurabile di Borel con $X$ localmente compatto e di Hausdorff. Denotata con $\mK$ la famiglia dei compatti su $X$ osserviamo che, poich\'e $X$ \e di Hausdorff, ogni $C \in \mK$ \e anche chiuso (vedi \cite[Prop.10.6]{Cam} o \cite[1.6.5]{Ped}), per cui $\mK \subset \mM$.

\begin{defn}
Sia $X$ uno spazio localmente compatto di Hausdorff e $\mu : \mM \to \wa \bR^+$ una misura di Borel su $X$. Diciamo che $\mu$ \e una \textbf{misura di Radon} se valgono le seguenti propriet\'a:

\noindent
\textbf{(1)} $\mu C < + \infty$, $\forall C \in \mK$;
\textbf{(2)} $\mu E = \sup \{ \mu C : C \subset E , C \in \mK \}$, $\forall E \in \mM$
{\footnote{
Questa propriet\'a \e detta \textbf{regolarit\'a interna}.}}.
\end{defn}
Analogamente al caso delle misure di Borel abbiamo una "famiglia privilegiata" di insiemi, quella dei compatti, a partire dalla quale calcolare per approssimazione la misura di un generico $A \in \mM$. Osserviamo che richiediamo che un compatto abbia misura finita, cosicch\'e ogni misura di Radon \e localmente finita (infatti, ogni $x \in X$ possiede un intorno compatto).

\begin{ex} \label{ex_MIS_Dirac}
{\it
Sia $X$ uno spazio localmente compatto di Hausdorff equipaggiato con una $\sigma$-algebra $\mM \supseteq \tau X$. Per ogni $x \in X$, definiamo la \textbf{misura di Dirac}
\[
\mu_x A := 
\left\{
\begin{array}{ll}
1 \ \ , \ \ x \in A
\\
0 \ \ , \ \ x \notin A
\end{array}
\right.
\ \ , \ \
A \in \mM
\ .
\]
Allora $\mu_x$ \e una misura di Radon, e chiaramente anche una misura di probabilit\'a.
}
\end{ex}

Assumiamo ora che $X$ sia, in particolare, compatto. Diciamo che una misura {\em con segno $\mu$ su $X$ \e di Radon} se $| \mu |$ \e di Radon, e denotiamo con $R(X)$ l'insieme delle misure di Radon con segno su $X$. Per definizione $R(X) \subseteq \Lambda_\beta^1(X)$ (vedi (\ref{def_l1})), dove $\beta := \beta X$ \e la $\sigma$-algebra dei boreliani. Visto che $\Lambda_\beta^1(X)$ \e uno spazio vettoriale troviamo $\lambda \mu + \nu \in \Lambda_\beta^1(X)$ per ogni $\lambda \in \bR$, $\mu , \mu' \in R(X)$, e si verifica facilmente che
\[
| \lambda \mu + \mu' | E 
\ = \
\sup \{ | \lambda \mu + \mu' | C : C \subset E , C \in \mK \}
\ \ , \ \
E \in \mM \cap \mM'
\ ;
\]
dunque $\lambda \mu + \mu' \in R(X)$ ed $R(X)$ \e uno spazio normato (con la stessa norma di $\Lambda_\beta^1(X)$). In effetti, $R(X)$ \e addirittura uno spazio di Banach (ci\'o in conseguenza del teorema di Riesz-Markov, vedi Esempio \ref{ex_dual_CX}). Per ulteriori dettagli sulle misure di Radon si veda \cite[\S 6.3]{Ped}.

\

\noindent {\em 3. Misure di Haar.} Sia $G$ un gruppo, nonch\'e uno spazio topologico localmente compatto e di Hausdorff. Diciamo che $G$ \e un {\em gruppo topologico} se l'applicazione
\[
G \times G \to G \ \ , \ \ g,g' \mapsto g^{-1}g'
\]
\e continua. Esempi di gruppi topologici sono gli spazi euclidei $\bR^n , \bC^m$, equipaggiati con l'operazione di somma, il toro complesso $\bT := \{ z \in \bC : |z|=1 \}$, equipaggiato con l'operazione di moltiplicazione, ed i gruppi di matrici ${\mathbb{GL}}(d,\bR)$, $\ud$, ..., $d \in \bN$ (ovvero, i cosiddetti gruppi di Lie classici), tutti equipaggiati con l'operazione del prodotto matriciale. Una misura di Radon $\mu \in R(G)$, $\mu : \mM \to \wa \bR^+$, si dice {\em misura di Haar} se verifica una delle due propriet\'a di {\em invarianza per traslazione}:
\begin{equation}
\label{eq_MIS_IT}
\mu A = \mu (Ag) 
\ \ , \ \
\mu A = \mu (gA)
\ \ , \ \
A \in \mM
\ ,
\end{equation}
dove
$Ag := \{ ag , a \in A \}$, $gA := \{ ga , a \in A \}$. Sull'esistenza ed unicit\'a delle misure di Haar, si veda \cite[\S 14.4-\S 14.6]{Roy}; oltre, sempre su questa strada, c'\e l'analisi armonica astratta (\cite{Fol,HR}).

\begin{ex}
\label{ex_misnumZ}
{\it
Consideriamo il gruppo additivo degli interi $(\bZ,+)$, sul quale definiamo la topologia discreta. Definiamo su $\bZ$ la misura di enumerazione $\mu : 2^\bZ \to \wa \bR^+$, costruita come nell'Esempio \ref{ex_misnum}. Poich\'e $\bZ$ \e discreto, \e evidente che $\mu$ \e di Radon (in particolare, i compatti di $\bZ$ sono i sottoinsiemi finiti). Inoltre \e chiaro che 
\[
\mu A = \mu (A+k)
\ \ , \ \
\forall k \in \bZ
\ ,
\]
dove $A+k := \{ h+k , h \in A \}$. Dunque $\mu$ \e invariante per traslazioni e quindi una misura di Haar.
}
\end{ex}

\

\noindent {\em 4. Misure di Hausdorff.} Sia ora $(X,d)$ uno spazio metrico e $\delta > 0$. Consideriamo $\eps > 0$ e definiamo
\begin{equation}
\label{eq.mis.Haus}
\mu^\delta_\eps (A) 
\ := \
\inf \
\{ \ 
\sum_k r_k^\delta 
\ : \ 
A \subseteq \cup_{k \in \bN} \Delta(x_k,r_k) 
\ , \  
x_k \in X \ , \ r_k < \eps 
\ \}
\ .
\end{equation}
Ora, $\mu^\delta_\eps(A)$ \e una funzione decrescente in $\eps$, per cui esiste il limite
\[
\mu^{\delta,*}(A) 
\ := \ 
\lim_{\eps \to 0} \mu^\delta_\eps (A) \ \ , \ \ A \in 2^X \ .
\]
Si pu\'o verificare che $\mu^{\delta,*} : 2^X \to \wa \bR$ \e una misura esterna dalla quale, con il metodo di Carath\'eodory, possiamo costruire una $\sigma$-algebra $\mM_\delta$ ed una misura $\mu^\delta : \mM_\delta \to \wa \bR$ nota con il nome di {\em misura di Hausdorff}.
Una nozione correlata alla costruzione precedente \e la {\em dimensione di Hausdorff}, nota anche come {\em dimensione di Hausdorff-Besicovitch}, la quale \e definita dall'espressione
\[
\dim_H X := \inf \{ \delta > 0 : \mu^\delta(X) = 0  \} \ .
\]
Si verifica che se $X$ \e uno spazio vettoriale a dimensione finita $n$, o una variet\'a di dimensione $n$, allora $\dim_H X = n$. Pi\'u in generale, la dimensione di Hausdorff trova applicazioni nella teoria dei frattali. Per ulteriori dettagli si veda \cite[\S 12.9]{Roy} e \cite[\S 2.6]{dBar}.

\

\noindent \textbf{Funzioni misurabili.} Passiamo ora ad introdurre l'importante concetto di funzione misurabile.

\begin{lem}
\label{lem_mis}
%
%
Sia $X$ un insieme ed $\mM \subseteq 2^X$ una $\sigma$-algebra. Presa una funzione $f : X \to \wa \bR$, le seguenti affermazioni sono equivalenti:
(1) Per ogni $\alpha \in \bR$, $f^{-1}\left( (\alpha,+\infty) \right) \in \mM$;
(2) Per ogni $\alpha \in \bR$, $f^{-1}\left( [\alpha,+\infty) \right) \in \mM$;
(3) Per ogni $\alpha \in \bR$, $f^{-1}\left( (-\infty,\alpha) \right) \in \mM$;
(4) Per ogni $\alpha \in \bR$, $f^{-1}\left( (-\infty,\alpha] \right) \in \mM$.
\end{lem}

\begin{proof}[Dimostrazione]
Se vale (1) allora, presa una successione monot\'ona crescente $\alpha_n \to \alpha$, abbiamo
$\cap_n f^{-1}\left( (\alpha_n,+\infty) \right) \in \mM$,
ma del resto
$\cap_n f^{-1}\left( (\alpha_n,+\infty) \right) = f^{-1}\left( [\alpha,+\infty) \right)$,
per cui vale (2). Supposto che valga (2), troviamo
$X - f^{-1}\left( [\alpha,+\infty) \right) \in \mM$,
ma del resto 
$X - f^{-1}\left( [\alpha,+\infty) \right) = f^{-1}\left( (-\infty,\alpha) \right)$,
e quindi vale (3). L'implicazione (3) $\Rightarrow$ (4) \e analoga a (1) $\Rightarrow$ (2), mentre (4) $\Rightarrow$ (1) \e analoga a (2) $\Rightarrow$ (3).
\end{proof}

\begin{rem}{\it
Dal Lemma precedente segue che se $f$ \e misurabile allora per ogni $\alpha \in \bR$ si ha
$f^{-1}(\alpha) = f^{-1}\left( (-\infty,\alpha] \right) \cap f^{-1}\left( [\alpha,+\infty) \right) \in \mM$. 
}
\end{rem}

\begin{defn}
\label{def_misfun}
Sia $X$ un insieme ed $\mM \subseteq 2^X$ una $\sigma$-algebra. Una funzione $f : X \to \wa \bR$ si dice \textbf{misurabile} se \e verificata una delle propriet\'a del lemma precedente. Denotiamo con $M(X)$ l'insieme delle funzioni misurabili su $X$.
\end{defn}

Per illustrare la connessione tra la definizione precedente e Def.\ref{def_mis_astr} consideriamo, in particolare, una funzione {\em a valori reali} $f \in M(X)$; segue allora banalmente dal Lemma \ref{lem_mis} che in effetti $f$ \e misurabile secondo Def.\ref{def_mis_astr} se equipaggiamo $\bR$ con la $\sigma$-algebra dei boreliani $\beta \bR$ (si noti infatti che $\beta \bR$ \e generata dagli intervalli).

Usando la stabilit\'a di $\mM$ rispetto all'unione ed all'intersezione \e facile verificare che $M(X)$ \e un'algebra (\cite[Prop.3.5.18]{Roy}). Chiaramente ogni funzione caratteristica $\chi_E$, $E \in \mM$, \e misurabile, e quindi ogni funzione semplice \e misurabile, ovvero $S(X) \subseteq M(X)$. Se $X$ \e uno spazio topologico ed $\mM$ contiene i boreliani allora ogni funzione boreliana \e misurabile; in particolare, ogni funzione continua \e misurabile.

\begin{thm}
\label{thm_MIS}
Sia $X$ un insieme ed $\mM \subseteq 2^X$ una $\sigma$-algebra. Date $f,g \in M(X)$ ed $\{ f_n \} \subset M(X)$, si ha
\begin{equation}
\label{eq_MIS_1}
\ovl h := \sup \{ f,g \} \in M(X)
\ \ , \ \
\unl h := \inf \{ f,g \} \in M(X)
\ ,
\end{equation}
\begin{equation}
\label{eq_MIS_2}
\ovl f (x) := \lim_n \sup f_n(x) \ , \ x \in X  \ \Rightarrow \ \ovl f \in M(X)
\ ,
\end{equation}
\begin{equation}
\label{eq_MIS_3}
\unl f (x) := \lim_n \inf f_n(x) \ , \ x \in X  \ \Rightarrow \ \unl f \in M(X)
\ .
\end{equation}
\end{thm}

\begin{proof}[Dimostrazione]
Per ogni $\alpha \in \bR$ troviamo $\ovl h^{-1}((\alpha,+\infty)) = f^{-1}((\alpha,+\infty)) \cup g^{-1}((\alpha,+\infty)) \in \mM$, dunque $\ovl h$ \e misurabile. Analogamente $\unl h$ \e misurabile. Ora, essendo $\mM$ chiusa per unioni ed intersezioni numerabili abbiamo che estremi superiori ed inferiori di successioni di funzioni misurabili sono funzioni misurabili; per cui, concludiamo che $\ovl f = \inf_n \sup_{k \geq n} f_k$, $\unl f = \sup_n \inf_{k \geq n} f_k$ sono misurabili.
\end{proof}

\begin{prop}
\label{prop_caressa}
Sia $(X,\mM,\mu)$ uno spazio misurabile ed $f \in M(X)$, $f \geq 0$. Allora esiste una succesione monot\'ona crescente $\{ \psi_n \} \subset S(X)$ che converge puntualmente ad $f$. Se $X$ \e $\sigma$-finito allora si pu\'o scegliere ogni $\psi_n$ in maniera tale che il supporto abbia misura finita.
\end{prop}

\begin{proof}[Dimostrazione]
Per ogni $n \in \bN$ e $t \in \bR^+$ esiste un unico $k_{t,n}$ tale che 
\[
t \in \left[ \ k_{t,n}2^{-n} \ , \ (k_{t,n}+1)2^{-n}  \ \right]
\ ,
\]
cosicch\'e definiamo
\[
\psi_n(x) := 
\left\{
\begin{array}{ll}
k_{f(x),n}2^{-n} \ \ , \ \ f(x) \in [ 0 , n )
\\
n  \ \ , \ \ f(x) \in [n,\infty]
\ .
\end{array}
\right.
\]
La successione $\{ \psi_n \}$ chiaramente soddisfa le propriet\'a desiderate. Se $X$ \e $\sigma$-finito allora abbiamo $X = \cup_n A_n$ con $\mu A_n < \infty$, e ponendo $A^{(n)} := \cup_m^n A_m$,
$\varphi_n := \chi_{A^{(n)}} \psi_n$
otteniamo, come desiderato, che ogni $\varphi_n$ ha supporto con misura finita e $\varphi_n \to f$.
\end{proof}

\

\noindent \textbf{Equivalenza e convergenza q.o..} Introduciamo ora un'importante terminologia: diciamo che una propriet\'a vale {\em quasi ovunque in $x \in X$ (q.o.)} se la misura del sottoinsieme di $X$ nel quale essa non \e verificata ha misura nulla. Ad esempio, useremo spesso l'espressione $f = g$ q.o., intendendo con ci\'o che 
\[
\mu \{ x : f(x) \neq g(x) \} = 0 \ .
\]
E' chiaro che l'uguaglianza q.o. di due funzioni definisce una relazione di equivalenza su $M(X)$. Nel seguito accadr\'a spesso che identificheremo $f \in M(X)$ con la sua classe di equivalenza q.o.. 

\begin{prop}
\label{prop_fgmis}
Sia $(X,\mM,\mu)$ completo, $g \in M(X)$ ed $f : X \to \wa \bR$ tale che $f = g$ q.o.. Allora $f$ \e misurabile.
\end{prop}

\begin{proof}[Dimostrazione]
Posto $E := \{ x \in X : f(x) \neq g(x) \}$ abbiamo $\mu E = 0$ e, preso $\alpha \in \bR$,
\[
f^{-1}(\alpha,+\infty) 
\ = \ 
( g^{-1}(\alpha,+\infty) \cap E^c )
\cup
( f^{-1}(\alpha,+\infty) \cap E )
\ .
\]
L'insieme $( g^{-1}(\alpha,+\infty) \cap E^c )$ \e misurabile, essendo esso intersezione di insiemi misurabili. L'insieme $( f^{-1}(\alpha,+\infty) \cap E )$, essendo contenuto nell'insieme di misura nulla $E$, \e anch'esso misurabile, per cui concludiamo che $f^{-1}(\alpha,+\infty)$ \e misurabile.
\end{proof}

\begin{defn}
\label{def_cqo}
Sia $\{ f_n \} \cup \{ f \} \subset M(X)$. Diciamo che la successione $\{ f_n \}$ \textbf{converge q.o.} ad $f$ se $\lim_n f_n(x) = f(x)$, q.o. in $x \in X$.
\end{defn}

\begin{thm}[Egoroff]
\label{thm_egoroff}
Sia $(X,\mM,\mu)$ uno spazio di misura finita ed $\{ f_n \} \subset M(X)$ una successione convergente q.o. ad una funzione $f$. Allora per ogni $\eps > 0$ esistono un insieme $A \subset X$ con $\mu A < \eps$ ed $n_0 \in \bN$ tali che $\sup_{X-A} | f_n(x) - f(x) | < \eps$ per ogni $n \geq n_0$.
\end{thm}

\begin{proof}[Dimostrazione]
Scelto $\eps > 0$ poniamo
\[
A_n := \{ x \in X :  | f_n(x) - f(x) | \geq \eps \}
\]
e 
\[
B_N 
:=
\bigcup_{n=N}^\infty A_n
=
\{ x \in X \ | \ \exists n \geq N : | f_n(x) - f(x) | \geq \eps  \}
\ .
\]
Osservando che $B_{N+1} \subseteq B_N$, usando la convergenza di $\{ f_n \}$ q.o. troviamo
\begin{equation}
\label{eq_BC0}
0 
\ = \
\mu \left( \bigcap_N B_N \right)
\ \stackrel{(*)}{=} \
\lim_{N \to \infty} \mu B_N
\ .
\end{equation}
Per l'uguaglianza (*) abbiamo usato il Lemma \ref{lem_caressa}, per applicare il quale \e necessario che $X$ (e quindi $B_1$) abbia misura finita (a tal proposito si veda l'Esercizio \ref{sec_MIS}.6). Concludiamo che esiste $N_0 \in \bN$ tale che $\mu B_{N_0} < \eps$, e posto $A := B_{N_0}$ risulta
\[
x \in X-A 
\ \Leftrightarrow \ 
|f_n(x)-f(x)| < \eps 
\ , \ 
\forall n \geq N_0
\ .
\]
\end{proof}

\subsection{La misura di Lebesgue sulla retta reale.}
\label{sec_lebR}

In questa sezione esponiamo la costruzione della misura di Lebesgue sulla retta reale, con l'idea di base che questa debba coincidere con la lunghezza di un intervallo qualora sia valutata su questo. Come primo passo fissiamo alcune notazioni,
\begin{equation}
\label{def_J}
\left\{
\begin{array}{ll}
\mI^{aa} := \{ (a,b) , a<b \in \bR \} \ , \
\mI^{ac} := \{ (a,b] , a<b \in \bR \} \ , \
\\
\mI^{ca} := \{ [a,b) , a<b \in \bR \} \ , \
\mI^{cc} := \{ [a,b] , a<b \in \bR \} \ ;
\\
\mI := \mI^{aa} \cup \mI^{ac} \cup \mI^{ca} \cup \mI^{cc}  \ .
\end{array}
\right.
\end{equation}
Definiamo quindi 
\[
l(I) := b-a \ \ , \ \ \forall I \in \mI \ ,
\]
ed introduciamo l'applicazione
\begin{equation}
\label{eq_misEst}
\mu^* : 2^\bR \to \wa \bR^+
\ \ , \ \
\mu^* A :=  
\inf \left\{ \sum_{n \in \bN} l(I_n) \ : \ A \subset \bigcup_{n \in \bN} I_n \ , \ \{ I_n \} \subset \mI^{aa} \right\} 
\ .
\end{equation}
I seguenti risultati mostrano che $\mu^*$ \e in effetti una misura esterna invariante per traslazioni. 
\begin{lem}
\label{lem_ME01}
Valgono le seguenti propriet\'a:
\begin{enumerate}
\item $\mu^*A \geq 0$, $A \in 2^\bR$;
\item $\mu^* \emptyset = 0$;
\item $A \subseteq B$ $\Rightarrow$ $\mu^*A \leq \mu^*B$, $A,B \in 2^\bR$;
\item $\mu^*\{ x \} = 0$, $x \in \bR$;
\item $\mu^* (A+x) = \mu^* A$, $A \in 2^\bR$, $A+x := \{ a+x : a \in A \}$ (\textbf{invarianza per traslazioni}).
\end{enumerate}
\end{lem}

\begin{proof}[Dimostrazione]
(1) Segue dal fatto che $l(I) > 0$, $\forall I \in \mI^{aa}$;
(2) Segue dal fatto che $\emptyset$ \e contenuto in intervalli di lunghezza piccola a piacere;
(3) Segue dal fatto che se $B \subset \cup_n I_n$ allora $A \subset \cup_n I_n$;
(4) Poich\'e $\{ x \} \subset (x-\eps,x+\eps)$ abbiamo $\mu^*\{ x \} < 2 \eps$, con $\eps > 0$ piccolo a piacere;
(5) Segue dall'uguaglianza $l(I+x) = l(I)$, $\forall I \in \mI^{aa}$.
\end{proof}

\begin{lem}
\label{lem_ME02}
Per ogni $a,b \in \bR$, $a<b$, si ha
\begin{equation}
\label{eq_ME02}
b-a = \mu^*(I) = \mu^*(\ovl I) = \mu^* (J) = \mu^*(J')
\ \ , \ \
I := (a,b)
\ , \
\ovl I := [a,b]
\ , \
J := [a,b)
\ , \
J' := (a,b]
\ .
\end{equation}
Se $I$ \e non limitato (ovvero $a = - \infty$ o $b = \infty$) allora $\mu^*I = \infty$.
\end{lem}

\begin{proof}[Sketch della dimostrazione]
Ci limitiamo a trattare il caso di un intervallo chiuso e limitato $[a,b]$, $a,b \in \bR$ (gli altri casi di intervalli limitati sono analoghi, mentre quelli non limitati sono semplici da dimostrare e vengono lasciati per esercizio, vedi ad esempio \cite[\S 2]{dBar} o \cite[Prop.1.3.3]{Acq}). 
Innanzitutto, osserviamo che per ogni $\eps > 0$ si ha, per definizione di $\mu^*$, 
\[
\mu^*([a,b]) \ \leq \ l((a-\eps,b+\eps)) \ = \ b-a+2 \eps \ ,
\]
per cui, data l'arbitrariet\'a di $\eps$, concludiamo che $\mu^*([a,b]) \leq b-a$. 
D'altra parte, per definizione di $\mu^*$ esiste una successione $\{ I_n := (a_n,b_n) \}$ che ricopre $[a,b]$ e tale che 
\begin{equation}
\label{eq_prop_l0}
\mu^*([a,b]) \geq \sum_n l(I_n) - \eps \ .
\end{equation}
Essendo $\{ I_n \}$ un ricoprimento aperto di $[a,b]$ possiamo estrarre un sottoricoprimento finito di intervalli distinti $\{ \wt I_k := ( \wt a_k , \wt b_k ) \}_{k=1}^N$. Chiaramente, possiamo ordinare $\{ \wt I_k \}$ in maniera tale che $\wt a_k \leq \wt a_{k+1}$, $k = 1 , \ldots , N$, per cui
\begin{equation}
\label{eq_prop_l2}
\wt b_N - \wt a_1 
\ = \
\sum_{k=1}^N (\wt b_k - \wt a_k)
- 
\sum_{k=1}^{N-1} (\wt b_k - \wt a_{k+1})
\ \leq \
\sum_{k=1}^N l(\wt I_k)
\ \leq \
\sum_n l(I_n)
\ .
\end{equation}
Concludiamo che 
\[
\mu^*([a,b])
\ \stackrel{ (\ref{eq_prop_l0}) }{ \geq } \
\sum_n l(I_n) - \eps
\ \stackrel{ (\ref{eq_prop_l2}) }{ \geq } \
\wt b_n - \wt a_1 - \eps 
\ \geq \
b-a - \eps 
\ .
\]
\end{proof}

\begin{lem}[Subadditivit\'a]
\label{lem_ME03a}
Sia $\{ A_n \}$ una successione di sottoinsiemi di $\bR$. Allora
\begin{equation}
\label{eq_ME03}
\mu^* \left( \bigcup_n A_n \right)
\ \leq \
\sum_n \mu^* A_n
\ .
\end{equation}
\end{lem}

\begin{proof}[Dimostrazione]
Possiamo assumere che $\sum_n \mu^* A_n < \infty$, altrimenti non vi \e nulla da dimostrare.
Per ogni $\eps > 0$ ed $n \in \bN$ esiste una successione $\{ I^{(n)}_k \} \subset \mI^{aa}$ tale che
\begin{equation}
\label{eq_ME03a}
A_n \subseteq \cup_k I^{(n)}_k
\ \ , \ \
\mu^* A_n \geq \sum_k l( I^{(n)}_k ) - 2^{-n} \eps
\ .
\end{equation}
Inoltre $\cup_n A_n \subseteq \cup_n \cup_k I^{(n)}_k$, per cui, per definizione di $\mu^*$,
\begin{equation}
\label{eq_ME03.2}
\mu^* \left( \cup_n A_n \right)
\ \leq \
\sum_{n,k} l(I^{(n)}_k) 
\ \stackrel{(\ref{eq_ME03a})}{\leq} \
\sum_n \mu^* A_n + \eps
\ .
\end{equation}
\end{proof}

\begin{lem}[Regolarit\'a esterna]
\label{lem_ME03}
Sia $A \in 2^\bR$. Allora per ogni $\eps > 0$ esiste un aperto $U \subseteq \bR$, $U \supseteq A$, tale che $\mu^*U \leq \mu^*A + \eps$.
\end{lem}

\begin{proof}[Dimostrazione]
E' sufficiente considerare il caso $\mu^*A < \infty$, cosicch\'e per definizione esiste una successione $\{ I_n \} \subset \mI^{aa}$ tale che
$A \subseteq \cup_n I_n$ e $\sum_n l(I_n) \leq \mu^*A  +\eps$.
Posto $U := \cup_n I_n$, applicando la subadditivit\'a numerabile troviamo 
$\mu^*U \leq \sum_n l(I_n)$
e quindi otteniamo quanto volevasi dimostrare.
\end{proof}

\begin{rem}
\label{rem_ME04}
{\it
Dai risultati precedenti seguono banalmente le seguenti propriet\'a: 
\textbf{(1)} Se $\mu^* A = 0$ allora $\mu^*(A \cup B) = \mu^* B$ per ogni $B \in 2^\bR$; 
\textbf{(2)} Se $A$ \e numerabile allora $\mu^* A = 0$.
}
\end{rem}

Ora, grazie ai Lemmi precedenti sappiamo che $\mu^*$ \e una misura esterna nel senso di \S \ref{sec_MIS1} per cui, specializzando la nozione di misurabilit\'a di Carath\'eodory al caso $X = \bR$, diciamo che un insieme $E \subseteq \bR$ \e {\em misurabile secondo Lebesgue} se 
\[
\mu^* A = \mu^*(E \cap A) + \mu^*(E^c \cap A)
\ , \
\forall A \subseteq \bR
\ .
\]
Denotiamo con $\mL \subset 2^\bR$ la classe degli insiemi misurabili secondo Lebesgue.

\begin{thm}
\label{thm_ML02}
La classe $\mL$ \e una $\sigma$-algebra e contiene quella dei boreliani su $\bR$.
\end{thm}

\begin{proof}[Dimostrazione]
Il fatto che $\mL$ \e una $\sigma$-algebra segue dal Lemma \ref{lem_car}. 
Verifichiamo che $\mL$ contiene i boreliani; a tale scopo \e sufficiente mostrare che ogni intervallo del tipo $J := (a,\infty)$, $a \in \bR$, appartiene ad $\mL$. Grazie alla regolarit\'a esterna possiamo prendere una successione $\{ I_n \}$ di intervalli aperti tali che, preso $A \subseteq \bR$,
\[
A \subseteq \cup_n I_n
\ \ , \ \
\mu^*A \geq \sum_n l(I_n) - \eps \ .
\]
Posto 
$I_n^- := J^c \cap I_n$, 
$I_n^+ := J \cap I_n$,
abbiamo, grazie al Lemma \ref{lem_ME02},
$l(I_n) = l(I_n^-) + l(I_n^+)$, 
cosicch\'e
\[
\begin{array}{ll}
\mu^* (A \cap J) + \mu^* (A \cap J^c)            & \leq
\mu^* \left(A \cap \bigcup_n I_n^+ \right) +
\mu^* \left(A \cap \bigcup_n I_n^- \right)    \\ & \leq
\sum_n l(I_n^+)  + \sum_n l(I_n^-)            \\ & \leq
\sum_n l(I_n)                                 \\ & \leq
\mu^*A + \eps
\ ;
\end{array}
\]
per arbitrariet\'a di $\eps$ concludiamo $\mu^* (A \cap J) + \mu^* (A \cap J^c) \leq \mu^*A$ e quindi (essendo la diseguaglianza opposta sempre verificata) che $J$ \e misurabile.
\end{proof}

\begin{rem}
{\it
L'inclusione dei boreliani in $\mL$ \e stretta, per cui $\mL$ \e una $\sigma$-algebra piuttosto "grande" (vedi \cite[\S 2.2]{dBar} o \cite[Es.3.1.8]{Acq}). 
Tuttavia esistono insiemi non misurabili secondo Lebesgue, il pi\'u famoso dei quali, l'\textbf{insieme di Vitali}, \e il sottoinsieme $V \subset [0,1]$ che si costruisce come segue: innanzitutto consideriamo la proiezione di $[0,1]$ sull'insieme delle classi di equivalenza modulo i razionali,
\[
\pi : [0,1] \to [0,1]_\bQ
\ \ , \ \
\pi(x) := \{ y \in [0,1] : y-x \in \bQ \}
\ ,
\]
e quindi definiamo $V$ come l'immagine di una sezione
{\footnote{Una sezione dell'applicazione suriettiva $\pi : Y \to X$ \e un'applicazione iniettiva $s : X \to Y$ tale che $\pi \circ s$ \e l'identit\'a su $X$. L'esistenza di una sezione in generale non \e assicurata, a meno che non si invochi l'assioma della scelta.}}
di $\pi$ (dettagli su \cite[\S 3.4]{Roy} o \cite[\S 1.8]{Acq}).
}
\end{rem}

Definiamo ora la {\em misura di Lebesgue}
\[
\mu : \mL \to \wa \bR^+
\ \ , \ \
\mu E := \mu^*E 
\ \ , \ \
E \in \mL
\ .
\]

\begin{thm}
\label{thm_ML03}
$\mu$ \e una misura completa, boreliana, regolare esterna, invariante per traslazioni ed estende la funzione lunghezza definita sugli intervalli. 
\end{thm}

\begin{proof}[Dimostrazione]
Il fatto che $\mu$ \e una misura segue dal Lemma \ref{lem_car}. La regolarit\'a esterna, il fatto che $\mu$ estende la funzione lunghezza, e l'invarianza per traslazioni seguono dai risultati precedenti. Infine, la completezza segue sempre dal Lemma \ref{lem_car}.
\end{proof}

Il risultato seguente implica (essendo evidente che $\mu C < \infty$ per ogni compatto $C \subset \bR$) che la misura di Lebesgue soddisfa anche la propriet\'a di regolarit\'a interna, per cui essa \e una misura di Radon. Avendosi anche invarianza per traslazioni concludiamo che $\mu$ \e la misura di Haar su $\bR$ inteso come gruppo topologico rispetto all'operazione di somma.

\begin{cor}
\label{cor_ri}
Sia $E \in \mL$. Per ogni $\eps > 0$ esiste un chiuso $W_\eps \subseteq E$ tale che $\mu E - \mu W_\eps < \eps$.
\end{cor}

\begin{proof}[Dimostrazione]
Applicando il Lemma \ref{lem_ME03} ad $E^c$ troviamo che per ogni $\eps > 0$ esiste un aperto $U_\eps \supseteq E^c$ tale che $\mu U_\eps - \mu E^c < \eps$.
Del resto
$\mu U_\eps - \mu E^c = 
 \mu ( U_\eps - E^c ) =
 \mu ( E - U_\eps^c )$,
ed essendo $U_\eps^c$ chiuso troviamo quanto volevasi dimostrare.
\end{proof}

\begin{rem}[Misura di Lebesgue e misure di Hausdorff]
{\it
Confrontando la definizione di misura esterna di Lebesgue (\ref{eq_misEst}) con quella di Hausdorff (\ref{eq.mis.Haus}), ed osservando che i dischi nello spazio metrico $\bR$ sono proprio gli intervalli, concludiamo che la misura di Lebesgue \e la misura di Hausdorff su $\bR$ di dimensione $\delta = 1$.
}
\end{rem}

\begin{ex}[Insiemi di Cantor]
\label{ex_cantor}
{\it
Sia $\theta \in (0,1/3]$. Togliendo da $[0,1]$ l'intervallo aperto $I_1^1$ di centro $1/2$ e lunghezza $\theta$ otteniamo due intervalli chiusi $J_1^2$, $J_2^2$; da questi togliamo i due intervalli aperti $I_1^2$, $I_2^2$ di lunghezza $\theta^2$ centrati nei due punti di mezzo di $J_1^2$, $J_2^2$. Iterando il procedimento, al passo $k$-esimo toglieremo da $2^{k-1}$ intervalli chiusi i $2^{k-1}$ intervalli aperti di mezzo di lunghezza $\theta^k$. Nel limite $k \to \infty$ otteniamo l'insieme
\[
W_\theta := [0,1] - \cup_{k=1}^\infty \cup_{n=1}^{2^{k-1}} I_n^k
\ .
\]
Chiaramente $W_\theta$ \e chiuso e la sua misura \e facilmente calcolabile per additivit\'a numerabile,
\[
\mu W_\theta = 
1 - \sum_k \sum_n^{2^{k-1}} \theta^k =
1 - \frac{1}{2} \sum_k (2 \theta)^k =
\frac{1-3 \theta}{1 - 2 \theta}
\ ,
\]
ci\'o nonostante le ben note propriet\'a "mostruose" degli insiemi $W_\theta$ a livello topologico (vedi \cite[\S 12.9]{Roy}).
Per $\theta = 1/3$ otteniamo il classico insieme di Cantor, il quale ha quindi misura nulla. In realt\'a la misura pi\'u adatta per trattare $W_{1/3}$ \e la misura di Hausdorff di cui abbiamo gi\'a accennato; la dimensione di Hausdorff di $W_{1/3}$ risulta essere $\dim_H W_{1/3} = \ln 2 / \ln 3$ (vedi \cite[\S 2.6]{dBar}).
}
\end{ex}

\

\noindent \textbf{Cenni sul Lemma di Vitali.} Sia $\mI$ una collezione di intervalli. Preso $A \in 2^\bR$, diciamo che $\mI$ ricopre $A$ {\em nel senso di Vitali} se per ogni $\eps > 0$ ed $x \in A$ esiste $I \in \mI$ tale che $x \in I$ e $l(I) < \eps$.

\begin{lem}\textbf{(Vitali, \cite[Lemma 5.1.1]{Roy}).}
Sia $A \subset \bR$ un insieme di misura esterna finita ed $\mI$ una collezione che ricopre $A$ nel senso di Vitali. Allora, preso $\eps > 0$, esiste un insieme finito $\{ I_1 , \ldots , I_n \}$ di intervalli mutualmente disgiunti tale che
$\mu^* \left( A - \dot{\cup}_k^n I_k \right) < \eps$.
\end{lem}

\begin{proof}[Sketch della dimostrazione]
Possiamo assumere che ogni intervallo di $\mI$ sia chiuso (altrimenti passiamo alla chiusura). Preso un aperto $U$ di misura finita che contiene $A$, possiamo assumere che ogni $I \in \mI$ sia contenuto in $U$. Costruiamo ora per induzione l'insieme finito cercato, nel seguente modo: (1) Scegliamo un arbitrario $I_1 \in \mI$; (2) Assumendo di aver scelto i primi $k$ intervalli disgiunti $I_1 , \ldots , I_k$, definiamo
\[
\lambda(k) 
\ := \ 
\sup \ 
\{ 
   l(I) 
   \ : \  
   I \in \mI \ ,\  I \cap \left( \dot{\cup}_{j=1}^k I_j \right) = \emptyset  
\} 
\ .
\]
Poich\'e ogni $I \in \mI$ \e contenuto in $U$ abbiamo $\lambda(k) < \infty$. Ora, a meno che non sia $A \subseteq \dot{\cup}_j^k I_k$, possiamo trovare $I \in \mI$ disgiunto da $I_1 , \ldots , I_k$ e tale che $l(I) > 1/2 \lambda(k)$; per cui abbiamo una successione $\{ I_k \}$  di intervalli disgiunti tali che $\sum_k l(I_k) < \mu U < \infty$; (3) Scegliamo ora $\eps > 0$; poich\'e la serie $\sum_k l(I_k)$ \e convergente troviamo $n \in \bN$ tale che 
\[
\sum_{k=n+1}^\infty l(I_k) < \eps 
\ .
\]
Per cui il Lemma sar\'a dimostrato se verifichiamo che $\mu^* (A-\dot{\cup}_{k=1}^nI_k) < \eps$.
\end{proof}

Il Lemma di Vitali si pu\'o generalizzare ad $\bR^n$, $n > 1$ (Teorema di Besicovitch), e ad arbitrari spazi metrici. In tal caso il ruolo degli intervalli viene assunto da palle di raggio opportuno. Su questo argomento si veda \cite{Sul} e referenze.

\

\noindent \textbf{Funzioni misurabili sulla retta reale. Funzioni a gradini.} Passiamo ora a studiare le funzioni misurabili sulla retta reale. In particolare, introduciamo un'interessante classe di funzioni semplici, le {\em funzioni a gradini}
\[
\psi = \sum_k^n \lambda_k \chi_{I_k}
\ \ , \ \
\lambda_k \in \bR
\ \ , \ \
I_k := (a_k , b_k)
\ , \
a_k,b_k \in \bR
\ .
\]
Denotiamo con $\mG(A)$ l'insieme delle funzioni a gradini definite su un insieme misurabile $A \subseteq \bR${\footnote{Osservare che quando la parte interna $\dot{A}$ \e l'insieme vuoto, abbiamo $\mG(A) = \{ 0 \}$.}}. Osserviamo che $\mG(\bR)$ rispecchia propriet\'a topologiche di $\bR$; infatti esso \e definito per mezzo della base degli intervalli, a prescindere dalla struttura di $\bR$ inteso come spazio misurabile. Chiaramente $\mG(\bR)$ \e una sottoalgebra di $S(\bR)$, ed entrambe sono sottoalgebre di $M(\bR)$. Usando la regolarit\'a esterna di $\mu$, si dimostra il seguente
\begin{lem}
\label{lem_sim_stp}
Siano $a,b \in \bR$ e $\varphi \in S([a,b])$. Allora per ogni $\eps > 0$ esiste una funzione a gradini $\psi_\eps \in \mG([a,b])$ tale che $\mu \{ x : | \varphi - \psi_\eps  | \geq \eps \} < \eps$.
\end{lem}

\begin{proof}[Sketch della dimostrazione]
Osserviamo che $\varphi$ \`e, per definizione, limitata, per cui esiste $M \in \bR^+$ tale che $\| \varphi \|_\infty \leq M$. Scegliamo $\eps > 0$ e partizioniamo $[-M,M]$ in intervalli 
$J_n := [ \alpha_n-\eps/2  \ , \ \alpha_n+\eps/2 ]$, $n = 1 , \ldots , N$;
poich\'e ogni $\varphi^{-1}(J_n)$ \e misurabile esiste un aperto $U_n \supseteq \varphi^{-1}(J_n)$ tale che 
$U_n - \varphi^{-1}(J_n)$
ha misura minore di $\eps / N$. Ora, \e semplice verificare che $U_n$ pu\'o essere espresso come unione disgiunta di intervalli aperti, $U_n = \dot{\cup}_i I_i^n$, per cui definiamo
$\psi_\eps (x) := \alpha_n$, $x \in I_i^n$.
Ripetiamo quindi il procedimento al variare di $n = 1 , \ldots , N$.
\end{proof}

\begin{prop}
\label{prop_app_mis}
Sia $f \in M([a,b])$ tale che 
$\mu \{ x : f(x) = \pm \infty \} = 0$.
Allora per ogni $\eps > 0$ esistono una funzione a gradini $\varphi_\eps : [a,b] \to \bR$ ed una funzione continua, lineare a tratti $g_\eps \in C([a,b])$, tali che
\[
\mu \{ x : | f(x) - \varphi_\eps(x) | \geq \eps \} < \eps
\ \ , \ \
\mu \{ x : | f(x) - g_\eps(x) | \geq \eps \} < \eps
\ .
\]
Inoltre, se $m \leq f \leq M$, allora $g_\eps$, $\varphi_\eps$ possono essere scelte in maniera tale che $m \leq g_\eps , \varphi_\eps \leq M$.
\end{prop}

\begin{proof}[Dimostrazione]
Si usa un argomento di tipo "$3$-$\eps$". Osserviamo che avendo $f^{-1}(\pm \infty)$ misura nulla, per regolarit\'a esterna per ogni $\eps > 0$ esiste $M \in \bR$ tale che $|f| \leq M$ tranne che su un insieme di misura $\eps / 3$. 
Come primo passo, partizioniamo $[-M,M)$ in intervalli semiaperti a sinistra di lunghezza $< \eps$, $[-M,M) = \dot{\cup}_k I_k$, $I_k = [y_k,y_{k+1})$, $y_{k+1} - y_k < \eps$, ed introduciamo la funzione semplice
\[
\varphi_\eps (x) := 
\left\{
\begin{array}{ll}
y_k \ \ , \ \ x \in f^{-1}(I_k)
\\
0   \ \ , \ \ {\mathrm{se}} \ |f(x)| > M
\ ,
\end{array}
\right.
\]
cosicch\'e $| f - \varphi_\eps| < \eps$ tranne che sull'insieme dove $|f| > M$ (che ha misura $\eps / 3$). 
Come secondo passo, usando Lemma \ref{lem_sim_stp} otteniamo una funzione a gradini $\psi_\eps$ tale che 
$\mu \{ x : | \varphi_\eps - \psi_\eps  | \geq \eps \} < \eps / 3$.
Come terzo ed ultimo passo, osserviamo che l'insieme $\{ x_n \}$ dei punti di discontinuit\'a di $\psi_\eps$ \e al pi\'u numerabile, per cui esiste una successione 
$\{ J_n := [ x_n - \delta_n , x_n + \delta_n ] \}$ 
di intervalli tra loro disgiunti e tali che 
$\sum_n (2 \delta_n) < \eps / 3$.
Definiamo quindi la funzione continua
\[
g_\eps (x) := 
\left\{
\begin{array}{ll}
\psi_\eps (x)
\ \ , \ \ 
x \in [a,b] - \dot{\cup}_n J_n
\\
\psi_\eps(x_n-\delta_n) +  \frac{ \psi_\eps(x_n+\delta_n) - \psi_\eps(x_n-\delta_n)  }
                                {2 \delta_n} \
                           (x-x_n+\delta_n)
\ \ , \ \ 
x \in J_n
\ .
\end{array}
\right.
\]
Per costruzione
$\mu \{ x : | g_\eps - \psi_\eps  | \geq \eps \} < \eps / 3$.
Stimando 
$| f - g_\eps | \leq | f - \varphi_\eps |         + 
                     | \varphi_\eps - \psi_\eps | +
                     | g_\eps - \psi_\eps |$,
ed, analogamente, $|f - \psi_\eps |$, otteniamo il risultato cercato.
\end{proof}

\begin{thm}[Lusin]
\label{thm_lusin}
Sia $f$ una funzione misurabile su un intervallo $[a,b]$. Allora per ogni $\eps > 0$ esiste una funzione continua $\varphi_\eps$ su $[a,b]$ tale che
\[
\mu \{ x : f(x) \neq \varphi_\eps(x) \} < \eps
\ .
\]
\end{thm}

\begin{proof}[Sketch della dimostrazione.]
Usando Prop.\ref{prop_app_mis} costruiamo una successione $\{ f_n \}$ di funzioni continue tali che 
$\mu \{ x : |f(x)-f_n(x)| \geq 1/n \} < 1 / n$.
Ci\'o fornisce una convergenza q.o. ed a questo punto usiamo il teorema di Egoroff.
\end{proof}

Il seguente risultato fornisce una caratterizzazione delle funzioni integrabili secondo Riemann dal punto di vista della misura di Lebesgue; si noti come la nozione di continuit\'a giochi un ruolo importante.
\begin{prop}\textbf{(Lebesgue-Vitali, \cite[Es.6.4.2]{Giu2}).}
Una funzione limitata $f : [a,b] \to \bR$ \e integrabile secondo Riemann se e soltanto se l'insieme dei punti di discontinuit\'a di $f$ ha misura (di Lebesgue) nulla.
\end{prop}

\begin{ex}[La funzione di Dirichlet]
\label{ex_dir}
{\it
Sia $f : [0,1] \to \bR$ la funzione definita come $f (t) = 0$, $t \in \bQ \cap [0,1]$, $f(t) = 1$, $t \in [0,1] - \{ [0,1] \cap \bQ \}$. Usando Oss. \ref{rem_ME04}(2) concludiamo che, essendo $\bQ \cap [0,1]$ numerabile, $\mu (\bQ \cap [0,1]) = 0$, per cui $f$ coincide quasi ovunque con la funzione costante $1$. Essendo la misura di Lebesgue completa (Teo.\ref{thm_ML03}), usando Prop.\ref{prop_fgmis} concludiamo che $f$ \e misurabile. D'altro canto $f$ non \e integrabile secondo Riemann, infatti l'insieme dei punti di discontinuit\'a di $f$ \e $[0,1]$, come si verifica osservando che per ogni $t \in [0,1]$ possiamo trovare successioni 
$\{ q_n \} \subset [0,1] \cap \bQ$, $\{ i_n \} \subset [0,1] \cap (\bR-\bQ)$
con limite $t$, cosicch\'e 
$\lim_n f(q_n) = 0$, $\lim_n f(i_n) = 1$.
Ovviamente, la non integrabilit\'a secondo Riemann di $f$ si pu\'o verificare pi\'u direttamente mostrando che non vale la condizione (\ref{def.Rint}) nella sezione seguente.
}
\end{ex}

\subsection{L'integrale di Lebesgue.}

Sia $(X,\mM,\mu)$ uno spazio di misura, $\varphi := \sum_i a_i \chi_{A_i}$, $\mu A_i < + \infty$ una funzione semplice ed $A \subseteq X$ un insieme misurabile (non necessariamente di misura finita). Definiamo rispettivamente l'{\em integrale di} $\varphi$ e l'{\em integrale di $\varphi$ su $A$} 
\begin{equation}
\label{def_int_simple}
\int_X \varphi := \sum_i a_i \mu A_i
\ \ , \ \
\int_A \varphi := \int_X (\varphi \chi_A)
\ \ .
\end{equation}
Si verifica banalmente che la precedente definizione di integrale non dipende dalla particolare decomposizione di $\varphi$ come combinazione lineare di funzioni caratteristiche. Abbiamo cos\'i definito un funzionale lineare su $S(X)$. Nel seguito, quando sar\'a chiaro dal contesto, ometteremo il simbolo "$X$" dalla notazione di integrale.

Ora, vogliamo estendere l'applicazione (\ref{def_int_simple}) alle funzioni misurabili. Come primo passo ci occuperemo del caso in cui $X$ ha misura finita, e successivamente generalizzeremo a spazi (quasi) qualsiasi.

\begin{prop}
\label{prop_mis_simple}
Sia $(X,\mM,\mu)$ di misura finita ed $f : X \to \bR$ limitata. Se $f$ \e misurabile allora
\begin{equation}
\label{eq_mis_simple}
\int_\geq f
\ := \
\inf_{f \leq \psi \in S(X)} \int \psi 
\ = \
\sup_{f \geq \varphi \in S(X)} \int \varphi
\ =: \
\int_\leq f
\ .
\end{equation}
Se $(X,\mM,\mu)$ \e completo, allora \e vero anche il viceversa.
\end{prop}

\begin{proof}[Dimostrazione]
Per semplicit\'a di notazione poniamo $M := \| f \|_\infty$. Se $f$ \e misurabile, allora gli elementi (mutualmente disgiunti) della successione
\[
E_k 
:= 
\{ \ x \in X : \frac{(k-1)M}{n} < f(x) \leq \frac{kM}{n} \ \}
\ \ , \ \
-n \leq k \leq n
\ ,
\]
sono misurabili. Definendo le successioni di funzioni semplici
\[
\psi_n    := \frac{M}{n} \sum_{k=-n}^n   k    \chi_{E_k}
\ \ , \ \
\varphi_n :=  \frac{M}{n} \sum_{k=-n}^n (k-1) \chi_{E_k}
\]
troviamo $\varphi_n \leq f \leq \psi_n$, e 
\[
0 
\ \leq  \
\int_\geq f
-
\int_\leq f
\ \leq \
\int \psi_n
-
\int \varphi_n
\ \leq \
\frac{M}{n} \sum_{k=-n}^n \mu E_k = \frac{M}{n} \mu X
\stackrel{n}{\to} 0
\ .
\]
Ci\'o implica che $f$ soddisfa (\ref{eq_mis_simple}). 
Supponiamo ora che $(X,\mM,\mu)$ sia completo e che valga (\ref{eq_mis_simple}), e mostriamo che $f$ \e misurabile. A tale scopo, osserviamo che \e sufficiente verificare che $f = g$ q.o. per qualche funzione misurabile $g$ (vedi Prop.\ref{prop_fgmis}, ed \e qui che sfuttiamo la completezza). Consideriamo, per ogni $n \in \bN$, funzioni semplici $\varphi_n$, $\psi_n$ tali che 
\begin{equation}
\label{eq_pvn}
\varphi_n \leq f \leq \psi_n
\ \ , \ \
\int \psi_n - \int \varphi_n < n^{-1} \ .
\end{equation}
Allora, le funzioni 
$\unl \psi := \inf_n \psi_n$, $\ovl \varphi := \sup_n \varphi_n$
sono misurabili (Teo.\ref{thm_MIS}) e $\ovl \varphi \leq f \leq \unl \psi$. Ora, posto
$\Delta := \{ x \in X : \ovl \varphi < \unl \psi \}$ 
abbiamo, per ogni $n \in \bN$,
\[
\Delta \ \subseteq \ \cup_k \Delta_{k,n}
\ \ , \ \
\Delta_{k,n} := \{ x \in X : \varphi_n < \psi_n - k^{-1} \} 
\ .
\]
Ora, usando (\ref{eq_pvn}) troviamo 
\[
k^{-1} \mu \Delta_{k,n} < \int_{\Delta_{k,n}} (\psi_n-\varphi_n) \leq n^{-1}
\ \Rightarrow \
\mu \Delta_{k,n} < kn^{-1}
\ .
\]
Per cui (potendo scegliere $n$ arbitrariamente grande) concludiamo che $\mu \Delta = 0$. Dunque $\unl \psi$, $\ovl \varphi$, e quindi $f$, coincidono quasi ovunque. 
\end{proof}

In accordo alla proposizione precedente, definiamo l'{\em integrale di $f$} e l'{\em integrale di $f$ su $A \in \mM$} rispettivamente come
\begin{equation}
\label{eq_def_intE}
\int f
\ := \
\int_\geq f
\ = \
\int_\leq f
\ \ \ , \ \ \
\int_A f := \int_X f \chi_A
\ \ , \ \
A \in \mM
\ .
\end{equation}

\begin{rem}
\label{oss_muA0}
{\it
Le seguenti propriet\'a sono banali da verificare nel caso di funzioni semplici, per cui nel caso generale seguono applicando la definizione precedente:
\textbf{(1)} Se $A$ ha misura nulla ed $f \in M(X)$ allora $\int_A f = 0$;
\textbf{(2)} Se $f \geq 0$ allora $\int f \geq 0$. 
}
\end{rem}

\noindent \textbf{Riemann vs. Lebesgue.} Siano $a,b \in \bR$ ed $f : [a,b] \to \bR$. Allora $f$ \e integrabile secondo Riemann se e solo se
\begin{equation}
\label{def.Rint}
\int_\geq^R f
\ := \
\inf_{f \leq \psi \in \mG([a,b])} \int \psi 
\ = \
\sup_{f \geq \varphi \in \mG([a,b])} \int \varphi
\ =: \
\int_\leq^R f
\ ,
\end{equation}
ed in tal caso poniamo 
\[
\int^R f 
\ := \
\int_\leq^R f 
\ = \
\int_\geq^R f 
\ .
\]
La precedente definizione possiede un contenuto intuitivo, il quale evidenzia la filosofia di base dell'integrale di Riemann: se $f$ \e Riemann-integrabile, allora $\int^R f$ deve essere approssimabile a mezzo di integrali di funzioni a gradini, ed \e facilmente verificabile che, presa 
$\varphi := \sum_i a_i \chi_{(x_i,x_{i+1})} \leq f$,
l'errore
\[
\delta (\varphi) := \int^R f - \int \varphi 
\]
\e migliorabile raffinando la partizione $P := \{ x_i \}$, ovvero considerando 
$P' := \{ x'_j \} \supset P$
e definendo 
$\varphi' := \sum_j a'_j \chi_{(x'_j,x'_{j+1})}$
tale che
$f(x) \geq a'_j \geq a_i$, $\forall x \in (x'_j,x'_{j+1}) \subseteq (x_i,x_{i+1})$.
Infatti ci\'o garantisce che $0 \leq \delta(\varphi') \leq \delta(\varphi)$.
Dunque l'idea di base \e: {\em raffina le partizioni, e spera che al restringersi degli intervalli $f$ sia abbastanza regolare da poter essere ben approssimata, in ogni intervallo, da una costante}.

Equipaggiamo ora $[a,b]$ con la misura di Lebesgue, cosicch\'e 
$\mG([a,b]) \subset S([a,b])$; 
allora
\[
\{ \psi \in \mG([a,b]) : \psi \leq f \} \subset \{ \psi \in S([a,b]) : \psi \leq f \}
\ \ , \ \
\{ \psi \in \mG([a,b]) : \psi \geq f \} \subset \{ \psi \in S([a,b]) : \psi \geq f \}
\ ,
\]
per cui concludiamo che 
\[
\int_\leq^R f
\ \leq \
\int_\leq f
\leq
\int_\geq f
\ \leq \
\int_\geq^R f
\ .
\]
Dunque se $f$ \e integrabile secondo Riemann allora \e integrabile secondo Lebesgue, ed i due integrali coincidono. D'altra parte, dall'Esempio \ref{ex_dir} abbiamo una funzione integrabile secondo Lebesgue che non \e integrabile secondo Riemann.

\

\noindent \textbf{Misure non finite.} Sia $(X,\mM,\mu)$ uno spazio di misura non necessariamente finita. Iniziamo denotando con
\[
L^\infty_{\mu,0}(X)
\]
l'algebra delle funzioni misurabili su $X$, limitate, e con supporto a misura finita. Chiaramente, possiamo definire l'integrale di $f \in L^\infty_{\mu,0}(X)$ su $X$ semplicemente come l'integrale sul suo supporto, come dalla definizione precedente. Inoltre, \e chiaro che tale integrale \e finito: 
$| \int_X f | \leq \| f \|_\infty \mu( {\mathrm{supp}}(f) )  < + \infty$.
Denotiamo con 
$M^+(X)$
l'insieme delle funzioni {\em definite su $X$, misurabili e non negative}. Diciamo che $f \in M^+(X)$ \e {\em integrabile} se
\begin{equation}
\label{def_int}
\int_X f 
\ := \
\sup \left\{  \int_X h \ , \ h \in L^\infty_{\mu,0}(X) \ , \ h \leq f \right\}
\ < \
+ \infty
\ .
\end{equation}
Passiamo infine al caso generale $f \in M(X)$. Per prima cosa, osserviamo che
\[
f   = f^+ - f^-
\ , \
|f| = f^+ + f^-
\ \ , \ \
f^+ := \max \{ f,0 \} \geq 0
\ , \
f^- := - \min \{ f,0 \} \geq 0
\ .
\]
Per (\ref{eq_MIS_1}), $f$ \e misurabile se e soltanto se $f^+$, $f^-$ sono misurabili. Diciamo che {\em $f$ \e integrabile se sia $f^+$ che $f^-$ sono integrabili su $X$}, ed in tal caso, definiamo
\[
\int_X f := \int_X f^+ - \int_X f^-
\ .
\]
Osserviamo che, in base alla definizione precedente, se $f$ \e integrabile allora anche $|f|$ \e integrabile. Qualora si render\'a necessario specificare la misura $\mu$ rispetto alla quale stiamo integrando, scriveremo
\[
\int_X f \ d \mu := \int_X f \ .
\]
Denotiamo con $L_\mu^1(X)$ lo spazio vettoriale delle funzioni integrabili su $X$ {\footnote{Sul motivo di questa notazione rimandiamo a \S \ref{sec_Lp}.}}. Osserviamo che segue banalmente dalla definizione che
\begin{equation}
\label{eq_int}
f \leq g \ \Rightarrow \ \int f \leq \int g
\ \ , \ \
\int cf = c \int f
\ , \
\forall c \in \bR
\ ;
\end{equation}
per dimostrare invece che l'integrale conserva l'operazione di somma sar\'a conveniente usare i teoremi di passaggio al limite sotto il segni d'integrale (vedi sezione seguente).

\begin{prop}
\label{lem_sf}
Ogni $f \in L_\mu^1(X)$ ha supporto $\sigma$-finito; se $f \geq 0$, allora essa \e limite puntuale di una successione monot\'ona crescente $\{ \varphi_n \}$ di funzioni semplici aventi supporto con misura finita.
\end{prop}

\begin{proof}[Dimostrazione]
Per ogni $n \in \bN$ definiamo
$E_n := \{ x \in X : |f(x)| \geq n^{-1} \}$
ed osserviamo che
\[
n^{-1} \mu E_n = \int n^{-1} \chi_{E_n} \leq \int_{E_n} |f| < \infty \ ;
\]
cosicch\'e $\mu E_n < \infty$, e del resto ${\mathrm{supp}}(f) = \cup_n E_n$.
Riguardo la seconda affermazione, prendiamo la successione $\{ \psi_n \}$ di Prop.\ref{prop_caressa} e definiamo 
$A_n := \cup_i^nE_i$, 
$\varphi_n := \chi_{A_n} \psi_n$, 
cosicch\'e
$f(x) = \lim_n \varphi_n(x)$, $\forall x \in X$.
\end{proof}

\begin{rem}{\it
La nostra definizione di integrale funziona bene solo nel caso in cui $(X,\mM,\mu)$ abbia misura localmente finita
{\footnote{
Osserviamo comunque che gli spazi di misura che si incontrano di solito nella vita sono localmente finiti.
}}.
Ad esempio, si prenda un qualsiasi insieme $X$, $\mM := \{ X , \emptyset \}$ con $\mu X := + \infty$, $\mu \emptyset := 0$, cosicch\'e le funzioni misurabili sono le costanti. In base alla nostra definizione di integrale $\int f = 0$ per ogni $f \in M(X)$, mentre invece ci aspetteremmo, ad esempio, che $\int 1 = \mu X = + \infty$.
Questa difficolt\'a \e aggirabile definendo $\int f$ come l'estremo superiore degli integrali di generiche funzioni semplici $\leq f$, non necessariamente con supporto di misura finita. In conseguenza di ci\'o andrebbe modificata la dimostrazione del Lemma di Fatou rispetto a quella che daremo nella prossima sezione (vedi \cite[\S 11.3]{Roy}). Gli altri risultati non necessitano modifiche. 
}
\end{rem}

\noindent \textbf{Misure con segno.} Se $(X,\mM,\mu)$ \e uno spazio di misura con segno allora, grazie alla decomposizione di Jordan, esistono uniche le misure $\mu^\pm$ tali che 
$\mu = \mu^+ - \mu^-$.
Presa una funzione misurabile $f \in M(X)$ diciamo che essa \e $\mu$-integrabile se $f \in L_{\mu^+}^1(X) \cap L_{\mu^-}^1(X)$, ed in tal caso poniamo
\[
\int f \ d \mu 
\ := \ 
\int f \ d \mu^+ - \int f \ d \mu^- 
\ .
\]
Denotiamo con $L_\mu^1(X)$ lo spazio delle funzioni $\mu$-integrabili.

\

\noindent \textbf{Il caso complesso.} Sia $(X,\mN,\nu)$ uno spazio di misura (eventualmente con segno). Ogni funzione $f : X \to \bC$ si pu\'o scrivere come $f=f_1+if_2$, dove 
$f_1 , f_2 : X \to \bR$.
Diciamo che $f$ \e integrabile se $f_1 , f_2 \in L_\nu^1(X)$; in tal caso, poniamo
\[
\int f \ d \nu 
\ := \
\int f_1 \ d \nu + i \int f_2 \ d \nu 
\ ,
\]
e scriviamo $f \in L_\nu^1(X,\bC)$. Ora, se $(X,\mM,\mu)$ \e uno spazio di misura complesso allora per definizione $\mu = \mu_1 + i \mu_2$, con $\mu_1 , \mu_2$ misure con segno. Presa $f : X \to \bC$, diciamo che essa \e $\mu$-integrabile se 
$f \in L_{\mu_1}^1(X,\bC) \cap L_{\mu_2}^1(X,\bC)$;
in tal caso, poniamo
\[
\int f \ d \mu 
\ := \
\int f \ d \mu_1 + i \int f_2 \ d \mu_2 
\ ,
\]
e scriviamo $f \in L_\mu^1(X,\bC)$.

\subsection{Limiti sotto il segno d'integrale.}

Nelle pagine che seguono approcceremo la questione del passaggio al limite sotto il segno di integrale, premettendo che questa operazione {\em non sempre funziona} (si veda Esempio \ref{ex_lebesgue} alla fine della sezione). D'altro canto, se $\{ f_n \}$ \e una successione uniformemente convergente ad una funzione $f$, allora una semplice stima mostra come gi\'a nell'ambito dell'integrale di Riemann risulti $\int f_n \to \int f$. La teoria di Lebesgue migliora tale risultato: vedremo infatti come sia sufficiente, sotto opportune ipotesi, una convergenza q.o..

\begin{prop}[Convergenza limitata]
\label{prop_conv_lim}
Sia $(X,\mM,\mu)$ uno spazio di misura finita ed $\{ f_n \} \subset M(X)$ una successione limitata rispetto alla norma dell'estremo superiore. Se q.o. in $x \in X$ esiste $\lim_n f_n(x) =:$ $f(x)$, allora $f$ \e misurabile e 
\[
\int_X f 
\ = \
\lim_n \int_X f_n
\ .
\]
\end{prop}

\begin{proof}[Dimostrazione]
La funzione $f$ \e misurabile grazie a Teo.\ref{thm_MIS}. Sia $M > 0$ tale che $\| f_n \|_\infty < M$, $n \in \bN$. Per il teorema di Egoroff (Teo.\ref{thm_egoroff}) scelto $\eps > 0$ esiste un insieme misurabile $A$ con 
\[
\mu A < \frac{\eps}{4M}
\]
ed $n_0 \in \bN$ tale che $\| (f_n-f)|_{X-A} \|_\infty <$ $\eps / (2\mu X)$ per ogni $n > n_0$. Per cui
\[
\left| \int_X f_n - \int_X f \right|
\leq
\int_X | f_n - f |
=
\int_{X-A} | f_n - f | + \int_A | f_n - f |
\leq
\mu X \frac{\eps}{2 \mu X}  +  2M \frac{\eps}{4M}
\ .
\]
\end{proof}

\begin{thm}[Lemma di Fatou]
\label{thm_fatou}
Sia $(X,\mM,\mu)$ uno spazio di misura ed $\{ f_n \} \subset M^+(X)$ una successione di funzioni convergente q.o. in $x \in X$ ad una funzione $f$. Allora $f \in M^+(X)$, e
\begin{equation}
\label{eq_fatou}
\int_X f 
\ \leq \
\lim_n \inf \int_X f_n
\ .
\end{equation}
\end{thm}

\begin{proof}[Dimostrazione]
L'idea della dimostrazione consiste nel ridursi al caso della convergenza limitata, nel seguente modo: consideriamo $h \in L^\infty_{\mu,0}(X)$ tale che $f \geq h \geq 0$. Posto $h_n :=$ $\min \{ f_n , h \}$, osserviamo che deve essere $h(x) = \lim_n h_n(x)$ q.o. per $x \in {\mathrm{supp}}(h)$ (comunque, al di fuori di ${\mathrm{supp}}(h)$ si trova $h \equiv h_n = 0$). Per convergenza limitata (Prop.\ref{prop_conv_lim}) si trova
\[
\int_X h \ = \ 
\int_{{\mathrm{supp}}(h)} h \ = \ 
\lim_n \int_{{\mathrm{supp}}(h)} h_n \ \leq \
\lim_n \inf \int_X f_n
\ .
\]
Passando all'estremo superiore per $h \leq f$, $h \in L^\infty_{\mu,0}(X)$, per definizione di integrale su un insieme di misura non necessariamente finita si trova quanto volevasi dimostrare.
\end{proof}

\begin{thm}[Teorema di convergenza monot\'ona, Beppo Levi]
\label{thm_conv_mon}
Sia $\{ f_n \} \subset M^+(X)$ una successione crescente di funzioni non negative tale che esista il limite $f (x) :=$ $\lim_n f_n (x)$, q.o. in $x \in X$. Allora
\begin{equation}
\label{eq_conv_mon}
\int_X f 
\ = \
\lim_n \int_X f_n
\ .
\end{equation}
\end{thm}

\begin{proof}[Dimostrazione]
Il Lemma di Fatou ci assicura che $\int f \leq \lim_n \inf \int f_n$. D'altro canto, per monoton\'ia la successione $\int f_n$ ammette limite $\lim_n \int f_n =$ $\lim_n \inf \int f_n$, per cui $\int f \leq \lim_n \int f_n$. Ora, sempre per monoton\'ia troviamo $f_n \leq f$, $n \in \bN$, per cui $\lim_n \int f_n \leq \int f$.
\end{proof}

\begin{rem}{\it 
Il teorema di Beppo Levi non \e valido nei seguenti casi: \textbf{(1)} nell'ambito dell'integrale di Riemann (si veda l'Esempio \ref{ex_RL} pi\'u avanti); \textbf{(2)} per successioni decrescenti (\cite[Ex.4.7(b)]{Roy}).
} \end{rem}

Grazie al teorema precedente possiamo dimostrare alcune propriet\'a di base dell'integrale:
\begin{cor}
\label{cor_int}
Valgono le seguenti propriet\'a:
\textbf{(1)} Se $f,g \in M^+(X)$ allora $\int (f+g) = \int f + \int g$;
\textbf{(2)} Se $f \in M^+(X)$ allora $\int f = 0$ se e solo se $f = 0$ q.o.;
\textbf{(3)} Se $f,g \in L_\mu^1(X)$ e $a \in \bR$ allora $\int (af +g) = a \int f + \int g$;
\textbf{(4)} Se $f,g \in L_\mu^1(X)$ e $f \geq g$ ($f=g$) q.o. allora $\int f \geq \int g$ ($\int f = \int g$).
\end{cor}

\begin{proof}[Dimostrazione]
\textbf{(1)} L'affermazione \e ovvia per funzioni semplici. Del resto $f,g$ sono limite puntuale di successioni monot\'one $\{ \psi_n \}$, $\{ \varphi_n \}$ di funzioni semplici (Prop.\ref{prop_caressa}), per cui applicando il teorema di Beppo Levi troviamo
\[
\int (f + g) =
\lim_n \int (\psi_n + \varphi_n) =
\lim_n \left( \int \psi_n + \int \varphi_n \right) =
\int f + \int g
\ .
\]
\textbf{(2)} Se $f = 0$ q.o. allora $\int f = 0$ (vedi Oss.\ref{oss_muA0}(1)); viceversa, se $\int f = 0$ allora, presa la successione $\{ E_n \} \subset \mM$ di Prop.\ref{lem_sf}, abbiamo $0 = \int f \geq n^{-1} \mu E_n$. Quindi ${\mathrm{supp}}(f) = \cup_n E_n$ ha misura nulla.
\textbf{(3)} Basta applicare il punto (1) e (\ref{eq_int}).
\textbf{(4)} Grazie al punto (1) si ha $\int f - \int g = \int (f-g)$ con $f-g \geq 0$ ($f-g=0$) q.o.. Dunque l'asserto per $f=g$ segue dal punto (2), mentre 
\[
\int (f-g) = 
\int_A(f-g) + 
\int_B(f-g) \stackrel{Oss.\ref{oss_muA0}(1)}{=}
\int_A(f-g) \stackrel{Oss.\ref{oss_muA0}(2)}{\geq} 0
\ ,
\]
avendo posto 
$A := \{ x : f(x)-g(x) \geq 0 \}$,
$B := \{ x : f(x)-g(x) <    0 \}$
con $\mu B = 0$.
\end{proof}

\begin{ex}(Riemann vs. Lebesgue in convergenza puntuale).
\label{ex_RL}
{\it
Consideriamo la successione di funzioni
\[
f_n : [0,1] \to \bR
\ \ , \ \
f_n(x) :=
\left\{
\begin{array}{ll}
1 \ \ , \ \ x \in \{ q_1 , \ldots , q_n \}
\\
0 \ \ , \ \ altrimenti
\end{array}
\right.
\ \ , \ \
n \in \bN
\ ,
\]
dove $\{ q_n , n \in \bN \}$ \e una enumerazione dei razionali in $[0,1]$ (ovvero, una corrispondenza 1-1 tra $\bN$ e $\bQ \cap [0,1]$).
Ogni $f_n$ \e continua in $[0,1] - \{ q_1 , \ldots , q_n \}$, per cui \e integrabile secondo Riemann con
\[
\int^R f_n = 0 \ .
\]
Ora, $\{ f_n \}$ converge puntualmente a
\[
f(x) =
\left\{
\begin{array}{ll}
1 \ \ , \ \ x \in \bQ \cap [0,1]
\\
0 \ \ , \ \ altrimenti \ ,
\end{array}
\right.
\]
la quale non \e integrabile secondo Riemann, come si verifica in maniera analoga all'Esempio \ref{ex_dir}.
Equipaggiamo ora $[0,1]$ con la misura di Lebesgue. Abbiamo che $\{ f_n \}$ \e monot\'ona crescente, limitata e definita su uno spazio di misura finita, per cui possiamo applicare sia il teorema di convergenza limitata che quello di Beppo Levi, concludendo che $f$ \e integrabile secondo Lebesgue con
\[
0 = \lim_n \int f_n = \int \lim_n f_n = \int f \ .
\]
}
\end{ex}

\begin{cor}
\label{lem_ic}
Sia $f \in L_\mu^1(X)$, $f \geq 0$. Allora per ogni $\eps > 0$ esiste un $\delta > 0$ tale che 
\begin{equation}
\label{eq_LPB1}
\int_A f < \eps 
\ \ , \ \
\forall A \ : \ \mu A < \delta
\ .
\end{equation}
\end{cor}

\begin{proof}[Dimostrazione]
Definiamo 
$f_n (x) := \inf \{ f (x) , n \}$, $n \in \bN$, 
cosicch\'e la successione $\{ f_n \}$ converge puntualmente ad $f$ ed \e monot\'ona crescente. Dal teorema di convergenza monot\'ona segue che preso $\eps > 0$, esiste $n_0 \in \bN$ tale che $\int_X f_n >$ $\int_X f - \eps$ per ogni $n \geq n_0$. D'altra parte, scelti $A$, $\delta$ tali che
\[
\mu A < \delta <  \frac{\eps}{n_0}
\ ,
\]
ed usando il fatto che $| f_{n_0}(x) | \leq n_0$, $x \in X$, troviamo
\[
\int_A f \leq
\int_A (f-f_{n_0}) + \int_A f_{n_0} \leq
\eps + \eps
\ .
\]
\end{proof}

Applicando il corollario precedente ad $X = [a,b]$, $a,b \in \bR$, troviamo:
\begin{cor}
\label{prop_ic}
Sia $f : [a,b] \to \bR$ una funzione integrabile. Allora $F (x) :=$ $\int_a^x f(t) \ dt$, $x \in [a,b]$, \e continua.
\end{cor}

\begin{thm}[Teorema di convergenza di Lebesgue]
\label{teo_lebesgue1}
Sia $\{ f_n \}$ una successione di funzioni misurabili con limite 
$f(x) := \lim_n f_n(x)$ q.o. in $x \in X$. 
Se esiste $g \in L_\mu^1(X)$ tale che
$|f_n| \leq g$, $\forall n \in \bN$,
allora
\begin{equation}
\label{eq_lebesgue1}
\int_X f 
\ = \ 
\lim_n \int_X f_n 
\ .
\end{equation}
\end{thm}
\begin{proof}[Dimostrazione]
Innanzitutto osserviamo che essendo $|f_n| \leq g$, troviamo che $f$, $f_n$, $n \in \bN$, sono integrabili, per cui esistono $\int_X f_n$, $\int_X f < + \infty$. Consideriamo la successione $\{ g - f_n \}$, con $g - f_n \geq 0$. Per il Lemma di Fatou (Teo.\ref{thm_fatou}) abbiamo
\[
\int_X (g-f) \leq \lim_n \inf \int_X (g-f_n) \ . 
\]
Del resto
\[
\lim_n \inf \int_X (g-f_n) 
=
\lim_n \inf \left( \int_X g - \int_X f_n \right) 
= 
\int_X g - \lim_n \sup \int_X f_n 
\]
e quindi
\[
- \int_X f \leq - \lim_n \sup \int_X f_n 
\ \Leftrightarrow \
\lim_n \sup \int_X f_n \leq \int_X f 
\ .
\]
Applicando lo stesso argomento alla successione $\{ g + f_n \}$ otteniamo la diseguaglianza
\[
\lim_n \inf \int_X f_n \geq \int_X f \ ,
\]
per cui il teorema \e dimostrato.
\end{proof}


\begin{ex}\textbf{(vedi \cite{Giu2}).}
{\it 
\label{ex_lebesgue}
Per ogni $\lambda \in \bR$, consideriamo la successione $\{ f_n \} \subset C([0,1])$ definita da
\[
f_n (x) := n^\lambda x e^{-nx}
\ \ , \ \
x \in [0,1]
\ .
\]
Ora, $f_n$ converge puntualmente alla funzione nulla. Semplici verifiche mostrano che la convergenza \e anche uniforme se e soltanto se $\lambda < 1$ (infatti, $\| f_n \|_\infty =$ $n^{\lambda-1}e^{-1}$).
Restringiamo ora il campo delle nostre argomentazioni al caso $\lambda \geq 0$, e studiamo il passaggio al limite sotto il segno di integrale. Un calcolo esplicito mostra che
\[
\int_0^1 f_n(x) \ dx 
\ = \
n^{\lambda-2} [ 1 - (n+1) e^{-n} ]
\ ,
\]
per cui, {\em se $\lambda < 2$} otteniamo $\lim_n \int_n f_n = 0$, il che \e quanto possiamo aspettarci avendosi $\lim_n f_n = 0$ puntualmente. Tuttavia, {\em per $\lambda \geq 2$} otteniamo $\lim_n \int f_n \neq 0$. Dal punto di vista dei teoremi di convergenza sotto il segno di integrale, osserviamo i seguenti fatti:
\textbf{(1)} Per $\lambda \in [0,1)$ abbiamo $f_n \to 0$ uniformemente, per cui abbiamo convergenza sotto il segno di integrale senza necessit\'a di usare i teoremi di Lebesgue;
\textbf{(2)} Per $\lambda \in [1,2)$ abbiamo la stima
\[
f_n(x) := n^\lambda x e^{-nx}  
\ \leq \ 
g(x) := \lambda^\lambda e^{-\lambda} x^{1-\lambda} 
\ ;
\]
per cui
{\footnote{La stima precedente si ottiene con il seguente trucco: si definisca 
$F(t) :=$ $t^\lambda e^{-tx}$, $t \in \bR$,
e si osservi che: (1)$\max F =$ $F(\lambda / x)$; (2) $f_n(x) = x F(n) \leq$ $x F(\lambda / x) =$ $g(x)$.}}, 
essendo $g$ integrabile, possiamo passare al limite sotto il segno di integrale grazie al teorema di Lebesgue, il quale in questo caso \e indispensabile non avendosi convergenza uniforme.
\textbf{(3)} Per $\lambda \geq 2$, chiaramente il teorema di Lebesgue non \e valido.
}\end{ex}

\

\noindent \textbf{Due applicazioni del teorema di Lebesgue.} Consideriamo la retta reale equipaggiata con la misura di Lebesgue $\mu$ ed una successione $\{ f_n \} \subset L_\mu^1(\bR)$ con
\[
f(x) \ := \ \sum_n f_n(x)
\ \ , \ \
{\mathrm{q.o. \ in}} \ x \in \bR
\ ;
\]
supponiamo che esista $g \in L_\mu^1(\bR)$ tale che
\[
\left| \sum_n^m f_n(x) \right| \ \leq \ g(x)
\ \ , \ \
{\mathrm{q.o. \ in}} \ x \in \bR
\ \ , \ \
\forall m \in \bN
\ .
\]
Allora il teorema di convergenza dominata ci dice che
\[
\int f \ = \ \sum_n \int f_n \ ,
\]
il che fornisce un metodo di calcolo dell'integrale di $f$ nel caso in cui siano facilmente calcolabili quelli delle funzioni $f_n$, $n \in \bN$. Ci\'o accade tipicamente quando $\sum_n f_n$ \e uno sviluppo in serie di Taylor ($f_n(x) = c_n x^n$) o di Fourier ($f_n(x) = a_n \sin nx + b_n \cos nx$); per esempi si veda \cite[\S 3.3]{dBar},\cite[Es.6.4.4]{Giu2}.

Sempre in conseguenza del teorema di Lebesgue, si ottiene il seguente risultato concernente la derivazione sotto il segno d'integrale:
\begin{thm}
\label{thm_der_int}
Sia $(X,\mM,\mu)$ uno spazio di misura, $U \subseteq \bR^k$ aperto ed $f : X \times U \to \bR$ tale che:
\textbf{(1)} $f(\cdot,t) \in L_\mu^1(X)$ per ogni $t \in U$;
\textbf{(2)} $f(x,\cdot) \in C^1(U)$ per ogni $x \in X$;
\textbf{(3)} Esistono $g_1 , \ldots , g_k \in L_\mu^1(X)$ tali che 
\[
\left |\frac{\partial f}{\partial t_i} (x,t) \right| \ \leq \ g_i(x)
\ \ , \ \
\forall x \in X \ , \ i = 1 , \ldots , k \ .
\]
Allora $F(t) := \int_X f(\cdot,t)$, $t \in U$, \e di classe $C^1$ in $U$ e 
\[
\frac{\partial F}{\partial t_i}(t) \ = \ \int_X \frac{\partial f}{\partial t_i}(x,t) \ dx \ .
\]
\end{thm}

\begin{proof}[Dimostrazione]
Tenendo fisse le variabili $t_j$, $j \neq i$, possiamo riguardare $F$ come funzione della sola variabile $t_i$ e limitarci a dimostrare il teorema per $k = 1$.
Presi $t,t' \in U$, usando il teorema di Lagrange stimiamo il rapporto incrementale
\[
\left| \frac{ F(t)-F(t') }{t-t'} \right|
\ \leq \
\int_X \left| \frac{ f(x,t)-f(x,t') }{t-t'} \right| \ d \mu
\ = \
\int_X \left| \frac{\partial f}{\partial t_i}(x,\xi) \right| \ d \mu
\ \leq \
\int_X g_i
\ \ , \ \
t \leq \xi \leq t'
\ ,
\]
cosicch\'e possiamo applicare il teorema di convergenza dominata e passare al limite $t \to t'$ sotto il segno d'integrale.
\end{proof}

\subsection{Il teorema di Radon-Nikodym.}
\label{sec_RN}

Sia $\mM$ una $\sigma$-algebra definita su un insieme $X$. Date due misure $\mu , \nu : \mM \to \wa \bR$, diciamo che $\nu$ \e {\em assolutamente continua rispetto a} $\mu$ se $\mu A = 0$ implica $\nu A = 0$ per ogni $A \in \mM$. In tal caso, scriviamo $\nu \prec \mu$.
Se $\mu , \nu : \mM \to \wa \bR$ sono misure con segno, allora diciamo che $\nu$ \e {\em assolutamente continua} rispetto a $\mu$ se $|\mu|A = 0$ implica $\nu A = 0$ per ogni $A \in \mM$.

\begin{ex} \label{ex_dirac}
{\it
\textbf{(1)} Sia $(X,\mM,\mu)$ uno spazio misurabile ed $f \in M^+(X)$. Definendo $\nu A := \int_A f \ d \mu$, $A \in \mM$, otteniamo una misura assolutamente continua rispetto a $\mu$ (vedi Oss.\ref{oss_muA0});
\textbf{(2)} Sia $(X,\mM,\mu)$ uno spazio misurabile tale che $\{ y \} \in \mM$ e $\mu\{ y \} = 0$, $\forall y \in X$. Preso $x \in X$, consideriamo la misura di Dirac $\mu_x : \mM \to \bR$ (vedi Esempio \ref{ex_MIS_Dirac}). Visto che $\mu X = \mu ( X - \{ y \} )$ per ogni $y \in X$, troviamo $\mu A = 0$ per ogni insieme finito $A$. D'altra parte $\mu_x A = 1$ per ogni $A \ni x$, per cui $\mu_x$ \textbf{non} \e assolutamente continua rispetto a $\mu$.
}
\end{ex}

\begin{thm}[Radon-Nikodym]
\label{thm_RN}
Sia $( X,\mM,\mu )$ uno spazio misurabile e $\nu \prec \mu$, con $\nu,\mu$ $\sigma$-finite. Allora esiste ed \e unica $f \in M^+(X)$ tale che
\begin{equation}
\label{eq_RN}
\nu A 
\ = \
\int_A f \ d \mu
\ \ , \ \
A \in \mM
\ .
\end{equation}
Lo stesso risultato vale nel caso in cui $\mu , \nu$ siano misure con segno (con $f \in M(X)$).
\end{thm}

\begin{proof}[Dimostrazione]
Iniziamo assumendo che $X$ \textbf{abbia misura finita}. Definiamo
\[
C := 
\{
g \in M^+(X)
\ : \
\int g \ d \mu  \leq \nu E
\ , \
\forall E \in \mM
\}
\ \ , \ \
\alpha := \sup_{g \in C} \int g
\ .
\]
Poich\'e $\mu X < \infty$ abbiamo che $\alpha < \infty$; inoltre, per definizione di estremo superiore, esiste una successione $g_n \in C$ tale che $\alpha = \lim_n \int g_n d \mu$. Definiamo
\[
f_n := \sup \{ g_1 , \ldots , g_n \}
\ \Rightarrow \
0 \leq f_n \nearrow f
\ ;
\]
grazie a Teo.\ref{thm_MIS} abbiamo $f \in M^+(X)$. Ora, per ogni $E \in \mM$ troviamo facilmente
\[
E = \dot{\cup}_{i=1}^n E_i
\ \ : \ \
f_n |_{E_i} = g_i
\ \ , \ \
i = 1 , \ldots , n
\ .
\]
Per cui 
\[
\int_E f_n \ d \mu 
\ = \ 
\sum_i^n \int_{E_i} g_i \ d \mu
\ \leq \ 
\sum_i^n \nu E_i 
\ = \ 
\nu E
\]
(ovvero $f_n \in C$), e per convergenza monot\'ona troviamo
\[
\int_E f \ d \mu 
\ = \
\lim_n \int_E f_n \ d \mu
\ \leq \
\nu E
\ \ , \ \
E \in \mM
\ .
\]
Dunque $f \in C$ e $\mu \{ x : |f(x)| = + \infty \} = 0$. Inoltre troviamo
\[
\alpha 
\ \geq \ 
\int f \ d \mu
\ = \
\lim_n \int f_n \ d \mu
\ \geq \
\lim_n \int g_n \ d \mu
\ = \
\alpha
\ \Rightarrow \
\int f \ d \mu = \alpha
\ .
\]
Definiamo ora la misura (non negativa)
\[
\nu_0 E := \nu E - \int_E f \ d \mu
\ \ , \ \
E \in \mM
\ ,
\]
e mostriamo che $\nu_0 \equiv 0$. Se per assurdo esistesse $A \in \mM$ con $\nu_0 A > 0$, allora troveremmo, grazie alla decomposizione di Hahn, un $B \subseteq A$ tale che 
\begin{equation}
\label{eq_RN01}
(\nu_0 - \eps \mu) B 
\ = \
\nu B - \int_B (f + \eps 1) \ d \mu 
\ > \ 
0
\ .
\end{equation}
Poniamo ora $\wt f := f + \eps \chi_B$ ed osserviamo che
\[
\int_E \wt f \ d \mu
= 
\int_E f \ d \mu + \eps \mu (E \cap B)
= 
\int_{E-B} f \ d \mu + \int_B (f + \eps 1) \ d \mu
\ \stackrel{ (\ref{eq_RN01}) }{ < } \
\nu (E-B) + \nu B
=
\nu E
\ .
\]
La diseguaglianza precedente implica che $\wt f \in C$. D'altra parte, 
\[
\int \wt f \ d \mu = \int f \ d \mu + \eps \mu B > \alpha \ ,
\]
e quindi $\wt f \notin C$. Ci\'o \e assurdo e concludiamo che $\nu_0 = 0$. L'unicit\'a di $f$ si dimostra in modo banale per assurdo, e ci\'o mostra il teorema per $X$ a misura finita.
La generalizzazione al caso in cui $X$ \e \textbf{$\sigma$-finito} si effettua decomponendo $X = \dot{\cup}_n X_n$, $\mu X_n < \infty$, ed applicando il risultato nel caso finito ad ogni $X_n$. Ci\'o produce funzioni $f_n \in M^+(X)$, $n \in \bN$, tali che $\nu (E \cap X_n) = \int_E f_n \ d \mu$, $E \in \mM$. Ponendo $f := \sum_n f_n$ otteniamo, per convergenza monot\'ona, che $f$ soddisfa la propriet\'a desiderata. 
La generalizzazione al caso in cui $\mu , \nu$ sono misure con segno si ottiene applicando la decomposizione di Hahn (vedi \cite[\S 8.3]{dBar}).
\end{proof}

\subsection{Funzioni BV ed AC.}
\label{sec_AC_BV}

In questa sezione consideriamo funzioni sullo spazio di misura $[a,b]$, $a,b \in \bR$, equipaggiato della misura di Lebesgue.

Lo studio delle funzioni {\em assolutamente continue} (AC) ed a {\em variazione limitata} (BV) \e motivato dalla questione della formulazione del teorema fondamentale del calcolo nell'ambito dell'integrale di Lebesgue. In particolare, attraverso tali classi di funzioni determineremo l'immagine dell'applicazione 
\begin{equation}
\label{def_int_ab}
L^1([a,b]) \to C([a,b])
\ \ , \ \
f \mapsto F \ : \ F(x) := \int_a^x f
\ ,
\end{equation}
che associa ad $f \in L^1([a,b])$ la sua primitiva (osservare che gi\'a sappiamo che $F$ \e continua, vedi Cor.\ref{prop_ic}).

Iniziamo approcciando la questione della derivabilit\'a. Sia $f : [a,b] \to \bR$ una funzione. Per ogni $x \in (a,b)$, definiamo le {\em quattro derivate}
\[
D^\pm f (x) := \lim_{h \to 0 \pm} \sup  \frac{ f(x+h)-f(x) }{h}
\ \ , \ \
D_\pm f (x) := \lim_{h \to 0 \pm} \inf  \frac{ f(x+h)-f(x) }{h}
\ .
\]
e diciamo che $f$ {\em \e derivabile in} $x$ se $D^+ f (x) =$ $D^- f (x) =$ $D_+ f (x) =$ $D_- f (x) \neq$ $\pm \infty$. {\em In tal caso, definiamo} $f'(x) :=$ $D^+ f (x)$.

Osserviamo che se $f \in C([a,b])$ e una delle quattro derivate \e non negativa in $(a,b)$, allora $f$ \e monot\'ona crescente. D'altra parte, abbiamo il seguente notevole risultato.

%
\begin{thm}[Lebesgue]
\label{thm_BV1}
Sia $f : [a,b] \to \bR$ monot\'ona. Allora $f$ \e derivabile q.o. in $[a,b]$. Inoltre, $f'$ \e misurabile, e
\begin{equation}
\label{eq_BV1}
\int_a^b f'(x) \ dx \ \leq \ f(b) - f(a) \ .
\end{equation}
\end{thm}

\begin{proof}[Dimostrazione.]
Per fissare le idee assumiamo che $f$ sia non decrescente e consideriamo le derivate $D^+f$ e $D_-f$. L'obiettivo \e quello di mostrare che la misura di $E := \{ x : D^+f(x) \neq D_-f(x) \}$ \e nulla. A tale scopo definiamo, per ogni $u < v \in \bQ$, $E_{uv} := \{ x : D^+f(x) > v > u > D_-f(x)  \}$, ed osserviamo che $E = \cup_{u,v}E_{uv}$. Mostriamo dunque che ogni $E_{uv}$ ha misura nulla (il che implica che $E$, essendo unione numerabile degli $E_{uv}$, ha misura nulla). 
Innanzitutto osserviamo che $E_{uv}$, essendo contenuto in $[a,b]$, ha certamente misura esterna finita. Applicando il Lemma di Vitali troviamo una collezione di intervalli disgiunti $\{ I_1 , \ldots , I_N \}$, $I_n := ( x_n - h_n , x_n )$, tale che
\[
A' := \dot{\cup}_n I_n \subseteq E_{uv} 
\ \ , \ \ 
\mu A' > \mu^*E_{uv} - \eps 
\ .
\]
Ora, per ogni $x \in A'$ esiste un intervallo $[x-h,x]$ tale che
\[
f(x) - f(x-h) < uh  
\ \Rightarrow \ 
\sum_n \left(  f(x_n) - f(x_n-h_n)  \right) \ < \ u \sum_n h_n \ < \ u ( \mu^*E_{uv} + \eps ) \ .
\]
D'altro canto, ogni $y \in A'$ \e tale che esiste $k > 0$ con $(y,y+k) \subseteq I_n \subseteq E_{uv}$. Applicando ancora il Lemma di Vitali, troviamo una collezione $\{ J_1 , \ldots , J_M  \}$, $J_m := ( y_m , y_m + k_m  )$, tale che
\[
A'' := \dot{\cup}_m J_m
\ \ , \ \
\mu A'' > \mu^*E_{uv} - 2 \eps \ .
\]
Per $y \in A''$ troviamo
\[
f(y+k) - f(y) \ > \ vk
\ \ , \ \
\sum_m \left( f(y_m+k_m) - f(y_m) \right) \ > \ v ( \mu^*E_{uv} - 2 \eps ) 
\ .
\]
Ora, per monoton\'ia abbiamo
\[
\sum_{m \ : \ J_m \subseteq I_n} \left( f( y_m + k_m ) - f(y_m) \right)
\ \leq \
f(x_n) - f(x_n-h_n)
\ ,
\]
da cui
\[
v ( \mu^*E_{uv} - 2 \eps )
\ < \
\sum_m \left( f( y_m + k_m ) - f(y_m) \right)
\ \leq \
\sum_n \left(  f(x_n) - f(x_n-h_n)  \right) 
\ < \
u ( \mu^*E_{uv} + \eps ) 
\ .
\]
Per cui $E_{uv}$ ha misura nulla ed $f$ \e derivabile q.o.. 
Infine, (\ref{eq_BV1}) si dimostra applicando il Lemma di Fatou alla successione $g_n(x) := n ( f(x+1/n) - f(x) )$, $x \in [a,b]$, che converge q.o. a $f'(x)$.
\end{proof}

\

\noindent \textbf{Funzioni BV.}
Una {\em partizione di} $[a,b]$ \e una successione finita del tipo 
\[
P := \left\{ x_0 = a , x_1 , \ldots , x_{n(P)} = b  \right\} \ .
\]
Denotiamo con $\mP$ l'insieme delle partizioni di $[a,b]$; \e chiaro che $\mP$ \e un insieme parzialmente ordinato rispetto all'inclusione; se $P \subset P'$, allora diremo che $P'$ \e un {\em raffinamento} di $P$. E' ovvio che ogni partizione $P$ ammette un raffinamento.
Per ogni $f : [a,b] \to \bR$, definiamo
\[
V_a^b f (P)
\ = \
\sum_{k=0}^{n(P)-1}
| f(x_{k+1}) - f(x_k) |
\ .
\]
Applicando la diseguaglianza triangolare troviamo 
\[
V_a^b f (P) \leq V_a^b f (P')
\ \ , \ \
P \subseteq P'
\ .
\]
Definiamo ora la {\em variazione totale}
\begin{equation}
\label{eq_BV2}
V_a^b f
\ := \
\sup_{P \in \mP}
V_a^b f (P)
\ = \
\sup_{P \in \mP}
\sum_{k=0}^{n(P)-1}
| f(x_{k+1}) - f(x_k) |
\ .
\end{equation}
\begin{defn}
Una funzione $f : [a,b] \to \bR$ si dice a \textbf{variazione limitata (BV)} se $V_a^b f < + \infty$. Denotiamo con $BV([a,b])$ l'insieme delle funzioni a variazione limitata su $[a,b]$.
\end{defn}

\begin{rem}
{\it
La nozione di funzione BV pu\'o essere interpretata anche in termini geometrici, nel senso che la condizione $V_a^bf < + \infty$ indica che la curva definita dal grafico di $f$ in $[a,b]$ \e rettificabile.
}
\end{rem}

Una semplice applicazione della diseguaglianza triangolare implica che $BV([a,b])$ \e uno spazio vettoriale. Per dare un'idea del contenuto intuitivo di (\ref{eq_BV2}), osserviamo che se $f$ \e in $C^1([a,b])$, allora 
\[
V_a^b f = \int_a^b |f'(x)| \ dx
\ .
\]

\begin{ex}{\it 
Consideriamo le funzioni $f_k : \left[ 0 , \frac{2}{\pi} \right] \to \bR$, $k = 0,1,2$,
\[
f_k(x) :=
\left\{
\begin{array}{ll}
0 \ \ , \ \ x = 0
\\
x^k \sin (1/x) \ \ , \ \ x \neq 0
\end{array}
\right.
\]
Risulta che $f_0$, $f_1$ non sono BV, mentre $f_2$ \e BV.
}\end{ex}

Il teorema seguente fornisce una caratterizzazione delle funzioni BV.
\begin{thm}
\label{thm_BV2}
Una funzione $f : [a,b] \to \bR$ \e BV se e soltanto se \e la differenza di due funzioni monot\'one. Per cui, se $f \in BV([a,b])$ allora la derivata prima $f'$ esiste q.o. in $[a,b]$ (Teo.\ref{thm_BV1}).
\end{thm}

\begin{proof}[Dimostrazione.]
Iniziamo introducendo la seguente notazione: se $a \in \bR$, allora $a^+ :=$ $\sup \{ a , 0 \} \geq 0$, $a^- :=$ $- \inf \{ a , 0 \} \geq 0$. Cosicch\'e, $|a| = a^+ + a^-$, $a = a^+ - a^-$.
Sia $f$ BV; definiamo
\[
V_a^bf^\pm 
:= 
\sup_{P \in \mP}
\sum_{k=0}^{n(P)-1}
\left( f(x_{k+1}) - f(x_k) \right)^\pm
\ ,
\]
cosicch\'e dalla relazione $|a| = a^+ + a^-$ si ottiene facilmente
\begin{equation}
\label{eq_BV3}
V_a^b f = V_a^b f^+ + V_a^b f^-
\ .
\end{equation}
Inoltre, dalla relazione $a = a^+ - a^-$ segue
\[
\sum_{k=0}^{n(P)-1}
\left( f(x_{k+1}) - f(x_k) \right)^+
\
-
\
\sum_{k=0}^{n(P)-1}
\left( f(x_{k+1}) - f(x_k) \right)^-
\
=
\
f(b) - f(a)
\]
per cui otteniamo
\begin{equation}
\label{eq_BV4}
f(b) - f(a)
=
V_a^b f^+ - V_a^b f^-
\ .
\end{equation}
Definiamo allora $f_\pm(x) :=$ $V_a^x f^\pm$ (osservare che $f$ BV implica $|f_\pm (x)| <$ $V_a^b f <$ $+ \infty$, cosicch\'e $f_\pm(x)$ \e ben definita per ogni $x$). Da (\ref{eq_BV4}) segue $f (x) = f_+ (x) - f_- (x) + f(a)$, $x \in [a,b]$, ed \e chiaro che $f_\pm$ sono monot\'one crescenti.
Viceversa, se $f = g - h$ con $g,h$ monot\'one crescenti, abbiamo
\[
\sum_{k=0}^{n(P)-1} | f(x_{k+1}) - f(x_k) |
\leq
\sum_{k=0}^{n(P)-1} 
\left[ \left( g(x_{k+1}) - g(x_k) \right)
+
\left( h(x_{k+1}) - h(x_k) \right)
\right]
\leq
g(b) - g(a) + h(b) - h(a)
,
\]
e ci\'o implica $f \in BV([a,b])$.
\end{proof}

\begin{ex}
{\it
La funzione caratteristica $f := \chi_{[0,1]} : [-1,1] \to \bR$ \e monot\'ona crescente e quindi a variazione limitata. La derivata prima di $f$ \e  -- a meno di equivalenza q.o. -- la funzione nulla, per cui (\ref{eq_BV1}) in questo caso \e una diseguaglianza stretta: $\int_{-1}^1 f' = 0 < f(1) - f(-1) = 1$.
}
\end{ex}

\begin{defn}
Una funzione $F : [a,b] \to \bR$ si dice una \textbf{primitiva di $f \in L^1([a,b])$} se
\[
F(x) = F(a) + \int_a^x f(t) dt 
\ \ , \ \ 
x \in [a,b] \ .
\]
\end{defn}

\begin{prop}
\label{prop_BV3}
Sia $f \in L^1([a,b])$ ed $F$ una sua primitiva. Allora $F \in BV([a,b]) \cap C([a,b])$.
\end{prop}

\begin{proof}[Dimostrazione.]
$F$ \e continua grazie a Cor.\ref{prop_ic}. Ora se $P$ una partizione di $[a,b]$ troviamo
\[
\sum_{k=0}^{n(P)-1} | F(x_{k+1}) - F(x_k) |
=
\sum_{k=0}^{n(P)-1} \left| \int_{x_k}^{x_{k+1}} f(t) \ dt \right|
\leq
\int_a^b |f(t)| \ dt
\ .
\]
Dunque $F \in BV([a,b])$.
\end{proof}

%
%
%
%

Concludiamo questa breve rassegna sulle funzioni BV menzionando il seguente risultato, connesso all'integrale di Stieltjes (\cite[10.36]{Kol}, \cite[\S 12.3]{Roy}):
\begin{thm}\textbf{(Helly, \cite[\S 10.36.5]{Kol}).} Sia $\{ f_n \} \subseteq BV([a,b])$ una successione puntualmente convergente ad una funzione $f : [a,b] \to \bR$ e tale che $\sup_n V_a^b f_n < + \infty$. Allora $f \in BV([a,b])$.
\end{thm}

\noindent \textbf{Funzioni AC.}
Una funzione $f : [a,b] \to \bR$ si dice {\em assolutamente continua (AC)}, se per ogni $\eps > 0$ esiste un $\delta > 0$ tale che, presa una qualsiasi collezione finita $C :=$ $\left\{ (x_k , x'_k) \right\}_k$ di intervalli disgiunti con
\[
l(C) := \sum_k (x'_k - x_k) < \delta
\ ,
\]
risulta 
\[
\sum_k | f(x'_k) - f(x_k) | < \eps \ .
\]
Denotiamo con $AC([a,b])$ l'insieme delle funzioni AC su $[a,b]$. Si dimostra facilmente che $AC([a,b])$ \e uno spazio vettoriale. Inoltre, \e ovvio che ogni funzione AC \e anche (uniformemente) continua.

\begin{ex}{\it
Sia $f : [a,b] \to \bR$ una funzione con costante di Lipschitz $L > 0$. Scegliendo $\delta < \eps / L$ nella definizione precedente, concludiamo che $f \in AC([a,b])$.
}
\end{ex}

\begin{lem}
\label{lem_AC1}
$AC([a,b]) \subseteq BV([a,b])$, cosicch\'e ogni funzione $AC$ \e derivabile q.o. in $[a,b]$.
\end{lem}

\begin{proof}[Dimostrazione.]
Sia $f \in AC([a,b])$. Fissiamo $\eps > 0$ tale che $\sum_k |f(x'_k) - f(x_k)| < \eps$ per ogni collezione $C := \{ (x_k , x'_k ) \}$ con $l(C) < \delta$ e $\delta > 0$ opportuno. 
Sia $P := \{ \ldots ,  x_k , \ldots \}$ una partizione di $(a,b)$. Effettuando eventualmente un raffinamento di $P$, possiamo costruire $N$ collezioni $C_i := \{ ( y_{ij} , y'_{ij} ) \}_j$ tali che
\[
[a,b] = \cup_i \ovl C_i
\ \ , \ \
P \subseteq \cup_{ij} \{ y_{ij} , y'_{ij} \}_{ij}
\ \ , \ \
\delta / 2 < l(C_i) < \delta
\ .
\]
Ora, il numero $N$ di collezioni che soddisfano le condizioni precedenti rimane limitato, a prescindere dalla partizione $P$:
\[
\sum_i \delta / 2 
\ \leq \
\sum_i l(C_i)
\ = \
b-a
\ \Rightarrow \
\delta \frac{N(N-1)}{4} \leq b-a
\ \Rightarrow \
N \leq N(N-1) \leq \frac{4}{\delta} \ (b-a)
\ .
\]
Inoltre, essendo $l(C_i) < \delta$ troviamo
\[
\sum_k | f(x_k) - f(x_{k+1}) | \leq 
\sum_{i,j} | f(y_{ij}) - f(y'_{ij}) | \leq 
\sum_i \eps = 
\frac{N(N-1)}{2} \ \eps 
\ .
\]
Poich\'e la maggiorazione precedente non dipende da $P$, concludiamo che $f$ \e BV.
\end{proof}

\begin{rem}{\it
L'inclusione $AC([a,b]) \subseteq BV([a,b])$ \e stretta. Un famoso esempio di funzione BV ma non AC \e la \textbf{funzione di Cantor
{\footnote{
Detta anche {\em la scala del Diavolo.}
}}
} (\cite[\S 9.33, Ex.2]{Kol},\cite[Es.5.3.4]{Giu2}), la quale si pu\'o definire come il limite della successione
\[
f_0(t) := t
\ \ , \ \
f_{n+1}(t) :=
\left\{
\begin{array}{ll}
1/2 f_n(3t)
\ \ \ \ \ \ \ \ \ \ \ \ \ \ \ \ \ \ \ , \ \ 
t \in [0,1/3]
\\
1/2
\ \ \ \ \ \ \ \ \ \ \ \ \ \ \ \ \ \ \ \ \ \ \ \ \ \ \ , \ \ 
t \in (1/3,2/3]
\\
1/2 \left( 1 + f_n \left( 3 \left( t - \frac{2}{3} \right)  \right)  \right) 
\ \ , \ \ 
t \in (2/3,1] 
\ .
\end{array}
\right.
\]
Si verifica facilmente che $\| f_{n+k}-f_n \|_\infty \leq 2^{-n}$, $\forall n,k \in \bN$, per cui esiste il limite uniforme $f \in C([0,1])$. A livello intuitivo, possiamo visualizzare $f$ come una funzione costante sugli "intervalli di mezzo" che appaiono nella costruzione dell'insieme di Cantor $W_{1/3}$ e crescente in $W_{1/3}$, il quale ha misura nulla (vedi Es.\ref{ex_cantor}).
Ora, essendo ogni $f_n$ monot\'ona crescente abbiamo che $f$ \e monot\'ona crescente (e quindi BV) ed uniformemente continua (Heine-Cantor). Si noti che $f' = 0$ q.o. in $t \in [0,1]$ e quindi 
$
\int_0^1 f' = 0 < f(1) - f(0) = 1$.
Ci\'o implica, in conseguenza del Lemma seguente, che $f$ non \e AC.
}
\end{rem}

A differenza dell'esempio precedente, abbiamo:
\begin{lem}
\label{lem_AC3}
Sia $f \in AC([a,b])$. Se $f' = 0$ q.o., allora $f$ \e costante.
\end{lem}

\begin{proof}[Dimostrazione]
Mostriamo che $f(a) = f(c)$ per ogni $c \in (a,b]$. A tale scopo definiamo
\[
E_c := \{ t \in (a,c) : f'(t) = 0 \}
\ \Rightarrow \
\mu E_c = c-a
\ ,
\]
e scegliamo arbitrari $\eps,\eps' > 0$. Ora, per ogni $t \in E_c$ esiste $h>0$ tale che $[t,t+h] \subseteq [a,c]$ e $|f(t+h)-f(t)| < \eps' h$. Sia ora $\delta > 0$ il numero reale associato ad $\eps$ nella definizione di assoluta continuit\'a; applicando il Lemma di Vitali, possiamo estrarre da $\{ [t,t+h] \}$ una collezione finita $\{ [t_k,t'_k] \}$ di intervalli disgiunti tali che
\begin{equation}
\label{eq_mudelta}
\mu \left( E_c - \dot{\cup}_k [t_k,t'_k] \right) \ < \ \delta \ .
\end{equation}
Osserviamo che si ha l'ordinamento
$t'_0 := a < t_1 < t'_1 \leq t_2 < \ldots < t'_n < t_{n+1} := c$, 
cosicch\'e ogni intevallo $(t'_k,t_{k+1})$ \e contenuto in $E_c - \dot{\cup}_k [t_k,t'_k]$. Per cui
\[
\sum_{k=0}^n | t_{k+1} - t'_k | < \delta
\ \ \Rightarrow \ \
\sum_{k=0}^n | f(t_{k+1}) - f(t'_k) | < \eps
\ .
\]
D'altra parte,
\[
\sum_{k=1}^n | f(t'_k) - f(t_k) | 
\ \leq \ 
\eps' \sum_{k=0}^n (t'_k - t_k) \leq \eps' (c-a)
\ .
\]
Concludiamo quindi
\[
\begin{array}{ll}
|f(c)-f(a)|
& \ = \
| \sum_k f(t_{k+1}) - f(t'_k) + f(t'_k) - f(t_k) |
\\ & \ \leq \
\sum_k | f(t_{k+1}) - f(t'_k) | + \sum_k | f(t'_k) - f(t_k) |
\leq 
\eps + \eps'(c-a)
\ .
\end{array}
\]
\end{proof}

Il seguente risultato risolve completamente la questione della determinazione dell'immagine dell'applicazione integrale (\ref{def_int_ab}):
\begin{thm}
\label{thm_AC2}
Una funzione $F : [a,b] \to \bR$ \e primitiva di una qualche $f \in L^1([a,b])$ se e soltanto se $F \in AC([a,b])$.
\end{thm}

\begin{proof}[Dimostrazione.]
Sia $F \in AC([a,b])$. Allora $F$ \e BV, derivabile q.o. (Lemma \ref{lem_AC1}) e differenza di due funzioni monot\'one $g$ ed $h$ (Teorema \ref{thm_BV2}). Per cui, per disuguaglianza triangolare,
\[
|F'(x)| \leq g'(x) + h'(x) 
\ \ , \ \  
{\mathrm{q.o. \ in \ }} x \in [a,b]  
\ .
\]
Quindi, integrando membro a membro troviamo
\[
\int_a^b |F'(x)| \ dx \leq g(b) - g(a) - h(b) + h(a) \ ,
\]
il che implica che $F'$ \e integrabile. Definendo
\[
\hat F (x) := F(a) + \int_a^x F'(t) \ dt  \ \ , \ \ x \in [a,b] \ ,
\]
troviamo $F' - \hat{F}' = 0$ q.o. e quindi (grazie a Lemma \ref{lem_AC3}) $F = \hat F$.
Viceversa, sia $F$ una primitiva. Allora $F$ \e continua e BV (Prop.\ref{prop_BV3}). Eventualmente sommando una costante, possiamo assumere che $f \geq 0$, per cui concludiamo che $F$ \e AC grazie a (\ref{eq_LPB1}). 
\end{proof}

\subsection{Funzioni convesse e diseguaglianza di Jensen.}
\label{sec_jensen}

Una funzione 
$\varphi : (a,b) \to \bR$
si dice {\em convessa} se per ogni $\lambda \in [0,1]$ risulta
\begin{equation}
\label{eq_Lp1}
\varphi ( (1-\lambda)x + \lambda y ) 
\leq
 (1-\lambda)\varphi(x) + \lambda \varphi(y) 
\ \ , \ \
a < x , y < b
\ .
\end{equation}
Ci\'o significa che per ogni $x,y \in (a,b)$ il grafico di $\varphi$ rimane confinato nella regione al di sotto della retta a secondo membro di (\ref{eq_Lp1}). Una semplice verifica mostra che vale la diseguaglianza
\begin{equation}
\label{eq_Lp2}
\frac{ \varphi(y) - \varphi (x) }{ y-x }
\ \leq \
\frac{ \varphi(y') - \varphi (x') }{ y'-x' }
\ \ , \ \
x \leq x' < y' 
\ , \
x' < y \leq y'
\ .
\end{equation}

\begin{prop}
\label{prop_conv}
Sia $\varphi : (a,b) \to \bR$ convessa. Allora
\begin{enumerate}
\item $\varphi$ \e localmente Lipschitz in $(a,b)$;
\item $\varphi$ \e AC in ogni sottointervallo chiuso di $(a,b)$;
\item $D^+\varphi = D_+ \varphi$, $D^-\varphi = D_- \varphi$ sono monot\'one crescenti;
\item Se $\{ \varphi_i \}_{i \in I}$ \e una famiglia di funzioni convesse, allora $\varphi := \sup_i \varphi_i$ \e convessa;
\item $\varphi$ \e derivabile in $(a,b)$ tranne che in un sottoinsieme numerabile.
\end{enumerate}
\end{prop}

\begin{proof}[Sketch della dimostrazione]
{\em Punto 1}: Se $[c,d] \subset [a',b'] \subset (a,b)$, allora per ogni $c < x,y < d$ troviamo
\begin{equation}
\label{eq_conv_lip}
\frac{ \varphi(c) - \varphi (a') }{ c-a' }
\ \leq \
\frac{ \varphi(y) - \varphi (x) }{ y-x }
\ \leq \
\frac{ \varphi(b') - \varphi (d) }{ b'-d }
\ ,
\end{equation}
il che implica che $\varphi$ \e Lipschitz in $[c,d]$. {\em Punto 2}: Ogni funzione lipschitziana \e AC. 
{\em Punto 3}: Eq. \ref{eq_Lp2} implica che ogni rapporto incrementale di $\varphi$ rispetto ad un $x_0$ fissato \e una funzione monot\'ona crescente. 
{\em Punto 4}: Si osservi che 
\[
\varphi_i ( (1-\lambda)x + \lambda y ) 
\leq
(1-\lambda)\varphi_i(x) + \lambda \varphi_i(y) 
\leq
(1-\lambda)\varphi(x) + \lambda \varphi(y) 
\ , \
i \in I
\ ,
\]
il che implica che $\varphi$ soddisfa la condizione di convessit\'a.
{\em Punto 5}: L'insieme di discontinuit\'a di una funzione monot\'ona pu\'o essere al pi\'u numerabile: applicando questo principio a $D^+\varphi$, otteniamo che $\varphi$ \e derivabile nei punti in cui $D^+\varphi$ \e continua.
\end{proof}
La costante di Lipschitz associata a $\varphi$ nel senso del Punto 2 dipende solo dai valori assunti da $\varphi$ agli estremi $c,d$, $a',b'$ (vedi (\ref{eq_conv_lip})), e questo \e un fattore di cui tenere conto nel caso in cui si abbia necessit\'a di effettuare stime connesse all'equicontinuit\'a di famiglie di funzioni convesse.

Osserviamo inoltre che \e semplice dimostrare l'affermazione reciproca del Punto 3: {\em se $\varphi$ \e continua in $(a,b)$ e se una delle sue derivate \e crescente, allora $\varphi$ \e convessa}. 

\

\noindent \textbf{La diseguaglianza di Jensen.} Una {\em retta di supporto in} $x_0 \in (a,b)$ per $\varphi$ \e una retta 
\[
y (x) = m(x-x_0) + \varphi(x_0)
\ \ , \ \
x \in (a,b)
\ ,
\]
tale che $y(x) \leq \varphi(x)$, $x \in (a,b)$. E' banale verificare che (proprio grazie alla convessit\'a di $\varphi$) esiste sempre una retta di supporto per $\varphi$ in ogni $x_0 \in (a,b)$.
\begin{prop}[Diseguaglianza di Jensen]
\label{prop_jensen}
Sia $\varphi : \bR \to \bR$ convessa, $(X,\mM,\mu)$ uno spazio di misura di probabilit\'a ed $f \in L_\mu^1(X)$. Allora
\begin{equation}
\label{eq_jensen}
\varphi \left( \int_X f \ d\mu \right)
\ \leq \
\int_X \varphi\circ f\ d\mu
\ .
\end{equation}
\end{prop}

\begin{proof}[Dimostrazione]
Innanzitutto osserviamo che, essendo $\varphi$ continua, $\varphi \circ f$ \e misurabile. Definiamo $\alpha := \int f$ e consideriamo una retta di supporto in $\alpha$,
\[
y(t) = m(t - \alpha) + \varphi(\alpha)
\ \ , \ \
t \in \bR
\ ,
\] 
cosicch\'e valutando su valori del tipo $f(x)$, $x \in X$, otteniamo $\varphi(f(x)) \geq m(f(x) - \alpha) + \varphi(\alpha)$. Integrando su $X$ (tenuto conto che $\mu X = 1$) si trova
\[
\int_X \varphi \circ f \ d\mu
\ \geq \
m \int_X (f-\alpha) \ d\mu \ + \ \varphi (\alpha)
\ = \
\varphi (\alpha)
\ ,
\]
e quindi otteniamo quanto volevasi dimostrare.
\end{proof}

\begin{ex}{\it
Per ogni $f \in L^1([0,1])$ si trova $e^{\int_0^1 f} \leq \int_0^1 e^{f(x)} \ dx$.
}
\end{ex}

\subsection{Esercizi.}
\label{ex_sec_lebesgue}

\textbf{Esercizio \ref{sec_MIS}.1 (Il Lemma di Riemann-Lebesgue).} {\it Sia $f : \bR \to \bR$ integrabile e $\varphi$ limitata, misurabile e tale che esista $\pi > 0$ con 
$\varphi (x+\pi) = - \varphi(x)$, $\forall x \in \bR$.
\textbf{(1)} Preso $k \in \bN$ si mostri che
\[
\int_\bR f(x) \varphi (kx) \ dx 
\ = \
- \int_\bR f \left( x + \frac{\pi}{k} \right) \varphi (kx) \ dx
\ .
\]
\textbf{(2)} Si assuma come noto che
\begin{equation}
\label{eq_exLEB01}
\lim_{h \to 0} \int_\bR | f(x+h) - f(x) | \ dx \ = \ 0
\end{equation}
(vedi Esercizio \ref{sec_Lp}.2), e, usando il punto (1), si mostri che
\begin{equation}
\label{eq_exLEB1}
\lim_{k \to \infty} \int_\bR f(x) \varphi(kx) \ dx \ = \ 0
\ .
\end{equation}

\noindent Soluzione.} \textbf{(1)} Usando la sostituzione $x \mapsto x + \pi k^{-1}$ e la periodicit\'a di $\varphi$ si trova
\[
I_k 
\ := \ 
\int f(x) \varphi (kx) \ dx
\ =  \
\int f \left( x + \frac{\pi}{k} \right) \varphi (kx + \pi) \ dx
\ = \
- \int f \left( x + \frac{\pi}{k} \right) \varphi (kx) \ dx
\ .
\]
\textbf{(2)} Usando il punto precedente e (\ref{eq_exLEB01}) otteniamo le stime
\[
2 |I_k|
\ \leq \
\int \left| f \left( x + \frac{\pi}{k} \right) - f(x) \right| | \varphi (kx) | \ dx
\ \leq \
\| \varphi \|_\infty \int \left| f \left( x + \frac{\pi}{k} \right) - f(x) \right| \ dx
\ \stackrel{k}{\to} 0
\ .
\]

\

\noindent \textbf{Esercizio \ref{sec_MIS}.2.} {\it Sia $\{ f_n \} \cup \{ f \} \subset C_0(\bR) \cap L^1(\bR)$, e $\lim_n \int | f_n - f | = 0$. Allora $f(x) = \lim_n f_n(x)$, $\forall x \in \bR$.}

\

\noindent {\it Soluzione.} Supponiamo per assurdo che esistano $x_0 \in \bR$ ed $\eps > 0$ tali che esiste una sottosuccessione $\{ x_{n_k} \}$ con
\[
| f(x_0) - f_{n_k} (x_0) | \geq \eps \ .
\]
Per continuit\'a, esiste un aperto $A \ni x_0$ con $\mu (A) > 0$ tale che
\[
| f(x) - f_{n_k} (x) | \geq \frac{\eps}{2}  \ \ , \ \ x \in A \ .
\]
Per cui,
\[
\int   | f_{n_k}(x) - f(x) | \ dx \geq
\int_A | f_{n_k}(x) - f(x) | \ dx \geq
\frac{\eps}{2} \ \mu A
\ ,
\]
il che contraddice l'ipotesi fatta.

\

\noindent \textbf{Esercizio \ref{sec_MIS}.3.} {\it Sia $f \in BV([a,b])$. Presa una funzione reale $g$, indicare le condizioni che $g$ deve soddisfare affinch\'e $g \circ f$ sia $BV$.}

\

\noindent \textbf{Esercizio \ref{sec_MIS}.4.} {\it Sia $A \subset [0,1]$ tale che $\mu A = 0$, dove $\mu$ \e la misura di Lebesgue. Dimostrare che: (1) Esiste una successione di aperti $[0,1] \supset A_1 \supset A_2 \supset \ldots$ tale che $A \subset A_n$ e $\mu A_n \leq 2^{-n}$ per ogni $n \in \bN$; (2) $g := \sum_n \chi_{A_n}$ appartiene ad $L^1([0,1])$; (3) $f(x) := \int_0^x g$, $x \in [0,1]$, appartiene ad $AC([0,1])$, ma non \e derivabile in nessun punto di $A$.

\

\noindent (Suggerimenti: per (1) si usi la regolarit\'a esterna; per (2) si osservi che 
$\| g \|_1 = \sum_n n \mu ( A_n - A_{n+1})$; 
per (3) si osservi che $g|_A = + \infty$.)
}

\

\noindent \textbf{Esercizio \ref{sec_MIS}.5 (Il Lemma di Borel-Cantelli).} {\it Sia $(X,\mM,\mu)$ uno spazio di misura finita ed $\{ A_n \} \subseteq \mM$ tale che $\sum_m \mu A_n < + \infty$. Posto $B_N := \cup_{n=N}^\infty A_n$, dimostrare che
\begin{equation}
\label{eq_BC}
\mu \left( \bigcap_N B_N \right) \ = \ 0 \ .
\end{equation}
}

\noindent {\it Soluzione.} Per ogni $N \in \bN$, si ha
$\mu \left( \cap_N B_N  \right)  \leq  \mu B_N \leq \sum_{n=N}^\infty \mu A_n$.

\

\noindent \textbf{Esercizio \ref{sec_MIS}.6 (Riguardo (\ref{eq_BC0})).} {\it Sia $A_n := [n,+\infty)$, $n \in \bN$. Si verifichi che posto $B_N := \cup_{n=N}^\infty A_n$, risulta 
\[
\mu \left( \bigcap_N B_N \right) = 0
\ \ , \ \
\lim_N \mu B_N = + \infty
\ .
\]
}

\noindent \textbf{Esercizio \ref{sec_MIS}.7} {\it Sia $X$ un insieme e $\beta$ una $\sigma$-algebra su $X$. Date $\mu , \nu \in \Lambda_\beta^1(X)$ (vedi (\ref{def_l1})), si mostri che:
\textbf{(1)} $| \mu E | \leq |\mu| E$, $\forall E \in \mM_\mu$;
\textbf{(2)} $\| \mu \| = 0$ $\Rightarrow$ $\mu = 0$ (ovvero, $\mu E = 0$ per ogni $E \in \mM$);
\textbf{(3)} $\| \mu + \nu \| \leq \| \mu \| + \| \nu \|$.
}

\

\noindent {\it Soluzione.} Consideriamo decomposizioni di Hahn $\{ X^\pm_\mu \}$ per $\mu$ e $\{ X^\pm_{\mu+\nu} \}$ per $\mu + \nu$. Riguardo (1), si osservi che 
$|\mu E| = 
 | \mu (E \cap X^+_\mu) + \mu (E \cap X^-_\mu) | \leq 
 \mu (E \cap X^+_\mu) - \mu (E \cap X^-_\mu) =
 |\mu|E$,
$E \in \mM_\mu$,
cosicch\'e (2) segue per monoton\'ia di $|\mu|$; riguardo (3), usando (1) troviamo
\[
\begin{array}{ll}
\| \mu + \nu \| & = | \mu + \nu |X = \\ & =
\{ \mu+\nu \}X^+_{\mu+\nu} - \{ \mu+\nu \}X^-_{\mu+\nu} = \\ & =
\mu X^+_{\mu+\nu} - \mu X^-_{\mu+\nu} + \nu X^+_{\mu+\nu} - \nu X^-_{\mu+\nu} \leq \\ & \leq
|\mu| X^+_{\mu+\nu} + |\mu| X^-_{\mu+\nu} + |\nu| X^+_{\mu+\nu} + |\nu| X^-_{\mu+\nu} = \\ & =
|\mu|X + |\nu| X \ .
\end{array}
\]  

\

\noindent \textbf{Esercizio \ref{sec_MIS}.8} {\it Sia $X$ uno spazio topologico. Si mostri che per ogni funzione boreliana $f$, $\lambda \in \bR$, e misure con segno, boreliane, finite $\mu , \nu$, risulta
\[
\int f \ d \{ \lambda \mu + \nu \} \ = \ \lambda \int f \ d \mu + \int g \ d \nu \ 
\ .
\]
(Suggerimento: si inizi dimostrando l'uguaglianza precedente su funzioni semplici e poi si usi la definizione di integrale.)
}

\

\noindent \textbf{Esercizio \ref{sec_MIS}.9} {\it Sia $X$ uno spazio topologico localmente compatto e di Hausdorff. Preso $x \in X$ ed una misura di Dirac $\mu_x : \mM \to \bR$ (vedi Esempio \ref{ex_MIS_Dirac}), si mostrino i seguenti punti:
\textbf{(1)} Due funzioni misurabili $f,g$ coincidono q.o. rispetto a $\mu_x$ se e solo se $f(x) = g(x)$;
\textbf{(2)} $\int f \ d \mu_x = f(x)$, per ogni $f \in L_{\mu_x}^1(X)$.

\

\noindent (Suggerimento: si osservi che, essendo $X$ di Hausdorff, l'insieme $\{ x \}$ \e boreliano e quindi misurabile; per cui, 
$\mu_x E = \mu_x(E-\{ x \}) + \mu_x \{ x \} = \mu_x \{ x \} = 1$ 
se $E \ni x$, mentre $\mu E = 0$ se $x \notin E$. Di conseguenza, prese $f,g \in M(X)$ si trova
\[
\mu_x \{ y : f(y) = g(y) \} = 
\left\{
\begin{array}{ll}
\mu_x \{ x \} = 1 \ \ , \ \ se \ f(x) = g(x)
\\
0 \ \ , \ \ se \ f(x) \neq g(x) \ .
\end{array}
\right.
\]
Per quanto riguarda il punto (2), si verifichi sulle funzioni semplici e poi si applichi la definizione di integrale).
}

\

\noindent \textbf{Esercizio \ref{sec_MIS}.10} {\it Sia $(X,\mM,\mu)$ uno spazio misurabile ed $f \in M^+(X)$. Presa la misura
\[
\mu_f E := \int_E f \ d \mu 
\ \ , \ \
E \in \mM
\ ,
\]
si mostri che $\int g \ d \mu_f = \int gf \ d \mu$ per ogni $g \in L_{\mu_f}^1(X)$.

\

\noindent (Suggerimento: si verifichi - al solito - sulle funzioni semplici e poi si applichi la definizione di integrale).
}

\

\noindent \textbf{Esercizio \ref{sec_MIS}.11.} {\it Richiamando la notazione (\ref{def_J}), si consideri l'applicazione 
\[
\mu^*_{ac}A \ := \ \inf \left\{ \sum_n l(J_n) \ : \ A \subset \cup_n J_n \ , \  \{ J_n \} \subset \mI^{ac}  \right\}
                   \ \ , \ \
                   \forall A \subseteq \bR
                   \ ,
\]
e si mostri che essa coincide con la misura esterna di Lebesgue.

\

\noindent (Suggerimenti: si inizi mostrando che $\mu^*_{ac} = \mu^*$ sugli intervalli, verificando che, dati
$a<b \in \bR$ e $I := (a,b)$, $J := (a,b]$, $J' := [a,b)$,
risulta
\[
\mu^*_{ac}(I) = \mu^*_{ac}(\ovl I) = \mu^*_{ac}(J) = \mu^*_{ac}(J') = b-a \ ,
\]
sulla linea del Lemma \ref{lem_ME02}. Si usi il fatto che per ogni $\eps > 0$ si ha
$\mu^*_{ac}(a,b+\eps] - \mu^*(a,b) = \eps$).
}

\

\noindent \textbf{Esercizio \ref{sec_MIS}.12 (Funzioni di distribuzione).} {\it Sia $\mu$ una misura boreliana finita su $\bR$. La \textbf{funzione di distribuzione} di $\mu$ si definisce come
\[
\wt \mu(x) := \mu(-\infty,x]
\ \ , \ \
x \in \bR
\ .
\]
\textbf{(1)} Si mostri che $\wt \mu$ \e:
(1.1) monot\'ona crescente;
(1.2) tale che $\lim_{x \to - \infty}\wt \mu(x) = 0$;
(1.3) continua a destra, ovvero $\lim_{\delta \to 0^+} \wt \mu(x+\delta) = \wt \mu(x)$, $\forall x \in \bR$;
\textbf{(2)} Si mostri che 
          \[
          \mu(a,b] = \wt \mu(b) - \wt \mu(a)
          \ \ , \ \
          \forall a < b
          \ .
          \]
\textbf{(3)} Si mostri che $\wt \mu$ \e limitata.
\textbf{(4)} Si mostri che $\wt \mu$ \e continua se e solo se $\mu \{ x \} = 0$, $\forall x \in \bR$.

\

\noindent (Suggerimenti: 
          (1.1) si usi la monoton\'ia di $\mu$;
          (1.2) si noti che $\cap_{n \in \bN} (-\infty,-n] = \emptyset$;
          (1.3) si ha
                \[
                \lim_n \wt \mu(x+1/n) \ = \
                \lim_n \mu \{ \cap_k^n (-\infty,x+1/k] \} \ \stackrel{ Lemma \ \ref{lem_caressa} }{=} \
                \mu(-\infty,x]
                \ .
                \]
          (2)   Si ha $[a,b) = (-\infty,b)-(-\infty,a)$;
          (3)   Si usi il fatto che $\wt \mu$ \e monot\'ona crescente e si osservi che 
                $\sup \wt \mu = \lim_{x \to \infty} \wt \mu(x) = \mu(\bR) < \infty$.
          (4)   Si ha
                $\mu \{ x \} = 
                 \mu(-\infty,x] - \cup_n \mu(-\infty,x-1/n] = 
                 \wt \mu(x) - \lim_n \wt \mu(x-1/n)$).
}

\

\noindent \textbf{Esercizio \ref{sec_MIS}.13 (Misure di Lebesgue-Stieltjes).} {\it Sia $\omega : \bR \to \bR$ una funzione monot\'ona crescente e continua a destra. 
\textbf{(1)} Si mostri che definendo 
          $l_\omega(J) := \omega(b)-\omega(a)$, $\forall J := (a,b] \in \mI^{ac}$,
          e
          \[
          \mu_\omega^*A := \inf \left\{ \sum_n l_\omega(J_n) \ : \ 
                                        A \subset \cup_n J_n \ , \  
                                        \{ J_n \} \subset \mI^{ac}  \right\}
          \ \ , \ \
          \forall A \subseteq \bR
          \ ,
          \]
          si ottiene una misura esterna.
\textbf{(2)} Si mostri che la misura associata $\mu_\omega$ \e boreliana.
\textbf{(3)} Supposto che $\mu_\omega$ sia finita, si mostri che la relativa funzione di distribuzione $\wt \mu_\omega$ coincide con $\omega$.
\textbf{(4)} Si assuma che $\omega|_{[\alpha,\beta]} \in AC([\alpha,\beta])$ per qualche $\alpha,\beta \in \bR$ e si mostri che
          \[
          \int_{[\alpha,\beta]} f \ d \mu_\omega \ = \ \int_\alpha^\beta f(t) \omega'(t) \ dt
          \ \ , \ \
          \forall f : [ \alpha,\beta ] \to \bR \ boreliana
          \]
          (il secondo integrale \e rispetto alla misura di Lebesgue).
          
          \
          
\noindent (Suggerimenti: per (1) e (2) si proceda in modo analogo alla misura di Lebesgue;  
                         per (3) si mostri, in analogia al Lemma \ref{lem_ME02}, che 
                         $\mu_\omega^*(J) = l_\omega(J)$, $\forall J \in \mI^{ac}$
                         (si veda anche l'Esercizio \ref{sec_MIS}.11);
                         per (4) si inizi verificando su funzioni caratteristiche di intervalli in $\mI^{ac}$
                         contenuti in $[\alpha,\beta]$).
}

\newpage
\section{Gli spazi $L^p$.}
\label{sec_Lp}

Esistono ampie ed interessanti classi di funzioni misurabili che non sono n\'e continue n\'e integrabili.
In tal caso pu\'o essere conveniente adottare opportune norme, tipicamente non euclidee, che generalizzano in modo naturale quelle definite su $\bR^n$, $n \in \bN$, per $p \in [1,+\infty)$,
\[
\| v \|_p := \left( \sum_1^n |v_i|^p \right)^{1/p}
\ \ , \ \
v \in \bR^n
\ .
\]
Ci\'o conduce alla teoria degli spazi $L^p$, la quale fornisce tra l'altro un'importante motivazione per la sistematizzazione di concetti (quelli di {\em norma, funzionale, dualit\'a}) che troveranno poi naturale collocazione nell'ambito degli spazi di Banach e di Hilbert, e quindi dell'analisi funzionale.

\subsection{Propriet\'a generali.}

Sia $(X,\mM,\mu)$ uno spazio misurabile completo. Preso $p \in (0,+\infty)$ denotiamo con $L_\mu^p(X)$ l'insieme delle funzioni misurabili su $X$ tali che 
\[
\int_X |f|^p \ < \ + \infty \ :
\]
poich\'e 
\[
| f + g |^p \ < \ 2^p \left( |f|^p + |g|^p \right)
\ ,
\]
concludiamo che ogni $L_\mu^p(X)$ \e uno spazio vettoriale. Nel caso in cui $X \subseteq \bR^n$, $n \in \bN$, sia equipaggiato con la misura di Lebesgue ometteremo il simbolo $\mu$, per cui scriveremo ad esempio $L^p(\bR)$, $L^p(X) := L^p(a,b)$, $X := (a,b)$, $a,b \in \wa \bR$. Inoltre, per economia di notazione d'ora in poi scriveremo $L^p := L^p([0,1])$, $p \in [1,+\infty)$, $\mG := \mG([0,1])$ (notare che in quest'ultimo caso stiamo considerando l'intervallo {\em chiuso} $[0,1]$).
Definiamo ora
\begin{equation}
\label{eq_normap}
\| f \|_p 
\ := \
\left(
\int_X |f|^p
\right)^{1/p}
\ \ , \ \
f \in L_\mu^p(X)
\ .
\end{equation}
Denotiamo inoltre con $L_\mu^\infty(X)$ lo spazio vettoriale delle funzioni limitate e misurabili su $X$, e definiamo
$
\| f \|_\infty :=$
$\sup_x |f(x)|$,
$f \in L_\mu^\infty(X)$.
Come suggerito dalla notazione, la nostra intenzione \e quella di interpretare $\| \cdot \|_p$, $p \in (0,+\infty]$, come una norma su $L_\mu^p(X)$. Tuttavia, \e chiaro che $\| \cdot \|_p$ non pu\'o essere una norma, in quanto svanisce su ogni funzione nulla q.o. su $X$. Per ovviare a ci\'o, con un abuso di terminologia identificheremo funzioni $L^p$ con le corrispondenti classi di equivalenza q.o. su $X$
{\footnote{In termini pi\'u precisi, passiamo da $L_\mu^p(X)$ al suo spazio quoziente rispetto al sottospazio delle funzioni nulle q.o.. Il nostro abuso di terminologia consiste allora nell'indicare tale quoziente ancora con $L_\mu^p(X)$.}}.
Nel caso $p = + \infty$ il passaggio a classi di equivalenza q.o. richiede un piccolo aggiustamento, che consiste nel definire {\em l'estremo superiore essenziale},
\begin{equation}
\label{def_ese}
\| f \|_\infty 
\ := \
\inf \left\{ M \in \bR \ : \ \mu \{ x \in X : |f(x)| > M \} = 0   \right\}
\ \ , \ \
f \in L_\mu^\infty(X)
\ ,
\end{equation}
in maniera tale che se $g = f$ q.o. allora $\| g \|_\infty = \| f \|_\infty$.

Osserviamo che ancora non sappiamo se $\| \cdot \|_p$ \e una norma, in quanto dobbiamo dimostrare la diseguaglianza triangolare; iniziamo osservando che, grazie a propriet\'a elementari dell'integrale, essa \e verificata nei casi $p = 1,\infty$. Inoltre, una semplice stima mostra che 
\[
\int_X |fg| 
\ \leq \
\| f \|_1 \| g \|_\infty 
\ \ , \ \
f \in L_\mu^1(X) \ , \ g \in L_\mu^\infty(X)
\ .
\]

\begin{prop}[Diseguaglianza di Holder]
\label{prop_holder}
Siano $p,q \in [0,+\infty]$ tali che
\begin{equation}
\label{eq_conj}
\frac{1}{p} + \frac{1}{q} = 1
\ .
\end{equation}
Allora
\begin{equation}
\label{eq_holder}
\int_X |fg| 
\ \leq \ 
\| f \|_p \| g \|_q
\ \ , \ \
f \in L_\mu^p(X) \ , \ g \in L_\mu^q(X)
\ .
\end{equation}
\end{prop}

\begin{proof}[Dimostrazione]
Come primo passo, osserviamo che per ogni $a,b > 0$ risulta
\begin{equation}
\label{eq_holder_point}
a^{1/p} b^{1/q} \ \leq \ \frac{a}{p} + \frac{b}{q} \ ,
\end{equation}
come si dimostra facilmente usando la convessit\'a di $\exp$:
\[
e^{ \frac{1}{p} \log a + \frac{1}{q} \log b   } \ \leq \ \frac{a}{p} + \frac{b}{q} \ .
\]
Ora, se $f,g$ sono come da ipotesi (e non nulle q.o., altrimenti non c'\e niente da dimostrare), poniamo $a := |f|^p \| f \|_p^{-p}$, $b := |g|^q \| g \|_q^{-q}$ ed otteniamo, usando (\ref{eq_holder_point}),
\[
\frac{|fg|}{ \| f \|_p \| g \|_q }
\ \leq \
\frac{1}{p} \frac{ |f|^p }{ \| f \|_p^p }
+
\frac{1}{q} \frac{ |g|^q }{ \| g \|_q^q }
\ ,
\]
ed integrando otteniamo (\ref{eq_holder}).
\end{proof}

Reali estesi $p,q \in [0,+\infty]$ che soddisfano (\ref{eq_conj}) si dicono {\em coniugati}. Chiaramente $q$ \e univocamente determinato da $p$, ed in tal caso sar\'a denotato con $\ovl p$.

\begin{prop}[Diseguaglianza di Minkowski, ovvero la disegueglianza triangolare]
\label{prop_mink}
Siano $p \in [1,+\infty]$ ed $f,g \in L_\mu^p(X)$. Allora $\| f+g \|_p \leq$ $\| f \|_p + \| g \|_p$.
\end{prop}

\begin{proof}[Dimostrazione]
Essendo il caso $p=1$ ovviamente verificato, supponiamo $p>1$, poniamo $q := \ovl p$, ed osserviamo che $q(p-1)=p$, cosicch\'e 
\[
\int_X \left( |f|^{p-1} \right)^q = \| f \|_p^p < + \infty
\ \ , \ \
f \in L_\mu^p(X)
\ \ \Rightarrow \ \
f^{p-1} \in L_\mu^q(X)
\ .
\]
Poi, stimiamo
\[
\begin{array}{ll}
\| f+g \|_p^p
& \ \ \ \leq \ \ \ 
\int_X |f| |f+g|^{p-1} + \int_X |g| |f+g|^{p-1}
\\ & \stackrel{Holder}{\leq} \
\| f \|_p \| (f+g)^{p-1} \|_q
+
\| g \|_p \| (f+g)^{p-1} \|_q
\\ & \ \ \ = \ \ \  
( \| f \|_p + \| g \|_p )
( \| f+g \|_p^{p/q} )
\ .
\end{array}
\]
\end{proof}

\begin{cor}
$L_\mu^p(X)$ \e uno spazio normato per ogni $p \in [1,+\infty]$. 
\end{cor}

\begin{cor}
\label{cor_holder}
Sia $(X,\mM,\mu)$ uno spazio di misura finita e $p \in [1,+\infty]$. Se $f \in L_\mu^p(X)$ allora $f \in L_\mu^1(X)$.
\end{cor}

\begin{proof}[Dimostrazione]
Basta osservare che la funzione costante $1$ appartiene ad $L_\mu^q(X)$, $q := \ovl{p}$, cosicch\'e
\[
\| f \|_1 = 
\int |f1| \leq 
\| f \|_p  \| 1 \|_q =
\| f \|_p  \left( \mu X \right)^{1/q} 
\ .
\]
\end{proof}

Passiamo ora a dimostrare la completezza degli spazi $L^p$. Per le nozioni di {\em spazio di Banach e di Hilbert} si veda \S \ref{sec_afunct}.
\begin{thm}[Fischer-Riesz]
\label{thm_RF}
$L_\mu^p(X)$ \e uno spazio di Banach per ogni $p \in [1,+\infty]$. Inoltre, se $\{ f_n \} \cup \{ f \} \subset L_\mu^p(X)$ con $\lim_n \| f_n-f \|_p = 0$, allora esiste una sottosuccessione $\{ n_k \}$ tale che $f(x) = \lim_k f_{n_k}(x)$ q.o. in $x \in X$.
\end{thm}

\begin{proof}[Dimostrazione]
Il caso $p = +\infty$ \e banale, per cui assumiamo $p \in [1,+\infty)$. L'idea \e quella di applicare il criterio delle serie convergenti in norma (Prop.\ref{prop_compl}), per cui consideriamo una successione $\{ f_i \} \subset L_\mu^p(X)$ tale che esista finito $M :=$ $\sum_i^\infty \| f_i \|_p$ con lo scopo di dimostrare che esiste la somma $\sum_i f_i \in L_\mu^p(X)$; grazie al gi\'a menzionato criterio, ci\'o sar\'a sufficiente a concludere che $L_\mu^p(X)$ \e uno spazio di Banach. Definiamo
\[
g_n(x) := \sum_{i=1}^n | f_i(x) |
\ \ , \ \
x \in X
\ ,
\]
ed osserviamo che $\| g_n \|_p \leq$ $\sum_i^n \| f_i \|_p \leq$ $M$. Ci\'o implica che $g_n \in L_\mu^p(X)$, $\forall n \in \bN$. Ora, per monoton\'ia per ogni $x \in X$ esiste $g(x) :=$ $\lim_n g_n(x) \in$ $[0,+\infty]$. La funzione $g$ cos\'i ottenuta \e misurabile e, applicando la diseguaglianza di Minkowski ed il Lemma di Fatou (Teo.\ref{thm_fatou}), otteniamo
\begin{equation}
\label{eq_RF}
\int_X g^p 
\ \leq \
\lim_n \inf \sum_i^n \int_X |f_i|^p 
\ \leq \
M^p
\ ,
\end{equation}
per cui $g \in L_\mu^p(X)$. Ci\'o implica che q.o. in $x \in X$ la serie $\sum_i f_i(x)$ \e assolutamente convergente per cui (essendo $\bR$ completo) esiste finito $s(x) :=$ $\lim_n \sum_i^n f_i(x)$. Definendo $s(x) := 0$ per ogni $x$ tale che $g(x) = + \infty$ otteniamo una funzione $s$ tale che
\[
s = \sum_i f_i \ \ {\mathrm{q.o.}}
\]
Ci\'o implica che $s$ \e misurabile (Teo.\ref{thm_MIS}), ed utilizzando (\ref{eq_RF}) si trova $s \in L_\mu^p(X)$. Concludiamo che $s$ \e la somma cercata e quindi $L_\mu^p(X)$ \e uno spazio di Banach. 
Sia ora $\lim_n \| f-f_n \|_p = 0$. Scegliamo una sottosuccessione $\{ n_k \}$ tale che $\| f_{n_{k+1}}-f_{n_k} \| < 2^{-n_k}$ e definiamo
\[
\delta_m (x) := \sum_{k=1}^m | f_{n_{k+1}}(x) - f_{n_k}(x) | \ .
\]
La successione precedente \e monot\'ona crescente, per cui esiste il limite $\delta (x)$ per ogni $x \in X$; per verificare che esso \e finito q.o., osserviamo che
\[
\| \delta_m \|_p 
\ \leq \
\sum_k \| f_{n_{k+1}}-f_{n_k} \|_p
\ \leq \ 
\sum_k 2^{-n_k} \leq 1
\ .
\]
Dunque $\delta \in L_\mu^p(X)$ (teorema di convergenza monot\'ona), e quindi $\delta (x) < + \infty$ q.o. in $x \in X$. Usando ripetutamente la diseguaglianza triangolare, nell'insieme dove $g(x) < +\infty$ troviamo
\begin{equation}
\label{eq_RF2}
| f_m(x)-f_l(x) | \leq \delta(x) - \delta_{l-1}(x)  
\end{equation}
e quindi $\{ f_m(x) \}$ \e di Cauchy. Definendo $\tilde f(x) := \lim_m f_m(x)$ otteniamo una funzione $\tilde f$ definita q.o. in $x \in X$, e vogliamo mostrare che in effetti $\tilde f = f$ q.o. in $x \in X$. A tale scopo, osserviamo che (\ref{eq_RF2}) implica, essendo $\delta_l(x) > 0$, che
\[
| \tilde f(x)-f_l(x) | \leq \delta(x) 
\ \ , \ \
{\mathrm{q.o. \ in \ }} x \in X
\ ;
\]
dunque, avendosi $\delta \in L_\mu^p(X)$, possiamo applicare il teorema di convergenza dominata e concludere 
$\lim_l \| \tilde f - f_l \|_p = \| \tilde f - f \|_p = 0$. 
Cosicch\'e $\tilde f = f$ q.o. in $x \in X$.
\end{proof}

\begin{rem}
{\it
La convergenza q.o. \e garantita soltanto per la sottosuccessione $\{ f_{n_k} \}$. A questo proposito si veda l'Esercizio \ref{sec_Lp}.6.
}
\end{rem}

\begin{rem}
{\it
Se $f \in L_\mu^p(X)$ allora certamente l'insieme $\{ x \in X : |f(x)| = \infty \}$ ha misura nulla, per cui ogni funzione in $L^p$ coincide q.o. con una funzione a valori in $\bR$. Nell'ambito degli spazi $L^p$ possiamo quindi sempre assumere di avere a che fare, a meno di equivalenza q.o., con funzioni a valori reali (e non a valori nei reali estesi).
}
\end{rem}

\begin{rem} \label{rem_Lp_C}
{\it
Sia $\mE$ uno spazio di Banach (vedi \S \ref{sec_afunct}). Allora possiamo introdurre su $\mE$ la $\sigma$-algebra $\beta \mE$ generata dai dischi $\Delta(v,r)$, $v \in \mE$, $r>0$ (in altri termini, $\beta \mE$ \e la $\sigma$-algebra dei boreliani associata alla topologia della norma di $\mE$). Per cui, dato il nostro spazio di misura $(X,\mM,\mu)$ ha senso considerare l'insieme $M(X,\mE)$ delle funzioni misurabili $f : X \to \mE$ nel senso di Def.\ref{def_mis_astr}. Presa $f \in M(X,\mE)$ denotiamo con $\| f(x) \| \in \bR$ la norma di $f(x)$ in $\mE$; la funzione $\{ X \ni x \mapsto \| f(x) \| \}$ \e chiaramente misurabile nel senso usuale, per cui definiamo
\[
L^p(X,\mE) 
\ := \
\{ f \in M(X,\mE) : \int \| f(x) \|^p \ d\mu < \infty \}
\ \ , \ \
p \in [1,\infty]
\ .
\]
In particolare, per $\mE = \bC$ (che identifichiamo con $\bR^2$ come spazio di Banach), otteniamo gli spazi $L^p$ complessi
\[
L_\mu^p(X,\bC) \ \ , \ \ p \in [1,+\infty] \ .
\]
Tutti i risultati precedenti (e successivi) rimangono validi per gli spazi $L_\mu^p(X,\bC)$, $L_\mu^p(X,\mE)$.
}
\end{rem}

\subsection{Approssimazione in $L^p$.}
\label{sec_appLp}

Diamo ora alcuni utili risultati di approssimazione.

\begin{prop}
\label{prop_sLp}
Dati $p \in [1,\infty]$ ed $f \in L_\mu^p(X)$, per ogni $\eps > 0$ esistono una funzione semplice $\psi_\eps$ e $g_\eps \in L_\mu^\infty(X)$, entrambe in $L_\mu^p(X)$ e tali che 
$\| f-\psi_\eps \|_p < \eps$,
$\| f-g_\eps \|_p < \eps$.
\end{prop}

\begin{proof}[Dimostrazione]
Effettuando la solita decomposizione $f = f^+-f^-$ possiamo ridurci al caso $f \geq 0$. 
Applicando Prop.\ref{prop_caressa} otteniamo una successione crescente di funzioni semplici e non negative $\{ \psi_n \}$ convergente puntualmente ad $f$. 
Per $p < \infty$ abbiamo $(f-\psi_n)^p \to 0$ ed $(f-\psi_n)^p \leq f^p$, con $f^p \in L_\mu^1(X)$. Dunque applicando il teorema di convergenza dominata concludiamo che
\[
\| f-\psi_n \|_p^p = \int (f-\psi_n)^p \to 0 \ .
\]
Nel caso $p = \infty$ osserviamo che per ipotesi l'insieme 
$E := \{ x : |f(x)| > \| f \|_\infty \}$
ha misura nulla. Per costruzione (vedi ancora Prop.\ref{prop_caressa}) la successione $\{ \psi_n \}$ converge uniformemente in $X-E$, e ci\'o conclude la dimostrazione per quanto riguarda le funzioni semplici.
Riguardo l'analoga affermazione per funzioni $L^\infty$ consideriamo la succcessione
$g_n(x) := \inf \{ f(x) , n \}$, $\forall x \in X$,
cosicch\'e $0 \leq g_n \leq f$, $g_n \to f$ puntualmente e 
$\| g_n \|_\infty \leq n$, $\| g_n \|_p \leq \| f \|_p$.
Possiamo quindi applicare il precedente argomento di convergenza dominata e concludere $\| g_n-f \|_p \to 0$.
\end{proof}

Nella proposizione seguente ci specializziamo al caso in cui $X$ sia uno spazio topologico normale (il che include gli spazi metrici).

\begin{prop}
\label{prop_cLp}
Sia $X$ uno spazio normale e $\mu : \mM \to \wa \bR^+$ una misura di Borel regolare esterna su $X$. 
Presi $p \in [1,\infty)$ ed $f \in L_\mu^p(X)$, per ogni $\eps > 0$ esiste una funzione 
$\varphi_\eps \in C(X) \cap L_\mu^\infty(X)$
con
$\mu \{ {\mathrm{supp}}(\varphi_\eps) \} < \infty$, 
tale che $\| f-\varphi_\eps \|_p < \eps$.
\end{prop}

\begin{proof}[Dimostrazione]
Al solito assumiamo $f \geq 0$. Applicando Prop.\ref{lem_sf} ad $f^p$ concludiamo che ${\mathrm{supp}}(f)$ \e $\sigma$-finito ed $f$ \e limite puntuale di una successione monot\'ona di funzioni semplici, non negative e con supporto di misura finita. Usando l'argomento della proposizione precedente troviamo che tale successione approssima $f$ in norma $\| \cdot \|_p$, per cui possiamo assumere che $f$ sia una funzione semplice, non negativa e con supporto di misura finita. Ma ancora, per linearit\'a ci possiamo ridurre al caso in cui $f$ sia una funzione caratteristica $\chi_E$, dove $E \subseteq X$ ha misura finita. 
Ora, per regolarit\'a esterna esiste un aperto $U \subset X$ con misura finita tale che 
$\ovl E \subset U$,
$\mu (U - E) < \eps$
(riguardo l'aver considerato la chiusura $\ovl E$ vedi Oss.\ref{oss_cbor}). Applicando il Lemma di Urysohn (Teorema \ref{thm_ury}) troviamo $\varphi_\eps \in C(X)$ tale che 
$\| \varphi_\eps \|_\infty = 1$, $\varphi_\eps |_{\ovl E} = 1$ e $\varphi_\eps|_{X-U} = 0$;
cosicch\'e $\varphi_\eps$ ha supporto con misura finita e 
\[
\| \chi_E - \varphi_\eps \|_p^p
\ = \
\int_E | 1-1 |^p +
\int_{U-E} | \chi_E - \varphi_\eps |^p
\ = \
\int_{U-E} | \varphi_\eps |^p
\ \leq \
\mu (U-E)
\ < \
\eps
\ .
\]
\end{proof}

Come applicazione del risultato precedente consideriamo lo spazio euclideo $\bR^d$, $d \in \bN$, equipaggiato con la misura prodotto di Lebesgue (vedi \S \ref{sec_misprod}), e diamo un risultato di approssimazione con le funzioni continue a supporto compatto.

\begin{cor}
\label{cor_appLp}
Per ogni $p \in [1,\infty)$, lo spazio $C_c(\bR^d)$ \e denso in $L^p(\bR^d)$ nella norma $\| \cdot \|_p$.
\end{cor}

\begin{proof}[Dimostrazione]
Grazie alla proposizione precedente \e sufficiente dimostrare che, presa $f \in C(\bR^d) \cap L^\infty(\bR^d)$ avente supporto con misura finita, e scelto $\eps > 0$, esiste $f_\eps \in C_c(\bR^d)$ tale che $\| f-f_\eps \|_p < \eps$. A tale scopo consideriamo la successione di funzioni
\[
\varphi_n (x) :=
\left\{
\begin{array}{ll}
1 \ \ , \ \ |x| < n
\\
n+1-|x| \ \ , \ \ |x| \in [n,n+1)
\\
0 \ \ , \ \ |x| \in [n+1,\infty) 
\end{array}
\right.
\ \ , \ \
x \in \bR^d
\ .
\]
Ovviamente troviamo $\varphi_n , f \varphi_n \in C_c(\bR^d)$ per ogni $n \in \bN$, nonch\'e
\[
|f(x) - f(x)\varphi_n(x) |^p \to 0 \, , \, \forall x \in \bR^d
\ \ \ {\mathrm{e}} \ \ \
|f - f \varphi_n |^p \leq |f|^p
\ .
\]
Essendo $f \in L^p(\bR^d)$ concludiamo, per convergenza dominata, che $\| f \varphi_n - f \|_p \to 0$.
\end{proof}

Osserviamo che usando il teorema di Stone-Weierstrass possiamo approssimare in norma $\| \cdot \|_\infty$ ogni funzione in $C_c(\bR^d)$ con una funzione $C_c^\infty(\bR^d)$, e da ci\'o segue che $C_c^\infty(\bR^d)$ \e denso in $L^p(\bR^d)$. D'altra parte, lo stesso risultato si pu\'o dimostrare usando i mollificatori (vedi Prop.\ref{prop_moll}).

\subsection{Il teorema di Riesz-Fr\'echet-Kolmogorov.}

Diamo ora una versione del Teorema di Ascoli-Arzel\'a per spazi $L^p$ nel caso in cui il soggiacente spazio misurabile sia $\bR^d$, $d \in \bN$, equipaggiato con la misura (prodotto) di Lebesgue; tale risultato, il {\em teorema di Riesz-Fr\'echet-Kolmogorov}, fornisce un criterio di compattezza per famiglie di funzioni rispetto alla topologia della norma $\| \cdot \|_p$. Per esporne l'enunciato richiamiamo la nozione di {\em funzione traslata}
\[
f_h (x) := f(x+h)
\ \ , \ \
f : \bR^d \to \bR
\ , \
x,h \in \bR^d 
\ .
\]
Inoltre, per ogni famiglia $\mF$ di funzioni da $\bR^d$ in $\bR$ ed $\Omega \subseteq \bR^d$, definiamo
\[
\mF_\Omega
:=
\{
f |_\Omega : f \in \mF
\}
\ .
\]
Per la dimostrazione del teorema seguente servono le nozioni di {\em convoluzione} (\S \ref{sec_conv}) e {\em mollificatore} (Def.\ref{def_moll}).

\begin{thm}[Riesz-Fr\'echet-Kolmogorov]
%
%
\label{thm_RFK}
Sia $p \in [1,+\infty)$ ed $\mF \subset L^p(\bR^d)$ limitato tale che 
$\lim_{h \to 0} \| f_h - f \|_p = 0$
uniformemente in $f \in \mF$. Allora $\mF_\Omega$ \e precompatto in $L^p(\Omega)$ per ogni $\Omega \subset \bR^n$ con misura finita.
\end{thm}

\begin{proof}[Dimostrazione]
L'idea \e quella di approssimare funzioni in $\mF_\Omega$ con funzioni continue usando una successione di mollificatori $\{ \rho_n \} \subseteq C_c^\infty(\bR^d)$ e quindi usare il Teorema di Ascoli-Arzel\'a.

\noindent \textbf{Passo 1.} Mostriamo che scelto $\eps > 0$ esiste un $\delta > 0$ tale che
\begin{equation}
\label{eq_RFK1}
\| f - \rho_n * f \|_p < \eps
\ \ , \ \
\forall f \in \mF
\ , \ 
n > \delta^{-1}
\ .
\end{equation}
Usando il fatto che $\rho_n(y) dy$ \e una misura di probabilit\'a (infatti $\int \rho_n = 1$ per ogni $n \in \bN$) troviamo
\[
|\rho_n*f(x)-f(x)| 
\ \leq \
\int | f(x-y)-f(x) | \rho_n(y) \ dy
\ \stackrel{Holder}{\leq} \
\left( \int |f(x-y)-f(x)|^p \rho_n(y) \ dy  \right)^{1/p}
,
\]
per cui
\[
\| \rho_n*f-f \|_p^p 
\ \leq \
\int \left\{ \int |f(x-y)-f(x)|^p \rho_n(y) \ dy \right\} dx 
\ \stackrel{Fubini}{=} \
\int \rho_n(y) \| f_{-y}-f \|_p^p \ dy 
\ .
\]
Riguardo l'ultimo integrale osserviamo che preso il $\delta$ dell'uniforme continuit\'a di $\mF$ abbiamo che il supporto di $\rho_n$ sar\'a contenuto in $\Delta(0,\delta)$ per $n > \delta^{-1}$; per cui, avendosi $y \in \Delta(0,\delta)$ troviamo $\| f_{-y}-f \|_p^p < \eps^p$ e quindi
\[
\| \rho_n*f-f \|_p^p 
\ \leq \
\int \rho_n(y) \eps^p \ dy 
\ \leq \ 
\eps^p 
\ .
\]
Notare che, grazie a Prop.\ref{prop_der_conv}, abbiamo $\rho_n * f \in C^\infty(\bR^d)$ per ogni $n \in \bN$.

\noindent \textbf{Passo 2.} Mostriamo che posto $c_n := \| \rho_n \|_q$ risulta
\begin{equation}
\label{eq_RFK2}
\| \rho_n*f \|_\infty 
\ \leq \ 
c_n \| f \|_p
\ \ , \ \
\forall f \in L^p(\bR^d)
\ .
\end{equation}
Infatti si trova, posto $q := \ovl p$,
\begin{equation}
\label{eq_RFK2a}
\left| \int f(x-y) \rho_n(y) \ dx \right| 
\ \leq \
\int |f(x-y) \rho_n(y)| \ dx 
\ \stackrel{Holder}{\leq} \ 
\| f_{-y} \|_p \| \rho_n \|_q 
\ = \
\| f \|_p \| \rho_n \|_q
\ .
\end{equation}
\noindent \textbf{Passo 3.} Esiste $c'_n > 0$ tale che
\begin{equation}
\label{eq_RFK3}
| \{ \rho_n*f \}(x_1) - \{ \rho_n*f \}(x_2) | 
\ \leq \
c'_n \| f \|_p \ |x_1-x_2|
\ \ , \ \
\forall f \in \mF
\ , \
x_1,x_2 \in \bR^d
\ .
\end{equation}
Per dimostrare questa stima osserviamo che, essendo $\rho_n \in C_c^\infty(\bR^d)$, abbiamo $\nabla \rho_n \in L^q(\bR^d,\bR^d)$, $\forall q \in [1,\infty]$, e quindi, applicando Prop.\ref{prop_der_conv},
\begin{equation}
\label{eq_nabla}
\nabla(\rho_n * f) = (\nabla \rho_n)*f \in L^p(\bR^d,\bR^d)
\ \ , \ \
\forall f \in \mF
\ .
\end{equation}
Ora, per il teorema di Lagrange abbiamo che per ogni $x_1 , x_2 , y \in \bR^d$ esiste $\xi \in \bR^d$ tale che
\[
\{ \rho_n*f \}(x_1-y) - \{\rho_n*f \}(x_2-y)
\ = \
\{\nabla (\rho_n*f)\}(\xi) \cdot (x_1-x_2)
\]
(qui con $\cdot$ intendiamo il prodotto scalare in $\bR^d$). Per cui, 
\[
\begin{array}{ll}
| \{ \rho_n*f \}(x_1) - \{ \rho_n*f \}(x_2) | 
& \ \ \stackrel{*}{\leq} \ \
|\{\nabla (\rho_n*f)\}(\xi)| \ |x_1-x_2|
\\ & \stackrel{(\ref{eq_nabla})}{=} \
|\{ (\nabla \rho_n)*f)\}(\xi)| \ |x_1-x_2|
\\ & \ \ \stackrel{**}{\leq} \ \
\| \nabla \rho_n \|_q \| f \|_p \ |x_1-x_2|\
\ ,
\end{array}
\]
avendo usato le diseguaglianze di Cauchy-Schwarz per (*) e di Holder per (**), quest'ultima usata in modo analogo a (\ref{eq_RFK2a}).

\noindent \textbf{Passo 4.} Preso $\eps > 0$ esiste $\Omega_\eps \subseteq \Omega$ limitato e misurabile tale che 
\begin{equation}
\label{eq_RFK4}
\| f|_{\Omega-\Omega_\eps} \|_p 
\ < \ 
\eps
\ \ , \ \
\forall f \in \mF
\ .
\end{equation}
Per questo basta osservare che 
\[
\| f|_{\Omega-\Omega_\eps} \|_p 
\ \leq \ 
\| f-\rho_n*f \|_p + \| (\rho_n*f) |_{\Omega-\Omega_\eps} \|_p
\ ,
\]
cosicch\'e per $n > \delta^{-1}$ da (\ref{eq_RFK1}) deduciamo
\[
\| f|_{\Omega-\Omega_\eps} \|_p 
\ \leq \ 
\eps + \| (\rho_n*f) |_{\Omega-\Omega_\eps} \|_p 
\ \leq \
\eps + \| \rho_n*f \|_\infty \ {\mathrm{vol}}(\Omega - \Omega_\eps)^{1/p}
\ ,
\]
e grazie a (\ref{eq_RFK2}), ed alla limitatezza di $\mF$, basta scegliere $\Omega_\eps$ in maniera tale che ${\mathrm{vol}}(\Omega - \Omega_\eps)$ sia sufficientemente piccolo.

\noindent \textbf{Passo 5.} Fissato $\Omega_\eps$ ed $n > \delta^{-1}$ come nei passi precedenti consideriamo la famiglia 
\[
\mF_{n,\eps} := 
\{ (\rho_n*f)|_{\ovl{\Omega}_\eps} : f \in \mF \} \subset 
C^\infty(\ovl{\Omega}_\eps)
\ ;
\]
grazie a (\ref{eq_RFK2}) e (\ref{eq_RFK3}) concludiamo che $\mF_{n,\eps}$ \e limitata ed equicontinua in $C(\ovl{\Omega}_\eps)$. Per Ascoli-Arzel\'a $\mF_{n,\eps}$ \e precompatto nella topologia della norma dell'estremo superiore, e quindi lo \e anche in $L^p(\ovl{\Omega}_\eps)$ (infatti la convergenza in $\| \cdot \|_\infty$ in $C(\ovl{\Omega}_\eps)$ implica quella in $\| \cdot \|_p$).

\noindent \textbf{Conclusione.} Sia $\{ f_i \} \subseteq \mF$. Fissati $\eps,n$ come sopra, grazie al passo precedente esiste una sottosuccessione 
$\{ (\rho_n*f_{i_k})|_{\Omega_\eps} \}$
di Cauchy in $L^p(\ovl{\Omega}_\eps)$ (per brevit\'a scriviamo $f_k \equiv f_{i_k}$). Usando (\ref{eq_RFK1}), (\ref{eq_RFK2}), (\ref{eq_RFK4}), abbiamo le stime
\[
\begin{array}{ll}
\| (f_h - f_k)|_\Omega \|_p   & \leq   
\| (f_h - \rho_n*f_h)|_\Omega  \|_p +
\| (\rho_n*f_h - \rho_n*f_k)|_\Omega \|_p +
\| (\rho_n*f_k - f_k)|_\Omega  \|_p 
\\ & \leq \
2 \eps + \| (\rho_n*f_h - \rho_n*f_k)|_\Omega \|_p
\\ & \leq \
2 \eps + 
\| (\rho_n*f_h - \rho_n*f_k)|_{\Omega-\Omega_\eps} \|_p +
\| (\rho_n*f_h - \rho_n*f_k)|_{\Omega_\eps} \|_p
\\ & \leq \
3 \eps + \| (\rho_n*f_h - \rho_n*f_k)|_{\Omega-\Omega_\eps} \|_p
\\ & \leq \
3 \eps + c_n ( \| f_h|_{\Omega-\Omega_\eps} \|_p + \| f_k|_{\Omega-\Omega_\eps} \|_p  )
\\ & \leq \
3 \eps + c_n \cdot 2 \eps
\ .
\end{array}
\]
Concludiamo quindi che $\{ f_k|_\Omega \}$ \e di Cauchy in $L^p(\Omega)$, e ci\'o mostra il teorema.
\end{proof}

\subsection{Gli spazi $L^p_{loc}$.}

Sia $(X,\mM,\mu)$ uno spazio misurabile boreliano e $p \in [1,+\infty]$. Una funzione misurabile $f : X \to \bR$ si dice {\em localmente $L^p$} se per ogni $x \in X$ esiste un intorno misurabile $U \ni x$ tale che
\[
f |_U \in L_{\mu_U}^p(U) 
\ \Leftrightarrow \
f \chi_U \in L_\mu^p(X)
\]
(vedi Oss.\ref{def_MIS02}). Denotiamo con $L_{\mu,loc}^p(X)$ l'insieme delle funzioni localmente $L^p$ (nel caso della misura di Lebesgue seguiremo le analoghe convenzioni degli spazi $L^p$).

Poich\'e per ipotesi ogni aperto \e misurabile, nella definizione precedente non \e restrittivo considerare intorni aperti (se l'intorno $U$ non \e aperto consideriamo un aperto $U' \subset U$). E' ovvio che se $f \in L_\mu^p(X)$ allora $f$ \e localmente $L^p$. Il viceversa \e falso: 

\begin{ex}
{\it
Sia $f(x) := x^{-1}$, $x \in (0,1)$. Allora \e chiaro che $f \notin L^p(0,1)$ per ogni $p \in [1,+\infty]$. Invece troviamo $f \in L_{loc}^p(0,1)$, visto che, preso $x \in (0,1)$ e $0 < a < x < b \leq 1$,
\[
+ \infty \ > \
\int_a^b |f(x)|^p \ dx \ = \ 
\left\{
\begin{array}{ll}
\frac{1}{p-1} \ \left| b^{1-p} - a^{1-p}  \right|
\ \ , \ \
p > 1
\\
\log b - \log a 
\ \ , \ \ 
p = 1 \ .
\end{array}
\right.
\]
}
\end{ex}

\begin{prop}
Sia $(X,\mM,\mu)$ uno spazio di misura di Radon, $p \in [1,+\infty]$ ed $f \in L_{\mu,loc}^p(X)$. 
Allora per ogni compatto $K \subseteq X$ si ha $f \chi_K \in L_\mu^p(X)$.
\end{prop}

\begin{proof}[Dimostrazione]
Per ogni $x \in K$ esiste un intorno aperto $U \ni x$ con $f \chi_U \in L_\mu^p(X)$, per cui la famiglia $\{ U \}$ costituisce un ricoprimento aperto di $K$. Consideriamo un sottoricoprimento finito $\{ U_k \}_{k=1}^n$ ed osserviamo che
\[
\int_X |f \chi_K |^p
\ = \
\int_X |f|^p \chi_K 
\ \leq \
\sum_k^n \int_X |f|^p \chi_{U_k}
\ < \
+ \infty
\ .
\]
\end{proof}

\begin{cor}
\label{cor_Lploc_LC}
Per ogni $p \in [1,+\infty]$, lo spazio $L_{\mu,loc}^p(X)$ ha una struttura di spazio localmente convesso{\footnote{Vedi \S \ref{sec_loc_conv}.}} rispetto alle seminorme
\[
\eta_U (f) := \left( \int_U |f|^p \right)^{1/p}
\ \ , \ \
U \in \tau X \ : \ \ovl U {\mathrm{ \ compatto}}
\ .
\]
\end{cor}

\begin{proof}[Dimostrazione]
Se $f,g \in L_{\mu,loc}^p(X)$ allora per ogni $x$ esistono intorni $U,V$ di $x$ con $f \chi_U \in L_\mu^p(X)$, $g \chi_V \in L_\mu^p(X)$. Preso un aperto $W \subseteq U \cap V$ con $x \in W$, troviamo $(f+g)\chi_W \in L_\mu^p(X)$, per cui concludiamo che $L_{\mu,loc}^p(X)$ \e uno spazio vettoriale. Inoltre, dalla proposizione precedente segue che $\eta_U(f) < + \infty$ per ogni intorno aperto $U$ a chiusura compatta. Chiaramente ogni $\eta_U$ \e una seminorma, e se $\eta_U(f) = 0$ per ogni precompatto $U$ allora $f = 0$ q.o..
\end{proof}

Diamo ora un'applicazione del concetto di funzione localmente $L^p$. Denotato con $L_c^p(\bR^d)$ lo spazio delle funzioni in $L^p(\bR^d)$ a supporto compatto, osserviamo che ha senso definire la {\em convoluzione} (vedi \S \ref{sec_conv})
\begin{equation}
\label{eq_conv_SC}
f*g (x) := \int_\bR f(x-y) g(y) \ dy
\ \ , \ \
x \in \bR^d
\ , \
f \in L_{loc}^1(\bR^d)
\ , \
g \in L_c^p(\bR^d)
\ ,
\end{equation}
infatti $f * g (x) = f|_{{\mathrm{supp}}g}*g(x)$ ed $f|_{{\mathrm{supp}}g} \in L^1(\bR^d)$; usando il Teorema \ref{thm_conv}, concludiamo che $f*g \in L^p(\bR^d)$.

\subsection{La dualit\'a di Riesz.}

In accordo con la notazione che verr\'a introdotta in \S\ref{sec_afunct}, per ogni $p \in [1,+\infty]$ denotiamo con $L^{p,*}_\mu(X)$ lo spazio duale di $L_\mu^p(X)$. L'osservazione alla base dei risultati che esporremo in questa sezione \e la seguente: presi $p \in [1,+\infty]$, $q := \ovl{p}$ e $g \in L_\mu^q(X)$ abbiamo l'applicazione lineare
\begin{equation}
\label{def_Fg}
F_g : L_\mu^p(X) \to \bR
\ \ , \ \
F_g(f) := \left \langle F_g , f \right \rangle := \int_X fg 
\ , \
\forall f \in L_\mu^p(X)
\ ;
\end{equation}
con la precedente notazione la diseguaglianza di Holder implica, chiaramente, che 
\[
| \left \langle F_g , f \right \rangle | \ \leq \ \| g \|_q \| f \|_p \ ,
\]
per cui, nella terminologia di \S \ref{sec_afunct}, abbiamo $F_g \in L^{p,*}_\mu(X)$ con norma $\| F_g \| \leq \| g \|_q$. In realt\'a si verifica che la norma di $F_g$ \e proprio $\| g \|_q$ (vedi Esercizio \ref{sec_Lp}.7), per cui abbiamo un'applicazione lineare isometrica 
\begin{equation}
\label{def_dual_riesz}
F : L_\mu^q(X) \to L^{p,*}_\mu(X) 
\ \ , \ \
g \mapsto F_g
\ .
\end{equation}
Nel caso complesso, si pu\'o analogamente definire
\[
F : L_\mu^q(X,\bC) \to L^{p,*}_\mu(X,\bC)
\ \ , \ \
\left \langle F_g , f \right \rangle := \int_X f \ovl g \ .
\]

\begin{lem}
\label{lem_Riesz1}
Sia $(X,\mM,\mu)$ uno spazio di misura finita e $p \in [1,+\infty)$. Se $g \in L_\mu^1(X)$ ed esiste $M \in \bR$ tale che
\begin{equation}
\label{eq_Riesz1}
\left| \int fg \right| 
\ \leq \
M \| f \|_p
\ \ , \ \
\forall f \in L_\mu^\infty(X)
\ ,
\end{equation}
allora $g \in L_\mu^q(X)$ con $q = \ovl{p}$, e $\| g \|_q \leq M$.
\end{lem}

\begin{proof}[Dimostrazione]
Iniziamo considerando il caso $p > 1$. Definiamo
\[
g_n(x) :=
\left\{
\begin{array}{ll}
g(x) \ \ , \ \ |g(x)| \leq n
\\
0    \ \ , \ \ |g(x)| >    n
\end{array}
\right.
\ \ \Rightarrow \ \
g_n \in L_\mu^\infty(X)
\ .
\]
Osserviamo che $gg_n = g_n^2$ e ${\mathrm{sgn}} (g_n) = {\mathrm{sgn}} (g)$
{\footnote{La funzione segno ${\mathrm{sgn}}$ \e definita nell'Esercizio \ref{sec_Lp}.7, che invitiamo a risolvere per meglio maneggiare i conti della presente dimostrazione.}}. 
Definiamo quindi $f_n := |g_n|^{q/p} {\mathrm{sgn}} (g_n)$, cosicch\'e $\| f_n \|_p = \| g_n \|_q^{q/p} = \| g_n \|_q^{q-1}$. Dunque per costruzione $g_n \in L_\mu^q(X) \cap L_\mu^\infty(X)$ ed $f_n \in L_\mu^p(X) \cap L_\mu^\infty(X)$, e possiamo stimare
\[
\| g_n \|_q^q
=
\int |g_n|^{q/p+1}
=
\int |g_n|^{q/p} |g|
=
\int |f_n||g|
=
\left| \int f_n g  \right|
\leq 
M \| f_n \|_p
= 
M \| g_n \|_q^{q-1}
\ ,
\]
per cui $\| g_n \|_q \leq M$. Ora, $\{ g_n \}$ \e una successione monot\'ona crescente e convergente a $g$ q.o. in $x \in X$ (ed analogamente $|g_n|^q \nearrow |g|^q$ q.o.); per cui, applicando il teorema di Beppo Levi concludiamo
\[
M^q 
\geq
\int |g_n|^q 
\stackrel{n}{\to}
\int |g|^q
\ \ \Rightarrow \ \ 
\int |g|^q \leq M^q
\ .
\]
Discutiamo ora il caso $p = 1$. Per ogni $\eps > 0$, poniamo $A := \{ x \in X : |g(x)| \geq M + \eps \}$, e, posto $f := {\mathrm{sgn}}(g) \chi_A$, otteniamo $\| f \|_1 = \mu A$. Ora,
\[
\left| \int fg \right| 
\leq
M \| f \|_1
= M \mu A
\ \ , \ \
\left| \int fg \right| 
\geq
\int_A |g|
\geq 
(M+\eps) \mu A
\ .
\]
Dalle due precedenti diseguaglianze, concludiamo che $\mu A = 0$ e quindi $\| g \|_\infty \leq M$.
\end{proof}

\begin{thm}[Dualit\'a di Riesz]
Sia $(X,\mM,\mu)$ uno spazio di misura $\sigma$-finita. Per ogni $p \in [1,+\infty)$, l'applicazione (\ref{def_dual_riesz}) definisce un isomorfismo di spazi di Banach $L_\mu^q(X) \to L^{p,*}_\mu(X)$.
\end{thm}

\begin{proof}[Dimostrazione]
Sappiamo che (\ref{def_dual_riesz}) \e isometrica (e quindi iniettiva), per cui dobbiamo verificarne soltanto la suriettivit\'a. Consideriamo allora $F \in L^{p,*}_\mu(X)$ e mostriamo che appartiene all'immagine di (\ref{def_dual_riesz}).
Come primo passo, assumiamo che \textbf{$X$ abbia misura finita}, e definiamo
\[
\nu A := \left \langle F , \chi_A \right \rangle 
\ \ , \ \
A \in \mM
\ .
\]
Poich\'e $\mu X < \infty$ abbiamo $\| \chi_A \|_p = (\mu A)^{1/p} < \infty$, per cui $\chi_A \in L_\mu^p(X)$ e la definizione precedente \e ben posta. La strategia \e ora quella di mostrare che $\nu$ \e una misura con segno. A tale scopo, osserviamo che $\nu \emptyset = \left \langle F , 0 \right \rangle = 0$; inoltre, se $A \cap B = \emptyset$ allora $\chi_{A \cup B} = \chi_A + \chi_B$ e $\nu (A \cup B) = \nu A + \nu B$. Per mostrare l'additivit\'a numerabile, consideriamo una successione $\{ A_k \}$ di insiemi misurabili mutualmente disgiunti, ed osserviamo che per il teorema di Lebesgue risulta
\[
0 =
\lim_n \| \chi_E - \chi_{E_n}  \|_p =
\lim_n \mu (E-E_n)^{1/p}
\ \ , \ \
E := \cup_k^\infty A_k
\ , \
E_n := \cup_k^n A_k
\ .
\]
Essendo $F$ continuo, concludiamo che
\[
0 =
\lim_n \left \langle F , \chi_E - \chi_{E_n} \right \rangle =
\nu E - \lim_n \nu E_n
\ .
\]
Dunque $\nu : \mM \to \bR$ \e una misura con segno. Ora, se $\mu A = 0$, $A \in \mM$, allora $\| \chi_A \|_p = 0$ e $\nu A = \left \langle F , \chi_A \right \rangle = 0$. Dunque $\nu$ \e assolutamente continua rispetto a $\mu$, ed il teorema di Radon-Nikodym implica che esiste $g \in L_\mu^1(X)$ tale che 
\begin{equation}
\label{eq_R01}
\left \langle F , \chi_A \right \rangle
=
\int_A g \ d \mu
\ \ , \ \
A \in \mM
\ .
\end{equation}
Per cui, per ogni $\varphi \in S(X)$ troviamo 
$\left \langle F , \varphi \right \rangle = \int \varphi g \ d \mu$.
Poich\'e $S(X)$ \e denso in $L_\mu^\infty(X)$ in norma $\| \cdot \|_\infty$, concludiamo che
\[
\left \langle F , f \right \rangle
=
\int fg \ d \mu
\ \ , \ \
f \in L_\mu^\infty(X)
\ \ \Rightarrow \ \
\left| \int fg \ d \mu \right|
\leq
\| F \| \| f \|_p
\ .
\]
Applicando il Lemma precedente concludiamo che $F = F_g$. 
Passiamo ora al caso in cui $X$ \e \textbf{$\sigma$-finito}: consideriamo una successione $\{ X_i \}$ di insiemi a misura finita e mutualmente disgiunti tali che $X = \cup_i X_i$; osserviamo quindi che -- grazie al passo precedente -- per ogni $n \in \bN$ esistono $g_1 , \ldots , g_n \in L_\mu^q(X)$ (con supporti contenuti rispettivamente in $X_1 , \ldots , X_n$) tali che
\[
\left \langle F , f \chi_{A_n} \right \rangle
=
\int f \sum_i^n g_i \ d \mu
\ \ , \ \
f \in L_\mu^p(X)
\ \ , \ \
A_n := \cup_i^n X_i
\ .
\]
In altri termini, ogni $F^{(n)} := F |_{L_\mu^p(A_n)}$, $n \in \bN$, \e il funzionale associato a $\sum_i^n g_i \in L_\mu^q(A_n)$ attraverso l'applicazione (\ref{def_dual_riesz}){\footnote{Qui abbiamo scritto, per semplicit\'a di notazione, $\mu \equiv \mu_{A_n}$.}}. Poich\'e chiaramente $\| F^{(n)} \| \leq \| F \|$ per ogni $n \in \bN$, concludiamo che
\[
\|  \sum_i^n g_i  \|_q
=
\| F^{(n)} \|
\leq
\| F \|
\ \ , \ \
n \in \bN
\ .
\]
Definiamo ora $g (x) := \sum_i^\infty g_i(x)$, $x \in X$. Usando il Lemma di Fatou, per ogni $n \in \bN$ troviamo
\[
\| g \|_q
\leq
\lim \inf_n \| \sum_i^n g_i \|_q
=
\lim \inf_n \| F^{(n)} \|
\leq
\| F \|
\ .
\]
Dunque $g \in L_\mu^q(X)$. Ora, per il teorema di convergenza di Lebesgue troviamo $\| f - f \chi_{A_n} \|_p \stackrel{n}{\to} 0$, per cui, per continuit\'a di $F$ abbiamo che
\[
\left \langle F , f \right \rangle                   =
\lim_n \left \langle F       , f \chi_{A_n} \right \rangle =
\lim_n \left \langle F^{(n)} , f \chi_{A_n} \right \rangle =
\lim_n \int f \sum_i^n g_i \ d \mu
\ .
\]
D'altro canto, usando ancora Lebesgue troviamo $\| g - \sum_i^n g_i \|_q \stackrel{n}{\to} 0$; per la diseguaglianza di Holder, concludiamo
\[
\int f \left( g -  \sum_i^n g_i \right) \ d \mu
\leq 
\| f \|_p \| g -  \sum_i^n g_i \|_q
\stackrel{n}{\to} 0
\ \Rightarrow \
\int fg \ d \mu = 
\lim_n \int f \sum_i^n g_i = 
\left \langle F , f \right \rangle 
\ .
\]
\end{proof}

\begin{rem}{\it
Il teorema di dualit\'a di Riesz rimane valido nel caso in cui $X$ abbia misura qualsiasi, a condizione che $p$ sia strettamente maggiore di $1$ (vedi \cite[Teo.8.5.12]{dBar} per i dettagli; per un controesempio al caso $p=1$ vedi \cite[Oss.9.2.3]{Acq}). Preso $F \in L^{p,*}_\mu(X)$, l'idea \e quella di costruire un insieme misurabile e $\sigma$-finito $X_0 \subseteq X$ tale che 
\[
f |_{X_0} = 0 
\ \Rightarrow \
\left \langle F , f \right \rangle = 0
\ \ , \ \
\forall f \in L_\mu^p(X)
\ ;
\]
osserviamo che in tal modo ogni $f_0 \in L_\mu^p(X_0)$ si scrive come $f_0 = f|_{X_0}$, $f \in L_\mu^p(X)$, avendo posto $f|_{X-X_0} := 0$, $f|_{X_0} := f_0$, cosicch\'e \e ben definito il funzionale
\[
F_0 \in L^{p,*}_\mu(X_0)
\ \ , \ \
\left \langle F_0 , f_0 \right \rangle := \left \langle F , f \right \rangle
\ \ , \ \
\forall f_0 = f|_{X_0} \in L_\mu^p(X_0)
\ .
\]
Scelta una successione $\{ f_n \} \subset L_\mu^p(X)$, $\| f_n \|_p \equiv 1$, tale che
\[
\| F \| ( 1 - 1/n ) 
\leq  
\left \langle F , f_n \right \rangle
\ \ , \ \
n \in \bN
\ ,
\]
il nostro candidato \e
\[
X_0 \ := \ \bigcup_n \ \{ x \in X : f_n(x) \neq 0 \} \ ,
\]
che sappiamo essere $\sigma$-finito grazie a Prop.\ref{lem_sf}. Al che si applica il teorema del caso $\sigma$-finito, il che fornisce una funzione $g \in L_\mu^q(X_0)$ tale che 
$F_0 = F_g$. 
Estendendo $g$ ad $X$ ponendo $g \equiv 0$ in $X-X_0$ otteniamo la funzione $g \in L_\mu^q(X)$ cercata.
}
\end{rem}

Alla regola $L^{p,*}_\mu(X) = L_\mu^q(X)$, $p \in [1,+\infty)$, $q = \ovl p$, fa eccezione lo spazio $L_\mu^\infty(X)$, il cui duale contiene {\em strettamente} $L_\mu^1(X)$ (vedi Esercizio \ref{sec_afunct}.2). In effetti, il duale $L_\mu^{\infty,*}(X)$ si pu\'o caratterizzare in termini di misure di Radon (vedi \S \ref{sec_MIS1}); tale caratterizzazione richiede due risultati fondamentali: il teorema di Riesz-Markov (Esempio \ref{ex_dual_CX}), ed il teorema di Gel'fand-Naimark (Teo.\ref{thm_GN}):

\begin{prop}
\label{prop_dual_Li}
Sia $(X,\mM,\mu)$ uno spazio misurabile. Allora esiste uno spazio compatto e di Hausdorff $X_\mu$ con un isomorfismo $L_\mu^{\infty,*}(X) \simeq R(X_\mu)$.
\end{prop}

\begin{proof}[Dimostrazione]
Immergiamo $L_\mu^\infty(X)$ in $L_\mu^\infty(X,\bC)$ ed osserviamo che quest'ultima \e una \sC algebra commutativa con identit\'a (Esempio \ref{ex_LiC}). Denotato con $X_\mu$ lo spettro di $L_\mu^\infty(X,\bC)$ nel senso di Def.\ref{def_GN}, abbiamo un isomorfismo di spazi di Banach (reali) $L_\mu^\infty(X) \simeq C(X_\mu)$. La tesi segue dunque dal teorema di Riesz-Markov. 
\end{proof}

Sull'argomento precedente, si veda anche \cite[\S 4.3.C]{Bre} (e referenze).

\

\noindent \textbf{La dualit\'a di Riesz su $[0,1]$.}
E' possibile dimostrare la dualit\'a di Riesz sull'intervallo $[0,1]$ usando delle tecniche diverse rispetto alla sezione precedente, che fanno ricorso alle funzioni assolutamente continue piuttosto che al teorema di Radon-Nicodym:
\begin{thm}[Teorema di rappresentazione di Riesz sugli intervalli]
\label{thm_Riesz}
Per ogni $p \in [1,+\infty)$ e $q := \ovl{p}$ si ha l'isomorfismo di spazi di Banach
\begin{equation}
\label{eq_Riesz0}
L^q \to L^{p,*}
\ \ , \ \
g \mapsto F_g
\ .
\end{equation}
\end{thm}

\begin{proof}[Dimostrazione]
Preso $F \in L^{p,*}$ dimostriamo che $F = F_g$ per qualche $g \in L^q$. A tale scopo consideriamo le funzioni caratteristiche $\chi_s := \chi_{[0,s]}$ e definiamo
\[
G(s) := \left \langle F , \chi_s \right \rangle
\ \ , \ \
s \in [0,1]
\ .
\]
Volendo mostrare che $G \in AC([0,1])$, consideriamo una collezione $C :=$ $\{ ( x_i , x'_i ) \}$ di intervalli disgiunti con lunghezza totale $l(C)$ minore di $\delta > 0$ ed osserviamo che
\[
\sum_i | G(x_i) - G(x'_i) | \ = \
\sum_i | \left \langle F , \chi_{(x_i,x'_i)} \right \rangle | \ = \
\left \langle F , f \right \rangle
\ ,
\]
dove
\[
f 
:= 
\sum_i \chi_{( x_i,x'_i)} \ 
{\mathrm{sgn}} \left( \left \langle F , \chi_{( x_i,x'_i)} \right \rangle \right)
\ \Rightarrow \
\| f \|_p = \delta^{1/p}
\ .
\]
Per cui,
\[
\sum_i | G(x_i) - G(x'_i) | 
\ \leq  \
\| F \| \| f \|_p
\ \leq \
\| F \| \delta^{1/p}
\ .
\]
Dunque $G$ \e AC e quindi $G(x) = \int_0^x g$ per qualche $g \in L^1$ (Teo.\ref{thm_AC2}), il che fa di $g$ il nostro candidato per avere $F = F_g$. Per definizione troviamo
\[
G(s) \ = \ 
\left \langle F , \chi_s \right \rangle 
\ = \ 
\int_0^1 g \chi_s
\ ,
\]
il che mostra che $F |_\mG = F_g$, dove $\mG$ \e lo spazio delle funzioni a gradini. Ora, ogni funzione $f \in L^\infty$ \e limite q.o. di una successione $\{ \psi_n \}$ di funzioni a gradini (Prop.\ref{prop_app_mis}), equilimitata da $\| f \|_\infty$. Per cui, il teorema di convergenza limitata (Prop.\ref{prop_conv_lim}) implica $\lim_n \| f - \psi_n \|_p = 0$. Poich\'e $F$ \e limitato, troviamo 
\[
\| \left \langle F , f-\psi_n \right \rangle \|_p
\leq
\| F \| \| f-\psi_n \|_p
\stackrel{n}{\to} 0
\ \Rightarrow \
\left \langle F , f \right \rangle 
= 
\lim_n \int_0^1 g \psi_n \stackrel{(*)}{=} \int fg
\]
dove in (*) si \e usato il teorema di convergenza di Lebesgue per la successione $|g \psi_n| \leq |gf| \in L^1$. Dunque $F |_{L^\infty} = F_g$. Possiamo applicare ora Lemma \ref{lem_Riesz1}, dal quale concludiamo $g \in L^q$. I risultati di approssimazione in $L^p$ (\S \ref{sec_appLp}) permettono infine di dimostrare $F = F_g$ su tutto $L^p$.  
\end{proof}

\subsection{Esercizi.}

\textbf{Esercizio \ref{sec_Lp}.1.} {\it Sia $f \in L^\infty$. Si mostri che $\lim_{p \to \infty} \| f \|_p =$ $\| f \|_\infty$.}

\

\noindent \textbf{Esercizio \ref{sec_Lp}.2 (Continuit\'a delle convoluzioni).} {\it Sia $f \in L^p(\bR)$, $p \in [1,+\infty)$, ed $f_h(x) := f(x+h)$, $x,h \in \bR$. Si provi che la funzione
\[
n_f(h) := 
\left( \int | f(x+h) - f(x) |^p dx \right)^{1/p} 
\equiv \
\| f_h - f \|_p
\ \ , \ \ 
\forall h \in \bR \ ,
\]
\'e continua. Inoltre, presi $q := \ovl p$ e $g \in L^q(\bR)$, si provi che \e ben definita, continua e limitata la funzione (detta \textbf{convoluzione}, si veda \S \ref{sec_conv})
\[
f*g(x) := \int f(x-y) g(y) \ dy \ \ , \ \ x \in \bR \ .
\]
}

\noindent {\it Soluzione.} $n_f$ \e ben definita in quanto la funzione integranda appartiene ad $L^p(\bR)$ per ogni $h \in \bR$; inoltre, osserviamo che  
\[
| n_f (h) - n_f (k) | \leq 
\| f_h - f_k \|_p    =
\| f_{h-k} - f \|_p  =
n_f (h-k)
\ ,
\]
per cui \e sufficiente verificare la continuit\'a soltanto in $h=0$. A tale scopo, assumiamo inizialmente che $f \in C_c(\bR)$, cosicch\'e $f$ \e uniformemente continua; scelto $\eps > 0$ esiste $\delta > 0$ tale che $|f_h(x) - f(x)| < \eps$ per $|h| < \delta$, e si trova
\[
n_f (h) < \eps \mu ( {\mathrm{supp}}(f) )^{1/p}
\ \ , \ \
|h| < \delta \ .
\]
Per cui, $n_f$ \e continua. Per una generica $f \in L^p(\bR)$, osserviamo che per ogni $\eps > 0$ esiste $g_\eps \in C_c(\bR)$ con $\| f - g_\eps \|_p < \eps$. Per cui, per $h$ abbastanza piccolo troviamo
\[
n_f (h) \leq
\| f_h-g_{\eps,h}      \|_p +
\| g_{\eps,h} - g_\eps \|_p +
\| g_\eps - f \|_p \leq
3 \eps
\ .
\]
Infine, valutiamo
\begin{equation}
\label{eq_DH}
|f*g(x)| 
\ \leq \ 
\int | f(x-y)||g(y) | \ dy 
\stackrel{Holder}{\leq}
\| f_x \|_p \| g \|_q
\ = \
\| f \|_p \| g \|_q \ .
\end{equation}
Per cui $f*g(x)$ \e ben definita e limitata per ogni $x \in \bR$. Per quanto riguarda la continuit\'a, abbiamo
\[
| f*g (x) - f*g(x_0) | 
\ \leq \
\int | f(x-y) - f(x_0-y) | |g(y)| \ dy
\ \stackrel{Holder}{\leq} \
\| g \|_q \| f_x - f_{x_0} \|_p  
\ = \
\| g \|_q n_f (x-x_0)
\ ,
\]
per cui, grazie alla continuit\'a di $n_f$, concludiamo che $f*g$ \e continua.

\

\noindent \textbf{Esercizio \ref{sec_Lp}.3.} {\it Sia $f \in L^p(\bR)$, $p \in [1,+\infty)$. Allora, $f_n := f \chi_{[-n,n]}$, $n \in \bN$, \e una funzione in $L^1(\bR) \cap L^p(\bR)$.}

\

\noindent {\it Soluzione.} E' ovvio che $f_n \in L^p(\bR)$. Poi, si osservi che
\[
\int |f_n| = \int_{-n}^n |f| 
\ \stackrel{Holder}{\leq} \
\left( \int_{-n}^n |f(x)|^p \right)^{1/p} \cdot \| \chi_{[-n,n]} \|_q
\ \leq \
\| f \|_p  (2n)^{1/q}
\ .
\]

\

\noindent \textbf{Esercizio \ref{sec_Lp}.4.} {\it Sia 
$\mX := \{ u \in C^1([0,1])  :  u(0) = 0   \}$.
Per ogni $\lambda \geq 0$ si consideri funzionale
\[
F_\lambda : \mX \to \bR
\ \ , \ \
F_\lambda(u) 
\ = \
\arctan \left( \int_0^1 [u'(t)]^2 \ dt \right) 
-
\lambda \arctan (u(1))
\ .
\]
\textbf{(1)} Usando la diseguaglianza di Holder si mostri che
\[
F_\lambda(u) \ \geq \ \arctan(u(1)^2) - \lambda \arctan (u(1)) 
\ \ , \ \
\forall u \in \mX
\ ;
\]
\textbf{(2)} Si trovi una famiglia $\mM$ di funzioni in $\mX$ tali che la diseguaglianza precedente si riduce ad un'eguaglianza; 
\textbf{(3)} Si calcoli $\min_{u \in \mM} F_\lambda(u)$ al variare di $\lambda$. 

\

\noindent (Suggerimento: la diseguaglianza di Holder implica che 
$\int (u')^2 \geq (\int u')^2 = u(1)^2$.
Una classe di funzioni tali che la diseguaglianza precedente si riduce ad un'eguaglianza \e data da quelle del tipo $u_\alpha(t) := \alpha t$, $t \in [0,1]$, dove $\alpha \in \bR$, e ci\'o consente di ridurre il nostro problema allo studio della funzione $g(\lambda,\alpha) := F_\lambda(u_\alpha)$).

}

\

\noindent \textbf{Esercizio \ref{sec_Lp}.5.} {\it Dimostrare che la naturale inclusione (con $L^p := L^p([0,1])$, $p \in [1,\infty]$)
\[
L^\infty \subset \bigcap_{p \in [1,\infty)} L^p
\]
\'e stretta, ovvero che esistono funzioni non limitate ma in $L^p$ per ogni $p \in [1,\infty)$.}

\

\noindent \textbf{Esercizio \ref{sec_Lp}.6 (\cite[Ex.3.12]{Can}).} {\it Si consideri $X := [0,1)$ equipaggiato con la misura di Lebesgue e la successione $\{ f_{in} \}_{i,n \in \bN} \subset L^p(X)$, $p \in [1,+\infty]$, definita da
\[
f_{in}(x) := 
\left\{
\begin{array}{ll}
1  \ \ , \ \ x \in [\  (i-1)n^{-1} \ , \ in^{-1} \ )
\\
0  \ \ , \ \ {{altrimenti \ .}}
\end{array}
\right.
\]
Si verifichi che $\{ \ldots , f_{n1} , f_{n2} , \ldots , f_{nn} , f_{n+1,1} \ldots \}$ converge a $0$ in $L^p(X)$, ma non q.o.. Si verifichi invece che $f_{1n} \stackrel{n}{\to} 0$ q.o..
}

\

\noindent \textbf{Esercizio \ref{sec_Lp}.7.} {\it Sia $(X,\mM,\mu)$ uno spazio di misura e $g \in M(X)$.
\textbf{(1)} Si consideri la funzione segno 
\[
\{ {\mathrm{sgn}}(g) \} (x) \ := \
\left\{
\begin{array}{ll}
1 \ , \ g(x) \geq 0 \ ,
\\
-1 \ , \ g(x) < 0 \ ,
\end{array}
\right.
\]
e si mostri che ${\mathrm{sgn}}(g) \in M(X)$.
\textbf{(2)} Si verifichi che se $g \in L_\mu^q(X)$ allora 
$f := |g|^{q/p}{\mathrm{sgn}}(g) \in L_\mu^p(X)$, dove $p := \ovl{q}$.
\textbf{(3)} Si mostri che $\| f \|_p = \| g \|_q^{q/p}$.
\textbf{(4)} Si mostri che 
$\left \langle F_g , f \right \rangle = \| g \|_q \| f \|_p$.

\

\noindent (Suggerimento: si osservi che 
$\int |f|^p = \int |g|^q$,
$q = 1 + q/p$, 
e
$\left \langle F_g , f \right \rangle = \int fg = \int |g|^q = \| g \|_q^q$.)
}

\

\noindent \textbf{Esercizio \ref{sec_Lp}.8 (Delta-approssimanti).} {\it Sia $x := \{ x_k  \} \subset [0,1]$ una successione monot\'ona crescente tale che $x_1 = 0$ e $\lim_k x_k = 1$.
\textbf{(1)} Presa $f \in L^p$, si mostri che la funzione a gradini
\[
\{ S_x f \}(t) 
\ := \ 
\frac{1}{x_{k+1}-x_k} \int_{x_k}^{x_{k+1}} f(s) \ ds
\ \ , \ \
\forall k \in \bN \ , \  t \in [x_k,x_{k+1})
\ ,
\]
appartiene ad $L^p$ e si concluda che  
\[
S_x : L^p \to L^p
\ \ , \ \
f \mapsto S_xf 
\ ,
\]
\e un'applicazione lineare tale che $\| S_x f \|_p \leq \| f \|_p$;
\textbf{(2)} Diciamo che $x$ \textbf{ha passo $\delta$} se $x_{k+1} - x_k < \delta$ per ogni $k \in \bN$.
Si mostri che per ogni $f \in C([0,1])$ ed $\eps > 0$ esiste $\delta > 0$ tale che, per ogni successione $x$ con passo $\delta$, risulta
\[
\sup_{t \in [0,1] - x}| f(t) - \{ S_xf \}(t) | \ < \ \eps \ ;
\]
\textbf{(3)} Si concluda che lo spazio vettoriale delle funzioni a gradini \e denso in $L^p$.

\

\noindent (Suggerimenti: per il punto (1) si osservi che per $t \in [x_k,x_{k+1})$
\[
| \{ S_xf \}(t) |^p \ = \ 
\frac{1}{(x_{k+1}-x_k)^p} \left| \int_{x_k}^{x_{k+1}} f \right|^p \ = \
\left| \int_{x_k}^{x_{k+1}} f \ d \mu \right|^p \ \stackrel{Jensen}{\leq} \
\frac{1}{(x_{k+1}-x_k)}  \int_{x_k}^{x_{k+1}} |f|^p
\ ,
\]
avendo definito la misura di probabilit\'a $d \mu(s) = ds / (x_{k+1}-x_k)$, per cui
\[
\| S_xf \|_p^p \ = \ 
\sum_k (x_{k+1}-x_k) \cdot |S_xf|^p |_{ \ [x_k,x_{k+1})} \ \leq \
\sum_k \int_{x_k}^{x_{k+1}} |f|^p \ = \
\| f \|_p^p \ .
\]
Per (2) si usi il teorema fondamentale del calcolo ed il teorema di Lagrange, che implicano
\[
\{ S_xf \}(t)  \ = \  
\frac{ F(x_{k+1}) - F(x_k) }{ x_{k+1}-x_k } \ = \
f(\xi_k)
\ ,
\]
dove $\xi_k \in [x_k,x_{k+1})$ ed $F$ \e una primitiva di $f$.).
}

\

\noindent \textbf{Esercizio \ref{sec_Lp}.9 (Basi di Haar-Schauder per $L^p$).} {\it Si consideri la successione di partizioni
\[
x_n := \{ x_{n,1} := 0 < \ldots < x_{n,k} < \ldots < x_{n,i(n)} := 1  \}
\ \ , \ \
n \in \bN
\ ,
\]
e si assuma che $\lim_n \delta_n := \sup_k (x_{n,k+1}-x_{n,k}) = 0$.
\textbf{(1)} Presa una successione $\{ \tau_h \} \subset L^\infty$ tale che 
\[
\kappa_n(t,s)
\ := \
\sum_h^n \tau_h(t) \tau_h(s) 
\ = \
\left\{
\begin{array}{ll}
( x_{n,k+1} - x_{n,k} )^{-1} \ \ , \ \  \forall s \in (x_{n,k},x_{n,k+1}) \ ,
\\
0 \ \ ,\ \ \forall s \in [0,x_{n,k}) \cup (x_{n,k+1},1] \ ,
\end{array}
\right.
\]
si mostri che per ogni $p \in [1,\infty)$ ed $f \in L^p$ risulta
\[
\int_0^1 \kappa_n(s,t) f(s) \ ds 
\ = \
\frac{1}{ x_{n,k+1}-x_{n,k} } \int_{x_{n,k}}^{x_{n,k+1}} f
\ \ , \ \
\forall t \in ( x_{n,k},x_{n,k+1})
\ .
\]
\textbf{(2)} Si mostri che le applicazioni lineari
\[
K_n : L^p \to L^p
\ \ , \ \
K_nf(t) := \int_0^1 \kappa_n(s,t) f(s) \ ds 
\ , \
t \in [0,1]
\ ,
\]
sono tali che $\| K_nf \|_p \leq \| f \|_p$, $\forall f \in L^p$.
\textbf{(3)} Si mostri che per ogni $f \in C([0,1])$ risulta
\[
\lim_n \sup_{t \in [0,1]-x} | K_nf(t) - f(t) | \stackrel{n}{\to} 0 
\ \ , \ \
x := \cup_n x_n
\ .
\]
\textbf{(4)} Si mostri che per ogni $f \in L^p$ esiste una successione 
$\{ a_h \} \subset \bR$ tale che $f = \lim_n \sum_h^n a_h \tau_h$ in norma $\| \cdot \|_p$, 
e che $\{ a_h \}$ \e unica
{\footnote{La propriet\'a che si chiede di dimostrare in questo punto equivale a dire che $\{ \tau_h \}$ \e una
{\em base di Schauder} per $L^p$ nel senso di \S \ref{sec_afunct_1}. Un esempio esplicito di $\{ \tau_h \}$
\e dato dalle cosiddette {\em funzioni di Haar}, vedi I.Singer: Bases in Banach spaces, Ex.2.3.}}
se valgono le relazioni di ortogonalit\'a $\int \tau_h \tau_m = 0$, $\forall h \neq m$.

\

\noindent (Suggerimenti: Per (2) e (3) si proceda come per l'esercizio precedente, mentre per (4) si ponga
\[
a_h := \int_0^1 f \tau_h \ \ , \ \ h \in \bN \ ,
\]
e si usi il punto (3).).
}


\newpage
\section{Funzioni di pi\'u variabili.}
\label{sec_func_n}

Da un punto di vista astratto possiamo pensare alle funzioni di pi\'u variabili come quelle definite su prodotti cartesiani del tipo $X \times Y$. Nelle sezioni seguenti approcceremo questioni relative alle ulteriori strutture con le quali possiamo arricchire il nostro insieme $X \times Y$, quali quella topologica (\S \ref{sec_top.prod}), quella differenziale con le sue applicazioni all'esistenza di funzioni implicite ed al calcolo variazionale (\S \ref{sec_func_n_der}, \S \ref{sec_teo_impl}, \S \ref{sec_fdif}, \S \ref{sec_calc_var}), e poi quella di spazio misurabile con l'applicazione dei prodotti di convoluzione (\S \ref{sec_misprod}, \S \ref{sec_conv}).

\subsection{Topologie prodotto e prodotti tensoriali.}
\label{sec_top.prod}

Le propriet\'a di base delle topologie prodotto sono ben note ed in questa sede ci limitamo a segnalare alcune referenze (\cite{Ser,Cam}). Passando ad un punto di vista pi\'u analitico vogliamo invece dimostrare un importante risultato di approssimazione.

Siano $X,Y$ spazi localmente compatti e di Hausdorff; allora $X \times Y$, equipaggiato della topologia prodotto, \e uno spazio localmente compatto e di Hausdorff. Date $f \in C_0(X)$, $g \in C_0(Y)$, definiamo la {\em funzione prodotto}
\[
f \otimes g : X \times Y \to \bR
\ \ , \ \
f \otimes g (x,y) := f(x)g(y)
\ .
\]
Una verifica immediata mostra che $f \otimes g \in C_0(X \times Y)$. Denotiamo con $C_0(X) \otimes C_0(Y)$ la sottalgebra di $C_0(X \times Y)$ generata dalle funzioni del tipo $f \otimes g$.

\begin{prop}
\label{prop_tens_prod}
Siano $X,Y$ spazi localmente compatti di Hausdorff. Allora $C_0(X) \otimes C_0(Y)$ \e densa in $C_0(X \times Y)$ nella topologia della convergenza uniforme.
\end{prop}

\begin{proof}[Dimostrazione]
L'idea \e quella di applicare il teorema di Stone-Weierstrass per gli spazi localmente compatti (Cor.\ref{cor_SW}). Per cui, posto $\mA := C_0(X) \otimes C_0(Y)$, per dimostrare la proposizione \e sufficiente verificare che: (1) $\mA$ separa i punti di $X \times Y$; (2) Per ogni $(x,y) \in X \times Y$ esiste $f \otimes g \in \mA$ tale che $f \otimes g (x,y) \neq 0$.
Ora grazie a Lemma \ref{lem_ury} abbiamo che, dati $x \in X$ ed $y \in Y$, esistono $f \in C_0(X)$ e $g \in C_0(Y)$ tali che $f(x) \neq 0$, $g(y) \neq 0$. Dunque $f \otimes g (x,y) \neq 0$, e ci\'o dimostra il punto (2). Riguardo il punto (1), se $(x,y) \neq (x',y')$ sono elementi distinti di $X \times Y$, allora almeno una delle due coordinate di questi deve essere distinta, diciamo $x \neq x'$. Applicando ancora Lemma \ref{lem_ury} troviamo $f \in C_0(X)$ tale che $f(x) \neq f(x')$; aggiungendo eventualmente una costante possiamo assumere che $f(x') = 0$. Presa infine $g \in C_0(Y)$ tale che $g(y) \neq 0$ concludiamo che $f \otimes g (x,y) \neq 0$, mentre $f \otimes g (x',y') = 0$. 
\end{proof}

\subsection{Derivabilit\'a e differenziabilit\'a.} 
\label{sec_func_n_der}

Come vedremo nelle righe che seguono la nozione di derivabilit\'a per funzioni di pi\'u variabili reali, ed il suo rapporto con la continuit\'a, \e una questione pi\'u delicata rispetto al caso ad una variabile. Iniziamo dando la seguente terminologia: dato $n \in \bN$, una {\em direzione} \e un vettore $v \in \bR^n$ con norma $1$.
\begin{defn}
\label{def_fRn1}
Sia $A \subseteq \bR^n$ aperto. Una funzione $f : A \to \bR$ si dice \textbf{derivabile in} $a \in A$ \textbf{lungo la direzione} $v$ se esiste finito il limite 
\begin{equation}
\label{eq_fRn1}
\frac{\partial f}{\partial v} (a)
\ := \
\lim_{t \to 0} \frac{ f(a+tv) - f(a) } {t}
\ .
\end{equation}
Nel caso in cui $v$ sia uno degli elementi $e_i$, $i = 1 , \ldots , n$, della base canonica di $\bR^n$, adottiamo la classica notazione
\begin{equation}
\label{def.dav}
\frac{\partial f}{\partial x_i}(a)
\ := \
\frac{\partial f}{\partial e_i}(a)
\ .
\end{equation}
\end{defn}

Ora, sappiamo bene che se una funzione di una variabile \e derivabile in $a \in \bR$, allora \e anche continua. Il seguente esempio mostra invece come una funzione di pi\'u variabili $f : A \to \bR$ possa essere derivabile in ogni direzione in $a \in A$, e tuttavia essere discontinua in $A$.

\begin{ex}{\it 
\label{ex_fRn1}
Si consideri
\[
f(x,y)
:=
\left\{
\begin{array}{ll}
\left( \frac{x^2y}{x^4+y^2} \right)^2
\ \ , \ \
(x,y) \neq \textbf{0}
\\
0
\ \ , \ \
(x,y) = \textbf{0}
\end{array}
\right.
\]
Si trova 
\[
\lim_{n \to \infty} f \left( \frac{1}{n} , \frac{1}{n^2} \right) 
= 
\frac{1}{4}
\ ,
\]
il che implica che $f$ \e discontinua in $\textbf{0}$. D'altro canto, per ogni direzione $v :=$ $(v_1,v_2)$ si trova
\[
\frac{\partial f}{\partial v}(\textbf{0})
\ = \
\lim_{t \to 0} t^{-1} f ( tv_1 , tv_2 ) = 0
\ .
\]
}\end{ex}

Una nozione pi\'u naturale per le propriet\'a inerenti la differenziazione di $f$ \e invece la seguente:
\begin{defn}
\label{def_fRn2}
Una funzione $f : A \to \bR$ si dice \textbf{differenziabile in} $a \in A$ se esiste un'applicazione lineare 
\[
df_a \in \bR^{n,*}
\ \ , \ \
df_a : \bR^n \to \bR
\ ,
\]
chiamato il \textbf{differenziale di $f$ in $a$}, tale che risulti
\[
\lim_{x \to a} 
\frac{ f(x) - f(a) - df_a(x-a) }{|x-a|} 
\ = \
0
\ .
\]
Se $f$ \e differenziabile per ogni $a \in A$, allora si dice differenziabile in $A$.
\end{defn}

In termini pi\'u geometrici, potremmo riguardare lo spazio vettoriale generato dalle direzioni (isomorfo ad $\bR^n$) come il tangente ad $A$ nel punto $a$. Da (\ref{eq_fRn1}), segue immediatamente che $f$ \e differenziabile in $a$ se e solo se 
\[
\lim_{t \to 0} \frac { f(a+tv)-f(a)-df_a(tv) } {t} 
\ = \ 
0
\ \ , \ \
\forall v
\ \ \Rightarrow \ \
\frac{\partial f}{\partial v} (a)
\ = \
df_a(v)
\ .
\]
Dunque possiamo riguardare $df_a$ come un elemento dello spazio dello spazio duale $\bR^{n,*}$ (in termini geometrici, lo {\em spazio cotangente}), il quale, valutato su $v := \sum_i v_i e_i \in \bR^n$, fornisce il valore della derivata parziale di $f$ lungo $v$. In particolare, per linearit\'a di $df_a$ troviamo, ricordando (\ref{def.dav}),
\begin{equation}
\label{eq_fRn2}
df_a(v)
=
\frac{\partial f}{\partial v}(a)
=
\sum_i v_i df_a(e_i)
=
\sum_i v_i \frac{\partial f}{\partial x_i}(a)
\ .
\end{equation}
D'altro canto, denotando con $dx_i \in \bR^{n,*}$ gli elementi della base canonica dello spazio cotangente, abbiamo per definizione
\[
df_a(v) = \sum_i (df_a)_i \ dx_i (v) = \sum_i (df_a)_i v_i
\ ,
\]
per cui, confrontando con (\ref{eq_fRn2}) troviamo $(df_a)_i =$ $\frac{\partial f}{\partial x_i}(a)$, e quindi la familiare espressione
\begin{equation}
\label{eq_dfa}
df_a = \sum_i \frac{\partial f}{\partial x_i}(a) \ dx_i
\ \ , \ \
a \in A
\ .
\end{equation}
Introducendo il {\em gradiente di $f$ in $a$}
\[
\nabla f (a)
\ := \
\left(  \frac{\partial f}{\partial x_i}(a)  \right)_i
\ \in \bR^n
\]
possiamo esprimere il differenziale in termini del prodotto scalare
\[
df_a(v)
\ = \
(  \nabla f (a) , v  )
\ \ , \ \
v \in \bR^n
\ .
\]
I due seguenti teoremi sono ben noti e reperibili su ogni testo di Analisi II, per cui ne omettiamo la dimostrazione.
\begin{thm}
\label{thm_fRn1}
Ogni funzione differenziabile \e continua. D'altro canto, se una funzione continua ha derivate parziali continue in un intorno di $a \in A$, allora \e differenziabile in $A$.
\end{thm}
\begin{thm}[Schwartz]
\label{thm_fRn2}
Se esistono e sono continue le derivate parziali miste di $f$ in un intorno di $a \in A$, allora
\[
\frac{\partial^2f}{\partial x_i \partial x_k}(a)
\ = \
\frac{\partial^2f}{\partial x_k \partial x_i}(a)
\ \ , \ \
i,k = 1 , \ldots , n
\ .
\]

\end{thm}

\

\noindent \textbf{Massimi e minimi.} Lo studio dei punti di massimo e minimo relativi di una funzione $f \in C^1(A)$ si effettua in primo luogo analizzando i {\em punti stazionari} di $f$, ovvero cercando le soluzioni $a \in A$ dell'equazione
\[
df_a = 0 
\ \ \Leftrightarrow \ \ 
\nabla f (a) = 0
\ .
\]
Se $f \in C^2(A)$, allora la ricerca di massimi e minimi relativi si effettua considerando la {\em matrice Hessiana}
\[
H(a) 
\ := \
\left( \frac{\partial^2f}{\partial x_i \partial x_j}(a) \right)_{ij}
\ ,
\]
la quale \e autoaggiunta in conseguenza del teorema di Schwartz. Se $a$ \e di minimo relativo, allora $H(a) \geq 0$; viceversa, se $H(a) > 0$, allora $a$ \e di minimo relativo. In modo analogo si studia il comportamento dei punti di massimo.

\

\noindent \textbf{La matrice jabobiana.} Sia $U \subset \bR^m$ ed $F : U \to \bR^n$ un'applicazione. Chiaramente possia- mo descrivere $F$ in termini delle funzioni coordinate 
$F_1 , \ldots , F_n : U \to \bR$ definite da $F(u) = ( F_1(u) , \ldots , F_n(u) )$, 
cosicch\'e diciamo che $F$ \e {\em differenziabile in $U$} se lo sono le sue componenti $F_i$, $\forall i = 1 , \ldots , n$. Denotiamo con $C^1(U,\bR^n)$ l'insieme delle applicazioni differenziabili da $U \subseteq \bR^m$ in $\bR^n$; ad uso futuro, introduciamo la {\em matrice jacobiana} di $F$,
\[
JF (u)
\ := \ 
\left( \frac{\partial F_i}{\partial w_j} (u)  \right)_{ij}
\ \in M_{n,m}(\bR)
\ \ , \ \
\forall u \in U
\ .
\]

\subsection{Il teorema delle funzioni implicite.}
\label{sec_teo_impl}

In geometria \e usuale presentare un luogo, sia esso una curva o una superficie, in termini dell'annullarsi di un'espressione del tipo
\begin{equation}
\label{eq_din0}
F(x,y) = 0
\ \ , \ \
(x,y) \in U \subseteq \bR^{m+n}
\ ,
\end{equation}
dove 
$F : U \to \bR^n$
\'e un'applicazione che svolge il ruolo di una relazione tra le variabili $x$ ed $y$.

Il teorema delle funzioni implicite permette di esprimere curve e superfici in termini di grafici di funzioni, almeno a livello locale. Ci\'o vuol dire che a partire da (\ref{eq_din0}) saremo in grado di esibire, in un intorno opportuno $A \subseteq \bR^m$, una funzione $f : A \to \bR^n$ tale che
\[
F (x,f(x)) = 0 \ \ , \ \ \forall x \in A \ .
\]
Nel seguito denoteremo con $x,a \in \bR^m$ i vettori delle prime $m$ variabili di $F$, e con $y,b \in \bR^n$ i vettori delle rimanenti $n$ variabili. Se $F \in C^1(U)$, allora spezziamo la jacobiana di $F$ come segue:
\begin{equation}
\label{def_jacobi}
\left\{
\begin{array}{ll}
J_xF : U \to M_{n,m}(\bR)
\ \ , \ \
J_xF (a,b)
:=
\left( \frac{\partial F_i}{\partial x_j} (a,b)  \right)_{ij}
\\
J_yF : U \to M_{n,n}(\bR)
\ \ , \ \
J_yF (a,b)
:=
\left( \frac{\partial F_i}{\partial y_h} (a,b)  \right)_{ih}
\end{array}
\right.
\end{equation}

\begin{thm}[Teorema delle funzioni implicite, o del Dini]
\label{thm_din1}
Sia $U \subseteq \bR^{m+n}$ aperto ed $F : U \to \bR^n$ un'applicazione di classe $C^1$. Se $(a,b) \in U$ \e tale che
\begin{equation}
\label{eq_din1}
F(a,b) = 0
\ \ , \ \
\det J_yF (a,b) \neq 0 \ ,
\end{equation}
allora esistono intorni $A \ni a$, $B \ni b$, con $A \subseteq \bR^m$, $B \subseteq \bR^n$, $A \times B \subseteq U$, ed un'applicazione
\[
f : A \to B \ \ , \ \ f \in C^1(A) 
\ \ , \ \ 
\]
tale che
\[
F ( x , f(x) ) = 0 \ \ , \ \ x \in A \ .
\]
Inoltre, la matrice Jacobiana di $f$ \e data da
\begin{equation}
\label{eq_din9a}
Jf(\xi)
\, = \,
- J_yF (\xi,f(\xi))^{-1} \, J_xF (\xi,f(\xi))
\ \ , \ \
\xi \in A
\ .
\end{equation}
\end{thm}

\begin{proof}[Dimostrazione]

\

\noindent {\em Passo 1.} Per semplicit\'a di notazione assumiamo $(a,b) = 0 \in \bR^{m+n}$, come del resto \e lecito fare applicando una traslazione all'aperto $U$. Inoltre, osserviamo che grazie a Teo.\ref{thm_fRn1} $F$ \e differenziabile in $U$. Per economia scriviamo
\[
T := J_yF (0) \in GL(n,\bR)
\]
(l'invertibilit\'a di $T$ \e assicurata da (\ref{eq_din1})). Per differenziabilit\'a di $F$ possiamo scrivere
\begin{equation}
\label{eq_din2}
F(x,y) \, = \, T y + J_xF(0) x + R(x,y) \ ,
\end{equation}
dove il resto $R : U \to \bR^n$ \e di classe $C^1(U)$ e tale che 

\begin{equation}
\label{eq_din3}
\lim_{(x,y) \to 0} \frac{R(x,y)}{|x|+|y|} = 0 \ .
\end{equation}
Da (\ref{eq_din2}) otteniamo

\begin{equation}
\label{eq_din4}
y = T^{-1} F(x,y) + Lx + \tilde R(x,y)
\end{equation}
dove 
\[
\tilde R : U \to \bR^n
\ \ , \ \
\tilde R(x,y) := - T^{-1} R(x,y)
\ ,
\]
ed $L : \bR^m \to \bR^n$ \e l'applicazione lineare
\[
Lx \, := \, - T^{-1} \, J_xF(0) \, x \ \ , \ \ x \in \bR^m \ .
\]
Per cui, concludiamo che
\begin{equation}
\label{eq_din5}
F(x,y) = 0 \ \Leftrightarrow \ y = Lx + \tilde R(x,y)
\ ,
\end{equation}
cosicch\'e {\em dimostrare l'esistenza della funzione implicita $f$ equivale a mostrare che esistono intorni $A \ni 0 \in \bR^m$, $B \ni 0 \in \bR^n$ tali che per ogni $x \in A$ esista un solo $y \in B$ che verifichi (\ref{eq_din5})}.

\

\noindent {\em Passo 2.} Dimostriamo che esistono $A \ni 0 \in \bR^m$, $B \ni 0 \in \bR^n$ tali che, per ogni $x \in A$, l'applicazione
\[
\phi_x : B \to \bR^n
\ \ , \ \
\phi_x (y) := Lx + \tilde R(x,y) \ ,
\]
\'e una contrazione. A tale scopo, osserviamo che $\tilde R$ \e differenziabile e chiaramente soddisfa (\ref{eq_din3}), per cui, in un opportuno intorno dell'origine abbiamo
\begin{equation}
\label{eq_din60}
|\tilde R(x,y)| < \frac{1}{2} ( |x|+|y| )
\ \ \ {\mathrm{e}} \ \ \
\lim_{x \to 0} \frac{ | \tilde R (x,0) | }{|x|}
\ , \
\lim_{y \to 0} \frac{ | \tilde R (0,y) | }{|y|}
\ = \ 
0
\ .
\end{equation}
Ci\'o implica che esistono dischi chiusi $\Delta_m \subset \bR^m$, $\Delta_n \subset \bR^m$ di centro l'origine e raggio $r > 0$, tali che per $a \in \Delta_m$, $b \in \Delta_n$,
\[
\left| J_x \tilde R_i(a,b) \right|
\ \leq \ 
\frac{1}{2 \sqrt{m}}
\ \ \ , \ \ \
\left| J_y \tilde R_i(a,b) \right| 
\ \leq \ 
\frac{1}{2 \sqrt{n}}
\]
(osservare che avendosi $\tilde R_i : U \to \bR$ per ogni $i=1,\ldots,m$, le jacobiane $J_x \tilde R_i$ e $J_y \tilde R_i$ sono in realt\'a dei vettori).
Siano ora $y_1,y_2 \in \Delta_n$; applicando il teorema di Lagrange troviamo che esistono $\xi_1 , \ldots , \xi_n \in \Delta_n$ appartenenti al segmento che congiunge $y_1$ con $y_2$, tali che
{\footnote{Per comodit\'a di notazione, nel corso della dimostrazione scriviamo $v \cdot v'$ per il prodotto scalare $(v,v')$; cosicch\'e ad esempio $J_y \tilde R_i(x,\xi_i) \cdot (y_1 - y_2)$ sta per $( \, J_y \tilde R_i(x,\xi_i)  \, , \, y_1 - y_2 \, )$.}}
\[
\tilde R_i(x,y_1) - \tilde R_i(x,y_2) \, = \, 
J_y \tilde R_i(x,\xi_i) \cdot (y_1-y_2) 
\ \ , \ \
i = 1 , \ldots , n
\ ,
\]
da cui
\[
| \tilde R_i(x,y_1) - \tilde R_i(x,y_2) |
\leq 
\frac{1}{2 \sqrt{n}} \ |y_1 - y_2|
\ \ \ 
{\mathrm{e}}
\ \ \
| \tilde R(x,y_1) - \tilde R(x,y_2) |
\leq 
\frac{1}{2} \ |y_1 - y_2|
\ .
\]
Usando le relazioni precedenti troviamo 
\begin{equation}
\label{eq_din7}
| \phi_x (y_1) - \phi_x (y_2) | =
| \tilde R(x,y_1) - \tilde R(x,y_2)  | \leq
\frac{1}{2} \ |y_1-y_2| \ .
\end{equation}
Per cui $\phi_x$ decrementa la distanza. Troviamo ora un dominio compatto $B \subset \Delta_n$ tale che $\phi_x(B) \subseteq B$. A tale scopo, applicando (\ref{eq_din60}) e ponendo 
$\| L \| := \sup_{|x|=1} | Lx |$, 
osserviamo che 
\[
|\phi_x(y)| \leq
\| L \| |x| + |\tilde R(x,y)| \leq
\| L \| |x| + \frac{1}{2} ( |x|+|y| ) \ .
\]
Per cui, scegliendo opportunamente $r_1 , r_2 > 0$ e dischi chiusi $\ovl A :=$ $\ovl{\Delta}(0,r_1) \subset \bR^m$, $B :=$ $\ovl{\Delta}(0,r_2) \subset \bR^n$, otteniamo $| \phi_x (y) | < r_2$. Dunque abbiamo una contrazione $\phi_x : B \to B$.

\

\noindent {\em Passo 3.} Applicando il teorema delle contrazioni (Teo.\ref{thm_contr}), otteniamo che per ogni $x \in A$ esiste ed \e unico $f(x) \in B$ punto fisso per $\phi_x$, ovvero $f(x)$ soddisfa (\ref{eq_din5}). In tal modo, \e definita l'applicazione $f : A \to B$, che chiaramente \e la nostra candidata a funzione implicita per $F$.

\

\noindent {\em Passo 4 (Continuit\'a di $f$).} Come in (\ref{eq_din7}), troviamo opportuni $\eta \in A$, $\xi \in B$ tali che
\[
\begin{array}{ll}
| f(x_1) - f(x_2) | & \leq 
\| L \| |x_1-x_2| + | \tilde R(x_1,f(x_1)) - \tilde R(x_2,f(x_2)) | 
\\ & \leq
\| L \| |x_1-x_2| + 
| \tilde R(x_1,f(x_1)) - \tilde R(x_1,f(x_2)) | +
| \tilde R(x_1,f(x_2)) - \tilde R(x_2,f(x_2)) |
\\ & \leq
\| L \| |x_1-x_2| + \frac{1}{2} \left( |f(x_1)-f(x_2)| + |x_1-x_2| \right) 
\ ;
\end{array}
\]
dunque, 
\begin{equation}
\label{eq_din8}
| f(x_1) - f(x_2) | \leq ( 2 \| L \| - 1 ) |x_1-x_2| \ ,
\end{equation}
ed $f$ \e continua.

\

\noindent {\em Passo 5 (Differenziabilit\'a e Jacobiana di $f$).}  Essendo $F$ di classe $C^1$ esistono opportuni $(\eta_1,\xi_1) , \ldots ,$ 
$(\eta_n,\xi_n)$ appartenenti al segmento che congiunge $(x_1,f(x_1))$ ed $(x_2,f(x_2))$, tali che per ogni $i = 1 , \ldots , n$
\[
0 \, = \,
F_i(x_1,f(x_1)) - F_i(x_2,f(x_2)) \, = \,
J_x F_i(\eta_i,\xi_i) \cdot (x_1-x_2) + J_y F_i(\eta_i,\xi_i) \cdot (f(x_1)-f(x_2)) 
\ ;
\]
introducendo (per somma e sottrazione) nell'equazione precedente le matrici
\[
M_1 := \left(
       \frac{\partial F_i}{\partial x_j}(\eta_i,\xi_i) - 
       \frac{\partial F_i}{\partial x_j}(x_2,f(x_2))
       \right)_{ij}
\ \ , \ \
M_2 := \left(
       \frac{\partial F_i}{\partial y_k}(\eta_i,\xi_i) - 
       \frac{\partial F_i}{\partial y_k}(x_2,f(x_2))
       \right)_{ik}
\ ,
\]
otteniamo
\[
0 \, = \, 
J_xF(x_2,f(x_2)) (x_1-x_2)        + 
J_yF(x_2,f(x_2)) (f(x_1)-f(x_2))  + 
M_1 (x_1-x_2) + M_2 (f(x_1)-f(x_2)) \ .
\]
Ora $M_1,M_2 \to 0$ per $x_1 \to x_2$; d'altro canto $J_yF(x_2,f(x_2))$ \e invertibile, cosicch\'e
\[
\label{eq_din9}
f(x_1)-f(x_2) \, = \,
- J_yF(x_2,f(x_2))^{-1}
         J_x(x_2,f(x_2)) (x_1-x_2)    +
Q(x_1,x_2)
\ ,
\]
avendo raggruppato in $Q(x_1,x_2)$ i termini restanti; da (\ref{eq_din8}) segue che 
\[
\lim_{(x_1 \to x_2)} \frac{Q(x_1,x_2)}{|x_1-x_2|} = 0 \ ,
\]
per cui $f$ \e differenziabile.
\end{proof}

\

\noindent \textbf{Il teorema dell'inverso locale.} Sia $A \subseteq \bR^n$ aperto, Un'applicazione $f \in C^1(A,\bR^n)$ si dice {\em diffeomorfismo locale in} $a \in A$ se esiste un intorno $V \ni a$ tale che $f|_V$ \e un diffeomorfismo.

\begin{thm}[Teorema dell'inverso locale]
\label{thm_din2}
Sia $A \subseteq \bR^n$ aperto ed $f: A \to \bR^n$ un'applicazione di classe $C^1$. Se $a \in A$ \e tale che 
\[
\det Jf(a) \ \neq \ 0 \ ,
\]
allora $f$ \e un diffeomorfismo locale in $a$.
\end{thm}

\begin{proof}[Dimostrazione]
L'idea consiste nell'applicare il teorema delle funzioni implicite a 
\[
F : A \times \bR^n \to \bR^n
\ \ , \ \  
F(x,y) := y - f(x) 
\ \ , \ \  
x \in A \ , \ y \in \bR^n 
\ .
\]
Osserviamo infatti che preso il nostro $a \in A$, allora risulta
\[
F(a,f(a)) = 0
\ \ , \ \
\det J_xF (a,f(a)) 
\ = \ 
- \det Jf(a)
\]
(notare lo scambio del ruolo di $x,y$ rispetto all'enunciato di Teo.\ref{thm_din1}). Per cui, il teorema delle funzioni implicite ci assicura che esistono intorni $V \subseteq A$, $a \in V$, $B \subseteq \bR^n$, $f(a) \in B$, ed una funzione $g : B \to V$ di classe $C^1$ tale che $F(g(y),y) = 0$, $y \in B$; in altri termini, $f \circ g (y) = y$, $y \in B$. Per cui $g$ \e l'inversa locale cercata.
\end{proof}

\begin{ex}
{\it
Esibiamo un esempio per il quale non vale il teorema dell'inverso locale. Sia
\[
f : \bR \to \bR
\ \ , \ \
f(x) :=
\left\{
\begin{array}{ll}
\frac{x}{2} + x^2 \sin \frac{1}{x} \ \ , \ \ x \neq 0
\\
0 \ \ , \ \ x = 0
\end{array}
\right.
\]
Si verifica che $f$ \e derivabile su $\bR$ e $f'(0) = 1/2 \neq 0$. Tuttavia $f$ ha infinite oscillazioni in ogni intorno di $0$, e quindi non pu\'o essere localmente invertibile. Si noti che $f'$ \e discontinua in $0$.
}
\end{ex}

\

\noindent \textbf{Moltiplicatori di Lagrange.} Sia $f \in C^1(A)$, $A \subseteq \bR^n$. Consideriamo una variet\'a
$\Sigma \subset A$
di classe $C^1$ e dimensione $m < n$, equipaggiata di una rappresentazione parametrica
\[
\varphi : U \to \Sigma 
\ , \
U \subseteq \bR^m
\ \ , \ \ 
\varphi = ( \varphi_1 (u) , \ldots , \varphi_n(u) )
\ , \
u := ( u_1 , \ldots , u_m ) \in U
\ .
\]
Si dicono {\em massimi e minimi vincolati di $f$ rispetto a $\Sigma$} i massimi e minimi di
$f \circ \varphi : U \to \bR$.
In tal caso, diremo che $\Sigma$ \e un {\em vincolo per $f$}. E' immediato dimostrare, tramite la formula di derivazione composta, che 
\[
\nabla (f \circ \varphi) 
\ = \
\left( \
\sum_{i=1}^n 
\frac{\partial f}{\partial x_i} (\varphi(u))
\
\frac{\partial \varphi_i}{\partial u_k} (u)
\ \right)_{k=1 , \ldots , m}
\ = \
\left\{ \
\left( 
\nabla f (\varphi(u)) \ , \ \frac{\partial \varphi}{\partial u_k}(u)
\right)
\ \right\}_{k = 1, \ldots , m}
\ .
\]
Per cui, se cerchiamo i punti stazionari $b \in U$ per $f \circ \varphi$, allora la condizione $\nabla (f \circ \varphi) = 0$ si traduce in
\begin{equation}
\label{eq_ML1}
\left( 
\nabla f (\varphi(b)) \ , \ \frac{\partial \varphi}{\partial u_k}(b)
\right)
\ = \
0
\ \ , \ \
k = 1 , \ldots , m
\ .
\end{equation}
In altri termini, {\em $\nabla f$ \e normale a $\Sigma$ nel punto $\varphi (b)$}.
Ora, in generale $\Sigma$ si pu\'o presentare in forma implicita, ovvero in termini di un luogo
\begin{equation}
\label{eq_ML2}
\Sigma
\ := \
\{
a \in A \ : \ F ( a_1 , \ldots , a_n ) = 0
\}
\ ,
\end{equation}
dove $F : A \to \bR^m$ \e un'applicazione $C^1$. In questo caso, abbiamo il seguente
\begin{thm}[Moltiplicatori Lagrangiani]
Sia $f : A \to \bR$, $A \subseteq \bR^n$, di classe $C^1$, e sia $\Sigma \subset A$ l'insieme degli zeri di un'applicazione $F : A \to \bR$, anch'essa di classe $C^1$. Allora, i punti di massimo e minimo (vincolato) della funzione $f |_\Sigma : \Sigma \to \bR$ sono i punti stazionari \textbf{liberi} della funzione
\[
H : A \times \bR \to \bR
\ \ , \ \
H(x,\lambda) := f(x) + \lambda F(x)
\ .
\]
\end{thm}

\begin{proof}[Dimostrazione]
Sia $a \in \Sigma$ tale che
\[
\frac{\partial F}{\partial u_n} (a) \neq 0
\ .
\]
Allora, per il teorema delle funzioni implicite (Teo.\ref{thm_din1}) esiste un intorno $V \ni (a_1 , \ldots , a_{n-1})$ e $g : V \to \bR$ tali che
\[
F (u_1 , \ldots , u_n) = 0
\ \ \Leftrightarrow \ \
u_n = g ( u_1 , \ldots , u_{n-1} )
\ .
\]
In tal modo, in un intorno di $a$ del tipo $V \times (a_n-\delta,a_n+\delta)$, il versore normale a $\Sigma$ si scrive 
\[
\textbf{n} 
\ = \
\frac{1}{\sqrt{1+|\nabla g|^2}} 
\ 
\left( - \frac{\partial g}{\partial u_1} 
         \ , \ldots , \ 
       - \frac{\partial g}{\partial u_{n-1}}
         \ , \
       1 \
\right)
\ \stackrel{(\ref{eq_din9a})}{=} \
\frac{\nabla F}{|\nabla F|}
\ .
\]
Usando (\ref{eq_ML1}), troviamo che se $a \in \Sigma$ \e stazionario per $f|_\Sigma$, allora $\nabla f$ deve essere proporzionale ad $\textbf{n}$ in $a$. In altri termini, esiste $\lambda \in \bR$ tale che
\begin{equation}
\label{eq_MLAN}
\nabla f (a) + \lambda \nabla F (a) = 0
\ .
\end{equation}
Quest'ultima equazione, accoppiata alla condizione $F(a) = 0$, \e equivalente all'annullarsi del gradiente di $H$ in $(a,\lambda)$.
\end{proof}

\begin{ex}
{\it
Siano $g,k > 0$ fissati. Consideriamo la circonferenza
$\Sigma := \{ (x,y) \in \bR^2 \ , \ x^2+y^2-1=0 \}$ 
e la funzione $f(x,y) := mgy+1/2 k [(x-1)^2+y^2]$, $x,y \in \bR$. Allora i punti di minimo di $f |_\Sigma$ sono le soluzioni del sistema
\[
\left\{
\begin{array}{ll}
k(x-1) + 2 \lambda x = 0
\\
mg + ky + 2 \lambda y = 0 \ .
\end{array}
\right.
\]
}
\end{ex}

\subsection{Cenni sulle forme differenziali e loro integrazione.}
\label{sec_fdif}

\noindent \textbf{Formule di Gauss-Green.} Sia $U \subseteq \bR^n$ aperto ed $F : U \to \bR^n$ un'applicazione di classe $C^1$. La {\em divergenza} di $F$ \e data dalla funzione
\[
{\mathrm{div}} F : U \to \bR
\ \ , \ \
{\mathrm{div}} F \ := \ \sum_i^n \frac{\partial F_i}{\partial x_i} \ =: \ (\nabla,F) \ .
\]
Un compatto $K \subset U$ si dice {\em dominio regolare} se il bordo $\partial K$ \e costituito da variet\'a $C^1$ a tratti (curve regolari a tratti
{\footnote{Per la definizione di curva regolare a tratti si veda \S \ref{sec_intc}.}}, 
nel caso $n=2$). La {\em formula di Gauss-Green} stabilisce che
\begin{equation}
\label{eq_gauss_green}
\int_K {\mathrm{div}} F \ = \ \int_{\partial K} (F,\textbf{n})
\ ,
\end{equation}
dove $\textbf{n}$ denota il versore normale esterno al bordo $\partial K \subset \bR^n$ e 
$(F,\textbf{n}) : U \to \bR$ 
\e la funzione definita dal prodotto scalare
\[ 
(F,\textbf{n})(x) \ := \ \sum_i F_i(x) \textbf{n}_i(x) \ \ , \ \ x \in U \ .
\]
Osserviamo che nel caso $n=1$ e $K = [a,b]$ abbiamo $\textbf{n}(a) = -1$ e $\textbf{n}(b) = 1$; inoltre la divergenza di $F$ \e la sua derivata, mentre l'integrale su $\partial K = \{ a,b \}$ \e semplicemente la somma 
\[
(F,\textbf{n})(b) + (F,\textbf{n})(a) \ = \ F(b)\textbf{n}(b) + F(a)\textbf{n}(a) \ = \ F(b)-F(a) \ .
\]
Dunque, (\ref{eq_gauss_green}) si riduce al teorema fondamentale del calcolo.
Nelle righe che seguono diamo l'idea della dimostrazione nel caso $n=2$, con bordo dato da curve regolari (per i dettagli si veda \cite{Giu2}):
\begin{itemize}
\item Fissiamo $r > 0$ ed consideriamo un ricoprimento $Q := \{ Q_k \}$ di $U$, dove ogni $Q_k$ \e un quadrato con centro $(t_k,a_k)$ e lato $2r$, $r > 0$. Osserviamo che, potendo assumere che $U$ \e limitato, possiamo sempre scegliere $Q$ finito;
\item Richiediamo che ogni $\partial K \cap Q_k$, $k \in \bN$, sia il grafico di una funzione 
$\alpha_k : I_k \to \bR$ per qualche intervallo limitato $I_k \subset \bR$. 
Osserviamo che grazie al Teorema delle funzioni implicite (\S \ref{sec_teo_impl}) $\partial K$ si pu\'o {\em sempre} interpretare {\em localmente} come il grafico di una funzione, dunque il ricoprimento di cui sopra esiste.
\end{itemize}
In tale scenario possiamo scrivere il versore normale esterno nel seguente modo:
\[
\textbf{n}(t,\alpha_k(t)) = (1+\alpha_k'(t)^2)^{-1/2} ( - \alpha_k'(t) , 1 ) 
\ \ , \ \  
t \in I_k \ .
\]
Sia ora $f \in C^1_0(U \cap Q_k)$, cosicch\'e $f$ \e nulla al bordo di $Q_k$. Assumiamo che $K \cap Q_k$ sia della forma
\begin{equation}
\label{eq.UQ}
K \cap Q_k = \{ (x_1,x_2) \in \bR^2 : x_1 \in I_k := (t_k-r,t_k+r) \ , \  x_2 \in (a_k-r,\alpha_k(x_1))  \} \ ,
\end{equation}
il che vuol dire che stiamo applicando il teorema del Dini esprimendo $\delta K \cap Q_k$ in funzione della coordinata $x_1$. Applicando Fubini troviamo, tenendo conto che per ipotesi $f(t,a_k-r) = 0$, $\forall t \in I_k$,
\[
\begin{array}{ll}
\int_{K \cap Q_k} \frac{\partial f}{\partial x_2} \ dx_1 dx_2 & = 
\int_{I_k} dt \int_{a_k-r}^{\alpha_k(t)} \frac{\partial f}{\partial x_2} \ dx_2 = \\ & =
\int_{I_k} f(t,\alpha_k(t)) \ dt  = \\ & =
\int_{I_k} f(t,\alpha_k(t)) \ \textbf{n}_2(t,\alpha_k(t)) \sqrt{1+\alpha_k'(t)^2} \ dt 
\ ;
\end{array}
\]
riguardando l'ultimo integrale otteniamo
\begin{equation}
\label{eq.div1}
\int_{K \cap Q_k} \frac{\partial f}{\partial x_2} \ dx_1 dx_2
\ = \
\int_{\partial K \cap Q_k} f \textbf{n}_2 \ ds 
\ \ , \ \
ds := \sqrt{1+\alpha_k'(t)^2} \ dt 
\ .
\end{equation}
Analogamente, avendosi $f(t_k-r,x_2) = f(t_k+r,x_2) = 0$, $\forall x_2 \in (a_k-r,a_k+r)$,
\begin{equation}
\label{eq.div2}
\begin{array}{ll}
\int_{K \cap Q_k} \frac{\partial f}{\partial t} \ dt dx_2 & = 
\int_{I_k} dt \ \left\{ \int_{a_k-r}^{\alpha_k(t)} \frac{\partial f}{\partial t} \ dx_2 \right\} = \\ & =
\int_{I_k} dt \left\{ \frac{d}{dt} \int_{a_k-r}^{\alpha_k(t)} f(t,x_2) dx_2 - f(t,\alpha_k(t)) \alpha_k'(t) \right\} = \\ & =
\int_{a_k-r}^{\alpha_k(t_k+r)} f(t_k+r,x_2) \ dx_2   - 
\int_{a_k-r}^{\alpha_k(t_k-r)} f(t_k-r,x_2) \ dx_2 \ - \
\int_{I_k} f(t,\alpha_k(t)) \alpha_k'(t) \ dt = \\ & =
- \int_{I_k} f(t,\alpha_k(t)) \alpha_k'(t) \ dt = \\ & =
\int_{\partial K \cap Q_k} f \textbf{n}_1 \ ds
\ .
\end{array}
\end{equation}
Si verifica facilmente che le uguaglianze precedenti rimangono vere anche per quegli indici $k$ per i quali scrivessimo $x_1$ in funzione di $x_2$, cosicch\'e le assumeremo valide per ogni $K \cap Q_k$, anche non della forma (\ref{eq.UQ}). Consideriamo ora (\ref{eq.div1},\ref{eq.div2}) sostituendo ad $f$ le funzioni $\rho_k F_i \in C^1_0(U \cap Q_k)$, $i = 1,2$, dove $\{ \rho_k \in C^1_0(Q_k) \}_k$ \e una partizione dell'unit\'a di $U$ subordinata a $Q$
{\footnote{Con ci\'o intendiamo che: 
           (1) ogni $\rho_k : U \to [0,1]$ \e una funzione $C^\infty$ con supporto in $Q_k$;
           (2) in un intorno di ogni $x \in U$ vi \e solo un numero finito di $\rho_k$ con $\rho_k(x) \neq 0$;
           (3) $\sum_k \rho_k(x) = 1$, $\forall x \in K$ (convergenza puntuale, quando la somma \e infinita).
           L'esistenza delle partizioni dell'unit\'a si dimostra con argomenti che fanno uso del Lemma 
           di Urysohn, vedi \cite[Prop.1.7.12]{Ped} per il caso di spazi normali e \cite{Giu2} per $\bR^2$.
           }},
e calcoliamo
\[
\begin{array}{ll}
\int_{\partial K} (F,\textbf{n}) & = 
\sum_k \int_{\partial K \cap Q_k} \rho_k (F,\textbf{n}) = \\ & =
\sum_k \int_{\partial K \cap Q_k} \rho_k (F_1 \textbf{n}_1 + F_2 \textbf{n}_2) \stackrel{(\ref{eq.div1},\ref{eq.div2})}{=} \\ & =
\sum_k \int_{K \cap Q_k} \left( \frac{\partial (\rho_k F_1)}{\partial x_1} + 
                                \frac{\partial (\rho_k F_2)}{\partial x_2} \right) \ dx_1dx_2
\ .
\end{array}
\]
Derivando per parti le funzioni integrande ed usando le identit\'a
\[
\sum_k \rho_k(x) = 1
\ \ \Rightarrow \ \
\sum_k \frac{\partial \rho_k}{\partial x_i}(x) = 0
\ \ , \ \
\forall x \in K  \ , \ i = 1,2
\ ,
\]
\e immediato verificare che i contributi delle $\rho_k$ si elidono, e ci\'o dimostra (\ref{eq_gauss_green}).

\

\noindent \textbf{Forme Differenziali.} Dato uno spazio vettoriale reale $V$ di dimensione finita $n$, denotiamo con $ \wedge^m V$ la sua potenza tensoriale {\em esterna}
{\footnote{Detta anche {\em antisimmetrica}, in quanto si hanno le relazioni 
$v_{p(1)} \wedge \ldots \wedge v_{p_m} = sgn(p) v_1 \wedge \ldots \wedge v_m$,
$\forall v_i \in V$,
dove $p$ \e una qualsiasi permutazione di ordine $m$. Ad esempio, nel caso $m=2$ abbiamo $v_1 \wedge v_2 = - v_2 \wedge v_1$.}} 
di ordine $m \leq n$; ricordiamo che $\wedge^0V := \bR$, $\wedge^1 V = V$ e $\wedge^n V = \bR$. Se $\{ e_i  \}$ \e una base di $V$, denotiamo con
\[
e_I := e_{i_1} \wedge \ldots \wedge e_{i_m}
\ \ , \ \
I := \{ i_1 , \ldots , i_m  \} 
\ ,
\]
gli elementi della base di $\wedge^mV$. Nel caso $V = \bR^{n,*}$, $n \in \bN$, abbiamo la base $\{ dx_i \}$, cosicch\'e $\wedge^m \bR^{n,*}$ ha base
$\{ dx_I := dx_{i_1} \wedge \ldots \wedge dx_{i_m} \}$.

Sia ora $U \subseteq \bR^n$ un aperto connesso. Una $m$-forma differenziale di classe $C^k$ \e il dato di un'applicazione di classe $C^k$
\[
\omega : U \to \wedge^m \bR^{n,*} \ ;
\]
considerando, per ogni multiindice $I$ di lunghezza $m$, i prodotti scalari $a_I(x) := ( \omega (x) , dx_I )$, $x \in U$, otteniamo una funzione $a_I \in C^k(U)$ e troviamo
\[
\omega (x) \ = \ \sum_{|I|=m} a_I(x) \ dx_I \ \ , \ \ x \in U \ .
\]
In accordo alle convenzioni precedenti, una $0$-forma \e semplicemente una funzione $f \in C^k(U)$. Denotiamo con $\Omega^m_k(U)$ lo spazio delle $m$-forme differenziali di classe $k$, $k \in \bZ^+ \cup \{ \infty  \}$. La {\em derivata esterna} \e definita dall'applicazione
\[
d \ : \  \Omega^m_k(U) \to \Omega^{m+1}_{k-1}(U)
\ \ , \ \
\omega \mapsto d \omega \ := \ \sum_{i=1}^n \frac{ \partial a_I  }{ \partial x_i } \ dx_i \wedge dx_I \ .
\]
Usando l'antisimmetria del prodotto esterno ed il teorema di Schwartz troviamo $d^2 := d \circ d = 0$. Una $m$-forma si dice {\em chiusa} se $d \omega = 0$, ed {\em esatta} se $\omega = d \varphi$ per qualche $m-1$-forma $\varphi$; poich\'e $d^2 = 0$, \e chiaro che ogni forma esatta \e chiusa. 
In particolare, una $1$-forma \e esatta se $\omega = df$ per qualche $f \in C^k(U)$, dove $df$ \e l'applicazione definita dal differenziale (vedi (\ref{eq_dfa}))
\[
df : U \to \bR^{n,*}
\ \ , \ \
a \mapsto df_a
\ \ , \ \
a \in U
\ .
\]
\textbf{Fattori integranti.} Come applicazione del concetto di forma differenziale presentiamo un metodo per la soluzione di una classe di problemi di Cauchy. Sia $U \subset \bR^2$, $(t_0,u_0) \in U$, $h,g \in C(U)$ con $g(t_0,u_0) \neq 0$. Consideriamo il problema
\begin{equation}
\label{eq_FI}
\left\{
\begin{array}{ll}
u' \ = \ - h(t,u(t)) \ g(t,u(t))^{-1}
\\
u(t_0) \ = \ u_0 \ .
\end{array}
\right.
\end{equation}
Definiamo la $1$-forma 
\[
\omega(t,u) := h(t,u) \ dt + g(t,u) \ du
\ \ , \ \
(t,u) \in U
\ ,
\]
e supponiamo che essa sia esatta, ovvero $\omega = df$, $f \in C^1(U)$. Sostituendo eventualmente $f$ con $f - f(t_0,u_0)$ possiamo assumere che $f(t_0,u_0) = 0$, per cui abbiamo
\[
f(t_0,u_0) = 0
\ \ , \ \
\frac{ \partial f }{ \partial y } (t_0,u_0) = g(t_0,u_0) \neq 0 \ .
\]
Sono dunque soddisfatte le ipotesi di Teo.\ref{thm_din1}, per cui esiste $A \subseteq \bR$ ed $u \in C^1(A)$ tale che $f(t,u(t)) = 0$, $t \in A$. Poich\'e, grazie a (\ref{eq_din9a}), si ha
\begin{equation}
\label{eq_FI.sol}
u'(t) 
\ = \ 
- \left( \frac{ \partial f }{ \partial y } (t,u(t)) \right)^{-1}
\
\frac{ \partial f }{ \partial x } (t,u(t))
\ = \
- g(t,u(t))^{-1} h(t,u(t))
\ ,
\end{equation}
concludiamo che $(A,u)$ \e soluzione di (\ref{eq_FI}). Qualora $\omega$ non sia esatta possiamo cercare un {\em fattore integrante}, ovvero una funzione $\phi \in C^1(U)$, $\phi(t_0,u_0) \neq 0$, tale che 
\[
\phi \omega (t,u) \ := \ \phi(t,u)h(t,u) \ dt + \phi(t,u) g(t,u) \ du
\ \ , \ \
(t,u) \in U
\ ,
\]
sia esatta, ovvero $\phi \omega = df_\phi$, $f_\phi \in C^1(U)$. In tal modo, 
$\{ \partial f_\phi / \partial y \}(t_0,u_0) = \phi(t_0,u_0) g(t_0,u_0) \neq 0$
e troviamo la soluzione $u_\phi$ di (\ref{eq_FI}) definita come in (\ref{eq_FI.sol}).

\

\noindent \textbf{Integrazione di forme e Teorema di Stokes.} Discutiamo ora la nozione di integrale di una forma differenziale. Iniziamo osservando che avendosi $\bR \simeq \wedge^0 \bR^{*,n} , \wedge^n \bR^{*,n}$, l'integrale di $0$-forme ed $n$-forme \e ben definito nella maniera usuale, visto che queste si possono riguardare come funzioni di classe $C^k$; in particolare, se $K \subset U$ \e compatto allora sappiamo che $| \int_K \omega | < \infty$, $\omega \in \Omega^0_k(U)$ o $\Omega^n_k(U)$.
Pi\'u in generale possiamo integrare $m$-forme, $1 \leq m < n$, nel modo che segue. Consideriamo l'aperto $A \subseteq \bR^m$ equipaggiato con la restrizione della misura prodotto di Lebesgue (\S \ref{sec_misprod}) e $\gamma \in C^1(A,U)$; data $\omega \in \Omega^m_k(U)$, l'{\em integrale} di $\omega$ su $\gamma$ si definisce come
\begin{equation}
\label{eq.int.form}
\int_\gamma \omega 
\ := \ 
\sum_{|I|=m} \int_A  a_I \circ \gamma (u) \cdot \det \partial_I \gamma (u) \ du_I
\ \ , \ \
A \subseteq \bR^m
\ \ , \ \
\omega \in \Omega^m_k(U)
\ ,
\end{equation}
dove, per ogni $I := \{ i_1 , \ldots , i_k , \ldots , i_m \}$, abbiamo definito l'applicazione
\[
\partial \gamma_I : A \to M_{m,m}(\bR)
\ \ , \ \
\partial \gamma_I (v) 
\ := \
\left( \frac{\partial \gamma_{i_k}}{\partial u_j} (v)  \right)_{j,k=1 , \ldots , m}
\ \ , \ \
v \in A \subseteq \bR^m
\ .
\]
L'integrazione delle forme \e strettamente connessa con la loro propriet\'a di esattezza: un importante risultato stabilisce che $\omega \in \Omega^1_0(U)$ \e esatta se e solo se $\int_\gamma \omega = 0$ per ogni curva chiusa e regolare a tratti $\gamma$ (\cite[Cor.8.2.1]{Giu2}). Osserviamo inoltre che la propriet\'a dell'esattezza di $\omega$ dipende anche dal dominio scelto: se $U$ \e {\em stellato}{\footnote{Ovvero esiste $x \in U$ tale che per ogni $y \in U$ risulta che i segmento che unisce $x$ ad $y$ \e contenuto in $U$.}} allora ogni $1$-forma chiusa \e anche esatta (\cite[Teo.8.2.2]{Giu2}).

\begin{ex}
{\it
Sia $A := \bR^2 - \{ 0 \}$ ed $\omega := (x^2+y^2)^{-1} [ -y \ dx + x \ dy ]$. Allora $\omega$ \e chiusa ma non esatta (si calcoli infatti $\int_\gamma \omega$, $\gamma(t) := (\cos t , \sin t)$, $t \in [0,1]$).
}
\end{ex}

Sia ora $K \subset U$ un dominio regolare con bordo $\partial K$. Il teorema di Stokes (vedi \cite[Eq.8.5.8]{Giu2} o \cite{Arb}) stabilisce che
\begin{equation}
\label{eq_stokes}
\int_{\partial K} \omega \ = \ \int_K d \omega 
\ \ , \ \
K \subset U \subseteq \bR^n
\ \ , \ \
\omega \in \Omega^m_k(U)
\ .
\end{equation}
Nel caso $m=1$, $n=2$ il teorema di Stokes \e una riformulazione della formula di Gauss-Green con $\omega = F_2 \ dx_1 - F_1 \ dx_2$.

\

\noindent \textbf{La dualit\'a di de Rham.} Il teorema di Stokes \e un tassello importante di una costruzione fondamentale in analisi e geometria, nota come la {\em dualit\'a di de Rham}. Abbiamo visto nei paragrafi precedenti come l'esattezza di una forma differenziale si traduca nella soluzione di un'equazione differenziale, per cui, data una variet\'a $M$ che per comodit\'a riguardiamo come un sottoinsieme $M \subseteq \bR^n$, \e di interesse caratterizzare le forme differenziali esatte su $M$. Ovviamente, condizione necessaria (e facilmente verificabile) per l'esattezza della forma $\omega \in \Omega^m_\infty(M)$ \e la propriet\'a di essere chiusa, ovvero $d \omega = 0$.

Per ogni $m \in \bN$ denotiamo con $Z_{dR}^m(M)$ lo spazio vettoriale delle $m$-forme chiuse e $C^\infty$; chiaramente $d \Omega^{m-1}_\infty(M)$ \e un sottospazio vettoriale di $Z^m_{dR}(M)$, e definiamo la {\em coomologia di de Rham di ordine $m$} come lo spazio quoziente 
\[
H_{dR}^m(M) \ := \ Z^m_{dR}(M) / d \Omega^{m-1}_\infty(M) \ .
\]
In base alle considerazioni precedenti, il fatto che $H_{dR}^m(M)$ sia non banale si traduce nell'esistenza di $m$-forme non esatte. Ebbene, $H_{dR}^m(M)$ pu\'o essere calcolato {\em in termini di propriet\'a prettamente topologiche di $M$}, e nelle righe seguenti accenneremo alla dimostrazione di questo fatto.

Iniziamo definendo, per ogni $m \in \bN$, il {\em simplesso standard}
\[
\Delta_m 
\ := \ 
\{ x \in \bR^m : x_i \geq 0 \ \forall i=1,\ldots,n \ , \ \sum_i x_i \leq 1 \}
\ \subset \bR^m
\ ;
\]
osserviamo che $\Delta_1$ \e un intervallo chiuso, $\Delta_2$ un triangolo chiuso, $\Delta_3$ un tetraedro.

Ora, preso un qualsiasi insieme $\Sigma$ possiamo definire lo spazio vettoriale $V(\Sigma)$ con elementi combinazioni lineari formali del tipo $\sum_i c_i a_i$, $c_1 \in \bR$, $a_i \in \Sigma$; per ogni $m \in \bN$ consideriamo quindi l'insieme $C^\infty(\Delta_m,M)$ e definiamo $C_m(M) := V(C^\infty(\Delta_m,M))$.
Lo scopo \e ora quello di introdurre sugli spazi $C_m(M)$ degli operatori 
$\partial : C_m(M) \to C_{m-1}(M)$
che formalizzino la nozione di bordo di un sottoinsieme di $\bR^n$ (nozione peraltro usata in (\ref{eq_stokes})). A tale scopo, per ogni $i = 0 , \ldots , m-1$ definiamo le applicazioni (dette {\em facce})
\[
j^i_m : \Delta_{m-1} \to \Delta_m
\ \ , \ \
j^i_m(x)
\ := \
\left\{
\begin{array}{ll}
(x_1 , \ldots , x_{i-1} , 0 , x_i , \ldots , x_{m-1}) \ \ , \ \ i \neq 0
\\
(1 - \sum_i x_i , x_1 , \ldots , x_{m-1})             \ \ , \ \ i = 0 \ ;
\end{array}
\right.
\]
il contenuto intuitivo della definizione precedente \e che un $m-1$-simplesso standard si presenta $m$ volte come una delle facce dell'$m$-simplesso.
Possiamo ora costruire {\em gli operatori di bordo} estendendo per linearit\'a sugli elementi di $C^\infty(\Delta_m,M) \subset C_m(M)$:
\[
\partial_m : C_m(M) \to C_{m-1}(M)
\ \ , \ \
\partial_m \sigma \ := \ \sum_i (-1)^i \sigma \circ j^i_m
\ \ , \ \
\sigma \in C^\infty(\Delta_m,M)
\ .
\]
Dei semplici (ma piuttosto tediosi) conti mostrano che $\partial_m \circ \partial_{m+1} = 0$, per cui ${\mathrm{Im}}(\partial_{m+1})$ \e un sottospazio vettoriale di $\ker(\partial_m)$. Ha quindi senso definire {\em l'omologia singolare di ordine $m$} come lo spazio quoziente
\[
H_m(M) \ := \ \ker(\partial_m) / {\mathrm{Im}}(\partial_{m+1}) \ .
\]
Presa $\omega \in Z^m_{dR}(M)$, definiamo il funzionale lineare
\[
\omega_* : \ker(\partial_m) \to \bR
\ \ , \ \
\omega_*(v)
\ := \
\sum_i a_i \int_{\sigma_i} \omega
\ \ , \ \
\forall v := \sum_i a_i \sigma_i \in \ker(\partial_m)
\ .
\]
Grazie al teorema di Stokes troviamo immediatamente
\[
\omega_*(v)
\ = \
\{ \omega + d \varphi \}_* (v + \partial_{m+1}w)
\ \ , \ \
\forall \varphi \in \Omega^{m-1}_\infty(M)
\ , \
w \in C_{m+1}(M)
\ ,
\]
per cui concludiamo che: 
(1) $\omega_* = \{ \omega + d \varphi \}_*$ dipende solo dalla classe di equivalenza $[\omega] \in H_{dR}^m(M)$; 
(2) $\omega_*(v) = \omega_*(v + \partial_{m+1}w)$ dipende solo dalla classe di equivalenza di $v$ in $H_m(M)$. 
Di conseguenza, abbiamo una ben definita applicazione lineare
\begin{equation}
\label{eq.dR}
H_{dR}^m(M) \to H_m^*(M)
\ \ , \ \
[\omega] \mapsto \omega_*
\ ,
\end{equation}
dove $H_m^*(M)$ \e lo spazio duale di $H_m(M)$. 

\begin{thm}[de Rham]
Sia $M$ una variet\'a differenziale di dimensione $d$. 
Allora $H_m(M)$ ha dimensione finita per ogni $m \in \bN$ e (\ref{eq.dR}) stabilisce un isomorfismo di spazi vettoriali. Di conseguenza si ha un isomorfismo 
$H_{dR}^m(M) \simeq H_m(M)$.
\end{thm}

%
Segnaliamo il fatto fondamentale (e non banale) che se costruissimo gli spazi di omologia a partire da applicazioni continue $\sigma \in C(\Delta_m,M)$ piuttosto che $C^\infty$ otterremmo lo stesso spazio vettoriale $H_m(M)$, il quale dipende quindi solo dalla topologia di $M$.
Per dettagli sulla dualit\'a di de Rham rimandiamo a \cite{Arb,Lee}; qui ci limitiamo ad elencare i casi delle palle unitarie ($D^n$) e delle sfere ($S^n$), $n \in \bN$, di interesse nell'ambito del teorema di Brouwer (Teo.\ref{teo_bro}):
\begin{equation}
\label{eq_omologia}
H_k(S^n) =
\left\{
\begin{array}{ll}
\bR     \ \ , \ \ k = 0,n
\\
\{ 0 \} \ \ , \ \ k \neq 0,n
\end{array}
\right.
\ \ \ , \ \ \
H_k(D^n) =
\left\{
\begin{array}{ll}
\bR     \ \ , \ \ k = 0
\\
\{ 0 \} \ \ , \ \ k > 0
\end{array}
\right.
\end{equation}

\subsection{Fondamenti di calcolo variazionale.}
\label{sec_calc_var}

Oggetto di studio del calcolo variazionale \e la minimizzazione di applicazioni (comunemente dette {\em funzionali}) definite su spazi di funzioni. In termini precisi, consideriamo
\[
f : \bR^3 \to \bR 
\ \ , \ \ 
f = f(t,x,p)
\ 
{\mathrm{di \ classe }} \ C^2
\ ,
\]
$a < b \in \bR$ e lo spazio topologico
\[
\mX 
:=
\{ u \in C^1([t_0,t_1]) : u(a) = 0 , u(b) = L \}
\ \ , \ \
L \in \bR
\ ;
\]
definiamo quindi il funzionale
\begin{equation}
\label{eq_CV1}
F : \mX \to \bR
\ \ , \ \
F(u) := \int_a^b f ( t , u(t) , u'(t) ) \ dt
\ .
\end{equation}
In nostro problema \e quello di trovare $u \in \mX$ che minimizzi $F$.
Come vedremo nelle sezioni seguenti, esiste una stretta relazione tra il calcolo variazionale e la teoria delle equazioni alle derivate parziali. 

\

Diamo ora una condizione necessaria all'esistenza di un minimo $u \in \mX$ per $F$. Consideriamo funzioni $\varphi \in C_0^1([a,b])$ (ovvero, $\varphi(a) = \varphi(b) = 0$). Per ogni $\lambda \in \bR$ osserviamo che $v + \lambda \varphi \in \mX$, $v \in \mX$, e definiamo
$g (\lambda) := F( u + \lambda \varphi)$,
$\lambda \in \bR$,
cosicch\'e
\begin{equation}
\label{eq_EL1}
g'(\lambda)
=
\int_a^b 
\left(
\frac{\partial f}{\partial x} (\cdot) \ \varphi (t)
+
\frac{\partial f}{\partial p} (\cdot) \ \varphi'(t)
\right) \ dt
\ .
\end{equation}
Chiaramente, se $u$ \e di minimo per $f$ allora $\lambda = 0$ \e di minimo per $g$, per cui esplicitando la condizione $g'(0) = 0$ otteniamo il seguente risultato:
\begin{prop}
\label{prop_CV1}
Se $u \in \mX$ \e di minimo per il funzionale $F$ definito da (\ref{eq_CV1}), allora
\begin{equation}
\label{eq_CV2}
\int_a^b 
\left(
\frac{\partial f}{\partial x} (t,u(t),u'(t)) \ \varphi (t)
+
\frac{\partial f}{\partial p} (t,u(t),u'(t)) \ \varphi'(t)
\right) \ dt
\ = \
0
\ \ , \ \
\forall \varphi \in C_0^1([a,b])
\ .
\end{equation}
\end{prop}

\begin{thm}[Teorema di Eulero]
\label{thm_CV1}
Sia $u \in \mX$ di minimo per il funzionale $F$ definito da (\ref{eq_CV1}), $u \in C^2([a,b])$. Allora $u$ soddisfa l'equazione di Eulero-Lagrange
\begin{equation}
\label{eq_CV3}
\frac{d}{dt} \frac{\partial f}{\partial p} ( t , u(t) , u'(t) )
\ = \
\frac{\partial f}{\partial x} ( t , u(t) , u'(t) )
\ \ , \ \
t \in (a,b)
\ .
\end{equation}
\end{thm}

\begin{proof}[Dimostrazione]
Integrando per parti otteniamo
\begin{equation}
\label{eq_EL2}
\int_a^b 
\frac{\partial f}{\partial p} ( t , u(t) , u'(t) ) \ \varphi'(t) \ dt
\ = \
- \int_a^b 
\frac{d}{dt} \frac{\partial f}{\partial p} ( t , u(t) , u'(t) ) \ \varphi(t) \ dt
\ ,
\end{equation}
per cui
\[
\int_a^b
\varphi (t)
\
\left(
\frac{d}{dt} \frac{\partial f}{\partial p} ( t , u(t) , u'(t) ) 
\ - \
\frac{\partial f}{\partial x} ( t , u(t) , u'(t) )
\right)
\ dt
\ = \
0 
\ \ , \ \
\varphi \in C_0^1([a,b])
\ .
\]
Da quest'ultima espressione \e semplice concludere che il fattore moltiplicato per $\varphi$ nella funzione integranda deve essere nullo, ed il teorema \e dimostrato.
\end{proof}


\

Il teorema di Eulero fornisce una condizione necessaria affinch\'e esista una soluzione $u$ del problema variazionale associato al funzionale (\ref{eq_CV1}). Un'importante classe di funzionali tali che l'equazione di Eulero-Lagrange \e anche condizione sufficiente \e quella dei {\em funzionali convessi}, ovvero
\begin{equation}
\label{eq_CV4}
F ( \lambda u + (1-\lambda)v  )
\leq
\lambda F(u) + (1-\lambda)F(v)
\ \ , \ \
\lambda \in [0,1]
\ , \
u,v \in \mX
\ .
\end{equation}

\begin{thm}
\label{thm_CV2}
Sia dato il problema variazionale (\ref{eq_CV1}), con $F$ convesso. Una funzione $u \in \mX \cap C^1([a,b])$ \e di minimo per $F$ se e solo se soddisfa l'equazione di Eulero-Lagrange (\ref{eq_CV3}).
\end{thm}

\begin{proof}[Dimostrazione]
Grazie al teorema di Eulero, per dimostrare il teorema dobbiamo soltanto verificare che una funzione $u$ che soddisfi (\ref{eq_CV3}) \e di minimo per $F$. Ora, per ogni $v \in \mX$ risulta che 
$\varphi := v - u \in C_0^1([a,b])$
(si ricordino le condizioni al bordo per elementi di $\mX$). Per convessit\'a, troviamo
\[
F( u + \lambda \varphi)
=
F( \lambda v + (1-\lambda)u )
\leq
\lambda F(v) + (1-\lambda) F(u)
\ .
\]
Definendo 
\[
g (\lambda) := F(u+\lambda \varphi) = F ( \lambda v + (1-\lambda) u )
\ ,
\]
abbiamo la condizione
\[
g(\lambda) \leq \lambda g(1) + (1-\lambda)g(0)
\ \Rightarrow \
g(1) - g(0) \geq \frac{g(\lambda) - g(0)}{\lambda}
\ ,
\]
ovvero
$g(1) \geq g(0) + g'(0)$.
Ora, applicando (\ref{eq_EL1},\ref{eq_EL2}) otteniamo
\[
g'(\lambda)
\ = \
\int_a^b 
\varphi (t)
\
\left(
\frac{\partial f}{\partial x} ( \cdot ) 
-
\frac{d}{dt} \frac{\partial f}{\partial p} ( \cdot ) 
\right) \ dt
\ .
\]
Dunque, poich\'e $u$ soddisfa (\ref{eq_CV3}) otteniamo $g'(0) = 0$. Concludiamo quindi che $g(1) \geq g(0)$, ovvero $F(v) \geq F(u)$.
\end{proof}

\begin{ex}{\it 
Si consideri il funzionale
\[
F(u) := \int_0^1 [u'(t)]^2 dt \ . 
\]
Poich\'e $f(t,x,p) = p^2$ \e convessa, $F$ \e convesso. L'equazione di Eulero-Lagrange associata ad $F$ \e $u''=0$, con condizioni al bordo $u(a) = 0$, $u(1) = L$. La soluzione \e quindi il minimo per $F$ \e
$u (t) = L(b-a)^{-1} (t-a)$.
}\end{ex}

\subsection{Misure ed integrali su spazi prodotto.}
\label{sec_misprod}

Torniamo ora nell'ambito della teoria della misura, occupandoci dell'integrazione su prodotti cartesiani di spazi misurabili.

\

\noindent \textbf{La costruzione di misure su spazi cartesiani.} Siano $(X,\mM,\mu)$, $(Y,\mN,\nu)$ spazi misurabili. Vogliamo equipaggiare il prodotto cartesiano $X \times Y$ di una misura, ed a tale scopo introduciamo la {\em famiglia dei rettangoli}
\[
\mR_{X,Y} := \{ A \times B \ : \ A \in \mM , B \in \mN   \} \ ,
\]
la quale 
{\footnote{Osservare che, nel caso $X = Y = \bR$ equipaggiato con la misura di Lebesgue, $\mR_{X,Y}$ ha elementi prodotti cartesiani di generici insiemi misurabili e non solo di intervalli, per cui il termine {\em rettangolo} \e da intendere in senso lato.}}
soddisfa le seguenti propriet\'a:

\begin{lem}
\label{lem_semial}
(1) $\emptyset \in \mR_{X,Y}$; 
(2) $R,R' \in \mR_{X,Y}$ $\Rightarrow$ $R \cap R' \in \mR_{X,Y}$; 
(3) Per ogni $R \in \mR_{X,Y}$ esistono $R_1 , \ldots , R_n \in \mR_{X,Y}$ disgiunti tali che $R^c = \dot{\cup}_{k=1}^n R_k$. 
\end{lem}

\begin{proof}[Dimostrazione]
(1) Ovviamente $\emptyset \times \emptyset = \emptyset$;
(2) Posto $R = A \times B$, $R' = A' \times B'$, abbiamo
    $R \cap R' = (A \cap A') \times (B \cap B')$,
    con $A \cap A' \in \mM$, $B \cap B' \in \mN$.
(3) Posto $R = A \times B$, abbiamo
    $R^c = (A^c \times B) \dot{\cup} (A \times B^c) \dot{\cup} (A^c \times B^c)$,
    con $A,A^c \in \mM$, $B,B^c \in \mN$.
\end{proof}

Dal Lemma precedente segue che se $R,R' \in \mR_{X,Y}$ ed $R \subset R'$ allora esistono $R_1 , \ldots , R_n \in \mR_{X,Y}$ disgiunti tali che
\begin{equation}
\label{eq_3p}
R' \, = \, R \, \dot{\cup} \, R_1 \, \dot{\cup} \ldots \dot{\cup} \, R_n \ ;
\end{equation}
infatti, usando (3) abbiamo $R^c = \dot{\cup}_k^n R'_k$ con $\{ R'_k \} \subset \mR_{X,Y}$; per cui
$R^c \cap R' = \dot{\cup}_k R_k$, $R_k := R' \cap R'_k$,
e quindi
$R' = R \dot{\cup} (R^c \cap R') = R \dot{\cup} R_1 \dot{\cup} \ldots \dot{\cup} R_n$.
Definiamo ora l'applicazione
\[
\lambda : \mR_{X,Y} \to \wa \bR^+
\ \ , \ \
\lambda (A \times B) := \mu(A) \cdot \nu(B)
\ \ , \ \
A \in \mM 
\ , \
B \in \mN
\ .
\]

\begin{lem}
\label{lem_semial2}
L'applicazione $\lambda$ soddisfa le seguenti propriet\'a: 
(i) $\lambda \emptyset = 0$; 
(ii) Se $\{ R_n \}$ \e una successione di rettangoli disgiunti tale che $R  := \dot{\cup}_n R_n \in \mR_{X,Y}$, allora $\lambda R = \sum_n \lambda R_n$;
(iii) Se $\{ R_n \}$ \e una successione di rettangoli ed $R = \cup_n R_n \in \mR_{X,Y}$, allora $\lambda R \leq \sum_n \lambda R_n$;
(iv) Se $R,R' \in \mR_{X,Y}$ ed $R \subseteq R'$ allora $\lambda R \leq \lambda R'$.
\end{lem}

\begin{proof}[Dimostrazione]
(i) \e ovvia. Riguardo (ii) scriviamo $R := A \times B$, $R_n := A_n \times B_n$ ed osserviamo che, pur essendo gli elementi di $\{ R_n \}$ a due a due disgiunti, potremmo tranquillamente avere $A_n \cap A_m \neq \emptyset$, $B_{n'} \cap B_{m'} \neq \emptyset$, per opportune coppie di indici $n,m$ e $n',m'$. Consideriamo la funzione $f$ che associa ad $x \in X$ la somma (eventualmente numerabile) delle misure di tutti i $B_n$ tali che $(x,y) \in R_n$ per qualche $y \in Y$; chiaramente, possiamo scrivere 
$f(x) = \sum_n \chi_{A_n}(x) \nu B_n$, $x \in X$.
D'altra parte, fissato $x \in A$ abbiamo che {\em ogni} $y \in B$ \e tale che $(x,y)$ appartiene {\em ad uno, ed un solo}, $R_n$, per cui deve essere, per additivit\'a numerabile di $\nu$,
$f(x) = \chi_A(x) \nu B$;
concludiamo che
\[
\chi_A(x) \nu B = 
\sum_n \chi_{A_n}(x) \nu B_n
\ .
\]
Integrando ed usando il Teorema di convergenza monot\'ona troviamo
\[
\lambda R = 
\mu A \cdot \nu B = 
\int \chi_A \ d \mu \cdot \nu B = 
\int \sum_n \chi_{A_n} \nu B_n \ d \mu  =
\sum_n \mu A_n  \nu B_n =
\sum_n \lambda R_n \ .
\]
Riguardo (iii), ripetiamo il ragionamento del punto precedente osservando che stavolta, fissato $x \in A$, la generica coppia $(x,y)$, $y \in Y$, apparterr\'a ad uno {\em o pi\'u} rettangoli $R_n$, per cui abbiamo
$\chi_A(x) \nu B \leq \sum_n \chi_{A_n}(x) \nu B_n$
e quindi, integrando membro a membro,
\[
\lambda R = 
\int \chi_A \ d \mu \cdot \nu B 
\ \leq \
\int \sum_n \chi_{A_n} \nu B_n \ d \mu  =
\sum_n \lambda R_n \ .
\]
Infine, per quanto attiene a (iv) osserviamo che usando (\ref{eq_3p}) e (ii) troviamo
\[
\lambda R' = \lambda R + \sum_i^n \lambda R_i \geq \lambda R \ .
\]
\end{proof}

\begin{thm}[La misura prodotto]
Dati gli spazi di misura completi $(X,\mM,\mu)$, $(Y,\mN,\nu)$, esiste ed \e unica la misura completa $\mu \times \nu$ su $X \times Y$ che estende $\lambda$ alla $\sigma$-algebra generata da $\mR_{X,Y}$.
\end{thm}

\begin{proof}[Dimostrazione]
Come primo passo mostriamo che
\begin{equation}
\label{eq_lambda}
\lambda^* : 2^{X \times Y} \to \wa \bR^+
\ \ , \ \
\lambda^*A 
\ := \
\inf_{A \subseteq \cup_n R_n 
:
\{ R_n \} \subset \mR_{X,Y} }
\sum_n \lambda R_n
\end{equation}
\e una misura esterna. Chiaramente $\lambda^* \emptyset = 0$; inoltre se $A \subseteq A'$ allora $A' \subseteq \cup_n R_n$, $\{ R_n \} \subset \mR_{X,Y}$, implica $A \subseteq \cup_n R_n$ e quindi (passando agli inf) $\lambda^* A \leq \lambda^* A'$. 
Rimane infine da verificare la subadditivit\'a numerabile. Se $\{ A_n \} \subset 2^{X \times Y}$ allora possiamo assumere $\lambda^*A_n < \infty$ per ogni $n \in \bN$, altrimenti non vi \e nulla da dimostrare. Per definizione di $\lambda^*$, per ogni $\eps > 0$ ed $n \in \bN$ esiste una successione $\{ R_{n,k} \}$ tale che
\[
A_n \subseteq \cup_k R_{n,k}
\ \ , \ \
\sum_k \lambda R_{n,k} 
< 
\lambda^* A_n + 2^{-n} \eps
\ .
\]
Poich\'e $\cup_n A_n \subseteq \cup_{n,k}R_{n,k}$, usando la definizione di $\lambda^*$ e le diseguaglianze precedenti troviamo
\[
\lambda^* \left( \cup_n A_n \right)
\leq
\sum_{n,k} \lambda R_{n,k}
< 
\sum_n \lambda^* A_n + \sum_n 2^{-n} \eps
= 
\sum_n \lambda^* A_n + \eps
\ ,
\]
e per arbitrariet\'a di $\eps$ segue la subadditivit\'a numerabile di $\lambda^*$, la quale \e quindi una misura esterna. Applicando il Lemma \ref{lem_car} otteniamo una misura completa $\mu \times \nu$ su $X \times Y$. 
Per ultimare la dimostrazione rimane da verificare soltanto il fatto che $\{ \mu \times \nu \}R = \lambda R$ per ogni $R \in \mR_{X,Y}$. A tale scopo osserviamo che, per costruzione, $\{ \mu \times \nu \}R = \lambda^* R$ (vedi Lemma \ref{lem_car}); d'altra parte \e evidente che l'inf attraverso il quale \e definito $\lambda^*R$ \e raggiunto proprio da $\lambda R$ (vedi Lemma \ref{lem_semial2}(iii)), per cui il teorema \e dimostrato.
\end{proof}

Denotiamo con $\mR_{\sigma \delta} \subseteq 2^{X \times Y}$ la classe dei sottoinsiemi di $X \times Y$ del tipo $E = \cap_n \left( \cup_m R_m^n \right)$, dove $R_m^n \in \mR_{X,Y}$ $\forall n,m \in \bN$; \e chiaro che ogni elemento di $\mR_{\sigma \delta}$ \e misurabile rispetto a $\mu \times \nu$.

\begin{lem}
\label{lem_sd}
Sia $E \subseteq X \times Y$ misurabile tale che $\{ \mu \times \nu \}E < \infty$. Allora esiste $R \in \mR_{\sigma \delta}$ tale che $E \subseteq R$ e $\{ \mu \times \nu \}E = \{ \mu \times \nu \}R$.
\end{lem}

\begin{proof}[Dimostrazione]
Per definizione di $\mu \times \nu$ (vedi (\ref{eq_lambda})) possiamo trovare una successione $\{ R_n \}$, $E \subseteq R_n$, $n \in \bN$, i cui elementi sono unione numerabile di rettangoli e tale che scelto $\eps > 0$ esiste $n_\eps \in \bN$ con
\[
\{ \mu \times \nu \}E \leq \{ \mu \times \nu \}R_n  \leq \{ \mu \times \nu \}E + \eps
\ \ , \ \
\forall n \geq n_\eps
\ \ \Rightarrow \ \
\{ \mu \times \nu \}E = \lim_n \{ \mu \times \nu \}R_n
\ .
\]
Eventualmente ridefinendo
$R_n \to \cap_i^n R_i$
possiamo assumere $R_n \subseteq R_{n+1}$, $n \in \bN$, con $\{ R_n \} \subseteq \mR_{\sigma \delta}$. Inoltre avendo $E$ misura finita possiamo assumere $\{ \mu \times \nu \} R_1 < \infty$. Definendo $R := \cap_n R_n$ ed applicando il Lemma \ref{lem_caressa} troviamo
$\{ \mu \times \nu \} R = \lim_n \{ \mu \times \nu \}R_n = \{ \mu \times \nu \}E$.
\end{proof}

\

\noindent \textbf{I teoremi di Fubini e Tonelli.} I risultati seguenti permettono di approcciare il calcolo di integrali di funzioni su spazi prodotto. Per prima cosa introduciamo il concetto di {\em sezione} di un insieme $E \subseteq X \times Y$:
\[
E_x := \{ y \in Y : (x,y) \in E \} \subseteq Y
\ \ , \ \
x \in X
\]
(ed analogamente si definisce $E_y \subseteq X$, $y \in Y$); abbiamo le ovvie propriet\'a
\[
(E^c)_x = (E_x)^c
\ \ , \ \
(\cup_i E_i)_x = \cup_i (E_i)_x
\ \ , \ \
\chi_{E_x}(y) = \chi_E (x,y) = \chi_{E_y}(x)
\ .
\]

\begin{lem}
\label{lem_fub1}
Sia $E \in \mR_{\sigma \delta}$. Allora $E_x$ \e misurabile per ogni $x \in X$.
\end{lem}

\begin{proof}[Dimostrazione]
Il Lemma \e banalmente vero se $E \in \mR_{X,Y}$. Assumendo quindi che
$E = \cup_n E_n$, $\{ E_n \} \subseteq \mR_{X,Y}$,
troviamo
\[
\chi_{E_x}(y) \, = \, \chi(x,y) \, = \, \sup_n \chi_{E_n}(x,y) \, = \, \sup_n \chi_{(E_n)_x}(y) \ .
\]
Poich\'e ogni $E_n \in \mR_{X,Y}$, abbiamo che ogni $(E_n)_x \subseteq Y$ \e misurabile; quindi $\chi_{(E_n)_x}$ \e misurabile e, grazie al Teo.\ref{thm_MIS}, 
$\chi_{E_x} = \sup_n \chi_{(E_n)_x}$
\'e misurabile, il che vuol dire che $E_x$ \e misurabile.
Infine prendiamo $E = \cap_i E_i$, dove ogni $E_i$ \e unione numerabile di rettangoli (per cui per ogni $E_i$ vale l'enunciato del Lemma). Allora
\[
\chi_{E_x}(y) \, = \, \chi_E(x,y) \, = \, \inf_i \chi_{E_i}(x,y) \, = \, \inf_i \chi_{(E_i)_x}(y) \ ,
\]
e ragionendo come in precedenza concludiamo che $E_x$ \e misurabile.
\end{proof}

\begin{lem}
\label{lem_fub2}
Sia $E \in \mR_{\sigma \delta}$ con $\{ \mu \times \nu \}E < \infty$. Allora la funzione 
$f_E(x) := \nu E_x$, $x \in X$,
\e misurabile e 
$\int f_E  d \mu = \{ \mu \times \nu \}E$.
\end{lem}

\begin{proof}[Dimostrazione]
Il Lemma \e banalmente vero se $E \in \mR_{X,Y}$. Se $E$ \e unione numerabile di rettangoli osserviamo che possiamo comunque esprimerlo, usando iterativamente (\ref{eq_3p}), come un'unione disgiunta:
$E = \dot{\cup}_n E_n$, $\{ E_n \} \subseteq \mR_{X,Y}$.
Consideriamo quindi le funzioni
$f_n(x) := \nu (E_n)_x$, $x \in X$,
che sono positive e misurabili, ed osserviamo che $f_E = \sum_n f_n$, cosicch\'e concludiamo che $f_E$ \e misurabile. Per convergenza monot\'ona ed additivit\'a numerabile di $\mu \times \nu$ abbiamo
\[
\int f_E \ d \mu \, = \,
\sum_n \int f_n \ d \mu \, = \, 
\sum_n \{ \mu \times \nu \}E_n \, = \,
\{ \mu \times \nu \}E
\ ,
\]
per cui il Lemma \e vero per unioni numerabili di rettangoli. 
Sia ora $E = \cap_n E_n$ tale che ogni $E_n$ \e unione numerabile di rettangoli (per cui per ogni $E_n$ vale l'enunciato del Lemma). Eventualmente ridefinendo $E'_n := \cap_i^n E_i$ possiamo assumere che $E_{n+1} \subseteq E_n$ per ogni $n \in \bN$. D'altra parte, essendo $E$ di misura finita, possiamo sempre assumere, ricordando la definizione (\ref{eq_lambda}), che 
$\{ \mu \times \nu \} E_1 < \infty$. 
A questo punto, posto 
$f_n(x) := \nu (E_n)_x$, $x \in X$,
grazie al Teorema \ref{thm_MIS} concludiamo che $f_E = \lim_n f_n$ \e misurabile. D'altra parte abbiamo
\[
f_E(x) \, = \,
\nu E_x \, \stackrel{Lemma \ref{lem_caressa}}{=}  \,
\lim_n \nu (E_n)_x \, = \,
\lim_n f_n(x)
\ \ , \ \
{\mathrm{q.o. \ in \ }} x \in X \ ;
\]
per cui, avendosi 
$f_n \leq f_1$
e
$\int f_1 \ d \mu = \{ \mu \times \nu \} E_1 < \infty$,
per il teorema di convergenza dominata di Lebesgue concludiamo
\[
\int f_E \ d \mu \, =  \,
\lim_n \int f_n \ d \mu \, = \,
\lim_n \{ \mu \times \nu \}E_n \, \stackrel{Lemma \ref{lem_caressa}}{=} \,
\{ \mu \times \nu \}E
\ .
\]
\end{proof}

\begin{lem}
\label{lem_fub3}
Sia $E \subseteq X \times Y$ misurabile e tale che $\{ \mu \times \nu \}E = 0$. Allora q.o. in $x \in X$ si ha $E_x$ misurabile e $\nu E_x = 0$.
\end{lem}

\begin{proof}[Dimostrazione]
Usando il Lemma \ref{lem_sd} troviamo che esiste $R \in \mR_{\sigma \delta}$ con $E \subseteq R$ e 
$\{ \mu \times \nu \}E = \{ \mu \times \nu \}R = 0$,
per cui usando il Lemma precedente troviamo
\[
\{ \mu \times \nu \}R = \int f_R \ d \mu = 0
\ \Rightarrow \
f_R(x) = \nu R_x = 0
\ \
{\mathrm{q.o. \ in \ }}
x \in X
\ .
\]
D'altro canto per costruzione $E_x \subseteq R_x$ per ogni $x \in X$, ed essendo $\nu$ completa concludiamo che $E_x$ \e misurabile con $\nu E_x = 0$ per ogni $x$ tale che $\nu R_x = 0$.
\end{proof}

\begin{lem}
\label{lem_fub4}
Sia $E \subseteq X \times Y$ misurabile e tale che $\{ \mu \times \nu \}E < \infty$. Allora la funzione 
$f_E(x) := \nu E_x$, $x \in X$,
\e misurabile e 
$\int f_E  d \mu = \{ \mu \times \nu \}E$.
\end{lem}

\begin{proof}[Dimostrazione]
Grazie al Lemma \ref{lem_sd} sappiamo che esiste $R \in \mR_{\sigma \delta}$ tale che $E \subseteq R$ e 
$\{ \mu \times \nu \}E = \{ \mu \times \nu \}R$.
Cosicch\'e
\[
\infty > 
\{ \mu \times \nu \}R =
\{ \mu \times \nu \}E + \{ \mu \times \nu \}(R-E)
\ \Rightarrow \
\{ \mu \times \nu \}(R-E) = 0
\ .
\]
Grazie al Lemma \ref{lem_fub3} troviamo $\nu(R-E)_x = 0$ q.o. in $x \in X$, per cui
$f_E(x) = f_R(x) := \nu R_x$, q.o. in $x \in X$.
Del resto $f_R$ \e misurabile per il Lemma \ref{lem_fub2}, e grazie a Prop \ref{prop_fgmis} concludiamo che $f_E$ \e misurabile.
Infine, sempre per il Lemma \ref{lem_fub2} concludiamo che
$\int f_E \ d \mu = \{ \mu \times \nu \}R = \{ \mu \times \nu \}E$.
\end{proof}

Il teorema seguente costituisce lo strumento principale per il calcolo esplicito di integrali su spazi prodotto e fornisce la famosa regola dello scambio dell'ordine di integrazione:
\begin{thm}[Fubini]
Sia $f \in L_{\mu \times \nu}^1(X \times Y)$.
Allora q.o. in $x \in X$ si ha $f(x,\cdot) \in L_\nu^1(Y)$,  e la funzione $F(x) := \int_Y f(x,y) d \nu$, $x \in X$, \e in $L_\mu^1(X)$. 
Inoltre, l'affermazione analoga \e vera scambiando i ruoli di $X$ ed $Y$, e
\begin{equation}
\label{eq_Fub1}
\int_{X \times Y} f(x,y) \ d \{ \mu \times \nu \}
\ = \
\int_X \left( \int_Y f (x,y) d \nu \right) d \mu
\ = \
\int_Y \left( \int_X f (x,y) d \mu \right) d \nu
\ .
\end{equation}
\end{thm}

\begin{proof}[Dimostrazione]
Grazie alla simmetria tra $x$ ed $y$ \e sufficiente dimostrare il teorema senza scambiare i ruoli delle due variabili. Se la conclusione del teorema \e vera per due funzioni allora \e vera anche per la loro differenza, per cui possiamo ridurci a considerare il caso $f \geq 0$. Il Lemma \ref{lem_fub4} afferma che il teorema \e vero se $f$ \e la funzione caratteristica di un insieme $E \subseteq X \times Y$ di misura finita, per cui esso \e vero anche per funzioni semplici che si annullano al di fuori di un insieme di misura finita (denotiamo con $S_c(X \times Y)$ l'insieme di tali funzioni).
Ora, ogni funzione integrabile non negativa \e limite puntuale di una successione $\{ \psi_n \} \subset S_c(X \times Y)$ monot\'ona crescente (vedi Prop.\ref{lem_sf}), per cui applicando il teorema di Beppo Levi
\[
F(x) := \int_Y f(x,y) \ d \nu = \lim_n \int_Y \psi_n(x,y) \ d \nu
\ ;
\]
essendo il teorema vero per ogni $\psi_n$ concludiamo che, essendo $F$ limite puntuale di funzioni misurabili, \e essa stessa misurabile (Teorema \ref{thm_MIS}). Applicando ancora il teorema di Beppo Levi, ed usando ancora il fatto che il teorema \e vero per ogni $\psi_n$, troviamo
\[
\int F \ d \mu = 
\int_X \int_Y f \ d \mu d \nu =
\lim_n \int_X \int_Y \psi_n \ d \mu d \nu = 
\lim_n \int_{X \times Y} \psi_n d \{ \mu \times \nu \} =
\int_{X \times Y} f \ d \{ \mu \times \nu \}
\ .
\]
Ci\'o conclude la dimostrazione.
\end{proof}

\begin{thm}[Tonelli]
Siano $(X,\mM,\mu)$,$(Y,\mN,\nu)$ spazi di misura $\sigma$-finiti ed $f : X \times Y \to \wa \bR^+$ una funzione misurabile non negativa. Allora:
(1) $f(x,\cdot) : Y \to \wa \bR^+$ \e una funzione misurabile q.o. in $x \in X$; 
(2) $F(x) := \int_Y f(x,y) \ d\nu$, $x \in X$, \e misurabile;
(3) le stesse propriet\'a (1,2) sono vere scambiando i ruoli di $x,y$;
(4) \e verificata l'uguaglianza (\ref{eq_Fub1}).
\end{thm}

\begin{proof}[Dimostrazione]
L'unico punto della dimostrazione del teorema di Fubini in cui usiamo l'integrabilit\'a di $f$ \e dove affermiamo che $f$ \e limite puntuale di funzioni in $S_c(X \times Y)$. Del resto, Prop.\ref{prop_caressa} afferma che $f$ \e limite puntuale di funzioni in $S_c(X \times Y)$ con la sola ipotesi di misurabilit\'a, a patto che $X \times Y$ sia $\sigma$-finito. Ci\'o \e senz'altro vero se $X,Y$ sono $\sigma$-finiti, per cui possiamo ripetere con successo l'argomento della dimostrazione del Teorema di Fubini ed ottenere le propriet\'a desiderate.
\end{proof}

Nel teorema precedente si considerano funzioni $f \geq 0$ e non si fa nessuna affermazione sull'inte- grabilit\'a di $f$. Tuttavia se si suppone che $f(x,\cdot) \in L_\mu^1(Y)$ q.o. in $x \in X$ ed $F \in L_\mu^1(X)$, allora il punto (4) permette di concludere che $f \in L_{\mu \times \nu}^1(X \times Y)$.

\subsection{Convoluzioni.}
\label{sec_conv}

Un'importante applicazione dei teoremi di Fubini e Tonelli \e quella dei prodotti di convoluzione, i quali a loro volta hanno un ruolo importante nell'ambito degli spazi $L^p$, di Sobolev, e nell'analisi di Fourier. In questa sezione consideriamo gli spazi euclidei $\bR^d$, $d \in \bN$, equipaggiati con la misura prodotto di Lebesgue (che otteniamo iterando $d-1$ volte sulla retta reale la costruzione della sezione precedente).

\begin{thm}
\label{thm_conv}
Siano $f \in L^1(\bR^d)$, $g \in L^p(\bR^d)$, $p \in [1,+\infty]$. Allora q.o. in $x \in \bR^d$ la funzione 
\[
K(x,y) := f(x-y)g(y)
\ \ , \ \
x,y \in \bR^d \ ,
\]
\e integrabile rispetto ad $y$ su $\bR^d$. Di conseguenza \e ben definita la funzione
\begin{equation}
\label{eq_conv1}
f*g (x) \ := \ \int_{\bR^d} f(x-y)g(y) \ dy \ \ , \ \ x \in \bR^d  \ ,
\end{equation}
la quale soddisfa la diseguaglianza
\begin{equation}
\label{eq_conv2}
\| f*g \|_p 
\ \leq \
\| f \|_1 \| g \|_p
\ ,
\end{equation}
per cui $f*g \in L^p(\bR^d)$.
\end{thm}

\begin{proof}[Dimostrazione]
Dimostriamo il teorema distinguendo i vari casi per $p \in [1,+\infty]$.
\textbf{(1) $p = \infty$}: l'enunciato \e ovvio. 
\textbf{(2) $p = 1$}. Q.o. in $y \in \bR^d$ si ha
\[
\int_{\bR^d} |K(x,y)| \ dx
\ = \
|g(y)| \int_{\bR^d} |f(x-y)| \ dx
\ = \
|g(y)| \ \| f \|_1
\ ,
\]
per cui, avendosi $g \in L^1(\bR^d)$ troviamo $K(\cdot,y) \in L^1(\bR^d)$ q.o. in $y \in \bR^d$. Inoltre
\begin{equation}
\label{eq_conv3}
\int_{\bR^d} dy \left( \int_{\bR^d} |K(x,y)| \ dx \right) 
\ \leq \ 
\| f \|_1 \| g \|_1 
\ < \ 
+ \infty \ .
\end{equation}
Per il teorema di Tonelli abbiamo $K \in L^1(\bR^d \times \bR^d)$, mentre per Fubini concludiamo che $K(x,\cdot) \in L^1(\bR^d)$ q.o. in $x \in \bR^d$, ovvero 
\[
\int_{\bR^d} |f(x-y)g(y)| dy < +\infty 
\ \ , \ \ 
{\mathrm{q.o. \ in}} \ x \in \bR^d \ ,
\]
il che dimostra (\ref{eq_conv1}). Per dimostrare che $f * g \in L^1(\bR^d)$, basta osservare che
\[
\begin{array}{lll}
\int_{\bR^d} dx \left| \int_{\bR^d} f(x-y)g(y) \ dy \right| 
& \ \ \ \leq &
\int_{\bR^d} dx \left( \int_{\bR^d} |f(x-y)g(y)| \ dy \right)
\\ \\ & \stackrel{Fubini}{=} &
\int_{\bR^d} dy \left( \int_{\bR^d} |f(x-y)g(y)| \ dx \right)
\\ \\ & \stackrel{(\ref{eq_conv3})}{\leq} &
\| f \|_1 \| g \|_1 \ .
\end{array}
\]
\textbf{(3) $p \in (1,+\infty)$}. Sia $g \in L^p(\bR^d)$ e $K_p (x,y) := |f(x-y)| |g(y)|^p$. Grazie a quanto mostrato per $p = 1$, abbiamo 
\[
K_p(x,\cdot) \in L^1(\bR^d)
\ \ {\mathrm{ovvero}} \ \
|K_p(x,\cdot)|^{1/p} \in L^p(\bR^d)
\ \ , \ \
{\mathrm{q.o. \ in}} \ x \in \bR^d \ .
\]
Poniamo $q := \ovl{p}$. Poich\'e $f \in L^1(\bR^d)$, abbiamo
\[
|\delta f (x,\cdot)|^{1/q} \in L^q(\bR^d)
\ \ , \ \
\delta f (x,y) := f(x-y)
\ \  , \ \ 
{\mathrm{q.o. \ in}} \ x \in \bR^d \ .
\]
Applicando la disuguaglianza di Holder otteniamo, q.o. in $x \in \bR^d$,
\[
+ \infty 
\ \stackrel{Holder}{>} \
\int_{\bR^d} |K_p(x,\cdot)|^{1/p} \cdot |\delta f (x,\cdot)|^{1/q}
\ = \
\int_{\bR^d} |f(x-y)|^{1/p} |g(y)| \cdot |f(x-y)|^{1/q} \ dy
\ \geq \
|f*g (x)|
\ .
\]
Per cui $f*g (x)$ \e definito q.o. in $x \in \bR^d$. Ora,
\[
\left\{
\begin{array}{ll}
\| |K_p(x,\cdot)|^{1/p} \|_p^p
=
\int_{\bR^d} |f(x-y)||g(y)|^p \ dy
= 
\left( |f|*|g|^p (x) \right)
\\ \\
\| |\delta f (x,\cdot)|^{1/q} \|_q^q
= 
\int_{\bR^d} |f(x-y)| \ dy 
= 
\| f \|_1
\ ,
\end{array}
\right.
\]
per cui scrivendo esplicitamente la precedente disuguaglianza di Holder otteniamo
\[
\left( |f|*|g|^p (x) \right)^{1/p} 
\cdot  
\| f \|_1^{1/q}
\ \geq \
|f*g(x)|
\ \Rightarrow \
|f|*|g|^p (x)
\cdot  
\| f \|_1^{p/q}
\ \geq \
|f*g(x)|^p
\ .
\]
Applicando a $|g|^p \in L^1(\bR^d)$ quanto mostrato nel caso $p=1$ abbiamo che $|f|*|g|^p \in L^1(\bR^d)$, per cui $|f*g|^p \leq |f|*|g|^p \cdot \| f \|_1^{p/q}$ \e integrabile.
\end{proof}

La funzione $f * g$ si dice {\em convoluzione} di $f$ e $g$. Qui di seguito ne elenchiamo alcune propriet\'a elementari, la cui dimostrazione \e lasciata per esercizio:
\begin{enumerate}
\item $f * g = g * f$, $f \in L^1(\bR^d)$, $g \in L^p(\bR^d)$;
\item $f_1 * ( f_2 * g ) = (f_1 * f_2) * g$, $f_1 ,f_2 \in L^1(\bR^d)$;
\item $f * (g+h) = f*g + f*h$, $h \in L^p(\bR^d)$;
\item $(af)*g = f * (ag)$, $a \in \bR^d$.
\end{enumerate}

Per approcciare la questione della derivabilit\'a di una convoluzione introduciamo alcune nozioni. Per prima cosa consideriamo $F \in L^1(\bR^d,\bR^m)$ nel senso di Oss.\ref{rem_Lp_C} ed osserviamo che possiamo scrivere, in termini vettoriali,
$F = (F_1 , \ldots , F_m )$, 
dove $F_i \in L^1(\bR^d)$, $i = 1 , \ldots , m$. Presa quindi $f \in L^p(\bR^d)$ definiamo la convoluzione
\[
F*f : \bR^d \to \bR^m
\ \ , \ \
(F*f)_i := F_i *f
\ \ , \ \
i = 1 , \ldots , m
\ ,
\]
cosicch\'e grazie al Teorema \ref{thm_conv} abbiamo $F*f \in L^p(\bR^d,\bR^m)$. Ora, se $g \in C^1(\bR^d)$ allora possiamo considerare il gradiente
\[
\nabla g : \bR^d \to \bR^d
\ \ , \ \
\nabla g := 
\left( \frac{\partial g}{\partial x_1} , 
       \ldots , 
       \frac{\partial g}{\partial x_d} \right)
\ ;       
\]
se $\nabla g \in L^1(\bR^d,\bR^d)$ ha quindi senso considerare la convoluzione 
$\nabla g * f \in L^p(\bR^d,\bR^d)$, $\forall f \in L^p(\bR^d)$.
Dalle definizioni precedenti si deduce banalmente il seguente risultato, conseguenza diretta dei teoremi di derivazione sotto il segno di integrale e della commutativit\'a del prodotto di convoluzione:
\begin{prop}
\label{prop_der_conv}
Sia $f \in L^1(\bR^d) \cap C^1(\bR^d)$ con $\nabla f \in L^1(\bR^d,\bR^d)$. Allora per ogni $g \in L^1(\bR^d)$ si ha che $f*g$ \e derivabile e 
\[
\nabla(f*g) = (\nabla f) * g
\ .
\]
Se anche $g$ \e derivabile con $\nabla g \in L^1(\bR^d,\bR^d)$, allora
$\nabla(f*g) = (\nabla f)*g = f*(\nabla g)$.
\end{prop}

\begin{rem}
\label{rem_conv_cont}
{\it
Le convoluzioni hanno la notevole propriet\'a di essere continue anche nel caso in cui n\'e $f$ n\'e $g$ lo siano: a questo proposito si veda l'Esercizio \ref{sec_Lp}.2. Ad esempio, invitiamo a verificare che, presi $0 < b \leq a$ e definite $\chi_a := \chi_{(-a,a)}$, $\chi_b := \chi_{(-b,b)}$, cosicch\'e $\chi_a , \chi_b \in L^p(\bR)$ $\forall p \in [1,+\infty]$, allora $\chi_a * \chi_b$ \e una funzione a supporto compatto, continua e lineare a tratti.
}
\end{rem}

\begin{rem}
\label{rem_conv_cont_2}
{\it
Sia $G$ un gruppo topologico localmente compatto di Hausdorff, e $\mu \in R(G)$ la misura di Haar (vedi \S \ref{sec_MIS1}). L'argomento della dimostrazione del teorema precedente si basa sui teoremi di Fubini-Tonelli, l'invarianza per traslazione e la diseguaglianza di Holder: queste propriet\'a sono tutte verificate dalla misura di Haar, per cui possiamo definire la convoluzione 
\[
f*g (s) := \int_G f(st^{-1}) g(t) \ d \mu (t)
\ , \
s \in G
\ \ , \ \
f \in L_\mu^1(G)
\ , \
g \in L_\mu^p(G)
\ ,
\]
anche in questo caso pi\'u generale
{\footnote{
Con la notazione $d \mu (t)$ intendiamo il fatto che stiamo integrando rispetto alla variabile $t \in G$.
}}. 
Per dettagli sull'argomento si veda \cite[\S 2.5]{Fol}  o\cite[Vol.I, Cap.5]{HR}.
}
\end{rem}

Ora, \e facilmente verificabile che nessuna funzione $g \in L^1(\bR)$ \e un'identit\'a rispetto al prodotto di convoluzione, ovvero $f * g = f$, $\forall f \in L^1(\bR)$. Tuttavia \e possibile introdurre dei buoni sostituti dell'identit\'a, che definiamo qui di seguito.
\begin{defn}
\label{def_moll}
Una successione di funzioni $\{ \rho_n \} \subset L^1(\bR^d)$ si dice \textbf{identit\'a approssimata} se
\[
\rho_n \geq 0
\ \ , \ \
{\mathrm{supp}}(\rho_n) \subseteq \ovl{\Delta(0,1/n)}
\ \ , \ \
\int \rho_n = 1
\ \ , \ \ 
\forall n \in \bN
\ .
\]
In particolare, diremo che $\{ \rho_n \}$ \e una successione di \textbf{mollificatori} se $\rho_n \in C_c^\infty(\bR^d)$ 
{\footnote{
Qui con $C_c^\infty(\bR^d)$ intendiamo lo spazio vettoriale delle funzioni $C^\infty$ su $\bR^d$ a supporto compatto.
}}
$\forall n \in \bN$.
\end{defn}

\begin{rem}
{\it
\textbf{(1)} Possiamo definire in modo del tutto analogo identit\'a approssimate indicizzate da un parametro a valori reali, piuttosto che dai numeri naturali; 
\textbf{(2)} Le nozioni di identit\'a approssimata e successione di mollificatori si possono dare senza variazioni nel caso di funzioni in $L^1(\bR^d,\bC)$.
\textbf{(3)} Se $f$ ha supporto compatto allora ogni $\rho_n*f$ ha supporto compatto.
}
\end{rem}

\begin{prop}
\label{prop_moll}
Sia $p \in [1,+\infty)$ e $\{ \rho_n \}$ un'identit\'a approssimata. Allora per ogni $g \in L^p(\bR^d)$ risulta $\| g - \rho_n * g \|_p \to 0$; in particolare, se $\{ \rho_n \}$ \e una successione di mollificatori allora ogni $\rho_n*g$ \e di classe $C^\infty$ e $C_c^\infty(\bR^d)$ \e denso in $L^p(\bR^d)$ in norma $\| \cdot \|_p$.
\end{prop}

\begin{proof}[Dimostrazione]
Iniziamo approssimando in norma $\| \cdot \|_p$ una funzione $g \in C_c(\bR^d)$. Per continuit\'a e compattezza del supporto abbiamo che $g$ \e uniformemente continua, per cui per ogni $\eps > 0$ esiste $\delta > 0$ tale che se $|y| < \delta$ allora $|g(x-y)-g(x)|< \eps$. Per $1/n < \delta$ troviamo
\[
\begin{array}{ll}
| (\rho_n * g) (x) - g(x) |
& \leq
\int_{\bR^d} \left| g(x-y)-g(x) \right| \rho_n(y) \ dy
\\ \\ & =
\int_{\Delta(0,1/n)} \left| g(x-y)-g(x) \right| \rho_n(y) \ dy
\\ \\ & \leq
\eps \int_{\bR^d} \rho_n = \eps
\ .
\end{array}
\]
Ci\'o implica $\| \rho_n*g-g \|_\infty \to 0$; avendo $g$ ed ogni $\rho_n*g$ supporto compatto (il quale non si ingrandisce al crescere di $n$) concludiamo che 
\[
\| \rho_n*g-g \|_p \leq K^{1/p} \| \rho_n*g-g \|_\infty \to 0
\ ,
\]
dove $K < \infty$ \e la misura del supporto di $\rho_1*g$. Se $f \in L^p(\bR^d)$, allora grazie a Cor.\ref{cor_appLp} esiste $f_\eps \in C_c(\bR^d)$ tale che $\| f-f_\eps \|_p < \eps$ e quindi
\[
\| \rho_n*f-f \|_p  
\leq
\| \rho_n*(f-f_\eps) \|_p + \| \rho_n*f_\eps-f_\eps \|_p + \| f_\eps-f \|_p
\leq
2 \| f-f_\eps \|_p + \| \rho_n*f_\eps-f_\eps \|_p 
<
3 \eps
\]
per $n$ abbastanza grande. Infine, se $\{ \rho_n \}$ \e una successione di mollificatori allora Prop.\ref{prop_der_conv} implica che $\rho_n * g \in C^\infty(\bR^d)$ per ogni $n$.
\end{proof}

\begin{rem}
\label{rem_moll}
{\it
L'argomento della proposizione precedente dimostra anche il seguente risultato: se $g \in C(\bR^d)$ \e uniformemente continua e limitata, allora $\| \rho_n*g-g \|_\infty \to 0$.
}
\end{rem}

\begin{ex}{\it 
\label{ex_moll_1}
Si consideri la funzione
\[
\rho (x) :=
\left\{
\begin{array}{ll}
e^{\frac{1}{x^2-1}} \ \ , \ \ |x| < 1
\\
0  \ \ , \ \ |x| \geq 1
\end{array}
\right.
\]
Allora $\rho \in L^1(\bR)$, e definendo $\rho_n (x) :=$ $\| \rho \|_1^{-1} n \rho (nx)$, $x \in \bR$, si ottiene una successione di mollificatori. 
}\end{ex}
\begin{ex}{\it 
\label{ex_moll_2}
Per ogni $n \in \bN$, consideriamo la funzione
\[
\rho_n
\ := \
n \chi_{ \left( 0 , \frac{1}{n}  \right) }
\ .
\]
Chiaramente $\{ \rho_n \}$ \e un'identit\'a approssimata. Osserviamo che applicando Prop.\ref{prop_moll} ed il Teorema di Fischer-Riesz otteniamo, per ogni $f \in L^1(\bR)$, 
\[
f * \rho_{n_k} (x) 
\ = \
n_k \int_\bR 
f(x-y) \chi_{ \left( 0 , 1 / n_k  \right) } (y) \ dy
\ = \ 
n_k \int_x^{x+1/{n_k}} f(s) \ ds
\ \stackrel{k}{\to} \ 
f(x)
\ \ , \ \
{\mathrm{q.o. \ in \ }} 
x \in \bR
\ .
\]
Quindi abbiamo dimostrato una versione dell'uguaglianza
\begin{equation}
\label{eq_int_der}
f(x) 
\ = \
\lim_{h \to 0} \frac{1}{h} \int_x^{x+h} f(s) \ ds
\ \ , \ \
f \in L^1(\bR)
\ \ , \ \
{\mathrm{q.o. \ in \ }} 
x \in \bR
\ ,
\end{equation}
dimostrabile altrimenti usando la derivabilit\'a q.o. delle funzioni primitive.
}\end{ex}

\noindent \textbf{Le convoluzioni dal punto di vista dell'analisi funzionale.} Nelle righe che seguono faremo uso delle nozioni di {\em norma di un operatore} (vedi (\ref{eq_defnorm})) ed {\em algebra di Banach} (Def.\ref{def_alg}). Denotiamo con 
$BL^p(\bR^d)$, $p \in [1,+\infty]$,
l'algebra di Banach degli operatori lineari limitati da $L^p(\bR^d)$ in s\'e. Teo.\ref{thm_conv} afferma che l'operatore
\begin{equation}
\label{def_Cp}
C_pf \in BL^p(\bR^d) \ \ , \ \ \{ C_pf \}(g) := f*g \ , \ \forall g \in L^p(\bR^d) \ ,
\end{equation}
ha norma $\leq \| f \|_1$. Considerando una successione di mollificatori $\{ \rho_n \} \subset L^1(\bR^d)$, otteniamo
\[
\| \{ C_1 f \} (\rho_n) \|_1
\ = \
\|  f*\rho_n \|_1
\ \stackrel{n}{\to} \
\| f \|_1
\ ,
\]
per cui $\| C_1 f \|$ ha norma esattamente pari a $\| f \|_1$. Inoltre \e ovvio che $f * (\wt f * g) = (f*\wt f)*g$, $f,\wt f \in L^1(\bR^d)$, $g \in L^p(\bR^d)$, per cui abbiamo dimostrato:
\begin{thm}
Lo spazio $L^1(\bR^d)$, equipaggiato del prodotto di convoluzione, \e un'algebra di Banach commutativa che denotiamo con $( L^1(\bR^d) , * )$, e (\ref{def_Cp}) definisce un'applicazione lineare isometrica
\[
C_p : ( L^1(\bR^d) , * ) \to BL^p(\bR^d)
\ \ , \ \
f \mapsto C_pf
\ ,
\]
tale che $C_p(f * \wt f) = \{ C_p f \} \circ \{ C_p \wt f \}$, per ogni $f , \wt f \in L^1(\bR^d)$.
\end{thm}
Un ulteriore aspetto interessante della convoluzione \e il suo rapporto con la trasformata di Fourier, come vedremo nel seguito.

\begin{rem}\label{rem_conv}{\it 
L'argomento della dimostrazione di Teo.\ref{thm_conv} funziona, pi\'u in generale, se invece della funzione $\delta f (x,y) := f(x-y)$ consideriamo
$\varphi \in L^1(\bR^d \times \bR^d)$
tale che esista $c > 0$ con
\[
\| \varphi (x,\cdot) \|_1 
\ , \
\| \varphi (\cdot,y) \|_1 
\ \leq \ 
c
\ \ , \ \
{\mathrm{q.o. \ in \ }} x,y \in \bR^d
\ .
\]
Cosicch\'e, per ogni $p \in [1,+\infty]$ e $g \in L^p(\bR^d)$, l'integrale
\[
C_\varphi g (x) := \int_{\bR^d} \varphi (x,y) g(y) \ dy 
\ \ , \ \
x \in \bR^d 
\ ,
\]
definisce un operatore lineare limitato $C_\varphi \in BL^p(\bR^d)$ tale che $\| C_\varphi \| \leq c$.
} \end{rem}

\noindent \textbf{La trasformata di Laplace}. Nello stesso ordine di idee dell'osservazione precedente consideriamo lo spazio $L^p(\bR^+)$, $\bR^+ := [0,+\infty)$, e definiamo {\em la trasformata di Laplace}
\begin{equation}
\label{eq_laplace_tr}
Lf (x) 
\ := \ 
\int_{\bR^+} e^{-xt} f(t) \ dt 
\ \ , \ \
f \in L^1(\bR^+)
\ \ , \ \
x \in \bR^+
\ .
\end{equation}
Abbiamo le seguenti propriet\'a:
\begin{itemize}
\item $L(f+ag) = Lf + a Lg$, $f,g \in L^1(\bR^+)$, $a \in \bR$; 
\item $Lf \in C_0(\bR^+)$; \\
      Infatti, poich\'e 
      $|e^{-xt} f(t)| \leq |f(t)|$, 
      possiamo applicare il teorema di Lebesgue e concludere che
      \[
      \left\{
      \begin{array}{ll}
      x_n \to x
      \ \Rightarrow \
      \lim_n Lf(x_n) = \lim_n \int_{\bR^+} e^{-x_nt} f(t) \ dt = Lf(x) \ ,
      \\
      x_n \to \infty
      \ \Rightarrow \
      \lim_n Lf(x_n) = \lim_n \int_{\bR^+} e^{-x_nt} f(t) \ dt = 0
      \ .
      \end{array}
      \right.
      \]
\item $\| Lf \|_\infty \leq \| f \|_1$; \\ 
      (basta osservare che $e^{-xt} \leq 1$, $x,t \in \bR^+$).
\item $L(f*g) = Lf \cdot Lg$; \\
      Infatti, prolunghiamo $f,g$ ad $\bR$ ponendo $f(x) = g(x) = 0$, $x < 0$,
      e, usando il teorema di Fubini, calcoliamo
      \begin{equation}
      \label{eq.pr.L}
      \begin{array}{ll}
      L \{ f*g \} (x) & =
      \int_{\bR^+} e^{-xt} f*g(t) \ dt \ =
      \\ & = 
      \int_{\bR^+} \int_\bR e^{-x(t-s)} f(t-s) \cdot e^{-xs}g(s) \ dtds \ =
      \\ & = 
      \int_{\bR^+} e^{-x \theta} f(\theta) \ d \theta \int_{\bR^+} e^{-xs}g(s) \ ds \ =
      \\ & = 
      Lf(x) \cdot Lg(x)
      \ .
      \end{array}
      \end{equation}
\end{itemize}
Cosicch\'e la trasformata di Laplace definisce un morfismo (limitato) di algebre di Banach
\[
L : ( L^1(\bR^+),* ) \to C_0(\bR^+) 
\ \ , \ \
f \mapsto Lf
\ .
\]

\subsection{Esercizi.}

\textbf{Esercizio \ref{sec_func_n}.1.} {\it Sia $p \in (1,+\infty)$. Si consideri l'identit\'a approssimata
\[
\rho := \frac{1}{2} \chi_{[-1,1]} 
\ \ , \ \
\rho_\eps := \frac{1}{\eps} \rho \left( \frac{x}{\eps} \right)
\ \ , \ \
\eps > 0
\ \ ,
\]
e, preso $\alpha \in (0,1/p)$, si calcoli la convoluzione $\rho_\eps * f$, dove
\[
f \in L^p(\bR)
\ \ , \ \
f(x) :=
\left\{
\begin{array}{ll}
x^{-\alpha} \ \ , \ \ x \in (0,1]
\\
0 \ \ , \ \  {\mathrm{altrimenti \ . }}
\end{array}
\right.
\]
Verificare che $\rho_\eps * f \in C_c(\bR) \cap L^p(\bR)$ per ogni $\eps > 0$, e che $\lim_{\eps \to 0} \| f - \rho_\eps * f \|_\infty = 0$.}

\

\noindent \textbf{Esercizio \ref{sec_func_n}.2.} {\it Sia
\[
f(x,y) := \frac{ x^2 - y^2 }{ (x^2+y^2)^2 }
\ \ , \ \
x,y \in [0,1]
\ .
\]
Si verifichi che $\int \left( \int f \ dx \right) dy \neq \int \left( \int f \ dy \right) dx$, e si dimostri che $f \notin L^1([0,1]^2)$.}

\

\noindent {\it Soluzione.} Per calcolare gli integrali di $f$ rispetto ad $x,y$, ricordiamo che
\[
\int \frac{s}{(1+s^2)} \ ds \ = \ - \frac{1}{2} \frac{1}{1+s^2}
\ \ \Rightarrow \ \
\int \frac{1}{(1+s^2)^2} \ ds \ = \  \frac{1}{2} 
                                     \left\{
                                     \arctan s + \frac{s}{1+s^2}
                                     \right\}
\ ,
\]
per cui
\[
\int \frac{s^2}{(1+s^2)^2} \ ds \ = \  \frac{1}{2} 
                                        \left\{
                                        \arctan s - \frac{s}{1+s^2}
                                        \right\}
\ .
\]
Possiamo ora calcolare
\[
\int_0^1 \frac{x^2-y^2}{ (x^2+y^2)^2 } \ dy
\ = \
x^2 \int_0^1 \frac{dy}{ x^4 \left( 1 + y^2/x^2 \right)^2 }
\ - \ 
\int_0^1 \frac{y^2}{ x^4 \left( 1 + y^2/x^2 \right)^2 } \ dy
\ = \
\]
\[
=
\frac{1}{x} 
\left\{ 
\int_0^{1/x} \frac{1}{(1+s^2)^2} \ ds
\ - \
\int_0^{1/x} \frac{s^2}{(1+s^2)^2} \ ds
\right\}
\ = \
\frac{1}{1+x^2}
\ .
\]
Usando la simmetria di $f$ rispetto ad $x,y$, otteniamo anche
\[
\int_0^1 \frac{x^2-y^2}{ (x^2+y^2)^2 } \ dx
\ = \
- \frac{1}{1+y^2}
\ ,
\]
per cui, integrando nuovamente rispetto ad $x$ ed $y$ rispettivamente
\[
\int_0^1 \left( \int_0^1 f(x,y) \ dy \right) dx = \frac{\pi}{4}
\ \ , \ \
\int_0^1 \left( \int_0^1 f(x,y) \ dx \right) dy = - \frac{\pi}{4}
\ .
\]
Dunque non valgono i teoremi di Fubini-Tonelli; ed infatti $f \notin L^1([0,1]^2)$, come si verifica ponendo $f^+(x,y) := \sup \{ f(x,y) , 0 \}$ e calcolando
\[
\int_{[0,1]^2} f^+(x,y) \ dx dy
\ = \
\int_0^1 \left(
         \int_0^x \frac{x^2-y^2}{(x^2+y^2)^2} \ dy
         \right) dx
\ = \
\frac{1}{2} \int_0^1 \frac{1}{x} \ dx
\ = \
+ \infty
\ .
\]

\

\noindent \textbf{Esercizio \ref{sec_func_n}.3.} {\it Si calcoli il minimo dei funzionali
\[
F(u) = \int_0^1 [ (u')^2 + 2tu ] \ dt 
\ \ , \ \
G(u) = \int_0^1 [ (u')^2 - 2tuu' + e^tu ] \ dt 
\ \ , \ \
u(0) = u(1) = 0 \ .
\]
}

\noindent \textbf{Esercizio \ref{sec_func_n}.4.} {\it Si trovi la soluzione $u$ del problema
\[
\left\{
\begin{array}{ll}
F(u) = \int_0^1 [ (u')^2 + 4u^2 ] \ dt
\\
\int_0^1 u = 1 
\ \ , \ \
u(0) = u(1) = 3
\end{array}
\right.
\]
}

\noindent \textbf{Esercizio \ref{sec_func_n}.5 (\cite[Es.8.1.1-2]{SCa}).} {\it Si studino i luoghi degli zeri di
\[
F(x,y) = y^2 + \lambda x^2 - x^3
\ \ , \ \
G(x,y) = e^{xy} (y^2 + 1) - y ( 1 + e^{2xy} )
\ ,
\]
al variare del parametro $\lambda \in \bR$.
}

\

\noindent \textbf{Esercizio \ref{sec_func_n}.6 (\cite[Es.7.2.6-7]{SCa}).} {\it Si dimostri che le seguenti forme differenziali sono esatte sui relativi domini:
\[
\left\{
\begin{array}{ll}
\omega 
\ = \  
\{ \ 1 - 2x [ (x^2+y^2)^{-2} ] \ \} \ dx + \{ \ 1 - 2y [ (x^2+y^2)^{-2} ] \ \} \ dy
\ \ , \ \
(x,y) \in \bR^2 - \{ 0 \}
\\
\varphi
\ = \
z(x+y)^{-1} ( dx + dy ) + \log (x+y) \ dz
\ \ , \ \
x,y,z > 0
\end{array}
\right.
\]
Inoltre, si trovi la primitiva $F$ di $\varphi$ tale che $F(1,1,1) = 1$. 

\

\noindent (Suggerimento: si osservi che il dominio di $\omega$ non \e stellato, per cui conviene usare \cite[Cor.8.2.1]{Giu2}).
}

\

\noindent \textbf{Esercizio \ref{sec_func_n}.7 (Convoluzioni di misure)} {\it Sia $\Lambda_\beta^1(\bR)$ lo spazio normato delle misure boreliane finite su $\bR$ (vedi Def.\ref{def_l1} ed Esercizio \ref{sec_MIS}.7). Per ogni $\lambda : \mM \to \bR$, $\lambda' : \mM' \to \bR$, $\lambda, \lambda' \in \Lambda_\beta^1(\bR)$, si mostri che:
\textbf{(1)} La misura prodotto $\lambda \times \lambda'$ \e boreliana e finita su $\bR^2$;
\textbf{(2)} Si definisca
\begin{equation}
\label{def_comv.mis}
\{ \lambda * \lambda' \}E := \int_{\bR^2} \chi_E(x+y) \ d \lambda(x)  d \lambda'(y)
\ \ , \ \
\forall E \in \mM \cap \mM'
\ ,
\end{equation}
dove $\chi_E$ \e la funzione caratteristica di $E$, e si mostri che $\lambda * \lambda' \in \Lambda_\beta^1(\bR)$;
\textbf{(3)} Si consideri l'applicazione 
\begin{equation}
\label{def_LL1}
\mu : L^1(\bR) \to \Lambda_\beta^1(\bR)
\ \ , \ \
f \mapsto \mu_f \ : \ \mu_fE := \int_E f \ \ , \ \ \forall E \in \mL \ .
\end{equation}
Si mostri che $\mu$ \e lineare, isometrica, e che $\mu_{f*g} = \mu_f * \mu_g$, $\forall f,g \in L^1(\bR)$. Si verifichi inoltre che $\mu$ non \e suriettiva.
\textbf{(4)} Si verifichi che la misura di Dirac $\delta_0$ appartiene a $\Lambda_\beta^1(\bR)$, e che
$\lambda * \delta_0 = \lambda$,
$\forall \lambda \in \Lambda_\beta^1(\bR)$.

\

\noindent (Suggerimenti: per il punto (2), riguardo l'additivit\'a numerabile di $\lambda * \lambda'$ si osservi che per ogni successione $\{ E_n \}$ di insiemi disgiunti la serie $\sum_n \chi_{E_n}$ \e positiva, monot\'ona crescente e puntualmente convergente a $\chi_E$, $E := \dot{\cup}_n E_n$, per cui si pu\'o applicare il Teorema di Beppo Levi; per il punto (3), riguardo la non suriettivit\'a si considerino le misure di Dirac).
}

\newpage
\section{Analisi Funzionale.}
\label{sec_afunct}

In questa sezione esponiamo le basi dell'analisi funzionale. Questo approccio fa uso di nozioni sia topologiche {\footnote{Si pensi alle locuzioni {\em spazio localmente convesso} od {\em operatore compatto}.}} che di algebra lineare ed i risultati che esso permette di conseguire hanno importanti applicazioni in analisi "hard".

Nelle pagine seguenti tratteremo prevalentemente spazi di Banach e di Hilbert a coefficienti sia reali che complessi. Il caso reale \e interessante per le applicazioni alla teoria delle equazioni alle derivate parziali (si veda \S \ref{sec_sobolev} e la relativa bibliografia). Gli spazi di Hilbert complessi hanno un ruolo importante in meccanica quantistica (\cite[Vol.I-IV]{RS}) ed in teoria delle rappresentazioni dei gruppi topologici (\cite{Fol,HR}), ivi compresa l'analisi di Fourier (\S \ref{sec_fourier}).

\subsection{Spazi di Banach e di Hilbert.}
\label{sec_afunct_1}

\noindent \textbf{Spazi di Banach.} Sia $\mE$ uno spazio vettoriale, reale o complesso. Una {\em seminorma} su $\mE$ \e una funzione $p : \mE \to \bR^+$ tale che
\[
p( \lambda v ) = |\lambda| p(v)
\ \ , \ \
p(v+w) \leq p(v)+p(w)
\ \ , \ \
\lambda \in \bR (\bC)  \ , \  v,w \in \mE
\ .
\]
In particolare, $p$ si dice {\em norma} se $p(v) = 0$ implica $v=0$. In tal caso, useremo la notazione $\| \cdot \| := p(\cdot)$ e diremo che $\mE$ \e uno {\em spazio normato}. Nel seguito, indicheremo con $\mE_1$ l'insieme degli elementi di $\mE$ con norma uguale a $1$. Se $f : \mE \to \bR$ \e un'applicazione lineare (ovvero, un {\em funzionale lineare}), allora introduciamo la notazione
\begin{equation}
\label{eq_AF1}
\| f \|
\, := \,
\sup_{v \in \mE} \frac{ | f(v) | }{\| v \|}
\, = \,
\sup_{v \in \mE_1} |f(v)|
\ ;
\end{equation}
se $\| f \| < + \infty$, allora diciamo che $f$ \e {\em limitato} o {\em continuo}. L'insieme $\mE^*$ dei funzionali lineari limitati da $\mE$ in $\bR$ (o $\bC$ qualora $\mE$ sia uno spazio complesso), equipaggiato con la norma (\ref{eq_AF1}), \e uno spazio normato, chiamato il {\em duale di $\mE$}. Nel seguito, utilizzeremo la notazione
\begin{equation}
\label{eq_AF2}
\left \langle f , v \right \rangle
\ := \
f(v)
\ \ , \ \
\forall f \in \mE^*
\ , \
v \in \mE
\ .
\end{equation}
Una successione $\{ v_n \} \subset \mE$ si dice {\em di Cauchy} se 
$| \leq \| v_n - v_m \| \stackrel{n,m}{\to} 0$;
osserviamo che
\[
| \ \| v_n \| - \| v_m \| \ | \leq \| v_n - v_m \| \stackrel{n,m}{\to} 0 \ ,
\]
per cui la successione delle norme $\{ \| v_n \| \}$ \e di Cauchy in $\bR$ e quindi convergente.
Una serie $\{ s_n := \sum_i^n v_i \} \subset \mE$ si dice {\em convergente} se esiste il limite $\lim_n s_n$, ed {\em assolutamente convergente} se esiste il limite della serie reale $\{ a_n := \sum_i^n \| v_i \| \} \subset \bR$.

Uno spazio normato $\mE$ si dice {\em di Banach} se esso \e completo rispetto alla topologia della norma (ovvero, ogni successione di Cauchy converge ad un elemento di $\mE$).

\begin{prop}
\label{prop_compl}
Un spazio normato $\mE$ \e di Banach se e solo se ogni serie assolutamente convergente \e convergente.
\end{prop}

\begin{proof}[Dimostrazione]
Se $\mE$ \e di Banach ed $\{ s_n := \sum_i^n v_i \}$ \e assolutamente convergente, allora 
\[
\| s_n - s_m \| \leq \sum_m^{i=n} \| v_i \| \stackrel{n,m}{\to} 0 \ .
\]
Viceversa, se $\mE$ \e tale che ogni serie assolutamente convergente \e convergente, consideriamo una successione di Cauchy $\{ v_n \}$. Per mostrare che questa ammette un limite, scegliamo una sottosuccessione $\{ n_k \} \subset \bN$ tale che per $n,m > n_k$ risulti $\| v_n - v_m \| < 2^{-k}$, e definiamo la sottosuccesione $\{ g_k := v_{n_k} - v_{n_{k-1}} \}$ (avendo posto $v_{n_0} := 0$). Allora
\[
\sum_k \| g_k \| \leq 
\| g_1 \| + \sum_k 2^{-k} = 
\| g_1 \| + 1
\ ,
\]
per cui la serie $\{ \sum_k \| g_k \| \}$ \e monot\'ona e limitata, e quindi convergente. Esiste dunque $v := \sum_k g_k$, ed \e semplice verificare che deve essere anche $v = \lim_n v_n$.
\end{proof}

\begin{prop}
\label{prop_dualBan}
Sia $\mE$ uno spazio normato. Allora il duale $\mE^*$ \e uno spazio di Banach.
\end{prop}

\begin{proof}[Dimostrazione]
Presa una successione di Cauchy $\{ f_n \} \subset \mE^*$, osserviamo che 
\[
| \left \langle f_n , v \right \rangle - \left \langle f_m , v \right \rangle | 
\leq
\| f_n - f_m \| \| v \|
\ \ , \ \
\forall v \in \mE
\ ,
\]
per cui la succesione $\left \langle f_n , v \right \rangle \subset \bR$ (o $\bC$) \e di Cauchy e possiamo definire
$f(v) := \lim_n \left \langle f_n , v \right \rangle$,
$v \in \mE$. Si tratta quindi di verificare che l'applicazione $f$ cos\'i definita \e un funzionale lineare continuo, e che 
$\lim_n \| f-f_n \| = 0$. 
Mostriamo che $f$ \e limitata: poich\'e $\{ f_n \}$ \e di Cauchy, otteniamo per diseguaglianza triangolare che 
$| \| f_n \| - \| f_m \| | \leq \| f_n - f_m \|$,
per cui $\| f_n \|$ \e di Cauchy in $\bR$ ed esiste $M > 0$ tale che $\| f_n \| < M$, $n \in \bN$. Ora, per ogni $\eps > 0$ esiste $m \in \bN$ tale che
$| f(v) - \left \langle f_n , v \right \rangle | < \eps$, $n \geq m$;
per cui, se $v \in \mE_1$ troviamo 
$| f(v) | \ \leq \ \eps + | \left \langle f_n , v \right \rangle |
\ \leq \
\eps + M$.
Che $f$ sia lineare \e evidente per linearit\'a dei limiti, per cui $f \in \mE^*$ e concludiamo la dimostrazione osservando che
\[
\lim_n \| f - f_n \| \ = \
\lim_n \sup_{v \in \mE_1} | \left \langle f - f_n , v \right \rangle |   \ = \
\lim_n \sup_{v \in \mE_1} \lim_m | \left \langle f_m - f_n , v \right \rangle | \ \leq \
\limsup_{m,n} \| f_m - f_n \|
\ \stackrel{m,n}{\to} \
0
\ .
\]
\end{proof}

Sia $\mE$ uno spazio di Banach. Un sottospazio vettoriale $\mE' \subseteq \mE$ si dice {\em di Banach} se esso \e completo. Preso un sottospazio vettoriale $\mE' \subseteq \mE$, {\em la chiusura di $\mE'$ in $\mE$} si definisce come lo spazio vettoriale degli elementi di $\mE$ che sono limite di successioni di Cauchy in $\mE'$, e si denota con $\ovl{\mE'}$; ovviamente $\ovl{\mE'}$ \e -- per costruzione -- un sottospazio di Banach. Preso un insieme $S \subset \mE$, il {\em sottospazio di Banach generato da $S$} si definisce come la chiusura dello spazio vettoriale delle combinazioni lineari di elementi di $S$.

Dato un sottospazio vettoriale $\mE' \subseteq \mE$ l'operazione di restrizione $f \mapsto R_{\mE'}f := f|_{\mE'}$, $f \in \mE^*$, induce un'applicazione lineare $R_{\mE'} : \mE^* \to {\mE'}^*$ e, chiaramente, $\| R_{\mE'}f \| \leq \| f \|$, $\forall f \in \mE^*$. Il teorema di Hahn-Banach, che mostreremo nel seguito, afferma che $R_{\mE'}$ \e suriettiva.

Una {\em base di Schauder} dello spazio di Banach $\mE$ \e una successione $\{ e_i \in E \}$ che soddisfa la seguente propriet\'a: per ogni $v \in \mE$ esiste {\em ed \e unica} la successione reale (o complessa) $\{ a_i \}$ tale che $v = \sum_i a_i e_i$. Non tutti gli spazi di Banach posseggono una base di Schauder (vedi i commenti in \cite[Cap.V]{Bre}).

\begin{ex}
\label{ex_CX_d}
{\it
Sia $X$ uno spazio topologico. Lo spazio $C_b(X)$ delle funzioni continue e limitate su $X$ a valori reali \e uno spazio di Banach, qualora equipaggiato della norma $\| \cdot \|_\infty$. Una famiglia di funzionali limitati su $C_b(X)$ \e data dalle \textbf{delta di Dirac}
\[
\left \langle \delta_x , f \right \rangle := f(x) 
\ \ , \ \
x \in X
\ \Rightarrow \
| \left \langle \delta_x , f \right \rangle | \leq \| f \|_\infty
\ .
\]
}
\end{ex}

\begin{ex}[Gli spazi $l^p$]
\label{ex_lp_spaces}
{\it
Gli spazi $L_\mu^p(X)$ sono degli spazi di Banach per ogni $p \in [1,+\infty]$ (Teorema di Fischer-Riesz). In particolare denotiamo con $l^p$ l'insieme delle successioni $z := \{ z_n \in \bR \}_n$ tali che
\[
\| z \|_p := \left( \sum_n |z_n|^p \right)^{1/p} 
\ < \
+ \infty
\ .
\]
Definiamo inoltre $l^\infty$ come lo spazio delle successioni tali che
\[
\| z \|_\infty := \sup_n |z_n| < + \infty \ .
\]
Ora
{\footnote{Le stesse definizioni possono essere formulate usando il campo complesso, ed in tal caso scriveremo $l^p_\bC$, $p \in [1,\infty]$.}}, 
ogni $z \in l^p$ si pu\'o riguardare come una funzione sullo spazio discreto $\bN$ equipaggiato della misura di enumerazione. Per cui, con le tecniche usate per dimostrare Prop.\ref{prop_holder} e Prop.\ref{prop_mink}, otteniamo delle versioni delle diseguaglianze di Holder e Minkowski,
\begin{equation}
\label{eq_HM_lp}
\sum_n |z_n||w_n|  \leq  \| z \|_p \| w \|_q
\ \ , \ \
\| z + z' \|_p \leq \| z \|_p + \| z' \|_p
\ \ , \ \
\frac{1}{p} + \frac{1}{q} = 1
\ ,
\end{equation}
ed il teorema di Fischer-Riesz implica che ogni $l^p$ \e uno spazio di Banach. 
}
\end{ex}

\begin{ex}[Il Teorema di Riesz-Markov] \label{ex_dual_CX}
{\it
Se $X$ \e uno spazio compatto allora $C(X)$, equipaggiato con la norma dell'estremo superiore, \e uno spazio di Banach. Assumiamo ora che $X$ sia di Hausdorff; per ogni misura di Radon con segno $\mu \in R(X)$ definiamo il funzionale
\[
\left \langle F_\mu , f \right \rangle
:=
\int_X f \ d \mu
\ \ , \ \
f \in C(X)
\ .
\]
E' semplice verificare che $\| F_\mu \| = | \mu |(X)$, cosicch\'e abbiamo un'applicazione lineare isometrica
\begin{equation}
\label{eq_RM}
R(X) \to C(X)^* \ \ , \ \ \mu \mapsto F_\mu \ .
\end{equation}
Il \textbf{Teorema di Riesz-Markov} (\cite[\S 13.4]{Roy}) afferma che (\ref{eq_RM}) \e anche suriettiva, cosicch\'e abbiamo una caratterizzazione delle misure di Radon su $X$ in termini dei funzionali lineari continui su $C(X)$. Per dare un'idea della dimostrazione consideriamo $\varphi \in C(X)^*$ \textbf{positivo}, ovvero tale che
$\left \langle \varphi , f \right \rangle \geq 0$ $\forall f \geq 0$,
ed introduciamo l'applicazione
\[
\mu_* A
\ := \
\sup
\{
\left \langle \varphi , f \right \rangle
\ , \
f \in C(X) \ , \ {\mathrm{supp}}(f) \subset A \ , \ 0 \leq f \leq 1
\}
\ \ , \ \
A \in 2^X
\ .
\]
Si verifica che la "misura interna" $\mu_*$ induce una misura di Radon $\mu$, definita su un'opportuna $\sigma$-algebra $\mM \supseteq \tau X$, e che $F_\mu = \varphi$. La tesi del teorema segue osservando che ogni funzionale su $C(X)$ si decompone in una differenza di funzionali positivi (\cite[Prop.13.24]{Roy}). 
Il teorema si estende senza difficolt\'a al caso complesso, cosicch\'e abbiamo un isomorfismo di spazi di Banach $R(X,\bC) \to C(X,\bC)^*$, dove $R(X,\bC)$ \e lo spazio delle misure di Radon complesse su $X$.
}
\end{ex}

\

\noindent \textbf{Spazi di Hilbert e basi ortonormali.} Uno spazio di Banach $\mH$ si dice di {\em Hilbert} se questo \e equipaggiato di un prodotto scalare 
\[
( u , v ) \in \bR
\ \ , \ \
u,v \in \mH
\ ,
\]
che ne induce la norma, ovvero 
$\| v \|^2 = (v,v)$, $\forall v \in \mH$.
Nel caso complesso il prodotto scalare \e per definizione {\em sesquilineare}, ovvero 
\begin{equation}
\label{eq.sesq}
(\lambda u , \mu v) = \ovl \lambda \mu (u,v)
\ , \
(v,u) = \ovl{(u,v)}
\ \ , \ \
\forall u,v \in \mH
\ , \
\lambda,\mu \in \bC
\end{equation}
(spesso in letteratura si richiede, a differenza di quanto facciamo noi, linearit\'a nella prima variabile ed antilinearit\'a nella seconda, ma chiaramente questa non \e una differenza sostanziale). Il prodotto scalare si pu\'o ricostruire dalla norma grazie all'{\em identit\'a di polarizzazione}
\[
(u,v) = 1/4 \sum_{k=0}^3 i^k \| u+i^kv \|^2
\ \ , \ \
\forall u,v \in \mH
\ ,
\]
la quale si dimostra banalmente usando (\ref{eq.sesq}) e le identit\'a
${\mathrm{Re}}(z) =  1/2 (z+\ovl z)$, 
${\mathrm{Im}}(z) = -i/2 (z-\ovl z)$,
$\forall z \in \bC$.
Una famiglia $\{ e_i \} \subset \mH$ si dice {\em ortogonale} se 
\[
( e_i , e_j ) = 0
\ \ , \ \
\forall i \neq j
\ ,
\]
ed {\em ortonormale} se, inoltre, $\left \langle e_i , e_i \right \rangle = 1$ per ogni $i \in \bN$. Una {\em base hilbertiana di } $\mH$ \e una famiglia ortonormale $\{ e_i \}$ tale che lo spazio vettoriale da essa generato \e denso in $\mH$. 
Si pu\'o dimostrare che se $\mH$ \e separabile allora ha una base {\em numerabile} (ovvero, la famiglia $\{ e_i \}$ \e una successione, vedi \cite[\S 4.16.3]{Kol}). L'esistenza delle basi hilbertiane pu\'o essere dimostrata usando il procedimento di ortogonalizzazione di Gram-Schmidt (ancora, si veda \cite[\S 4.16.3]{Kol}). Chiaramente, una base hilbertiana \e una base di Schauder (per verificarlo, si definiscano gli $a_i$ nella definizione di base di Schauder come i prodotti scalari $\left \langle v , e_i \right \rangle $).

\begin{ex}
\label{ex_base}
{\it
Come vedremo in \S \ref{sec_fourier}, una base per lo spazio di Hilbert reale $L^2([0,1])$ \e quella delle funzioni trigonometriche
\[
\mB := \{ \sin(2 \pi nx) \ , \ \cos(2 \pi mx) \ , \ x \in [0,1]  \}_{n,m \in \bN} \ .
\]
Passando al caso complesso, le stesse argomentazioni del caso reale permettono di concludere che
\[
\mB_\bC := \{ e^{2 \pi i n x} \ , \ x \in [0,1]  \}_{n \in \bZ}
\]
\e una base per $L^2([0,1],\bC)$ (osservare che, avendosi 
\[
\cos(2 \pi mx) = 1/2  ( e^{2 \pi i n x} + e^{- 2 \pi i n x} )
\ \ , \ \
\sin(2 \pi nx) = -i/2 ( e^{2 \pi i n x} - e^{- 2 \pi i n x} )
\ \ , \ \
\forall x \in [0,1]
\ ,
\]
\e possibile ottenere elementi di $\mB$ come combinazioni lineari di elementi di $\mB_\bC$).
}
\end{ex}

\begin{prop}[Bessel, Parseval]
\label{prop_bessel}
Dato un insieme ortonormale $\{ e_i \}$ dello spazio di Hilbert $\mH$, per ogni $u \in \mH$ si ha
\begin{equation}
\label{eq_fou}
\sum_i | (e_i,u) |^2
\leq
\| u \|^2 
\ .
\end{equation}
Se, in particolare, $\{ e_i \}$ \e una base hilbertiana allora
$u = \sum_i (e_i,u) e_i$ e
$\| u \|^2 = \sum_i | (e_i,u) |^2$.
\end{prop}

\begin{proof}[Dimostrazione]
Posto $c_i := (e_i,u)$ troviamo (nel caso reale)
\begin{equation}
\label{eq_bessel}
\begin{array}{ll}
0 
\ \leq \
\| u - \sum_i^n a_i e_i \|^2
& = 
\left( u - \sum_i^n a_i e_i \ , \ u - \sum_i^n a_i e_i \right)
\\ & =
\| u \|^2 - 2 \sum_i a_i c_i + \sum_i^n a_i^2
\\ & =
\| u \|^2 - \sum_i^n c_i^2 + \sum_i^n ( a_i - c_i )^2
\ .
\end{array}
\end{equation}
L'ultima espressione assume il suo minimo proprio per $a_i \equiv c_i$, in concomitanza del quale troviamo $0 \leq \| u \|^2 - \sum_i c_i^2$ ovvero (\ref{eq_fou})
{\footnote{
D'altra parte, nel caso complesso (\ref{eq_bessel}) prende la forma
\[
\| u - \sum_i^n a_i e_i \|^2
=
\| u \|^2 - 2 \sum_i \ {\mathrm{Re}} \{ \ovl a_i c_i \} + \sum_i^n |a_i|^2
=
\| u \|^2 - \sum_i^n |c_i|^2 + \sum_i^n | a_i - c_i |^2
\]
e possiamo argomentare in modo analogo al caso reale.}}.
Quando $\{ e_i \}$ \e una base hilbertiana sappiamo che esiste una serie $\sum_i a_i e_i$ con limite $u$, per cui imponendo in (\ref{eq_bessel}) $\lim_n \| u - \sum_i^n a_i e_i \|^2 = 0$ troviamo, valutando nel minimo $a_i \equiv c_i$, sia
\[
0 \stackrel{n}{\leftarrow} \| u - \sum_i^n a_i e_i \|^2  \geq \| u - \sum_i^n c_i e_i \|^2
\ \ {\mathrm{ovvero}} \ \
u = \sum_i c_i e_i
\ ,
\]
che
$0 = \lim_n \| u - \sum_i^n a_i e_i \|^2 \geq \lim_n \{ \| u \|^2 - \sum_i^nc_i^2 \}$.
\end{proof}

\noindent I prodotti scalari $(e_i,u)$, $i \in \bN$, si dicono i {\em coefficienti di Fourier di $u$}. 

\ 

\noindent \textbf{Diseguaglianza di Cauchy-Schwarz e dualit\'a di Riesz.} Siano $u,v \in \mH$ e $\lambda > 0$; allora
\[
0 \leq \| u-\lambda v \|^2
\ = \
\| u \|^2 - 2 ( u,\lambda v ) + \lambda^2 \| v \|^2
\ \Rightarrow \
2 ( u,v ) 
\ \leq \
\lambda^{-1} \| u \|^2 + \lambda \| v \|^2
\ ,
\]
cosicch\'e per $\lambda = \| u \| \| v \|^{-1}$ otteniamo
$
( u,v ) \leq \| u \| \| v \|$.
Nel caso complesso troviamo, con $\lambda \in \bC$,
\[
0 \leq \| u-\lambda v \|^2 =
\| u \|^2 - 2 {\mathrm{Re}} \{ \lambda (v,u) \} + |\lambda|^2 \| v \|^2
\ ,
\]
cosicch\'e, valutando per $\lambda = (u,v) \| v \|^{-1}$ otteniamo la {\em diseguaglianza Cauchy-Schwarz}
\begin{equation}
\label{eq_CS}
| ( u,v ) | \ \leq \ \| u \| \| v \| 
\ \ , \ \
\forall u,v \in \mH
\ .
\end{equation}

\begin{rem}
{\it
L'argomento della dimostrazione della diseguaglianza di Cauchy-Schwarz vale per ogni forma bilineare (sesquilineare), simmetrica e definita positiva 
$A  :\mE \times \mE \to \bR$ ($\bC$),
nel senso che 
\begin{equation}
\label{eq_CSg}
| A(u,v) | \ \leq \ A(u,u)^{1/2} A(v,v)^{1/2}
\ \ , \ \
\forall u,v \in \mE
\ .
\end{equation}
Per verificare (\ref{eq_CSg}), nei conti che mostrano (\ref{eq_CS}) si sostituisca $\| u \|^2$ con $A(u,u)$.
}
\end{rem}

Ora, (\ref{eq_CS}) implica che per ogni $u \in \mH$ il funzionale
\[
f_u : \mH \to \bR \ (\bC)
\ \ : \ \
\left \langle f_u , v \right \rangle := ( u,v )
\ \ , \ \
v \in \mH
\ ,
\]
\e limitato ed ha norma $\leq \| u \|$; d'altra parte, usando il fatto che
$| \left \langle f_u , v \right \rangle | = \| u \|$, $v := u / \| u \|$,
otteniamo
\begin{equation}
\label{eq_riesz0}
\| f_u \| = \| u \|
\ \ , \ \
u \in \mH
\ .
\end{equation}

\begin{thm}[Riesz]
\label{thm_riesz}
Sia $\mH$ uno spazio di Hilbert. Allora l'applicazione 
\begin{equation}
\label{eq_riesz}
\mH \to \mH^* \ \ , \ \ u \mapsto f_u \ ,
\end{equation}
\'e lineare (antilineare nel caso complesso), isometrica e suriettiva.
\end{thm}

\begin{proof}[Dimostrazione (caso separabile)]
L'applicazione (\ref{eq_riesz}) \e chiaramente lineare e (\ref{eq_riesz0}) implica che essa \e isometrica e quindi iniettiva; per cui resta da verificarne soltanto la suriettivit\'a. Sia $\{ e_k \}$ una base ortonormale per $\mH$. Poniamo $a_k :=$ $\left \langle f , e_k \right \rangle$ ed osserviamo che 
\[
\sum_{k=1}^n a_k^2
=
\left \langle f , \sum_{k=1}^n a_k e_k \right \rangle
\leq 
\| f \| \| \sum_{k=1}^n a_k e_k \|
= 
\| f \| \left( \sum_{k=1}^n a_k^2 \right)^{1/2}
\ ,
\]
per cui
\begin{equation}
\label{eq_riesz1}
\left( \sum_{k=1}^n a_k^2 \right)^{1/2} 
\leq 
\| f \|
\ , \
n \in \bN
\ \Rightarrow \
\sum_n a_n^2 < + \infty
\ ,
\end{equation}
cosicch\'e esiste $u := \lim_n \sum_k^n a_k e_k \in \mH$. Ora,
\[
( u , v ) = 
\sum_k a_k b_k = 
\sum_k \left \langle f , e_k \right \rangle b_k =
\left \langle f , v \right \rangle
\ \ , \ \
\forall v := \sum_k b_k e_k
\ ,
\]
per cui $f = f_u$ ed il teorema \e dimostrato.
\end{proof}

\begin{rem}{\it 
\textbf{(1)} Dal teorema precedente segue immediatamente che uno spazio di Hilbert \e riflessivo nel senso di Def.\ref{def_rifl}; \textbf{(2)} Il teorema di Riesz \e valido anche nel caso non separabile: la dimostrazione si basa su un accorto uso delle proiezioni ortogonali sui sottospazi di $\mH$ (vedi \cite[\S 3.1.6 - 3.1.9]{Ped}).
} \end{rem}

\begin{ex}
{\it
$L_\mu^2(X)$ \e uno spazio di Hilbert rispetto al prodotto scalare $(f,g) := \int fg$. Nel caso complesso abbiamo invece il prodotto scalare
\[
(f,g) := \int \ovl f g
\ \ , \ \
f,g \in L_\mu^2(X,\bC)
\ .
\]
}
\end{ex}

\subsection{Operatori limitati e $\bf{C^*}$-algebre.}
\label{sec.oper}

Siano $\mE,\mF$ spazi normati. Si dice {\em operatore limitato}, o {\em continuo}, un'applicazione lineare $T : \mE \to \mF$ tale che, per qualche $c \in \bR^+$ fissato, $\| Tv \| \leq c \| v \|$, $\forall v \in \mE$. Definendo
\begin{equation}
\label{eq_defnorm}
\| T \| 
\ :=  \
\inf \{ 
c \in \bR : \| Tv \| \leq c \| v \| \ \forall v \in \mE  
\}
\ = \
\sup \{ \| Tv \| \ , \ \| v \| = 1  \}
\end{equation}
e la struttura di spazio vettoriale
\[
(T + \lambda T') v := Tv + \lambda Tv' 
\ \ , \ \
T,T' \in B(\mE,\mF)
\ \ , \ \
v,v' \in \mE
\ , \ 
\lambda \in \bR
\ ,
\]
otteniamo che l'insieme $B(\mE,\mF)$ degli operatori limitati da $\mE$ in $\mF$ \e uno spazio normato.

\begin{prop}
Se $\mE$ \e uno spazio normato ed $\mF$ uno spazio di Banach allora $B(\mE,\mF)$ \e uno spazio di Banach.
\end{prop}

\begin{proof}[Dimostrazione]
Se $\{ T_n \} \subseteq B(\mE,\mF)$ \e una successione di Cauchy allora per ogni $v \in \mE$ risulta
\[
\| T_n v - T_mv \| \leq \| T_n - T_m \| \| v \| \stackrel{n,n}{\to} 0
\ ;
\]
per cui, essendo $\mF$ completo, possiamo definire $Tv := \lim_n T_n v \in F$. L'applicazione $T : \mE \to \mF$ cos\'i definita \e chiaramente lineare, e per ogni $v \in \mE_1$, $\eps > 0$ ed $n > n_\eps$ troviamo
\[
\| Tv \| \leq \| Tv - T_nv \| + \| T_n v \| < 
\eps + \sup_n \| T_n \| \| v \| <
\eps + \sup_n \| T_n \|
\ ;
\]
poich\'e $\{ T_n \}$ \e di Cauchy abbiamo che $c := \sup_n \| T_n \| < \infty$, per cui $\| Tv \| \leq \eps + c$ e $T$ \e limitato. Infine  abbiamo
\[
\| Tv - T_nv \|   
\ = \
\lim_m \| T_mv - T_nv \| 
\ \leq \
\limsup_{m,n} \| T_m - T_n \| \| v \|
\ \stackrel{m,n}{\to} \
0
\ ,
\]
per cui $T$ \e limite di $\{ T_n \}$ e concludiamo che $B(\mE,\mF)$ \e uno spazio di Banach. 
\end{proof}

Se $\mF'$ \e uno spazio normato allora ha senso effettuare la composizione di operatori $T' \circ T : \mE \to \mF'$ (spesso scriveremo pi\'u brevemente $T'T$). Visto che
$\| T'Tv \| \leq \| T' \| \| Tv \| \leq \| T' \| \| T \| \|v \|$, $v \in \mE$, concludiamo che
\begin{equation}
\label{eq_op.norm}
\| T'T \| \leq \| T' \| \| T \|
\ \ , \ \
T \in B(\mE,\mF) \ , \ T' \in B(\mF,\mF')
\ .
\end{equation}
In particolare, per $\mE=\mF$ usiamo la notazione $B(\mE) := B(\mE,\mE)$.

\begin{ex}
\label{ex_Ta}
{\it
Sia $a := \{ a_n \} \in l^\infty_\bC$. Allora, per ogni $p \in [1,+\infty]$, l'operatore
\[
T_a \in B(l^p_\bC)
\ \ , \ \
T_a z := \{ a_n z_n \}_n 
\ \ , \ \
z := \{ z_n \} \in l^p_\bC
\ ,
\]
\e limitato ed ha norma $\| T_a \|  =  \| a \|_\infty$.
}
\end{ex}

Preso $T \in B(\mE,\mF)$, consideriamo l'immagine
$T(\mE) := \{ Tu  : u \in \mE \} \subseteq \mF$
e diciamo che $T$ \e {\em chiuso} se $T(\mE)$ \e un sottospazio chiuso di $\mE$. Il {\em nucleo} di $T$ si definisce come il sottospazio
$\ker T := \{ u \in \mE : Tu=0 \} \subseteq \mE$,
il quale \e chiuso in quanto se $v = \lim_n v_n$, $\{ v_n \} \subset \ker T$, allora
\[
\| Tv \| \leq \| T \| \| v-v_n \| + \| Tv_n \| = \| T \| \| v-v_n \| \to 0 \ .
\]

\begin{ex}
\label{ex_int}
{\it
Consideriamo gli spazi di Banach $L^1([0,1])$, $\mF := C([0,1])$, equipaggiati rispettivamente con le norme $\| \cdot \|_1$ ed $\| \cdot \|_\infty$. Allora la somma diretta $\mE := L^1([0,1]) \oplus \bR$ \e uno spazio di Banach rispetto alla norma $\| f \oplus \lambda \| := \sup \{ \| f \|_1 , |\lambda| \}$, e l'applicazione
\[
T : \mE \to \mF
\ \ , \ \
\{ T(f \oplus \lambda) \}(x) := \lambda + \int_0^x f(t) \ dt
\ \ , \ \
\forall x \in [0,1]
\ ,
\]
\e un operatore limitato tale che
\[
\| T(f \oplus \lambda) \|_\infty 
\ \leq \
| \lambda | + \| f \|_1 \leq 2 \| f \oplus \lambda \| 
\ .
\]
L'immagine di $T$ coincide con lo spazio $AC([0,1])$ delle funzioni assolutamente continue, il quale \e strettamente contenuto e denso in $\mF$. Dunque, $T$ non \e chiuso.
}
\end{ex}

\

\noindent \textbf{Operatori aggiunti.} Un'importante operazione \e quella che associa a $T \in B(\mE,\mF)$ l'{\em operatore aggiunto}
\[
T^* : \mF^* \to \mE^* 
\ , \
f \mapsto T^* f 
\ \ : \ \
\left \langle T^* f , v \right \rangle 
\ := \
\left \langle f , Tv \right \rangle 
\ \ , \ \
\forall v \in \mE
\ .
\]
L'applicazione $T \mapsto T^*$ \e isometrica, e ci\'o si dimostra tramite il Teorema di Hahn-Banach che dimostreremo nel seguito (Teo.\ref{thm_HB}):

\begin{prop}
Per ogni $T \in B(\mE,\mF)$, si ha $\| T^* \| = \| T \|$.
\end{prop}

\begin{proof}[Dimostrazione]
Per ogni $v \in \mE$ ed $f \in \mE^*$ abbiamo la stima
$| \left \langle T^*f , v \right \rangle | =
 | \left \langle f , Tv   \right \rangle | \leq 
 \| f \| \| T \| \| v \|$,
per cui, passando al $\sup$ per $\| v \| \leq 1$ prima e $\| f \| \leq 1$ poi, otteniamo 
$\| T^* \| \leq \| T \|$. 
Per dimostrare la diseguaglianza opposta, prendiamo $x \in \mE$, poniamo 
$y := \| Tx \|^{-1}Tx \in \mF_1$, 
e definiamo i funzionali $g : \bR y \to \bR$ (o $g : \bC y \to \bC$ nel caso complesso),
$\left \langle g , \lambda y \right \rangle := \lambda$, $\lambda \in \bR$ (o $\bC$). 
Per costruzione, i nostri funzionali hanno norma 1: $\| g \| = 1$.
Grazie a Teo.\ref{thm_HB}, per ogni $g$ esiste $\wt g \in \mF^*$ tale che $\| \wt g \| = \| g \| = 1$ e $\wt g|_{\bR y} = g$ (o $\wt g|_{\bC y} = g$). Inoltre, avendosi
$1 = \left \langle \wt g     , y  \right \rangle$,
troviamo
$\| Tx \| = 
  | \left \langle \wt g    , Tx \right \rangle | = 
  | \left \langle T^*\wt g , x \right  \rangle | \leq  
 \|T^* \wt g \| \| x \| \leq
 \| T^* \| \| x \|$.
Per cui $\| T \| \leq \| T^* \|$, il che conclude la dimostrazione.
\end{proof}

Quando abbiamo a che fare con un spazio di Hilbert $\mH$ possiamo adattare la definizione di operatore aggiunto usando la dualit\'a di Riesz: 
\[
T^* \in B(\mH) 
\ \ : \ \
(T^*u,v) := (u,Tv)
\ \ , \ \
u,v \in \mH
\ , \
T \in B(\mH)
\ ;
\]
dunque abbiamo un'applicazione isometrica $* : B(\mH) \to B(\mH)$, $T \mapsto T^*$, lineare quando $\mH$ \e reale ed antilineare 
{\footnote{
Con il termine {\em antilineare} intendiamo che 
$(\lambda T_1 + T_2)^* = \ovl \lambda T_1^* + T_2^*$,
$\lambda \in \bC$,
$T_1 , T_2 \in B(\mH)$.
}}
quando $\mH$ \e complesso. 
Un operatore $T \in B(\mH)$ si dice {\em autoaggiunto} se $T=T^*$. Diciamo invece che $U \in B(\mH)$ \e {\em unitario} qualora $U^*U = UU^* = 1$, e lasciamo come esercizio dimostrare che $U$ \e unitario se e solo se esso \e suriettivo ed isometrico. Infine, un {\em proiettore} $P \in B(\mH)$ \e un operatore autoaggiunto ed idempotente, $P = P^* = P^2$.

\begin{ex}[L'operatore di Volterra]
\label{ex_volterra}
{\it
Osservando che $\mH := L^2([0,1],\bC) \subset L^1([0,1],\bC)$, defi- niamo l'operatore
\[
F \in B(\mH)
\ \ , \ \
Fu(x) := i \int_0^x u(t) \ dt
\ , \ 
\forall x \in [0,1]
\ , \
u \in \mH
\ .
\]
Per calcolarne l'aggiunto, osserviamo che
\[
(v,Fu) 
\ = \ 
i \int_0^1 \int_0^1 \ovl{v(x)} \chi_{[0,x]}(t) u(t) \ dtdx
\ = \
(F^*v,u)
\ \ , \ \
\forall u,v \in \mH
\ ,
\]
cosicch\'e
\[
F^*v(t) \ = \ 
-i \int_0^1 v(x) \chi_{[0,x]}(t) \ dx  \ = \
-i \int_t^1 v(x) \ dx
\ \ , \ \
\forall t \in [0,1]
\ , \
v \in \mH
\ .
\]
Ad uso futuro, osserviamo che
\begin{equation}
\label{eq_ff}
\{ F-F^* \}f(x) 
\ =  \
i \left\{ \int_0^x f(t) \ dt + \int_x^1 f(t) \ dt   \right\}
\ = \
Ff(1)
\ , \
\forall x \in [0,1]
\ , \
f \in \mH
\ .
\end{equation}
}
\end{ex}

\begin{ex}[Operatori di traslazione]
\label{ex.trans}
{\it
Consideriamo lo spazio di Hilbert $\mH := L^2(\bR,\bC)$ e, preso $t \in \bR$, definiamo l'operatore
\[
U_t : \mH \to \mH
\ \ : \ \
U_tf(x) := f(x+t)
\ , \
\forall x \in \bR
\ , \
f \in \mH
\ .
\]
Chiaramente $\int |f|^2 = \int |U_tf|^2$ per ogni $f \in \mH$, per cui $U_t$ \e isometrico e quindi limitato. Inoltre $U_t$ ha, ovviamente, inverso $U_t^{-1} = U_{-t}$, e
\[
(f,U_tg) = \int f(s)g(s+t) \ ds = \int f(s'-t)g(s') \ ds' = (U_{-t}f,g)
\ \ , \ \
f,g \in \mH
\ .
\]
Comcludiamo che $U_t$ \e unitario, con $U_t^* = U_{-t}$.
}
\end{ex}

\begin{rem}[Alcune propriet\'a elementari dei proiettori]
\label{rem_proj}
{\it
Ogni proiettore $P \in B(\mH)$ definisce il sottospazio $\mH_P := \{ v \in \mH : Pv = v \}$, il quale \e di Hilbert (ovvero, \e chiuso nella topologia della norma). Viceversa, ogni sottospazio di Hilbert $\mH' \subseteq \mH$ con base hilbertiana $\{ f_k \}$ definisce l'operatore 
\[
Pv := \sum_k (f_k,v) f_k
\ \ , \ \
v \in \mH
\ ,
\]
il quale \e un proiettore (lasciamo la verifica di questi fatti come esercizio).
Essendo $\{ f_k \}$ un insieme ortonormale per ogni proiettore $P \in B(\mH)$ si ha, usando (\ref{eq_fou}),
\[
\| Pv \|^2 
\ = \ 
\sum_k | (f_h,v) |^2
\ \leq \
\| v \|^2
\ \Rightarrow \
\| P \| \leq 1
\ ;
\]
pi\'u precisamente abbiamo $\| P \| = 1$, visto che 
$\| P f_k \| = \| f_k \| = 1$, $\forall k \in \bN$.
Infine, si verifica facilmente che 
\begin{equation}
\label{eq.rem_proj}
\mH_P \subseteq H_{P'}       \ \Leftrightarrow \ P'P = PP' = P
\ \ , \ \
\mH_P \cap H_{P'} = \{ 0 \}  \ \Leftrightarrow \ P'P = PP' = 0
\ .
\end{equation}
Quando $PP' = P'P = 0$ \e ovvio che $P + P'$ \e un proiettore, ed il relativo sottospazio \e in effetti la chiusura del sottospazio 
\begin{equation}
\label{eq.rem_proj2}
\mH_P + H_{P'}
:=
\{ u+v : u \in \mH_P , v \in \mH_{P'}  \}
\ .
\end{equation}
}
\end{rem}

\noindent \textbf{Complementi ortogonali, nuclei e immagini.} Preso un sottospazio $\mV \subset \mH$ definiamo il {\em complemento ortogonale}
\begin{equation}
\label{def_perp}
\mV^\perp := \{ v \in \mH : (v,u) = 0 , \forall u \in \mV \} \ .
\end{equation}
La disuguaglianza di Cauchy-Schwarz implica che se $v = \lim_n v_n$, $\{ v_n \} \subset \mV$, allora
\[
|(v,u)| \leq \| v-v_n \|\| u \| + |(v_n,u)| = \| v-v_n \|\| u \| \to 0 \ ,
\]
per cui $v \in \mV^\perp$ e $\mV^\perp$ \e chiuso. D'altro canto, un ragionamento analogo mostra che $\mV^\perp = (\ovl \mV)^\perp$, dove $\ovl \mV$ \e la chiusura di $\mV$.

Ora, denotiamo con $P$ il proiettore su $\ovl \mV$, cosicch\'e $w=Pw$, $\forall w \in \ovl \mV$; preso $u \in \mH$ definiamo $v := Pu$ e, posto $v' := u - v$, troviamo
\[
(v',w) = (u-Pu,w) = (u,w) - (Pu,w) = (u,Pw)-(Pu,w) = (u,Pw)-(u,Pw) = 0
\ \ , \ \
\forall w \in \ovl \mV
\ ,
\]
cosicch\'e $v' \in \mV^\perp$. Inoltre \e ovvio che $\ovl \mV \cap \mV^\perp = \{ 0 \}$, dunque se $u = z+z'$, $z \in \ovl \mV$, $z' \in \mV^\perp$ allora
\[
u = v+v' = z+z'
\ \Leftrightarrow \
\ovl \mV \ni v-z = z'-v' \in \mV^\perp
\ \Leftrightarrow \
0 = v-z = z'-v'
\ \Leftrightarrow \
v=z \ , \ v'=z'
\ .
\]
In conclusione, abbiamo la {\em decomposizione ortogonale}
\begin{equation}
\label{eq_ort}
\mH = \ovl \mV \oplus \mV^\perp \ ,
\end{equation}
il che significa che ogni $u \in \mH$ si scrive {\em in modo unico} come $u = v+v'$, $v \in \ovl \mV$, $v' \in \mV^\perp$, e $\| u \|^2 = \| v \|^2 + \| v' \|^2$.

\begin{lem}
\label{lem_pp}
Per ogni sottospazio $\mV \subset \mH$, si ha $\mV^{\perp \perp} = \ovl \mV$.
\end{lem}

\begin{proof}[Dimostrazione]
Applicando (\ref{eq_ort}) a $\mV^\perp$ otteniamo, essendo $\mV^\perp$ chiuso,
\begin{equation}
\label{eq_orta}
\mH = \mV^\perp \oplus \mV^{\perp \perp} \ ,
\end{equation}
per cui, per l'unicit\'a della decomposizione ortogonale, confrontando (\ref{eq_orta}) con (\ref{eq_ort}) concludiamo che $\mV^{\perp \perp} = \ovl \mV$.
\end{proof}

Sia ora $T \in B(\mH)$. Dall'identit\'a 
$(T^*u,v) = (u,Tv)$, $\forall u,v \in \mH$,
e dal Lemma precedente, segue immediatamente che
\begin{equation}
\label{eq_kerim}
\{ T(\mH) \}^\perp = \ker T^* 
\ \Leftrightarrow \
\ovl{T(\mH)} = \{ \ker T^* \}^\perp
\ \ , \ \
\forall T \in B(\mH)
\ ;
\end{equation}
in particolare, quando $T$ \e chiuso otteniamo $T(\mH) = \{ \ker T^* \}^\perp$.

\

\noindent\textbf{Operatori e forme bilineari.} Un'applicazione bilineare $A : \mH \times \mH \to \bR$ si dice \textbf{limitata} se esiste $c > 0$ tale che $\| A (u,v) \| \leq c \|u \| \| v \|$, $u,v \in \mH$, e \textbf{simmetrica} se $A(u,v) = A(v,u)$, $u,v \in \mH$.

Chiaramente, $A$ \e limitata se e solo se essa \e continua nella topologia prodotto su $H \times H$ indotta dalla norma.

Ora, fissato $v \in \mH$ abbiamo che l'applicazione $f_{A,v}(u) := A(u,v)$, $u \in \mH$, \e in effetti un funzionale lineare limitato tale che $\| f_{A,v} \| \leq c \| v \|$, per cui il teorema di Riesz implica che esiste ed \e unico $T_Av \in \mH$ tale che 
\[
\left \langle f_{A,v} , u \right \rangle = A(u,v) = ( u,T_Av ) \ .
\]
Per bilinearit\'a di $A$ troviamo che $T_A : \mH \to H$ \e un'aplicazione lineare, mentre la limitatezza di $A$ implica che $\| T_A \| \leq c$. In particolare, $A$ \e simmetrica se e solo se $T_A = T_A^*$. Viceversa, ogni operatore limitato $T \in B(\mH)$ definisce la forma bilineare limitata $A_T(u,v) := (u,Tv)$, $u,v \in \mH$.
Analogo risultato vale nel caso complesso considerando forme {\em sesquilineari}, ovvero forme antilineari nella prima variabile e lineari nella seconda.

\

\noindent \textbf{$\bf{C^*}$-algebre.} Le seguenti nozioni formalizzano in termini intrinseci alcune propriet\'a fondamentali degli operatori limitati.
\begin{defn}
\label{def_alg}
\textbf{(1)} Un' \textbf{algebra di Banach} (reale o complessa) \e un'algebra $\mA$ che soddisfa le seguenti propriet\'a: 
(i) $\mA$ \e uno spazio di Banach; 
(ii) per ogni $a,a' \in \mA$ risulta $\| aa' \| \leq \| a \| \| a' \|$.
\textbf{(2)} Diciamo che $\mA$ ha un'\textbf{identit\'a} se esiste $1 \in \mA$ tale che $1a = a1 = a$, $a \in \mA$, e che $\mA$ \e \textbf{commutativa} se $aa' = a'a$ per ogni $a,a' \in \mA$. 
\textbf{(3)} Un'algebra di Banach complessa $\mA$ \e una \textbf{*-algebra di Banach} se esiste un'applicazione 
$* : \mA \to \mA$, $a \mapsto a^*$
antilineare, isometrica ed idempotente (ovvero $a^{**} = a$ per ogni $a \in \mA$). 
\textbf{(4)}  Una *-algebra di Banach si dice \textbf{C*-algebra} se \e verificata \textbf{l'identit\'a C*}
\[
\| a^* a \| = \| a \|^2 
\ \ , \ \
a \in \mA
\ .
\]
\end{defn}

\begin{ex}{\it
Sia $\mE$ uno spazio di Banach. Allora $B(\mE)$ \e un'algebra di Banach, reale o complessa qualora $\mE$ sia reale o complesso (vedi (\ref{eq_op.norm})). 
}
\end{ex}

\begin{ex}{\it
Sia $\mH$ uno spazio di Hilbert complesso. Preso $T \in B(\mH)$ allora (\ref{eq_op.norm}) implica $\| T^*T \| \leq \| T \|^2$, ma d'altra parte
\begin{equation}
\label{eq_idc}
\| Tu \|^2 =
( Tu,Tu ) =
( u,T^*Tu) \leq
\| T^*T \| \| u \|^2
\ \ , \ \
\forall u \in \mH
\ ,
\end{equation}
per cui concludiamo che $\| T^*T \| = \| T \|^2$ e quindi $B(\mH)$ \e una \sC algebra. In realt\'a si dimostra che per ogni \sC algebra $\mA$ esiste un opportuno spazio di Hilbert complesso $\mH$ con un'inclusione $\mA \subseteq B(\mH)$; questo risultato \e noto come la costruzione di Gel'fand-Naimark-Segal ("GNS", si veda il seguente Teorema \ref{lem_GNS} e \cite[Ex.4.3.16-18]{Ped}).
}
\end{ex}

\begin{ex}
{\it
Sia $X$ compatto. Allora $C(X)$, equipaggiata con l'usuale moltiplicazione e norma delle'estremo superiore, \e un'algebra di Banach commutativa con identit\'a la funzione costante $1$. Se $X$ \e localmente compatto, allora $C_0(X)$ \e un'algebra di Banach priva di identit\'a. Passando al caso complesso, definendo
$f^*(x) := \ovl{f(x)}$,
$\forall x \in X$,
$f \in C(X,\bC)$, 
abbiamo che $C(X,\bC)$ (per $X$ compatto) e $C_0(X,\bC)$ (per $X$ localmente compatto) sono \sC algebre.
}
\end{ex}

\begin{ex}
\label{ex_conv}
{\it
$L^1(\bR)$, equipaggiato con il prodotto di convoluzione e la norma $\| \cdot \|_1$, \e un'algebra di Banach commutativa (vedi Teo.\ref{thm_conv} e successive osservazioni). Analogamente $L^1(\bR,\bC)$ \e un'algebra di Banach, la quale \e di interesse nell'ambito della trasformata di Fourier. Definiamo ora, per ogni $f \in L^1(\bR,\bC)$, 
$f^*(t) := \ovl{f(-t)}$, $t \in \bR$;
\'e immediato verificare che l'applicazione $f \mapsto f^*$ \e antilineare, isometrica ed idempotente, cosicch\'e $L^1(\bR,\bC)$ \e una $*$-algebra di Banach.
}
\end{ex}

\begin{ex} \label{ex_LiC}
{\it
Sia $( X ,\mM,\mu )$ uno spazio misurabile. Allora $L_\mu^\infty(X,\bC)$ (equipaggiata con le stesse operazioni di $C(X,\bC)$) \e una \sC algebra commutativa, con identit\'a la funzione costante $1$.
}
\end{ex}

Sia $\mA$ una \sC algebra. Un funzionale lineare $\omega \in \mA^*$ si dice {\em positivo} se 
$\left \langle \omega , a^*a \right \rangle \in \bR^+$ per ogni $a \in \mA$;
nel seguito denoteremo con $\mA^*_+$ l'insieme dei funzionali positivi. E' possibile dimostrare che ogni $a = a^* \in \mA$ si scrive
\[
a = a_1^* a_1 - a_2^* a_2
\ \ , \ \
a_1 , a_2 \in \mA
\]
(vedi \cite[Chap.1]{Ped2}), cosicch\'e se $\omega$ \e positivo allora 
$\left \langle \omega , a \right \rangle \in \bR$ per ogni $a = a^*$. 
Ora, per un generico $a \in \mA$ possiamo scrivere
\[
a = a_+ + ia_-
\ \ , \ \
a_+ := 1/2(a+a^*)
\ , \
a_- := - i/2(a-a^*)
\]
e, ovviamente,
\[
a_+^* = a_+ 
\ \ , \ \
a_-^* = a_-
\ \ , \ \
a^* = a_+^* - i a_-^*
\ .
\]
Quindi abbiamo
$\left \langle \omega , a_+ \right \rangle , 
 \left \langle \omega , a_- \right \rangle  \in \bR$ 
e
\[
\left \langle \omega , a^* \right \rangle = 
\left \langle \omega , a_+^* \right \rangle 
 - i \left \langle  \omega , a_-^* \right \rangle =
\left \langle \omega , a_+ \right \rangle 
 - i \left \langle \omega , a_- \right \rangle =
\ovl{ \left \langle \omega , a_+ \right \rangle} 
 -i \ovl{ \left \langle \omega , a_- \right \rangle} =
\ovl{ \left \langle \omega , a \right \rangle}
\ ,
\]
da cui l'utile propriet\'a
\begin{equation}
\label{eq_omSA}
\left \langle \omega , a^* \right \rangle = \ovl{ \left \langle \omega , a \right \rangle}
\ \ , \ \
\forall a \in \mA
\ , \
\omega \in \mA^*_+
\ .
\end{equation}
Inoltre, applicando la diseguaglianza di Cauchy-Schwarz (\ref{eq_CSg}) con $A(b,a) := \left \langle \omega , b^*a \right \rangle$ troviamo
\begin{equation}
\label{eq_CSo}
| \left \langle \omega , b^*a \right \rangle|^2
\ \leq \
\left \langle \omega , b^*b \right \rangle
\left \langle \omega , a^*a \right \rangle
\ \ , \ \
\forall a,b \in \mA
\ .
\end{equation}
L'abbondanza di funzionali positivi pu\'o essere dimostrata usando la teoria spettrale ed il Teorema di Hahn-Banach (vedi \cite[Ex.4.3.12-15]{Ped}). Il risultato seguente mostra come usando i funzionali positivi sia possibile interpretare una \sC algebra in termini di operatori su uno spazio di Hilbert.

\begin{thm}[Gel'fand-Naimark]
\label{lem_GNS}
Sia $\mA$ una \sC algebra con identit\'a $1 \in \mA$ ed $\omega \in \mA^*_+$. Allora esistono uno spazio di Hilbert complesso $\mH_\omega$, una rappresentazione 
{\footnote{Con il termine {\em rappresentazione} intendiamo che $\pi_\omega$ \e un operatore lineare limitato tale che $\pi_\omega(a^*) = \pi_\omega(a)^*$, $\pi_\omega(aa') = \pi_\omega(a)\pi_\omega(a')$, $\forall a,a' \in \mA$.}}
$\pi_\omega : \mA \to B(\mH_\omega)$
ed $u_\omega \in \mH_\omega$ tali che
\[
\left \langle \omega , a \right \rangle = (u_\omega,\pi_\omega(a)u_\omega)
\ \ , \ \
\forall a \in \mA \ .
\]
\end{thm}

\begin{proof}[Dimostrazione]
Consideriamo l'insieme 
$I_\omega := \{ z \in \mA : \left \langle \omega , z^*z \right \rangle = 0 \}$.
Poich\'e $\omega$ \e continuo per ogni successione convergente $z_n \to z$ troviamo $z^*z = \lim_n z_n^*z_n$ e quindi 
$0 = 
 \left \langle \omega , z^*z \right \rangle = 
 \lim_n \left \langle \omega , z_n^*z_n \right \rangle$, 
dunque $I_\omega$ \e chiuso. Grazie a (\ref{eq_CSo}), per ogni $a \in \mA$ e $z \in I_\omega$ troviamo
\[
| \left \langle \omega , az \right \rangle |^2
\ \leq \
\left \langle \omega , aa^* \right \rangle
\left \langle \omega , z^*z \right \rangle
\ = \
0
\ ,
\]
dunque $az \in I_\omega$; usando tale propriet\'a concludiamo che, presi $z' \in I_\omega$, $\lambda \in \bC$,
\[
\left \langle \omega , (z+ \lambda z')^*(z+ \lambda z') \right \rangle = 0
\ ,
\]
cosicch\'e $I_\omega$ \e anche un sottospazio chiuso e quindi un ideale sinistro di $\mA$. Consideriamo ora lo spazio quoziente 
\[
V := 
\mA / I_\omega := 
\{ [v] := \{ v + z : z \in I_\omega  \} : v \in \mA  \}
\ ;
\]
su $V$ definiamo la forma sesquilineare
\begin{equation}
\label{eq_GNS1}
([v],[v']) := \left \langle \omega , v^*v' \right \rangle
\ \ , \ \
v,v' \in \mA
\ ,
\end{equation}
la quale \e ben posta in quanto se $v = v_0 + z$, $v' = v'_0 + z'$, $z,z' \in I_\omega$, allora, usando il fatto che $I_\omega$ \e un ideale sinistro e (\ref{eq_omSA}),
\[
\left \langle \omega , v^*v'     \right \rangle = 
\left \langle \omega , (v_0^* + z^*)(v'_0+z') \right \rangle = 
\left \langle \omega , v_0^*v'_0      \right \rangle +
\left \langle \omega , v_0^*z'        \right \rangle +
\ovl{\left \langle \omega , {v'_0}^*z \right \rangle} +
\left \langle \omega , z^*z'          \right \rangle =
\left \langle \omega , v_0^*v'_0 \right \rangle
\ .
\]
D'altro canto (\ref{eq_GNS1}) \e anche un prodotto scalare, in quanto
$([v],[v]) = 0$
se e solo se $v \in I_\omega$, ovvero $[v] = 0$. Definiamo $\mH_\omega$ come lo spazio di Hilbert ottenuto completando $V$ rispetto a (\ref{eq_GNS1}). Per ogni $a \in \mA$ definiamo l'applicazione
\[
\pi_\omega (a) : \mH_\omega \to \mH_\omega
\ \ , \ \
\{ \pi_\omega (a) \}[v] := [av]
\ ,
\]
la quale \e ben posta in quanto $[av] = [a(v+z)]$ per ogni $z \in I_\omega$. E' ovvio che $\pi_\omega(a)$ \e lineare, e 
\[
\| \{ \pi_\omega(a) \}[v] \|^2 = 
( [av],[av] ) =
\left \langle \omega , v^*a^*av \right \rangle \leq
\| \omega \| \| v^* \| \| a^*a \| \| v \| =
\| \omega \| \| a \|^2 \| v \|^2 \ ,
\]
per cui $\pi_\omega(a) \in B(\mH_\omega)$. E' altrettanto ovvio che 
$\pi_\omega(a+a') = \pi_\omega(a)+\pi_\omega(a')$,
$\pi_\omega(aa') = \pi_\omega(a) \pi_\omega(a')$,
$\forall a,a' \in \mA$,
nonch\'e
\[
([v],[av']) = 
\left \langle \omega , v^*av' \right \rangle =
\left \langle \omega , (a^*v)^*v' \right \rangle =
([a^*v],[v'])
\ ,
\]
per cui $\pi_\omega(a^*) = \pi_\omega(a)^*$. Ci\'o mostra che $\pi_\omega$ \e una rappresentazione.
Infine, ponendo $u_\omega := [1] \in \mH_\omega$ troviamo
\[
(u_\omega,\pi_\omega(a)u_\omega) = 
([1],[a1]) =
\left \langle \omega , a \right \rangle 
\ ,
\]
e ci\'o conclude la dimostrazione.
\end{proof}

Il risultato precedente \e un passo fondamentale della gi\'a menzionata costruzione GNS, la quale permette di concludere che ogni \sC algebra ammette una rappresentazione {\em iniettiva}. Questa propriet\'a pu\'o essere dedotta direttamente dal teorema precedente solo nel caso in cui $\omega$ sia {\em fedele} (ovvero 
$\left \langle \omega , a^*a \right \rangle = 0$
implica $a = 0$): infatti, se $a \neq 0$ allora
\[
\| \{ \pi_\omega (a) \}u_\omega \| = 
\| [a1] \| =
\| [a] \| = 
\left \langle \omega , a^*a \right \rangle^{1/2}
\neq 0
\ ,
\]
per cui $\pi_\omega(a)$ non pu\'o essere l'operatore nullo. In generale, per ottenere la rappresentazione iniettiva desiderata occorre passare ad un'opportuna somma diretta di spazi di Hilbert del tipo $\mH_\omega$ (vedi \cite[Ex.4.3.17]{Ped}). Quando $\mA$ \e commutativa il teorema precedente si pu\'o interpretare in termini di misure di Radon (vedi Esempio \ref{ex_dual_CX} e l'Esercizio \ref{sec_afunct}.8).

\subsection{Uniforme limitatezza ed applicazioni aperte.}

\begin{lem}[Lemma di Baire]
Sia $X$ uno spazio metrico completo ed $\left\{ X_n \right\}$ una successione di chiusi tali che $\dot{X}_n = \emptyset$, $n \in \bN$ (ovvero, ogni $X_n$ \e  \textbf{rado}). Allora $[ \cup_n X_n ]^\cdot =$ $\emptyset$ (L'unione \textbf{numerabile} di insiemi radi \e un insieme rado).
\end{lem}

\begin{proof}[Dimostrazione]
L'affermazione da dimostrare equivale a verificare che presa la successione di aperti $A_n :=$ $X - X_n$, $n \in \bN$, con $\ovl A_n =$ $X$, risulta $\ovl{\cap_n A_n} = X$. Consideriamo allora un aperto $U \subset X$. Il nostro compito \e dimostrare che $U \cap \ovl{(\cap_n A_n)}$ \e non vuoto. A tale scopo, scegliamo $x_0 \in U$, $r_0 > 0$ tali che $\Delta ( x_0 , r_0 ) \subset U$. Consideriamo quindi, induttivamente 
\[
( x_{n+1} , r_{n+1} ) 
\ : \
x_{n+1} \in \Delta ( x_n , r_n ) \cap A_n \cap U
\ \ , \ \
n \in \bN 
\]
(notare che la precedente definizione \e ben posta proprio perch\'e ogni $A_n$ \e denso). Si verifica facilmente che $\left\{ x_n \right\}$ \e di Cauchy, per cui esiste ed \e unico il limite $x$. Per costruzione $x \in U \cap (\cap_n A_n)$.
\end{proof}

\begin{thm}[Teorema di Banach-Steinhaus]
Siano $\mE,\mF$ spazi di Banach e $\left\{ T_i \right\}_{i \in I} \subset B(\mE,\mF)$ una famiglia tale che $\sup_i \| T_i v \| <$ $\infty$ per ogni $v \in \mE$. Allora esiste una costante $c > 0$ tale che $\| T_i \| < c$, $i \in I$.
\end{thm}

\begin{proof}[Dimostrazione]
Poniamo
\[
X_n
\ := \
\left\{
v \in \mE : \| T_i v  \| \leq n 
\ , \
i \in I
\right\}
\ \ , \ \
n \in \bN
\ .
\]
Poich\'e per ipotesi $\sup_i \| T_i v \| < $ $\infty$, abbiamo che $\cup_n X_n =$ $\mE$. Il lemma di Baire implica che deve esistere $n_0 \in \bN$ tale che $\dot X_n \neq$ $\emptyset$. Per cui, esistono $v_0 \in \mE$, $r_0 > 0$ tali che $\Delta (  v_0 , r_0 ) \subset X_{n_0}$. Quindi otteniamo, per ogni $w \in \mE$, $\| w \| \leq 1$,
\[
\| T_i ( v_0 + r_0 w ) \|
\leq
n_0
\ \Rightarrow \
r_0 \| T_i w \| 
\leq  
n_0 + \| T_i v_0 \| 
< 2 n_0
\ .
\]
\end{proof}

\begin{cor}
\label{cor_BS}
Sia $\left\{ T_n \right\} \subset B(\mE,\mF)$ una successione tale che $\left\{ T_n v \right\}$ converge per ogni $v \in \mE$. Allora:
\textbf{(1)} $\sup_n \| T_n \| < \infty$;
\textbf{(2)} Posto $Tv :=$ $\lim_n T_n v$, si ha che $T$ \e un operatore limitato e $\|  T  \| \leq \lim \limits_n \inf \| T_n \|$.
\end{cor}

\begin{rem}{\it 
Nel corollario precedente \textbf{non} si afferma che $\| T_n - T \| \to 0$. Tuttavia si pu\'o verificare con un argomento del tipo "$3$-$\eps$" che $\sup_{v \in K} \| T_n v - T v \| \to 0$ per ogni compatto $K \subset \mE$. 
} \end{rem}

\begin{cor}
Sia $B \subseteq \mE$ tale che $f(B)$ \e limitato per ogni $f \in \mE^*$. Allora $B$ \e limitato.
\end{cor}

%
%

A corollario dei risultati precedenti segnaliamo la seguente terminologia: dato uno spazio topologico $X$, un sottoinsieme $Y \subseteq X$ si dice {\em di tipo} $G_\delta$ se \e intersezione numerabile di insiemi aperti. 
Osserviamo che l'intersezione di insiemi $G_\delta$ \e $G_\delta$, e l'unione finita di insiemi $G_\delta$ \e $G_\delta$.
Inoltre, se $X$ \e metrico, allora ogni chiuso \e di tipo $G_\delta$.
Come esempio, segnaliamo $X = \bR$, $Y = \bR - \bQ$.

\

A seguire alcune notazioni. Per ogni $\delta, r > 0$ scriviamo $\mE_{\leq r} := \{ v \in \mE : \| v \| \leq r \}$. Inoltre definiamo
$\delta \mE_{\leq r} := \{ \delta w : w \in \mE_{\leq r} \} = \mE_{\leq \delta r}$
e
$v + \mE_{\leq r} := \{ v + w : w \in \mE_{\leq r} \}$,
cosicch\'e se $T \in B(\mE,\mF)$ allora
$T(v + \mE_{\leq r}) = Tv + T(\mE_{\leq r}) := \{ Tv + Tv' , v' \in \mE_{\leq r} \}$.

\begin{lem}
Siano $\mE,\mF$ spazi di Banach e $T \in B(\mE,\mF)$ tale che $T(\mE_{\leq 1})$ \e denso in qualche $\mF_{\leq r}$, $r > 0$. Allora per ogni $\eps > 0$ risulta 
$\mF_{(1-\eps)r} \subset T(\mE_{\leq 1})$.
\end{lem}

\begin{proof}[Dimostrazione]
Presi $w \in \mF_{\leq r}$ ed $\eps \in (0,1)$, per ipotesi esiste $w_1 \in T(\mE_{\leq 1})$ tale che 
$\| w - w_1 \| < \eps r$.
Ora, sempre per ipotesi, abbiamo che 
$\eps T(\mE_{\leq 1})$
\'e denso in $\mF_{\leq \eps r}$, per cui troviamo che esiste $w_2 \in \eps T(\mE_{\leq 1})$ tale che
$\| w - w_1 - w_2 \| < \eps^2 r$. Procedendo per induzione otteniamo una successione $\{ w_n \in \eps^{n-1}T(\mE_{\leq 1}) \}$ tale che
\[
\| w - \sum_k^n w_k \| < \eps^n r \ .
\]
Scegliamo ora $v_n \in \mE_{\leq 1}$ tali che $w_n = \eps^{n-1}Tv_n$, $\forall n \in \bN$. Allora la serie $\sum_n \eps^{n-1} v_n$ \e assolutamente convergente, e quindi convergente a $v \in \mE$. Per costruzione
$Tv = \sum_n \eps^{n-1}Tv_n = \sum_n w_n = w$
e
$\| v \| \leq \sum_n \eps^{n-1} \leq (1-\eps)^{-1}$, e ci\'o conclude la dimostrazione.
\end{proof}

\begin{thm}[Teorema dell'applicazione aperta]
Siano $\mE,\mF$ spazi di Banach e $T \in B(\mE,\mF)$ suriettivo. Allora $T(A)$ \e aperto in $\mF$ per ogni aperto $A \subseteq \mE$ (ovvero $T$ \e un'applicazione aperta).
\end{thm}

\begin{proof}[Dimostrazione]
Per l'ipotesi di suriettivit\'a l'insieme $\cup_{n \in \bN} \ovl{T(\mE_{\leq n})}$ \e denso in $\mF$, e per il Lemma di Baire esiste almeno un $m \in \bN$ tale che $\ovl{T(\mE_{\leq m})}$ contiene un intorno del tipo
\[
w + \mF_{\leq \eps}
\ , \  
w \in \mF
\ , \
\eps > 0 \ ,
\]
cosicch\'e $\ovl{T(\mE_{\leq 1})} \supseteq m^{-1}w + \mF_{\eps m^{-1}}$. Del resto, se $w_0 \in \mF_{\leq \eps m^{-1}}$ allora 
$m^{-1}w+w_0 \in m^{-1}w + \mF_{\leq \eps m^{-1}}$ 
e quindi troviamo successioni $\{ v_k \} , \{ v'_h \} \subset \mE_{\leq 1}$ tali che 
$m^{-1}w+w_0 = \lim_k Tv_k$, $m^{-1}w = \lim_h Tv'_h$; per cui
\[
1/2 w_0 = 1/2 \lim_k T(v_k-v'_k)
\ \ , \ \
{\mathrm{con}}
\ \
1/2(v_k-v'_k) \in \mE_{\leq 1} \ \forall k \in \bN
\ .
\]
Dunque $T(\mE_{\leq 1})$ \e denso in $\mF_{1/2 \eps m^{-1}}$. Usando il Lemma precedente, e la linearit\'a di $T$, concludiamo che per ogni $v \in \mE$ e $\rho > 0$ esiste un $\delta > 0$ tale che
\begin{equation}
\label{eq_app.ap}
Tv + \mF_{\leq \delta} \subset T ( v + \mE_{\leq \rho} )
\ .
\end{equation}
Ora, preso un aperto $U \subset \mE$ e $v \in U$ troviamo 
$v \in v + \mE_{\leq \rho} \subset U$ per qualche $\rho > 0$, 
e da (\ref{eq_app.ap}) concludiamo che $T ( v + \mE_{\leq \rho} )$ contiene l'intorno $Tv + \mF_{\leq \delta}$ di $Tv$. In altre parole $T$ \e un'applicazione aperta ed il teorema \e dimostrato.
\end{proof}

\begin{cor}[Teorema dell'inverso continuo]
Sia $T \in B(\mE,\mF)$ biettivo. Allora $T^{-1} \in B(\mF,\mE)$.
\end{cor}

\begin{proof}[Dimostrazione]
Il fatto che $T^{-1}$ \e lineare segue dalla linearit\'a di $T$, mentre la continuit\'a (ovvero limitatezza) di $T^{-1}$ segue dal fatto che $T$ \e un'applicazione aperta.
\end{proof}

\begin{cor}[Teorema del grafico chiuso]
Siano $\mE,\mF$ spazi di Banach e $T : \mE \to \mF$ un'applica- zione lineare. Se il \textbf{grafico} 
$G(T) := \{ v \oplus T(v) , v \in \mE \}$
\'e un sottoinsieme chiuso di $\mE \oplus \mF$, allora $T$ \e un operatore limitato.
\end{cor}

\begin{proof}[Dimostrazione]
La somma diretta $\mE \oplus \mF$ \e uno spazio vettoriale nella maniera ovvia, e definendo $\| (v,w) \| := \sup \{ \| v \| , \| w \| \}$ possiamo riguardare $\mE \oplus \mF$ come uno spazio di Banach. Osserviamo che abbiamo gli operatori
$P  \in B( \mE \oplus \mF , \mE )$, $P(v,w) := v$,
$P' \in B( \mE \oplus \mF , \mF )$, $P(v,w) := w$,
i quali chiaramente hanno norma $1$. Ora, per ipotesi $G(T)$ \e un sottospazio di Banach (ovvero chiuso) di $\mE \oplus \mF$, e per costruzione $S := P|_{G(T)}$ ha inverso
$S^{-1} \in B(\mE,G(T))$, $S^{-1}v := (v,Tv)$.
Per il teorema precedente $S^{-1}$ \e limitato, e di conseguenza $T = P' \circ S^{-1}$ \e limitato.
\end{proof}

\subsection{Il teorema di Hahn-Banach.}

Il teorema di Hahn-Banach \e uno dei risultati cardine dell'analisi funzionale. Oggetto del teorema \e l'esistenza di funzionali lineari limitati su spazi vettoriali equipaggiati con seminorme.

\begin{thm}[Hahn-Banach]
\label{thm_HB}
Sia $\mE$ uno spazio vettoriale e $p : \mE \to \bR$ una seminorma. Se $\mE'$ \e un sottospazio di $\mE$ ed $f : \mE' \to \bR$ \e un'applicazione lineare che soddisfa
\begin{equation}
\label{eq_HB1}
f(v') \leq p(v') 
\ \ , \ \ 
\forall v' \in \mE'
\ ,
\end{equation}
allora esiste $\tilde f : \mE \to \bR$ tale che $\tilde f |_{\mE'} = f$, $\tilde f(v) \leq$ $p(v)$, $\forall v \in \mE$.
\end{thm}

La diseguaglianza (\ref{eq_HB1}) pu\'o essere riguardata come una condizione di limitatezza per $f$, ed il teorema di Hahn-Banach ci assicura che l'estensione $\tilde f$ soddisfa la medesima minorazione di $f$ rispetto a $p$. Il caso che abbiamo in mente \e quello in cui $\mE$ \e uno spazio normato ed $f \in {\mE'}^*$, con
\[
\| f \| := \sup_{v' \in \mE'_1} | \left \langle f , v' \right \rangle |
\ \ , \ \
p(v) := \| f \| \|v\|
\ \ , \ \
v \in \mE
\ .
\]

\begin{proof}[Dimostrazione]
Si tratta di applicare il Lemma di Zorn al seguente insieme parzialmente ordinato:
\[
\mP :=
\{
h : \mE_h \to \bR 
\ : \
\mE' \subseteq \mE_h \subseteq \mE
\ , \
h |_{\mE'} = f
\ , \
h(v) \leq p(v)
\ , \
v \in \mE_h
\}
\]
equipaggiato con la relazione d'ordine
\[
h \leq h' 
\ \Leftrightarrow \
\mE_h \subset \mE_{h'} \ \ {\mathrm{e}} \ \ h'|_{\mE_h} = h
\ .
\]
Sia $\mC \subset \mP$ totalmente ordinato; allora, esiste un elemento massimale $k$ per $\mC$, definito da 
\[
\mE_k := \bigcup_{h \in \mC} \mE_h 
\ \ , \ \
k(v) := h(v)
\ \ , \ \
v \in \mE_h
\ .
\]
Dunque, per il Lemma di Zorn otteniamo che $\mP$ ammette un elemento massimale, che denotiamo con $\tilde f$. Resta da verificare che $\tilde f$ \e il funzionale che stiamo cercando. Innanzitutto, osserviamo che se il dominio $\mE_f$ di $\tilde f$ non coincidesse con $\mE$ allora potremmo considerare $v_0 \in \mE - \mE_f$, il sottospazio proprio $\mE_0 := \mE_f + \bR v_0$ di $\mE$, e definire
\[
f'(v+\lambda v_0) := \tilde f(v) + \lambda \delta 
\ \ , \ \
v \in \mE_f \ , \ \lambda \in \bR \ ,
\]
dove $\delta \in \bR$ \e una costante che sceglieremo in modo tale che 
\[
f'(v+\lambda v_0) = \tilde f(v) + \lambda \delta 
\ \leq \ 
p(v+\lambda v_0) 
\ ;
\]
ci\'o sarebbe in contraddizione con la massimalit\'a di $\tilde f$ ed il teorema sarebbe dimostrato. Del resto, usando (\ref{eq_HB1}), si verifica che basta scegliere
\[
\sup_{v \in \mE'} \{ f(v)-p(v-v_0) \}
\ \leq \
\delta
\ \leq \
\inf_{v \in \mE'} \{ p(v+v_0)-f(v) \}
\ .
\]
\end{proof}

\begin{cor}
Sia $\mE$ uno spazio normato, ed $\mE' \subseteq \mE$ un sottospazio vettoriale. Allora per ogni $f \in {\mE'}^*$ esiste $\tilde f \in \mE^*$ con $\tilde f |_{\mE'} = f$ e $\|  \tilde f  \| =$ $\| f \|$.
\end{cor}

\begin{cor}
Per ogni $v \in \mE$ si ha
\begin{equation}
\label{eq_HB2}
\| v \|
\ = \ 
\max_{f \in \mE^*_1} | \left \langle f , v \right \rangle |
\ .
\end{equation}
\end{cor}

\begin{proof}[Dimostrazione]
Assumiamo $v \neq 0$. E' ovvio che $\| v \| \geq$ $| \left \langle f , v \right \rangle |$, $f \in \mE^*_1$. D'altra parte, considerando il sottospazio $\mE' := \bR v$ ed il funzionale $f \in {\mE'}^*$, $\left \langle f , \lambda v \right \rangle :=$ $\lambda \|v\|^2$, $\lambda \in \bR$, ed applicando il teorema di Hahn-Banach, otteniamo l'uguaglianza cercata.
\end{proof}

%
%
%
%
%

\subsection{Operatori compatti ed il Teorema di Fredholm.}

Siano $\mE,\mF$ spazi di Banach (reali o complessi). Un operatore limitato $T \in B(\mE,\mF)$ si dice {\em compatto} se $T(\mE_{\leq 1})$ \e precompatto nella topologia della norma di $\mF$ (dove $\mE_{\leq 1}$ \e la palla unitaria di $\mE$). Denotiamo con $K(\mE,\mF)$ l'insieme degli operatori compatti da $\mE$ in $\mF$.

\begin{lem}
Sia $T \in K(\mE,\mF)$ compatto. Allora $T(A)$ \e compatto per ogni $A \subset \mE$ limitato.
\end{lem}

\begin{proof}[Dimostrazione]
Essendo $A$ limitato esso \e contenuto in una palla $\mE_{\leq r} := \{ v \in \mE : \| v \| \leq r \}$ per qualche $r \in \bR$. Ora, presa una successione $\{ Tv_n \} \subset T(\mE_{\leq r})$ troviamo che $\{ r^{-1} Tv_n \} \in T(\mE_{\leq 1})$, per cui possiamo estrarre una sottosuccessione convergente $\{ r^{-1} Tv_{n_k} \}$; di conseguenza anche $\{ Tv_{n_k} \}$ \e convergente, per cui $T(\mE_{\leq r})$ \e precompatto. Concludiamo che la chiusura di $T(A)$, essendo contenuta nel precompatto $T(\mE_{\leq r})$, \e compatta
{\footnote{Infatti in generale un insieme chiuso contenuto in un compatto \e compatto.}}.
\end{proof}

\begin{prop}
Siano $\mE,\mF$ spazi di Banach. Allora $K(\mE,\mF)$ \e uno spazio di Banach chiuso rispetto a composizioni con elementi di $B(\mF)$, $B(\mE)${\footnote{In particolare, se $\mE=\mF$ allora $K(\mE) := K(\mE,\mE)$ \e un ideale bilatero chiuso di $B(\mE)$.}}. Inoltre, se $T \in B(\mE,\mF)$ ha rango finito allora \e compatto.
\end{prop}

\begin{proof}[Dimostrazione]
L'immagine di un insieme compatto attraverso un'applicazione continua \e compatta, per cui se $S' \in B(\mF)$, $S \in B(\mE)$, $T \in K(\mE,\mF)$ allora applicando il Lemma precedente concludiamo che $S(\mE_{\leq 1})$ \e limitato (per continuit\'a di $S$), $T \circ S (\mE_{\leq 1})$ precompatto (per compattezza di $T$) ed $S' \circ T \circ S (\mE_{\leq 1})$ precompatto (per continuit\'a di $S'$). Ci\'o mostra che $K(\mE,\mF)$ \e chiuso rispetto a composizioni con elementi di $B(\mF)$, $B(\mE)$. 
Inoltre, la sfera unitaria di uno spazio a dimensione finita \e compatta e ci\'o mostra che ogni operatore $T$ a rango finito (ovvero, tale che $T(\mE)$ ha dimensione finita) \e compatto. 
Mostriamo che $K(\mE,\mF)$ \e uno spazio vettoriale: se $T,T' \in K(\mE,\mF)$ ed $A \subset \mE$ \e limitato allora $T(A) \times T'(A)$ \e compatto (essendo il prodotto di compatti) e l'applicazione 
\[
T(A) \times T'(A) \to \mF
\ \ , \ \
(Tv,T'v') \mapsto Tv + T'v'
\]
\'e continua. Ci\'o implica che $(T+T')(A)$ \e compatto e quindi $K(\mE,\mF)$ \e uno spazio vettoriale (che $\lambda T$, $\lambda \in \bR , \bC$, $T \in K(\mE,\mF)$, sia compatto non ci dovrebbero essere dubbi).
Infine mostriamo che $K(\mE,\mF)$ \e chiuso in norma. Consideriamo una successione di Cauchy $\{ T_i \} \subset K(\mE,\mF)$ e mostriamo che il limite $T$ \e compatto. A tale scopo \e sufficiente verificare che, data una successione limitata $\{ v_n \} \subset \mE$, esiste una sottosuccessione convergente di $\{ T v_n \}$, e ci\'o si dimostra utilizzando il seguente argomento diagonale. Per compattezza di $T_1$, esiste certamente una sottosuccessione $\{ v_n^{(1)} \}$ di $\{ v_n \}$ tale che $\{ T_1 v_n^{(1)} \}$ \e convergente; analogamente, esiste una sottosuccessione $\{ v_n^{(2)} \}$ di $\{ v_n^{(1)} \}$ tale che $\{ T_2 v_n^{(2)} \}$ converge; iterando il procedimento otteniamo la sottosuccessione $\{ v_1^{(1)} , v_2^{(2)} , \ldots , v_n^{(n)} , \ldots  \}$, la quale -- per costruzione -- \e tale che ogni $\{ T_i v_n^{(n)} \}$ \e convergente nel limite $n \to \infty$. Ora, abbiamo la stima
\[
\| T   v_n^{(n)} - Tv_m^{(m)}    \| \leq 
\| T   v_n^{(n)} - T_i v_n^{(n)} \| + 
\| T_i v_n^{(n)} - T_i v_m^{(m)} \| +
\| T_i v_m^{(m)} - T   v_m^{(m)} \|
\ ;
\]
il primo ed il terzo termine si annullano nel limite $i \to \infty$ per convergenza di $\{ T_i \}$ e limitatezza di $\{ v_n^{(n)} \}$, mentre il secondo termine si annulla nel limite $m,n \to \infty$ per convergenza di $\{ T_i v_n^{(n)} \}$. Ne segue che $\{ Tv_n^{(n)} \}$ \e convergente e dunque $T$ \e compatto.
\end{proof}

\begin{ex}{\it
Sia $X$ uno spazio metrico compatto, $K \in C(X \times X)$ e $\mu$ una misura di Radon su $X$. Definiamo l'operatore 
\[
T_K : L_\mu^1(X) \to C(X)
\ \ , \ \ 
T_Kf (y) := \int_X K(x,y)f(x) \ dx  
\ \ , \ \ 
y \in X \ .
\]
Verifichiamo che $T_K$ sia ben definito: essendo $K$ uniformemente continua (Heine-Cantor), scelto $\eps > 0$ troviamo $\delta > 0$ tale che $| K(x,y) - K(x,y') | < \eps$ per ogni $x,y,y'$ tali che
$d_2( (x,y) , (x,y') ) = d(y,y') < \delta$
(qui, e di seguito, usiamo la notazione dell'Esercizio \ref{sec_top}.4). Poich\'e $\delta$ \e indipendente da $x,y,y' \in X$ concludiamo che
\begin{equation}
\label{eq.st.cp}
| T_Kf(y) - T_Kf(y') | < \eps \| f \|_1 \ ,
\end{equation}
dunque $T_Kf \in C(X)$. Inoltre 
$| T_Kf(y) | \leq \| \kappa_y \|_\infty \| f \|_1 \leq \| K \|_\infty \| f \|_1$,
dunque $T_K$ \e limitato e $\| T_K \| \leq \| K \|_\infty$. %
Infine, se $\| f \|_1 \leq 1$ allora usando (\ref{eq.st.cp}) troviamo 
$| T_Kf(y) - T_Kf(y') | < \eps$
per $d(y,y') < \delta$; poich\'e il nostro $\delta$ non dipende dalla scelta di $f$ concludiamo che la famiglia $\{ T_Kf \}_{\| f \|_1 \leq 1}$ \e equicontinua e quindi, per Ascoli-Arzel\'a, precompatta. Dunque $T_K$ \e un operatore compatto.
Consideriamo ora, per ogni $p \in [1,\infty)$, l'operatore
\[
I_p : C(X) \to L_\mu^p(X)
\]
che associa ad $f \in C(X)$ la sua classe in $L_\mu^p(X)$; allora 
$\| I_pf \|_p \leq \| f \|_\infty \mu X^{1/p}$
e quindi $I_p$ \e limitato. Concludiamo dalla proposizione precedente che 
$I_p \circ T_K : L_\mu^1(X) \to L_\mu^p(X)$ 
\e compatto.
}
\end{ex}

\begin{ex}{\it
Sia $(X,\mM,\mu)$ uno spazio misurabile, $h_i \in L_\mu^r(X)$, $g_i \in L_\mu^q(X)$ per ogni $i = 1 , \ldots , n \in \bN$, e $p := \ovl q$. Posto 
$K(x,y) := \sum_i h_i(x) g_i(y)$, $x,y \in X$,
abbiamo l'operatore lineare
\[
T_K : L_\mu^p(X) \to L_\mu^r(X)
\ \ \ , \ \ \
T_Kf(x) \, :=  \,
\int_X K(x,y)f(y) \ dy \, = \,
\sum_i \left( \int_X fg_i \right) \cdot h_i(x)
\ ,
\]
per ogni $f \in L_\mu^p(X)$, $x \in X$.
$T_K$ ha rango finito, e quindi \e compatto. Riguardo la norma di $T_K$, usando la disuguaglianza di Holder troviamo la maggiorazione 
$\| T_Kf \|_r \, \leq \, \| f \|_p \, \sum_i \| h_i \|_r \| g_i \|_q$.
}
\end{ex}

\begin{rem}{\it
Lo spazio vettoriale degli operatori a rango finito \e denso in $K(\mE,\mF)$ se $\mF$ possiede una base di Schauder, e quindi in particolare se $\mF$ \e uno spazio di Hilbert. Per esempi di casi in cui tale risultato non vale si veda \cite[Cap.VI, Oss.1]{Bre}.
}
\end{rem}

\noindent \textbf{Complementi ortogonali in spazi di Banach.} Diamo ora una naturale generalizzazione della nozione di complemento ortogonale in uno spazio di Hilbert. Sia $\mE$ uno spazio di Banach ed $\mM \subset \mE$, $\mN \subset \mE^*$ sottospazi vettoriali. I {\em complementi ortogonali} di $\mM , \mN$ si definiscono rispettivamente come
\[
\left\{
\begin{array}{ll}
\mM^\perp 
\ := \ 
\{ f \in \mE^* : \left \langle f,v \right \rangle = 0  \ , \ \forall v \in \mM \}
\ \subseteq \mE^* \ ,
\\
\mN^\perp 
\ := \ 
\{ v \in \mE : \left \langle f,v \right \rangle = 0  \ , \ \forall f \in \mN \}
\ \subseteq \mE
\ .
\end{array}
\right.
\]
Analogamente al caso di uno spazio di Hilbert, si verifica facilmente che 
\begin{equation}
\label{eq_perpB}
\mM^{\perp \perp} = \ovl \mM \ ;
\end{equation}
per dettagli sui complementi ortogonali in spazi di Banach rimandiamo a \cite[II.5]{Bre}.


Sia ora $T \in B(\mE,\mF)$; presi $v \in \mE$ ed $f \in \mF^*$ si ha l'identit\'a
$
\left \langle f   , Tv \right \rangle =
\left \langle T^*f,  v \right \rangle
$,
per cui $f \in \ker T^*$ se e soltanto se $f \in \{ T(\mE) \}^\perp$, ovvero
\begin{equation}
\label{eq_kerimB}
\{ T(\mE) \}^\perp = \ker T^* 
\ \stackrel{(\ref{eq_perpB})}{\Leftrightarrow} \
\ovl{T(\mE)} = \{ \ker T^* \}^\perp
\ .
\end{equation}
Quando $T$ \e chiuso, chiaramente abbiamo $T(\mE) = \{ \ker T^* \}^\perp$.

\

\noindent \textbf{L'alternativa di Fredholm.} Possiamo ora dimostrare un risultato fondamentale per la soluzione di equazioni agli autovalori per operatori compatti.

\begin{lem}
\label{lem_fred}
Sia $T \in K(\mE)$. Allora $\ker (1-T)$ ha dimensione finita e $1-T \in B(\mE)$ \e un operatore chiuso.
\end{lem}

\begin{proof}[Dimostrazione]
Se $\ker (1-T)$ avesse dimensione infinita allora troveremmo una successione 
\[
\{ v_n \in \ker (1-T) \cap \mE_1  \}
\]
priva di sottosuccessioni convergenti in norma (vedi Esercizio \ref{sec_afunct}.4). D'altro canto, essendo $T$ compatto deve esistere una sottosuccessione convergente $\{ Tv_{n_k} \}$. Poich\'e $Tv_{n_k} = v_{n_k}$, $k \in \bN$, anche $\{ v_{n_k} \}$ converge, e questa \e una contraddizione.
Dimostriamo che $1-T$ \e chiuso: se $\{ v_n \} \subset \mE$ \e una successione tale che $\{ v_n - Tv_n \}$ converge a $w \in \mE$, occorre verificare che $w = v-Tv$ per qualche $v \in \mE$. Considerata la distanza 
$d_n := d( v_n , \ker (1-T) )$, $n \in \bN$,
osserviamo che avendo $\ker (1-T)$ dimensione finita deve esistere $z_n \in \ker (1-T)$ tale che 
$d_n = \| v_n - z_n \|$.
Supponiamo di aver mostrato che $\{ v_n - z_n\}$ \e limitata; in tal caso potremmo estrarre una sottosuccessione $\{ v_{n_k} - z_{n_k} \}$ tale che
\[
\exists
w_0 = \lim_k T( v_{n_k} - z_{n_k} ) = \lim_k Tv_{n_k}
\ \Rightarrow \
v_{n_k} = ( v_{n_k} - Tv_{n_k} ) + Tv_{n_k} \stackrel{k}{\to} w+w_0
\ ,
\]
al che
\[
w 
\ = \ 
\{ 1-T \} \lim_k v_{n_k}
\ = \
\{ 1-T \}(w+w_0)
\ ,
\]
e $w$ apparterrebbe all'immagine di $1-T$, il quale sarebbe quindi chiuso.
Dunque per dimostrare il Lemma rimane da verificare solo che $\{ v_n - z_n \}$ \e limitata. Se per assurdo cos\'i non fosse, troveremmo infiniti indici $k \in \bN$ tali che
$d_k \to \infty$,
e posto 
$s_k := d_k^{-1} (v_k-z_k) \subset \mE_1$ 
avremmo
\begin{equation}
\label{eq_fred}
s_k - Ts_k \ = \ d_k^{-1} (v_k - Tv_k) \ \stackrel{k}{\to} \ 0 \cdot w \ = 0 \ ;
\end{equation}
del resto, $\{ s_k \}$ \e limitata e quindi possiamo estrarre una sottosuccessione $\{ s_{i_k} \}$ tale che $\{ T s_{i_k} \}$ converge ad $s \in \mE$ il quale, grazie a (\ref{eq_fred}), \e tale che 
$Ts = s$,
ovvero $s \in \ker(1-T)$. Ma d'altra parte per costruzione
\[
d(s_{i_k},\ker(1-T)) \ = \ 
d_k^{-1} \ d( v_k-z_k , \ker(1-T) ) \ = \
d_k^{-1} \ d( v_k , \ker(1-T) ) \ = \
1
\ ,
\]
il che contraddice il fatto che $\{ s_{i_k} \}$ ha limite in $\ker(1-T)$.
\end{proof}

\begin{thm}[Fredholm]
Per ogni $T \in K(\mE)$ valgono le seguenti propriet\'a:
\begin{enumerate}
\item $\ker (1-T)$ ha dimensione finita;
\item $(1-T)(\mE) = \ker (1-T^*)^\perp$;
\item $\ker(1-T) = \{ 0 \}$ $\Leftrightarrow$ $(1-T)(\mE) = \mE$.
\end{enumerate}
\end{thm}

\begin{proof}[Dimostrazione]
(1) Vedi Lemma \ref{lem_fred}.
(2) Grazie al Lemma \ref{lem_fred} sappiamo che $1-T$ \e chiuso, e scrivendo (\ref{eq_kerimB}) per $1-T$ si trova $(1-T)(\mE) = \ker (1-T^*)^\perp$ come desiderato.
(3) Sia $\ker(1-T) = \{ 0 \}$. Supponendo per assurdo che $\mE_1 := (1-T)(\mE)$ sia strettamente contenuto in $\mE$ troviamo subito 
    $T |_{\mE_1} \subseteq \mE_1$, $T |_{\mE_1} \in K(\mE_1)$
    e, grazie al Lemma precedente, abbiamo che $\mE_2 := (1-T) |_{\mE_1}$ \e chiuso. Del resto $1-T$ \e iniettivo, per cui definendo induttivamente
    \[
    \mE_k := (1-T)(\mE_{k-1})
    \ \ , \ \
    k = 2, \ldots \ ,
    \]
    otteniamo una sequenza di sottospazi chiusi di $\mE$ contenuti {\em strettamente} l'uno nell'altro. Grazie al Lemma di Riesz (Esercizio \ref{sec_afunct}.4) esiste una successione $\{ u_k \in \mE_k \}$ tale che 
    $\| u_k \| \equiv 1$ e $d( u_k , \mE_{k+1} ) \geq 1/2$,
    per cui, per ogni $h < k$,
    \[
    \| T u_h - Tu_k \| =
    \| u_h - \{ (1-T)u_h - (1-T)u_k + u_k \} \| \geq
    d( u_h , \mE_{h+1} ) \geq 1/2
    \ .
    \]
    Ci\'o contraddice la compattezza di $T$, per cui $1-T$ deve essere suriettivo.
    Viceversa sia $1-T$ suriettivo. Allora (essendo $1-T$ chiuso per il Lemma precedente) troviamo $\ker (1-T^*) = \{ (1-T)(\mE) \}^\perp = \{ 0 \}$ (vedi (\ref{eq_kerimB})); per cui, essendo $T^*$ compatto (Esercizio \ref{sec_afunct}.3), possiamo applicare a quest'ultimo l'argomento precedente, concludendo che $1-T^*$ \e suriettivo. D'altra parte $1-T^*$ \e chiuso (sempre grazie al Lemma precedente), per cui applicando ancora (\ref{eq_kerimB}) concludiamo che $\{ 0 \} = \{ (1-T^*)(\mE^*) \}^\perp = \ker (1-T)$.
\end{proof}

\begin{rem}
\label{rem_FRED}
{\it
\textbf{(1)} Il teorema di Fredholm si pu\'o enunciare nel seguente modo: l'equazione
\[
u - Tu = v
\]
\textbf{o} ammette soluzione unica per ogni $v \in \mE$, \textbf{oppure} l'equazione omogenea
$u - Tu = 0$
ammette un numero finito di soluzioni linearmente indipendenti (l'\textbf{alternativa}, appunto). In tal caso $v \in \ker (1-T^*)^\perp$, per cui abbiamo una condizione di ortogonalit\'a di $v$ rispetto alle soluzioni dell'omogenea. 
\textbf{(2)}
L'alternativa di Fredholm \e una delle principali motivazioni di una nozione che svolge un ruolo fondamentale in analisi funzionale ed in geometria (!): dato lo spazio di Hilbert $\mH$, un operatore $S \in B(\mH)$ si dice di \textbf{Fredholm} se: (i) $\ker S$ ha dimensione finita; (ii) $S(\mH)$ ha codimensione finita (ovvero, $S(\mH)^\perp$ ha dimensione finita). In tal caso \e ben definito \textbf{l'indice}
\begin{equation}
\label{eq.index}
{\mathrm{ind}}S := {\mathrm{dim}}\ker S - {\mathrm{dim}}S(\mH)^\perp \in \bZ \ .
\end{equation}
Ad esempio, se $T = T^* \in K(\mH)$, allora per il teorema di Fredholm $1 - T$ \e un operatore di Fredholm con indice nullo. 
Denotiamo con $F(\mH)$ l'insieme degli operatori di Fredholm su $\mH$, il quale diventa naturalmente uno spazio topologico se equipaggiato con la topologia della norma. Per dare una vaga idea di come la geometria sia coinvolta nella nozione di indice segnaliamo \textbf{il teorema di Atiyah-Janich} (\cite[Appendix]{Ati}), il quale afferma che, dato lo spazio compatto di Hausdorff $X$, l'insieme delle classi di omotopia di funzioni continue 
$f : X \to F(\mH)$
\e isomorfo al gruppo $K^0(X)$ della \textbf{K-teoria} di $X$. Quest'ultimo \e un importante invariante topologico di $X$, nonch\'e la nozione fondamentale su cui poggia il famoso \textbf{teorema di Atiyah-Singer}.
}
\end{rem}

\subsection{I teoremi di Stampacchia e Lax-Milgram.}
\label{sec_StLM}

I teoremi di Stampacchia e Lax-Milgram sono dei risultati di estrema utilit\'a per la soluzione sia di problemi variazionali che di equazioni alle derivate parziali, in particolare nell'ambito degli spazi di Sobolev (vedi \S \ref{sec_sobolev}). La tecnica della dimostrazione si basa sulla geometria dei convessi negli spazi di Hilbert ed il teorema delle contrazioni.

\begin{thm}\textbf{(Proiezione su un convesso, \cite[Teo.V.2]{Bre}).}
Sia $\mH$ uno spazio di Hilbert e $K \subset \mH$ un convesso chiuso e non vuoto. Allora, per ogni $u \in \mH$ esiste ed \e unico $P_Ku \in K$ tale che $\| u - P_Ku \| = \min_K \| u-v \|$. E ci\'o avviene se e solo se $( u-P_Ku , v-P_Ku ) \leq 0$ per ogni $v \in K$.
\end{thm}

\begin{defn}
Sia $\mH$ uno spazio di Hilbert. Una forma bilineare $A : \mH \times \mH \to \bR$ si dice \textbf{coercitiva} se esiste $\alpha > 0$ tale che $A (u,u) \geq$ $\alpha \| u \|^2$, $u \in \mH$.
\end{defn}

\begin{thm}[Stampacchia]
\label{thm_SLM1}
Sia $\mH$ uno spazio di Hilbert, $K \subseteq \mH$ un convesso chiuso e non vuoto. Presa una forma bilineare, limitata e coercitiva $A$, per ogni $\varphi \in \mH^*$ risulta quanto segue:
\begin{enumerate}
\item Esiste ed \e unico $u_0 \in K$ tale che $A ( u_0 , v - u_0 ) \geq$ $\left \langle \varphi , v - u_0 \right \rangle$, $v \in K$;
\item Se $A$ \e simmetrica, allora $u_0 \in K$ soddisfa la condizione precedente se e soltanto se 
\begin{equation}
\label{eq_St}
\frac{1}{2} A ( u_0 , u_0 ) - \left \langle \varphi , u_0 \right \rangle
\ = \
\min_{v \in K}
\left\{  \frac{1}{2} A (v,v) - \left \langle \varphi , v \right \rangle  \right\}
\ .
\end{equation}
\end{enumerate}
\end{thm}

\begin{proof}[Dimostrazione]
Innanzitutto osserviamo che grazie al teorema di Riesz esistono unici $T \in B(\mH)$, $f \in \mH$ tali che $A (u,v) = ( Tu,v )$, $\left \langle \varphi , v \right \rangle =$ $(f,v)$, $\forall v \in \mH$. Per cui il nostro compito \e quello di trovare $u_0 \in \mH$ tale che
\[
( Tu_0 - f , v - u_0 ) \geq 0 
\ \ , \ \
v \in K
\ .
\]
L'equazione precedente \e equivalente a richiedere che preso un $\delta > 0$ si abbia
\[
( \delta f - \delta Tu_0 + u_0 - u_0 \ , \ v - u_0 ) \ \leq \ 0 
\ \ , \ \
v \in K
\ ,
\]
il che -- ricordando la definizione di proiezione $P_K$ su un sottoinsieme chiuso $K$ di uno spazio euclideo -- si pu\'o leggere come il fatto che
\begin{equation}
\label{eq_St3}
u_0 
\ = \
P_K ( \delta f - \delta Tu_0 + u_0 )
\ .
\end{equation}
Per cui, adesso l'idea \e quella di trovare $\delta$ affinch\'e sia verificata (\ref{eq_St3}), usando il teorema delle contrazioni. A tale scopo, definiamo
\[
S_\delta : K \to K
\ \ , \ \
v \mapsto P_K ( \delta f - \delta Tv + v )
\]
osservando che (\ref{eq_St3}) equivale a richiedere che $u_0$ sia un punto fisso per $S_\delta$. A questo punto, stimiamo
\[
\|  S_\delta v - S_\delta v'  \|^2
\leq
\| v - v' \|^2 - 
2 \delta ( v-v' , T(v-v') ) + 
\delta^2 \|  T(v-v')  \|
\leq
\| v - v' \|^2 ( 1 - 2 \alpha \delta + c^2 \delta^2 )
\ .
\]
Per cui, affinch\'e $S_\delta$ sia una contrazione (ed abbia quindi un punto fisso), \e sufficiente che sia verificata la disequazione $c^2 \delta^2 - 2 \alpha \delta < 0$, la quale ammette certamente soluzioni $\delta > 0$. Ci\'o dimostra il Punto 1 dell'enunciato del teorema.
Infine, se $A$ \e simmetrica allora $A(\cdot,\cdot)$ \e un prodotto scalare su $\mH$ e per il teorema di Riesz esiste ed \e unica $g \in \mH$ tale che
\begin{equation}
\label{eq_St4}
\left \langle \varphi , v \right \rangle
\ =  \
A ( g,v )
\ \ , \ \
v \in \mH
\ .
\end{equation}
Il Punto 1 dell'enunciato \e equivalente a 
\[
A ( g - u_0 , v - u_0 ) \leq 0
\ \ , \ \
v \in \mH
\ \Leftrightarrow \
u_0 = P_K g
\ ,
\]
il che \e equivalente a minimizzare, al variare di $v$ in $K$, la quantit\'a
$A ( g - v , g - v )$
o, equivalentemente, minimizzare $A(v,v) - 2 A(g,v)$.
\end{proof}

\begin{thm}[Lax-Milgram]
\label{thm_SLM2}
Sia $A : \mH \times \mH \to \bR$ una forma bilineare, limitata e coercitiva. Allora, per ogni $\varphi \in \mH^*$ esiste ed \e unico $u_0 \in \mH$ tale che
\begin{equation}
\label{eq_LM1}
A (u_0,v) = \left \langle \varphi , v \right \rangle \ .
\end{equation}
Inoltre, se $A$ \e simmetrica, $u_0$ soddisfa (\ref{eq_LM1}) se e soltanto se
\begin{equation}
\label{eq_LM2}
\frac{1}{2} A (u_0,u_0) - \left \langle \varphi , u_0 \right \rangle 
=
\min_{v \in \mH} 
\left\{ 
\frac{1}{2} A (v,v) - \left \langle \varphi , v \right \rangle
\right\}
\ .
\end{equation}
\end{thm}

\begin{proof}[Dimostrazione]
Da (\ref{eq_St3}) otteniamo (essendo $K = \mH$) $f = Tu_0$, da cui segue (\ref{eq_LM1}).
\end{proof}

\begin{rem}
{\it
Applicando il teorema di Riesz abbiamo l'operatore autoaggiunto $T \in B(\mH)$ e $g \in \mH$ invece della forma $A$ e di $\varphi \in \mH^*$, dunque possiamo esprimere il risultato precedente come segue: se $T$ \e tale che la relativa forma bilineare \e coercitiva, allora $T$ \e biettivo (infatti, la condizione $(Tu_0,v) = (g,v)$, $\forall v \in \mH$, \e equivalente a dire che $Tu_0 = g$).
}
\end{rem}

\subsection{Teoria spettrale.}

Gli operatori limitati, in particolare quelli compatti, appaiono frequentemente in analisi sotto la forma di operatori integrali, per cui \e importante conoscerne le propriet\'a spettrali, che si traducono in termini di soluzioni di problemi integro-differenziali.

\

\noindent \textbf{Spettro e risolvente.} Per ogni algebra di Banach (reale o complessa) $\mA$ con identit\'a $1$ denotiamo con
$\mA^{-1}$
il gruppo degli elementi invertibili di $\mA$, ovvero di quei $T \in \mA$ tali che esiste $T^{-1} \in \mA$ con $TT^{-1} = T^{-1}T = 1$. I seguenti due risultati si applicano (chiaramente) al caso particolare in cui $\mA = B(\mE)$ per qualche spazio di Banach $\mE$.

\begin{lem}
$\mA^{-1}$ \e aperto nella topologia della norma.
\end{lem}

\begin{proof}[Dimostrazione]
Sia $T \in \mA^{-1}$. Mostriamo che esiste $\delta > 0$ tale che ogni $T' \in \Delta(T,\delta)$ \e invertibile. Definiamo $S := 1-T^{-1}T'$ ed osserviamo che scegliendo $\delta < \| T^{-1} \|^{-1}$ otteniamo $\| S \| < 1$. Ora, la serie $\sum_{k=0} S^k$ (per convenzione poniamo $S^0 := 1$) \e assolutamente convergente e quindi convergente ad $A \in \mA$. Inoltre 
$(1-S)A = \lim_n (1-S) \sum_{k=0}^n S^k = \lim_n (1-S^n) = 1$, 
per cui $A = (1-S)^{-1} = (T^{-1}T')^{-1}$. Ne segue che $T'$ ha inverso $AT^{-1}$.
\end{proof}

\begin{defn}
\label{def_spettro}
Sia $\mA$ un'algebra di Banach (reale o complessa) con identit\'a $1$, e $T \in \mA$. Il \textbf{risolvente} di $T$ \e dato dall'insieme
\[
\rho (T) := \{  \lambda \in \bK : T-\lambda 1 \in \mA^{-1}  \} 
\ \ , \ \
\bK = \bR , \bC
\ .
\]
Lo \textbf{spettro} di $T$ \e dato da $\sigma(T) := \bK - \rho(T)$. 
Supponiamo ora che $\mA = B(\mE)$ per qualche spazio di Banach $\mE$; diciamo che $\lambda \in \bK$ \e un \textbf{autovalore} di $T$ se esiste un \textbf{autovettore} $v \in \mE$, ovvero $Tv = \lambda v$; l'insieme degli autovalori di $T$ si denota con $\sigma p (T)$.
\end{defn}

Osserviamo che \e del tutto evidente che $\sigma p (T) \subseteq \sigma (T)$, ed in effetti se $\mE$ ha dimensione finita allora $\sigma (T) = \sigma p (T)$; d'altro canto, consideriamo lo {\em shift}
\begin{equation}
\label{eq_shift}
S \in B(l^2) 
\ : \
Sx := ( 0 , x_1 , x_2 , \ldots )
\ \ , \ \
x := ( x_1 , x_2 , \ldots  ) \in l^2
\ ;
\end{equation}
si verifica immediatamente che $0 \in \sigma (T) - \sigma p (T)$ (infatti $S$ \e iniettivo ma non suriettivo).

\begin{prop}
\label{prop_spettro}
Sia $\mA$ un'algebra di Banach (reale o complessa) con identit\'a $1$, e $T \in \mA$. Allora $\sigma(T)$ \e compatto e si ha l'inclusione 
$\sigma(T) \subseteq 
\ovl{\Delta(0,\| T \|)} = 
\{ \lambda : |\lambda| \leq \| T \| \}$. 
\end{prop}

\begin{proof}[Dimostrazione]
Innanzitutto dimostriamo che $\sigma(T) \subseteq \ovl{\Delta(0,\| T \|)}$; assumendo che $\| T \| > 0$, prendiamo $\lambda$ tale che $| \lambda | > \| T \|$ (cosicch\'e $\lambda \neq 0$) ed osserviamo che $\| \lambda^{-1} T \| < 1$. Ragionando come nel Lemma precedente abbiamo che la serie $\sum_{n=0} \lambda^{-n} T^n$ \e assolutamente convergente, e quindi convergente ad $A \in \mA$ tale che $(1-\lambda^{-1}T)A = A(1-\lambda^{-1}T) = 1$. Dunque $1-\lambda^{-1}T$ (e quindi $T- \lambda 1$) \e invertibile.
%
%
%
%
Verifichiamo ora che $\sigma(T)$ \e chiuso. Supponiamo per assurdo che esista una successione $\{ \lambda_n  \} \subset \sigma(T)$ tale 
che $\lambda_n \to \lambda \in \rho(T)$. Cio' vuol dire che $T-\lambda 1$ \e invertibile, ma del resto
\[
\| (T-\lambda_n 1) - (T - \lambda 1)  \| = | \lambda - \lambda_n | \stackrel{n}{\to} 0 \ .
\]
Poich\'e $\mA^{-1}$ \e aperto nella topologia della norma (Lemma precedente) otteniamo una contraddizione, e ci\'o dimostra la proposizione.
\end{proof}

\begin{rem}{\it
\label{rem_sp_ad}
Sia $\mA$ una $*$-algebra di Banach con identit\'a $1$ e $T \in \mA$. Preso $\lambda \in \bC$ abbiamo che $T - \lambda 1$ ha inverso $B \in \mA$ se e solo se $T^* - \ovl \lambda 1$ ha inverso $B^*$. Dunque $\rho (T^*) = \ovl{\rho(T)}$ e $\sigma (T^*) = \ovl{\sigma(T)}$.
In particolare, se $\mH$ \e uno spazio di Hilbert e $T = T^* \in B(\mH)$ allora $T$ ha spettro reale
$\sigma(T) \subset \bR$.
}
\end{rem}

Il risultato seguente permette di ottenere informazioni sullo spettro di un operatore autoaggiunto su uno spazio di Hilbert:
\begin{prop}[Il principio del minimax]
\label{prop_sTT}
Sia $T = T^* \in B(\mH)$. Posto 
\[
\lambda^+ := \sup_{\| u \| = 1} (u,Tu)
\ \ , \ \
\lambda^- := \inf_{\| u \| = 1} (u,Tu)
\ ,
\]
si ha $\sigma(T) \subseteq [ \lambda^- , \lambda^+ ]$ e $\lambda^\pm \in \sigma(T)$.
\end{prop}

\begin{proof}[Dimostrazione]
Abbiamo $(u,Tu) \leq \lambda^+ \| u \|^2$ per ogni $u \in \mH$, per cui se $\lambda > \lambda^+$ allora esiste $\eps > 0$ tale che
\[
( u,(\lambda 1 - T)u ) \ \geq \ 
\lambda \| u \|^2 - \lambda^+ \| u \|^2 \ > \
\eps \| u \|^2
\ .
\]
Dunque la forma definita da $\lambda 1 - T$ \e coercitiva e per il teorema di Lax-Milgram $\lambda 1 - T$ \e biettivo, il che vuol dire che $\lambda \notin \sigma(T)$. 
Passiamo ora al secondo enunciato: per mostrare che $\lambda^+ \in \sigma(T)$ mostreremo che $\lambda^+ 1 - T$ non pu\'o essere invertibile.
A tale scopo osserviamo che 
\[
u,v \mapsto ( u, (\lambda^+ 1 - T)v )
\ \ , \ \
u,v \in \mH
\ ,
\]
\e una forma sesquilineare simmetrica e definita positiva, per cui la diseguaglianza di Cauchy-Schwarz (\ref{eq_CSg}) implica che
\[
| ( u, (\lambda^+ 1 - T)v ) | \ \leq \
( u, \lambda^+ u - Tu )^{1/2} ( v, \lambda^+ v - Tv )^{1/2}
\ \ , \ \
\forall u,v \in \mH
\ .
\]
Passando al $\sup$ al variare di $u \in \mH$, $\| u \| = 1$, troviamo la stima
\begin{equation}
\label{eq.sTT}
\| (\lambda^+ 1 - T)v \| \ \leq \ c ( v, \lambda^+ v - Tv )^{1/2} 
\ \ , \ \
\forall v \in \mH
\ ,
\end{equation}
dove $c := \sup_{\| u \| = 1} ( u, \lambda^+u - Tu )^{1/2}$. Ora, per definizione di $\lambda^+$ esiste una successione $\{ v_n \}$ tale che
$\| v_n \| \equiv 1$ e $( v_n , \lambda^+v -Tv_n ) \to 0$,
dunque 
\[
\| (\lambda^+ 1 - T) v_n \| 
\ \stackrel{(\ref{eq.sTT})}{\leq} \ 
c ( v_n , \lambda^+ v_n - Tv_n )^{1/2}
\ \to \
0
\ ,
\]
e se $\lambda^+ 1 - T$ avesse inverso $B \in B(\mH)$ troveremmo $v_n = B(\lambda^+ 1 - T) v_n \to 0$, il che contraddice la condizione $\| v_n \| \equiv 1$. Concludiamo che $\lambda^+ \in \sigma(T)$.
Ripetendo il ragionamento per $\lambda^-$ (con diseguaglianze invertite) otteniamo quanto desiderato.
\end{proof}

\begin{cor}
Se $T \in B(\mH)$ \e autoaggiunto allora
\begin{equation}
\label{eq_sp_Ta}
\sigma (T) = \{ 0 \} \ \Rightarrow \ T = 0 \ .
\end{equation}
\end{cor}

\begin{proof}[Dimostrazione]
Per la proposizione precedente deve essere $(u,Tu) = 0$ per ogni $u \in \mH$, e quindi
$2(u,Tv) = ( u+v,T(u+v) ) -  (u,Tu) - (v,Tv)  = 0$, $\forall u,v \in \mH$.
Dunque $T=0$.
\end{proof}

\

\begin{ex}{\it
Consideriamo lo spazio di Hilbert $l^2_\bC$ e l'operatore
\[
Tx := \{ 0 , x_1 , \ldots , x_n/n  , \ldots \}
\ \ , \ \
x := \{ x_n \} \in l^2_\bC
\ .
\]
Allora $T$ \e compatto (infatti \e limite degli operatori di rango finito 
$T_nx := ( 0 . x_1 , \ldots , x_n/n , 0 , \ldots )$, $x \in l^2_\bC$) 
ma non autoaggiunto (si calcoli infatti $T^*$ e si verifichi che $T^* \neq T$). Inoltre abbiamo
$\{ Tx-\lambda x \}_{n+1} = x_n/n - \lambda x_{n+1}$, $\forall n \in \bN$, 
per cui $T - \lambda 1$ \e invertibile per ogni $\lambda \neq 0$ 
{\footnote{Lasciamo a chi legge questa verifica, come semplice esercizio di algebra lineare.}}, 
mentre evidentemente $T$ non \e suriettivo. Concludiamo quindi che $\sigma(T) = \{ 0 \}$, in contrasto con (\ref{eq_sp_Ta}).
}
\end{ex}

\begin{ex}
\label{ex_op_int}
{\it
Sia $X$ uno spazio compatto e di Hausdorff, $\mu$ una misura di Radon su $X$ e $K \in C(X \times X)$; definiamo l'operatore lineare 
\[
T_K : C(X) \to C(X) \ \ , \ \ T_Kf (y) := \int_X K(x,y)f(x) \ dx  \ \ , \ \ y \in X \ .
\]
Osserviamo che $T_K$ \e limitato in quanto $\| T_K f \|_\infty \leq \| K \|_\infty \mu X \| f \|_\infty$. Preso $\lambda \in \bR$, troviamo che esso \e un autovalore di $T_K$ se e solo se esiste $u \in C(X)$ soluzione del problema
\begin{equation}
\label{eq_autovett}
\lambda u(y) = \int_X K(x,y)u(x) \ dx  
\ \ , \ \
y \in X
\ ,
\end{equation}
noto come \textbf{equazione di Volterra omogenea di seconda specie}.
}
\end{ex}

\begin{ex}
\label{ex_op.mol}
{\it Consideriamo lo spazio di Hilbert $L^2 \equiv L^2([0,1])$ e l'operatore
\[
T \in BL^2 
\ \ : \ \
Tu (x) := x u(x)
\ , \
x \in [0,1]
\ , \
u \in L^2
\ .
\]
Poich\'e 
$\| Tu \|_2^2 = \int x^2 u(x)^2 dx \leq \int u(x)^2 dx$ 
troviamo subito $\| T \| \leq 1$
{\footnote{In effetti $\| T \| = 1$; lasciamo la verifica di questa uguaglianza come esercizio.}}.
Valutiamo gli operatori $T - \lambda 1$ al variare di $\lambda \in \bR$; per prima cosa, osserviamo che, preso $u \in L^2$,
\[
\{ T - \lambda 1 \} u(x) = 0 \ \ {\mathrm{q.o.}} 
\ \ \Rightarrow \ \
(x - \lambda) u(x)       = 0 \ \ {\mathrm{q.o.}} 
\ \ \Rightarrow \ \
u (x)                    = 0 \ \ {\mathrm{q.o.}} 
\ ,
\]
per cui ogni $T - \lambda 1$ \e iniettivo al variare di $\lambda \in \bR$.
Ora, \e conveniente definire la funzione
$f_\lambda (x) := (x - \lambda)^{-1}$,
$x \in [0,1]$,
ed osservare che se $\lambda \in \bR - [0,1]$ allora l'operatore 
\[
S_\lambda \in BL^2
\ \ : \ \
S_\lambda u (x) := f_\lambda(x) u(x)
\ , \
x \in [0,1]
\ , \
u \in L^2
\ ,
\]
\e limitato e tale che $(T-\lambda 1) S_\lambda = S_\lambda (T-\lambda 1) = 1$, dunque $\bR - [0,1] \subseteq \rho(T)$. D'altro canto, se $\lambda \in [0,1]$ allora definendo 
$u_0 \in L^2$, $u_0(x) := 1$, $x \in [0,1]$,
e supponendo $u_0 = ( T - \lambda 1 ) v$ per qualche $v \in L^2$, troviamo la contraddizione
\[
\| v \|_2^2 
\ = \ 
\| f_\lambda u_0 \|_2^2 
\ = \ 
\int_0^1 \frac{ dx }{ (x-\lambda)^2 }
\ = \
\infty
\ .
\]
Dunque $T-\lambda 1$ non \e suriettivo e $\lambda \in \sigma(T)$. Concludiamo che
$\sigma (T) = [0,1]$,
$\sigma p (T) = \emptyset$.
}
\end{ex}

\

\noindent \textbf{Il teorema spettrale per gli operatori compatti.} Passiamo ora a studiare lo spettro di un operatore compatto.
\begin{lem}
Sia $T \in K(\mE)$ e $\{ \lambda_n \} \subseteq \sigma p(T)$ una successione di numeri reali (complessi) distinti e non nulli convergente a $\lambda$. Allora $\lambda = 0$.
\end{lem}

\begin{proof}[Dimostrazione]
Per ogni $n \in \bN$ consideriamo un autovettore $v_n \in \mE$, $Tv_n = \lambda_n v_n$, e definiamo $\mV_n := {\mathrm{span}} \{ v_1 \ldots v_n \}$. Mostriamo ora che $v_1 , \ldots , v_n$ sono linearmente indipendenti. Poich\'e ci\'o \e chiaramente vero per $n=1$, procediamo induttivamente e, assunto che $v_1 , \ldots , v_n$ siano linearmente indipendenti, mostriamo che $v_1 , \ldots , v_{n+1}$ sono linearmente indipendenti. Se per assurdo fosse $v_{n+1} = \sum_i^n a_i v_i$, avremmo
\[
Tv_{n+1} = \sum_i^n a_i \lambda_i v_i = \lambda_{n+1} \sum_i^n a_i v_i
\ \Rightarrow \
a_i (\lambda_i - \lambda_{n+1}) = 0 \ \ \forall i = 1 , \ldots , n
\ .
\]
Poich\'e $\lambda_{n+1} \neq \lambda_i$ per ogni $i$ otteniamo una contraddizione, e dunque $v_1 , \ldots , v_{n+1}$ sono linearmente indipendenti. Ci\'o implica che abbiamo inclusioni proprie $\mV_n \subset \mV_{n+1}$ per ogni $n \in \bN$. D'altra parte, per costruzione,
\[
(T-\lambda_n 1) \mV_n \subset \mV_{n-1} \ .
\]
Usando il Lemma di Riesz (Esercizio \ref{sec_afunct}.4(1)) troviamo che esiste una successione $\{ u_n \} \subset \mE_1$ tale che $u_n \in \mV_n$ e $d(u_n,\mV_{n-1}) \geq 1/2$. Ora, se $m < n-1$ abbiamo $\mV_m \subset \mV_{n-1}$ e
\begin{equation}
\label{eq_SP_T}
\| \lambda_n^{-1} Tu_n - \lambda_m^{-1} Tu_m  \|
\ = \
\|  u_n - v  \|
\ \geq \
1/2
\end{equation}
dove
\[
v 
\ := \ 
u_m - \lambda_n^{-1} (T-\lambda_n 1)u_n + \lambda_m^{-1}(T-\lambda_m 1)u_m
\ \in \ 
\mV_{n-1}
\ .
\]
Se, per assurdo, $\lambda = \lim_n \lambda_n$ fosse non nullo troveremmo che la successione$\{ \lambda_n^{-1} Tu_n \}$ sarebbe limitata; per compattezza di $T$ potremmo estrarne una sottosuccessione convergente, e ci\'o contraddice (\ref{eq_SP_T}).
\end{proof}

\begin{thm}
Sia $\mE$ uno spazio di Banach (reale o complesso) a dimensione infinita e $T \in K(\mE)$. Allora: 
(1) $0 \in \sigma(T)$; 
(2) $\sigma(T) - \{ 0 \} = \sigma p(T) - \{ 0 \}$;
(3) $\sigma(T)$ \e un insieme al pi\'u numerabile, ed in tal caso $0$ \e l'unico punto di accumulazione.
\end{thm}

\begin{proof}[Dimostrazione]
(1) Supponendo per assurdo che $T$ sia invertibile troviamo che $T^{-1}(\mE_1)$ \e limitato e quindi $\mE_1 = T (T^{-1}(\mE_1))$ \e compatto. Ci\'o contraddice il fatto che $\mE$ ha dimensione infinita. 
(2) Sia $\lambda \in \sigma(T)- \{ 0 \}$. Se $\ker (T-\lambda 1) = \{ 0 \}$ allora per l'alternativa di Fredholm troviamo $(T-\lambda 1)\mE = \mE$ per cui $\lambda \in \rho(T)$. Ci\'o contraddice l'ipotesi $\lambda \in \sigma (T)$, per cui deve essere $\lambda \in \sigma p(T)$.
(3) Poniamo
$A_n := \sigma (T) \cap \{ \lambda  : |\lambda| \geq n^{-1} \}$;
grazie al punto precedente sappiamo che $A_n \subset \sigma p(T)$. Inoltre, ogni $A_n$ pu\'o essere al pi\'u finito, altrimenti potremmo estrarne una sottosuccessione convergente a $\lambda \neq 0$, in contraddizione con il lemma precedente. Per cui, avendosi $\sigma (T) - \{ 0 \} = \cup_n A_n$ troviamo che $\sigma (T)$ \e numerabile, e grazie al Lemma precedente sappiamo che $0$ \e l'unico eventuale punto di accumulazione.
\end{proof}

Il risultato seguente caratterizza lo spettro degli operatori compatti autoaggiunti su uno spazio di Hilbert:
\begin{thm}[Il teorema spettrale per operatori autoaggiunti compatti]
\label{thm_sp.cp}
Sia $\mH$ uno spazio di Hilbert separabile (reale o complesso) e $T \in K(\mH)$ autoaggiunto. Allora $\mH$ ammette una base di autovettori di $T$.
\end{thm}

\begin{proof}[Dimostrazione]
Poniamo $\mH_0 := \ker T$ ed $\mH_n := \ker \{ T - \lambda_n 1 \}$. Per il teorema di Fredholm (vedi anche Oss.\ref{rem_FRED}) abbiamo che ogni $\mH_n$ ha dimensione finita. Poich\'e $(u,Tv) = \lambda_m(u,v) = (Tu,v) = \lambda_n (u,v)$ concludiamo che $\mH_n \perp \mH_m$, $n \neq m$. Vogliamo ora verificare che $\wt \mH := \cup_{n \geq 0} \mH_n$ \e denso in $\mH$. A tale scopo osserviamo che $T(\wt \mH) \subseteq \wt \mH$ e che, avendosi $(Tv,u) = (v,Tu) = 0$, $u \in \mH$, $v \in \wt \mH^\perp$, ha senso definire l'operatore $T_0 := T |_{\wt \mH^\perp}$. Osserviamo che $T_0$ \e autoaggiunto e compatto. Ora, se $\lambda \in \sigma (T_0) - \{ 0 \}$ allora, per il teorema precedente, esiste un autovettore $v \in \wt \mH^\perp$ tale che $T_0u = Tu = \lambda u$; per cui si avrebbe $\lambda \in \sigma p(T)$ e $u \in \wt \mH^\perp \cap \mH_n$, il che \e assurdo. Concludiamo che $\sigma (T_0) = \{ 0 \}$, per cui $T_0 = 0$ grazie a (\ref{eq_sp_Ta}). Ora, $T_0 = 0$ equivale ad affermare che $\wt \mH^\perp \subseteq \ker T$. Concludiamo che
$\wt \mH^\perp \subseteq \ker T \subseteq \wt \mH$, per cui $\wt \mH = \{ 0 \}$.
\end{proof}

\

\noindent \textbf{Il teorema spettrale per gli operatori autoaggiunti.} Esiste una versione del teorema precedente per il caso degli operatori autoaggiunti limitati, la quale si pu\'o esprimere con il linguaggio della teoria della misura. Allo scopo di facilitarne la comprensione gettiamo uno sguardo un p\'o diverso al caso compatto: preso $T = T^* \in K(\mH)$, osserviamo che lo spettro $\sigma(T) = \{ \lambda_k \}$, essendo numerabile, pu\'o essere riguardato in modo naturale come uno spazio di misura $( \sigma(T),\mM,\mu )$ con $\sigma$-algebra $\mM := 2^{\sigma(T)}$ (vedi Esempio \ref{ex_misnum}). 
Preso $\lambda_k \in \sigma(T)$ consideriamo il proiettore $P(\lambda_k) \in B(\mH)$ sul sottospazio di Hilbert $\ker (T-\lambda_k 1)$ cosicch\'e, grazie al Teorema \ref{thm_sp.cp}, abbiamo la decomposizione ortogonale
\[
v = \sum_k P(\lambda_k) v
\ \ \Rightarrow \ \
\| v \|^2 = \sum_k \| P(\lambda_k) v \|^2
\ \ , \ \
\forall v \in \mH
\ .
\]
Per ogni $u,v \in \mH$ definiamo 
\[
\mu_{uv} : \mM \to \bC
\ \ , \ \
\mu_{uv}E := \sum_{\lambda_k \in E} (u,P(\lambda_k)v)
\ ;
\]
\e immediato verificare che $\mu_{uv}$ \e una misura complessa su $\sigma(T)$ tale che, in particolare,
\[
\mu_{uv}\{ \lambda_k \} = (u,P(\lambda_k)v)
\ \ , \ \
\forall \lambda_k \in \sigma(T)
\ .
\]
Presa una funzione limitata $f : \sigma(T) \to \bC$ (che ovviamente \e $\mu$-misurabile), abbiamo che la forma sesquilineare 
\[
A_f(u,v) := \sum_k f(\lambda_k) \mu_{uv} \{ \lambda_k \}
\ \ , \ \
u,v \in \mH
\ ,
\]
\e limitata, in quanto 
\[
| A_f(u,v) |^2                                            \ \leq \ 
\sum_k |(u, f(\lambda_k) P(\lambda_k)v)|^2                \ \leq \ 
\| u \|^2 \sum_k |f(\lambda_k)|^2 \| P(\lambda_k) v \|^2  \ \leq \
\| u \|^2 \| f \|_\infty^2 \| v \|^2
\ ,
\]
per cui esiste ed \e unico l'operatore $f(T) \in B(\mH)$ tale che $A_f(u,v) = (u,f(T)v)$.
In particolare, poich\'e $T P(\lambda_k)v = \lambda_k P(\lambda_k) v$, abbiamo la seguente "diagonalizzazione" di $T$:
\begin{equation}
\label{eq_dec.sp}
(u,Tv) 
\ = \
\sum_k ( u , T P(\lambda_k) v ) 
\ = \
\sum_k \lambda_k \mu_{uv} \{ \lambda_k \} 
\ \ , \ \
\forall u,v \in \mH
\ .
\end{equation}
L'idea alla base dei ragionamenti precedenti \e quella che, dati lo spettro di $T$ e la famiglia di misure $\{ \mu_{uv} \}$, riusciamo a ricostruire non solo $T$, bens\'i anche ogni operatore del tipo $f(T)$, dove $f$ \e una qualsiasi funzione limitata sullo spettro di $T$.

\

{\em Consideriamo ora operatori $T \in B(\mH)$ non necessariamente compatti}. Chiaramente ora non \e pi\'u detto che lo spettro $\sigma (T)$ sia numerabile, n\'e che $\mH$ ammetta una base di autovettori (vedi Esempio \ref{ex_op.mol}). Tuttavia si pu\'o ancora dimostrare un analogo di (\ref{eq_dec.sp}), e qui di seguito ne esponiamo in buon dettaglio la dimostrazione.

Iniziamo assumendo che $\mH$ {\em sia uno spazio di Hilbert complesso}. Ci\'o sia perch\'e in questo ambito le argomentazioni seguenti sono valide anche per operatori pi\'u generici di quelli autoaggiunti 
{\footnote{
Ovvero i cosiddetti operatori {\em normali} $T \in B(\mH)$ tali che $TT^* = T^*T$. Osserviamo che nel caso reale operatori non autoaggiunti non sono in genere diagonalizzabili gi\'a in dimensione finita (quando il polinomio caratteristico ha redici complesse).
}}, 
sia perch\'e in alcuni passi sar\'a cruciale usare tecniche di analisi complessa.

Come primo passo denotiamo con $P(\sigma(T),\bC) \subset C(\sigma(T),\bC)$ la $*$-algebra delle funzioni polinomiali 
$f(\lambda) := \sum_k^n z_k \lambda^k$, $z_k \in \bC$, $\lambda \in \sigma(T)$
(per $k=0$ poniamo $\lambda^0 = 1$), e definiamo l'applicazione
\begin{equation}
\label{eq_SM1}
P(\sigma(T),\bC) \to B(\mH)
\ \ , \ \
f \mapsto f(T) := \sum_{k=0}^n z_k T^k
\ .
\end{equation}
Osserviamo che in particolare $T = I(T)$, dove 
\[
I(\lambda) := \lambda
\ \ , \ \
\forall \lambda \in \sigma(T)
\ .
\]
Inoltre \e evidente che 
\begin{equation}
\label{eq_SM1a}
\{ f+zg \}(T) = f(T) + zg(T)
\ \ , \ \
\{ fg \}(T) = f(T) g(T)
\ \ , \ \
\{ f^* \}(T) = f(T)^*
\ ,
\end{equation}
$\forall f,g \in P(\sigma(T),\bC)$,
$z \in \bC$.
Il teorema di Stone-Weierstrass (Teo.\ref{thm_sw}) ci dice che $P(\sigma(T),\bC)$ \e denso in $C(\sigma(T),\bC)$, ed \e possibile dimostrare che
{\footnote{Per l'uguaglianza seguente vedi \cite[Prop.4.3.15]{Ped} o \cite[Theorem VII.1]{RS}; \e in questo punto che \e importante considerare spazi di Hilbert complessi.}}
\[
\| f \|_\infty = \| f(T) \|
\ \ , \ \
\forall f \in P(\sigma(T),\bC)
\ .
\]
Dunque estendendo per continuit\'a otteniamo un operatore isometrico
\begin{equation}
\label{eq_SM2}
C(\sigma(T),\bC) \to B(\mH)
\ \ , \ \
f \mapsto f(T)
\ ,
\end{equation}
noto come {\em il calcolo funzionale continuo di} $T$; osserviamo che (\ref{eq_SM2}) \e in realt\'a una {\em rappresentazione}, il che vuol dire che continuano a valere le (\ref{eq_SM1a}) per generici elementi di $C(\sigma(T),\bC)$.
Ora, per ogni $u,v \in \mH$ definiamo il funzionale lineare
\[
\varphi_{uv} : C(\sigma(T),\bC) \to \bC
\ \ , \ \
\left \langle \varphi_{uv} , f \right \rangle := (u,f(T)v)
\ \ , \ \
\forall f \in C(\sigma(T),\bC)
\ ,
\]
il quale \e limitato in quanto
\begin{equation}
\label{eq_SM20}
| \left \langle \varphi_{uv} , f \right \rangle | \ \leq \
\| f(T) \| \| u \| \| v \| \ = \
\| f \|_\infty \| u \| \| v \|  
\ .
\end{equation}
Applicando il teorema di Riesz-Markov (Esempio \ref{ex_dual_CX}) troviamo che esiste ed \e unica la misura di Radon complessa $\mu_{uv} : \mM_{uv} \to \bC$ su $\sigma(T)$ tale che
\begin{equation}
\label{eq_SM2a}
\left \langle \varphi_{uv} , f \right \rangle
\ = \
(u,f(T)v)
\ = \
\int_{\sigma(T)} f(\lambda) \ d \mu_{uv}(\lambda)
\ \ , \ \
\forall f \in C(\sigma(T),\bC)
\ .
\end{equation}
Osserviamo che in particolare $\mu_{uu}$ \e una misura di Radon {\em reale e positiva} per ogni $u \in \mH$. Abbiamo dunque un'applicazione
\begin{equation}
\label{eq_SM3}
\mH \times \mH \to R(\sigma(T),\bC)
\ \ , \ \
u,v \mapsto \mu_{uv}
\ ,
\end{equation}
la quale \e chiaramente sesquilineare, ovvero
\begin{equation}
\label{eq_SM3a}
\mu_{au+u',bv+v'} 
\ =  \
\ovl ab \mu_{uv} + \ovl a \mu_{uv'} + b \mu_{u'v} + \mu_{u'v'}
\ .
\end{equation}
Ora, le uguaglianze (\ref{eq_SM2a}) ci dicono che l'applicazione
$u,v \mapsto \int f \ d \mu_{uv}$, $u,v \in \mH$,
\e in realt\'a una forma sesquilineare limitata, e che $f(T) \in B(\mH)$ \e in effetti l'operatore ad essa associato. In particolare, visto che $T = I(T)$, troviamo 
\begin{equation}
\label{eq.diag}
(u,Tv) = \int_{\sigma(T)} \lambda \ d \mu_{uv}(\lambda)
\ \ , \ \
\forall u,v \in \mH
\ ,
\end{equation}
ovvero un analogo non numerabile di (\ref{eq_dec.sp}). Diciamo che $\mu := \{ \mu_{uv} \}$ \e una {\em misura spettrale di $T$}.

Possiamo ora estendere il calcolo funzionale continuo, nel modo che segue. Denotiamo con $B^\infty(\sigma(T),\bC)$ la \sC algebra delle funzioni a valori complessi, boreliane e limitate su $\sigma(T) \subset \bR$. Visto che ogni $\mu_{uv}$ \e una misura boreliana troviamo che 
$B^\infty( \sigma(T) , \bC ) \subseteq L_{\mu_{uv}}^\infty(\sigma(T),\bC)$,
$\forall u,v \in \mH$. 
Per cui, per ogni $f \in B^\infty( \sigma(T) , \bC )$ \e ben definita l'applicazione
\begin{equation}
\label{eq_SM4}
\mH \times \mH \to \bC
\ \ , \ \
u,v \mapsto \int_{\sigma(T)} f \ d \mu_{uv}
\ ,
\end{equation}
la quale, grazie a (\ref{eq_SM3a}) ed all'Esercizio \ref{sec_MIS}.8, \e sequilineare. Inoltre essa \e limitata: per verificare ci\'o \e sufficiente effettuare una stima nel caso $u=v$,
\[
\left| \int_{\sigma(T)} f \ d \mu_{uu} \right|
\ \leq \
\| f \|_\infty  \int_{\sigma(T)} d \mu_{uu}
\ \leq \
\| f \|_\infty \| u \|^2
\ ,
\]
e trattare il caso generale usando l'identit\'a di polarizzazione
$\mu_{uv} = 1/4 \sum_{k=0}^3 i^k \mu_{u+i^kv,u+i^kv}$. 
Denotando con $f(T) \in B(\mH)$ l'operatore associato a (\ref{eq_SM4}) otteniamo un'estensione del calcolo funzionale (\ref{eq_SM2}),
\begin{equation}
\label{eq_SM5}
B^\infty( \sigma(T) , \bC ) \to B(\mH)
\ \ , \ \
f \mapsto f(T)
\ ,
\end{equation}
che chiamiamo il {\em calcolo funzionale boreliano di} $T$. Si verifica che (\ref{eq_SM5}) \e una rappresentazione tale che 
$\| f(T) \| \leq \| f \|_\infty$, $\forall f \in B^\infty( \sigma(T) , \bC )$ 
(vedi \cite[Theorem 4.5.4]{Ped}; osserviamo che comunque (\ref{eq_SM5}) \e isometrica su $C(\sigma(T),\bC) \subset B^\infty( \sigma(T) , \bC )$). 
%
%
%
%
%
Le precedenti considerazioni portano al seguente teorema:
\begin{thm}[Hilbert]
Se $T = T^* \in B(\mH)$ allora esiste una misura spettrale $\mu$ di $T$ tale che \e verificata (\ref{eq.diag}). Si hanno poi le seguenti propriet\'a:
\textbf{(1)} Per ogni $f \in B^\infty( \sigma(T) , \bC )$, l'operatore $f(T) \in B(\mH)$ definito da (\ref{eq_SM5}) \e limitato e $\| f(T) \| \leq \| f \|_\infty$. In particolare $T = I(T)$, dove $I(\lambda) := \lambda$, $\forall \lambda \in \sigma(T)$.
\textbf{(2)} L'operatore (limitato)
\[
B^\infty( \sigma(T) , \bC ) \to B(\mH)
\ \ , \ \
f \mapsto f(T)
\]
\e una rappresentazione. 
\end{thm}

Per una dimostrazione dettagliata del teorema precedente rimandiamo a \cite[\S 4.5]{Ped} o \cite[\S VII.2]{RS}. Qui ci limitiamo ad osservare che, mentre nel caso $T \in K(\mH)$ abbiamo che ogni $\{ \lambda \} \subseteq \sigma(T)$ ha misura non nulla, nel caso generale $T \in B(\mH)$ possiamo trovare $\mu \{ \lambda \} = 0$; in effetti, si pu\'o dimostrare che $\mu \{ \lambda \} \neq 0$ se e solo se $\lambda \in \sigma p(T)$. Per propriet\'a di continuit\'a della rappresentazione del punto (2) rimandiamo all'Esercizio \ref{sec_afunct}.16, il quale include anche il caso degli operatori non limitati (\S \ref{sec_onl}).

\begin{ex}{\it 
Consideriamo lo spazio di Hilbert $\mH := L^2([0,1],\bC)$ e l'operatore $T \in B(\mH)$ definito come nell'Esempio \ref{ex_op.mol}. Gli stessi argomenti fatti in tale esempio mostrano che $T = T^*$ e $\sigma(T) = [0,1]$, $\sigma p(T) = \emptyset$. Denotiamo con $\mM$ la $\sigma$-algebra dei boreliani di $[0,1]$, e per ogni $f,g \in L^2$ definiamo
\[
\mu_{fg}E = \int_E \ovl{f}g
\ \ , \ \
\forall E \in \mM
\ ,
\]
cosicch\'e
\[
(f,Tg) = 
\int_{[0,1]} \ovl{f(s)} s g(s) \ ds =
\int_{[0,1]} s \cdot \ovl{f(s)}g(s) \ ds =
\int_{[0,1]} s \ d \mu_{fg}(s)
\ ,
\]
e $\mu := \{ \mu_{fg} \}$ \e la misura spettrale di $T$ (Osservare che $\mu\{ \lambda \} = 0$, $\forall \lambda \in \sigma (T)$).
Chiaramente, l'analogo argomento vale nel caso di uno spazio di Hilbert reale.
}
\end{ex}

\subsection{Topologie deboli e spettri di algebre di Banach.}
\label{sec_topdeb}

Viene spesso richiesto che successioni in spazi funzionali convergano (si pensi ad esempio alla successione di Peano-Picard). Purtroppo la topologia della norma di uno spazio di Banach risulta essere troppo ricca di aperti per avere buone propriet\'a rispetto alla compattezza, per cui in generale \e impossibile costruire successioni che ammettano sottosuccessioni convergenti.
E' quindi conveniente introdurre topologie meno dotate di aperti rispetto a quella della norma, in modo da migliorare le propriet\'a inerenti la convergenza. 

\

\noindent \textbf{La topologia debole.} Sia $\mE$ uno spazio di Banach. La {\em topologia debole $\sigma(\mE,\mE^*)$}, definita su $\mE$, \e la topologia meno fine che rende continui i funzionali lineari $f \in \mE^*$. Visto che ogni $f \in \mE^*$ \e anche limitato (ovvero, continuo in norma), abbiamo che $\sigma(\mE,\mE^*)$ \e meno fine della topologia della norma di $E$.

Siano $v \neq v' \in E$; usando il teorema di Hahn-Banach estendiamo il funzionale 
$\left \langle f_0 , \lambda(v-v') \right \rangle := \lambda \| v-v' \|$, $\forall \lambda \in \bR$,
ottenendo $f \in \mE^*$ tale che 
$\left \langle f,v \right \rangle \neq \left \langle f,v' \right \rangle$.
Cosicch\'e esiste $\alpha \in \bR$ tale che
$\left \langle f , v \right \rangle 
<
\alpha
<
\left \langle f , v' \right \rangle$,
e gli aperti in $\sigma(\mE,\mE^*)$
\[
A := f^{-1}(-\infty,\alpha)
\ \ , \ \
A' := f^{-1}(\alpha,+\infty)
\]
hanno intersezione vuota e contengono rispettivamente $v$, $v'$. Concludiamo che $( \mE , \sigma(\mE,\mE^*) )$ {\em \e uno spazio di Hausdorff}. Diciamo che una successione  $\{ v_n \} \subset \mE$ \e {\em debolmente} convergente a $v \in \mE$ se essa converge rispetto alla topologia debole, cosa che accade se e solo se
$\left \langle f , v \right \rangle = \lim_n \left \langle f , v_n   \right \rangle$,
$\forall f \in \mE^*$;
in tal caso, scriviamo
\[
v_n \stackrel{\sigma}{\to} v
\ \ {\mathrm{ovvero}} \ \
v = \lim_n^\sigma v_n
\ .
\]
Elenchiamo nel seguito alcune propriet\'a elementari della topologia debole:
\begin{enumerate}
\item $v_n \stackrel{\sigma}{\to} v$ 
      $\Leftrightarrow$
      $\left \langle f , v_n \right \rangle \to \left \langle f , v \right \rangle$,
      $\forall f \in \mE^*$;
\item $v_n \to v$ $\Rightarrow$ $v_n \stackrel{\sigma}{\to} v$;
\item Se $v_n \stackrel{\sigma}{\to} v$, allora $\{ \|v_n \| \}$ \e limitata e 
      $\|v\| \leq \lim \inf_n \|v_n\|$. Ci\'o segue osservando
      che per ogni $f \in \mE^*$ la successione
      $\left\{ | \left \langle f , v_n \right \rangle | \right\}$
      converge a $| \left \langle f , v \right \rangle |$), 
      ed applicando Cor.\ref{cor_BS}.
\item Se $v_n \stackrel{\sigma}{\to} v$ e $f_m \to f$, allora
      $\left \langle f_n , v_n \right \rangle 
       \to 
       \left \langle f , v \right \rangle$. Infatti, si trova
      \[
      | \left \langle f , v \right \rangle - \left \langle f_n , v_n \right \rangle |
      \leq
      | \left \langle f , v - v_n \right \rangle |
      +
      \| f - f_n \| \|v_n \|
      \ .
      \]
\end{enumerate}

\begin{ex}
{\it
Sia $p \in [1,+\infty)$, $q := \ovl p$ e $(X,\mM,\mu)$ uno spazio misurabile $\sigma$-finito. Una successione $\{ f_n \} \subset L_\mu^p(X)$ \e debolmente convergente ad $f \in L_\mu^p(X)$ se e solo se
\[
\int_X ( f_n - f ) g \ \stackrel{n}{\to} \ 0
\ \ , \ \
\forall g \in L_\mu^q(X)
\ .
\]
Ed in tal caso, esiste $M \in \bR$ tale che $\| f_n \|_p \leq M$, $n \in \bN$.
}
\end{ex}

\begin{ex}
{\it
Sia $\mH$ uno spazio di Hilbert e 
$\{ e_k \} \subset \mH_1 := \{ v \in \mH : \| v \| = 1 \}$ 
una successione ortonormale. Usando il teorema di Riesz e Prop.\ref{prop_bessel} troviamo
\[
\left \langle f_u , e_k \right \rangle = (u,e_k) \stackrel{k}{\to} 0
\ \ , \ \
\forall u \in \mH
\ .
\]
Ci\'o implica che $e_k \stackrel{\sigma}{\to} 0$, nonostante il fatto che $\| e_k \| \equiv 1$.
}
\end{ex}

Usando le propriet\'a precedenti e l'esistenza di una base finita, si pu\'o dimostrare in modo molto semplice che se $\mE$ ha dimensione finita, allora la topologia debole e quella della norma coincidono. D'altra parte, se $\mE$ ha dimensione non finita, allora
\begin{itemize}
\item La sfera unitaria $\mE_1 := \{ v \in \mE : \| v \| = 1 \}$ non \e chiusa in $\sigma(\mE,\mE^*)$
\item Il disco unitario $\mE_{< 1} := \{ v \in \mE : \| v \| < 1 \}$ non \e aperto in $\sigma(\mE,\mE^*)$
\end{itemize}
(si veda \cite[\S III.2]{Bre}). Concludiamo questa breve rassegna su $\sigma(\mE,\mE^*)$ richiamando i seguenti due risultati, conseguenze piuttosto semplici del teorema di Hahn-Banach:

\begin{prop}
Un sottoinsieme $C \subseteq \mE$ convesso \e debolmente chiuso se e solo se \e chiuso nella topologia della norma.
\end{prop}

\begin{prop}
Sia $T \in B(\mE,\mF)$. Allora $T$ \e un'applicazione continua da $( \mE , \sigma(\mE,\mE^*) )$ in $( \mF , \sigma(\mF,\mF^*) )$.
\end{prop}

\noindent \textbf{La topologia $*$-debole.} Passiamo ora a considerare un ulteriore tipo di topologia debole, stavolta sul duale $\mE^*$. Consideriamo l'applicazione lineare canonica
\begin{equation}
\label{eq_TD01}
\varphi : \mE \to \mE^{**}
\ \ , \ \
v \mapsto \varphi_v : 
\left \langle \varphi_v , f \right \rangle
=
\left \langle f , v \right \rangle
\ ,
\end{equation}
peraltro isometrica grazie a (\ref{eq_HB2}). La topologia $*$-debole su $\mE^*$ \e la topologia meno fine che rende continua la famiglia di applicazioni $\{ \varphi_v , v \in \mE \}$, e si denota con $\sigma(\mE^*,\mE)$. Poich\'e $\varphi(\mE) \subseteq \mE^{**}$, abbiamo che $\sigma(\mE^*,\mE)$ \e meno fine di $\sigma(\mE^*,\mE^{**})$.
Lo stesso argomento usato per la topologia debole (ma stavolta non \e neanche necessario invocare il teorema di Hahn-Banach) mostra che lo spazio $( \mE^* , \sigma(\mE^*,\mE) )$ {\em \e di Hausdorff}. Una successione $\{ f_n \} \subset \mE^*$ converge ad $f \in \mE^*$ nella topologia $*$-debole se e solo se
$\left \langle f , v \right \rangle = \lim_n \left \langle f_n , v \right \rangle$,
$\forall v \in \mE$;
per indicare la convergenza $*$-debole di $\{ f_n \}$ useremo le notazioni
\[
f_n \stackrel{*}{\to} f
\ \ {\mathrm{ovvero}} \ \
f = \lim_n^* f_n
\ .
\]
In maniera analoga al caso della topologia debole, si dimostrano le seguenti propriet\'a:
\begin{enumerate}
\item $f_n \stackrel{*}{\to} f$ 
      $\Leftrightarrow$
      $\left \langle f_n , v \right \rangle \to \left \langle f , v \right \rangle$,
      $v \in \mE$;
\item $f_n \to f$ $\Rightarrow$ $f_n \stackrel{*}{\to} f$;
\item Se $f_n \stackrel{*}{\to} f$, allora $\{ \| f_n \| \}$ \e limitata e 
      $\| f \| \leq \lim \inf_n \| f_n \|$.
\item Se $f_n \stackrel{*}{\to} f$ e $v_m \to v$, allora
      $\left \langle f_n , v_n \right \rangle 
       \to 
       \left \langle f , v \right \rangle$.
\end{enumerate}

Diamo, senza dimostrazione, la seguente
\begin{prop}\textbf{\cite[Prop.III.13]{Bre}.}
Sia $\varphi : \mE^* \to \bR$ un'applicazione lineare e $*$-debolmente continua. Allora esiste $v \in \mE$ tale che $\varphi = \varphi_v$.
\end{prop}

Una propriet\'a fondamentale (e, in un certo senso, motivante) della topologia $*$-debole \e data dal teorema seguente:
\begin{thm}[Alaoglu]
Sia $\mE$ uno spazio di Banach. Allora, la palla unitaria 
\[
\mE^*_{\leq 1} := \{ f \in \mE^* : \| f \| \leq 1 \}
\]
\'e $*$-debolmente compatta.
\end{thm}

\begin{proof}[Dimostrazione]
Consideriamo lo spazio prodotto $\bR^\mE$ con elementi $\{ \omega_v \in \bR \}_{v \in \mE}$, equipaggiato con le proiezioni (continue) $\pi_v : \bR^\mE \to \bR$, $v \in \mE$, e l'applicazione (manifestatamente iniettiva)
\[
T : \mE^* \to \bR^\mE
\ \ , \ \
T(f) := \{ \left \langle f , v \right \rangle \}_{v \in \mE}
\ .
\]
Vogliamo mostrare che $T : \mE^* \to T(\mE^*)$ \e un omeomorfismo. Ora, si ha $\varphi_v = \pi_v \circ T$, $v \in \mE$; essendo $\varphi_v$ $*$-debolmente continua concludiamo che ogni $\pi_v \circ T$ \e $*$-debolmente continua per cui, per definizione di topologia prodotto, $T$ \e continua. Per dimostrare che $T^{-1} : T(\mE^*) \to \mE^*$ \e continua, \e sufficiente verificare che, per ogni $v \in \mE$, l'applicazione
\[
\varphi_v \circ T^{-1} : T(\mE) \to \bR
\ \ , \ \
T(f) \mapsto \left \langle \varphi_v , f \right \rangle
\]
\'e continua. Ma ci\'o \e evidente per definizione di topologia prodotto, in quanto $\varphi_v \circ T^{-1} = \pi_v |_{T(\mE^*)}$. Ora,
\[
T (\mE^*_{\leq 1}) = K_1 \cap K_2
\ ,
\]
dove
\[
\left\{
\begin{array}{ll}
K_1 := \{ \omega \in \bR^\mE : |\omega_v| \leq \|v\| , v \in \mE  \}
    =  \prod_v \left[ - \|v\| , \|v\| \right]
\\
K_2 := \{ \omega \in \bR^\mE : \omega_{v+aw} = \omega_v + a \omega_w ,
          v,w \in \mE  ,  a \in \bR \}
\end{array}
\right.
\]
E' evidente che $K_1$ \e compatto (Tychonoff). D'altra parte,
\[
K_2 
\ = \
\bigcap_{v,w,a} 
\left( \pi_{v+aw} - \pi_v - a \pi_w \right)^{-1} (0)
\]
dunque (per continuit\'a di $\pi_{v+aw}$, $\pi_v$, $a \pi_w$) esso \e chiuso. Essendo quindi $T(\mE^*_{\leq 1})$ intersezione di un chiuso con un compatto, \e esso stesso compatto. Essendo $T^{-1}$ continua, concludiamo che $\mE^*_{\leq 1}$ \e compatto. 
\end{proof}

\begin{cor}
Sia $\mE$ uno spazio di Banach. Allora la sfera unitaria $\mE^*_1$ \e *-debolmente compatta.
\end{cor}

\begin{proof}[Dimostrazione]
Infatti $\mE^*_1$ \e chiusa ed \e contenuta nel compatto $\mE^*_{\leq 1}$.
\end{proof}

\

\begin{defn}
\label{def_rifl}
Uno spazio di Banach $\mE$ si dice \textbf{riflessivo} se l'iniezione canonica (\ref{eq_TD01}) \e suriettiva, cosicch\'e si ha un isomorfismo $\mE \simeq \mE^{**}$. 
\end{defn}

Il teorema seguente, che forniamo senza dimostrazione, fornisce una caratterizzazione topologica degli spazi riflessivi.

\begin{thm} \textbf{(Kakutani, \cite[Teo.III.16]{Bre}).}
Uno spazio di Banach $\mE$ \e riflessivo se e solo se la palla unitaria $\mE_{\leq 1} := \{ v \in \mE : \| v \| \leq 1 \}$ \e debolmente compatta.
\end{thm}

Parte del teorema di Kakutani \e semplice da dimostrare. Se $\mE$ \e riflessivo allora (\ref{eq_TD01}) si restringe ad un omeomorfismo isometrico $\varphi : \mE_{\leq 1} \to \mE^{**}_{\leq 1}$. Il teorema di Alaoglu implica che $( \mE^{**}_{\leq 1} , \sigma(\mE^{**},\mE^*) )$ \e compatto, per cui \e sufficiente verificare che 
\[
\varphi^{-1} : 
( \mE^{**}_{\leq 1} , \sigma(\mE^{**},\mE^*) ) 
\to
( \mE_{\leq 1} , \sigma(\mE,\mE^*) )
\]
\'e un'applicazione continua; ovvero, per definizione della topologia debole su $\mE$, che per ogni $f \in \mE^*$, l'applicazione
\[
f \circ \varphi^{-1} : \mE^{**}_{\leq 1} \to \bR
\]
sia $*$-debolmente continua. Poich\'e \e immediato verificare che
$f \circ \varphi^{-1} = \varphi^*_f$
(dove $\varphi^*_f \in \mE^{***}$ \e definito secondo l'applicazione canonica $\varphi^* : \mE^* \to \mE^{***}$), concludiamo che $f \circ \varphi^{-1}$ \e $*$-debolmente continua, e quindi $\varphi^{-1}$ \e continua.

\

Passiamo ora ad elencare alcune propriet\'a concernenti la separabilit\'a degli spazi di Banach per la topologia debole:
\begin{enumerate}
\item Sia $\mE$ uno spazio di Banach tale che $\mE^*$ sia separabile. Allora $\mE$ \e separabile (\cite[Teo.III.23]{Bre}). Il viceversa \e falso (ad esempio, si prenda $\mE = L^1$);
\item Un spazio di Banach $\mE$ \e riflessivo e separabile se e solo $\mE^*$ \e riflessivo e separabile (\cite[Cor.III.24]{Bre}, l'esempio canonico \e $\mE = L^p$, $p \neq 1,\infty$);
\item Uno spazio di Banach $\mE$ \e separabile se e solo se $( \mE^*_{\leq 1} , \sigma(\mE^*,\mE) )$ \e metrizzabile (\cite[Teo.III.25]{Bre}). Per dare un'idea della dimostrazione, {\em se assumiamo che $\mE$ sia separabile}, allora possiamo definire la metrica
\[
d(f,g) 
\ := \
\sum_n 2^{-n} | \left \langle f - g , v_n \right \rangle |
\ \ , \ \
f,g \in \mE^*_{\leq 1}
\ ,
\]
dove $\{ v_n \} \subset \mE$ \e denso.
\item \textbf{(Eberlein-Smulian, \cite[Teo.III.27,Teo.III.28]{Bre}).} Uno spazio di Banach $\mE$ \e riflessivo se e solo se $( \mE_{\leq 1} , \sigma(\mE,\mE^*) )$ \e sequenzialmente compatta (ovvero, da ogni successione limitata in $E$ \e possibile estrarre una sottosuccessione debolmente convergente). 
\end{enumerate}

Per quanto concerne gli spazi $L^p$, nel caso in cui $X$ sia separabile abbiamo la seguente situazione:
\begin{itemize}
\item $L_\mu^p(X)$, $p \in (1,\infty)$, \e separabile (\cite[Teo.IV.13]{Bre}), e riflessivo (dualit\'a di Riesz);
\item $L_\mu^1(X)$ \e separabile (\cite[Teo.IV.13]{Bre}), ma {\em non} riflessivo (vedi Prop.\ref{prop_dual_Li});
\item $L_\mu^\infty(X)$ {\em non} \e separabile (a meno che $X$ non sia un insieme numerabile, vedi \cite[Lemma IV.2]{Bre}) e neanche riflessivo (vedi Prop.\ref{prop_dual_Li}). Tuttavia, essendo $L_\mu^\infty(X)$ il duale di $L_\mu^1(X)$ (dualit\'a di Riesz), la palla unitaria di $L_\mu^\infty(X)$ \e compatta rispetto alla topologia *-debole (Teorema di Alaoglu).
\end{itemize}

\noindent \textbf{Lo spettro di una $*$-algebra di Banach ed il teorema di Gel'fand-Naimark.} 
Il teorema di Alaoglu si dimostra senza alcuna variazione nel caso di spazi di Banach complessi, ed un'importante applicazione di questo risultato \e la possibilit\'a di rappresentare, in modo pi\'u o meno accurato a seconda dei casi, una generica $*$-algebra di Banach commutativa come un'algebra di funzioni continue. E' questo l'oggetto del risultato noto come {\em il teorema di Gel'fand-Naimark}.

Sia $\mA$ una $*$-algebra di Banach commutativa. Denotiamo con $\wa \mA$ l'insieme dei {\em caratteri di $\mA$}, ovvero di quegli $\omega \in \mA^*$ {\em non nulli} tali che
\begin{equation}
\label{def_char}
\left\{
\begin{array}{ll}
\left \langle \omega , aa' \right \rangle 
\ = \
\left \langle \omega , a \right \rangle  \left \langle \omega , a' \right \rangle \ ,
\\
\left \langle \omega , a^* \right \rangle = \ovl{ \left \langle \omega , a \right \rangle  }
\ \ , \ \
\forall a,a' \in \mA \ ;
\end{array}
\right.
\end{equation}
osserviamo che $\wa \mA$ \e naturalmente equipaggiato della topologia *-debole, rispetto alla quale esso \e di Hausdorff.

\begin{defn}
\label{def_GN}
Lo spazio di Hausdorff $\wa \mA$ si dice \textbf{lo spettro di $\mA$}. 
\end{defn}

\noindent Introduciamo ora la \textbf{trasformata di Gel'fand}
\begin{equation}
\label{def_GT}
\mA \to C(\wa \mA,\bC)
\ \ , \ \
a \mapsto \wa a \ : \ \wa a (\omega) := \left \langle \omega , a \right \rangle 
\ \ , \ \
\forall \omega \in \wa \mA
\ .
\end{equation}
Le funzioni $\wa a$, $a \in \mA$, sono effettivamente continue proprio per definizione di topologia $*$-debole (infatti $\wa a = \varphi_a$, nella notazione del paragrafo precedente), per cui (\ref{def_GT}) \e ben definita come applicazione. Osserviamo che per costruzione (\ref{def_GT}) \e lineare e tale che
\begin{equation}
\label{rel_GN}
\wa{aa'} = \wa a \wa a'
\ \ , \ \
\wa{a^* \ }(\omega) = \ovl{\wa a (\omega)}
\ \ , \ \
\forall a , a' \in \mA 
\ , \
\omega \in \wa \mA \ .
\end{equation}
Supponiamo ora che $\mA$ abbia identit\'a $1 \in \mA$ e poniamo
$\lambda :=  \left \langle \omega , 1 \right \rangle$;
allora troviamo 
$\left \langle \omega , a \right \rangle = \lambda \left \langle \omega , a \right \rangle$
per ogni $a \in \mA$ e quindi, scegliendo $a=1$, $\lambda^2 = \lambda$; essendo $\omega$ non nullo troviamo $\lambda \neq 0$, e concludiamo che deve essere
$\lambda = \left \langle \omega , 1 \right \rangle = 1$
per ogni $\omega \in \wa \mA$. Nel risultato seguente stabiliamo una connessione tra $\wa \mA$ e gli spettri degli elementi di $\mA$.

\begin{lem}
\label{lem_GN}
Sia $\mA$ una $*$-algebra di Banach commutativa e con identit\'a $1 \in \mA$. Per ogni $a \in \mA$ ed $\omega \in \wa \mA$ si ha $\left \langle \omega , a \right \rangle \in \sigma (a)$, cosicch\'e $| \left \langle \omega , a \right \rangle | \leq \| a \|$.
\end{lem}

\begin{proof}[Dimostrazione]
Iniziamo dimostrando che $T_\omega := a - \left \langle \omega , a \right \rangle 1$ non \e invertibile. A tale scopo, osserviamo che (avendosi $\left \langle \omega , 1 \right \rangle = 1$) troviamo
$\left \langle \omega , T_\omega \right \rangle = 0$,
per cui $T_\omega$ appartiene al nucleo di $\omega$, che denotiamo con $\ker \omega$. Ora, $\ker \omega$ \e chiuso rispetto a prodotti con generici elementi di $\mA$ 
{\footnote{
Ovvero, $\ker \omega$ \e un ideale di $\mA$.
}},
infatti
$\left \langle \omega , Ta \right \rangle = 0 \cdot \left \langle \omega , a \right \rangle = 0$
per ogni $T \in \ker \omega$ ed $a \in \mA$. Osserviamo che $\ker \omega$ non pu\'o avere elementi invertibili, altrimenti avremmo la contraddizione
$1 = \left \langle \omega , TT^{-1} \right \rangle = 0 \cdot \left \langle \omega , T^{-1} \right \rangle = 0$,
$T \in \ker \omega$.
Concludiamo che $T_\omega$ non \e invertibile e $\left \langle \omega , a \right \rangle \in \sigma (a)$. Il fatto che 
$| \left \langle \omega , a \right \rangle | \leq \| a \|$
segue da Prop.\ref{prop_spettro}.
\end{proof}

In conseguenza del Lemma precedente troviamo $\wa \mA \subseteq \mA^*_{\leq 1}$; inoltre \e immediato verificare che $\wa \mA$ \e chiuso (il limite nella topologia $*$-debole di una successione di caratteri \e un carattere), per cui il teorema di Alaoglu implica che $( \wa \mA , \sigma(\mA^*,\mA) )$ \e compatto, essendo $\wa \mA$ un chiuso contenuto in un compatto. Dunque, ogni $\wa a \in C(\wa \mA,\bC)$, $a \in \mA$, ha norma
$\| \wa a \|_\infty := \sup_{\omega \in \wa \mA} | \left \langle \omega , a \right \rangle | \leq \| a \|$.

\begin{prop}[Gel'fand-Naimark]
Sia $\mA$ una $*$-algebra di Banach commutativa e con identit\'a $1$. Allora la trasformata di Gel'fand (\ref{def_GT}) \e un operatore limitato ed ha immagine densa in $C(\wa \mA,\bC)$.
\end{prop}

\begin{proof}[Dimostrazione]
La linearit\'a di (\ref{def_GN}) \e ovvia e la limitatezza segue dal Lemma \ref{lem_GN}, il quale implica che $\| \wa a \|_\infty \leq \| a \|$, $a \in \mA$.
Denotiamo ora con $\mB \subseteq C(\wa \mA,\bC)$ l'immagine di $\mA$ attraverso la trasformata di Gel'fand;
per mostrare la densit\'a di (\ref{def_GN}) l'idea \e quella di utilizzare il teorema di Stone-Weierstrass (Teo.\ref{thm_sw}), per cui occorre verificare che $\mB$ \e chiusa rispetto al passaggio all'aggiunto in $C(\wa \mA,\bC)$ (e ci\'o \e ovvio da (\ref{rel_GN})), che $\mB$ contiene le funzioni costanti $z \in C(\wa \mA,\bC)$, $z \in \bC$ (e ci\'o \e ovvio in quanto $\wa{z1} (\omega) = z \left \langle \omega , 1 \right \rangle = z$, $\forall \omega \in \wa \mA$), e che $\mB$ separa i punti di $\wa \mA$ (e anche questo \e ovvio, visto che $\omega \neq \omega'$ se e solo se $\left \langle \omega , a \right \rangle \neq \left \langle \omega' , a \right \rangle$ per qualche $a \in \mA$, ovvero $\wa a (\omega) \neq \wa a (\omega')$). Dunque concludiamo che $\mB$ \e densa in $C(\wa \mA,\bC)$ come desiderato.
\end{proof}

Ora, peculiarit\'a delle \sC algebre commutative \e che 
$\| a \| = \sup_{\omega \in \wa \mA} | \left \langle \omega , a \right \rangle |$
per ogni $a \in \mA$ (vedi \cite[\S 4.1.10 e Lemma 4.3.11]{Ped}). In altri termini la trasformata di Gel'fand \e isometrica, e quindi chiusa come applicazione da $\mA$ in $C(\wa \mA,\bC)$. Di conseguenza, l'immagine di $\mA$ attraverso (\ref{def_GT}) \e sia chiusa che densa in $C(\wa \mA,\bC)$, e quindi coincide con $C(\wa \mA,\bC)$. Dunque abbiamo mostrato il seguente risultato:

\begin{thm}\textbf{(Gel'fand-Naimark, \cite[Theorem 4.3.13]{Ped}).}
\label{thm_GN}
Sia $\mA$ una \sC algebra commutativa con identit\'a $1$. Allora la trasformata di Gel'fand (\ref{def_GT}) \e iniettiva e suriettiva.
\end{thm}

\begin{ex}{\it
\textbf{(1)} Sia $X$ compatto e di Hausdorff, $\mA := C(X,\bC)$; allora $\wa \mA$ \e omeomorfo ad $X$ (vedi Esercizio \ref{sec_afunct}.9); 
\textbf{(2)} Sia $\mA := L^1(\bR,\bC) + \bC \delta_0$, dove $L^1(\bR,\bC)$ \e equipaggiato del prodotto di convoluzione e $\delta_0$ \e la delta di Dirac (qui abbiamo immerso $L^1(\bR,\bC)$ in $\Lambda_\beta^1(\bR,\bC)$, vedi Esercizio \ref{sec_func_n}.7); allora si ha un omeomorfismo $\theta : \wa \mA \to \bT$, e la trasformata di Gel'fand coincide, in essenza, con la trasformata di Fourier (vedi Esercizio \ref{sec_afunct}.10 e \cite[\S 4.2.8]{Ped}).
\textbf{(3)} Sia $T = T^* \in B(\mH)$ ed $\mA \subseteq B(\mH)$ l'immagine di $C(\sigma(T),\bC)$ rispetto al calcolo funzionale continuo (\ref{eq_SM2}); allora si ha un omeomorfismo $\wa \mA \simeq \sigma(T)$ (vedi \cite[Prop.4.3.15]{Ped}).
}
\end{ex}

\subsection{Cenni su spazi localmente convessi e distribuzioni.}
\label{sec_loc_conv}

Sia $V$ uno spazio vettoriale. Diciamo che una famiglia $\{ p_i \}_{i \in I}$ di seminorme su $V$ \e {\em separante} se 
\[
p_i(v) = 0 \ \ , \ \ \forall i \in I
\ \Rightarrow \
v = 0
\ .
\]
Diciamo che $V$ \e {\em localmente convesso} se ammette una famiglia separante di seminorme. La {\em topologia naturale} di $V$, che denotiamo con $\sigma V$, \e la pi\'u debole topologia che rende continue le applicazioni $p_i$ al variare di $i$ in $I$. Chiaramente, $\sigma V$ \e di Hausdorff; per definizione, una successione $\{ v_n \} \subset V$ converge a $v \in V$ nella topologia $\sigma V$ se e soltanto se $p_i(v_n) \stackrel{n}{\to} p_i(v)$ per ogni $i \in I$.

\begin{ex}
\label{ex_lc}
{\it
\textbf{(1)} Ogni spazio di Banach \e localmente convesso (in quanto la norma separa i punti);
\textbf{(2)} Gli spazi $L_{\mu,loc}^p(X)$ sono localmente convessi (vedi Cor.\ref{cor_Lploc_LC}); 
\textbf{(3)} Sia $V$ uno spazio vettoriale e $\{ f_i \in V^* \}$ una famiglia di funzionali lineari su $V$; allora abbiamo la famiglia di seminorme 
$
P := \{ p_i : p_i(v) := |f_i(v)| , \forall v \in V \}
$.
Se $P$ \e separante allora $V$ viene equipaggiato di una struttura di spazio localmente convesso. Ad esempio, se $\mE$ \e uno spazio di Banach allora la famiglia di seminorme
\[
\mP := 
\{ 
p_f : p_f(v) := | \left \langle f,v \right \rangle | , \forall v \in \mE 
\}_{f \in \mE^*}
\]
\e separante ed induce la topologia debole. Analogamente, la famiglia
\[
\mP^* 
:= 
\{ 
p_v : p_v(f) := | \left \langle f,v \right \rangle | , \forall f \in \mE^* 
\}_{v \in \mE}
\]
\e separante ed induce la topologia $*$-debole su $\mE^*$.
\textbf{(4)} Dato uno spazio di Hilbert complesso $\mH$, la \sC algebra $B(\mH)$ \e equipaggiata con la famiglia di seminorme 
\[
\mD := \{ p_{u,v}(T) := |(u,Tv)| , T \in B(\mH) \}_{u,v \in \mH} \ ,
\]
la quale \e separante in quanto per ogni $T \in B(\mH)$ risulta
\[
p_{Tu,u}(T) = | (Tu,Tu) | = \| Tu \|^2 
\ \ , \ \
\forall u \in \mH
\ .
\]
La topologia naturale associata a $\mD$ \e nota con il nome di \textbf{topologia debole di $B(\mH)$} ed \e molto usata nella teoria delle algebre di operatori. Una $*$-algebra $\mA \subseteq B(\mH)$ contenente l'identit\'a e chiusa rispetto alla topologia debole si dice \textbf{algebra di von Neumann} (vedi \cite[\S 4.6]{Ped}).
\textbf{(5)} Sia $\mE$ uno spazio di Banach; l'algebra di Banach $B(\mE)$ \e equipaggiata con la famiglia separante di seminorme 
\[
\mS := \{ p_v(T) := \| Tu \| , T \in B(\mE) \}_{v \in \mE} \ .
\]
La relativa topologia naturale \e nota con il nome di \textbf{topologia forte di $B(\mE)$} (vedi \cite[\S 4.6]{Ped}). 
In particolare, preso uno spazio di Hilbert $\mH$ ed una successione $\{ T_n \} \subset B(\mH)$, le disuguaglianze
\[
| (u, \{ T_n - T_m \} v) | \ \leq \ 
\| u \| \| \{ T_n - T_m \} v \| \ \leq \ 
\| u \| \| T_n - T_m \| \| v \|
\ \ , \ \
\forall u,v \in \mH
\ ,
\]
mostrano che la convergenza in norma di $\{ T_n \}$ implica quella in topologia forte, la quale a sua volta implica la convergenza nella topologia debole del punto (4). Dunque, ogni algebra di von Neumann $\mA \subseteq B(\mH)$ \e chiusa sia rispetto alla topologia forte che quella della norma (dunque $\mA$ \e anche una \sC algebra).
}
\end{ex}

Un sottoinsieme $C$ di uno spazio vettoriale $V$ si dice: (1) {\em convesso}, se per ogni $x,y \in C$, $t \in [0,1]$ risulta $tx+(1-t)y \in C$; (2) {\em bilanciato}, se per ogni $x \in C$ e $\lambda$ di modulo $1${\footnote{Qui potremmo avere $\lambda \in \bR$ o $\lambda \in \bC$ a seconda del caso in cui $V$ \e reale o complesso.}} risulta $\lambda x \in C$; (3) {\em assorbente}, se $\cup_{t \geq 0} tC = V$.
%

Uno spazio vettoriale $V$ si dice {\em topologico} se su esso \e definita una topologia $\tau$ tale che le operazioni di somma e moltiplicazione scalare sono applicazioni continue
\[
V \times V \to \bK
\ \ , \ \
\bK \times V \to V
\ ,
\]
dove $\bK = \bR,\bC$ a seconda del caso reale o complesso. Ad esempio, ogni spazio localmente convesso \e uno spazio vettoriale topologico rispetto alla topologia naturale (semplice esercizio!). Il seguente risultato caratterizza gli spazi vettoriali topologici che sono localmente convessi.

\begin{thm}
Sia $(V,\tau)$ uno spazio vettoriale topologico di Hausdorff. Allora $V$ \e localmente convesso (ovvero, $\tau$ \e indotta da una famiglia separante di seminorme) se e solo se $0 \in V$ ammette una base di intorni convessi, bilanciati ed assorbenti.
\end{thm}

\begin{proof}[Sketch della dimostrazione]
Sia $V$ localmente convesso. E' immediato verificare che se $p_1 , \ldots , p_n$ sono seminorme su $V$ allora ogni
\[
C_{\eps ; p_1 , \ldots , p_n} := \{ x \in V : |p_k(x)| < \eps \ , \ k = 1 , \ldots , n   \}
\ \ , \ \
\eps > 0
\ ,
\]
\'e convesso, bilanciato ed assorbente. 
Viceversa, se $(V,\tau)$ \e tale che $0$ ammette una base $\mC := \{ C_i  \}$ di intorni convessi, bilanciati e assorbenti, allora introducendo i {\em funzionali di Minkowski}
\begin{equation}
\label{eq_FMIN}
p_i (x) := \inf \{ t \in \bR : x \in tC_i  \}
\ \ , \ \
x \in V
\ ,
\end{equation}
otteniamo una famiglia separante di seminorme.
\end{proof}

Il teorema di Hahn-Banach si applica agli spazi localmente convessi, per cui questi sono ricchi di funzionali lineari continui rispetto alla topologia naturale. Lo spazio vettoriale di tali funzionali si denota con $V^*$, ed \e chiamato il {\em duale topologico} di $V$; si pu\'o dimostrare che $V^*$ separa i punti di $V$ (si veda \cite[\S V.1]{RS}). Ora, per ogni $v \in V$ definiamo il funzionale
\[
\varphi_v : V^* \to \bR 
\ \ , \ \
\left \langle \varphi_v , f  \right \rangle := \left \langle f , v  \right \rangle
\ \ , \ \
f \in V^*
\ .
\]
Equipaggiamo quindi $V^*$ della topologia debole $\sigma (V^*,V)$ che rende continua la famiglia di applicazioni $\{ \varphi_v , v \in V  \}$.

Dati due spazi localmente convessi $V$ e $V'$ con famiglie di seminorme $\{ p_i \}$, $\{ q_j \}$, possiamo considerare operatori lineari $T : V \to V'$ continui rispetto alle rispettive topologie naturali. Si verifica facilmente (vedi \cite[Teo.V.2]{RS}) che un operatore lineare $T : V \to V'$ \e continuo se e solo se per ogni seminorma $q_j$ esistono seminorme $p_1 , \ldots , p_n$ su $V$ e $c > 0$ tali che
\begin{equation}
\label{eq_CSLC}
q_j(Tv) \leq c \sum_{k=1}^n p_k(v) \ \ , \ \ v \in V \ .
\end{equation}

Ora, alcuni semplici argomenti topologici mostrano che, dato il nostro spazio localmente convesso $V$, le seguenti condizioni sono equivalenti:
(1) $V$ \e metrizzabile;
(2) $0 \in V$ ammette una base numerabile di intorni;
(3) $\sigma V$ \e indotta da una famiglia numerabile di seminorme.
Per dare un'idea della dimostrazione, se assumiamo che $\sigma V$ sia indotta da una famiglia numerabile $\{ p_n  \}$ allora possiamo definire la metrica
\begin{equation}
\label{eq_MLC}
d(v,w) := \sum_n 2^{-n} \ \frac{ p_n(v-w)  }{ 1+p_n(v-w)  }
\ \ , \ \
v,w \in V
\ .
\end{equation}

\begin{defn}
Si dice \textbf{spazio di Fr\'echet} uno spazio localmente convesso, completo e metrizzabile.
\end{defn}

\noindent \textbf{Funzioni a decrescenza rapida e distribuzioni temperate.} Sia $n \in \bN$ e $\bZ^{+,n}$ l'insieme delle $n$-ple di interi non negativi. Definiamo
\[
| \alpha | := \sum_{k=1}^n \alpha_k
\ , \
\alpha := ( \alpha_1 , \ldots , \alpha_n ) \in \bZ^{+,n}
\ \ , \ \
x^\alpha := \prod_{k=1}^n x_k^{\alpha_k} 
\ , \
x \in \bR^n
\ ,
\]
e l'operatore
\[
D^\alpha := \frac{ \partial^{|\alpha|}  }{ \partial x_1^{\alpha_1} \ldots \partial x_n^{\alpha_n} }
\ \ , \ \
\alpha \in \bZ^{+,n}
\ .
\]
Una funzione $f \in C^\infty(\bR^n)$ si dice {\em a decrescenza rapida} se per ogni $\alpha,\beta \in \bZ^{+,n}$ risulta
\begin{equation}
\label{eq_DR}
\| f  \|_{\alpha.\beta} := \sup_x | x^\alpha D^\beta f(x) | < \infty \ .
\end{equation}
Denotiamo con $\mS(\bR^n)$ lo spazio delle funzioni a decrescenza rapida. Osserviamo che (\ref{eq_DR}) ci dice che $f \in \mS(\bR^n)$, ed ogni sua derivata, decrescono pi\'u rapidamente dell'inverso di un qualsiasi polinomio. E' chiaro che $\{ \| \cdot \|_{\alpha,\beta}  \}$ \e una famiglia numerabile di seminorme, cosicch\'e $\mS(\bR^n)$ ha una struttura di spazio localmente convesso metrizzabile (vedi (\ref{eq_MLC})). Usando gli usuali teoremi di uniforme convergenza dimostriamo facilmente che $\mS(\bR^n)$ \e completo, per cui esso \e uno spazio di Fr\'echet. Il duale topologico $\mS^*(\bR^n)$ \e chiamato lo {\em spazio delle distribuzioni temperate}.

\begin{ex}
{\it
Per ogni $x \in \bR^n$ consideriamo la \textbf{delta di Dirac} $\delta_x \in \mS^*(\bR^n)$, $\left \langle \delta_x , f \right \rangle := f(x)$. Poich\'e $\| \left \langle \delta_x , f \right \rangle  \| \leq \| f \|_{0,0}$ concludiamo che $\delta_x \in \mS^*(\bR^n)$. 
Definiamo ora $\delta'_x : \mS(\bR) \to \bR$, $\left \langle \delta_x , f \right \rangle := f'(x)$. Poich\'e $\| \left \langle \delta'_x , f \right \rangle  \| \leq \| f \|_{0,1}$ concludiamo che $\delta'_x \in \mS^*(\bR)$; tuttavia, a differenza della delta di Dirac, $\delta'_x$ non \e associata a nessuna misura su $\bR$.
}
\end{ex}

\begin{ex}
{\it
Definiamo
\begin{equation}
\label{eq_PCauchy}
\left \langle \mP_{1/x} , f \right \rangle
\ := \
\lim_{\eps \to 0^+} \int_{|x| \geq \eps} \frac{1}{x} f(x) \ dx
\ = \
\lim_{\eps \to 0^+} \int_\eps^\infty \frac{1}{x} \left( f(x) - f(-x) \right) \ dx
\ \ , \ \
\forall f \in \mS(\bR)
\ .
\end{equation}
Poich\'e $\lim_{x \to 0} x^{-1}(f(x)-f(-x)) = 2 f'(0)$ troviamo che l'integrale improprio nell'espressione precedente converge a 
\[
\int_0^\infty \frac{1}{x} \left( f(x) - f(-x) \right) \ dx 
\ < \
\infty
\ .
\]
Dunque $\mP_{1/x}$ \e ben definita come applicazione lineare. Per mostrarne la continuit\'a, osserviamo che
$| x^{-1} ( f(x) - f(-x) ) | \leq 2 \| f' \|_\infty$
per cui, spezzando (\ref{eq_PCauchy}) tra $[0,1]$ ed $(1,\infty)$ troviamo
\[
| \left \langle \mP_{1/x} , f \right \rangle |
\ \leq \
2 \int_0^1 \| f' \|_\infty \ dx
\ + \
\int_1^\infty \left| xf(x) \right| \frac{1}{x^2} \ dx
\ \leq \
2 \| f \|_{0,1} + \| f  \|_{1,0}
\ .
\]
Dunque $\mP_{1/x} \in \mS^*(\bR)$, ed \e noto come \textbf{la parte principale di Cauchy}.
}
\end{ex}

Da (\ref{eq_DR}) segue che $\mS(\bR^n) \subset L^p(\bR^n)$ per ogni $p \in [1,+\infty]$, dunque abbiamo un'applicazione canonica
\[
i_p : \mS(\bR^n) \to L^p(\bR^n)
\ .
\]
Si verifica che $i_p$ \e continua. Ad esempio, nel caso $n=1$, $p=1$ abbiamo
\begin{equation}
\label{eq_l1_s}
\| i_1 f \|_1 
\ = \ 
\| f \|_1 
\ = \ 
\int_\bR \frac{1}{1+x^2} \ \{ (1+x^2)|f(x)|  \} \ dx  
\ \leq \ 
\pi \{  \| f \|_\infty + \| x^2 f  \|_\infty  \}
\ ,
\end{equation}
per cui la continuit\'a di $i_1$ segue da (\ref{eq_CSLC}). Per $p$ qualsiasi, poniamo $q := \ovl p$ e troviamo
$\| f \|_p  \leq  \| |f|^{p/q} |f| \|_1^{1/p}  \leq  c \| f \|_{0,0}^{1/p}$, dove $c := \| |f|^{1/q} \|_p < \infty$. I casi $n > 1$ si trattano in maniera analoga.

\

Per dualit\'a di Riesz ogni $g \in L^q(\bR^n)$ definisce il funzionale continuo $F_g \in L^{p,*}(\bR^n)$ (vedi (\ref{def_Fg})); per cui $F_g \circ i_p : \mS(\bR^n) \to \bR$ \e un funzionale continuo rispetto alla topologia naturale di $\mS(\bR^n)$. Dunque, abbiamo un'immersione
\begin{equation}
\label{eq_Lq_dS}
L^q(\bR^n) \to \mS^*(\bR^n)
\ \ , \ \
g \mapsto F_g \circ i_p
\ : \
\left \langle F_g \circ i_p , f \right \rangle = \int fg
\ , \
\forall f \in \mS(\bR^n)
\ .
\end{equation}
Equipaggiando $\mS^*(\bR^n)$ della topologia debole $\sigma ( \mS^*(\bR^n) , S (\bR^n) )$, ed usando la continuit\'a di $i_p$, concludiamo che (\ref{eq_Lq_dS}) \e continua.
Restringendo (\ref{eq_Lq_dS}) a funzioni del tipo $i_q g \in L^q(\bR^n)$, $g \in \mS(\bR^n)$, otteniamo l'applicazione continua
\begin{equation}
\label{eq_emb_ss}
I : \mS(\bR^n) \to \mS^*(\bR^n) 
\ \ , \ \  
g \mapsto I(g)
\ : \
\left \langle I(g),f \right \rangle = \int fg
\ , \
\forall f \in \mS(\bR^n)
\ .
\end{equation}
Usando i polinomi di Hermite si dimostra che $I(\mS(\bR^n))$ \e denso in $\mS^*(\bR^n)$ (\cite[Theorem V.14]{RS}): in termini espliciti, per ogni $\varphi \in \mS^*(\bR^n)$ esiste una successione $\{ g_n \} \subset \mS(\bR^n)$ tale che
\begin{equation}
\label{eq_appr_dis}
\left \langle \varphi , f \right \rangle 
\ = \
\lim_n \int g_n(x) f(x) \ dx
\ \ , \ \
\forall f \in \mS(\bR^n)
\ .
\end{equation} 
La densit\'a in $\mS^*(\bR^n)$ dell'immagine di (\ref{eq_emb_ss}) ha un'importante conseguenza: data un'applicazione continua $T : \mS(\bR^n) \to \mS(\bR^n)$ osserviamo che $I \circ T : \mS(\bR^n) \to \mS^*(\bR^n)$ \e continua, per cui estendendo per continuit\'a possiamo definire
\[
T^* : \mS^*(\bR^n) \to \mS^*(\bR^n)
\ \ : \ \
T^* \circ I := I \circ T
\ .
\]
Chiamiamo $T^*$ {\em l'applicazione aggiunta} di $T$.

\begin{ex}[La derivata debole]
\label{ex_DB}
{\it
Per ogni $\alpha \in \bZ^{+,n}$, l'operatore di derivazione definisce l'applicazione continua
\[
D^\alpha : \mS(\bR^n) \to \mS(\bR^n)
\ \ , \ \
\| D^\alpha f  \|_{\gamma,\delta} \leq \| f  \|_{\gamma,\delta+\alpha}
\ .
\]
L'operatore aggiunto $D^{\alpha,*} : \mS^*(\bR^n) \to \mS^*(\bR^n)$ \e noto come la \textbf{derivata debole}. Ora, per ogni $f,g \in \mS(\bR^n)$ troviamo
\[
\left \langle D^{\alpha,*} \circ I(f) , g \right \rangle
\ = \
\int D^\alpha f (x) g(x) \ dx
\ \stackrel{per \ parti}{=} \
(-1)^{|\alpha|} \int f (x) D^\alpha g(x) \ dx
\ = \
\left \langle  I(f) , (-1)^{|\alpha|} D^\alpha g \right \rangle
\ .
\]
Dunque se $\varphi \in \mS^*(\bR^n)$ per densit\'a troviamo
\[
\left \langle D^{\alpha,*} (\varphi) , g \right \rangle
\ = \
\left \langle \varphi , (-1)^{|\alpha|} D^\alpha g \right \rangle
\ \ , \ \
g \in \mS(\bR^n)
\ .
\]
}
\end{ex}

Diciamo che una distribuzione temperata $\varphi \in \mS^*(\bR^n)$ si annulla in un aperto $\Omega \subseteq \bR^n$ qualora $\left \langle \varphi , f \right \rangle = 0$ per ogni $f \in \mS(\bR^n)$ tale che ${\mathrm{supp}}(f) \subseteq \Omega$. Il {\em supporto di $\varphi$}, che denotiamo con ${\mathrm{supp}}(\varphi)$, si definisce come il complemento del pi\'u grande insieme aperto sul quale $\varphi$ si annulla. Si dimostra che
\begin{equation}
\label{eq_supp0}
{\mathrm{supp}}(\varphi) = \{ 0 \} 
\ \Rightarrow \
\varphi = \sum_{\alpha \in I} c_\alpha D^{\alpha,*} \circ \delta_0
\end{equation}
dove $I \subset \bZ^{+,n}$ \e un insieme finito e $c_\alpha \in \bR$ (vedi \cite[Theorem V.11]{RS}).

\begin{ex}[Rinormalizzazione]
{\it
Consideriamo la funzione $x^{-1}_+ (t) := \chi_{(0,+\infty)}(t) \cdot t^{-1}$, $t \in \bR$. Poich\'e $\int_0^1 x^{-1}_+(t) \ dt = + \infty$ concludiamo che $x^{-1}_+$ \textbf{non} definisce una distribuzione su $\mS(\bR)$.
Tuttavia, se consideriamo il sottospazio $\mS_0(\bR)$ di $\mS(\bR)$ delle funzioni a decrescenza rapida che si annullano nell'origine, troviamo che il funzionale definito dall'integrale improprio
\[
\left \langle \mP_{1/x}^+ , f  \right \rangle 
\ := \
\int_0^{+\infty} \frac{1}{t} f(t) \ dt
\ \ , \ \
f \in \mS_0(\bR)
\ ,
\]
\'e ben definito e continuo. Applicando il teorema di Hahn-Banach, concludiamo che esistono estensioni ad $\mS(\bR)$ di $\mP_{1/x}^+$, note come \textbf{rinormalizzazioni}. Prese due rinormalizzazioni $\varphi_1 , \varphi_2$, da (\ref{eq_supp0}) concludiamo
$\varphi_1 - \varphi_2 = \sum_\alpha c_\alpha D^{\alpha,*} \circ \delta_0$.
In effetti, si dimostra facilmente (\cite[Ex.V.3.9]{RS}) che tutte e sole le rinormalizzazioni di $\mP_{1/x}^+$ sono quelle del tipo
\[
\varphi_c \in \mS^*(\bR)
\ : \
\left \langle \varphi_c , f \right \rangle
\ := \
\int_0^c \frac{ f(x)-f(0) }{x} \ dx \ + \ \int_c^\infty \frac{ f(x)  }{x} \ dx
\ \ , \ \
c > 0
\ ,
\]
il che implica $\varphi_c - \varphi_{c'} = -\log \frac{c}{c'} \ \delta_0$.
}
\end{ex}

\

\noindent \textbf{La topologia di Fr\'echet su $C_c^\infty(\Omega)$, distribuzioni e soluzioni deboli.} Sia $\Omega \subseteq \bR^n$ un aperto connesso e $C_c^\infty(\Omega)$ lo spazio delle funzioni $C^\infty$ aventi supporto compatto e contenuto in $\Omega$. Abbiamo gi\'a mostrato che $C_c^\infty(\Omega)$ non \e uno spazio di Banach (vedi Esercizio \ref{sec_top}.7), tuttavia \e possibile verificare che esso \e uno spazio di Fr\'echet rispetto alla famiglia di seminorme
\begin{equation}
\label{eq.snD}
\| D^\alpha f \|_\infty
\ \ , \ \
\alpha \in \bZ^{+,n}
\ \ , \ \
f \in C_c^\infty(\Omega)
\ .
\end{equation}
La dimostrazione del fatto che (\ref{eq.snD}) induce una topologia di Fr\'echet \e non banale e coinvolge la nozione di limite induttivo rispetto alla rappresentazione di $C_c^\infty(\Omega)$ come l'unione
\[
C_c^\infty(\Omega) \ = \  \bigcup_{K \subset \Omega \ : \ K \ compatto } C_0^\infty(K)
\]
(vedi \cite[Problem V.46]{RS}). Qui diamo per buono tale risultato e denotiamo lo spazio di Fr\'echet cos\'i ottenuto con $\mD(\Omega)$. Per definizione, un funzionale lineare $\varphi : \mD(\Omega) \to \bR$ \e continuo se e soltanto se per ogni compatto $K \subset \Omega$ esistono $c_K > 0$, $m_K \in \bZ^{+,n}$ tali che
\begin{equation}
\label{eq_DISTR}
| \left \langle \varphi , f \right \rangle |
\ \leq \
c_K \sum_{|\alpha| \leq m_K} \| D^\alpha f  \|_\infty
\ \ , \ \
\forall f \in C_0^\infty(K) \subset \mD(\Omega)
\ .
\end{equation}
Lo spazio dei funzionali lineari continui su $\mD(\Omega)$, che denotiamo con $\mD^*(\Omega)$, \e detto lo {\em spazio delle distribuzioni}, e viene equipaggiato della topologia debole indotta da $\mD(\Omega)$.

\begin{ex}
{\it
Per ogni $x \in \bR$ ed $\alpha \in \bZ^{+,n}$, abbiamo, ovviamente, che 
$\left \langle \delta_x^\alpha , f \right \rangle := D^\alpha f(x)$, $f \in \mD(\Omega)$, 
\e una distribuzione.
}
\end{ex}

\begin{ex}[Ancora sulla derivata debole]
\label{ex_DD_D}
{\it
Per ogni $g \in C(\bR^n)$, $\alpha \in \bZ^{+,n}$, \e ben definita l'applicazione lineare
\[
\left \langle D^{\alpha,*} g , f  \right \rangle 
\ := \
(-1)^{|\alpha|} \int g (x) D^\alpha f (x) \ dx
\ \ , \ \
f \in \mD(\Omega)
\ ,
\]
la quale \e continua in quanto, per ogni compatto $K \subset \Omega$,
\[
| \left \langle D^{\alpha,*} g , f  \right \rangle |
\ \leq \
{\mathrm{vol}}(K) \ \| g|_K \|_\infty \ \| D^\alpha f  \|_\infty
\ \ , \ \
\forall f \in C_0^\infty(K)
\ .
\]
}
\end{ex}

\begin{ex}
\label{ex_D_L1loc}
{\it
Per ogni $g \in L_{loc}^1(\bR^n)$, \e ben definita la distribuzione
\[
\left \langle F_g , f  \right \rangle \ := \ \int fg
\ \ , \ \
f \in \mD(\Omega)
\ \ : \ \
| \left \langle F_g , f  \right \rangle |
\ \leq \
\| g |_K  \|_1 \| f \|_\infty
\ \ , \ \
\forall f \in C_0^\infty(K)
\ , \
K \subset \Omega
\ .
\]
}
\end{ex}

Introduciamo ora l'operazione di convoluzione tra una distribuzione ed una funzione in $\mD(\Omega)$. Innanzitutto richiamiamo le operazioni di antipodo e traslazione: per ogni $f \in \mD(\Omega)$, $x \in \bR^n$ definiamo
\[
\epsilon f(y) := f(-y)
\ \ , \ \
f_x(y) := f(y+x)
\ \ , \ \
y \in \bR^n
\ \ \Rightarrow \ \
\{ \epsilon f \}_x(y) = f(x-y)
\ ;
\]
\'e del tutto ovvio che $\epsilon f , f_x , \{ \epsilon f \}_x \in \mD(\Omega)$, con $\| D^\alpha \{ \epsilon f \}_x  \|_\infty = \| D^\alpha f  \|_\infty$, $\alpha \in \bZ^{+,n}$. Osserviamo quindi che per ogni $x \in \bR^n$ e $\varphi \in \mD^*(\Omega)$ con associate costanti $c_K$, $m_K$, $K \subset \Omega$, nel senso di (\ref{eq_DISTR}), risulta
\[
|  \left \langle  \varphi , \{ \epsilon f \}_x   \right \rangle   |
\ \leq \
c_K \sum_{|\alpha| \leq m_K} \| D^\alpha \{ \epsilon f \}_x   \|_\infty
\ = \
c_K \sum_{|\alpha| \leq m_K} \| D^\alpha f \|_\infty
\ < \ 
\infty
\ \ , \ \
f \in C_0^\infty(K)
\ .
\]
Per cui, \e ben definita la funzione
\[
\varphi * f : \bR^n \to \bR
\ \ , \ \
\varphi * f (x) := \left \langle  \varphi , \{ \epsilon f \}_x   \right \rangle 
\ \ , \ \
x \in \bR^n
\ .
\]
Facciamo qualche semplice osservazione sulla regolarit\'a di $\varphi * f$. Poich\'e $f \in C_0^\infty(K)$ \e uniformemente continua, troviamo che per ogni $\eps > 0$ esiste $\delta > 0$ tale che 
\[
|x-x'| \leq \delta 
\ \Rightarrow \
\| (\eps f)_x - (\eps f)_{x'} \|_\infty < \eps
\ .
\]
Ripetendo l'argomento precedente per ogni insieme {\em finito} $\{ D^\alpha f \}_{|\alpha| \leq m} \subset C_0^\infty(K)$, concludiamo che
\[
| \varphi * f (x) - \varphi * f (x') |
\ \leq \
c_K \sum_{|\alpha| \leq m_K} \| D^\alpha (\eps f)_x - D^\alpha (\eps f)_{x'} \|_\infty
\ \leq \
c_K M_K \eps
\ ,
\]
dove $M_K$ \e la somma dei multiindici $\alpha \in \bZ^{+,n}$ tali che $| \alpha | \leq m_K$. Dunque $\varphi * f \in C(\bR^n)$, e ripetendo l'argomento precedente sulle derivate di $f$ concludiamo che $\varphi * f \in C^\infty(\bR^n)$.

\begin{ex}
{\it
Data $g \in L_{loc}^1(\bR^n)$, in accordo all'Esempio \ref{ex_D_L1loc} per ogni $f \in \mD(\Omega)$ troviamo
\[
F_g * f (x) = \int g(y) f(x-y) \ dy
\ \ , \ \
x \in \bR^n
\ ,
\]
per cui $\varphi * f$ \e una naturale generalizzazione della convoluzione. Questo giustifica la classica notazione
$\varphi * f (x) = \int \varphi (y) f(x-y) dy$.
Osserviamo che Prop.\ref{prop_der_conv} implica $F_g * f \in C^\infty(\bR^n)$.
}
\end{ex}

\begin{ex}
\label{ex_dirac_conv}
{\it
Consideriamo la delta di Dirac $\delta_0 \in \mD^*(\Omega)$. Allora $\delta_0 * f(x) = \{ \epsilon f \}_x(0) = f(x)$, $f \in \mD(\Omega)$, e quindi $\delta_0 * f = f$.
}
\end{ex}

In modo analogo a quanto accade con le funzioni a decrescenza rapida, si pu\'o verificare che la dualit\'a di Riesz induce un'immersione $I : \mD(\Omega) \to \mD^*(\Omega)$ con immagine densa. Per cui, ogni applicazione continua $T : \mD(\Omega) \to \mD(\Omega)$ si estende in modo univoco ad un'applicazione continua $T^* : \mD^*(\Omega) \to \mD^*(\Omega)$, $T^* \circ I = I \circ T$. 

Un'importante applicazione delle idee precedenti \e la seguente: consideriamo un polinomio 
$p(x) = \sum_{|\alpha| < k} c_\alpha x^\alpha$,
$c_\alpha \in \bR$.
Allora \e definito l'operatore continuo
\[
p(D) : \mD(\Omega) \to \mD(\Omega)
\ \ , \ \ 
p(D)f := \sum_{|\alpha| < k} D^\alpha (c_\alpha f)
\ \ , \ \
\forall f \in \mD(\Omega)
\ ,
\]
il quale si estende in modo unico a
\[
p(D^*) : \mD^*(\Omega) \to \mD^*(\Omega)
\ \ , \ \ 
\left \langle p(D^*) \varphi , f \right \rangle
\ = \
\sum_{|\alpha| \leq k} 
 (-1)^{|\alpha|}  \ 
 \left \langle \varphi , D^\alpha (c_\alpha f) \right \rangle
\ \ , \ \
\forall f \in \mD(\Omega)
\ .
\]
Osserviamo che $p(D)$, $p(D^*)$ sono ben definiti anche qualora $c_\alpha \in C^\infty(\bR)$.

\begin{defn}
Sia $p$ un polinomio a coefficienti in $C^\infty(\bR)$ di grado massimo $k \in \bN$ e $g \in C(\bR)$. Una \textbf{soluzione debole} dell'equazione alle derivate parziali $p(D)u = g$ \e una distribuzione $\varphi \in \mD^*(\bR^n)$ tale che
\[
p(D^*) \varphi = I_g
\ \ \Leftrightarrow \ \
\left \langle p(D^*) \varphi , f \right \rangle
\ = \
\int g(x) f(x) \ dx
\ \ , \ \
\forall f \in \mD(\bR^n)
\ .
\]
Inoltre, $\varphi \in \mD^*(\bR^n)$ si dice \textbf{soluzione fondamentale} se $p(D^*) \varphi = \delta_0$.
\end{defn}

Sia $\varphi \in \mD^*(\bR^n)$ una soluzione fondamentale per l'operatore $p(D)$ ed $f \in \mD(\bR^n)$. Allora $u := \varphi * f$ \e $C^\infty$ ed ha senso considerare la funzione $p(D) u$. Ora, si pu\'o verificare che vale la propriet\'a di associativit\'a
\[
p(D) (\varphi * f) = ( p(D^*) \varphi ) * f \ ,
\]
per cui concludiamo che $p(D)u = \delta_0 * f = f$. In altri termini, $u$ \e una soluzione classica dell'equazione $p(D)u=f$. Il {\em Teorema di Malgrange-Ehrenpreis} (\cite[Theorem IX.23]{RS2}) afferma che ogni operatore differenziale $p(D)$ a coefficienti costanti ammette una soluzione fondamentale. La dimostrazione si basa su tecniche di analisi complessa e l'estensione, tramite il teorema di Hahn-Banach, del funzionale
\[
\varphi_D   \ : \   p(-D) \ (  \mD(\bR^n,\bC) )  \ \to \   \bC
\ \ , \ \
p(-D)f \mapsto f(0)
\ ,
\]
ad una distribuzione $\varphi \in \mD^*(\bR^n,\bC)$, dove 
$\mD(\bR^n,\bC) := \{ f+ig : f,g \in \mD(\bR^n) \}$. 
Per dettagli si veda \cite[Theorem IX.23]{RS2} o \cite{OW}.

\begin{ex}\textbf{\cite[\S IX.5]{RS2}.}
{\it
La soluzione fondamentale dell'\textbf{equazione di Laplace-Poisson} $\Delta u = f$, $f \in C_0^\infty(\bR^n)$, si pu\'o scrivere, \textbf{formalmente}, come $\varphi = F_\Phi$, dove
\[
\Phi(x) :=
\left\{
\begin{array}{ll}
1/2 |x|               \ \ , \ \ n=1 
\\
(2 \pi)^{-1} \log |x| \ \ , \ \ n=2
\\
- ( 4 \pi |x| )^{-1}  \ \ , \ \ n=3
\end{array}
\right.
\]
Osserviamo che in generale $\Phi \notin L_{loc}^1(\bR^n)$, per cui la notazione $F_\Phi$ v\'a intesa nel senso formale.
}
\end{ex}

\subsection{Operatori non limitati.}
\label{sec_onl}

Nonostante gli approcci che possiamo adottare a livello topologico (distribuzioni, spazi di Sobolev) sta di fatto che gli operatori differenziali, che ovviamente sono di grande importanza in analisi, non sono limitati n\'e in norma $\| \cdot \|_\infty$ n\'e in norma $\| \cdot \|_p$, e quindi non sono continui nelle topologie indotte da tali norme. In questa sezione presentiamo alcuni risultati sugli operatori non limitati su spazi di Hilbert.

\begin{defn}
Un \textbf{operatore} su uno spazio di Hilbert $\mH$ \e il dato di una coppia $(D,T)$, dove $D$ (il \textbf{dominio} di $T$) \e un sottospazio vettoriale di $\mH$ e $T : D \to \mH$ \e un'applicazione lineare. Diremo che $(D,T)$ \e \textbf{densamente definito} se $D$ \e denso in $\mH$.
\end{defn}

Ovviamente ogni $T \in B(\mH)$ \e un operatore densamente definito nel senso della definizione precedente, con $D = \mH$. D'altro canto, come esempio fondamentale portiamo la derivata
\[
\mH := L^2(\bR)
\ \ , \ \
D := C^1_c(\bR)
\ \ , \ \
Tf := f'
\ , \
\forall f \in D
\ .
\]
Osserviamo che in generale \e gi\'a difficoltoso effettuare la somma o il prodotto di operatori, in quanto occorre fare attenzione ai domini. Invece, quando $T$ \e iniettivo possiamo facilmente definire {\em l'inverso} di $T$ come l'operatore
\[
(T(D),T^{-1})
\ \ : \ \
T^{-1}(Tu) := u
\ , \
\forall u \in D
\ ;
\]
rimarchiamo il fatto che potremmo avere $T^{-1} \in B(\mH)$ anche quando $T$ \e non limitato. Diciamo che $(D',T')$ {\em estende} $(D,T)$ se $D \subset D'$ e $T'|_D = T$, ed in tal caso scriviamo $(D,T) \prec (D',T')$; osserviamo che pu\'o capitare che $(D,T)$ abbia come estensione un operatore limitato.

Analogamente al caso limitato, diciamo che $u \in D$ \e un autovettore di $(D,T)$ se $Tu=\lambda u$ per qualche $\lambda \in \bC$.

Diciamo che $(D,T)$ \e {\em chiuso} se il grafico
\[
G(D,T)
:=
\{ u \oplus Tu \in D \oplus \mH \}
\]
\e un sottospazio chiuso di $\mH \oplus \mH$, ed in tal caso, grazie al teorema del grafico chiuso, abbiamo che se $D=\mH$ allora $T$ \e limitato. Preso un generico operatore $(D,T)$, diremo che esso \e {\em chiudibile} se la chiusura di $G(D,T)$ in $\mH \oplus \mH$ \e il grafico di un operatore $(\ovl D , \ovl T)$; chiaramente, in tal caso $(\ovl D , \ovl T)$ \e chiuso e per definizione estende $(D,T)$. Chiamiamo $(\ovl D , \ovl T)$ la {\em chiusura} di $(D,T)$.

\

\noindent \textbf{Operatori simmetrici ed autoaggiunti.} Preso $(D,T)$ {\em densamente definito}, per ogni $u \in \mH$ consideriamo l'applicazione lineare
\[
f_{T,u} : D \to \bK
\ \ , \ \
f_{T,u}(v) := ( u,Tv )
\]
(dove $\bK = \bR,\bC$ a seconda del caso reale o complesso), la quale non necessariamente \e limitata; {\em quando essa lo \`e}, essendo $D$ denso possiamo estendere $f_{T,u}$ per continuit\'a ad un funzionale di $\mH$, per cui il Teorema di Riesz ci assicura che esiste ed \e unico $T^*u \in \mH$ tale che
\begin{equation}
\label{def_agg}
(T^*u,v) = \left \langle f_{T,u} , v \right \rangle
\ \  ,\ \
\forall v \in \mH
\ .
\end{equation}
L'{\em operatore aggiunto} di $(D,T)$ si definisce come la coppia $(D^*,T^*)$, dove
\[
D^* := \{ u \in \mH : \| f_{T,u} \| < \infty  \}
\]
e $T^*u$ \e definito da (\ref{def_agg}); ovviamente in particolare troviamo
$( T^*u,v) = (u,Tv)$,
$\forall v \in D$,
il che giustifica la nostra terminologia. Osserviamo che in generale non \e detto che $(D^*,T^*)$ sia densamente definito.

\begin{prop}
\label{prop_OnL01}
Sia $(D,T)$ densamente definito. Allora $(D^*,T^*)$ \e chiuso.
\end{prop}

\begin{proof}[Dimostrazione]
Se $u \oplus z \in G(D,T)^\perp$ allora, per ogni $v \in D$, abbiamo
\[
0 = ( u \oplus z , v \oplus Tv ) = (u,v) + (z,Tv)
\ \Leftrightarrow \
f_{T,z}(v) = - (u,v)
\ .
\]
Essendo $D$ denso, l'uguaglianza precedente \e equivalente ad affermare che $z \in D^*$ e $T^*z = -u$. Possiamo dunque definire l'operatore
\begin{equation}
\label{eq_OnL01}
U : G(D^*,T^*) \to G(D,T)^\perp
\ \ , \ \
U(z \oplus T^*z) := (-T^*z \oplus z)
\ ,
\end{equation}
il quale chiaramente \e isometrico; del resto le considerazioni precedenti ci dicono che $U$ \e anche suriettivo e quindi invertibile. Ora, l'ortogonale di un sottospazio \e sempre chiuso, per cui, essendo $U$ invertibile ed isometrico, $G(D^*,T^*) = U^{-1}G(D,T)^\perp$ \e chiuso.
\end{proof}

\begin{cor}
\label{cor_OnL01}
Con le notazioni della proposizione precedente, si ha la decomposizione ortogonale
\[
\mH \oplus \mH
\ = \
\ovl{G(D,T)} \oplus UG(D^*,T^*)
\ .
\]
\end{cor}

\begin{proof}[Dimostrazione]
Usando (\ref{eq_OnL01}) gi\'a sappiamo che $G(D,T)^\perp = UG(D^*,T^*)$. Del resto $G(D,T)^{\perp \perp} = \ovl{G(D,T)}$ e ci\'o conclude la dimostrazione. 
\end{proof}

Un operatore densamente definito $(D,T)$ si dice {\em simmetrico} se 
\[
(Tu,v) = (u,Tv)
\ \ , \ \
\forall u,v \in D
\ .
\]
L'espressione precedente ci dice che ogni $u \in D$ \e tale che $f_{T,u}$ \e limitato e $Tu = T^*u$. Dunque, $(D,T)$ \e simmetrico se e solo se $(D,T) \prec (D^*,T^*)$; ci\'o implica che anche $(D^*,T^*)$ \e densamente definito.
Se $(D,T)$ \e simmetrico allora possiede solo autovalori reali, infatti se $u \in D$ e $Tu=\lambda u$, $\lambda \in \bC$, allora troviamo
\[
(u,Tu) = \lambda (u,u) = (Tu,u) = \ovl \lambda (u,u) \ .
\]

\begin{defn}
Un operatore simmetrico $(D,T)$ si dice: (1) \textbf{autoaggiunto}, se $D=D^*$ (il che implica $T=T^*$); (2) \textbf{essenzialmente autoaggiunto}, se la chiusura $(\ovl D,\ovl T)$ \e un operatore autoaggiunto.
\end{defn}

\begin{cor}
Ogni operatore simmetrico \e chiudibile, ed ogni operatore autoaggiunto \e chiuso.
\end{cor}

\begin{proof}[Dimostrazione]
Segue immediatamente da Prop.\ref{prop_OnL01}.
\end{proof}

\begin{rem}
\label{oss_OA}
{\it
Se $(D,T)$ \e autoaggiunto ed $S = S^* \in B(\mH)$ allora $(D,T+S)$ \e autoaggiunto; infatti
\[
| f_{T,u}(v) | - \| u \| \| S \| \| v \|
\ \leq \ 
| f_{T+S,u}(v) |
\ \leq \
| f_{T,u}(v) | + \| u \| \| S \| \| v \|
\ \ , \ \
\forall u \in \mH
\ , \
v \in D
\ ,
\]
per cui $\{ u \in B(\mH) : \| f_{T+S,u} \| < \infty \} = D^* = D$.
}
\end{rem}

\begin{ex}[La derivata]
\label{ex_dx}
{\it
Consideriamo lo spazio di Hilbert complesso $\mH := L^2([0,1],\bC)$ e l'operatore di Volterra
\[
Fu (x) := i \int_0^x u(t) \ dt
\ \ , \ \
u \in \mH
\ .
\]
Se $u \in \mH$ \e tale che $Fu = 0$, allora per ogni $x,y \in (0,1)$ troviamo
\[
0 = ( \chi_{[x,y]} , u ) = \int_x^y u = -i (Fu(y)-Fu(x)) = 0
\ ,
\]
dunque $Fu$ \e ortogonale ad ogni funzione a gradini; poich\'e l'insieme delle funzioni a gradini \e denso in $\mH$ concludiamo che $Fu=0$, cosicch\'e $F$ \e iniettivo. Inoltre $F$ ha immagine $F(\mH)$ densa in $\mH$, in quanto essa contiene il sottoinsieme (denso)
$\{ f \in C^1([0,1],\bC) : f(0)=0 \} \subset \mH$.
Definiamo dunque l'operatore
\[
D := \{ Fu+z \ : \ z \in \bC \ , \ u \in \mH \ , \ Fu(0)=Fu(1)=0 \}
\ \ , \ \
T(Fu+z) := u
\ .
\]
Si verifica facilmente (semplice esercizio sugli integrali multipli!) che $(D,T)$ \e simmetrico e che
\[
D^* = \{ Fu+z : z \in \bC , u \in \mH \} \ ,
\]
cosicch\'e $(D,T)$ non \e autoaggiunto (per dettagli si veda \cite[\S 5.1.16]{Ped}). Si osservi che, applicando il teorema fondamentale del calcolo,
\[
Tg = -i g'
\ \ , \ \
\forall g \in D \cap C^1([0,1],\bC)
\ .
\]
}
\end{ex}

\begin{lem}
Se $S \in B(\mH)$ \e autoaggiunto e iniettivo allora $(D,S^{-1})$, $D := S(\mH)$, \e un operatore autoaggiunto.
\end{lem}

\begin{proof}[Dimostrazione]
Innanzitutto osserviamo che avendosi 
\[
\ovl{D} \ \stackrel{Lemma \ref{lem_pp}}{=} \
D^{\perp \perp} \ \stackrel{ (\ref{eq_kerim}) }{=} \
(\ker S)^\perp \ = \
\{ 0 \}^\perp \ = \
\mH
\ ,
\]
concludiamo che $D$ \e denso in $\mH$.
Per ogni $u:= Su_0$, $v := Sv_0 \in D$ osserviamo che 
\[
(u,S^{-1}v) = (Su_0,v_0) = (u_0,Sv_0) = (S^{-1}u,v) \ ,
\]
cosicch\'e $(D,S^{-1})$ \e simmetrico e quindi $D \subseteq D^*$, il dominio dell'aggiunto $(D^*,S^{-1,*})$. D'altro canto, $w \in \mH$ appartiene a $D^*$ se e solo se esiste $w' \in \mH$ tale che
\[
(w',u) = (w,S^{-1}u)
\ \ , \ \
\forall u \in D
\ .
\]
Ma
\[
(w',u) = (w',Su_0) = (Sw',u_0) = (Sw',S^{-1}u)
\ \Rightarrow \
(w,S^{-1}u) = (Sw',S^{-1}u)
\ \ , \ \
\forall u \in D
\ ,
\]
per cui avendosi $\mH = \{ S^{-1}u , u \in D \}$ concludiamo che $w = Sw'$ e quindi $w \in D$. Da ci\'o si deduce $D=D^*$ e ci\'o conclude la dimostrazione.
\end{proof}

\begin{thm}[von Neumann]
\label{thm_OnL01}
Sia $(D,T)$ densamente definito e chiuso. Allora esiste un sottospazio $D_+ \subseteq D$, denso in $\mH$, tale che $(D_+,T^*T)$ \e autoaggiunto e
$1+T^*T : D_+ \to \mH$
\e biettivo, con $(1+T^*T)^{-1} \in B(\mH)$.
\end{thm}

\begin{proof}[Dimostrazione]
Poich\'e $T$ \e chiuso abbiamo 
$G(D,T) = \ovl{G(D,T)}$
e grazie a Cor.\ref{cor_OnL01} possiamo introdurre gli operatori
\[
\mH \oplus \mH \stackrel{P}{\to} G(D,T)  \stackrel{P'}{\to} D
\ \ , \ \
\mH \oplus \mH \stackrel{1-P}{\to} UG(D^*,T^*) \stackrel{P'U^{-1}}{\to} D^*
\ ,
\]
dove $P$ \e il proiettore su $G(D,T)$, $1 \in B(\mH \oplus \mH)$ \e l'identit\'a e $P'$ \e la proiezione di $\mH \oplus \mH$ sulla prima componente della somma diretta. Definiamo gli operatori
\[
\left\{
\begin{array}{ll}
S : \mH \to D   \ \ , \ \ Su := P'P(u \oplus 0)
\\
R : \mH \to D^* \ \ , \ \ Ru := \{ P'U^{-1}(1-P) \} (u \oplus 0)
\ .
\end{array}
\right.
\]
Essendo $R , S$ composizioni di proiettori ed operatori isometrici troviamo
$\| R \| , \| S \| \leq 1$.
Applicando la decomposizione ortogonale di Cor.\ref{cor_OnL01} e (\ref{eq_OnL01}), per ogni $u \in \mH$ abbiamo
\[
u \oplus 0 =
(Su \oplus TSu) + (T^*Ru \oplus -Ru ) = 
(S+T^*R)u \oplus (TS-R)u
\ ,
\]
cosicch\'e
\begin{equation}
\label{eq_OnL02}
u = (S+T^*R)u
\ \ , \ \
0 = (TS-R)u
\ \Rightarrow \
R = TS
\ \ , \ \
(1+T^*T)S = 1
\ .
\end{equation}
Posto $D_+ := S(\mH) \subseteq D$, queste ultime uguaglianze ci dicono che:
(1) $T(D_+) \subseteq R(\mH) \subseteq D_*$, cosicch\'e $1+T^*T$ \e ben definito su $D_+$;
(2) $S$ \e iniettivo, e quindi anche $1+T^*T$ \e iniettivo su $D_+$;
(3) l'immagine di $1+T^*T$ coincide con $\mH$.
Inoltre troviamo
\[
(S^*u,u) = 
(S^*(1+T^*T)Su,u) = 
\| Su \|^2 + \| Ru \|^2 \in \bR \ ,
\]
cosicch\'e $S = S^*$. Usando il Lemma precedente otteniamo che $(D_+,S^{-1})$ \e autoaggiunto, e confrontando con l'ultima uguaglianza di (\ref{eq_OnL02}) concludiamo che $(D_+,S^{-1}) = (D_+,1+T^*T)$, cosicch\'e $(D_+,1+T^*T)$ \e autoaggiunto. Infine, Oss.\ref{oss_OA} implica che anche $(D_+,T^*T)$ \e autoaggiunto.
\end{proof}

\begin{ex}{\it
Consideriamo l'operatore simmetrico $(D,T)$ dell'Esempio \ref{ex_dx}. Abbiamo che $(D,T)$ \e densamente definito, ed \e semplice verificare che esso \e anche chiuso. D'altra parte troviamo
\[
T^*T u = - u''
\ \ , \ \
\forall u \in D \cap C^2([0,1],\bC)
\ ,
\]
per cui il teorema di von Neumann implica che l'operatore di derivata seconda ammette un'estensione autoaggiunta. Sempre invocando il teorema di von Neumann, concludiamo che per ogni $f \in \mH$ esiste un unico $u \in D$ tale che $u + T^*T u = f$. Se, in particolare, $u \in C^2([0,1],\bC)$ allora $u$ \e soluzione unica dell'equazione
\[
u-u'' = f 
\ \ , \ \
u \in D
\ .
\]
Notare che il precedente problema differenziale presenta delle condizioni al bordo ($u(0) = u(1)$), "nascoste" nella definizione del dominio $D$.
}
\end{ex}

Un operatore simmetrico $(D,T)$ si dice {\em semicoercitivo} se esiste $\alpha \geq 0$ tale che
$(Tu,u) \geq \alpha \| u \|^2$ 
per ogni $u \in D$; in particolare diciamo che $T$ \e {\em coercitivo} se $\alpha > 0$. Ad esempio, ogni operatore {\em positivo}, ovvero tale che $(Tu,u) \geq 0$, $\forall u \in D$, \e semicoercitivo (con $\alpha = 0$). Per segnalare l'importanza della nozione di coercitivit\'a nella teoria delle equazioni alle derivate parziali invitiamo a dare un'occhiata alle sezioni \ref{sec_StLM} e \ref{sec_SobEDP}. E' possibile dimostrare che ogni operatore simmetrico semicoercitivo ha un'estensione autoaggiunta; questo risultato \e noto come {\em l'estensione di Friedrichs} (\cite[\S 5.1.13]{Ped}).

\begin{ex}[L'operatore di Laplace]
\label{ex_laplace}
{\it
Sia $\Omega \subseteq \bR^d$ aperto. Sullo spazio di Hilbert $L^2(\Omega)$ definiamo l'operatore
\[
D := C_c^\infty(\Omega)
\ \ , \ \
Lu := - \Delta u
\ , \
\forall u \in D 
\ .
\]
Le formule di Green implicano che 
$( u,Lu ) = \int \nabla u \cdot \nabla u \geq 0$, $\forall u \in D$,
per cui usando il teorema di Friedrichs concludiamo che $(D,L)$ ammette un'estensione autoaggiunta. Il dominio di questa estensione \e noto come lo spazio di Sobolev $H^2(\Omega)$ (vedi \S \ref{sec_sobolev}).
}
\end{ex}

\

\noindent \textbf{La trasformata di Cayley.} In questo paragrafo consideriamo spazi di Hilbert {\em complessi}. Sia $(D,T)$ un operatore simmetrico e $\lambda = x+iy \in \bC$; allora anche $(D,T-x 1)$ \e simmetrico e per ogni $u \in D$ risulta
\[
\| (T - \lambda 1)u \|^2 \ =  \
(  (T - x1)u - iyu  ,  (T - x1)u - iyu   ) \ = \
\| (T - x       1)u \|^2 + y^2 \| u \|^2 \ \geq \
y^2 \| u \|^2
\ .
\]
Cosicch\'e $T-\lambda 1$ \e iniettivo {\em quando $y \neq 0$}, e l'operatore
\[
(T-\lambda 1)^{-1} : \{ T - \lambda 1 \}(D) \to \mH
\]
\e limitato con norma $\leq y^{-1}$. Definiamo la {\em trasformata di Cayley} di $(D,T)$ come l'operatore lineare
\begin{equation}
\label{def_cayley}
\mC T : \{ T+i1 \}(D) \to \{ T-i1 \}(D)
\ \ , \ \
\mC T := (T-i1)(T+i1)^{-1} 
\ .
\end{equation}
\begin{lem}
\label{lem_cayley}
Si hanno le seguenti propriet\'a:
\textbf{(1)} $\mC T$ \e isometrico e suriettivo, per cui si estende ad un operatore $\ovl{\mC} T \in B(\mH)$, isometrico su $\ovl{\{ T+i1 \}(D)}$ e nullo sul suo complementare in $\mH$
{\footnote{Operatori di questo tipo sono detti \textbf{isometrie parziali}. Si osservi che il nucleo di un'isometria parziale \e per costruzione il complementare del sottospazio sul quale essa \e isometrica.}}; 
\\
\textbf{(2)} L'operatore $1-\mC T$ \e iniettivo e con immagine $D$, in maniera tale che
\begin{equation}
\label{rel_cayley}
i(1+ \mC T)( 1-\mC T )^{-1} = T \ .
\end{equation}
\end{lem}

\begin{proof}[Dimostrazione]
Le stesse considerazioni fatte appena prima della definzione di $\mC T$ mostrano che
\[
\| (T+i1)u \|^2 = \| Tu \|^2 + \| u \|^2 = \| (T-i1)u \|^2
\ \ , \ \
\forall u \in D
\ ,
\]
cosicch\'e per ogni $v = (T+i1)u$ abbiamo $\| \mC T v \| = \| v \|$ e $\mC T$ si estende per continuit\'a a $\ovl{\{ T+i1 \}(D)}$. Possiamo quindi definire 
$\ovl{\mC} T v := \mC T \circ Pv$, $v \in \mH$,
dove $P \in B(\mH)$ \e il proiettore sul sottospazio 
$\ovl{\{ T+i1 \}(D)} \subseteq \mH$; 
ci\'o mostra il punto (1)
{\footnote{Nel seguito identificheremo $\mC T$ con $\ovl{\mC} T$.}}.
Riguardo il punto (2) osserviamo che
\[
\mC T v = v
\ \Leftrightarrow \
(T-i1)u = (T+i1)u
\ \Leftrightarrow \
u=-u = 0
\ \Leftrightarrow \
v=0
\ \ , \ \
\forall v = (T+i1)u
\ ,
\]
per cui $1-\mC T$ \e iniettivo. Per valutare l'immagine di $1- \mC T$ poniamo 
$W := \{ T+i1 \}(D)$ 
e calcoliamo
\[
\left\{
\begin{array}{ll}
1|_W - \mC T =
(T+i1)(T+i1)^{-1} - (T-i1)(T+i1)^{-1} = 
2i(T+i1)^{-1}
\\
1|_W + \mC T = 2T(T+i1)^{-1} \ ,
\end{array}
\right.
\]
da cui segue che
\[
\{ 1 - \mC T \} (Tu+iu) = 2iu
\ \ , \ \
\forall u \in D
\ \Rightarrow \
\{ 1 - \mC T \}(W) = D
\ ,
\]
nonch\'e
\[
i(1+\mC T)(1-\mC T)^{-1} = 2iT(T+i1)^{-1} (2i)^{-1}(T+i1) = T
\ .
\]
\end{proof}

Il seguente risultato mostra il vantaggio apportato dalla trasformata di Cayley, consistente nella possibilit\'a di descrivere gli operatori simmetrici (non limitati) in termini di isometrie parziali:
\begin{thm}
\label{thm_cayley}
La trasformata di Cayley definisce una corrispondenza biunivoca dall'insieme degli operatori simmetrici su quello delle isometrie parziali $U \in B(\mH)$ tali che $(1-U)|_{\ker U^\perp}$ ha immagine densa in $\mH$. Inoltre, se $(D,T) \prec (D',T')$ allora $\mC T'u = \mC T u$ per ogni $u \in \{ T+i1 \}(D) \subseteq \{ T'+i1 \}(D')$.
\end{thm}

\begin{proof}[Dimostrazione]
Che l'applicazione $\{ (T,D) \mapsto \mC T \}$ sia iniettiva segue da (\ref{rel_cayley}). D'altra parte \e ovvio che se $(T',D')$ estende $(T,D)$ allora $\mC T'$ estende $\mC T$ nel senso dell'enunciato, per cui rimane da verificare solo che la trasformata di Cayley \e suriettiva.
A tale scopo consideriamo un'isometria parziale $U \in B(\mH)$ tale che $(1-U)|_{\ker U^\perp}$ ha immagine densa; per definizione di isometria parziale abbiamo
\[
(Uv,Uv') = (v,v')
\ \ , \ \
\forall v,v' \in \ker U^\perp
\ ,
\]
cosicch\'e se $Uw = w$ per qualche $w \in \ker U^\perp$ allora
\[
(w,(1-U)v) = (w,v)-(w,Uv) = (Uw,Uv)-(w,Uv) = (Uw-w,Uv) = 0
\ ;
\]
poich\'e $(1-U)|_{\ker U^\perp}$ ha immagine densa concludiamo che deve essere $w=0$, per cui $1-U$ \e iniettivo su $\ker U^\perp$. Siamo ora in grado di definire
\[
D := \{ 1-U \}(\ker U^\perp)
\ \ , \ \
T (1-U)v := i(1+U)v
\ \ , \ \
v \in \ker U^\perp
\ ,
\]
e dei semplici conti mostrano che
\[
( T(1-U)v , (1-U)v ) \in \bR
\ \ , \ \
\forall v \in \ker U^\perp
\]
(cosicch\'e $(D,T)$ \e simmetrico), e che $U = \mC T$.
\end{proof}

I seguenti risultati mostrano l'utilit\'a della trasformata di Cayley nel determinare se l'operatore simmetrico $(D,T)$ \e autoaggiunto.

\begin{prop}
\label{prop_CaAA}
$(D,T)$ \e autoaggiunto se e solo se $\mC T$ \e unitario.
\end{prop}

\begin{proof}[Dimostrazione]
Se $(D,T)$ \e autoaggiunto allora $i,-i$ non sono autovalori di $T$, per cui 
$\{ 0 \} = \ker ( T \pm i 1 ) = \{ T \mp i 1 \}(D)^\perp$,
il che implica che $\{ T \pm i 1 \}(D)$ sono densi in $\mH$. Ma $T$, essendo autoaggiunto, \e anche chiuso, per cui $\{ T \pm i 1 \}(D)$ sono spazi chiusi e quindi coincidenti con $\mH$; concludiamo che $\mC T$ \e unitario.
Viceversa, se $\mC T$ \e unitario allora $T+i1$ \e suriettivo e per ogni $u \in D^*$ esiste $v \in D$ tale che
\[
(T^* + i1)u = (T+i1)v
\ \Leftrightarrow \
\{ T^* + i1 \} (u-v) = 0
\]
(infatti $T^*u=Tu$ avendosi $u \in D \subseteq D^*$). Del resto, per ipotesi abbiamo
$\ker ( T^*+i1 ) = \{ T-i1 \}(D)^\perp = \{ 0 \}$,
e quindi $u=v \in D$, da cui $D=D^*$.
\end{proof}

Gli {\em indici di difetto di} $(D,T)$ si definiscono come
\[
n_+ := {\mathrm{dim}}\{ T+i1 \}(D)^\perp
\ \ , \ \
n_- := {\mathrm{dim}}\{ T-i1 \}(D)^\perp
\ \ , \ \
n_+ , n_- \in \bN \cup \{ \infty \}
\ .
\]
Segue direttamente da (\ref{def_cayley}) e dal Lemma \ref{lem_cayley} che $\mC T$ (o, per essere pignoli, la sua estensione per continuit\'a ad $\mH$) \e unitario se e solo se $n_+ = n_- = 0$, per cui $(D,T)$ \e autoaggiunto se e solo se $n_+ = n_- = 0$. 
Se invece $(D,T)$ ha indici di difetto non nulli ma uguali, allora esiste un operatore unitario 
\[
V : \{ T+i1 \}(D)^\perp \to \{ T-i1 \}(D)^\perp
\]
(costruito nella maniera banale, mettendo in corrispondenza biunivoca gli elementi delle basi), per cui possiamo estendere $\mC T$ ponendo
\[
U (u+v) := \mC T u + Vv
\ \ , \ \
u \in \ovl{\{ T+i1 \}(D)}
\ , \
v \in \{ T+i1 \}(D)^\perp
\ ;
\]
per costruzione $U$ \e unitario e tale che $1-U$ ha immagine densa, per cui esso \e la traformata di Cayley di un operatore autoaggiunto $(D',T')$ il quale, grazie al Teorema \ref{thm_cayley}, estende $(D,T)$. Dunque se $n_+ = n_-$ allora $(D,T)$ ha un'estensione autoaggiunta, non unica in quanto dipendente dall'operatore $V$ definito poc'anzi. Per esempi di calcolo di indici di difetto rimandiamo all'Esercizio \ref{sec_afunct}.17 e \cite[Es.13.2.9]{Car} (lo shift).

\begin{rem}
\label{rem_defect}
{\it
In modo analogo al caso limitato abbiamo
$\ker T^* = T(D)^\perp$
per ogni operatore $(D,T)$ densamente definito. Dunque troviamo
\[
n_+ = {\mathrm{dim}} \ker(T^*-i1)
\ \ , \ \
n_- = {\mathrm{dim}} \ker(T^*+i1)
\ ,
\]
cosicch\'e per calcolare gli indici di difetto possiamo procedere risolvendo le equazioni 
$T^*u = \pm iu$, $u \in D^*$.
}
\end{rem}

\

\noindent \textbf{Teoria spettrale degli operatori autoaggiunti.} Sia $(D,T)$ un operatore sullo spazio di Hilbert complesso $\mH$. Il {\em risolvente} di $(D,T)$ \e per definizione l'insieme dei $z \in \bC$ tali che esiste $R_z \in B(\mH)$ con $R_z(\mH) \subseteq D$ e
\[
(z1-T)R_z = 1
\ \ , \ \
R_z(z1-T) = 1|_D
\]
(in termini colloquiali, $R_z = (z1-T)^{-1}$). Lo {\em spettro di $(D,T)$} si definisce come il complementare del risolvente e si denota con $\sigma(D,T)$. Come vedremo, in questo ambito pi\'u generale abbiamo che $\sigma(D,T)$ \e tipicamente non limitato, mentre invece ritroviamo la propriet\'a di chiusura del caso degli operatori limitati:

\begin{prop}
Lo spettro $\sigma(D,T) \subseteq \bC$ \e chiuso, per ogni operatore $(D,T)$. 
\end{prop}

\begin{proof}[Dimostrazione]
Sia $z \in \bC - \sigma(D,T)$ e $w \in \bC$ con $|w| \| R_z \|^{-1} < 1$. Allora la serie 
$\sum_{n=0} R_z^nw^n$
\e assolutamente convergente e, come nel caso limitato, troviamo che $1-wR_z$ \e invertibile, con
$
(1-wR_z)^{-1} = \sum_n R_z^nw^n$.
Per cui
\[
(1-wR_z)^{-1} R_z = 
\{ (z1-T) (1-wR_z) \}^{-1} = 
\{ (z1-T) (1-w(z1-T)^{-1}) \}^{-1} =
\{ (z-w)1-T \}^{-1}
\ ,
\]
e quindi $\bC - \sigma(D,T)$ \e aperto.
\end{proof}

In generale non \e detto che lo spettro di un operatore simmetrico sia contenuto in $\bR$; infatti, mentre da un lato sappiamo che $z1-T$ \e iniettivo se ${\mathrm{Im}}(z) \neq 0$ (vedi paragrafo sulla trasformata di Cayley), non \e detto che $z1-T$ sia suriettivo. Tuttavia, troviamo:
\begin{prop}
Se $(D,T)$ \e autoaggiunto allora $\sigma(D,T) \subseteq \bR$.
\end{prop}

\begin{proof}[Dimostrazione]
Posto $z = a+ib$, $b \neq 0$, abbiamo che $z1-T$ \e invertibile se e solo se $S+i1$ \e invertibile, dove $(D,S)$, $S:=b^{-1}(a1-T)$, \e autoaggiunto.
Essendo $(D,S)$ simmetrico abbiamo che $S+i1$ \e iniettivo (vedi inizio paragrafo sulla trasformata di Cayley), e quanto visto nella dimostrazione di Prop.\ref{prop_CaAA} implica che $S+i1$ \e anche suriettivo e quindi invertibile.
\end{proof}

In analogia con il caso limitato ritroviamo il teorema spettrale per operatori autoaggiunti. Denotiamo con $B(\sigma(D,T))$ e $B^\infty(\sigma(D,T),\bC)$ rispettivamente lo spazio delle funzioni boreliane e la \sC algebra delle funzioni boreliane limitate su $\sigma(T)$, a valori complessi in entrambi i casi.
\begin{thm}[von Neumann]
\label{thm_sonl}
Sia $(D,T)$  un operatore autoaggiunto su uno spazio di Hilbert $\mH$. Allora:
\textbf{(1)} Esiste una misura spettrale $\mu := \{ \mu_{uv} \}_{u,v \in D}$ per $T$, tale che \e verificata (\ref{eq.diag}) per ogni $u,v \in D$;
\textbf{(2)} Per ogni $f \in B(\sigma(D,T),\bC)$, l'applicazione
\[
A_f : D_f \times D_f \to \bC
\ \ , \ \
u,v \mapsto \int_{\sigma(D,T)} f \ d\mu_{uv}
\ \ , \ \
D_f := \{ u \in \mH : \int_{\sigma(D,T)} |f|^2 \ d \mu_{uu} < \infty \}
\ ,
\]
definisce una forma sesquilineare e quindi un operatore $(D_f,f(T))$ tale che
$A_f(u,v) = (u,f(T)v)$, $\forall u,v \in D_f$.
In particolare $(D,T) = (D_I,I(T))$, dove $I(\lambda) := \lambda$, $\forall \lambda \in \sigma(D,T)$;
\textbf{(3)} Se $f \in B^\infty(\sigma(D,T),\bC)$ allora $f(T)$ \e limitato con $\| f(T) \| \leq \| f \|_\infty$, e si ha una rappresentazione
\[
B^\infty(\sigma(D,T),\bC) \to B(\mH)
\ \ , \ \
f \mapsto f(T)
\ .
\]
\end{thm}

Referenze per una dimostrazione dettagliata del teorema precedente sono \cite[VIII.3]{RS},\cite[5.3]{Ped}; l'ordine di idee \e essenzialmente quello del caso limitato, con la differenza sostanziale che stavolta $\sigma(D,T)$ \e uno spazio {\em localmente} compatto, per cui occorre apportare delle modifiche all'argomento che utilizza il calcolo funzionale continuo ed il teorema di Riesz-Markov.
Osserviamo inoltre che, a differenza del caso limitato, funzioni continue su $\sigma(D,T)$ possono essere non limitate e quindi tali che la loro immagine rispetto al calcolo funzionale del punto (2) sia un operatore non limitato, come nel caso dello stesso $(D,T)$. Invece le misure spettrali del punto (1) rimangono comunque finite.

\begin{ex}[La derivata]
\label{eq.der.I}
{\it Poniamo $\mH := L^2([0,1],\bC)$, cosicch\'e $\mH \subset L^1([0,1],\bC)$, e denotiamo con $H^1_\bC \subset \mH$ lo spazio delle primitive di funzioni in $\mH$.
Definiamo quindi
\[
D := \{ u \in H^1_\bC : u(0)=u(1) \}
\ \ , \ \
Tu := -iu'
\ , \
\forall u \in D
\ .
\]
L'operatore $(D,T)$ \e banalmente simmetrico (per verificare si integri per parti) ed autoaggiunto (Esercizio \ref{sec_afunct}.17), per cui lo spettro di $(D,T)$ \e reale, ed in effetti $\sigma(D,T) = 2 \pi \bZ$ (Esercizio \ref{sec_afunct}.18).
Consideriamo ora la trasformata di Fourier discreta (Esempio \ref{ex_pontr})
\[
\wa u (n) := \int_0^1 u(x) e^{2 \pi inx} dx
\ \ , \ \
\forall n \in \bZ
\ \ , \ \
u \in \mH
\ .
\]
Essendo 
$\mB_\bC := \{ e_n(x) := e^{-2 \pi inx} , x \in [0,1] \}$ 
una base di $\mH$ (Esempio \ref{ex_base}), abbiamo
$(u,v) = \sum_n \ovl{\wa u(n)} \wa v(n)$ per ogni $u,v \in \mH$.
Definiamo la famiglia di misure complesse
\[
\mu_{uv}E \ := \  \sum_{2 \pi n \in E} \ovl{\wa u(n)} \wa v (n)
\ \ , \ \
\forall E \subseteq 2^{\sigma(D,T)}
\ , \
u,v \in \mH
\ .
\]
Poich\'e $\mB_\bC \subset D$ possiamo calcolare
\[
Te_n = -i e_n' = 2 \pi n e_n
\ \ , \ \
\forall n \in \bZ
\ ,
\]
per cui, per ogni $u,v \in D$, troviamo
\[
(u,Tv) \ = \
\sum_n \ovl{\wa u(n)} (e_n,Tv) \ = \
\sum_n \ovl{\wa u(n)} (Te_n,v) \ = \
\sum_n 2 \pi n \ovl{\wa u(n)} \wa v(n)         \ = \
\sum_{\lambda \in \sigma(D,T)} \lambda \mu_{uv} \{ \lambda \}
\ ,
\]
e $\{ \mu_{uv} \}$ \e una misura spettrale per $(D,T)$.
}
\end{ex}

\

\noindent \textbf{Gruppi ad un parametro.} Sia $(D,T)$ un operatore autoaggiunto. Consideriamo la famiglia di funzioni $\{ e_t \}_{t \in \bR}$,
\[
e_t : \bR \to \bC
\ \ , \ \
e_t(\lambda) := e^{i \lambda t}
\ \ , \ \
\forall \lambda \in \bR
\ .
\]
E' ovvio che ogni $e_t$ \e continua e limitata, per cui 
$\{ e_t \} \subset B^\infty(\sigma(D,T),\bC)$.
Dunque, grazie al teorema precedente possiamo definire
\[
U_t := e_t(T) \in B(\mH)
\ \ , \ \
\forall t \in \bR
\ .
\]
Ora, abbiamo
\[
(e_t)^* = e_{-t}
\ , \
e_t e_{-t} = 1
\ , \
e_{t+s} = e_t e_s
\ \ , \ \
\forall t,s \in \bR
\ ,
\]
e visto che il calcolo funzionale boreliano conserva le operazioni di moltiplicazione e passaggio all'aggiunto troviamo
\begin{equation}
\label{eq.scopug}
U_t^* = U_{-t}
\ \ , \ \
U_t U_{-t} = 1
\ \ , \ \
U_{t+s} = U_t U_s
\ \ , \ \
\forall t,s \in \bR
\ ,
\end{equation}
cosicch\'e in particolare ogni $U_t$, $t \in \bR$, \e unitario. Chiamiamo $U := \{ U_t \}$ il {\em gruppo ad un parametro definito da $(D,T)$}.

Visto che $e_t - 1 \to 0$ puntualmente per $t \to 0$, in conseguenza dell'Esercizio \ref{sec_afunct}.16 e della relazione $U_{t+s} = U_t U_s$ troviamo che $U$ soddisfa la seguente propriet\'a di continuit\'a rispetto alla topologia forte:
\begin{equation}
\label{eq.scopug2}
\lim_{t \to 0} \| U_tv - v \|        \ = \
\lim_{t \to 0} \| U_{t+s}v - U_sv \| \ = \
0
\ \ , \ \
\forall v \in \mH
\ , \
s \in \bR
\ .
\end{equation}
Ora, per ogni $\lambda \in \sigma(D,T)$ abbiamo
\[
| e^{i\lambda t}-1 |  \leq  | \lambda t |
\ , \
\forall t \in \bR
\ ,
\]
per cui per ogni $u \in D$ troviamo
\[
\begin{array}{ll}
\| \{ t^{-1}(U_t-1) - iT \} u \|^2  & =
( u \ , \  \{ t^{-1}(U_t-1) - iT \}^* \{ t^{-1}(U_t-1) - iT \} u ) \\
& =
\int | \{ t^{-1}(e_t-1) - i \lambda \}^* 
       \{ t^{-1}(e_t-1) - i \lambda \}| \ d \mu_{uu}(\lambda)  \\
& =
\int | t^{-1}(e^{i\lambda t}-1) - i\lambda |^2 \ d \mu_{uu}(\lambda)  \\
& \leq
\int \{ \lambda^2 + 2 \lambda^2  + \lambda^2 \} d \mu_{uu}(\lambda) \ ;
\end{array}
\]
grazie al Teorema \ref{thm_sonl} sappiamo che la funzione $I^2(\lambda) := \lambda^2$, $\lambda \in \sigma(D,T)$, \e $\mu_{uu}$-integrabile per ogni $u \in D$, per cui possiamo applicare il teorema di convergenza dominata e passare al limite $t \to 0$ sotto il segno di integrale. Ma
\[
\lim_{t \to 0} t^{-1}(e^{i \lambda t}-1)  -  i \lambda = 0
\ ,
\]
per cui concludiamo che
\begin{equation}
\label{eq.scopug3}
\lim_{t \to 0} t^{-1}(U_t-1) u = iTu
\ \ , \ \
\forall u \in D
\ ,
\end{equation}
il che vuol dire che possiamo esprimere $T$ in termini del suo gruppo ad un parametro. In realt\'a, si dimostra il seguente risultato:

\begin{thm}[Stone]
Ogni famiglia di operatori unitari $\{ U_t \}_{t \in \bR} \subset B(\mH)$ che soddisfi (\ref{eq.scopug}) e (\ref{eq.scopug2}) \e il gruppo ad un parametro di un operatore autoaggiunto $(D,T)$.
\end{thm}

\begin{proof}[Sketch della dimostrazione]
Occorre, prima di tutto, mostrare che l'insieme 
\[
D := \{ u \in \mH : \exists \lim_{t \to 0} t^{-1}(U_t-1) u \in \mH \}
\]
\e un sottospazio denso di $\mH$; si tratta poi di verificare che {\em definendo} $Tu$, $u \in D$, a partire da (\ref{eq.scopug3}) si ottiene effettivamente un operatore autoaggiunto. Per dettagli rimandiamo a \cite[Theorem VIII.8]{RS}. 
\end{proof}

\begin{ex}[La derivata]
\label{ex_der}
{\it
Consideriamo la famiglia $U = \{ U_t \}$ dell'Esempio \ref{ex.trans}; allora \e immediato verificare che $U$ soddisfa (\ref{eq.scopug}), e del resto
\[
\| U_tv-v \|_2^2 
\ = \ 
\int |v(t+s)-v(s)|^2 ds 
\ \stackrel{Es.\ref{sec_Lp}.2}{\to} \
0
\ \ , \ \
\forall v \in \mH
\ .
\]
Dunque $U$ soddisfa le ipotesi del teorema di Stone. Consideriamo ora lo spazio delle funzioni complesse a decrescenza rapida 
\[
\mS(\bR,\bC) := \{ f + ig : f,g \in \mS(\bR) \}
\ .
\]
Allora per ogni $u \in D := \mS(\bR,\bC)$ troviamo
\[
\lim_{t \to 0} \{ t^{-1}(U_t-1) u \}(s) 
\ = \ 
\lim_{t \to 0} t^{-1} \{ u(t+s) - u(s) \} 
\ = \ 
u'(s)
\ \ , \ \
\forall s \in \bR
\ ;
\]
poich\'e $u' \in \mS(\bR,\bC) \subset \mH$, da (\ref{eq.scopug3}) concludiamo subito che l'operatore associato ad $U$ \e l'estensione autoaggiunta di $(D,T)$, $Tu := -iu'$, $u \in D$.
}
\end{ex}

I gruppi ad un parametro hanno importanti applicazioni in meccanica quantistica. Conside- riamo lo spazio di Hilbert $\mH := L^2(\bR^d,\bC)$ ed un operatore autoaggiunto $(D,T)$ su $\mH$. L'{\em equazione di Schroedinger} associata a $T$ con condizione iniziale $\psi_0 \in D$ \e data dal problema differenziale
\begin{equation}
\label{Schroedinger}
\frac{\partial \psi}{\partial t}
\ = \
iT \psi
\ \ , \ \
\psi (\cdot,t_0) = \psi_0
\ ,
\end{equation}
la cui funzione incognita $\psi : \bR^d \times \bR \to \bC$ \e tale che $\psi (\cdot,t) \in D$ per ogni $t \in \bR$. Denotato con $U = \{ U_t \}$ il gruppo a un parametro associato a $(D,T)$, usando (\ref{eq.scopug3}) concludiamo che (\ref{Schroedinger}) ammette soluzione
\begin{equation}
\label{sol.Schroedinger}
\psi (x,t) := \{ U_t \psi_0 \} (x)
\ \ , \ \
\forall x \in \bR^d
\ , \
t \in \bR
\ .
\end{equation}
Portiamo ad esempio l'equazione di Schroedinger per la particella libera, che si ottiene per $d = 3$ e $(D,T)$ l'estensione autoaggiunta dell'operatore di Laplace $- \Delta$ dell'Esempio \ref{ex_laplace}; osserviamo che in tal caso la soluzione \e calcolabile usando la trasformata di Fourier (\S \ref{sec_fou_tra}),
\begin{equation}
\label{sol.Schroedinger2}
\psi (x,t) 
:= 
\int_{\bR^3} \wa \psi_0(\lambda) e^{-i |\lambda|^2t + \lambda \cdot x} \ d \lambda
\ \ , \ \
x \in \bR^3
\ , \
t \in \bR
\end{equation}
(per una verifica si derivi sotto il segno di integrale).
Non sempre per\'o ci troviamo in situazioni cos\'i comode, cosicch\'e diventa essenziale sfruttare le propriet\'a astratte dei gruppi ad un parametro; ed in tal caso, come passo intermedio \e necessario stabilire se il nostro operatore $(D,T)$, presentato tipicamente in forma differenziale, sia (essenzialmente) autoaggiunto. Come esempio menzioniamo l'equazione di Schroedinger per l'atomo con $n$ elettroni, che si ottiene per
\[
D := C_0^\infty(\bR^{3n},\bC)
\ \ , \ \
Tu := - \Delta u + fu
\ \ , \ \
u \in D
\ ,
\]
dove $f$ \e un termine associato al potenziale elettrico,
\[
f(r_1 , \ldots , r_n) 
\ := \ 
- n \sum_{k=1}^n |r_k|^{-1} + \sum_{k < h} |r_h - r_k|^{-1}
\ \ , \ \
\forall r_k \in \bR^3 - \{ 0 \}
\ , \
k = 1 , \ldots , n
\ .
\]
Ovviamente il calcolo esplicito di una soluzione \e ora pi\'u difficoltoso, ed il fatto che $(D,T)$ \e essenzialmente autoaggiunto, acclarato dal Teorema di Kato-Rellich, \e un risultato non banale, per la cui dimostrazione rimandiamo a \cite[Theorem X.16]{RS2}.
Una volta mostrata l'essenziale autoaggiuntezza di $(D,T)$ possiamo calcolare esplicitamente il relativo gruppo ad un parametro attraverso la {\em formula di Trotter-Kato} (vedi \cite[\S VIII.8]{RS}):
\begin{thm}
Siano $(D,H)$, $(D',V)$ autoaggiunti sullo spazio di Hilbert $\mH$ e tali che $(D \cap D',H+V)$ sia essenzialmente autoaggiunto. Allora, per ogni $t \in \bR$,
\[
e_t(H+V)u \ = \ \lim_{k \to \infty} \, \{  e_t(H/k) \, e_t(V/k) \}^k u
\ \ , \ \
\forall u \in \mH
\ .
\]
\end{thm}

Osserviamo che la formula precedente \e tutt'altro che banale, causa il fatto che $(D,H)$, $(D',V)$ non necessariamente commutano. Il teorema precedente si applica all'atomo con $n$ elettroni ponendo $(D,H) = ( \mS(\bR^{3n},\bC) , - \Delta )$ e definendo $(D',V)$ come l'operatore $Vu := f u$, $\forall u \in D' := \{ u \in \mH : fu \in \mH \}$. L'espressione esplicita di $e_t(H/k)$ si ottiene immediatamente da (\ref{sol.Schroedinger2}),
\[
\{ e_t(H/k) u \}(x)
\ = \
\int_{\bR^{3n}} \wa u(\lambda) e^{-i |\lambda|^2 k^{-1} t  +  \lambda \cdot x} \ d \lambda
\ \ , \ \
x \in \bR^{3n}
\ ,
\]
mentre $e_t(V/k)u = \exp(itf / k) u$; 
%
%
in tal modo, calcolando il limite del teorema precedente otteniamo la famosa {\em formula di Feynman-Kac} (vedi \cite[\S X.11]{RS2}).

\subsection{Il Teorema di Schauder.}
\label{sec_gfp}

Come rimarcato in \S \ref{sec_fp} i teoremi di punto fisso rivestono una grossa importanza in analisi: soluzioni di problemi integro-differenziali si possono presentare come punti fissi di applicazioni continue definite su spazi di funzioni, come accade ad esempio nel caso del problema di Cauchy rispetto all'operatore di Volterra (\ref{eq_tc1}).
In questa sezione facciamo una breve rassegna dei teoremi pi\'u noti e delle loro applicazioni. Enfatizziamo il fatto che la convessit\'a svolge un ruolo di primo piano per la dimostrazione dei risultati che seguono.

\begin{thm}[Schauder]
Sia $\mE$ uno spazio di Banach e $C \subset \mE$ un convesso compatto e non vuoto. Allora ogni applicazione continua $f : C \to C$ ammette \textbf{almeno} un punto fisso.
\end{thm}

\begin{proof}[Dimostrazione]
Preso $\eps > 0$ consideriamo il ricoprimento $X_\eps := \{ \Delta (v,\eps) : v \in C \}$ ed estraiamo un sottoricoprimento finito $X_{\eps,n} := \{ \Delta (v_k,\eps) : k = 1,\ldots,n \}$. 
Definiamo quindi le funzioni continue
\[
g_1, \ldots, g_n : C \to \bR^+
\ \ , \ \
g_k(v):=
\begin{cases}
\eps - \| v-v_k \|, & \;\;\textrm{se $\| v - v_k \| \leq \eps$} 
\\
0, & \;\;\textrm{se $\| v-v_k \| \geq \eps$}
\end{cases}
\]
Osserviamo che $\sum_{k=1}^n g_k(v) > 0$, $v \in C$, per cui definiamo
\[
G_n : C \to C_n
\ \ , \ \
G_n(v) := \frac{\sum_{k=1}^n g_k(v) v_k}{\sum_{k=1}^n g_k(v)}
\ ,
\]
dove $C_n \subseteq C$ \e l'inviluppo convesso di $\{ v_1 , \ldots , v_n \}$. Osserviamo che
\begin{equation}
\label{eq_SCH01}
\| G_n(v)-v \| \leq \eps \ \ , \ \  v \in C \ .
\end{equation}
Consideriamo ora la funzione continua $f_n := G_n \circ f : C \to C_n$ e la restrizione $\ovl f_n : C_n \to C_n$; essendo $C_n$ compatto, convesso, e contenuto nello spazio a dimensione finita generato da $v_1 , \ldots , v_n$, per il teorema di Brouwer troviamo che esiste $w \in C_n$ tale che
\[
G_n \circ f (w) = f_n(w) = \ovl f_n (w) = w \ .
\]
Inoltre, grazie a (\ref{eq_SCH01}) troviamo
\[
\| f(w) - w \| = \| f(w) - G_n \circ f (w)) \| \leq \eps \ .
\]
Osserviamo che $w$ dipende da $\eps$. Scegliendo $\eps = 1/m$, $m \in \bN$, scriviamo $w \equiv w_m$ ed osserviamo che le considerazioni precedenti implicano che
\begin{displaymath}
\forall m \in \bN \ \exists w_m \in C  \ : \ \| f(w_m) - w_m \| \leq m^{-1} \ .
\end{displaymath}
Per compattezza di $C$, esiste una sottosuccessione $w_{m_i}$ tale che $\{ f(w_{m_i}) \}$ converge ad un $w_0 \in C$. La stima
\[
\| w_{m_i} - w_0 \|
\leq
\| w_{m_i} - f(w_{m_i}) \|  +  \| f(w_{m_i}) - w_0 \|
\]
mostra che anche $\{ w_{m_i} \}$ converge a $w_0$; per cui, essendo $f$ continua troviamo $f(w_0) = w_0$.
\end{proof}

\begin{rem}{\it
Usando lo stesso argomento della dimostrazione precedente si ha la seguente versione del teorema di Schauder: se $C \subset \mE$ un convesso, chiuso, limitato e non vuoto, allora ogni applicazione compatta $f : C \to C$ ammette un punto fisso.
}
\end{rem}

Tra le applicazioni del teorema di Schauder, segnaliamo:
\begin{enumerate}
\item {\em Una dimostrazione alternativa del Teorema di Peano}. Con le notazioni di Teo.\ref{thm_P}, poniamo $I := [t_0-r,t_0+r]$, consideriamo lo spazio di Banach $\mE := C(I)$ e definiamo $C := \{ u \in \mE : \| u-u_0  \|_\infty \leq r  \}$. Si verifica facilmente che $C$ \e convesso, chiuso e limitato. A questo punto, si tratta di verificare che l'operatore di Volterra
\[
Fu(t) := u_0 + \int_{t_0}^t f ( s , u(s) ) \ ds \ \ , \ \ u \in C \ ,
\]
definisce un'applicazione compatta da $C$ in s\'e. Il punto fisso di $F$ fornisce la soluzione cercata.
\item {\em Teoremi di esistenza di soluzioni per equazioni integrali}. Un'{\em equazione di Urysohn di seconda specie} \e un problema del tipo
\begin{equation}
\label{eq_URI}
u(t) = \lambda \int_a^b \varphi(t,s,u(s)) \ ds
\ \ , \ \
\varphi \in C(A)
\ \ , \ \
A := [a,b] \times [a,b] \times [-r,r]
\ .
\end{equation}
Poniamo $\mE := C([a,b])$ e $C := \{ u \in \mE : \| u \|_\infty \leq r \}$. Al solito, occorre verificare che $C$ \e convesso, chiuso e limitato, e che l'operatore 
\[
Fu(t) := \lambda \int_a^b \varphi(t,s,u(s)) \ ds \ \ , \ \ u \in C \ ,
\]
sia un'applicazione compatta da $C$ in s\'e: in effetti, cos\'i \e  se $\lambda \leq r \| \varphi \|_\infty^{-1} (b-a)^{-1}$ 
(vedi \cite[\S 16.1]{PM} o \cite[\S 3]{Gra}).
%
%
%
%
\end{enumerate}

Menzioniamo infine il seguente risultato:
\begin{thm}\textbf{(Leray-Schauder-Tychonoff, \cite[Theorem V.19]{RS}).}
Sia $S$ un convesso compatto e non vuoto in uno spazio localmente convesso $V$. Allora ogni applicazione continua $f : S \to S$ ammette punti fissi.
\end{thm}

\begin{ex}{\it
La compattezza \e una condizione essenziale per i teoremi precedenti; ad esempio, l'applicazione
\[
F : l^2_{\leq 1} \to l^2_{\leq 1}
\ \ , \ \
x := \{ x_n \} \mapsto Fx 
\ : \ 
\left\{
\begin{array}{ll}
(Fx)_0 := \sqrt{ 1 - \| x \|_2^2} 
\\
(Fx)_n := x_{n-1}
\end{array}
\right.
\]
\'e continua, ma evidentemente priva di punti fissi.
}
\end{ex}

\subsection{Esercizi.}

\textbf{Esercizio \ref{sec_afunct}.1.} {\it
Sia $\mE := C([0,1])$. Fissato $\alpha \in [0,1]$, sia $T_\alpha : D(T_\alpha) \subset \mE \to \mE$ l'operatore lineare definito da
\[
T_\alpha f (x) = \int_0^x y^{-\alpha} f(y) \ dy
\ \ , \ \
x \in [0,1]
\ ,
\]
dove $D(T_\alpha) := \left\{ f \in \mE : x^{-\alpha} f(x) \in L^1 \right\}$. Stabilire se le seguenti affermazioni sono vere o false: (a) $T_1$ non \e continuo; (b) $T_\alpha$ \e compatto per $\alpha \in [0,1)$. Inoltre, determinare lo spettro di $T_\alpha$. (Suggerimento: si applichino i teoremi di Ascoli-Arzel\'a e Fredholm).
}

\

\noindent \textbf{Esercizio \ref{sec_afunct}.2.} {\it
Sia $n \in \bN$ ed $X \subseteq \bR^n$, $\dot{X} \neq \emptyset$, equipaggiato della misura di Lebesgue $\mu$. Preso $x \in X$, dimostrare che la delta di Dirac $\delta_x \in C_0(X)^*$ si estende ad un funzionale $F \in L^{\infty,*}_\mu(X)$, e che $F$ non appartiene all'immagine dell'applicazione canonica (\ref{def_Fg}).
}

\

\noindent {\it Sketch della soluzione.} Applicando il teorema di Hahn-Banach troviamo che esiste il funzionale $F \in L^{\infty,*}_\mu(X)$ cercato. Supponiamo ora per assurdo che $F = F_g$ per qualche $g \in L^1_\mu(X)$, ovvero
\[
\left \langle F , f \right \rangle 
= 
\int fg \ \ , \ \ f \in L^\infty_\mu(X) \ .
\]
Presa una successione $\{ f_n \} \subset C_c(X)$, $0 \leq f_n \leq 1$, $f_n(x) \equiv 1$, tale che $f_n \to \chi_{ \{ x \} }$ puntualmente, per convergenza di Lebesgue ($|f_ng| \leq |g|$, $n \in \bN$) troviamo
\[
1 = 
f_n(x) = 
\left \langle F , f_n \right \rangle = 
\int f_n g \ d \mu
\ \stackrel{n}{\to} \ 0
\ ,
\]
il che \e assurdo. Osserviamo che $F = F_g$ implicherebbe che $\mu_x \prec \mu$, dove $\mu_x$ \e la misura di Dirac, il che \e assurdo (vedi Esempio \ref{ex_dirac}).

\

\noindent \textbf{Esercizio \ref{sec_afunct}.3 (Schauder).} {\it
Siano $\mE,\mF$ spazi di Banach e $T \in K(\mE,\mF)$. 
\textbf{(1)} Si ponga $X := \ovl{T(\mE_{\leq 1})}$ (chiusura nella topologia della norma) e si mostri che $X$ \e uno spazio metrico compatto.
\textbf{(2)} Presa una successione $\{ f_n \} \subset \mF^*_{\leq 1}$, si mostri che la famiglia 
\[
\{ \varphi_n \} \subset C(X) 
\ \ , \ \ 
\varphi_n (x) := \left \langle f_n , x \right \rangle 
\ \ , \ \ 
x \in X 
\ ,
\]
\e equilimitata ed equicontinua.
\textbf{(3)} Usando i punti precedenti, si mostri $\{ T^* f_n  \} \subset \mE^*$ ammette una sottosuccessione convergente, cosicch\'e $T^*$ \e compatto.
\textbf{(4)} Si prenda ora un generico $T \in B(\mE,\mF)$ e si mostri che se $T^* \in K(\mF^*,\mE^*)$ allora $T$ \e compatto. 
\textbf{(5)} Si concluda che $T \in K(\mE,\mF)$ se e solo se $T^* \in K(\mF^*,\mE^*)$.
}

\

\noindent {\it Soluzione.} 
\textbf{(1)} Segue dalla compattezza di $T$.
\textbf{(2)} Visto che
          $| \left \langle f_n , Tu \right \rangle | \leq \| f_n \| \| T \| \| u \| \leq \| T \|$
          per ogni $u \in \mE_{\leq 1}$, troviamo
          $\| \varphi_n \|_\infty \leq \| T \|$
          e
          $| \varphi_n(x)-\varphi_n(x') | \leq |x-x'|$
          per ogni $n \in \bN$ ed $x,x' \in X$.
\textbf{(3)} Usando Ascoli-Arzel\'a troviamo una sottosuccessione $\{ \varphi_{n_k} \}$ 
          uniformemente convergente, per cui
          \[
          0 \stackrel{h,k}{\leftarrow}
          \| \varphi_{n_h} - \varphi_{n_k}  \|_\infty   
          \ \geq \
          \sup_{u \in \mE_{\leq 1}} \ 
          | \left \langle f_{n_h} , Tu \right \rangle  -  \left \langle f_{n_k} , Tu \right \rangle |
          \ = \
          \| T^* f_{n_h} - T^* f_{n_k} \|
          \ .
          \]
\textbf{(4-5)} Si argomenta in modo analogo ai punti precedenti, usando stavolta successioni 
          $\{ v_n \} \subset \mE_{\leq 1}$.

\

\noindent \textbf{Esercizio \ref{sec_afunct}.4 (Riesz).} {\it (1) Sia $\mE$ uno spazio normato ed $M \subset \mE$ chiuso. Allora per ogni $\eps > 0$ esiste $u \in \mE_1$ tale che 
\begin{equation}
\label{eq_lem_riesz}
d(u,M) := \inf_{w \in M} \| u-w \| \ \geq \ 1-\eps \ .
\end{equation}
(2) Usando il punto precedente, si concluda che se $\mE_1$ \e compatto nella topologia della norma allora $\mE$ ha dimensione finita. 
}

\

\noindent {\it Soluzione.} (1) Sia $v \in \mE-M$. Poich\'e $M$ \e chiuso (e diverso da $\mE$) abbiamo $d := d(v,M) > 0$. Per costruzione, per ogni $\eps > 0$ esiste $w \in M$ tale che $d \leq \| v-w \| \leq (1-\eps)^{-1}d$. Ponendo $u := (v-w) / \| v-w \|$ si trova il vettore cercato. (2) Supponiamo per assurdo che $\mE$ abbia dimensione infinita: allora esiste una successione di sottospazi propri $\mV_n \subset \mV_{n+1}$, $n \in \bN$. Usando il punto precedente riusciamo a costruire una successione $\{ u_n \in (\mV_n)_1 \}$ con $d(u_{n+1},\mV_n) \geq 1/2$ e quindi $\| u_n -u_m \| \geq 1/2$, $m < n$. Ci\'o contraddice l'ipotesi di compattezza di $\mE_1$.

\

\noindent \textbf{Esercizio \ref{sec_afunct}.5.} {\it Si consideri la seguente successione $\{ f_n \} \subset \bigcap_{p \in [1,\infty]} L^p(\bR)$:
\[
f_n(x) :=
\left\{
\begin{array}{ll}
2^{-n} \ \ , \ \ x \in [ 2^n,2^{n+1} ]
\\
0  \ \ , \ \  {altrimenti.}
\end{array}
\right.
\]
Si mostri che: (1) $\| f_n \|_p \to 0$ per ogni $p \in (1,\infty]$; (2) $\{ f_n \}$ non converge debolmente in $L^1(\bR)$.
}

\

\noindent \textbf{Esercizio \ref{sec_afunct}.6 (\cite[Ex.5.52]{Can}).} {\it Si consideri la successione $\{ x^{(n)} \} \subset l^\infty$ definita da
\[
x^{(n)}_k :=
\left\{
\begin{array}{ll}
0 \ \ , \ \ k \leq n
\\
1 \ \ , \ \ k > n \ .
\end{array}
\right.
\]
Si dimostri che $\{ x^{(n)} \}$ converge nella topologia $*$-debole, ma non in quella debole. (Suggerimento: riguardo la convergenza debole, si applichi il Teorema di Hahn-Banach al sottospazio $\mE := \{ x = \{ x_k \} \in l^\infty : \exists \lim_k x_k \}$ ed al funzionale $F \in \mE^*$, $\left \langle F , x \right \rangle := \lim_k x_k$).
}

\

\noindent \textbf{Esercizio \ref{sec_afunct}.7 (Ulteriori propriet\'a dello shift).} {\it
Dato l'operatore $S \in B(l^2)$ definito in (\ref{eq_shift}), si dimostrino le seguenti propriet\'a:
\textbf{(1)} $S^*x = ( x_2 , x_3 , \ldots  )$, $x := ( x_1 , x_2 , \ldots ) \in l^2$;
\textbf{(2)} $S^*S = 1$, dove $1 \in B(l^2)$ \e l'identit\'a;
\textbf{(3)} Per ogni $x,y \in l^2$, risulta $\lim_n (x,S^ny) = 0$. Si interpreti ci\'o in termini della topologia debole di $B(l^2)$ (vedi Esempio \ref{ex_lc}(4)).
\textbf{(4)} Si mostri che, preso $n \in \bN$, risulta
\[
{\mathrm{ind}} S^n = -n
\ \ , \ \
{\mathrm{ind}}(S^*)^n = n
\ ,
\]
dove {\em ind} denota l'indice nel senso di (\ref{eq.index}).
}

\

\noindent \textbf{Esercizio \ref{sec_afunct}.8 (Misure e rappresentazioni).} {\it Sia $X$ uno spazio compatto di Hausdorff. Presa una misura di Radon $\mu$ su $X$ consideriamo lo spazio di Hilbert $\mH_\mu := L_\mu^2(X)$ e, per ogni $f \in C(X)$, definiamo l'operatore lineare
\begin{equation}
\label{eq_muf}
\pi_\mu(f) : \mH_\mu \to \mH_\mu \ : \ \pi_\mu (f) v := fv \ \ , \ \ v \in \mH_\mu \ .
\end{equation}
\textbf{(1)} Si mostri che 
$\| \pi_\mu(f) \| := \sup_{v \in \mH_{\mu,\leq 1}} \| \pi_\mu(f)v \| \leq \| f \|_\infty$;
\textbf{(2)} Si verifichi che (\ref{eq_muf}) definisce una rappresentazione
%
%
%
\[
\pi_\mu : C(X) \to B(\mH_\mu)
\]
di $C(X)$ (ovvero, si mostri che $\pi_\mu$ \e lineare e che $\pi_\mu(fg) = \pi_\mu(f)\pi_\mu(g)$, $\forall f,g \in C(X)$);
\textbf{(3)} Si verifichi che, preso $v \in \mH_\mu$, l'insieme $[v] := \{ \pi_\mu(f) v , f \in C(X) \}$ \e un sottospazio vettoriale di $\mH_\mu$;
\textbf{(4)} Sia $\nu$ una misura di Radon assolutamente continua rispetto a $\mu$; si dimostri che esiste un operatore lineare $T \in B(\mH_\nu,\mH_\mu)$ tale che
\[
\left\{
\begin{array}{ll}
\| Tw \| = \| w \| \ \ , \ \ \forall w \in \mH_\nu \ ,
\\
T \circ \pi_\nu(f) = \pi_\mu(f) \circ T \ \ ,\ \ \forall f \in C(X) \ .
\end{array}
\right.
\]
\textbf{(5)} Si mostri che $\mH_\mu$ ha dimensione $1$ nel caso in cui $\mu$ sia una misura di Dirac.
\textbf{(6)} Supposto che $\mu$ sia regolare esterna, si mostri che esiste $v_0 \in \mH_\mu$ tale che $[v_0]$ \e denso in $\mH_\mu$.
%

\

\noindent (Suggerimenti: \textbf{(4)} Il teorema di Radon-Nikodym fornisce una funzione positiva $F \in L_\mu^1(X)$ tale che $\nu E = \int_E F \ d \mu$, $E \in \mM$; si osservi quindi che $F^{1/2}w \in \mH_\mu$ per ogni $w \in \mH_\nu$; \textbf{(6)} Si prenda la funzione $v_0(x) := 1$, $x \in X$, e si usi Prop.\ref{prop_cLp}.)
}

\

\noindent \textbf{Esercizio \ref{sec_afunct}.9}.
{\it
Sia $X$ uno spazio compatto e di Hausdorff, ed $\mA := C(X,\bC)$. Si mostrino i seguenti punti:
\textbf{(1)} Data $f \in C(X)$, si ha $\sigma(f) = f(X)$ (ovvero, lo spettro di $f$ coincide con la sua immagine);
\textbf{(2)} per ogni $x \in X$, si mostri che l'applicazione
\[
\omega_x : \mA \to \bC
\ \ , \ \
\left \langle \omega_x , f \right \rangle := f(x) \ , \ \forall f \in \mA \ ,
\]
\'e un carattere di $\mA$;
\textbf{(3)} Usando i punti precedenti, si mostri che l'applicazione 
$X \to \wa \mA$, $x \mapsto \omega_x$,
\'e un omeomorfismo.
}

\

\noindent \textbf{Esercizio \ref{sec_afunct}.10}.
{\it
Consideriamo la $*$-algebra di Banach $\mA := L^1(\bR,\bC)$ (vedi Esempio \ref{ex_conv}).  \\
\textbf{(1)} Per ogni $\lambda \in \bT$ si mostri che l'applicazione
\[
\omega_\lambda : \mA \to \bC
\ \ : \ \ 
\left \langle \omega_\lambda , f \right \rangle
:=
\int f(t) \lambda^{-t} dt
\ , \
\forall f \in \mA
\]
\'e un carattere di $\mA$ e si verifichi che $\| \omega_\lambda \| \leq 1$.
\textbf{(2)} Data $f \in \mA$, verificare che la funzione 
$\wa f (\lambda) := \left \langle \omega_\lambda , f \right \rangle$,
$\lambda \in \bT$, 
\'e continua e limitata.
\textbf{(3)} Presa $g \in L^\infty(\bR,\bC)$ si consideri $F_g \in \mA^*$ e si mostri che affinch\'e si abbia $F_g \in \wa \mA$, deve essere 
$\ovl{g(t)} = g(-t)$, $g(t+s) = g(t)g(s)$, per ogni $t,s \in \bR$. Si concluda che $g(t) = \lambda^t$ per qualche $\lambda \in \bT$.

\

\noindent (Suggerimenti: per il punto (1) si usino i Teoremi di Fubini-Tonelli ed il Teorema di convergenza dominata. Per il punto (2) si mostri che l'applicazione $\{ \bT \ni \lambda \mapsto \omega_\lambda \in \wa \mA \}$ \e continua e poi si usi la definizione di topologia $*$-debole. Per il punto (3) si scriva esplicitamente $\left \langle F_g, f \right \rangle := \int f \ovl g$ e si impongano le condizioni (\ref{def_char}).)
}

\

\noindent \textbf{Esercizio \ref{sec_afunct}.11}.
{\it
Sia $\{ f_n \} \subset L^1(\bR)$ un'identit\'a approssimata (vedi Def.\ref{def_moll}). 
\textbf{(1)} Si mostri che i funzionali
\[
\left \langle F_n , g \right \rangle := \int f_ng
\ \ , \ \
g \in C_0(\bR)
\ ,
\]
sono lineari e continui, ovvero $\{ F_n \} \subset C_0(\bR)^*$.
\textbf{(2)} Si verifichi che $F_n \stackrel{*}{\to} \delta_0$, dove $\delta_0 \in C_0(\bR)^*$ \e la delta di Dirac.
}

\

\noindent \textbf{Esercizio \ref{sec_afunct}.12}.
{\it
Si fissi $p = 2$ nell'Esempio \ref{ex_Ta} e si calcoli lo spettro dell'operatore $T_a \in B(l^2_\bC)$, $a \in l^\infty_\bC$. Si trovino le condizioni su $a$ tali che $T_a$ sia compatto.
}

\

\noindent \textbf{Esercizio \ref{sec_afunct}.13 (Commutanti e algebre di von Neumann)}.
{\it
Sia $\mH$ uno spazio di Hilbert complesso. Preso un sottoinsieme $\mS \subseteq B(\mH)$, definiamo il \textbf{commutante di $\mS$} 
\[
\mS' := \{ T \in B(\mH) : TA = AT \ , \ \forall A \in \mS \} \ .
\]
Si dimostri che:
\textbf{(1)} $\mS'$ \e un'algebra;
\textbf{(2)} $\mS'$ \e chiuso rispetto alla topologia debole di $B(\mH)$ (nel senso dell'Esempio \ref{ex_lc}(4));
\textbf{(3)} se $\mS = \mS^* := \{ A^* : A \in \mS \}$ allora $\mS'$ \e una $*$-algebra e quindi, per il punto precedente, un'algebra di von Neumann.

\

\noindent (Suggerimenti: riguardo il punto (2), occorre verificare che se $T \in B(\mH)$ appartiene alla chiusura debole di $\mS'$ allora $TA = AT$ per ogni $A \in \mS$, ovvero $(u,ATv) = (u,TAv)$ per ogni $u,v \in \mH$. Del resto, per ipotesi esiste una successione $\{ T_n \} \subseteq \mS'$ tale che $\lim_n (u,T_nv) = (u,Tv)$ per ogni $u,v$.)

}

\

\noindent \textbf{Esercizio \ref{sec_afunct}.14 (Algebre di von Neumann commutative)}.
{\it 
Sia $(X,\mM,\mu)$ uno spazio di misura finita ed $\mH := L_\mu^2(X,\bC)$. Si consideri la rappresentazione
\[
\pi : L_\mu^\infty(X,\bC) \to B(\mH)
\ \ , \ \
f \mapsto \pi_f 
\ : \
\pi_fu := fu \ , \ \forall u \in \mH
\ .
\]
\textbf{(1)} Si verifichi che $\pi$ \e iniettiva, e posto $\mR := \pi(L_\mu^\infty(X,\bC))$ si verifichi che $\mR$ \e una $*$-sottoalgebra di $B(\mH)$;
\textbf{(2)} Si consideri la funzione $u_0 \in \mH$, $u_0(x) := 1$, $\forall x \in X$, e preso $T \in B(\mH)$ si ponga $\varphi := Tu_0 \in \mH$. Si dimostri che se $T \in \mR'$ (ovvero $T \pi_f = \pi_f T$, $\forall f \in L_\mu^\infty(X,\bC)$), allora
\[
Tf = f \varphi
\ \ , \ \
\forall f \in L_\mu^\infty(X,\bC) \subset \mH
\ .
\]
\textbf{(3)} Si mostri che in realt\'a $\varphi \in L_\mu^\infty(X,\bC)$, con $\| \varphi \|_\infty \leq \| T \|$.
\textbf{(4)} Si mostri che $T = \pi_\varphi$, cosicch\'e $\mR' \subseteq \mR$.
\textbf{(5)} Si verifichi che $\mR \subseteq \mR'$, cosicch\'e $\mR = \mR'$, e si concluda che $\mR$ \e un'algebra di von Neumann.

\

\noindent (Suggerimenti: 
(1) Se $\pi_f = 0$ allora in particolare $\pi_fu_0 = f = 0$;
(2) Imponendo la condizione $T \pi_f = \pi_f T$ si trova $\pi_f Tu_0 = T \pi_f u_0 = Tf$;
(3) Scrivendo esplicitamente la disuguaglianza
    $\| Tf \|_2^2 \leq \| T \|^2 \| f \|_2^2$ 
    si trova
    \[
    0 \leq \int |\varphi|^2 |f|^2 \leq \| T \|^2 \int |f|^2
    \ \ , \ \
    \forall f \in L_\mu^\infty(X,\bC) \subset \mH
    \ ,
    \]
    e da quest'ultima segue facilmente la stima cercata;
(4) Si usi il punto (2) e la densit\'a di $L_\mu^\infty(X,\bC)$ in $\mH$ 
    (vedi Prop.\ref{prop_sLp});
(5) Per l'inclusione $\mR \subseteq \mR'$ si usi il fatto che $L_\mu^\infty(X,\bC)$
    \e un'algebra commutativa, e per il fatto che $\mR$ \e chiusa nella topologia
    debole si usi l'esercizio \ref{sec_afunct}.13.)

}

\

\noindent \textbf{Esercizio \ref{sec_afunct}.15 (Le relazioni di Heisenberg)}.
{\it Sia $\mH$ uno spazio di Hilbert complesso. 
\textbf{(1)} Preso $T=T^* \in B(\mH)$, mostrare che
\[
\| T^{2^k} \| = \| T \|^{2^k}
\ \ , \ \
\forall k \in \bN
\ .
\]
\textbf{(2)} Usando il punto precedente, dimostrare che non possono esistere $P,Q \in B(\mH)$ autoaggiunti che soddisfino le relazioni di Heisenberg
\begin{equation}
\label{eq.comm}
PQ - QP = -i1 \ .
\end{equation}
\textbf{(3)} Esibire uno spazio di Hilbert complesso $\mH$ con un sottospazio denso $D \subseteq \mH$ ed operatori simmetrici $(D,P)$, $(D,Q)$ tali che 
\begin{equation}
\label{eq.comm2}
\{ PQ - QP \} u = -iu \ \ , \ \ \forall u \in D \ .
\end{equation}
(Suggerimenti: riguardo (1) si osservi che l'identit\'a da dimostrare \e certamente vera nel caso $k=1$, in quanto $\| T^2 \| = \| T^*T \|$ (vedi (\ref{eq_idc})), e si proceda per induzione.
Riguardo (2) si argomenti per assurdo e si osservi che (\ref{eq.comm}) implica 
$PQ^n - Q^nP = -inQ^{n-1}$,
$\forall n \in \bN$,
cosicch\'e
\[
(2^k+1) \| Q^{2^k} \| 
\ \stackrel{(1)}{=} \
(2^k+1) \| Q \|^{2^k}
\ \leq \
2 \| P \| \| Q \|^{2^k+1}
\ \ , \ \
\forall k \in \bN
\ .
\]
Riguardo (3) si prenda $\mH = L^2(\bR,\bC)$, 
$D = \mS(\bR,\bC) := \{ f + ig : f,g \in \mS(\bR) \}$,
e
\[
Qu(x) := x u(x)
\ , \
x \in \bR
\ \ , \ \
Pu := -iu'
\ \ ,  \ \
u \in D
\ .
\]
}

\noindent \textbf{Esercizio \ref{sec_afunct}.16}
{\it
Si consideri l'operatore autoaggiunto $(D,T)$ ed il relativo calcolo funzionale boreliano del Teorema \ref{thm_sonl}. Presa una successione $\{ f_n \} \subset B^\infty(\sigma(D,T),\bC)$ limitata in norma $\| \cdot \|_\infty$ e convergente puntualmente a $f \in B^\infty(\sigma(D,T),\bC)$, si dimostri che
$\{ f_n(T) \}$
converge a $f(T)$ nella topologia forte di $B(\mH)$ (vedi Esempio \ref{ex_lc}(5)).

\

\noindent (Suggerimenti: occorre verificare che 
$\lim_n \| \{ f(T)-f_n(T) \} u \| = 0$, $\forall u \in B(\mH)$;
d'altra parte
\[
\| \{ f(T)-f_n(T) \} u \|^2 \ = \ 
( u , \{ f(T)-f_n(T) \}^* \{ f(T)-f_n(T) \} u  ) \ = \
\int_{\sigma(D,T)} g_n \ d \mu_{uu} \ ,
\]
dove 
\[
g_n := (f-f_n)^*(f-f_n) \in B^\infty(\sigma(D,T),\bC)
\ \ , \ \
\forall n \in \bN
\ ,
\]
tende puntualmente a $0$. Per cui, avendo $\sigma(D,T)$ misura $\mu_{uu}$-finita, basta invocare il teorema di convergenza limitata Prop.\ref{prop_conv_lim}.)
}

\

\noindent \textbf{Esercizio \ref{sec_afunct}.17.}
{\it
Si consideri lo spazio di Hilbert $\mH := L^2([0,1],\bC)$, l'operatore $(D,T)$ dell'Esempio \ref{eq.der.I}, e l'operatore di Volterra $F \in B(\mH)$ (Esempio \ref{ex_volterra}). Si consideri poi la funzione $u_0(x) := 1$, $\forall x \in [0,1]$, e si osservi che 
$(u_0,u) = \int u$, $\forall u \in \mH$.
\textbf{(1)} Si mostri che $Fg \in D$ per ogni 
$g \in V := \{ g \in \mH : (u_0,g) = 0 \}$, 
e che se $f \in \mH$ \e ortogonale a tutte le funzioni in $V$ allora $f$ \e costante;
\textbf{(2)} Usando il punto precedente e (\ref{eq_ff}), si mostri che
\[
(FT^*v,g) = (v,g)
\ , \
\forall v \in D^*
\ , \
g \in V
\ ;
\]
\textbf{(3)} Si mostri quindi che per ogni $v \in D^*$ esiste $c \in \bC$ tale che
$v = FT^*v + c$, 
cosicch\'e se $v \in D^*$ allora $v$ \e la primitiva di una funzione in $\mH$;
\textbf{(4)} Si mostri che
\[
v(1) - v(0) \ = \ i (v',u_0) \ = \ 0
\ \ , \ \
\forall v \in D^*
\ .
\]
\textbf{(5)} Si concluda che $D=D^*$, cosicch\'e $(D,T)$ \e autoaggiunto;
\textbf{(6)} Usando Oss.\ref{rem_defect}, si mostri che gli indici di difetto di $(D,T)$ sono nulli.

\

\noindent (Suggerimenti:
\textbf{(1)} Si noti che $\int g = Fg(1)$ e che $V^{\perp \perp} = \ovl{\bC u_0}$;
\textbf{(2)} L'equazione (\ref{eq_ff}) ed il punto (1) implicano che
$(FT^*v,g) = (F^*T^*v + \{ FT^*v(1) \}u_0 ,  g ) = (T^*v,Fg) = (v,TFg)$,
e, del resto, $TFg=g$;
\textbf{(3)} Dalle uguaglianze precedenti segue che $FT^*v - v$ \e ortogonale ad ogni $g \in V$, dunque per il punto (1) troviamo che $FT^*v - v$ \e costante;
\textbf{(4)} Si ha $v(1)-v(0) = \int_0^1 v'$;
\textbf{(6)} L'equazione $T^*u = -iu' = \pm iu$ ha soluzioni $u(x) := c e^{\pm x}$, $x \in [0,1]$, $c \in \bC$, le quali non sono in $D$.
)
}

\

\noindent \textbf{Esercizio \ref{sec_afunct}.18}
{\it
Si consideri l'operatore $(D,T)$ dell'Esempio \ref{eq.der.I}, e si mostri che:
\textbf{(1)} $(D,T)$ ha spettro puntuale $2 \pi \bZ := \{ 2 \pi k , k \in \bZ \}$; 
\textbf{(2)} Per ogni $\lambda \in \bR - 2 \pi \bZ$, l'operatore $(D,T-\lambda 1)$ \e biettivo;
\textbf{(3)} Si concluda che $\sigma (D,T) = \sigma p (D,T) = 2 \pi \bZ$.

\

\noindent (Suggerimenti: in via preliminare si osservi che se $u \in D$ allora $u(0) = u(1)$. Riguardo i quesiti:
\textbf{(1)} Occorre risolvere l'equazione
\[
-iu' - \lambda u = 0
\ \ , \ \
u \in D
\ ,
\]
che ha soluzione in $D$ se e solo se $\lambda \in 2 \pi \bZ$ e $u(x) := c e^{i \lambda x}$, dove $c \in \bC$.
\textbf{(2)} Grazie al punto precedente gi\'a sappiamo che $(D,T-\lambda 1)$ \e iniettivo se $\lambda \in \bR - 2 \pi \bZ$. D'altro canto, con il metodo di variazione delle costanti  troviamo che, presa $f \in L^2([0,1],\bC) \subset L^1([0,1],\bC)$ e posto
\[
\wa f_\lambda (t) := \int_0^t f(s) e^{-i\lambda s} ds
\ \ , \ \
t \in [0,1]
\ ,
\]
l'equazione
\[
u' - i \lambda u = f
\ \ , \ \
u \in D
\ ,
\]
ha soluzione 
\[
u(x) \ = \  
e^{- i \lambda x} 
\left\{ 
\frac{\wa f_\lambda(1)}{e^{i \lambda}-1}
+
\wa f_\lambda(x)
\right\}
\ \ , \ \
x \in [0,1]
\ .
\]
Dunque $(D,T-\lambda 1)$ \e suriettivo e quindi biettivo. Infine, osservando che $\wa f_\lambda$ \e continua concludiamo facilmente che l'operatore
\[
\{ S_\lambda f \} (x)
\ := \
e^{- i \lambda x} 
\left\{ 
\frac{\wa f_\lambda(1)}{e^{i \lambda}-1}
+
\wa f_\lambda(x)
\right\}
\ , \
x \in [0,1]
\ \ , \ \ 
f \in L^2([0,1],\bC)
\ ,
\]
\e limitato ed \e in effetti l'inverso di $(D,T-\lambda 1)$.)
}

\

\noindent \textbf{Esercizio \ref{sec_afunct}.19 (Misure e rappresentazioni, II).}
{\it Sia $X$ uno spazio compatto e di Hausdorff con famiglia di boreliani $\beta X$, e $B^\infty(X,\bC)$ la \sC algebra delle funzioni limitate e boreliane su $X$. Si assuma che esiste uno spazio di Hilbert complesso $\mH$ ed una rappresentazione
\[
\pi : B^\infty(X,\bC) \to B(\mH)
\ ,
\]
tale che
$f_n \stackrel{n}{\to} f$ puntualmente $\Rightarrow$ $\pi(f_n) \stackrel{n}{\to} \pi(f)$ nella topologia forte
(vedi Esempio \ref{ex_lc}(5)).
\textbf{(1)} Si mostri che per ogni coppia $u,v \in \mH$ esiste una misura boreliana $\mu_{uv} : \beta X \to \bC$, reale per $u = \lambda v$, $\lambda \in \bR$, e complessa altrimenti, tale che 
\[
\int_X f \ d \mu_{uv} \ = \ ( u , \pi(f)v )
\ \ , \ \
\forall f \in B^\infty(X,\bC)
\ .
\]
\textbf{(2)} Si consideri un compatto del tipo 
\begin{equation}
\label{eq.af19}
X \, := \, I \, \dot{\cup} \, S  \, \subset \, \bR \ ,
\end{equation}
dove $I$ \e l'unione di un numero finito di intervalli compatti mutualmente disgiunti ed 
$S := \{ x_n \} \dot{\cup} \{ x_0 \}$ 
\e dato dalla successione $\{ x_n \}$ con punto limite $x_0 \in \bR$.
Si mostri che esistono uno spazio di Hilbert $\mH$ e $T=T^* \in B(\mH)$ tale che $\sigma(T) = X$, $\sigma p(T) = S$.
\textbf{(3)} Si concluda che per ogni compatto $X \subset \bR$ del tipo (\ref{eq.af19}) esistono uno spazio di Hilbert $\mH$ ed una rappresentazione come  in (1).

\

\noindent (Suggerimenti: (1) Si osservi che ogni funzione caratteristica $\chi_E$, $E \in \beta X$, \e boreliana e che $\pi(\chi_E)$ \e un proiettore; si definisca quindi, per ogni $u,v \in \mH$,
\[
\mu_{uv}E := (u,\pi(\chi_E)v) \ \ , \ \ \forall E \in \beta X \ .
\]
(2) Si definisca la misura
\[
\omega : \beta X \to \bR
\ \ , \ \
\omega E := \mu(E \cap I) + \nu (E \cap S) \ ,
\]
dove $\mu$ \e la misura di Lebesgue e $\nu$ quella di enumerazione, e l'operatore
\[
T f(x) := x f(x)
\ \ , \ \
\forall x \in X
\ , \
f \in \mH := L_\omega^2(X,\bC)
\ ,
\]
tenendo a mente l'Esempio \ref{ex_op.mol}. (3) Si usi l'Esercizio \ref{sec_afunct}.16.)

}

%
%

\newpage
\section{Analisi di Fourier.}
\label{sec_fourier}

Lo sviluppo di una funzione in serie di Fourier contiene il germe di un concetto molto importante in analisi, quello di base di uno spazio di Hilbert. L'idea ulteriore che gli elementi di tale base (le funzioni trigonometriche) forniscano, per separazione di variabili, le soluzioni di un'equazione alle derivate parziali (l'equazione del calore) \e un secondo concetto fondamentale, usato poi anche in altri ambiti (l'equazione dell'oscillatore armonico con i polinomi di Hermite (\cite[Compl.II.III.1-3]{TS}), l'equazione dell'atomo di idrogeno con i polinomi di Laguerre \cite[Compl.II.III.2-3]{TS}, $\ldots$). 
Altrettanto importante \e la trasformata di Fourier, la quale trova applicazioni in svariati campi tra cui le equazioni alle derivate parziali e la teoria della trasmissione dei segnali.

In questa sezione esponiamo i fondamenti dell'analisi di Fourier, in particolare sviluppi in serie e trasformate, accennando poi alle sue generalizzazioni nell'ambito dei gruppi topologici, che formano l'oggetto di studio dell'analisi armonica astratta.

\subsection{Serie di Fourier.}

Consideriamo una successione uniformemente convergente del tipo
\begin{equation}
\label{eq_FOU01}
f(x) :=
\frac{a_0}{2} + \sum_{k=1}^{+\infty} \left\{ a_k \cos kx + b_k \sin kx \right\}
\ \ , \ \
x \in \bR
\ .
\end{equation}
Vista la periodicit\'a delle funzioni trigonometriche, la funzione $f$ \e completamente determinata dai suoi valori nell'intervallo $[-\pi,\pi]$ (o -- equivalentemente -- su $[0,2\pi]$). Ci chiediamo ora come la condizione di uniforme convergenza influisce sulle propriet\'a dei coefficienti $a_k,b_k$. A tale scopo, richiamiamo le relazioni
\begin{equation}
\label{eq_FOU02}
\int_{-\pi}^\pi \cos kx \cos mx \ dx = \delta_{km} \pi
\ \ , \ \
\int_{-\pi}^\pi \cos kx \sin mx \ dx = 0
\ \ , \ \
\int_{-\pi}^\pi \sin kx \sin mx \ dx = \delta_{km} \pi
\ ,
\end{equation}
dove $\delta_{km}$ \e il simbolo di Kronecker. Moltiplicando (\ref{eq_FOU01}) prima per $\cos mx$, poi per $\sin mx$, $m \in \bN$, ed integrando sull'intervallo $[-\pi,\pi]$ troviamo
\begin{equation}
\label{eq_FOU03}
a_k = \frac{1}{\pi} \int_{-\pi}^\pi f(x) \cos kx \ dx
\ \ , \ \
b_k = \frac{1}{\pi} \int_{-\pi}^\pi f(x) \sin kx \ dx
\ \ , \ \
a_0 = \frac{1}{\pi} \int_{-\pi}^\pi f(x) \ dx 
\ .
\end{equation}
Con le precedenti espressioni per i coefficienti $a_k,b_k$, chiamiamo (\ref{eq_FOU01}) lo {\em sviluppo in serie di Fourier di $f$}.

Sorge in modo naturale la questione di quali funzioni ammettano uno sviluppo in serie di Fourier, e di che tipo di convergenza (puntuale, uniforme, ...) questa abbia. Riguardando (\ref{eq_FOU03}), risulta chiaro che una condizione necessaria \e $f \in L^1([-\pi,\pi])$. Ora, (\ref{eq_FOU02}) suggerisce che possiamo riguardare 
\[
\mB
\ := \
\{ \cos mx , \sin kx \ , \ k,m \in \bN \} 
\]
come un insieme ortogonale di funzioni per il prodotto scalare di $L^2([-\pi,\pi])$; in realt\'a, come sar\'a chiaro alla fine della sezione, $\mB$ \e in effetti {\em una base} per $L^2([-\pi,\pi])$. Come primo passo, ci dedicheremo alla dimostrazione di teoremi di convergenza della serie di Fourier per funzioni regolari a tratti. Il seguente lemma, inerente lo "sviluppo di Fourier" della funzione costante $1$, si dimostra per induzione, e ne omettiamo la dimostrazione.
\begin{lem}
Per ogni $n \in \bN$, vale la relazione
\begin{equation}
\label{eq_FOU04}
\frac{1}{2} + \sum_{k=0}^n \cos ky \ = \ \frac{\sin (n+1/2)y}{2\sin y/2}
\ ;
\end{equation}
per cui,
\begin{equation}
\label{eq_FOU05}
\frac{1}{\pi} \int_0^\pi \frac{\sin (n+1/2)y}{2\sin y/2} \ dy 
\ = \ 
\frac{1}{2} 
\ .
\end{equation}
\end{lem}

La seguente {\em diseguaglianza di Bessel} segue immediatamente dalla considerazione che $\mB$ \e un insieme ortogonale in $L^2([-\pi,\pi])$ (vedi Prop.\ref{prop_bessel}):
\begin{lem}
Sia $f \in L^2([-\pi,\pi])$. Allora 
\begin{equation}
\label{eq_FOU06}
\frac{a_0^2}{2} + \sum_{k=1}^\infty (a_k^2+b_k^2)
\ \leq \
\frac{1}{\pi} \int_{-\pi}^{\pi} f(x)^2 dx 
\ .
\end{equation}
\end{lem}
Osserviamo che sapendo che $\mB$ \e una {\em base} ortonormale (\ref{eq_FOU06}) diventa un'uguaglianza. Usando (\ref{eq_exLEB1}), otteniamo immediatamente la forma classica del Lemma di Riemann-Lebesgue:
\begin{lem}
Sia $f \in L^1([-\pi,\pi])$. Allora
\begin{equation}
\label{eq_FOU07}
\lim_{k \to \infty} \int_{-\pi}^\pi f(x) \cos kx \ dx
\ = \
\lim_{k \to \infty} \int_{-\pi}^\pi f(x) \sin kx \ dx
\ = \ 
0
\ .
\end{equation}
\end{lem}

Prima di procedere introduciamo le seguenti convenzioni. Data $f : \bR \to \bR$ definiamo, qualora esistano, i limiti
\[
f_+(x) := \lim_{t \to x^+} f(t)
\ \ , \ \
f_-(x) := \lim_{t \to x^-} f(t)
\ \ , \ \
\ovl f (x) := \frac{1}{2} \left( f_+(x) - f_-(x) \right)
\ \ , \ \
x \in \bR
\ .
\]
Una funzione $f : \bR \to \bR$ sia dice {\em periodica con periodo} $a > 0$ se $f(x+a)=f(x)$, $x \in \bR$. E' chiaro che se $f$ \e una funzione sviluppabile in serie di Fourier, allora \e periodica con periodo $2 \pi$. Una propriet\'a elementare delle funzioni periodiche \e la seguente:
\begin{equation}
\label{eq_IFP}
\int_{-a}^a f \ = \ \int_{-a-x}^{a-x} f \ \ , \ \ x \in \bR \ .
\end{equation}

\begin{thm}
\label{thm_FOU_p}
Sia $f : \bR \to \bR$ una funzione con periodo $2 \pi$ e regolare a tratti. Allora la serie di Fourier (\ref{eq_FOU01}) converge puntualmente alla funzione $\ovl f$, e quindi ad $f$ stessa nei suoi punti di continuit\'a. 
\end{thm}

\begin{proof}[Dimostrazione]
Sia $n \in \bN$ ed $S_n$ la somma parziale $n$-esima della serie di Fourier di $f$. Allora
\[
\begin{array}{ll}
S_n(x) & =
\frac{1}{\pi} \int_{-\pi}^{\pi} f(t) 
\left\{
     \frac{1}{2} + \sum_{k=1}^n \cos kt \cos kx + \sin kt \sin kx
\right\}
\ dt
\\ \\ & =
\frac{1}{\pi} \int_{-\pi}^{\pi} f(t)
\left\{
     \frac{1}{2} + \sum_{k=1}^n \cos k(t-x) 
\right\}
\ dt
\\ \\ & \stackrel{u:=t-x}{=}
\frac{1}{\pi} \int_{-\pi-x}^{\pi+x} f(x+u)
\left\{
     \frac{1}{2} + \sum_{k=1}^n \cos ku 
\right\}
\ du
\\ \\ & \stackrel{(\ref{eq_IFP})}{=}
\frac{1}{\pi} \int_{-\pi}^{\pi} f(x+u)
\left\{
     \frac{1}{2} + \sum_{k=1}^n \cos ku 
\right\}
\ du
\\ \\ & \stackrel{(\ref{eq_FOU04})}{=}
\frac{1}{\pi} \int_{-\pi}^{\pi} f(x+u)
\frac{ \sin (n+1/2)u }{ 2 \sin u/2 } 
\ du
\ .
\end{array}
\]
Possiamo allora procedere a stimare
\[
S_n(x) - \ovl f(x)
=
\frac{1}{\pi} \int_{-\pi}^0 
\frac{ f(x+u) - f_-(x)  }{ 2 \sin u/2 }
\sin (n+1/2)u 
\ du
\ + \
\frac{1}{\pi} \int_0^{\pi} 
\frac{ f(x+u) - f_+(x)  }{ 2 \sin u/2 }
\sin (n+1/2)u 
\ du
\ .
\]
A questo punto, \e conveniente definire e studiare la funzione
\[
g(u) :=
\left\{
\begin{array}{ll}
\frac{1}{\pi}  
\frac{ f(x+u) - f_-(x) }{ 2 \sin u/2 }
\ \ , \ \
u \in [-\pi,0)
\\ \\
0 
\ \ , \ \
u=0
\\ \\
\frac{1}{\pi}
\frac{ f(x+u) - f_+(x)  }{ 2 \sin u/2 }
\ \ , \ \
u \in (0,\pi]
\end{array}
\right.
\]
Per costruzione $g$ \e limitata; inoltre l'insieme $A$ dei punti di discontinuit\'a di $g$ \e costituito al pi\'u dai punti del tipo $u = y-x$, dove $y$ \e di discontinuit\'a per $f$, ed eventualmente da $0$. Poich\'e $f$ \e regolare a tratti concludiamo che $A$ \e un insieme finito, per cui $g$ \e integrabile. A questo punto, osserviamo che
\[
S_n(x) - \ovl f(x)
=
\frac{1}{\pi} \int_{-\pi}^{\pi} 
g(u) \sin (n+1/2)u 
\ du
=
\frac{1}{\pi} 
\int_{-\pi}^{\pi} 
\left(
g(u) \sin \frac{u}{2} \cdot \cos nu
+
g(u) \cos \frac{u}{2} \cdot \sin nu
\right)
\ du
\ ,
\]
e (\ref{eq_FOU07}) implica che $\lim_n |S_n(x)-\ovl f(x)| = 0$.
\end{proof}

\begin{thm}
Sia $f : \bR \to \bR$ di periodo $2 \pi$ continua e regolare a tratti. Allora la serie di Fourier di $f$ converge totalmente, e quindi uniformemente, ad $f$. 
\end{thm}

\begin{proof}[Dimostrazione]
Per avere la convergenza totale (e quindi uniforme) \e sufficiente dimostrare che la serie
\[
s_n := \sum_{k=1}^n \left( |a_k| + |b_k| \right)
\ \ , \ \
n \in \bN \ ,
\]
\e convergente. A tale scopo, osserviamo che poich\'e $f$ \e regolare a tratti la derivata $f'$ \e ben definita e continua in $[-\pi,\pi]$ tranne che in un numero finito di punti, nei quali poniamo $f' := 0$. In tal modo abbiamo che $f' \in L^1([-\pi,\pi]) \cap L^2([-\pi,\pi])$, ed integrando per parti possiamo calcolarne i coefficienti di Fourier
\begin{equation}
\label{eq_FOU08}
a'_k := \frac{1}{\pi} \int_{-\pi}^\pi f'(t) \cos kt \ dt =   k b_k
\ \ , \ \
b'_k := \frac{1}{\pi} \int_{-\pi}^\pi f'(t) \sin kt \ dt = - k a_k
\ .
\end{equation}
Applicando la diseguaglianza di Bessel (\ref{eq_FOU06}), otteniamo
\begin{equation}
\label{eq_FOU09}
\sum_k^n \left( (a'_k)^2 + (b'_k)^2  \right) =
\sum_k^n \left( k^2a_k^2 + k^2b_k^2  \right) \leq
\frac{1}{\pi} \int_{-\pi}^\pi f'(t)^2 dt
\ \ , \ \
n \in \bN
\ ,
\end{equation}
il che implica che la serie $\sum_k k^2 (a_k^2+b_k^2)$ converge. Ora, applicando la diseguaglianza $xy \leq 1/2 (x^2+y^2)$ a
\[
2|a_k| = \frac{2}{k} \cdot k|a_k|  \leq  \frac{1}{k^2} + k^2 a_k^2
\ \ , \ \
2|b_k| = \frac{2}{k} \cdot k|b_k|  \leq  \frac{1}{k^2} + k^2 b_k^2
\ ,
\]
otteniamo 
\[
2 \sum_k^n ( |a_k|+|b_k| ) 
\ \leq \
\sum_k^n \frac{1}{k^2} +  
\frac{1}{2} \sum_k^n \left( k^2a_k^2 + k^2b_k^2 \right)
\ \ , \ \
n \in \bN
\ ,
\]
e l'ultimo termine \e una serie convergente.
\end{proof}

\begin{rem}{\it
Lo spazio delle funzioni continue e regolari a tratti su $[-\pi,\pi]$ \e denso in $L^2([-\pi,\pi])$ nella norma $\| \cdot \|_2$ e ci\'o implica che $\mB$ ne \e effettivamente una base ortonormale. Inoltre, osservando che
\[
\frac{1}{\pi} \int_{-\pi}^{\pi} f(t)^2 \ dt
\ = \ 
\sum_k (a_k^2+b_k^2)
\ \leq \
\sum_k (k^2a_k^2+k^2b_k^2)
\ ,
\]
da (\ref{eq_FOU09}) otteniamo una versione della diseguaglianza di Poincar\'e (\ref{eq_Poincare2}).
}
\end{rem}

\begin{ex}{\it
Consideriamo la funzione $f^\sharp$ con periodo $2\pi$ che prolunga $f(x) := x$, $x \in [-\pi,\pi]$.
Allora abbiamo lo sviluppo
\[
b_k = - \frac{2}{k} \cos k\pi = (-1)^{k+1} \frac{2}{k}
\ \Rightarrow \
f^\sharp (x) = x = 2 \sum_{k=1} (-1)^{k+1} \frac{\sin kx}{k}
\]
(Poich\'e $f^\sharp$ \e dispari, otteniamo una serie di soli seni). Definendo invece $f(x) := x^2$, otteniamo una funzione pari $f^\sharp$ con sviluppo di Fourier
\[
f^\sharp(x) = x^2 = \frac{\pi^2}{3} + 4 \sum_{k=1}^\infty (-1)^k \frac{\cos kx}{k^2}
\ .
\]
Le espressioni precedenti possono essere usate per calcolare esplicitamente la somma delle serie numeriche che si ottengono fissando dei valori di $x$. Si veda ad esempio i casi $x = \pi , \pi/2$.
}
\end{ex}

\noindent \textbf{L'equazione del calore.} Siano $\omega \in \bR$ ed $f \in C^2([-\pi,\pi])$; consideriamo il problema alle derivate parziali
\begin{equation}
\label{eq_CALOR}
\left\{
\begin{array}{ll}
\partial_t u = \omega^2 \partial_{xx} u
\\ \\
u(x,0) = f(x)
\end{array}
\right.
\ \ , \ \
x \in [-\pi,\pi]
\ , \
t \in [0,+\infty)
\ .
\end{equation}
Supponendo che una soluzione $u$ di (\ref{eq_CALOR}) debba essere di classe $C^2$ rispetto ad $x$, possiamo sviluppare $u (\cdot , t)$ in serie di Fourier. Inoltre osserviamo che le funzioni
\[
e^{- \omega^2 t} \cos x
\ \ , \ \
e^{- \omega^2 t} \sin x
\]
sono delle soluzioni di (\ref{eq_CALOR}), se non si tiene conto della condizione iniziale. L'idea \e quindi quella di scrivere lo sviluppo
\begin{equation}
\label{eq_GR_HEAT}
u(x,t) 
= 
\frac{a_0}{2}
+
\sum_k \
e^{-k^2 \omega^2 t}
\left\{
a_k \cos kx + b_k \sin kx
\right\}
\ .
\end{equation}
Assumendo le necessarie condizioni di regolarit\'a, troviamo
\[
\partial_x u 
= 
\sum_k \ 
k e^{-k^2 \omega^2 t}
\left\{
- a_k \cos kx + b_k \sin kx
\right\}
\ \Rightarrow \
\partial_{xx} u 
= 
-
\sum_k \ 
k^2 e^{-k^2 \omega^2 t}
\left\{
a_k \cos kx + b_k \sin kx
\right\}
\ ,
\]
e quindi $u$ \e una soluzione di (\ref{eq_CALOR}). I termini $a_k$, $b_k$, $k \in \bN$, si possono determinare imponendo la condizione iniziale:
\[
f(x) = u(x,0) 
=
\frac{a_0}{2}
+
\sum_k \ 
\left\{
a_k \cos kx + b_k \sin kx
\right\}
\]
per cui 
\[
a_k = \frac{1}{\pi} \int_{-\pi}^{\pi} f(x) \cos kx \ dx
\ \ , \ \
b_k = \frac{1}{\pi} \int_{-\pi}^{\pi} f(x) \sin kx \ dx
\ .
\]
Il richiedere la convergenza della serie per $\partial_{xx}u$ equivale a richiedere $f \in C^2([-\pi,\pi])$. 
In maniera analoga, possiamo risolvere prolemi alle derivate parziali con condizioni miste piuttosto che condizioni iniziali.

\begin{rem}[Il nucleo del calore]
{\it
Consideriamo il problema (\ref{eq_CALOR}) nel caso ad $n$ dimensioni (ovvero, $\partial_t u = \omega^2 \Delta u$, $u : \bR^n \times \bR^+ \to \bR$, con $u(\cdot,0) = f$, $f : \bR^n \to \bR$) e la funzione
\[
\Phi (x,t) := \frac{1}{ (4 \pi \omega^2 t)^{n/2}  } \ e^{ - |x|^2 / (4 \omega^2 t)  }
\ \ , \ \
(x,t) \in \bR^n \times \bR^+
\ .
\]
Osserviamo che $\Phi \notin L_{loc}^1(\bR^n \times \bR^+)$, tuttavia si pu\'o verificare che integrando formalmente rispetto ad $x$ otteniamo la famiglia di distribuzioni
\[
F_t \in \mD^*(\bR^n) \ , \ t \in \bR
\ \ : \ \
\left \langle F_t , f \right \rangle := \int f(y) \Phi (y,t) \ dy
\ , \
f \in \mD(\bR^n)
\ .
\]
In particolare per $t = 0$ abbiamo la delta di Dirac, $F_0 = \delta_0$, per cui effettuando la convoluzione troviamo $F_0 * f = f$. Inoltre, per $t>0$ abbiamo $\partial_t \Phi = \omega^2 \Delta \Phi$, dunque ponendo $u(x,t) := F_t * f$ otteniamo una soluzione di (\ref{eq_CALOR}), come si verifica derivando formalmente l'espressione esplicita di $u$:
\[
u(x,t) := \int f(y) \Phi(t,x-y) \ dy \ \ , \ \ x \in \bR^n \ , \ t>0 \ .
\]
Per questo motivo $\Phi$ \e detta la \textbf{soluzione fondamentale, o nucleo}, dell'equazione del calore.
Nel caso $n = 1$, sviluppando una soluzione $u = u(x,t)$ in serie di Fourier (si veda (\ref{eq_GR_HEAT}) e la dimostrazione di Teo.\ref{thm_FOU_p}), troviamo
\[
u(x,t) 
\ = \ 
\frac{1}{\pi} \
\int f(y) \left\{ \frac{1}{2} + \sum_k e^{- k^2 \omega^2 t} \cos k(x-y) \right\} \ dy
\ \ , \ \
(x,t) \in \bR^2
\ ,
\]
per cui 
\[
\Phi(x,t) \ = \ \frac{1}{2 \pi} + \frac{1}{\pi} \sum_k e^{- k^2 \omega^2 t} \cos kx
\ \ , \ \
(x,t) \in \bR^2
\ ,
\]
esprime lo "sviluppo di Fourier" della soluzione fondamentale (vedi (\ref{eq.fou.dis}) per un significato preciso dell'espressione precedente). Osservare che la serie nell'espressione precedente non converge per $t=0$.
}
\end{rem}

\subsection{La trasformata di Fourier.}
\label{sec_fou_tra}

Da un punto di vista intuitivo la trasformata di Fourier pu\'o essere vista come un analogo continuo delle serie di Fourier, oppure come una "continuazione analitica" della trasformata di Laplace (si veda (\ref{eq_laplace_tr})). In questa sezione esporremo in buon dettaglio il caso della retta reale, limitandoci ad accennare ai casi pi\'u generali di $\bR^d$, $d > 1$, e dei gruppi localmente compatti abeliani.

Iniziamo con il definire, per ogni $f \in L^1(\bR,\bC)$
\begin{equation}
\label{def_fourier_tr}
\wa f (x) 
\ := \
\frac{1}{\sqrt{2\pi}} \int_\bR f(t) e^{ixt} \ dt 
\ \ , \ \
x \in \bR
\ .
\end{equation}
Innanzitutto osserviamo che l'espressione precedente \e ben definita per ogni $x \in \bR$ in quanto
\begin{equation}
\label{eq_st_fourier}
| \wa f (x) |  \leq  \frac{1}{\sqrt{2\pi}} \| f \|_1
\ \ , \ \
x \in \bR
\ .
\end{equation}
Elenchiamo alcune propriet\'a elementari della trasformata di Fourier:
\begin{itemize}
\item $(f+ag)\wa{} = \wa f + a \wa g$, $f,g \in L^1(\bR,\bC)$, $a \in \bR$; 
\item $\wa f \in C_0(\bR,\bC)$; \\
      Infatti, sia $\{ x_n \} \subset \bR$ con $x_n \to x$; poich\'e 
      $|e^{ix_nt} f(t)| = |f(t)|$, possiamo applicare il teorema di Lebesgue e 
      concludere che
      $
      \lim_n \wa f(x_n) = 	
      \wa f(x)
      $.
      Il fatto che $\wa f$ svanisce all'infinito segue dal lemma di
      Riemann-Lebesgue (si veda (\ref{eq_exLEB1})). 
\item $\|  \wa f  \|_\infty \leq (2\pi)^{-1/2} \| f \|_1$; \\ 
      Ci\'o segue da (\ref{eq_st_fourier}).
\item $(2\pi)^{-1/2} (f*g) \ \wa{}  \ = \ \wa f \cdot \wa g$; \\
      Infatti, basta usare il teorema di Fubini in modo analogo a (\ref{eq.pr.L}), 
      avendosi $f,g \in L^1(\bR,\bC)$.
\end{itemize}

Nelle righe che seguono stabiliremo alcune propriet\'a di una successione di funzioni che giocher\'a un ruolo importante per le trasformate di Fourier, costruita a partire dalla cosiddetta {\em misura Gaussiana}. Innanzitutto, definiamo
\[
\rho (x) := \frac{1}{\sqrt{2\pi}} \ e^{-x^2/2} \ \ , \ \ x \in \bR
\ \ \Rightarrow \ \
\rho \in C_0^\infty (\bR) \ \cap \bigcap_{p \in [1,+\infty]} L^p(\bR)
\ ,
\]
e dimostriamo le seguenti propriet\'a:
\begin{equation}
\label{eq_gauss_1}
\rho (x) = \wa \rho (x) \ \ , \ \ x \in \bR \ .
\end{equation}
\begin{equation}
\label{eq_gauss_2}
\rho_n (x) := n \rho(nx) \ \ , \ \ x \in \bR \ ,\ n \in \bN
\ \ \Rightarrow \
\| \rho_n \|_1 \equiv 1 \ .
\end{equation}
\begin{equation}
\label{eq_gauss_3}
\rho_n (x) = 
\frac{1}{\sqrt{2\pi}} \int_\bR \rho \left( \frac{t}{n} \right) e^{itx} \ dt
\ \ , \ \
\wa \rho_n(x) = \wa \rho \left( \frac{x}{n} \right) = \rho \left( \frac{x}{n} \right)
\ .
\end{equation}
\begin{equation}
\label{eq_gauss_4}
f * \rho_n (x) 
\ = \ 
\int_\bR  \wa f (t) \rho \left( \frac{t}{n} \right) e^{-itx} \ dt
\ \ , \ \
f \in L^1(\bR,\bC)
\ .
\end{equation}
Per quanto riguarda \textbf{(\ref{eq_gauss_1})}, osserviamo che derivando sotto il segno di integrale ed integrando per parti troviamo
\[
\wa \rho \ '(x) = 
\frac{i}{2\pi} \int_\bR t e^{ixt - t^2/2} dt =
\frac{i}{2\pi} 
\left(
\left[
- e^{ixt-t^2/2}
\right]_{-\infty}^{+\infty} +  i \int_\bR e^{ixt-t^2/2} x \ dt
\right)
\ .
\]
L'uguaglianza precedente ci dice che $\wa \rho$ \e soluzione del problema di Cauchy
\[
\left\{
\begin{array}{ll}
u' (x) = - (2\pi)^{-1/2} x u(x)
\\
u(0) = 1 \ ,
\end{array}
\right.
\]
il quale d'altra parte ha soluzione unica $u = \rho$; per cui, $\rho = \wa \rho$. La dimostrazione di \textbf{(\ref{eq_gauss_2})} si effettua usando la sostituzione $t \mapsto nt$ nell'integrale $\| \rho_n \|_1$. \textbf{(\ref{eq_gauss_3})} si dimostra usando le uguaglianze
\[
\rho_n (x) = n \rho (nx) = n \wa \rho (nx) = 
\frac{n}{2 \pi} \int_\bR e^{itnx-t^2/2} \ dt 
\]
ed applicando la sostituzione $t \mapsto nt$. Infine, \textbf{(\ref{eq_gauss_4})} si dimostra calcolando
\[
\begin{array}{ll}
f * \rho_n (x) & =
\int_\bR f(x-y) \rho_n(y) \ dy 
\stackrel{(\ref{eq_gauss_3})}{=} 
\frac{1}{\sqrt{2\pi}} \int f(x-y) \rho(t/n) e^{iyt} \ dtdy 
\\ \\ & =
- \frac{1}{\sqrt{2\pi}} \int f(s) \rho(t/n) e^{i(x-s)t} \ dtds 
=
-
\int_\bR \left[ 
         \frac{1}{\sqrt{2\pi}} \int_\bR f(s) e^{-ist} \ ds  
         \right] 
         \rho(t/n) e^{ixt} \ dt 
\\ \\ & =
- \int_\bR \wa f (-t) \rho(t/n) e^{ixt} \ dt 
=
\int_\bR \wa f (t) \rho(t/n) e^{-ixt} \ dt
\ ,
\end{array}
\]
avendo usato il teorema di Fubini.

\

Richiamiamo ora la notazione $f_y(x) := f(x+y)$, $x,y \in \bR$, e diamo il seguente
\begin{lem}
\label{lem_FT01}
Si ha
\begin{equation}
\label{eq_FT01}
\lim_{n} \| f * \rho_n - f  \|_1 \ \stackrel{n}{\to} \ 0
\ \ , \ \
f \in L^1(\bR,\bC)
\ ,
\end{equation}
\begin{equation}
\label{eq_FT02}
\lim_{n} \| f * \rho_n - f  \|_2 \ \stackrel{n}{\to} \ 0
\ \ , \ \
f \in L^1(\bR,\bC) \cap L^2(\bR,\bC)
\ .
\end{equation}
\end{lem}

\begin{proof}[Dimostrazione]
Per dimostrare (\ref{eq_FT01}), effettuiamo la stima
\[
\| f * \rho_n - f  \|_1
\ \leq \
\int | f(x-y) - f (x) | \rho_n(y) \ dx dy
\ = \ 
\int \| f_{-y} - f \|_1 \rho_n(y) \ dy \ .
\]
Ora, usando l'Esercizio \ref{sec_Lp}.2 troviamo che $g(y) := \| f_{-y} - f \|_1$ \e continua, per cui scelto $\eps > 0$ esiste $\delta > 0$ tale che $|y| < \delta$ implica $\| f_{-y} - f \|_1 < \eps$. Dunque,
\[
\| f * \rho_n - f  \|_1
\ \leq \
\eps \int_{-\delta}^{\delta} \rho_n(y) \ dy
+
2 \| f \|_1 \int_{|y| \geq \delta} \rho_n(y) \ dy
\ ,
\]
e poich\'e il secondo addendo nell'espressione precedente diventa arbitrariamente piccolo nel limite $n \to \infty$, troviamo quanto volevasi dimostrare. 
Per dimostrare (\ref{eq_FT02}), osserviamo che $d \mu (t) := \rho_n(t) \ dt$ \e una misura di probabilit\'a, per cui possiamo usare la diseguaglianza di Jensen (\ref{eq_jensen}) e concludere
\[
| f * \rho_n (x) - f (x) |^2 
\ = \ 
\left( \int |f(x-y) - f(x)| \rho_n(y) \ dy \right)^2
\ \stackrel{Jensen}{\leq} \
\int |f(x-y) - f(x)|^2 \rho_n(y) \ dy
\ .
\]
Dunque, otteniamo la diseguaglianza
\[
\| f*\rho_n - f \|_2 
\ \leq \ 
\int \| f_{-y} - f \|_2 \rho_n(y) \ dy 
\]
e l'argomento usato per dimostrare (\ref{eq_FT01}) implica quanto volevasi dimostrare.
\end{proof}

\begin{rem}{\it
Il lemma precedente stabilisce che $\{ \rho_n \}$ si comporta come una successione di mollificatori, nonostante i supporti ${\mathrm{supp}}(\rho_n)$ non soddisfino la propriet\'a enunciata in Def.\ref{def_moll}. Quando $f \in C_c(\bR)$ il risultato pu\'o essere migliorato, ottenendo una convergenza uniforme (vedi Oss.\ref{rem_moll} o \cite[Lemma 2.8.1]{Giu2}). Poich\'e ogni $\rho_n$ \e analitica, sviluppando in serie di Taylor troviamo
\[
\rho_n * f (x) 
\ = \
\frac{1}{\sqrt{2\pi}} \
\lim_{m \to \infty}
\sum_{k=0}^m
\frac{ (-1)^k n^{2k+1} }{ {k!} }
\int_\bR
(x-t)^{2k} f(t) \ dt
\ .
\]
Notare che il termine a destra \e un polinomio in $x$ di grado $2m$; \e questa l'essenza della dimostrazione del teorema di densit\'a di Weierstrass (\cite[Teo.2.8.1]{Giu2}).
}\end{rem}

Con la seguente notazione, introduciamo l'applicazione nota come {\em antipodo}:
\[
\epsilon f (x) := \ovl{f(-x)} \ \ , \ \ x \in \bR \ , \ f \in L^1(\bR,\bC)
\ \ \Rightarrow \ \ 
\wa{\epsilon f} (x) = \ovl{\wa f (x)}
\ \ , \ \
x \in \bR
\ .
\]

\begin{thm}[Parseval]
\label{thm_FT02}
Se $f \in L^1(\bR,\bC) \cap L^2(\bR,\bC)$ allora $\wa f \in L^2(\bR,\bC)$ e $\| f \|_2 = \|  \wa f  \|_2$.
\end{thm}

\begin{proof}[Dimostrazione]
Ponendo $g := f * \epsilon f$, troviamo $g \in L^1(\bR,\bC)$, e
\[
g(y) = \int_\bR f(y-x) \ovl{f(-x)} \ dx 
\ \ \Rightarrow \ \ 
g(0) = \| f \|_2^2 \ .
\]
Inoltre,
\begin{equation}
\label{eq_FT04}
\wa g(x) 
\ = \
(f*\epsilon f) \wa{} \ (x) 
\ = \
\wa f (x) \ovl{\wa f (x)} 
\ \ , \ \
x \in \bR
\ \ \Rightarrow \ \
\int_\bR \wa g = \| \wa f \|_2^2
\ .
\end{equation}
Essendo $f , \epsilon f \in L^2(\bR,\bC)$, usando l'Esercizio \ref{sec_Lp}.2 troviamo che $g$ \e continua, e grazie a Prop.\ref{prop_der_conv} lo stesso \e vero per ogni
$g * \rho_n$, $n \in \bN$.
Usando (\ref{eq_FT01}) troviamo $\| g*\rho_n-g \|_1 \stackrel{n}{\to} 0$, per cui l'Esercizio \ref{ex_sec_lebesgue} implica 
\[
\lim_n g*\rho_n (x) = g(x) \ \ , \ \ x \in \bR \ .
\]
In particolare,
\[
\lim_n g*\rho_n (0) = g(0) \ = \ \| f \|_2^2 \ .
\]
D'altra parte, usando (\ref{eq_gauss_4}) troviamo
\begin{equation}
\label{eq_FT03}
g * \rho_n (x) 
\ = \
\int_\bR \wa g (t) \ \rho \left( \frac{t}{n} \right) e^{-ixt} \ dt
\ \ , \ \
x \in \bR
\ \Rightarrow \
g * \rho_n (0) 
\ = \
\int_\bR \wa g (t) \ \rho \left( \frac{t}{n} \right) \ dt
\ ;
\end{equation}
ora, osserviamo che la successione
\begin{equation}
\label{eq_FT05}
\lambda_n(x) := \rho \left( \frac{x}{n} \right)
\ \ , \ \
x \in \bR
\ ,
\end{equation}
\'e puntualmente convergente alla costante $(2\pi)^{-1/2}$ nonch\'e monot\'ona crescente; per cui, per convergenza monot\'ona (Teo.\ref{thm_conv_mon}) e grazie a (\ref{eq_FT03}), troviamo
\[
g(0) = 
\lim_n \int \wa g (t) \rho \left( \frac{t}{n} \right) \ dt
= 
\frac{1}{\sqrt{2\pi}} \int \wa g (t) \ dt
= 
\|  \wa f  \|_2^2
\ ,
\]
avendo usato (\ref{eq_FT04}).
\end{proof}

\begin{thm}[Teorema di inversione di Fourier]
Sia $f \in L^1(\bR,\bC)$ tale che $\wa f \in L^1(\bR,\bC)$.
Allora
\begin{equation}
\label{eq_FT06}
f(x) = \frac{1}{\sqrt{2\pi}} \int \wa f (t) e^{-ixt} dt
\ \ , \ \
{\mathrm{q.o. \ in \ }} 
x \in \bR
\ .
\end{equation}
\end{thm}

\begin{proof}[Dimostrazione]
Usando (\ref{eq_gauss_4}) ed applicando il teorema di convergenza monot\'ona alla successione (\ref{eq_FT05}) troviamo
\[
f * \rho_n (x) 
\ = \ 
\int_\bR  \wa f (t) \rho \left( \frac{t}{n} \right) e^{-itx} \ dt
\ \stackrel{n}{\to} \
\frac{1}{\sqrt{2\pi}} \int \wa f (t) e^{-ixt} dt
\ .
\]
D'altra parte, grazie a (\ref{eq_FT01}) ed al teorema di Fischer-Riesz (Teo.\ref{thm_RF}), concludiamo che esiste una sottosuccessione $\{ n_k \}$ tale che $f(x) = \lim_n f*\rho_{n_k}(x)$, q.o. in $x \in \bR$.
\end{proof}

\begin{cor}
Sia $f \in L^1(\bR,\bC)$ tale che $\wa f (x) = 0$ q.o. in $x \in \bR$. Allora $f(x) = 0$ q.o. in $x \in \bR$.
\end{cor}

Osserviamo ora che -- essendo $L^1 \cap L^2$ denso in $L^2$ -- la trasformata di Fourier si estende ad un operatore
\[
F \in BL^2(\bR,\bC) \ ,
\]
il quale, grazie al Teorema di Parseval, \e isometrico. Con il prossimo teorema dimostriamo che $F$ \e in effetti un operatore unitario.

\begin{thm}[Teorema di Fourier-Plancherel]
\label{thm_FP}
L'estensione $F$ della trasformata di Fourier \e un operatore unitario di $L^2(\bR,\bC)$ in s\'e.
\end{thm}

\begin{proof}[Dimostrazione]
Visto che gi\'a sappiamo che $F$ \e isometrico l'unica propriet\'a che occorre verificare \e la suriettivit\'a. Per ogni $g \in L^2(\bR,\bC)$ definiamo $g_n := g \chi_{[-n,n]}$, $n \in \bN$, ed osserviamo che $g_n \in L^1(\bR,\bC) \cap L^2(\bR,\bC)$ (Esercizio \ref{sec_Lp}.3); per cui, \e ben definita $\wa g_n \in C_0(\bR,\bC)$. Il primo passo della dimostrazione sar\'a quello di verificare che ponendo
\begin{equation}
\label{eq_breve}
\breve g (x) 
\ := \ 
\lim_n 
\frac{1}{\sqrt{2\pi}} \int g_n(t) e^{-ixt} \ dt 
\ = \
\lim_n
\wa g_n(-x)
\ \ , \ \
{\mathrm{q.o. \ in \ }} 
x \in \bR
\ ,
\end{equation}
otteniamo una ben definita funzione in $L^2(\bR,\bC)$. Al che, mostreremo che $\{ g \mapsto \breve g \}$ fornisce l'inversa di $F$. A tale scopo, osserviamo che chiaramente $\lim_n \| g - g_n \|_2 = 0$ per cui -- per isometria di $F$ -- troviamo $\lim_n \| Fg - \wa g_n \|_2 = 0$. Applicando Fischer-Riesz e (\ref{eq_breve}) troviamo 
\[
Fg(x) = \lim_k \wa g_{n_k} (x) = \breve g (-x)
\ \ , \ \
{\mathrm{q.o. \ in \ }} 
x \in \bR
\ ,
\]
e dunque $\breve g (x)$ \e ben definito q.o. in $x \in \bR$ ed ivi coincidente con $Fg (-x)$. Ci\'o implica
\[
\int |\breve g (x)|^2 \ dx 
\ = \
\int | Fg(-x) |^2 \ dx
\ \Rightarrow \
\| \breve g \|_2 = 
\| Fg \|_2       = 
\| g \|_2        < 
+ \infty \ .
\]
Abbiamo quindi costruito un operatore isometrico $F' \in BL^2(\bR,\bC)$, $F'g := \breve g$. Mostriamo ora che $F'F$ \e l'identit\'a su $L^2(\bR,\bC)$; a tale scopo, per densit\'a, \e sufficiente verificare solo per $g \in L^1(\bR,\bC) \cap L^2(\bR,\bC)$, e si ha
\[
g*\rho_n (x) 
\ \stackrel{ (\ref{eq_gauss_4}) }{=} \
\int_\bR  \wa g (t) \rho \left( \frac{t}{n} \right) e^{-itx} \ dt
\ \stackrel{ (\ref{eq_gauss_3}) }{=} \
\int_\bR  \wa g (t) \wa \rho_n(t) e^{-itx} \ dt 
\ = \
\frac{1}{\sqrt{2\pi}}
\int (g*\rho_n) \wa{} \ (t) \ e^{-itx} \ dt
\ .
\]
L'uguaglianza precedente si pu\'o leggere come $g*\rho_n = F'F (g*\rho_n)$, $n \in \bN$. 
Poich\'e per (\ref{eq_FT02}) si ha $\lim_n \| g - g*\rho_n \|_2 = 0$, otteniamo $g = F'F g$, ed il teorema \e dimostrato.
\end{proof}

\

\noindent \textbf{La trasformata di Fourier in $\bR^d$.} La teoria della trasformata di Fourier si generalizza facilmente al caso $\bR^d$, $d \in \bN$. Presa $f \in L^1(\bR^d,\bC)$ (misura prodotto di Lebesgue) definiamo
\[
\wa f(x) 
\ := \ 
(2\pi)^{-n/2}
\int_{\bR^d} 
f(t) e^{i x \cdot t} \ dt 
\ \ , \ \
x \in \bR^d
\ ,
\]
dove $x \cdot t$ denota il prodotto scalare. Gli strumenti di lavoro principali della sezione precedente, le convoluzioni e le misure gaussiane, si utilizzano senza problemi in questo caso pi\'u generale, le prime senza variazioni e le seconde definendo
\[
\rho_n(x) := (2\pi)^{-n/2} e^{-|x|^2/2}
\ \ , \ \
x \in \bR^d
\ .
\]
In modo analogo al caso $d=1$ possiamo definire l'antitrasformata
\[
\breve f (x) := \wa f (-x) 
\ , \
\forall x \in \bR^d
\ \ , \ \
f \in L^1(\bR^d,\bC)
\ .
\]
I risultati principali della sezione precedente rimangono, ovviamente, veri:

\begin{thm}[Parseval, Fourier, Plancherel]
La trasformata di Fourier definisce un operatore lineare limitato
\begin{equation}
\label{eq.foud}
L^1(\bR^d,\bC) \to C_0(\bR^d,\bC)
\ \ , \ \
f \mapsto \wa f
\ ,
\end{equation}
il quale ha inverso $\{ g \mapsto \breve g \}$ in $C_0(\bR^d,\bC) \cap L^1(\bR^d,\bC)$. Inoltre 
\[
\| \wa f \|_2 \ = \ \| f \|_2 
\ \ , \ \
\forall f \in L^1(\bR^d,\bC) \cap L^2(\bR^d,\bC) \ ,
\]
e (\ref{eq.foud}) si estende ad un operatore unitario
$F \in BL^2(\bR^d,\bC)$.
\end{thm}

La dimostrazione del teorema precedente si effettua adattando le tecniche del caso unidimensionale con un uso massiccio del teorema di Fubini. Per dettagli in merito segnalamo \cite[Chap.IX]{RS2}, dove un approccio leggermente diverso rispetto a quello della sezione precedente viene adottato con l'uso dello spazio $\mS(\bR^d,\bC)$ delle funzioni complesse a decrescenza rapida.

\

\noindent \textbf{Trasformata di Fourier e distribuzioni temperate.} L'utilizzo di $\mS(\bR^d,\bC)$ di cui abbiamo appena accennato \e motivato dal fatto che la trasformata di Fourier si restringe ad un operatore lineare
\[
T : \mS(\bR^d,\bC) \to \mS(\bR^d,\bC)
\ \ , \ \
Tf := \wa f
\ \ , \ \
\forall f \in \mS(\bR^d,\bC) \subset L^1(\bR^d,\bC)
\ ,
\]
il quale \e continuo nella topologia naturale di $\mS(\bR^d,\bC)$ (vedi \cite[Theorem IX.1]{RS2}); ci\'o implica che \e ben definita la {\em trasformata di Fourier sulle distribuzioni temperate}, come l'aggiunto di $T$:
\begin{equation}
\label{eq.fou.dis}
T^* : \mS^*(\bR^d,\bC) \to \mS^*(\bR^d,\bC)
\ \ : \ \
T^* \circ I = I \circ T
\ ,
\end{equation}
dove $I : \mS(\bR^d,\bC) \to \mS^*(\bR^d,\bC)$ \e l'immersione canonica nel senso di (\ref{eq_emb_ss}).

\

\noindent \textbf{La trasformata di Fourier dei gruppi localmente compatti abeliani.} Sia $G$ un gruppo topologico localmente compatto e di Hausdorff. Il {\em gruppo dei caratteri di $G$} \e l'insieme $G^*$ delle funzioni continue del tipo
\[
\chi : G \to \bT
\ \ : \ \
\chi (ts) = \chi(s) \chi(s)
\ , \
\forall t,s \in G
\ ,
\]
equipaggiato con il prodotto $\chi \chi' (s) := \chi(s) \chi'(s)$, inverso $\chi^{-1}(s) := \ovl{\chi(s)}$, ed identit\'a $e(s) := 1$, $s \in G$. Introduciamo su $G^*$ la topologia della convergenza uniforme sui compatti:
\[
\chi_n \stackrel{n}{\to} \chi 
\ \Leftrightarrow \
\sup_{s \in K} d ( \chi_n(s) , \chi (s) ) \ \stackrel{n}{\to} \ 0
\ , \
\forall K \subseteq G \ {\mathrm{compatto}}
\ ,
\]
dove $d : \bT \times \bT \to \bR$ \e la metrica di $\bT$ (il quale \e omeomorfo al cerchio).
E' possibile dimostrare che $G^*$ \e a sua volta localmente compatto e di Hausdorff (oltre che, ovviamente, abeliano, ovvero, $\chi \chi' = \chi' \chi$ per ogni $\chi , \chi' \in G^*$). Diamo alcuni esempi fondamentali di gruppi di caratteri:
\begin{enumerate}
\item Se $G = \bR^d$, $d \in \bN$ (come gruppo additivo) allora $G^* \simeq \bR^d$; infatti tutti e soli i caratteri di $\bR^d$ sono quelli del tipo
\[
\chi_x : \bR^d \to \bT 
\ \ , \ \ 
\chi_x(t) := e^{ix \cdot t} 
\ , \ \forall t \in \bR^d 
\ \ , \ \
x \in \bR^d
\ .
\]
\item Se $G = \bT$ (gruppo moltiplicativo) allora $G^* \simeq \bZ$. Infatti i caratteri di $\bT$ sono tutti e soli quelli del tipo
\[
\chi_k : \bT \to \bT 
\ \ , \ \ 
\chi_k(z) := z^k 
\ , \ 
\forall z \in \bT 
\ \ , \ \
k \in \bZ
\ .
\]
\item Se $G = \bZ$ (gruppo additivo) allora $G^* \simeq \bT$, con caratteri
\[
\chi_z : \bZ \to \bT 
\ \ , \ \ 
\chi_z(k) := z^k 
\ , \ 
\forall k \in \bZ 
\ \ , \ \
z \in \bT
\ .
\]
\end{enumerate}
Un'occhiata agli esempi precedenti mostra che si hanno isomorfismi $\{ \bR^d \}^{**} \simeq \bR^d$, $\bT^{**} \simeq \bT$, $\bZ^{**} \simeq \bZ$; questo non \e un fatto casuale, in quanto si ha il seguente teorema:

\begin{thm}[Pontryagin-van Kampen]
Se $G$ \e un gruppo localmente compatto abeliano allora si ha un isomorfismo $G \simeq G^{**}$.
\end{thm}

\begin{proof}[Sketch della dimostrazione]
Si tratta di verificare che l'applicazione
\begin{equation}
\label{eq.Pon}
G \to G^{**}
\ \ , \ \
s \mapsto \varphi_s 
\ \ : \ \ 
\varphi_s(\chi) := \chi(s)
\ , \
\forall \chi \in G^*
\ , \
s \in G
\ ,
\end{equation}
\e un isomorfismo. Le verifiche che (\ref{eq.Pon}) \e iniettiva e conserva il prodotto sono semplici e vengono lasciate come esercizio. La continuit\'a e la suriettivit\'a sono i punti pi\'u delicati, e per essi rimandiamo a \cite[\S 4.3]{Fol} o \cite[Vol.1, \S 24.8]{HR}.
\end{proof}

Osserviamo che il teorema precedente si applica solo nel caso in cui $G$ sia abeliano; infatti un generico gruppo topologico potrebbe essere privo di caratteri non banali, come ad esempio il gruppo $SU(2)$ delle matrici complesse $2 \times 2$ con determinante $1$.

Denotiamo ora con $\mu$ la misura di Haar di $G$ (vedi \S \ref{sec_MIS1}) e definiamo, per ogni $f \in L_\mu^1(G,\bC)$,
\begin{equation}
\label{def_a_fou}
\wa f (\chi) := \int_G f(s) \chi(s) \ d \mu (s)
\ \ , \ \
\forall \chi \in G^*
\ .
\end{equation}
Analogamente al caso $G = \bR$, usando il teorema di convergenza dominata troviamo che $\wa f$ \e una funzione continua, cosicch\'e abbiamo un analogo astratto della trasformata di Fourier. Il fatto che (quando $G^*$ non \e compatto) $\wa f$ svanisce all'infinito \e pi\'u delicato da dimostrare, ma comunque vero, cosicch\'e $\wa f \in C_0(G^*, \bC)$, cos\'i come rimangono veri, {\em nel caso $G$ abeliano}, i teoremi di Fourier, Parseval e Plancherel, con la modifica che stavolta abbiamo un operatore unitario
\[
F : L_\mu^2(G,\bC) \to L_{\mu^*}^2(G^*,\bC) \ ,
\]
dove $\mu^*$ \e la misura di Haar su $G^*$. Per dettagli si veda \cite[\S 4.2]{Fol} o \cite[Vol.1, Cap.6]{HR}.

\begin{ex}
\label{ex_pontr}
{\it
Nel caso $G = \bR^d$ ritroviamo (a meno di radici di $2 \pi$) la classica trasformata di Fourier in $\bR^d$; nel caso $G = \bT$ ($G^* \simeq \bZ$) abbiamo invece la cosiddetta trasformata di Fourier discreta
\[
\wa f(k) := \int_\bT f(z) z^k d \mu(z)
\ \ , \ \
\forall k \in \bZ
\ \ , \ \
f \in L^1_\mu(\bT,\bC)
\ ;
\]
l'integrale precedente \e effettuato sulla misura di Haar $\mu$ di $\bT$, la quale coincide essenzialmente con la misura di Lebesgue sull'intervallo $[0,1]$ una volta usato il cambiamento di variabile 
\[
[0,1] \to \bT
\ \ , \ \
\theta \mapsto e^{2 \pi i \theta}
\ .
\]
Infine, nel caso $G = \bZ$ ($G^* \simeq \bT$) la misura di Haar coincide con la misura di enumerazione (Esempio \ref{ex_misnumZ}), cosicch\'e lo spazio delle le funzioni integrabili su $\bZ$ \e dato da
\[
l^1(\bZ,\bC) \ := \ \{ \{ f_k \in \bC \}_{k \in \bZ} \ : \ \sum_k |f_k| < \infty  \}
\ .
\]
Di conseguenza, ogni $f \in l^1(\bZ,\bC)$ si pu\'o riguardare come la successione dei coefficienti di Fourier della sua trasformata:
\[
\wa f (z) := \sum_{k \in \bZ} f_k z^k
\ \ , \ \
\forall z \in \bT
\ \ , \ \
f \in l^1(\bZ,\bC)
\ .
\]
}
\end{ex}

\subsection{Esercizi.}

\noindent \textbf{Esercizio \ref{sec_fourier}.1.} {\it Scrivere gli sviluppi in serie di Fourier delle seguenti funzioni:
(1) $e^{\alpha x}$, $\alpha \in \bR$;
(2) $x \cos x$;
(3) $x (1+\cos x)$;
(4) ${\mathrm{sgn}} x$;
(5) $\chi_{[-1,1]}$.
}

\

\noindent \textbf{Esercizio \ref{sec_fourier}.2.} {\it Calcolare la trasformata di Fourier di $f_n := n \chi_{[0,1/n]}$, $n \in \bN$, e studiare la convergenza della successione $\{ \wa f_n \} \subset C_0(\bR,\bC)$.}

\

\noindent \textbf{Esercizio \ref{sec_fourier}.3.} {\it Sia $f \in L_{loc}^1(\bR,\bC)$ una funzione $2 \pi$-periodica, e
\[
\wa f (n) 
\ := \ 
\int_{-\pi}^\pi f(x) e^{-inx} \ dx
\ \ , \ \
n \in \bZ 
\ .
\]
Dimostrare che: \textbf{(1)} Valgono le uguaglianze
\[
\wa f (n) 
\ = \
\frac{1}{2} \
\int_{-\pi}^{\pi} 
\left(  f(x) - f \left( x + \frac{\pi}{n} \right) \right) \ e^{-inx} \ dx 
\ \ , \ \
n \in \bZ \ ;
\]
\textbf{(2)} Se $f$ \e lipschitziana, allora si ha una stima
$| \wa f(n) | \leq c n^{-1}$,
$n \in \bN$;
\textbf{(3)} Se $f \in C^1(\bR,\bC)$, allora
$\lim_{n \to \pm \infty} n \wa f(n) = 0$.

\

\noindent (Suggerimento: per (1) si usi lo stesso metodo usato per dimostrare il Lemma di Riemann-Lebesgue).
}

\noindent \textbf{Esercizio \ref{sec_fourier}.4 (La trasformata di una misura finita).} 
{\it Si consideri lo spazio normato $\Lambda_\beta^1(\bR^d,\bC)$ delle misure boreliane complesse (e quindi finite) su $\bR^d$ 
{\footnote{Vedi Def.\ref{def_l1} ed Esercizio \ref{sec_MIS}.7.}}.
\textbf{(1)} Presa $\mu \in \Lambda_\beta^1(\bR^d,\bC)$, si mostri che la funzione
\[
\wa \mu (x) := \int e^{i x \cdot t} d \mu (t)
\ \ , \ \
x \in \bR^d
\ ,
\]
\e continua e tale che $\| \wa \mu \|_\infty \leq | \mu |(\bR^d) < \infty$.
\textbf{(2)} Presa $f \in L^1(\bR^d,\bC)$ e la misura 
\[
\mu_f E := \int_E f
\ \ , \ \
E \subseteq \bR^d
\ ,
\]
si mostri che $\mu_f \in \Lambda_\beta^1(\bR^d,\bC)$ e che $\wa f = \wa \mu_f$.
\textbf{(3)} Si considerino le misure di Dirac $\mu_a$, $a \in \bR^d$, e si mostri che
$\wa \mu_a (x) = e^{i x \cdot a}$, $x \in \bR^d$ (cosicch\'e, in particolare, per $a=0$ otteniamo la funzione costante $1$).  

\

\noindent (Suggerimenti: per il punto (1) si usi il teorema di convergenza dominata, mentre per il punto (2) si osservi che $\int g \ d \mu_f = \int gf$ per ogni $g \in L_{\mu_f}^1(\bR^d)$).
}

\

\noindent \textbf{Esercizio \ref{sec_fourier}.5 (Trasformata di Fourier e gruppi ad un parametro).} 
{\it Sia $U := \{ U_t \}$ un gruppo ad un parametro sullo spazio di Hilbert complesso $\mH$ (vedi (\ref{eq.scopug}) e (\ref{eq.scopug2})). Per ogni $f \in L^1(\bR,\bC)$ si definisca
\begin{equation}
\label{eq.exUt}
A_f(u,v) \ := \ \int f(t) ( u,U_tv ) \ dt
\ \ , \ \
u,v \in \mH
\ .
\end{equation}
\textbf{(1)} Si mostri che $A_f$ \e una forma bilineare continua, per cui esiste ed \e unico l'operatore $T_f \in B(\mH)$ tale che
$A_f(u,v) = (u,T_fv)$, $\forall u,v \in \mH$;
\textbf{(2)} Si mostri che
\begin{equation}
\label{eq.exUt2}
T_{af + g} = aT_f + T_g \ , \ T_{f*g} = T_f T_g \ , \ \| T_f \| \leq \| f \|_1
\ \ , \ \
\forall a \in \bC \ , \ f,g \in L^1(\bR,\bC)
\ .
\end{equation}
\textbf{(3)} Si prenda $\mH := L^2(\bR,\bC)$ e si verifichi che definendo
\begin{equation}
\label{eq.exUt3}
U_tu(s) := e^{its}u(s)
\ \ , \ \
\forall u \in \mH
\ , \
t \in \bR
\ ,
\end{equation}
si ottiene un gruppo ad un parametro. Infine, si mostri che
\[
T_fu = \wa f u
\ \ , \ \
\forall f \in L^1(\bR,\bC) \ , \  u \in \mH \ .
\]
(Suggerimenti: Per (1) si osservi che, essendo ogni $U_t$ unitario, troviamo $|(u,U_tv)| \leq \| u \| \| v \|$ e quindi
$\| A_f(u,v) \| \leq \| f \|_1 \| u \| \| v \|$.
Per (2), in particolare l'uguaglianza che coinvolge la convoluzione $f*g$, si confrontino i prodotti scalari
$( u, T_{f*g}v )$ e $(u,T_f T_g v)$
usando (\ref{eq.scopug}) ed il teorema di Fubini.
Per (3) si usino i teoremi di Lebesgue e di Fubini).
}

\

\noindent \textbf{Esercizio \ref{sec_fourier}.6.} {\it Si mostri che 
$\wa \varphi \in C_0^\infty(\bR,\bC) \cap L^2(\bR,\bC)$ 
per ogni $\varphi \in C_c^\infty(\bR,\bC)$.

\noindent (Suggerimenti: si usi il Teorema di Parseval. Per la differenziabilit\'a si usi il Teorema \ref{thm_der_int}).
}

\

\noindent \textbf{Esercizio \ref{sec_fourier}.7 (Trasformata di Fourier e derivate deboli).} {\it Si consideri la trasformata di Fourier come un operatore unitario $F \in BL^2(\bR,\bC)$, cosicch\'e 
$Ff = \wa f$, $F^*f = \breve f$, $\forall f \in L^1(\bR,\bC) \cap L^2(\bR,\bC)$.
\textbf{(1)} Si determini l'operatore autoaggiunto $(D,T)$ associato, nel senso di (\ref{eq.scopug3}), al gruppo ad un parametro (\ref{eq.exUt3});
\textbf{(2)} Si mostri che $\{ \wa U_t := FU_tF^* \}$ \e, a sua volta, un gruppo ad un parametro e se ne dia un'espressione esplicita
          sul sottospazio $L^1(\bR,\bC) \cap L^2(\bR,\bC)$;
\textbf{(3)} Si mostri che l'operatore autoaggiunto $(\wa D , \wa T)$ associato ad $\{ \wa U_t \}$ \e dato da
          $\wa D = F(D)$, $\wa T = FTF^*$;
\textbf{(4)} Si verifichi (usando l'Esercizio precedente) che 
          \[
          -i \frac{d}{ds} \wa \varphi(s) \ = \ FT \varphi(s)
          \ \ , \ \
          \forall \varphi \in C_c^\infty(\bR,\bC)
          \ ,
          \]
          cosicch\'e, per definizione di $\wa T$, si ha 
          $\wa T F \ = \ -i d / ds \circ F$ sul sottospazio $C_c^\infty(\bR,\bC) \subset L^2(\bR,\bC)$
          {\footnote{In realt\'a si verifica che $(\wa D , \wa T)$ \e la derivata debole
          \[
          \left\{
          \begin{array}{ll}
          \wa D  = \{ u \in L^2(\bR,\bC) \ | \
                      \exists u' \in L^2(\bR,\bC) : 
                      \int u \varphi' = - \int u' \varphi \ , \ 
                      \forall \varphi \in C_c^\infty(\bR,\bC)  \} \ ,
          \\
          \wa T u  = -iu'
          \ ;
          \end{array}
          \right.
          \]
          il dominio $\wa D$ \e noto come lo spazio di Sobolev (complesso) $H^1(\bR,\bC)$ (vedi \S \ref{sec_sobolev});
          per dettagli si veda \cite[\S III.3]{Mal}.}}
          ;
          
          \
          
          \noindent (Suggerimenti: (1) Calcolando il limite (\ref{eq.scopug3}) si trova
                     \[
                     D = \left\{ u \in L^2(\bR,\bC) : \int s^2 |u(s)|^2 < \infty \right\}
                     \ \ , \ \
                     Tu(s) := s u(s)
                     \ , \
                     \forall s \in \bR \ , \ u \in D
                     \ .
                     \]
                     (2) Per ogni $f \in L^1(\bR,\bC) \cap L^2(\bR,\bC)$ e $t,x \in \bR$ si trova 
                     \[
                     \begin{array}{ll}
                     \{ \wa U_t f \}(x) & = 
                     \{ F U_t \breve f \}(x) = \\ & =
                     \int e^{its} \breve f(s) e^{ixs} \ ds = \\ & =
                     \int f(\lambda) e^{is(t+x-\lambda)} \ d \lambda ds  = \\ & =
                     f(x+t) \ ,
                     \end{array}
                     \]
                     per cui $\wa U_t$ \e un operatore di traslazione (Esempio \ref{ex.trans}) e l'operatore autoaggiunto
                     associato \e l'estensione autoaggiunta della derivata (Esempio \ref{ex_der}).
                     (3) Semplici manipolazioni algebriche.
                     (4) Si calcoli la trasformata di Fourier di $T \varphi$ usando il punto (1)).
}

\

\noindent \textbf{Esercizio \ref{sec_fourier}.8 (La rappresentazione regolare).} 
{\it
Sia $G$ un gruppo topologico (localmente compatto, di Hausdorff) e $\mu : \mM \to \wa \bR^+$, $\mM \subset 2^G$, la sua misura di Haar (invariante a sinistra, ovvero $\mu(gE) = \mu E$, $\forall E \in \mM$). Si consideri lo spazio di Hilbert $\mH := L_\mu^2(G,\bC)$ e si denoti con $U \mH$ il gruppo degli operatori unitari.
\textbf{(1)} Preso $h \in G$ si definisca
$u_h(g) := u(hg)$, $\forall u \in \mH$, $g \in G$,
e si mostrino le seguenti uguaglianze:
\[
\{ \alpha u + v  \}_h = \alpha u_h + v_h
\ \ , \ \
\| u \|_2 = \| u_h \|_2
\ \ , \ \
\forall \alpha \in \bC \ , \ u,v \in \mH 
\ ,
\]
cosicch\'e l'applicazione $U_h u := u_h$, $u \in \mH$, definisce un operatore unitario $U_h \in U \mH$.
\textbf{(2)} Si mostri che $U_{hh'} = U_h U_{h'}$, $\forall h,h' \in G$.
\textbf{(3)} Assumendo che lo spazio delle funzioni continue a supporto compatto $C_c(G,\bC)$ sia denso in $\mH$
          {\footnote{Ci\'o \e sempre garantito quando $G$ \e uno spazio metrico (e quindi normale), vedi \S \ref{sec_appLp}. 
          In particolare, $C_c(G,\bC)$ \e denso in $\mH$ quando $G$ \e un gruppo di Lie.}},
          si mostri che 
          \[
          \lim_{h \to h'} \| U_hu - U_{h'}u \|_2  \to 0 \ \ , \ \ \forall u \in \mH \ .
          \]
\textbf{(4)} Si mostri che, presa $f \in L_\mu^1(G,\bC)$ e denotato con $(\cdot,\cdot)$ il prodotto scalare di $\mH$, 
          l'applicazione
          \[
          A_f : \mH \times \mH \to \bC
          \ \ , \ \
          A_f(u,v) := \int_G f(h) (u,U_hv) \ d \mu(h)
          \ ,
          \]
          \e una forma sesquilineare limitata, che definisce quindi un operatore limitato $T_f \in B(\mH)$ tale che
          $(u,T_fv) = A_f(u,v)$, $\forall u,v \in \mH$. Si verifichino le seguenti relazioni:
          \[
          T_{f*l} = T_f T_l \ , \ 
          T_{\alpha f + l} = \alpha T_f + T_l \ , \
          \| T_f \| \leq \| f \|_1
          \ \ , \ \
          \forall f,l \in L_\mu^1(G,\bC)
          \ , \
          \alpha \in \bC
          \ .
          \]
\textbf{(5)} Si verifichino i punti precedenti nei casi:(i) $G = \bR$ (gruppo additivo), dove $\mu$ \e la misura di Lebesgue;
          (ii) $G = \bZ$ (sempre gruppo additivo), dove $\mu$ \e la misura di enumerazione (vedi Esempio \ref{ex_misnumZ}).
\textbf{(6)} Sia assuma che $G$ sia compatto e si verifichi che vale l'inclusione $G^* \subset \mH$, dove $G^*$ \e il gruppo 
          dei caratteri di $G$. 
\textbf{(7)} Si verifichi che 
          $\{ U_h \chi \}(g) = \chi(h) \chi(g)$, $\forall \chi \in G^* \subset \mH$, $h,g \in G$,
          e che
          $\{ T_f \chi \}(g) = \wa f(\chi) \chi (g)$, 
          $\forall f \in L_\mu^1(G,\bC)$, $\chi \in G^*$, $g \in G$.

\

\noindent (Suggerimenti: (1) La linearit\'a \e ovvia, mentre per l'isometria si usi l'invarianza per traslazioni di $\mu$; 
                         (3) Si verifichi per $u \in C_c(G,\bC)$ e poi si argomenti per densit\'a;
                         (4) Si proceda come nell'Esercizio 8.5;
                         (6) Si osservi che $\chi \in G^*$ \e continua e che $\mu$ \e una misura finita.
                         (7) Si osservi che, grazie al teorema di Riesz, basta verificare che
                             $A_f(u,\chi) = \wa f(\chi) (u,\chi )$, $\forall u \in \mH$,
                             e si usi il teorema di Fubini sull'integrale doppio $A_f(u,\chi)$).
}

%
%
%

\newpage
\section{Analisi Complessa.}
\label{sec_compl}

Consideriamo un'applicazione $f : U \to \bC$, con $U \subseteq \bC$ aperto. Scrivendo, per ogni $z \in U$, $z = x +iy$, $x,y \in \bR$, ed $f(z) = u+iv$, $u,v \in \bR$, troviamo che $f$ pu\'o essere riguardata come una coppia di funzioni reali di due variabili reali, cosicch\'e scriviamo
\[
f(z) \ \equiv \ f(x,y) \ \equiv \ u(x,y) + i v(x,y) \ . 
\]
Dunque valgono per le funzioni complesse tutti i risultati dimostrati per le funzioni di due variabili reali a valori in $\bR^2$, in particolare quelli inerenti le forme differenziali (\S \ref{sec_fdif}). Tuttavia l'analisi delle funzioni di variabile complessa presenta delle importanti peculiarit\'a, e gioca un ruolo importante in svariati ambiti, dalla geometria (si pensi ad esempio al teorema di Riemann-Roch) all'ingegneria elettrica, per non parlare della teoria quantistica dei campi.

\subsection{Funzioni olomorfe.}

In analogia al caso reale (in pi\'u variabili, vedi (\ref{eq_fRn1})), diremo che $f$ \e derivabile in $\zeta \in U$ lungo la direzione 
$v \in \bT := \{ z \in \bC : |z|=1 \}$,
se esiste finito il limite (con $t \in \bR$)
\begin{equation}
\label{eq_olom0}
\frac{\partial f}{\partial v} (\zeta) 
\ := \
\lim_{t \to 0} \frac{ f(\zeta+tv)-f(\zeta) }{ t }
\ .
\end{equation}
In particolare, scegliendo $x = 1$, $y = i$, abbiamo le {\em derivate parziali}
\[
\frac{\partial f}{\partial x} (\zeta) = 
\frac{\partial u}{\partial x} (\zeta) + i \frac{\partial v}{\partial x} (\zeta)
\ \ , \ \
\frac{\partial f}{\partial y} (\zeta) = 
\frac{\partial u}{\partial y} (\zeta) + i \frac{\partial v}{\partial y} (\zeta)
\ .
\]
Al solito, parleremo di funzioni di classe $C^k(U,\bC)$, $k = 1 , \ldots , +\infty$. 

\

Ora, vogliamo introdurre una nozione diversa di derivata, intesa stavolta come limite rispetto alla variabile $z$,
\begin{equation}
\label{eq_olom2}
f'(\zeta) := \lim_{z \to \zeta} \frac{ f(z)-f(\zeta) }{ z-\zeta } 
\ \ , \ \
\forall \zeta \in U 
\ ,
\end{equation}
mettendo in evidenza il fatto che $z$ tende a $\zeta$ a prescindere dalla direzione. Diremo che una funzione $f : U \to \bC$ \e {\em olomorfa} in $U$ se per ogni $\zeta \in U$ esiste in $\bC$ il limite (\ref{eq_olom2}), che chiameremo la {\em derivata complessa di $f$ in $\zeta$}.

Come si esprime la nozione di olomorfia in termini delle usuali derivate parziali? Per rispondere a questa domanda osserviamo (\ref{eq_olom0}) e notiamo che a denominatore del rapporto incrementale in effetti non appare la differenza 
$z - \zeta = tv$, $z := \zeta + tv$,
bens\'i il parametro $t$. Dunque se vogliamo parlare di derivata nel limite $z \to \zeta$ occorre innanzitutto dividere $\partial f / \partial v$ per $v$; l'indipendenza dalla direzione si traduce invece col fatto che la funzione
\begin{equation}
\label{eq_olom1}
U \times \bT \to \bC
\ \ , \ \
(\zeta , v) \mapsto \frac{1}{v} \frac{\partial f}{\partial v} (\zeta)
\ ,
\end{equation}
non dipende da $v \in \bT$ e coincide proprio con la derivata complessa di $f$.

\

Applicando i consueti risultati di derivazione di combinazioni lineari e prodotti di funzioni otteniamo che l'insieme delle funzioni olomorfe in $U$ \e un'algebra, la quale sar\'a denotata con $\mO(U)$; l'algebra delle funzioni olomorfe in ogni intorno di $U$ sar\'a denotata invece con $\mO(\ovl U)$.

Diamo ora un'utile caratterizzazione delle funzioni olomorfe. Come vedremo negli esempi successivi, "moltissime" funzioni in $C^\infty(U,\bC)$ non sono olomorfe, e lo strumento pi\'u comodo per verificarlo \e dato proprio dal Lemma seguente.
\begin{lem}
\label{lem_CR}
Sia $f \in C^\infty(U,\bC)$, $f(z) = u(x,y) + iv(x,y)$. Allora $f$ \e olomorfa in $U$ se e solo se valgono le \textbf{equazioni di Cauchy-Riemann}
\begin{equation}
\label{eq_CR}
\frac{\partial f}{\partial x} + i \frac{\partial f}{\partial y} = 0
\ \Leftrightarrow \ 
\frac{\partial u}{\partial x} =   \frac{\partial v}{\partial y}
\ \ , \ \
\frac{\partial u}{\partial y} = - \frac{\partial v}{\partial x} 
\ .
\end{equation}
\end{lem}

\begin{proof}[Dimostrazione]
Sia $f$ olomorfa. Scegliendo di passare al limite $z \to \zeta$ lungo la direzione $v$ abbiamo $z - \zeta = tv$; cosicch\'e, scegliendo $v = 1$, $w = i$, troviamo
\[
\frac{\partial f}{\partial x} (\zeta) =
\frac{1}{v} \frac{\partial f}{\partial v} (\zeta) =
\frac{1}{w} \frac{\partial f}{\partial w} (\zeta) =
-i \frac{\partial f}{\partial y} (\zeta) \ .
\]
Scrivendo esplicitamente $f(z) = u(x,y)+iv(x,y)$, otteniamo le relative equazioni in $u,v$.
Supponiamo ora che $f$ soddisfi (\ref{eq_CR}). Introduciamo le notazioni
\[
\zeta = x_0 + i y_0
\ \ , \ \
z     = x   + i y
\ \ , \ \
\xi   = x   - x_0
\ \ , \ \
\eta  = y   - y_0
\ ,
\]
cosicch\'e
\[
z - \zeta = \xi + i \eta \ .
\]
Essendo $f \in C^\infty(U,\bC)$ abbiamo, per $z$ appartenente ad un opportuno intorno $U' \subseteq U$
\[
f(z) - f(\zeta) =
\xi  \frac{\partial f}{\partial x} (\zeta) + 
\eta \frac{\partial f}{\partial y} (\zeta) +
\xi^2 H_x + \eta^2 H_y + \xi \eta H_{xy} 
\ ,
\]
dove $H_x , H_y , H_{xy} \in C^\infty(U',\bC)$. Applicando le equazioni di Cauchy-Riemann troviamo
\[
f(z) - f(\zeta) =
\frac{\partial f}{\partial x} (\zeta) (z-\zeta) +
\xi^2 H_x + \eta^2 H_y + \xi \eta H_{xy} \ .
\]
Ora, per diseguaglianza triangolare abbiamo $\xi,\eta \leq |z-\zeta| = |\xi + i \eta|$; per cui, posto $M := \max_{\ovl{U'}} \{ |H_x| , |H_y| , |H_{xy}| \}$, concludiamo che
\[
\left| \frac{f(z) - f(\zeta)}{z-\zeta}  - 
       \frac{\partial f}{\partial x} (\zeta)
\right|
=
\frac{1}{|\xi + i \eta|}
|\xi^2 H_x + \eta^2 H_y + \xi \eta H_{xy}|
\leq
M(\xi + 2\eta)
\stackrel{(\xi,\eta) \to 0}{\longrightarrow} 0
\ .
\]
Dunque, $f'(\zeta)$ \e ben definito come limite.
\end{proof}

\begin{cor}
Se $f(z) = u(x,y) + i v(x,y)$ \e olomorfa allora $u,v \in C^\infty(U)$ soddisfano l'equazio- ne di Laplace:
\[
\Delta u := \frac{\partial^2u}{\partial x^2} + \frac{\partial^2u}{\partial y^2}  =  0
\ \ , \ \
\Delta v  =  0 \ .
\]
\end{cor}

\begin{ex}[\cite{Arb}]{\it 
Usando le equazioni di Cauchy-Riemann, si trova facilmente che le funzioni
\[
f(z) = x^3 + iy
\ \ , \ \
f(z) = \ovl z
\ \ , \ \
f(z) = |z|^2 
\ \ , \ \
f(z) = \sin x + i \cos y
\]
\textbf{non} sono olomorfe. D'altra parte, ogni funzione del tipo $f(z) := z^k$, $k \in \bN$, \e olomorfa.
}\end{ex}

\subsection{Serie di potenze e funzioni analitiche.}
\label{sec_anal}

Una serie di potenze si presenta nel seguente modo:
\begin{equation}
\label{eq_SPA1}
\sum_{n=0}^\infty a_n z^n \ \ , \ \ a_n \in \bC \ .
\end{equation}
La questione della convergenza di una serie di potenze \e completamente risolta dal seguente
\begin{thm}[Abel]
Data una serie di potenze del tipo (\ref{eq_SPA1}), esiste un reale esteso $r \in [0,+\infty]$, detto \textbf{raggio di convergenza}, con le seguenti propriet\'a:
per ogni $z \in \bC$ con $|z| < r$, risulta che $\sum_n a_n z^n$ \e assolutamente convergente, e quindi convergente; se $\rho \in (0,r)$, allora (\ref{eq_SPA1}) converge \textbf{uniformemente} per ogni $z$ tale che $|z| \leq \rho$; se $z \in \bC$ e $|z| > r$, allora (\ref{eq_SPA1}) non converge.
\end{thm}

\begin{proof}[Dimostrazione]
Definiamo $r$ tramite la {\em formula di Hadamard}
\begin{equation}
\label{eq_HAD}
\frac{1}{r}
\ := \
\lim_n \sup |a_n|^{1/n} \ .
\end{equation}
Se $|z| < r$, allora per ogni $\rho \in (|z|,r)$ si ha chiaramente
\[
\frac{|z|}{r} \ < \ \frac{\rho}{r} \ < \ 1 
\ \ , \ \
\frac{|z|}{\rho} \ < \ 1
\ . 
\]
Inoltre, per definizione di $r$, esiste $n_0 \in \bN$ tale che 
\[
\frac{1}{\rho} \ > \ \frac{1}{r} \ > \ |a_n|^{1/n}
\ \ , \ \
n > n_0
\ .
\]
Dunque,
\[
1 > 
\left( \frac{|z|}{\rho} \right)^n >
|a_n| |z|^n
\ ,
\]
per cui (\ref{eq_SPA1}) \e assolutamente convergente (e quindi convergente). 
\end{proof}

\begin{defn}
Sia $U \subseteq \bC$ un aperto. Una funzione $f : U \to \bC$ si dice \textbf{analitica} in $U$ se per ogni $\zeta \in U$ esiste un disco $\Delta = \Delta ( \zeta , r )$, $\Delta \subseteq U$, tale che
\begin{equation}
\label{eq_FA}
f(z) = \sum_{n=0}^\infty a_n ( z-\zeta )^n \ \ , \ \ z \in \Delta \ ,
\end{equation}
per opportuni coefficienti $a_n \in \bC$.
\end{defn}

Allo scopo di iniziare a chiarire la natura delle funzioni analitiche, ed in particolare dei coefficienti $a_n$, $n \in \bN$, diamo il seguente
\begin{lem}
\label{lem_SPA2}
Ogni funzione analitica $f : U \to \bC$ \e olomorfa. Inoltre, ogni derivata $n$-esima $f^{(n)}$ \e analitica in $U$ (e quindi olomorfa), ed $f$ si sviluppa nella \textbf{serie di Taylor}
\begin{equation}
\label{eq_SPA2}
f(z) 
\ = \
\sum_{n=0}^\infty \frac{1}{n!} \ f^{(n)}(\zeta) (z-\zeta)^n 
\ \ , \ \
z \in \Delta \ .
\end{equation}
\end{lem}

\begin{proof}[Dimostrazione]
Per mostrare che $f$ \e olomorfa verifichiamo che soddisfa le equazioni di Cauchy-Riemann. A tale scopo, consideriamo al solito $\zeta \in \Delta \subseteq U$ tale che $f(z)$ \e della forma (\ref{eq_FA}), e ne denotiamo con 
\begin{equation}
\label{eq_SPA3}
g_k (z) := \sum_{n=0}^k a_n (z-\zeta)^n \ \ , \ \ k \in \bN \ ,
\end{equation}
le somme parziali. Per il Teorema di Abel, esiste $\rho > 0$ tale che $g_n \to f$ {\em uniformemente} in $\Delta_\rho := \{ z \in U : |z-\zeta| \leq \rho \} \subset \Delta$. Inoltre, poich\'e
\[
\lim_n \sup |a_n|^{1/n}
\ = \ 
\lim_n \sup (n|a_n|)^{1/n}
\ ,
\]
la serie delle derivate
\begin{equation}
\label{eq_SPA4}
g'_k(z) := \sum_{n=1}^k n a_n (z-\zeta)^{n-1}
\end{equation}
ha lo stesso raggio di convergenza di (\ref{eq_SPA3}), e quindi converge uniformemente in $\Delta_\rho$. Per gli usuali teoremi di derivazione di serie uniformemente convergenti, $f$ \e derivabile, e
\[
\frac{\partial g_k}{\partial x} \stackrel{k}{\to} \frac{\partial f}{\partial x}
\ \ , \ \
\frac{\partial g_k}{\partial y} \stackrel{k}{\to} \frac{\partial f}{\partial y}
\ .
\]
Per cui, poich\'e ogni $g_k$ soddisfa (\ref{eq_CR}), abbiamo che $f$ soddisfa (\ref{eq_CR}) e quindi \e olomorfa. L'argomento precedente mostra anche che $f' = \lim_k g'_k$ \e analitica e quindi olomorfa. Infine, confrontando i termini di (\ref{eq_SPA3}) e (\ref{eq_SPA4}) otteniamo $f^{(n)}(\zeta) = n! a_n$, da cui (\ref{eq_SPA2}).
\end{proof}

Un esempio fondamentale di funzione analitica (e quindi olomorfa) \e dato da
\[
f(z) := \frac{1}{z} \ \ , \ \ z \in U := \bC - \{ 0 \} \ .
\]
Infatti, per ogni $\zeta \in U$ e $z : |z-\zeta| < |\zeta|$ troviamo che la serie
\[
\sum_{n=0}^\infty
(-1)^n \left( \frac{z-\zeta}{\zeta} \right)^n
\]
\'e assolutamente convergente, e quindi uniformemente convergente in ogni disco chiuso $\Delta \subset \{ z : |z-\zeta| < |\zeta| \}$. La somma \e chiaramente data (per la regola di divisione tra polinomi) da 
\[
\frac{1}{1 - (-1) (z-\zeta)\zeta^{-1} } \ = \frac{\zeta}{z} \ ,
\]
per cui
\begin{equation}
\label{eq_1z}
f(z) 
\ = \ 
\frac{1}{z} 
\ = \
\sum_{n=0}^\infty
\frac{ (-1)^n }{ \zeta^{n-1} } ( z-\zeta)^n
\ \ , \ \ 
|z-\zeta| < |\zeta|
\ .
\end{equation}

\subsection{Integrazione complessa.}
\label{sec_intc}

In realt\'a le propriet\'a di analiticit\'a ed olomorfia sono del tutto equivalenti. Il fatto che ogni funzione olomorfa \e analitica si dimostra utilizzando il teorema di Cauchy, il quale stabilisce l'annullarsi dell'integrale di una funzione olomorfa su una curva chiusa.

\

\noindent \textbf{Curve a valori complessi.} Iniziamo richiamando le seguenti terminologie: una {\em curva}
\[
\gamma : I := [a,b] \to \bC 
\ \ , \ \
\gamma(t) = x(t) + i y(t)
\ , \
t \in I
\]
(dove $x,y : I \to \bR$, $k=1,2$, al solito sono le componenti di $\gamma$) si dice {\em chiusa} se $\gamma(a) = \gamma (b)$, e {\em semplice} se $\gamma |_{[a,b)}$, $\gamma |_{(a,b]}$ sono iniettive. 
Diremo inoltre che $\gamma$ \e {\em regolare} se essa \e continua ed ha derivata
\[
\gamma' : I \to \bC
\ \ , \ \
\gamma'(t) := x'(t) + i y'(t)
\]
continua 
{\footnote{Qui definiamo $\gamma'(a)$ e $\gamma'(b)$ rispettivamente come i limiti per $t \to a^+$ e $t \to b^-$.}}
in $\dot{I} := (a,b)$. 
Infine, diremo che la curva continua $\gamma$ \e regolare {\em a tratti} se $I$ pu\'o essere suddiviso in un numero finito $I_1 , \ldots , I_n$ di sottointervalli chiusi tali che ogni restrizione $\gamma |_{I_k}$, $k = 1 , \ldots , n$, sia regolare. In tal caso $\gamma'$ pu\'o essere comunque definita su tutto $I$, estendendo per continuit\'a a sinistra in ogni estremo destro di $I_k$, $k = 1 , \ldots , n$, per cui $\gamma'$ risulta essere continua a tratti.
Preso un aperto $U \subset \bC$, denotiamo con $C^{reg}(I,U)$ l'insieme delle curve regolari a tratti su $I$ con immagine contenuta in $U$.

\begin{ex}{\it
La curva seguente $\gamma : [0,4] \to \bC$ descrive il perimetro di un quadrato in $\bC$ ed \e regolare a tratti,
\[
\gamma(t) :=
\left\{
\begin{array}{ll}
t \ \ , \ \ t \in [0,1) \ ,
\\
1 + i(t-1)   \ \ , \ \ t \in [1,2) \ ,
\\
i+3-t \ \ , \ \ t \in [2,3) \ ,
\\
i(4-t) \ \ , \ \ t \in [3,4] \ .
\end{array}
\right.
\ \Rightarrow \
\gamma'(t) =
\left\{
\begin{array}{ll}
1 \ \ , \ \ t \in [0,1] \ ,
\\
i   \ \ , \ \ t \in (1,2] \ ,
\\
-1 \ \ , \ \ t \in (2,3] \ ,
\\
-i \ \ , \ \ t \in (3,4] \ .
\end{array}
\right.
\]
%
}
\end{ex}

\noindent \textbf{Integrali curvilinei.} Sia ora $f \in C(U,\bC)$; per ogni $\gamma \in C^{reg}(I,U)$ definiamo
\begin{equation}
\label{def_int.comp}
\int_\gamma f \ dz
\ := \
\int_a^b f \circ \gamma(t) \cdot \gamma'(t) \ dt
\ .
\end{equation}
La funzione integranda $f \circ \gamma(t) \cdot \gamma'(t)$ \e per definizione continua a tratti, per cui l'integrale precedente esiste gi\'a nel senso di Riemann.
Si noti che cambiando l'orientazione di $\gamma$, il che corrisponde ad effettuare il cambio di variabile $t \mapsto \ovl \gamma(t) := \gamma (b+a-t)$, $t \in I$, si ottiene l'inversione di segno
\[
\int_{\ovl \gamma} f \ dz \ = \ 
- \int_a^b f \circ \gamma (b+a-t) \cdot \gamma'(b+a-t) \ dt \ = \
- \int_a^b f \circ \gamma(s) \cdot \gamma'(s) \ ds \ = \
- \int_\gamma f \ dz \ .
\]
Infine, osserviamo che $\int_\gamma f dz$ si pu\'o esprimere anche in termini di integrali di forme differenziali, definendo la $1$-forma
\[
\omega_f : U \to \bR^{2,*} \simeq \bR^2 \simeq \bC
\ \ , \ \
\omega_f (x,y) := f(x+iy) \{ dx + i \ dy \}
\ ,
\]
e procedendo per integrazione di forme differenziali: infatti, usando (\ref{eq.int.form}) con $m=1$, $n=2$, si verifica immediatamente che (\ref{def_int.comp}) coincide con l'integrale di $\omega_f$. {\em Nel seguito, seguendo una notazione standard, scriveremo $\omega_f \equiv f dz$}.

\begin{ex}
\label{ex_int.curv}
{\it
\textbf{(1)} Siano $\zeta, \zeta' \in \bC$ e $\gamma : I := [0,1] \to \bC$ definita da $\gamma (t) := (1-t) \zeta' + t \zeta$. Allora $\gamma'(t) = \zeta - \zeta'$ e
\begin{equation}
\label{eq_triv}
\int_\gamma 1 \, dz
\ = \
\int_0^1 1 \cdot (\zeta -\zeta') \, dt 
\ = \
\zeta - \zeta' \ .
\end{equation}
\textbf{(2)} Sia $\zeta \in \bC$, $\eps > 0$ e $\gamma : I \to \bC - \{ \zeta \}$ la curva (chiusa e semplice)
\[
\gamma (t) \, := \, \zeta + \eps e^{2 \pi i t } \ \ , \ \ t \in I := [0,1] \ .
\]
Allora
\begin{equation}
\label{eq_IC01}
\int_\gamma \, \frac{dz}{z-\zeta} 
\ = \
2 \pi i
\ .
\end{equation}
Per la verifica, osserviamo che l'integrale precedente \e ben definito perch\'e
$f(z) := (z-\zeta)^{-1}$ \e olomorfa per $z \in \bC - \{ \zeta \}$;
dunque
$f \circ \gamma(t) = \eps^{-1} e^{-2 \pi i t }$
e 
$\gamma'(t) = \eps 2 \pi i e^{2 \pi i t }$,
cosicch\'e (\ref{eq_IC01}) si riduce all'integrale su $I$ della funzione costante $2 \pi i$, il quale evidentemente \e uguale proprio a $2 \pi i$.
}
\end{ex}

\

\noindent \textbf{Domini regolari ed il teorema di Cauchy.} Un {\em dominio regolare} \e un aperto limitato $U \subset \bC$ il cui bordo $\partial U$ \e costituito da un numero finito di curve chiuse, semplici, e regolari a tratti. 
Osser- viamo che dall'ipotesi di limitatezza di $U$ segue che $\bC - U$ possiede una, ed una sola, componente connessa non limitata $V$; chiamiamo {\em frontiera esterna} di $U$ l'insieme di curve 
$\partial U_{est} := V \cap \ovl{U}$.
Chiamiamo invece {\em frontiera interna} l'insieme di curve 
$\partial U_{int} := \partial U - \partial U_{est}$.
Per distinguere la frontiera interna da quella esterna si prende, per convenzione, orientazione antioraria per le curve in $\partial U_{est}$, e oraria per quelle in $\partial U_{int}$.

\begin{ex}
{\it
\textbf{(1)} Consideriamo la corona circolare aperta $U \subset \bC$,
$U := \{ z \in \bC : 1 < |z| < 2  \}$.
Allora $U$ \e un dominio regolare, il cui bordo \e l'unione disgiunta $\partial U = \partial U_{int} \dot{\cup} \partial U_{est}$ delle curve
\[
\left\{
\begin{array}{ll}
\partial U_{int} : [0,1] \to \bC \ \ , \ \ \partial U_{int}(t) := e^{2 \pi i (1-t)} \ ,
\\
\partial U_{est} : [0,1] \to \bC \ \ , \ \ \partial U_{est}(t) := 2 e^{2 \pi i t } \ .
\end{array}
\right.
\]
Osservare l'orientazione inversa di $\partial U_{int}$ rispetto a $\partial U_{est}$.
\textbf{(2)} Come variazione del tema precedente consideriamo 
$\Delta(\zeta,r) := \{ z \in \bC : |z-\zeta| < r \}$
e definiamo
\[
U := \Delta(0,4) - \{ \ovl{\Delta(2,1)} \cup \ovl{\Delta(-2,1)} \} \ .
\]
Allora $\partial U_{est}$ \e la circonferenza di raggio $4$ centrata in $0$, con orientazione antioraria, mentre $\partial U_{int}$ \e l'unione delle circonferenze di raggio $1$ centrate, rispettivamente, in $2$ e $-2$, orientate entrambe in senso orario.
}
\end{ex}

Ora, per additivit\'a dell'integrale di Riemann troviamo che se 
$\partial U = \gamma_1 \dot{\cup} \ldots \dot{\cup} \gamma_n$,
dove ogni $\gamma_k$, $k = 1 , \ldots , n$, \e una curva chiusa, semplice e regolare a tratti, allora
\[
\int_{\partial U} f \ dz \ = \ \sum_k \int_{\gamma_k} f \ dz \ .
\]
Sottolineiamo il fatto che l'orientazione delle $\gamma_k$ della frontiera interna \e opposta rispetto alle rimanenti altre, il che comporta, come abbiamo visto poc'anzi, il cambio di segno dei relativi integrali qualora volessimo scriverle con l'orientazione antioraria.
\begin{thm}[Cauchy-Goursat]
\label{thm_CG}
Sia $U \subset \bC$ un dominio regolare ed $f \in \mO(\ovl U)$. Allora 
\begin{equation}
\label{eq_thm_cauchy}
\int_{\partial U} f \ dz \ = \ 0 \ .
\end{equation}
\end{thm}

E' disponibile in letteratura una serie di dimostrazioni del teorema di Cauchy, di natura sia analitica che geometrica. La tecnica pi\'u classica \e quella di approssimare $U$ con un insieme di rettangoli sui quali \e pi\'u semplice verificare la validit\'a di (\ref{eq_thm_cauchy}), si veda \cite{Ahl}. Nelle righe seguenti mostriamo come il teorema di Cauchy sia conseguenza del {\em teorema di Stokes
{\footnote{Visto che $\bC$ ha dimensione reale $2$ il teorema di Stokes si riduce alla formula di Gauss-Green trattata in \S \ref{sec_fdif}.}}
}: se $\omega := p(x,y)dx + q(x,y)dy$ \e una $1$-forma $C^\infty(\ovl U,\bC)$, $U \subset \bR^2$ dominio regolare, allora
\begin{equation}
\label{q_stokes}
\int_U d\omega \ = \ \int_{\partial U} \omega 
\ \ , \ \
d\omega 
:= 
\left( \frac{\partial q}{\partial x} - \frac{\partial p}{\partial y}  \right)
\ dx dy
\ .
\end{equation}
Ora, l'osservazione cruciale \e che se $f \in \mO(U)$, allora applicando (\ref{eq_CR}) troviamo
\[
d(fdz) \ = \
d(fdx+ifdy) \ = \
      \frac{\partial f}{\partial y} \ dxdy \
- \ i \frac{\partial f}{\partial x} \ dxdy \ = \
0
\ .
\]
Dunque, {\em se $f \in \mO(\ovl U)$, allora la $1$-forma $fdz$ \e chiusa}. Applicando il teorema di Stokes a $\omega = fdz$ otteniamo immediatamente (\ref{eq_thm_cauchy}). 
Possiamo ora dimostrare il seguente fondamentale teorema:
\begin{thm}[Formula di Cauchy]
Sia $U \subset \bC$ un dominio regolare ed $f \in \mO(\ovl U)$. Allora
\begin{equation}
\label{eq_IC02}
f(\zeta) 
\ = \
\frac{1}{2 \pi i} \int_{\partial U} \frac{f(z)}{z-\zeta} \ dz
\ \ , \ \
\forall \zeta \in U
\ .
\end{equation}
\end{thm}

\begin{proof}[Dimostrazione]
La funzione $\{ z \mapsto (z-\zeta)^{-1} \}$ \e analitica in $U - \{ \zeta \}$ (si veda (\ref{eq_1z})), e quindi olomorfa (Lemma \ref{lem_SPA2}). Per cui,
\[
z \mapsto
\frac{f(z)}{z-\zeta}
\ \ , \ \
z \in U - \{ \zeta \}
\ ,
\]
\'e una funzione olomorfa. Sia ora $\Delta = \Delta (\zeta,\eps) \subset U$, cosicch\'e $f(z)(z-\zeta)^{-1}$ \e olomorfa in $U-\ovl \Delta$. Definendo $\gamma$ come nell'Esempio \ref{ex_int.curv}(2) per il teorema di Cauchy si trova, tenendo conto delle orientazioni
{\footnote{In termini espliciti abbiamo $\partial (U-\ovl \Delta) = \partial U \dot{\cup} \ovl \gamma$. Per cui, visto che l'integrale cambia segno passando da $\ovl \gamma$ a $\gamma$, per ogni $g$ olomorfa su $U-\ovl \Delta$ si trova 
$\int_{\partial(U-\ovl \Delta)} g dz = 
 \int_{\partial U} g dz - \int_\gamma g dz \stackrel{Teo.\ref{thm_CG}}{=} 
 0$.}},
\begin{equation}
\label{eq_IC03}
0
\ \stackrel{Teo.\ref{thm_CG}}{=} \
\int_{\partial (U-\ovl \Delta)} \frac{f(z)}{z-\zeta} \ dz
\ \Leftrightarrow \
\int_{\partial U} \frac{f(z)}{z-\zeta} \ dz
\ = \
\int_\gamma \frac{f(z)}{z-\zeta} \ dz
\ .
\end{equation}
Per cui, per dimostrare (\ref{eq_IC02}) \e sufficiente valutare l'integrale su $\gamma$ nell'uguaglianza precedente,
\[
\int_\gamma \frac{f(z)}{z-\zeta} \ dz
=
\int_0^1 
\frac
     { f( \zeta + \eps e^{2 \pi i t}) }
     { \eps e^{2 \pi i t} } \
\eps 2 \pi i e^{2 \pi i t} \ dt
=
2 \pi i \int_0^1 f( \zeta + \eps e^{2 \pi i t}) \ dt
\ .
\]
Ora, (\ref{eq_IC03}) implica che l'integrale precedente in realt\'a non dipende dalla scelta di $\eps > 0$, per cui passando al limite $\eps \to 0$ otteniamo (\ref{eq_IC02}).
\end{proof}

\begin{cor}
Ogni funzione olomorfa su un dominio regolare $U \subset \bC$ \e analitica. Per cui, tutte e sole le funzioni analitiche su $U$ sono le funzioni olomorfe.
\end{cor}

\begin{proof}[Dimostrazione]
Dopo Lemma \ref{lem_SPA2}, \e sufficiente dimostrare la prima affermazione. A tale scopo, presa $f \in \mO(\ovl U)$, applicando la formula integrale di Cauchy, i teoremi di passaggio al limite sotto il segno di integrale, e (\ref{eq_1z}), troviamo, per $\zeta \in \Delta (z_0,\eps)$ e $\gamma$ definita alla solita maniera,
\begin{equation}
\label{eq_IC04}
\begin{array}{ll}
f (\zeta) & =
\frac{1}{2 \pi i} \ \int_\gamma \frac{f(z)}{z-\zeta} \ dz 
\\ \\ & =
\frac{1}{2 \pi i} \ \int_\gamma 
                    \frac{f(z)}
                         { (z-z_0) 
                            \left( 1 - \frac{\zeta-z_0}{z-z_0} \right) }
                    \ dz  
\\ \\ & \stackrel{(\ref{eq_1z})}{=}
\frac{1}{2 \pi i} \ \int_\gamma 
                    \frac{f(z)}{z-z_0}
                    \sum_{n=0}^\infty \left( \frac{\zeta-z_0}{z-z_0} \right)^n \ dz
\\ \\ & =
\sum_{n=0}^\infty  \left[
                   \frac{1}{2 \pi i} \ 
                   \int_\gamma 
                         \frac{f(z)}{(z-z_0)^{n+1}}
                         \ dz
                   \right]
                   (\zeta-z_0)^n \ .
\end{array}
\end{equation}
\end{proof}

\

\noindent \textbf{Alcune conseguenze della formula di Cauchy.} La formula integrale di Cauchy ha una serie di implicazioni, importanti sia dal punto di vista analitico che algebrico-geometrico. L'esposizione di queste prender\'a il resto della sezione.

\begin{rem}
{\it
Da (\ref{eq_IC04}) e (\ref{eq_SPA2}) segue anche la formula
\begin{equation}
\label{eq_IC05}
f^{(n)}(\zeta)
=
\frac{n!}{2 \pi i} \ \int_\gamma \frac{f(z)}{(z-\zeta)^{n+1}} \ dz 
\ \ , \ \
\zeta \in U
\ \ , \ \
n = 0 , 1 , \ldots
\ ,
\end{equation}
la quale implica, nel caso in cui $f$ sia costante,
\begin{equation}
\label{eq_IC06}
\int_\gamma \frac{1}{(z-\zeta)^{n+1}} \ dz \ = \ 0
\ \ , \ \
\zeta \in U
\ \ , \ \
n = 1 , 2 , \ldots \ .
\end{equation}
}
\end{rem}

\begin{thm}[Liouville]
Sia $f \in \mO(\bC)$ limitata. Allora $f$ \e costante.
\end{thm}

\begin{proof}[Dimostrazione]
Per ipotesi esiste $M \in \bR$ tale che $\| f \|_\infty < M$. 
Usando $z-\zeta = \eps e^{2 \pi i t}$ e $dz = \eps 2 \pi i e^{2 \pi i t} dt$
otteniamo, da (\ref{eq_IC02}),
\[
|f^{(n)}(\zeta)|
=
\left| 
\frac{n!}{2 \pi i} \ \int_\gamma \frac{f(z)}{(z-\zeta)^{n+1}} \ dz 
\right|
\leq
\frac{n!}{\eps^n} \ |f(\zeta)|
\leq
\frac{n!}{\eps^n} \ M
\ .
\]
Potendo scegliere $\eps > 0$ arbitrariamente grande, troviamo $f^{n}(\zeta) = 0$
e quindi, esprimendo $f$ in serie di Taylor, otteniamo che $f$ \e costante.
\end{proof}

Il seguente risultato mostra che in effetti (\ref{eq_thm_cauchy}) {\em caratterizza} le funzioni olomorfe:
\begin{thm}[Morera]
\label{thm.morera}
Sia $f \in C(U,\bC)$ tale che, per ogni $\gamma : I \to U$ chiusa e semplice,
\begin{equation}
\label{eq_morera}
\int_\gamma f \ dz = 0 \ .
\end{equation}
Allora $f$ \e olomorfa in $U$.
\end{thm}

\begin{proof}[Dimostrazione]
Fissiamo $z_0 \in U$, e, per ogni $\zeta \in U$, consideriamo una curva semplice $\gamma \in C^{reg}(I,U)$, $\gamma (0) = z_0$, $\gamma (1) = \zeta$; definiamo quindi
\[
F(\zeta) := \int_\gamma f \ dz 
\ .
\]
Da (\ref{eq_morera}) segue che se $\gamma_1$ \e tale che $\gamma_1(0) = z_0$, $\gamma_1(1) = \zeta$, allora 
\[
\int_{\gamma_1 * \gamma^{-1}} f \ dz = 
\int_{\gamma_1} f \ dz - \int_\gamma f \ dz = 0 \ ,
\]
per cui $F(\zeta)$ non dipende dalla scelta di $\gamma$. Mostriamo che $F$ \e olomorfa e che $f = F'$, il che implica che $f$ \e olomorfa (Lemma \ref{lem_SPA2}). Scelto $\zeta \in U$ osserviamo che, per continuit\'a di $f$, per ogni $\eps > 0$ esiste $\delta > 0$ tale che $|f(\zeta)-f(\zeta')| < \eps$ per $\zeta' \in \Delta (\zeta,\delta) \subseteq U$. Denotato con $\gamma_2 : I \to U$ il segmento $\gamma_2(t) = (1-t) \zeta' + t\zeta$, $t \in [0,1]$, stimiamo
\[
\left| \frac{ F(\zeta)-F(\zeta') }{ \zeta-\zeta' } - f(\zeta)  \right|
=
\left| \frac{1}{\zeta-\zeta'} \int_{\gamma_2} [f - f(\zeta)] \ dz \right|
\leq
\frac{1}{|\zeta-\zeta'|} \ \eps |\zeta-\zeta'|
=
\eps
\ ,
\]
avendo usato (\ref{eq_triv}). Per cui $f = F'$ ed il teorema \e dimostrato.
\end{proof}

\begin{thm}
\label{thm.morera.2}
Sia $U \subset \bC$ un dominio regolare. Allora $\mO(\ovl U)$ \e uno spazio di Banach rispetto alla norma dell'estremo superiore su $\ovl U$.
\end{thm}

\begin{proof}[Dimostrazione]
Sia $\{ f_n \} \subset \mO(\ovl U)$ una successione uniformemente convergente ad $f$. Allora $f$ \e continua. Inoltre, per convergenza uniforme, per ogni curva chiusa e semplice $\gamma$ troviamo
\[
\int_\gamma f \ dz 
\ = \
\lim_n \int_\gamma f_n \ dz 
\ = \ 
0
\]
Dunque, applicando il teorema di Morera, $f$ \e analitica.
\end{proof}

\begin{thm}[Principio di continuazione analitica]
Sia $U \subset \bC$ un dominio regolare connesso ed $f \in \mO(\ovl U)$ non nulla. Allora l'insieme degli zeri di $f$ in $U$ non ha punti di accumulazione.
\end{thm}

\begin{proof}[Dimostrazione]
Supponiamo che esista $\zeta \in U$ ed una successione $\{ \zeta_n \}$ (con $\zeta_n \neq \zeta$, $n \in \bN$) tali che $\zeta = \lim_n \zeta_n$, $f(\zeta) = f(\zeta_n) = 0$, $n \in \bN$. Allora, esiste un disco $\Delta = \Delta (\zeta,\eps)$ contenente ogni $\zeta_n$ per $n$ maggiore o uguale di un opportuno $k$, e possiamo scrivere $f(z) = \sum_k a_k (z-\zeta)^k$, $z \in \Delta$. Se $j$ \e il primo indice tale che $a_j \neq 0$, allora
\[
0 \ = \ 
\lim_n \frac{ f(\zeta_n) }{ (\zeta - \zeta_n)^j }  \ = \ 
\sum_{k=j}^\infty a_k \lim_n (\zeta-\zeta_n)^{k-j} 
\ = \
a_j 
\ .
\]
Per cui, $f = 0$.
\end{proof}

Il teorema precedente consente di estendere ai complessi le funzioni reali classiche, considerando gli sviluppi in serie
\[
\sin z := \sum_{n=0}^\infty 
          (-1)^n
          \frac{z^{2n+1}}{(2n+1)!} 
\ \ , \ \ 
z \in \bC 
\ ,
\]
\[
\cos z := \sum_{n=0}^\infty 
          (-1)^n
          \frac{z^{2n}}{(2n)!}     
\ \ , \ \ 
z \in \bC 
\ ,
\]
\[
e^z := \sum_{n=0}^\infty 
       \frac{z^{n}}{n!}     
\ \ , \ \ 
z \in \bC 
\ ,
\]
\[
\log z := \sum_{n=1}^\infty 
          (-1)^{n-1}
          \frac{z^{2n}}{(2n)!}     
\ \ , \ \ 
|z-1| < 1 
\ ,
\]
ed osservando che per $z \in \bR$ otteniamo esattamente le usuali funzioni reali $\sin$, $\cos$, $\exp$, $\log$. Il principio di continuazione analitica assicura che le funzioni sopra definite sono le {\em uniche} funzioni {\em analitiche} che estendono quelle reali date: ad esempio, se $f \in \mO(\bC)$ e $f(t) = \sin t$, $\forall t \in \bR$, allora $f(z) - \sin z$ ha insieme degli zeri contenente $\bR$ e quindi avente punti di accumulazione, per cui deve essere $f(z) - \sin z \equiv 0$. 
Enfatizziamo il fatto che l'unicit\'a dell'estensione sussiste soltanto se ci restringiamo a considerare funzioni analitiche, infatti possiamo trovare sempre un'infinit\'a di estensioni a $\bC$ continue o $C^\infty$.

\begin{ex}[Il catino di Cauchy]
{\it
Consideriamo la funzione reale
\[
f(x)
:=
\left\{
\begin{array}{ll}
e^{-1/x} \ \ , \ \ x>0
\\
0 \ \ , \ \ x \leq 0 \ .
\end{array}
\right.
\]
Un semplice studio mostra che $f \in C_0^\infty(\bR)$. L'insieme degli zeri di $f$ \e dato da $\{ x \leq 0 \}$ ed ha chiaramente punti di accumulazione, per cui non esiste un'estensione analitica di $f$ definita su un aperto $U$ che contenga $\bR$.
}
\end{ex}

\begin{lem}
Sia $U \subset \bC$ un dominio regolare e connesso, ed $f \in \mO(\ovl U)$ tale che $|f|$ sia costante in $U$. Allora $f$ \e costante.
\end{lem}

\begin{proof}[Dimostrazione]
Definendo gli operatori
\[
\frac{\partial}{\partial \ovl z}
\ := \
\frac{\partial}{\partial x} + i \frac{\partial}{\partial y}
\ \ , \ \
\frac{\partial}{\partial z}
\ := \
\frac{\partial}{\partial x} - i \frac{\partial}{\partial y}
\ ,
\]
troviamo che l'equazione di Cauchy-Riemann per $f$ si scrive 
\[
\frac{\partial f}{\partial \ovl z} = 
\frac{\partial \ovl f}{\partial z} = 
0
\ \Leftrightarrow \
f' = \frac{\partial f}{\partial z} 
\]
(vedi (\ref{eq_CR})). Per cui, in $U$ troviamo
\[
0 = 
\frac{\partial}{\partial z} |f|^2 =
\frac{\partial}{\partial z} f \ovl f =
f' \ovl f + f \frac{\partial \ovl f}{\partial z} =
f' \ovl f
\ ,
\]
il che implica $f'=0$.
\end{proof}

\begin{thm}[Principio del massimo]
Sia $U \subset \bC$ un dominio regolare connesso, $f \in \mO(\ovl U)$. Allora il massimo di $|f|$ in $\ovl U$ \e assunto in un punto del bordo $\partial U$.
\end{thm}

\begin{proof}[Dimostrazione]
Supponiamo per assurdo che $\zeta \in U$ sia un punto di massimo per $|f|$. Allora esiste un disco $\Delta$ di centro $\zeta$ e raggio $\eps$ tale che $\ovl \Delta \subseteq U$, e grazie alla formula di Cauchy troviamo
\[
f(\zeta) = 
\int_{\partial \Delta} \frac{f(z)}{z-\zeta} \ dz =
\frac{1}{2\pi} \int_0^{2\pi} f(\zeta+\eps e^{i\theta}) \ d \theta \ ,
\]
il che implica
\begin{equation}
\label{eq_PM}
|f(\zeta)| =: M \leq 
\frac{1}{2\pi} \int_0^{2\pi} |f(\zeta+\eps e^{i\theta})| \ d\theta
\ .
\end{equation}
Supponiamo che esista $z' := \zeta + \eps e^{i \theta} \in \partial \Delta$ con $|f(z')| < M$. Allora troviamo $|f(z)| < M$ in un intorno $V \subseteq \partial \Delta$ di $z'$, e quindi (visto che in $\partial \Delta - V$ abbiamo $|f(z)| \leq M$)
\[
M > 
\frac{1}{2\pi} \int_0^{2\pi} |f(\zeta+\eps e^{i\theta})| \ d\theta
\ ,
\]
il che contraddice (\ref{eq_PM}). Dunque $|f(z)| = M$, $z \in \partial \Delta$.
Applicando questo ragionamento per ogni disco di raggio minore di $\eps$, concludiamo che $|f|_\Delta = M$. Applicando il teorema precedente, troviamo che $f$ \e costante in $\Delta$ e, per continuazione analitica, concludiamo che $f$ \e costante, il che fornisce una contraddizione.
\end{proof}

\begin{rem}{\it 
Si hanno versioni del teorema precedente per funzioni armoniche (vedi \cite[IV.2.3]{TS}), e per soluzioni del problema di Dirichlet non omogeneo (vedi \cite[VIII.5,IX.7]{Bre}). 
}
\end{rem}

\begin{cor}[Teorema fondamentale dell'algebra.]
Sia $p \in \bC[z]$ non costante. Allora $p$ ammette almeno uno zero in $\bC$.
\end{cor}

\begin{proof}[Dimostrazione]
Posto $p(z) = \sum_{k=0}^n a_kz^k$, consideriamo un arbitrario $\rho > 0$ ed osserviamo che, applicando ricorsivamente la diseguaglianza triangolare, per $|z| > \rho$ abbastanza grande troviamo
\[
|p(z)| \ \geq \
|z^n| \left\{ |a_n| - \sum_{k=0}^n \frac{|a_k|}{|z|^{n-k}} \right\} \ \geq \
c |z|^n
\ ,
\]
dove $c$ \e un'opportuna costante positiva. Se $p$ fosse privo di zeri avremmo che $p^{-1}$ sarebbe olomorfa, ed applicando il principio di massimo ad $U = \Delta := \Delta (0,\rho)$ troveremmo
\[
\frac{1}{|p(z)|} \leq \frac{1}{c} \frac{1}{|z|^n}
\ \Rightarrow \
M := \max_{\Delta} \frac{1}{|p(z)|} \leq \frac{1}{|c\rho^n|} 
\stackrel{\rho \to \infty}{\lra} 0
\ ,
\]
il che fornisce una contraddizione.
\end{proof}

\subsection{Funzioni meromorfe ed il teorema dei residui.}

Sia $U \subseteq \bC$ aperto ed $f \in \mO(U)$. Un punto $z_0 \in U$ si dice {\em zero di $f$ di ordine $n \in \bN$} se esiste un intorno $V \ni z$ tale che
\[
f(z) = a_n (z-z_0)^n + a_{n+1}(z-z_0)^{n+1} + \ldots
\ \ , \ \
a_k \neq 0 
\ , \
z \in V 
\ ;
\]
in termini equivalenti, esiste una funzione olomorfa $h \in \mO(V)$ tale che $h(z_0) \neq 0$ e
\[
f(z) = (z-z_0)^n h(z) \ \ , \ \ z \in V \ .
\]
Un {\em punto singolare (o singolarit\'a)} per $f$ \e un punto $z_0 \in \bC$ sul quale $f$ non \e definita; diremo che $z_0$ \e {\em isolato} se esiste un intorno $V \ni z_0$ tale che $f$ sia definita in $V - \{ z_0 \}$. Una singolarit\'a isolata $z_0$ \e un {\em polo di ordine $n \in \bN$} se esiste un intorno $V \ni z_0$ tale che 
\[
h(z) := (z-z_0)^n f (z) \ \ , \ \ z \in V - \{ z_0 \} \ ,
\]
si estende per continuit\'a ad una funzione {\em olomorfa e non nulla} in $V$. Usando lo sviluppo di Taylor di $h$, concludiamo che $f$ ammette uno {\em sviluppo di Laurent}
\begin{equation}
\label{eq_LAURENT}
f(z) = 
a_{-n}(z-z_0)^{-n} +
\ldots             +
a_{-1}(z-z_0)^{-1} +
a_0                +
a_1(z-z_0)         +
\ldots             +
a_k(z-z_0)^k       +
\ldots
\ ,
\end{equation}
$z \in V - \{ z_0 \}$. Viceversa, ogni funzione $f$ che ammetta uno sviluppo in serie del tipo precedente ha in $\zeta$ un polo di ordine $n$.

Sia ora $\zeta \in \bC$ una singolarit\'a isolata per $f$ ed $U \ni \zeta$ un {\em dominio regolare} tale che $\ovl U - \{ \zeta \}$ sia contenuto nel dominio di olomorfia di $f$. Dal teorema di Cauchy segue che 
\begin{equation}
\label{eq_RES01}
\res_\zeta f(z) \ dz 
\ := \
\int_{\partial U} f(z) \ dz
\end{equation}
non dipende dalla scelta di $U$. Chiameremo la quantit\'a definita dall'equazione precedente il {\em residuo in $\zeta$ della forma differenziale $f \ dz$}. Osserviamo che (\ref{eq_RES01}) fornisce un interpretazione del residuo come un'ostruzione per la forma $f \ dz$ ad essere chiusa. Chiaramente,
\[
f \in \mO(V) 
\ \Rightarrow \
\res_\zeta f(z) \ dz  =  0
\ , \ 
\forall \zeta \in V
\ .
\]
Sia ora $f$ del tipo 
\[
f(z) = \sum_{n=-\infty}^{+\infty} a_n (z-\zeta)^n
\ \ , \ \
z \in V
\ .
\]
dove i coefficienti $a_{-k}$ sono nulli per $k \in \bN$ abbastanza grande. Allora
\begin{equation}
\label{eq_RES02}
\res_\zeta f(z) \ dz = a_{-1} \ .
\end{equation}
Infatti, sfruttando il passaggio al limite sotto il segno di integrale, e calcolando il residuo sul bordo di un opportuno disco $\Delta$ di centro $\zeta$, otteniamo
\[
\res_\zeta f(z) \ dz
=
\frac{1}{2\pi i}
\sum_{n=-\infty}^{+\infty} a_n 
\int_{\partial \Delta} (z-\zeta)^n dz
\ ,
\]
ed i termimi con $n \geq 0$ svaniscono per olomorfia della funzione $(z-\zeta)^n$, mentre i termini con $n \leq -2$ svaniscono grazie all'uguaglianza (\ref{eq_IC06}).

\begin{defn}
Una funzione di variabile complessa $f$ di dice \textbf{meromorfa} in un aperto $U \subseteq \bC$ se esiste un insieme di punti isolati $A \subset U$ tale che 
\begin{enumerate}
\item $f \in \mO(U-A)$;
\item $f$ possiede in $A$ singolarit\'a al pi\'u polari.
\end{enumerate}
Denotiamo con $\mM(U)$ l'algebra delle funzioni meromorfe in $U$, e con $\mM(\ovl U)$ l'algebra delle funzioni meromorfe in ogni intorno di $\ovl U$ olomorfe in un intorno di $\partial U$. 
\end{defn}

\begin{ex}{\it
Ogni funzione razionale \e meromorfa.
}
\end{ex}

\begin{ex}{\it
La funzione $f(z) := (e^z - 1)^{-1}$, $z \in \bC$, \e meromorfa ed ha poli $\zeta_k := 2 \pi i k$, $k \in \bZ$. Ogni polo ha ordine $1$, avendosi
\[
\lim_{z \to \zeta_k} \frac{z}{e^z - 1} 
\ \stackrel{de \ l'Hopital}{=} \
\frac{1}{e^{\zeta_k}}
\ = \
1
\ .
\]
}
\end{ex}

\begin{thm}[Teorema dei residui]
Sia $U \subset \bC$ un dominio regolare ed $f \in \mM(\ovl U)$. Allora
\begin{equation}
\label{eq_RES03}
\frac{1}{2 \pi i} \int_{\partial U} f(z) \ dz
\ = \
\sum_{\zeta \in U} \res_\zeta f(z) \ dz
\ .
\end{equation}
\end{thm}

\begin{proof}[Dimostrazione]
Poich\'e $U$ \e precompatto l'insieme $A$ dei poli di $f$ \e finito. Consideriamo per ogni $\zeta \in A$ un disco $\Delta_\zeta$ di centro $\zeta$, in maniera tale che $\ovl \Delta_\zeta \cap \ovl \Delta_\xi = \emptyset$ per $\zeta \neq \xi \in A$. Per costruzione $f$ \e olomorfa in ogni $\ovl \Delta_\zeta - \{ \zeta \}$, per cui
\[
\frac{1}{2 \pi i} 
\sum_{\zeta \in A} 
\int_{\partial \Delta_\zeta}  f(z) \ dz
\ = \
\sum_{\zeta \in A} \res_\zeta f(z) \ dz
\ .
\]
D'altro canto $f$ \e olomorfa in $V := U - \dot{\cup}_\zeta \ovl \Delta_\zeta$, per cui usando il teorema di Cauchy troviamo
\[
0 
\ = \
\frac{1}{2 \pi i} \int_{\partial V} f(z) \ dz
\ = \
\frac{1}{2 \pi i} \int_{\partial U} f(z) \ dz
-
\sum_{\zeta \in A} \res_\zeta f(z) \ dz
\ .
\]
\end{proof}

Il teorema dei residui \e di estrema utilit\'a nel calcolo di integrali complessi (ed anche reali). Infatti, il calcolo del residuo di una funzione meromorfa (ovvero, il termine destro di (\ref{eq_RES03})) \e un'operazione piuttosto semplice: posto $n$ uguale all'ordine del polo di $f$ in $\zeta$, grazie a (\ref{eq_LAURENT}) e (\ref{eq_RES02}), otteniamo immediatamente
\begin{equation}
\label{eq_RES05}
\res_\zeta f(z) \ dz 
\ = \ 
a_{-1} 
\ = \
\frac{1}{(n-1)!} \ \lim_{ z \to \zeta} 
\frac{d^{n-1}}{dz^{n-1}} \left\{ (z-\zeta)^n f(z)  \right\}
\ .
\end{equation}

\subsection{Esercizi.}

\noindent \textbf{Esercizio \ref{sec_compl}.1 (\cite{Spi}).} {\it Calcolare l'integrale
\[
\int_0^\infty \frac{dx}{1+x^6} \ .
\]
}

\noindent {\it Soluzione.} Si consideri la funzione meromorfa $f(z) := (1+z^6)^{-1}$, definita su $\bC$ privato dell'insieme $P := \{ e^{ (2k+1) \pi i / 6 } \ , \ k = 0 , \ldots , 5 \}$. Chiaramente, gli elementi di $P$ sono tutti poli semplici per $f$. Preso $R > 0$, consideriamo il dominio regolare $U$ delimitato dal segmento $\{ | {\mathrm{Re}}(z) | \leq R \}$ e dalla semicirconferenza $S = \{ z = R e^{i \theta} , \theta \in (0,\pi) \}$. Per $R$ abbastanza grande, saranno contenuti in $U$ tutti e soli i poli di $f$ contenuti nel semipiano $\{ | {\mathrm{Im}}(z) | \geq 0 \}$, che sono 
\[
\zeta_1 := e^{i \pi / 6 } \ \ , \ \ 
\zeta_2 := e^{i \pi / 2 } \ \ , \ \
\zeta_3 := e^{i \pi 5/6 } \ .
\]
Applicando (\ref{eq_RES05}) abbiamo $\res_{\zeta_k} f(z) \ dz = \lim_{z \to \zeta_k} (z-\zeta_k)f(z)$, $k = 1,2,3$, ovvero
\[
\res_{\zeta_1} f(z) \ dz 
\ = \
\lim_{z \to \zeta_1} \left( \frac{ z - e^{i \pi / 6 } }{ z^6 + 1 } \right)
\ \stackrel{de \ l'Hopital}{=} \
\lim_{z \to \zeta_1} \frac{1}{6z^5} 
\ = \
\frac{ e^{- i \pi 5/6 } }{6}
\ ,
\]
ed analogamente
\[
\res_{\zeta_2} f(z) \ dz 
\ = \
\frac{ e^{- i \pi 5/2 } }{6}
\ \ , \ \
\res_{\zeta_3} f(z) \ dz 
\ = \
\frac{ e^{- i \pi 25/6 } }{6}
\ .
\]
Per cui,
\[
\int_\gamma \frac{dz}{1+z^6} 
\ = \
\frac{2 \pi i }{6} 
\left(
e^{- i \pi 5/6 } + e^{- i \pi 5/2 } + e^{- i \pi 25/6 }
\right)
\ = \
\frac{2}{3} \ \pi 
\ ,
\]
e quindi
\[
\frac{2}{3} \ \pi
\ = \
\int_{-R}^R \frac{dx}{1+x^6} + \int_S \frac{dz}{1+z^6} \ . 
\]
Valutiamo ora, con $R$ abbastanza grande, l'integrale su $S$:
\[
\left| \int_S \frac{dz}{1+z^6} \right|
\ = \
\left| \int_0^1 \frac{ R 2 \pi i e^{2 \pi i t} }{ 1+R^6e^{2 \pi i t} } \ dt \right|
\ \leq \
\frac{ 2 \pi R }{ R^6 - 1 }
\ .
\]
Passando al limite per $R \to \infty$ concludiamo
\[
\int_\bR \frac{dx}{1+x^6} \ = \ \frac{2}{3} \ \pi
\ \ \Rightarrow \ \ 
\int_0^\infty \frac{dx}{1+x^6} \ = \ \frac{\pi}{3} 
\ .
\]

\

\noindent \textbf{Esercizio \ref{sec_compl}.2.} {\it Calcolare l'integrale
\[
I := \int_0^{2 \pi} \frac{dx}{ 3 - 2\cos x + \sin x } \ .
\]
}

\

\noindent {\it Soluzione.} Effettuando la sostituzione
\[
z = e^{ix}
\ \ , \ \
\sin x = \frac{z-z^{-1}}{2}
\ \ , \ \
\cos x = \frac{z+z^{-1}}{2}
\ \ , \ \
dz = iz \ dx
\ ,
\]
troviamo
\[
I 
\ = \
\int_S \frac{2 \ dz}{ (1-2i)z^2 + 6 i z - 1 - 2i } 
\ ,
\]
dove $S := \{ z : |z| = 1 \}$. Proseguiamo quindi applicando il teorema dei residui.

\

\noindent \textbf{Esercizio \ref{sec_compl}.3.} {\it Si calcolino gli integrali
\[
\int_{|z|=1} \frac{e^z}{z} \ dz
\ \ \ , \ \ \
\int_{|z|=2} \frac{dz}{z^2+1}
\ \ .
\]
}

\

\noindent \textbf{Esercizio \ref{sec_compl}.4 (\cite[Ex.7.51]{Roy}).} {\it Sia $U := \{ z : |z|<1  \}$ ed $\mF \subseteq \mO(U)$ una famiglia equilimitata nella norma dell'estremo superiore. Usando la formula di Cauchy, si dimostri che $\mF$ \e anche equicontinua.

\

\noindent (Suggerimento: presi $\zeta , \zeta' \in U$ ed un dominio regolare $\ovl V \subset U$ tale che $\zeta , \zeta' \in V$, si verifichi che per ogni $f \in \mF$ si ha la stima
\[
| f(\zeta) - f(\zeta') | 
\ \leq \ 
\frac{\| f \|_\infty}{2 \pi}  \int_{\partial V} \left| \frac{1}{ z - \zeta } - \frac{1}{ z - \zeta' } \right| \ dz
\ ;
\]
quindi si usi l'equilimitatezza).
}

\

\noindent \textbf{Esercizio \ref{sec_compl}.5. (L'indicatore logaritmico).} {\it Sia $f \in \mM(U)$ priva di zeri in $\partial U$. Si mostri che
\begin{equation}
\label{eq_RES04}
\frac{1}{2 \pi i} \int_{\partial U} \frac{f'(z)}{f(z)} \ dz 
\ = \
\sharp \{ {\mathrm{zeri \ di \ }} f \} - \sharp \{ {\mathrm{poli \ di \ }} f \}
\ ,
\end{equation}
dove zeri e poli di $f$ sono contati con la relativa molteplicit\'a.
}

\

\noindent {\it Soluzione.} Preso $\xi \in U$ consideriamo un disco $\Delta$ con centro $\xi$ tale che $\xi$ \e l'unico zero o polo di $f$ in $\ovl \Delta$. Per ipotesi esistono $n \in \bZ$ ed una funzione $h$ olomorfa e non nulla in $\Delta$ tale che $f$ \e della forma 
\[
f(z) = (z-\xi)^nh(z)
\ \ , \ \
z \in \Delta
\ .
\]
Se $n > 0$ allora $f$ ha uno zero in $\xi$, mentre per $n <0$ abbiamo un polo; in entrambi i casi $n$ \`e, a meno del segno, la relativa molteplicit\'a. Dunque
\[
\frac{f'(z)}{f(z)}
\ = \
n \frac{1}{z-\xi} + \frac{h'(z)}{h(z)}
\ ,
\]
dove $h'/h$ \e olomorfa in $U$. Quindi 
\[
\res_\xi \frac{f'(z)}{f(z)} \ dz
\ = \
\frac{1}{2 \pi i} \int_{\partial \Delta} \frac{f'(z)}{f(z)} \ dz
\ = \
\frac{n}{2 \pi i} \int_{\partial \Delta} \frac{1}{z-\xi} \ dz
\ = \
n
\ .
\]
Ripetendo il ragionamento per ogni $\xi \in U$ troviamo la formula desiderata.

\

\noindent \textbf{Esercizio \ref{sec_compl}.6.} {\it Si calcoli
\begin{equation}
\label{eq_exRES}
\int_\bR \frac{ e^{ax} + e^{-ax} }{ e^x + e^{-x} } \ dx 
\ \ , \ \
|a| < 1
\ .
\end{equation}

\

\noindent (Suggerimenti: conviene calcolare separatamente gli integrali complessi
\[
I := \int_\gamma \frac{ e^{az} }{ e^z + e^{-z} } \ dz 
\ \ , \ \
J := \int_\gamma \frac{ e^{-az} }{ e^z + e^{-z} } \ dz
\ ,
\]
dove la curva chiusa $\gamma$ v\'a scelta in modo astuto. Procediamo al calcolo di $I$. In primo luogo, osserviamo che i poli della funzione integranda di $I$ si trovano per
\[
e^{2z} = -1 
\ \Rightarrow \
z = (k+1/2)\pi i 
\ \ , \ \
k \in \bZ
\ ,
\]
e sono tutti semplici. Scegliamo quindi $r > 0$ e $\gamma$ come la curva il cui grafico \e il rettangolo con vertici
\[
-r \ \ , \ \ r \ \ , \ \ r+\pi i \ \ , \ \ -r + \pi i
\ .
\]
In tal modo l'unico polo interno al grafico di $\gamma$ \e $\xi := 1/2 \pi i$ ed il relativo residuo \e
\[
\res_\xi \frac{ e^{az} }{ e^z + e^{-z} } \ dz =
\lim_{z \to \xi} ( z - \xi ) \frac{ e^{az} }{ e^z + e^{-z} } =
- i/2 e^{a i \pi / 2}
\ .
\]
Questo residuo \e uguale ad $I$, e del resto
\[
I  =
\int_{-r}^r \frac{e^{ax}}{e^x+e^{-x}} \ dx +
\int_0^\pi \frac{e^{a(r+iy)}}{e^{r+iy}+e^{-r-iy}} \ dy +
\int_{-r}^r \frac{e^{a(x+\pi i)}}{e^{x+\pi i}+e^{-x+\pi i}} \ dx +
\int_\pi^0 \frac{e^{a(-r+iy)}}{e^{-r+iy}+e^{-r-iy}} \ dy
\ .
\]
Effettuando una stima al limite $r \to \infty$ otteniamo che gli integrali tra $0,\pi$ tendono a zero; i restanti due integrali si scrivono
\[
\int_{-r}^r \frac{e^{ax}}{e^x+e^{-x}} \ dx +
\int_{-r}^r \frac{e^{a(x+\pi i)}}{e^{x+\pi i}+e^{-x+\pi i}} \ dx =
K_r + e^{a \pi i} K_r
\ ,
\]
dove
\[
K_r := \int_r^r \frac{e^{ax}}{e^x+e^{-x}} \ dx 
\]
\e, passando al $\lim_{r \to \infty}$, l'integrale reale che vogliamo calcolare. Dunque
\[
I = - i/2 e^{a i \pi / 2} =  K_r + e^{a \pi i} K_r
\]
e quindi
\[
K_r = \frac{1}{2} \frac{\pi}{\cos(a \pi / 2)} \ .
\]
Il calcolo di $J$ si effettua in maniera analoga.)
}

\

\noindent \textbf{Esercizio \ref{sec_compl}.7.} {\it Sia $\{ x \} := x-[x] \in (0,1)$ la parte frazionaria di $x \in \bR$. 
Si dimostri che la funzione
\[
g(z) := \int_1^\infty \frac{ \{ x \} }{ x^{z+1} } \ dx
\ \ , \ \
{\mathrm{Re}}(z) > 0
\]
(integrale reale) \e olomorfa.

\

\noindent (Suggerimento: La funzione $x,z \mapsto \{ x \} / x^{z+1}$ \e di classe $L^1$ su $[1,\infty)$ come funzione di $x$ ed olomorfa come funzione di $z$ nel semipiano ${\mathrm{Re}}(z) > 0$, infatti $x^{-z-1} = e^{-(z+1) \log x}$. Applicando il Teorema \ref{thm_der_int} possiamo derivare rispetto a $z$ sotto il segno d'integrale, concludendo che $g$ \e olomorfa).

}

\

\noindent \textbf{Esercizio \ref{sec_compl}.8. (La zeta di Riemann).} {\it Si ponga 
$U_{ > \lambda} := \{ z \in \bC : {\mathrm{Re}}(z) > \lambda \}$, $\forall \lambda \in \bR$,
e si consideri un aperto $U$ tale che $\ovl U \subset U_{>1}$. Si prenda quindi la successione 
\[
f_n(z) \ := \ \sum_{k=1}^n \frac{1}{k^z}
\ \ , \ \
z \in U
\ .
\]
\textbf{(1)} Si mostri che ogni $f_n$ \e olomorfa in $\ovl U$;
\textbf{(2)} Si verifichi che $\{ f_n \}$ converge uniformemente ad una funzione $f$, la quale \e quindi olomorfa in $\ovl U$;
\textbf{(3)} Si mostri che $f$ \e prolungabile analiticamente ad una funzione $F_{>1} \in \mO(U_{>1})$.
\textbf{(4)} Usando l'esercizio precedente e le identit\'a
          \[
          \frac{1}{k^z} \ = \ z \int_k^\infty \frac{dx}{x^{z+1}}
          \ \ , \ \
          z \in U_{>1}
          \ , \
          k = 1,2,\ldots
          \]
          (integrali reali) si verifichi che $F_{>1}$ \e prolungabile ad una funzione meromorfa $F_{>0} \in \mM(U_{>0})$ 
          con un polo di ordine $1$ in $z = 1$
          {\footnote{In realt\'a esiste un ulteriore prolungamento $F \in \mM(\bC)$ avente come unico polo $z=1$.
                     E' questa la funzione zeta di Riemann propriamente detta (\cite{Bom,Edw}),
                     famosa per la seguente:
                     
                     \
                     
                     \noindent \textbf{Congettura (Riemann).} Tutti e soli gli zeri di $F$ contenuti nella {\em striscia critica}
                               $\{ 0 \leq {\mathrm{Re}}(z) \leq 1  :  z \neq 1 \}$
                               si trovano sulla retta $\{ {\mathrm{Re}}(z) = 1/2 \}$.
                     
                     \
                     
           \noindent Ulteriore aspetto interessante inerente la zeta di Riemann \e la {\em formula di Eulero}
                     \[
                     F_{>1}(z) \ = \ \prod_{p \ {\mathrm{primo}} } \frac{1}{1-p^{-z}} \ \ , \ \ z \in U_{>1} \ .
                     \]
          }}.
          
\

\noindent (Suggerimenti: 
(2) Posto $\alpha := \inf_{z \in \ovl U} {\mathrm{Re}}(z)$ ed osservato che $\alpha > 1$, per ogni $n > m$ si ha la stima
    \[
    \sup_U | f_n(z) - f_m(z) | \ \leq \ \sum_{k=m}^n \frac{1}{k^\alpha} \ .
    \]
(3) Si usi il principio di continuazione analitica.
(4) Si ha
    \[
    \sum_k \frac{1}{k^z} \ = \
    \sum_k z \int_k^\infty \frac{dx}{x^{z+1}} \ = \
    z \int_1^\infty \frac{ x- \{ x \} }{x^{z+1}} \ dx \ ,
    \]
    dunque
    \[
    F_{>1}(z) = 1 + \frac{1}{z-1} - z \int_1^\infty \frac{ \{ x \} }{x^{z+1}} \ dx \ ,
    \]
    e quest'ultima espressione \e ben definita come funzione meromorfa su $U_{>0}$, in particolare l'integrale \e olomorfo grazie all'esercizio precedente ed il polo \e di ordine $1$).

}

\newpage
\section{Cenni sugli Spazi di Sobolev.}
\label{sec_sobolev}

Gli spazi di Sobolev sono stati introdotti come strumento per la dimostrazione di teoremi di esistenza ed unicit\'a della soluzione di problemi differenziali con condizioni al bordo. L'idea di fondo \e quella di combinare il concetto di derivata debole con la teoria degli spazi $L^p$.

In un certo senso la teoria degli spazi di Sobolev sovverte il punto di vista del classico calcolo variazionale (\S \ref{sec_calc_var}). Se da un lato, classicamente, la minimizzazione di un funzionale veniva effettuata risolvendo un'equazione differenziale (l'equazione di Eulero-Lagrange), ora l'esistenza e l'unicit\'a della soluzione di un problema alle derivate parziali vengono dimostrate associando ad esso un problema variazionale. Il vantaggio di questo approccio risiede nel fatto che per portare a buon fine il procedimento di minimizzazione abbiamo a disposizione gli strumenti dell'analisi funzionale (in particolare, i Teoremi di Stampacchia-Lax-Milgram). 

\

In questa sezione considereremo intervalli aperti $(a,b)$ non necessariamente limitati, ovvero sono ammesse le possibilit\'a $a,b = \pm \infty$.

\subsection{Propriet\'a di base.}
Presi $a,b \in \bR$, consideriamo il problema differenziale
\begin{equation}
\label{eq_Sob1}
\left\{
\begin{array}{ll}
- u'' + u = f
\\
u(a) = u(b) = 0
\end{array}
\right.
\end{equation}
Ovviamente il problema (\ref{eq_Sob1}) non \e particolarmente interessante, visto che siamo in grado di risolverlo con metodi classici. Tuttavia ce ne serviremo per introdurre alcuni importanti concetti.

Consideriamo una soluzione $u$ di $(\ref{eq_Sob1})$. Allora per ogni $\varphi \in C^1_0(a,b)$ risulta, integrando per parti, $- \int_a^b u'' \varphi =$ $\int_a^b u' \varphi'$, per cui
\begin{equation}
\label{eq_Sob2}
\int_a^b \left( u' \varphi' + u \varphi \right) = \int_a^b f \varphi \ .
\end{equation}
Osserviamo che l'equazione precedente, soddisfatta da ogni soluzione di (\ref{eq_Sob1}), ha senso pi\'u in generamente per funzioni $u \in C^1(a,b)$. In questo modo, possiamo pensare di sostituire il problema iniziale (\ref{eq_Sob1}) con la ricerca di una funzione $u$ che soddisfi l'equazione (\ref{eq_Sob2}), con il vantaggio che lo spazio delle possibili soluzioni \e a priori molto pi\'u grande (infatti, cerchiamo funzioni $C^1$ piuttosto che $C^2$).

\begin{lem}
\label{lem_Sob1}
Siano $a,b \in \wa \bR$, $p \in [1,+\infty]$ ed $u \in L^p(a,b)$. Se esiste $v \in L^p(a,b)$ tale che
\begin{equation}
\label{eq_Sob3}
- \int_a^b v \varphi = \int_a^b u \varphi'
\ \ , \ \
\forall \varphi \in C^1_0(a,b)
\ ,
\end{equation}
allora $v$ \e unico.
\end{lem}

\begin{proof}[Dimostrazione]
Supposto che esista $w \in L^p(a,b)$ che soddisfi (\ref{eq_Sob3}) troviamo $\int_a^b z \varphi = 0$, $\varphi \in C^1_0(a,b)$, dove $z := v-w \in$ $L^p(a,b)$. Per densit\'a di $C^1_0(a,b)$ in $L^p(a,b)$ concludiamo che il funzionale $F_z \in L^{q,*}(a,b)$, $q := \ovl{p}$, \e nullo, per cui $z = 0$ q.o. per dualit\'a di Riesz.
\end{proof}

La funzione $v$ del Lemma precedente si dice {\em la derivata debole di} $u$ e nel seguito sar\'a denotata con $u'$. Si osservi che $u'$ \e la derivata debole di $u$ anche nel senso delle distribuzioni (Esempio \ref{ex_DB}), con l'ulteriore propriet\'a di appartenere ad $L^p(a,b)$. L'insieme delle funzioni $L^p$ debolmente derivabili \e chiaramente uno spazio vettoriale; chiameremo tale spazio {\em spazio di Sobolev}, e lo denoteremo con
\[
W^{1,p} (a,b)
\ := \
\{
u \in L^p(a,b) \ | \ \exists u' \in L^p : 
- \int_a^b u' \varphi = \int_a^b u \varphi'
\ , \
\varphi \in C^1_0(a,b)
\}
\]
(osservare che il suffisso "1" suggerisce che stiamo considerando un analogo di $C^1$). Osserviamo che nella definizione precedente potremmo considerare equivalentemente $\varphi \in C_0^\infty(a,b)$ o $\varphi \in C_c^\infty(a,b)$.

\begin{rem}{\it
Sia $C_c(a,b)$ lo spazio delle funzioni continue a supporto compatto in $(a,b)$; allora \e chiaro che $C_c(a,b) \subset L^p(a,b)$, $\forall p \in [1,+\infty]$. Se 
$u \in C^1_c(a,b) := C^1(a,b) \cap C_c(a,b)$ 
allora la derivata di $u$ (nel senso classico) appartiene a $C_c(a,b)$, e chiaramente essa coincide con la derivata debole. Per cui abbiamo applicazioni canoniche
\[
C^1_c(a,b) \to W^{1,p}(a,b) \ \ , \ \ \forall p \in [1,+\infty] \ .
\]
}
\end{rem}

Tornando al nostro problema iniziale, osserviamo che affinch\'e sia ben definito l'integrale del termine sinistro di (\ref{eq_Sob2}) \e sufficiente che sia $u \in W^{1,p}(a,b)$ per un qualche $p \in [1,+\infty]$. Inoltre, siamo passati da un problema differenziale ad uno integrale.

\begin{ex}{\it 
\label{ex_Sob1}
Posto $a = -1$, $b = 1$, allora $u(x) := 1/2 ( |x|+x )$, $x \in [-1,1]$, appartiene a $W^{1,p}(a,b)$ per ogni $p \in [1,+\infty]$. La derivata debole di $u$ \e la \textbf{funzione di Heaviside}
\[
H(x) :=
\left\{
\begin{array}{ll}
0 \ \ , \ \ x \in (-1,0)
\\
1 \ \ , \ \ x \in [0,1)
\end{array}
\right.
\]
D'altra parte, $H$ non appartiene a $W^{1,p}$ per nessun $p \in [1,+\infty]$ (lasciamo la verifica di questo fatto come semplice esercizio).
}\end{ex}

Consideriamo ora la seguente norma su $W^{1,p}$:
\begin{equation}
\label{eq_nW1p}
\| u \|_{W,p}
\ := \
\| u \|_p + \| u' \|_p
\ \ , \ \
u \in W^{1,p}
\ .
\end{equation}
Osserviamo che nel caso $p = 2$ allora possiamo vedere $\| u \|_{W,2}$ come la norma associata al prodotto scalare 
\[
(u,u)_{H,1} 
\ := \
\int_a^b uv + \int_a^b u'v'
\ .
\]
Introduciamo la notazione $H^1(a,b) := W^{1,2}(a,b)$, cosicch\'e $H^1(a,b)$ \e uno spazio pre-Hilbertiano.

\begin{prop}[Completezza degli spazi di Sobolev]
\label{prop_Sob1}
Per ogni $p \in [1,+\infty]$, lo spazio di Sobolev $W^{1,p}(a,b)$ \e completo rispetto alla norma (\ref{eq_nW1p}). Inoltre, $W^{1,p}(a,b)$ \e riflessivo per $p \in (1,+\infty)$, e separabile per $p \in [1,+\infty)$.
\end{prop}

\begin{proof}[Sketch della dimostrazione]
La completezza segue osservando che se $\{ u_n \} \subset W^{1,p}(a,b)$ \e di Cauchy allora esistono $u := \lim_n u_n \in L^p$ e $u_1 := \lim_n u'_n \in L^p$; un passaggio al limite per l'uguaglianza
\[
\int_a^b u_n \varphi' = - \int_a^b u'_n \varphi
\ \ , \ \
\varphi \in C^1_c([a,b])
\ ,
\]
mostra che effettivamente $u_1$ \e la derivata debole di $u$, dunque $W^{1,p}(a,b)$ \e completo. La riflessivit\'a segue considerando l'isometria canonica
\[
W^{1,p}(a,b) \to L^p(a,b) \times L^p(a,b)
\ \ , \ \
u \mapsto (u,u')
\ ,
\]
il che permette di esibire $W^{1,p}(a,b)$ come un sottospazio chiuso dello spazio riflessivo $L^p(a,b) \times L^p(a,b)$. Poich\'e in generale un sottospazio chiuso di uno spazio riflessivo \e riflessivo concludiamo che $W^{1,p}(a,b)$ \e riflessivo. Allo stesso modo, la separabilit\'a di $W^{1,p}(a,b)$ segue dalla separabilit\'a di $L^p(a,b) \times L^p(a,b)$.
\end{proof}

\begin{rem}
{\it
Consideriamo l'applicazione canonica
\[
D : W^{1,p}(a,b) \to  L^p(a,b)
\ \ , \ \
u \mapsto u'
\ .
\]
Allora $D$ \e un operatore lineare tale che $\| Du \|_p = \| u' \|_p \leq \| u \|_{W,p}$. Per cui gli spazi di Sobolev, analogamente a quanto accade nella teoria delle distribuzioni, sono atti a rendere la derivata (debole) un'applicazione continua.
}
\end{rem}

\begin{thm}[Esistenza del rappresentante continuo]
\label{thm_Sob1}
Sia $p \in [1,+\infty]$ ed $a,b \in \bR$. Per ogni $u \in W^{1,p}(a,b)$, esiste ed \e unica $\tilde u \in C([a,b])$ tale che $u = \tilde u$ q.o. in $(a,b)$, in maniera tale che
\begin{equation}
\label{eq_Sob4}
\tilde u (x) - \tilde u (y) = \int^x_y u'(t) \ dt 
\ \ , \ \
x,y \in (a,b)
\ .
\end{equation}
\end{thm}

\begin{proof}[Sketch della dimostrazione]
Fissato $x_0 \in (a,b)$, l'idea \e quella di considerare la funzione continua 
\[
u_0(x) := \int_{x_0}^x u'(t) \ dt
\ \ , \ \
x \in (a,b)
\ .
\]
Osserviamo che $u_0$ \e ben definita in quanto $u'|_{(x_0,x)} \in L^p([x_0,x]) \subset L^1([x_0,x])$ (vedi Cor.\ref{cor_holder}). Cosicch\'e, per ogni $\varphi \in C_0^1(a,b)$ troviamo
\[
\begin{array}{ll}
\int_a^b u_0 \varphi' & =
\int_a^b \int_{x_0}^x u'(t) \varphi'(x) \ dtdx = 
\\ & = 
- \int_a^{x_0} dx  \int_x^{x_0}  u'(t) \varphi'(x) dt
+ \int_{x_0}^b dx  \int_{x_0}^x  u'(t) \varphi'(x) dt
\ \stackrel{Fubini}{=} \
- \int_a^b u' \varphi \ .
\end{array}
\]
L'uguaglianza precedente implica 
\[
\int_a^b (u_0 - u) \varphi' = 0
\ \ , \ \
\varphi \in C^1_0(a,b)
\ .
\]
Da quest'ultima uguaglianza si pu\'o dedurre (in modo non banale, vedi \cite[Lemma VIII.1,Cor.IV.24]{Bre}) che $u - u_0$ coincide q.o. con una costante $c \in \bR$. Poniamo allora $\tilde u := u_0 + c$.
\end{proof}

Presi $a,b \in \bR$ diciamo che $u \in W^{1,p}(a,b)$ {\em si annulla al bordo} se ci\'o accade al suo rappresentante continuo, ed in tal caso scriveremo $u \in W^{1,p}_0(a,b)$. Nel seguito, identificheremo $u$ col suo rappresentante continuo.

\begin{rem}{\it 
\label{rem_Sob1}
Siano $a,b \in \bR$ ed $u_1 \in L^p(a,b)$ con $p \in [1,+\infty]$. Poich\'e $(a,b)$ ha misura finita abbiamo che $u_1$ \e integrabile (Cor.\ref{cor_holder}) ed $u :=$ $\int_a^x u_1(t) dt$, $x \in [a,b]$, \e assolutamente continua e quindi limitata. Ci\'o implica $u \in L^p(a,b)$ e quindi $u \in W^{1,p}(a,b)$ con derivata debole $u_1$. Dunque, su intervalli limitati le funzioni in $W^{1,p}(a,b)$ sono tutte e sole le primitive di funzioni $L^p$. In particolare,
\[
W^{1,1}(a,b) = AC([a,b]) \ \ , \ \  a,b \in \bR \ .
\]
} \end{rem}


\begin{prop}[Diseguaglianza di Poincar\'e]
\label{prop_SobPoi}
Siano $a,b \in \bR$, $a < b$, $p \in [1,+\infty]$. Allora esiste una costante $c = c(a,b)$ tale che
\begin{equation}
\label{eq_Poincare}
\| u \|_{W,p} 
\ \leq \ 
c \| u' \|_p 
\ \ , \ \ 
u \in W^{1,p}_0(a,b)
\ .
\end{equation}
Cosicch\'e la seminorma $n(u) :=$ $\| u' \|_p$ , $u \in W^{1,p}_0(a,b)$, \e equivalente a $\| \cdot \|_{W,p}$ su $W^{1,p}_0(a,b)$.
\end{prop}

\begin{proof}[Dimostrazione]
Poich\'e $u(a) = 0$, ricordando Oss.\ref{rem_Sob1} troviamo
\begin{equation}
\label{eq_Sob6}
|u(x)| = 
|u(x) - u(a)| = 
\left| \int_a^x u'(t) \ dt \right| \leq 
\| u' \|_1
\ .
\end{equation}
Per cui, $\| u \|_\infty \leq \| u' \|_1$, e quindi
\[
(b-a)^{-1/p} \| u \|_p 
\leq
\| u \|_\infty
\leq
\| u' \|_1
\ \stackrel{Cor.\ref{cor_holder}}{\leq} \
\| 1 \|_q \| u' \|_p
\ ,
\]
con $q := \ovl{p}$. In termini pi\'u espliciti, abbiamo dimostrato la disuguaglianza
\begin{equation}
\label{eq_Poincare2}
\int_a^b |u|^p
\ \leq \
(b-a)^p \int_a^b |u'|^p
\ \ , \ \
u \in W^{1,p}_0(a,b)
\ .
\end{equation}
\end{proof}

\subsection{Immersioni compatte di $W^{1,p}$.}

\

\noindent \textbf{Approssimazione in $W^{1,p}$.} Il teorema seguente costituisce un importante risultato di approssimazione; la dimostrazione si basa sui risultati di approssimazione per gli spazi $L^p$ (vedi \S \ref{sec_appLp}):
\begin{thm}\textbf{\cite[Teo.VIII.6]{Bre}}
\label{thm_c1W1p}
Sia $p \in [1,+\infty)$ e $u \in W^{1,p}(a,b)$. Allora esiste una successione $\{ u_n \} \subset C_c^\infty(\bR)$ tale che $\| u_n |_{(a,b)} - u \|_{W,p} \stackrel{n}{\to} 0$.
\end{thm}

\

\noindent \textbf{Diseguaglianza di Sobolev e conseguenze.} Abbiamo visto in precedenza che se $(a,b)$ \e limitato allora vale la stima (\ref{eq_Sob6}), la quale implica  
\[
\| u \|_\infty 
\ \leq \
\| 1 \|_q \| u' \|_p
\ \leq \
\| 1 \|_q \| u' \|_p + \| 1 \|_q \| u \|_p
\ = \
(b-a)^{1/q} \| u \|_{W,p}
\ \ , \ \
u \in W^{1,p}_0(a,b)
\ .
\]
Una versione della diseguaglianza precedente vale anche per intervalli non limitati e generiche $u \in W^{1,p}(a,b)$. Ci\'o ha importanti conseguenze sulla struttura degli spazi di Sobolev:

\begin{lem}\textbf{(La diseguaglianza di Sobolev, \cite[Teo.VIII.7]{Bre}).}
\label{lem_immcont}
Esiste una costante non nulla $c = c(b-a)$, $b-a \in (0,+\infty]$, tale che
\begin{equation}
\label{eq_infty_p}
\| u \|_\infty  
\ \leq \
c \| u \|_{W,p}
\ \ , \ \
u \in W^{1,p}(a,b) 
\ , \
p \in [1,+\infty]
\ .
\end{equation}
\end{lem}




I seguenti interessanti risultati si dimostrano con l'uso di stime derivanti da (\ref{eq_infty_p}) e Teo.\ref{thm_c1W1p}:
\begin{cor}[Derivata di un prodotto]
Sia $p \in [1,+\infty]$ ed $u,v \in W^{1,p}(a,b)$. Allora $uv \in W^{1,p}(a,b)$ e
\begin{equation}
\label{eq_der_W1p}
(uv)' = u'v + uv'
\ \ , \ \
\int_y^x u'v 
\ = \
u(x) v(x) - u(y) v(y) - \int_y^x uv'
\ \ , \ \
x,y \in [a,b]
\ ,
\end{equation}
cosicch\'e $W^{1,p}(a,b)$ \e un'algebra.
\end{cor}

\begin{cor}[Derivata di una composizione]
Sia $\varphi \in C^1(\bR)$ tale che $\varphi(0) = 0$ ed $u \in W^{1,p}(a,b)$. Allora $\varphi \circ u \in W^{1,p}(a,b)$ e $(\varphi \circ u)' = ( \varphi' \circ u ) u'$.
\end{cor}

\noindent \textbf{L'immersione continua in $L^\infty$.} Il Lemma \ref{lem_immcont} ci dice che \e ben definito e limitato l'operatore lineare canonico
\begin{equation}
\label{map_infty_p}
I_p^\infty : W^{1,p}(a,b) \to L^\infty(a,b)
\ \ , \ \
p \in [1,+\infty]
\ ,
\end{equation}
che assegna alla funzione $u \in W^{1,p}(a,b)$ la sua classe in $L^\infty(a,b)$ (osservare che $\| I_p^\infty \| \leq c$). Chiaramente $I_p^\infty$ \e anche iniettivo, infatti se $u \in W^{1,p}(a,b)$ \e tale che $\| u \|_\infty = 0$ allora $u=0$ q.o..

\

\noindent \textbf{Immersioni compatte.} Sia ora $(a,b)$ limitato. Applicando il teorema del rappresentante continuo troviamo che (\ref{map_infty_p}) prende valori in $C([a,b])$, per cui \e ben definito l'operatore limitato ed iniettivo
\begin{equation}
\label{map_c_p}
I_p^{cont} : W^{1,p}(a,b) \to C([a,b])
\ \ , \ \
a,b \in \bR
\ , \
p \in [1,+\infty]
\ .
\end{equation}
Del resto, applicando la diseguaglianza di Holder troviamo $\| f \|_q \leq (b-a)^{1/q} \| f \|_\infty$ per ogni $f \in L^\infty(a,b)$ e $q \in [1,+\infty)$, per cui abbiamo un operatore limitato ed iniettivo
\begin{equation}
\label{map_q_p}
I_p^q : W^{1,p}(a,b) \to L^q(a,b)
\ \ , \ \
a,b \in \bR
\ , \
p \in [1,+\infty]
\ , \
q \in [1,+\infty)
\ .
\end{equation}

\begin{thm}\textbf{(Immersione compatta, \cite[Teo.VIII.7]{Bre}).}
\label{thm_Sob2}
Per ogni $a,b \in \bR$, $a<b$, si ha quanto segue:
\begin{enumerate}
\item  L'operatore $I_p^{cont} : W^{1,p}(a,b) \to C([a,b])$, definito da (\ref{map_c_p}), \e compatto per ogni $p \in (1,+\infty]$;
\item  L'operatore $I_1^q : W^{1,1}(a,b) \to L^q(a,b)$, definito come in (\ref{map_q_p}), \e compatto per ogni $q \in [1,+\infty)$.
\end{enumerate}
\end{thm}

\begin{proof}[Sketch della dimostrazione]
Verifichiamo la compattezza di (\ref{map_c_p}). A tale scopo denotiamo con $\mF$ l'immagine attraverso $I_p^{cont}$ della palla unitaria di $W^{1,p}(a,b)$ ed osserviamo che
\[
| u(x) - u(y) | 
\leq
\int_y^x |u'|
\ \stackrel{Holder}{\leq} \
\| u' \|_p |x-y|^{1/q}
\leq
|x-y|^{1/q}
\ \ , \ \
u \in \mF
\ , \
q := \ovl{p}
\ .
\]
Dunque $\mF$ \e equilimitato ed equicontinuo, ed il Teorema di Ascoli-Arzel\'a (Teo.\ref{thm_AA1}) implica che $\mF$ \e precompatto.
Infine, per quanto riguarda la compattezza di $I_1^q$, l'idea \e quella di ripetere il ragionamento precedente applicando il Teorema di Riesz-Fr\'echet-Kolmogorov.
\end{proof}

\begin{rem}{\it
L'operatore $I_1^{cont}$ \e continuo ed iniettivo, ma non compatto, anche quando $(a,b)$ \e limitato (si veda \cite[Cap.VIII]{Bre}).
}
\end{rem}

\begin{cor}
Siano $a,b \in \bR$ ed $\{ f_n \} \subset W^{1,p}(a,b)$ \e una successione limitata in norma $\| \cdot \|_{W,p}$. 
\textbf{(1)} Se $p \in (1,+\infty]$ allora esiste una sottosuccessione $\{ f_{k_n} \}$ convergente in norma $\| \cdot \|_\infty$, e quindi convergente in norma $\| \cdot \|_q$ per ogni $q \in [1,+\infty]$;
\textbf{(2)} Se $p = 1$ allora esiste una sottosuccessione $\{ f_{k_n} \}$ convergente in norma $\| \cdot \|_q$ per ogni $q \in [1,+\infty)$.
\end{cor}

\begin{rem}{\it
Se $p \in (1,+\infty]$ e $\{ f_n \} \subset W^{1,p}(a,b)$ \e una successione debolmente convergente allora quanto visto in \S \ref{sec_topdeb} implica che $\{ f_n \}$ \e limitata, e quindi si applicano i risultati del corollario precedente.
}
\end{rem}

\subsection{Ordini e dimensioni generali.}

Introduciamo ora (in modo ricorsivo) gli spazi di Sobolev di ordine superiore al primo:
\[
\left\{
\begin{array}{ll}
W^{m,p}(a,b) := \{ u \in W^{m-1,p}(a,b) : u' \in W^{m-1,p}(a,b) \}
\ , \ 
m = 2,3,\ldots
\\
H^m(a,b) := W^{m,2}(a,b) 
\end{array}
\right.
\]
Per definizione, $u \in W^{m,p}(a,b)$ se e soltanto se esistono $du , \ldots , d^mu \in$ $L^p(a,b)$ tali che
\[
\int_a^b d^ku \varphi 
\ = \
(-1)^k \int_a^b u d^k \varphi
\ \ , \ \
k = 1,\ldots,m
\ , \
\varphi \in C_c^\infty(a,b)
\ .
\]
Sugli spazi $W^{m,p}(a,b)$ sono definite le norme
\[
\| u \|_{m,p}
:=
\| u \|_p
+
\sum_{k=1}^m \| d^ku \|_p
\ ,
\]
ed in particolare $H^m(a,b)$ possiede norma indotta da prodotto scalare. Si puo dimostrare che la norma $\| \cdot \|_{m,p}$ \e equivalente alla norma
\[
n_{m,p}(u) := \| u \|_p
+
\| d^mu \|_p
\]
(quando $(a,b)$ \e limitato ed $u \in W^{m,p}_0(a,b)$, ci\'o segue applicando ricorsivamente (\ref{eq_Poincare2})). Altre propriet\'a di base di $W^{1,p}(a,b)$ si estendono in modo analogo: ad esempio, esiste una applicazione continua ed iniettiva
\[
W^{m,p}(a,b) \to C^{m-1}([a,b])
\ ,
\]
in analogia al Punto 1 di Teo.\ref{thm_Sob2}.

Concludiamo la sezione introducendo gli spazi di Sobolev nel caso multidimensionale. Sia $n \in \bN$, $\Omega \subseteq \bR^n$ un aperto, e $p \in [1,+\infty]$. Definiamo
\[
W^{1,p}(\Omega)
:=
\{
u \in L^p(\Omega)
:
\exists \frac{\partial u}{\partial x_1} , 
        \ldots , 
        \frac{\partial u}{\partial x_n} 
        \in L^p(\Omega)
\ , \
\int_\Omega u \frac{\partial \varphi}{\partial x_i}
= 
- \int_\Omega \frac{\partial u}{\partial x_i} \varphi
\ , \
\varphi \in C_c^\infty(\Omega)
\}
\ ,
\]
che equipaggiamo della norma
\[
\| u \|_{W,p}
:=
\| u \|_p
+
\sum_{i=1}^n \| \frac{\partial u}{\partial x_i} \|_p
\ .
\]
In particolare, definiamo $H^1(\Omega) := W^{1,2}(\Omega)$, la cui norma \e indotta dal prodotto scalare
\begin{equation}
\label{def_H1n}
( u,v )_{H,1}
:=
\int_\Omega uv 
+ 
\sum_{i=1}^n 
\int_\Omega \frac{\partial u}{\partial x_i} \frac{\partial v}{\partial x_i}
\ .
\end{equation}
Usando esattamente la stessa tecnica del caso unidimensionale, otteniamo
\begin{prop}
Per ogni aperto $\Omega \subseteq \bR^n$, lo spazio $W^{1,p}(\Omega)$ \e riflessivo per $p \in (1,+\infty)$ e separabile per $p \in [1,+\infty)$. In particolare, $H^1(\Omega)$ \e uno spazio di Hilbert separabile. 
\end{prop}
Altri risultati, come l'approssimazione con funzioni $C_c^\infty(\Omega)$ (Teorema di Friedrichs, \cite[Teo.IX.2]{Bre}), la diseguaglianza di Poincar\'e (\cite[Cor.IX.19]{Bre}), e l'immersione continua di $W^{1,p}(\Omega)$ in $L^\infty(\Omega)$, rimangono veri anche in pi\'u dimensioni, seppure in quest'ultimo caso con alcune ipotesi aggiuntive sulla coppia $(p,n)$. In particolare, se $p>n$ ed $\Omega \subset \bR^n$ \e limitato e di classe $C^1$ allora si ha un'immersione compatta $W^{1,p}(\Omega) \to C(\ovl \Omega)$ (Teorema di Rellich-Kondrachov, \cite[Teo.IX.16]{Bre}){\footnote{Del resto, anche in Teo.\ref{thm_Sob2} si richiede $p>1$ per avere la compattezza.}}; al solito, utilizzeremo le notazioni $W_0^{1,p}(\Omega)$, $H_0^1(\Omega)$ per denotare i sottospazi delle funzioni nulle al bordo di $\Omega$.

\subsection{Applicazioni alle equazioni alle derivate parziali.}
\label{sec_SobEDP}

In questa sezione mostriamo come l'uso combinato degli spazi di Sobolev e del teorema di Lax-Milgram permette di dimostrare risultati di esistenza ed unicit\'a per problemi differenziali di tipo ellittico con condizioni al bordo.

Nel seguito, indicheremo con $\Omega \subset \bR^n$ un aperto {\em limitato}. Allo scopo di avere una notazione pi\'u agile, denotiamo con
\[
\nabla u \cdot \nabla v  \in L^2(\Omega)
\ \ , \ \
\nabla u := 
\left( \frac{\partial u}{\partial x_1} , \ldots , \frac{\partial u}{\partial x_1} \right)
\ ,
\]
la funzione ottenuta effettuando il prodotto scalare dei gradienti di $u,v \in H^1(\Omega)$.
\begin{thm}[Il principio di Dirichlet] Sia dato il problema 
\begin{equation}
\label{eq_D}
\left\{
\begin{array}{ll}
- \Delta u + u = f
\\
u |_{\partial \Omega} = 0
\end{array}
\right.
\end{equation}
dove $f \in L^2(\Omega)$. Allora esiste ed \e unica la \textbf{soluzione debole} $u \in H_0^1(\Omega)$ di (\ref{eq_D}), ovvero
\[
\int_\Omega ( \nabla u \cdot \nabla v + uv )
=
\int_\Omega fv
\ \ , \ \ 
\forall v \in H_0^1(\Omega)
\ .
\]
La funzione $u$ si ottiene come soluzione del problema variazionale
\[
u 
\ = \
\min_{v \in H_0^1(\Omega)}
\left\{
\frac{1}{2} \int_\Omega \left( | \nabla v |^2 + v^2 \right)
-
\int_\Omega fv
\right\}
\ .
\]
\end{thm}

\begin{proof}[Dimostrazione]
Sia applica il teorema di Lax-Milgram alla forma bilineare indotta dal prodotto scalare (\ref{def_H1n}), ed al funzionale
\[
\left \langle \varphi, v \right \rangle := \int_\Omega fv 
\ \ , \ \
v \in H_0^1(\Omega)
\ .
\]
\end{proof}

\begin{rem}
\textbf{(Riguardo il concetto di soluzione debole del problema di Dirichlet).}
{\it Supponiamo che $u_c \in C^2_0(\Omega)$ sia una soluzione classica di (\ref{eq_D}); allora, una semplice integrazione per parti implica che 
\[
\int_\Omega ( \nabla u_c \cdot \nabla v + u_cv )
=
\int_\Omega fv
\ \ , \ \ 
\forall v \in C_c^\infty(\Omega)
\ .
\]
Usando i teoremi di densit\'a, troviamo che la precedente  uguaglianza \e verificata per $v \in H_0^1(\Omega)$, e quindi $u_c$ \e soluzione debole. Viceversa, dimostrare che una soluzione debole $u_d$ \e anche regolare (ovvero $u_d \in C^2(\Omega)$) \e un risultato non banale; una volta dimostrato che $u_d \in C^2(\Omega)$ una semplice integrazione per parti permette di concludere che $u_d$ \e una soluzione classica.
} \end{rem}

\begin{thm}[Il problema di Dirichlet non omogeneo] Sia dato il problema 
\begin{equation}
\label{eq_D1}
\left\{
\begin{array}{ll}
- \Delta u + u = f
\\
u |_{\partial \Omega} = g
\end{array}
\right.
\end{equation}
dove $f \in L^2(\Omega)$, $g \in C(\partial \Omega)$. Se $g = \tilde g |_{\partial \Omega}$ per qualche $\tilde g \in H^1(\Omega) \cap C(\ovl \Omega)$, allora esiste ed \e unica la soluzione debole $u \in H^1(\Omega)$ di (\ref{eq_D1}), come soluzione del problema variazionale
\[
u 
\ = \
\min_{v \in K}
\left\{
\frac{1}{2} \int_\Omega \left( | \nabla v |^2 + v^2 \right)
-
\int_\Omega fv
\right\}
\ ,
\]
dove
\[
K 
:=
\{
v \in H^1(\Omega) : v - \tilde g \in H_0^1(\Omega)
\}
\ .
\]
\end{thm}

\begin{proof}[Sketch della dimostrazione]
Sia osserva che $K$ \e un chiuso convesso in $H^1(\Omega)$, non dipendente da $\tilde g$ ma solo da $g$. A questo punto, si applica il teorema di Stampacchia.
\end{proof}

\begin{prop}[Il problema di Sturm-Liouville]
Sia dato il problema
\begin{equation}
\label{eq_SL}
\left\{
\begin{array}{ll}
- (pu')' + qu = f
\\
u(0) = u(1) = 0
\end{array}
\right.
\end{equation}
dove $p \in C^1([0,1])$, $p \geq \alpha$ con $\alpha \in \bR^+ - \{ 0 \}$, $q \in C([0,1])$, $f \in L^2$. Se $q \geq 0$, allora esiste ed \e unica la soluzione debole di (\ref{eq_SL}), come soluzione del problema variazionale
\[
u 
\ = \
\min_{v \in H_0^1(0,1)}
\left\{
\frac{1}{2} \int_0^1 [ p(v')^2 + qv^2 ]
-
\int_0^1 fv
\right\}
\ .
\]
\end{prop}

\begin{proof}[Dimostrazione]
Poich\'e $q \geq 0$ abbiamo
\[
\int_0^1 q u^2 \geq 0
\ .
\]
Usando la diseguaglianza precedente e la diseguaglianza di Poincar\'e troviamo
\[
\int_0^1 [ p(u')^2 + q u^2 ] 
\ \geq \ 
\alpha \int_0^1 (u')^2 
\ \stackrel{(\ref{eq_Poincare})}{\geq} \
\alpha c \| u \|_{W,2}^2
\ .
\]
Per cui la forma
\[
A(u,v) := \int_0^1 ( pu'v' + q uv ) 
\ \ , \ \
u,v \in H_0^1(0,1)
\ ,
\]
\'e coercitiva. Possiamo quindi applicare il teorema di Lax-Milgram.
\end{proof}

Con metodi analoghi (e qualche piccola variante), siamo in grado di dimostrare il seguente
\begin{prop}[Il problema di Neumann]
Sia dato il problema
\begin{equation}
\label{eq_N}
\left\{
\begin{array}{ll}
- u'' + u = f
\\
u'(0) = u'(1) = 0
\end{array}
\right.
\end{equation}
con $f \in L^2$. Allora esiste ed \e unica la soluzione debole $u \in H^2(0,1)$ di (\ref{eq_SL}), come soluzione del problema variazionale
\[
u 
\ = \
\min_{v \in H_0^1(0,1)}
\left\{
\frac{1}{2} \int_0^1 [ (v')^2 + v^2 ]
-
\int_0^1 fv
\right\}
\ .
\]
\end{prop}

%

\end{document}